\documentclass{article}
\usepackage{amsmath, amsfonts, amssymb, mathrsfs, amsthm, accents, sectsty,hyphenat, palatino, calc, url, stmaryrd, mathdots,array}

\usepackage{authblk}
\usepackage[all]{xy}
\usepackage[usenames]{color}

\usepackage[T3,T1]{fontenc}
\DeclareSymbolFont{tipa}{T3}{cmr}{m}{n}
\DeclareMathAccent{\invbreve}{\mathalpha}{tipa}{16}

\def\mf#1{\mathfrak{#1}}
\def\mc#1{\mathcal{#1}}
\def\mb#1{\mathbb{#1}}
\def\tx#1{\mathrm{#1}}
\def\tb#1{\textbf{#1}}

\def\ts#1{\textsf{#1}}
\def\tr{\tx{tr}\,}
\def\R{\mathbb{R}}

\def\C{\mathbb{C}}
\def\Q{\mathbb{Q}}
\def\A{\mathbb{A}}
\def\H{\mathbb{H}}
\def\Z{\mathbb{Z}}
\def\N{\mathbb{N}}
\def\F{\mathbb{F}}
\def\G{\mathbb{G}}

\def\rel{\textrm{-rel}}
\def\Grel{\mathbb{G}\textrm{-rel}}

\def\fka{\mathfrak{a}}
\def\fke{\mathfrak{e}}
\def\fkh{\mathfrak{h}}

\def\fkX{\mathfrak{X}}
\def\cA{\mathcal{A}}
\def\cC{\mathcal{C}}
\def\cE{\mathcal{E}}
\def\cG{\mathcal{G}}
\def\cH{\mathcal{H}}
\def\cI{\mathcal{I}}
\def\cL{\mathcal{L}}
\def\cS{\mathcal{S}}
\def\cN{\mathcal{N}}
\def\cZ{\mathcal{Z}}
\def\cX{\mathcal{X}}
\def\cY{\mathcal{Y}}
\def\mC{\mathscr{C}}
\def\dblbs{\backslash\!\backslash}

\def\lmod{\setminus}
\def\bs{\backslash}
\def\Gal{{\rm{Gal}}}
\def\ol#1{\overline{#1}}

\def\hat{\widehat}

\def\sspl{\rm{ss}}
\def\Res{\rm{Res}}
\def\der{\rm{der}}
\def\bdd{\rm{bdd}}

\def\Int{\tx{Int}}
\def\iso{\tx{iso}}
\def\simp{\tx{sim}}
\def\Im{\tx{Im}}
\def\Cent{\tx{Cent}}
\def\Irr{\tx{Irr}}
\def\isom{\stackrel{\sim}{\ra}}

\def\GL{\tx{GL}}
\def\SL{\tx{SL}}
\def\U{\tx{U}}
\def\Sp{\tx{Sp}}
\def\SU{\tx{SU}}
\def\SO{\tx{SO}}
\def\O{\tx{O}}
\def\Hom{\tx{Hom}}
\def\End{\tx{End}}
\def\Aut{\tx{Aut}}
\def\aut{\tx{aut}}
\def\Out{\tx{Out}}
\def\disc{\tx{disc}}
\def\cusp{\tx{cusp}}
\def\reg{\tx{reg}}
\def\el{\tx{ell}}

\def\temp{\tx{temp}}
\def\unit{\tx{unit}}
\def\cont{\tx{cont}}
\def\scusp{\tx{scusp}}

\def\id{\tx{id}}
\def\sgn{\tx{sgn}}
\def\Sym{\tx{Sym}}
\def\As{\tx{As}}

\def\Ind{\tx{Ind}}
\def\Res{\tx{Res}}
\def\EXC{\tx{exc}}
\def\rad{\tx{rad}}

\def\rw{\rightarrow}

\def\lrw{\longrightarrow}
\def\hrw{\hookrightarrow}

\def\ra{\rightarrow}

\def\lra{\longrightarrow}
\def\hra{\hookrightarrow}

\def\dirlim#1{\lim\limits_{\substack{\longrightarrow\\#1}}}

\def\sm{\smallsetminus}
\def\<{\langle}
\def\>{\rangle}
\def\lg{\langle}
\def\rg{\rangle}

\def\tilde{\widetilde}

\def\srad{S^\tx{rad}}

\newtheorem{thm}{Theorem}[subsection]
\newtheorem{thm*}[thm]{Theorem*}
\newtheorem{lem}[thm]{Lemma}
\newtheorem{lem*}[thm]{Lemma*}
\newtheorem{pro}[thm]{Proposition}
\newtheorem{prop*}[thm]{Proposition*}

\newtheorem{cor}[thm]{Corollary}
\newtheorem{cor*}[thm]{Corollary*}
\newtheorem{dfn}[thm]{Definition}
\newtheorem{fct}[thm]{Fact}

\newtheorem{rem}[thm]{Remark}

\newtheorem{lemm}[thm]{Lemma}
\newtheorem{prop}[thm]{Proposition}
\newtheorem{hypo}[thm]{Hypothesis}

\newtheorem{coro}[thm]{Corollary}

\providecommand{\abs}[1]{\lvert#1\rvert}

\sectionfont{\center\sc\normalsize}
\subsectionfont{\bf\normalsize}

\numberwithin{equation}{subsection}

\title{Endoscopic Classification of Representations: Inner Forms of Unitary Groups}

\author{Tasho Kaletha\thanks{partially supported by NSF grant DMS-161489.}}
\author{Alberto Minguez\thanks{partially supported by the
ANR ArShiFo ANR-BLAN-0114, MTM2010-19298 (FEDER) and P12-FQM-2696.}}
\author{Sug Woo Shin\thanks{partially supported by NSF grant DMS-1162250
  and a Sloan Fellowship.}}
\author{Paul-James White}
\affil{Harvard/Princeton}
\affil{Jussieu}
\affil{UC Berkeley/MIT}
\affil{Oxford}

\begin{document}

\date{\today}
\maketitle
\let\thefootnote\relax\footnote{
The research leading to these results has received funding from the European Research Council under the
European Community's Seventh Framework Programme (FP7/2007-2013) / ERC Grant agreement n. 290766 (AAMOT).
}

\newpage

\tableofcontents
\newpage

The principal goal of our paper is to classify the automorphic representations (over number fields) or the irreducible admissible representations (over local fields) of unitary groups which are not quasi-split, under the assumption that the same is known for quasi-split unitary groups. The classification of automorphic representations is given in terms of automorphic representations of general linear groups. At the core of our proof is the trace formula as in Arthur's seminal work \cite{Arthur} but our novelty is to treat the new phenomena which do not appear in the quasi-split case. We note that the local classification result for general reductive groups over local archimedean fields has been known due to Langlands and Shelstad. %
 Concerning the analogue for unitary groups over non-archimedean fields, complete results are due to Rogawski \cite{Rog90} for unitary groups in up to three variables and some partial results are due to Moeglin \cite{Moe11} in general. Our secondary aim is to build some foundation for the endoscopic study of more general non-quasi-split reductive groups. Arthur himself outlined a strategy to deal with inner forms of quasi-split reductive groups, most notably in the last chapter of \cite{Arthur}. The basic materials we develop are largely complementary to his.

The current work constitutes the first installment in a series of three papers. Here in the first part we state the main results (Theorems \ref{thm:locclass-single}, \ref{thm:main-global}, and \ref{thm:lir}) in full generality and establish most of the main argument, but obtain complete results only for tempered parameters (i.e. $A$-parameters which are generic in Arthur's terminology) locally and globally and only for unitary groups arising from Hermitian spaces associated to quadratic extensions of local or global fields. In the local case, this accounts for all unitary groups, but in the global case, this omits unitary groups arising from more general central simple algebras equipped with an involution of the second kind. In particular we obtain a precise version of the local Langlands correspondence for all unitary groups, which includes a partition of the tempered spectrum into $L$-packets according to tempered parameters, endoscopic character identities for tempered representations, and internal parametrisation of each tempered packet via certain characters of the centralizer group of the tempered parameter. (The Langlands quotient theorem allows us to extend the construction of $L$-packets from the tempered case to the general case.) We are going to extend our results to the case of non-tempered local and global $A$-parameters (but still for pure inner twists) in the second part of the series and to all unitary groups in the final part. As each extension requires a good amount of additional technical work, we decided to separate it to keep the first paper readable and in a reasonable length. We marked each main theorem as well as some lemmas and propositions with symbol $*$ to indicate that the theorem, lemma, or proposition is only partially proven in this paper and to be completed in the next two papers.

As the work of Arthur \cite{Arthur} and Mok \cite{Mok} is the most relevant to our paper, we would like to explain the relevance at the outset. Arthur established the endoscopic classification for quasi-split symplectic and special orthogonal groups over local and global fields and also outlined some general strategy for other groups. Modeling on his work, Mok gave the analogous classification for quasi-split unitary groups. Our work borrows the basic structure of proof (as one can see from the table of contents) and the proof of a number of intermediate assertions from Arthur. From Mok's work we strived to import only the main theorems (other than the basic setup) so as to make our paper more self-contained. In this spirit we usually repeat or sketch the argument also when we could have referred to a result in \cite{Arthur} (with just a few words for modifications). We tried to choose compatible conventions (if not always notation) with \cite{Arthur} and \cite{Mok}, but see \S\ref{sub:notation} below for differences.

Having mentioned that we take on faith the main results of \cite{Mok}, which will be explained more precisely in \S\ref{sub:results-qsuni} below, our results in this first paper are unconditional otherwise (except we are still awaiting the sequel to \cite{CLWFL1} and \cite{CLWFL2}, where Chaudouard and Laumon will establish the weighted fundamental lemma in general; we need it for the stabilization of the untwisted trace formula for inner forms of unitary groups but note that the sequel is already needed in the quasi-split case).%
At the time of writing Mok's main results are themselves conditional (see \cite{Mok} for a fuller discussion), most notably on the stabilization of the twisted trace formula. Waldspurger in partial collaboration with Moeglin has been actively releasing a series of papers on the latter problem lately so we are hopeful that presumably the most serious hypothesis will be removed in the near future.

Now we discuss more details about inner forms of a quasi-split (unitary) group $G^*$ over a local or global field $F$ of characteristic zero and their endoscopic classification, which are the main subject of this paper and the sequels. More precisely we consider inner twists, which consist of $(G,\xi)$ where $G$ is an inner form of $G^*$ and $\xi:G^*\ra G$ is an isomorphism over $\ol{F}$, an algebraic closure of $F$, such that $\xi^{-1}\sigma(\xi)$ is an inner automorphism of $G^*$. Given an $L$-parameter or an $A$-parameter, it is believed that the corresponding packet is determined by $(G,\xi)$ in both local and global cases. However in the local case, it is not known how to give a natural internal parametrisation of each packet when $G$ is not quasi-split; in the quasi-split case there is a canonical way relying on an extra choice of Whittaker datum for $G=G^*$. A closely related problem is that there is no natural way to normalize local endoscopic transfer factors for $G$ in general. One might try to proceed by fixing an arbitrary normalization of transfer factors at each place but this would entail multiple issues. For example, one would need to worry about compatibility of choices when dealing with multiple groups simultaneously. One would also need to incorporate these choices in all stated theorems. In particular, the local classification theorem will have to give the internal parametrisation of $L$-packets dependent on the arbitrary choice of a transfer factor.

In response to these issues, various rigidifications of inner twists have been suggested. A full discussion of this matter would lead us too far off field and we refer the reader to \cite{Kal14} and the introduction to \cite{Kal13}. Here we limit ourselves to a brief overview. The prototype is Vogan's notion of pure inner twists. An (equivalence class of) pure inner twist corresponds to an element of the Galois cohomology set $H^1(\tx{Gal}(\ol{F}/F),G^*)$ over any ground field $F$, in contrast to an (equivalence class of) an inner twist, which corresponds to an element of $H^1(\tx{Gal}(\ol{F}/F),G^*_\tx{ad})$. When $F$ is local, a pure inner twist and a Whittaker datum for $G^*$ determine a canonical normalization of the transfer factors. When $F$ is a number field, a pure inner twist and a Whittaker datum for $G^*$ lead to normalized transfer factors at each place which satisfy the necessary product formula, as can be shown by means of results on Galois cohomology for algebraic groups. However the limitation of pure inner twists is that not all inner twists can be rigidified as pure inner twists. Extensions of the concept of pure inner twists in the local context have been put forward by Adams-Barbasch-Vogan \cite{ABV92} in the archimedean case and by Kaletha \cite{Kal13} in both the archimedean and the non-archimedean cases, the latter nontrivially coinciding with the former in the archimedean case. These extensions have the capacity to rigidify all inner twists over local fields but their global analogues have not been found yet. A different extension stems from the work of Kottwitz on isocrystals with additional structure. It was initially only available in the non-archimedean context \cite{Kot85}, \cite{Kot97}, but was recently extended to the archimedean context as well as to the global context \cite{Kot14}. This extension has the capacity to rigidify all inner twists of a quasi-split group with connected center. It is this latter extension of the concept of pure inner twists that we have chosen to use in this article. Since the center of a quasi-split unitary group is connected, every inner twist  can be rigidified using an isocrystal, and we call the result an extended pure inner twist.

Once an inner twist and its rigidification are given the precise statements of the main theorems in the local and global cases are given in \S\ref{sub:main-local-thm} and \S\ref{sub:main-global-thm} below. In the course of proving them, however, we were surprised by the subtlety of the local intertwining relation (LIR) for inner forms (of unitary groups). Not only its proof but its correct formulation already seems highly nontrivial to us even with the notion of extended pure inner twist in hand, much of the difficulty being rooted in the correct normalization of intertwining operators for parabolically induced representations. It is perhaps fair to say that the difficulty and novelty of our work are mainly concentrated on the LIR. Actually the LIR is the main reason why we divide our work into three papers; it seems sensible to establish the LIR in three steps as quite a few technical results may be reasonably isolated to the second and third steps. Once LIR is proven in full, the main local theorem (i.e. the local classification theorem) for unitary groups will follow rather painlessly by imitating the arguments in the quasi-split case. On the other hand it is worth pointing out that the LIR is still far from trivial for inner twists of a general linear group and will be completed in the third article, even though the local classification theorem is already known for them. Unlike the quasi-split case, the main global theorem is an immediate consequence of the LIR and the local classification theorem.

The general structure of our argument is similar to that in the quasi-split case but there are differences. The spine of the argument is again a long induction that is resolved at the end of the paper. However, we have a stronger starting point than is the case in \cite{Arthur} and \cite{Mok} since we are given some fundamental facts about the parameters and stable linear forms on quasi-split groups (such as the two seed theorems, cf. Propositions \ref{prop:1st-seed-thm}, \ref{prop:2nd-seed-thm}, and \ref{prop:local-stable-linear} below) as well as the stable multiplicity formula (Proposition \ref{p:stable-multiplicity}). The twisted trace formula also plays a smaller role than in \cite{Arthur} and \cite{Mok}. In fact, in the current paper we can make do without it, but it will be used in the third part of our series to establish the LIR in the remaining cases. Thanks to all this our use of the trace formula is mostly more transparent and simpler than \cite{Arthur} and \cite{Mok}. It is also worth noting that our proof of LIR is quite different. It might look like a great advantage to assume the LIR for quasi-split groups from the start, but this does not take us too far because we cannot always realize a non-quasi-split local unitary group as the localization of a global unitary group at a place, with the property that the group is quasi-split away from that particular place, when the number of variables is even. It turns out that we should generally allow the global unitary group to be non-quasi-split at another auxiliary place. This roughly means that if the LIR is known at one non-quasi-split place then one hopes to deduce the LIR at the other non-quasi-split place by a global method based on the trace formula. So the basic strategy for proving LIR is to build a reduction step by globalising the given non-quasi-split local unitary group and its tempered parameter to a global unitary group and a global parameter such that the global parameter at the auxiliary place is easier to deal with than the original local parameter. Combining this idea with other reduction steps of purely local nature due to Arthur, we reduce the proof of LIR in some sense to the proof of LIR for the real unitary group $U(3,1)$ relative to its Levi subgroup isomorphic to $U(1,1)\times \C^\times$ and some special parameters amenable to explicit computations.

Our style of exposition is deliberately mixed. We elaborate, sometimes in great length, on the foundational or new materials regarding inner twists and their endoscopic study, such as how extended pure inner twists are tied with the normalization of transfer factors and local intertwining operators. We hope that this will serve as a stepstone to further investigation of inner twists. In contrast we often try to be brief and refer to an external reference (most likely \cite{Arthur}) for details especially when we are proving assertions which can be shown in the same way as for quasi-split groups with rather obvious modifications. For instance this applies to the so-called standard model in \S\ref{sub:std-model}, the initial step in the comparison of the trace formulas. On the other hand we occasionally give more details than Arthur even if the statements and proofs are similar, when we believe the effort is worthwhile. This is the case, for example, in the reduction steps when proving the LIR, cf. \S\ref{sub:prelim-local-intertwining1} and \S\ref{sub:prelim-local-intertwining} below. Let us also mention that the contents are arranged in such a way that whenever possible, we explain what happens for general reductive groups first and then specialize to (inner twists of) unitary groups and general linear groups.

Now we explain the organization of the paper with some commentaries. Chapter \ref{s:chapter0} is the preliminary chapter where basic definitions and properties of unitary groups and their $L$-groups, their inner twists and extended pure inner twists as well as their Levi subgroups, parabolic subgroups and the associated Weyl groups. Part of the chapter would offer an introduction to the relatively new notion of extended pure inner twists with examples.
Chapter \ref{s:param-main-thm} begins with some basic definitions and technical lemmas in endoscopy. Next we introduce local and global parameters for unitary and general linear groups. As in \cite{Arthur} and \cite{Mok} our global parameters are formally constructed from automorphic representations of general linear groups to avoid the issue with the hypothetical Langlands group. There is nothing new except \S\ref{subsub:relevance-global} on the relevance of parameters, which matters only outside the quasi-split case, and \S\ref{subsub:canonical-def} on canonical global parameters (without reference to any $L$-morphism). The endoscopic correspondence for parameters in \S\ref{sub:endo-correspondence} is not new either but we included a version which is suitable for extended pure inner twists. After giving a precise list of main results we are importing from \cite{Mok}, we state the local and global main theorems in the last two sections of Chapter \ref{s:param-main-thm}. The proof of these theorems will take up this paper and the two forthcoming papers except that we establish on the spot Theorem \ref{thm:LLC-GLmD} concerning the inner twists of local general linear groups. Chapter \ref{chapter2} is all about local intertwining operators and LIR. Our sections \S\ref{sec:diag} through \ref{sec:lir} correspond to much of \cite[2.3-2.4]{Arthur} but contain new materials and constructions as we remarked earlier. Notice that our statement of the LIR in Theorem \ref{thm:lir} consists of three parts, only the last of which was referred to as the LIR in \cite{Arthur}. We initiate the reduction steps for LIR by purely local methods in \S\S\ref{sub:prelim-local-intertwining1}-\ref{sub:prelim-local-intertwining} and then verify LIR for $U(3,1)$ in some special cases, which will enter the proof of LIR as the basis of the inductive argument. Chapter \ref{chapter3} is concerned with the global machinery based on the trace formula and analogous to \S3.1 through \S4.5 of \cite{Arthur}. We claim little originality here other than taking care of the necessary changes and additional justifications for inner twists. Chapter \ref{chapter4} is similar to \S6.2 through \S6.7 of \cite{Arthur} but the proof of LIR is different. As explained above we need different kinds of globalizations as constructed in \S\ref{sub:Globalization-param} to feed into the trace formula argument in the following two subsections. All this effort toward LIR comes together in \S\ref{sub:LIR-proof} to finish the proof for tempered parameters on local unitary groups. The remainder of the chapter flows in a similar manner as the parallel part in \cite{Arthur} and leads to the proof of the main local theorem for the same parameters. Chapter \ref{chapter5} carries out the rather short and straightforward proof of the main global theorem assuming all the local results. So the proof is complete only for global generic parameters on pure inner twists of a quasi-split unitary group. Finally Appendix \ref{sec:appendix} extends Ban's result on the invariance of the Knapp-Stein $R$-groups under Aubert involutions to the case of unitary groups and beyond. Thereby it fills in the gap in the case of quasi-split unitary groups and will provide us with an input in the sequels to this paper.

We end introduction by mentioning the results which are not completed in this paper. The immediate sequel \cite{KMS_A} will be devoted to the study of non-generic local $A$-parameters, thereby completing the proof of the main local theorem (Theorem \ref{thm:locclass-single}) in the remaining cases. Lemma \ref{lem:iop1m} will be fully proved along the way. In the final paper \cite{KMS_B} we will justify some facts about transfer factors which are unproven outside the case of pure inner twists (see \S\ref{subsub:normtf} below) and settle Lemmas \ref{lem:iop1pm}, \ref{lem:lirdesc1}, and \ref{lem:lirdesc2} as well as Proposition \ref{pro:iop3lg} and Theorem \ref{thm:main-global} in complete generality.

\bigskip

\noindent\textbf{Acknowledgment}

  We are grateful to Jussieu and MIT for their hospitality during our visit and collaboration in the summer of 2012 and 2013, respectively. Our project began in Jussieu and would have been impossible without the generous help and support from Michael Harris. It will be obvious to the reader that our work owes immensely to the lifelong work of James Arthur. Special thanks are due to him for his patience to answer our questions and his kindness to allow access to his book \cite{Arthur} before publication. We are grateful to Robert Kottwitz for providing the cohomological foundations of extended pure inner twists and preparing the manuscript \cite{Kot14} that is used throughout this paper. We would like to thank Ioan Badulescu, Wee Teck Gan, Kaoru Hiraga, Atsushi Ichino, and Colette Moeglin for helpful discussions and communications.

\setcounter{section}{-1}

\section{Chapter 0: Inner Twists of Unitary Groups}\label{s:chapter0}

\subsection{Notation}\label{sub:notation}

  In this paper we declare once and for all that every field is of characteristic zero. So a local field is either $\R$, $\C$, or a finite extension of $\Q_p$ for some prime $p$. A global field is going to be synonymous to a number field, a finite extension of $\Q$.
  Let $F$ be a field (of characteristic zero). We denote its full Galois group $\Gal(\ol{F}/F)$ by $\Gamma $ or $\Gamma_F$.
  Now let $E/F$ be a quadratic extension of number fields, and let $v$ be a place of $F$. We write $F_v$ for the completion of $F$ with respect to $v$, and $\Gamma_v$ or $\Gamma_{F_v}$ for $\Gal(\ol{F}_v/F_v)$. For such an $F$ we fix an $F$-embedding $\ol{F}\hra \ol{F}_v$ between algebraic closures for every place $v$ of $F$ so that there is an induced embedding $\Gamma_v\hra \Gamma$ at every $v$. Moreover we fix an $F$-embedding $E\hra \ol{F}$ once and for all. This determines an embedding $E\hra \ol{F}_v$ for each $v$, singling out a distinguished place $w$ of $E$ above $v$. When $v$ splits in $E$, we often write $\ol w$ for the other place above $v$. The ring of ad\'eles over a number field $F$ is denoted $\A_F$, or simply $\A$ if the field $F$ is well understood. For a finite set $S$ of places of $F$ we write $\A_F^S$ for the ring of ad\'eles away from $S$ (i.e. the restrict product is taken away from $S$).

  When $G$ is an algebraic group over a field $E$, which is an extension of a field $F$, we use $\Res_{E/F} G$ to denote the Weil restriction of scalars. The center of a group $G$ is written as $Z(G)$.

  Every infinite tensor product (over the places of a number field) is denoted by the usual symbol for tensor product $\otimes$ but understood as a restricted tensor product.

  By $\SU(2)$ we denote the compact special unitary group in two variables.

  The following is the list of major instances where our notation or convention is quite different from \cite{Arthur} and \cite{Mok}.
\begin{itemize}
  \item Here $G$ is typically an inner form of $U_{E/F}(N)$ (or a connected reductive group in a general discussion). We write $G^*$ for the quasi-split form of $G$, which is hence $U_{E/F}(N)$ in most instances.
   \item In \cite{Mok} an endoscopic datum for the quasi-split unitary group $U_{E/F}(N)$ is denoted $(G',s',\xi')$ (or $U(N_1)\times U(N_2),\zeta_\chi$), often abbreviated as $G'$, and the transfer of a function $f$ on $U_{E/F}(N)$ to $G'$ is denoted $f'$. Our notation for an endoscopic datum is $(G^{\fke},s^\fke,\eta^\fke)$ and abbreviated as $\fke$. (In the twisted case, appearing rarely, write $\tilde\fke=(G^{\tilde\fke},s^{\tilde\fke},\eta^{\tilde\fke})$.) The transfer from $G$ to $G^\fke$ is denoted $f^\fke$. To be precise the transfer $f^\fke$ also depends on the structure of inner twist or extended pure inner twist on $G$, but we omit the reference to such a structure in the notation.
   \item When $\psi$ is an (Arthur) parameter for $G$, over either a local or global field, their group $\cS_\psi$ is denoted by $\ol{\cS}_\psi$ in our paper. We generally use the bar notation to indicate that the definition involves taking quotient by the Galois invariants in the center of $\hat G$. (See diagram \eqref{eq:diagram-mod-central-global} below for instance.) Our $\cS_\psi$ modulo such Galois invariants is their $\cS_\psi$.
  \item The conjugate dual of an irreducible representation $\pi$ is denoted $\pi^*$ in \cite{Mok}. We write $\pi^\star$ to avoid using $*$ too much.
  \item Our formulation of local intertwining operators and local intertwining relations is a little different. It is technical to explain here but the reader is referred to Chapter \ref{chapter2} below for details.

\end{itemize}

\subsection{Reductive groups and their $L$-groups}\label{sub:L-groups}

\subsubsection{The based root datum and the Langlands dual group} \label{subsub:brd}

Let $G$ be connected reductive group defined over an algebraically closed field. Associated to $G$ there is the based root datum $\tx{brd}(G)=(X,\Delta,Y,\Delta^\vee)$ obtained as follows. For any Borel pair $(T,B)$ of $G$ we have the based root datum
\[ \tx{brd}(T,B)=(X^*(T),\Delta(T,B),X_*(T),\Delta(T,B)^\vee). \]
Given two pairs $(T_i,B_i)$, $i=1,2$, there exists a canonical isomorphism $T_1 \rw T_2$ given by conjugation by any element of $g$ which sends $(T_1,B_1)$ to $(T_2,B_2)$. This isomorphism produces an isomorphism $\tx{brd}(T_1,B_1) \rw \tx{brd}(T_2,B_2)$. In this way, we obtain a system $\tx{brd}(T,B)$ indexed by the set of Borel pairs for $G$, and taking its inverse limit, we obtain $\tx{brd}(G)$.

In fact, in the same way one can define the absolute Borel pair $(T,B)$ of $G$.

Let $G$ now be a quasi-split connected reductive group defined over a field $F$ of characteristic zero. Fix an algebraic closure $\ol{F}$ of $F$ and let $\Gamma$ denote the Galois group of $\ol{F}/F$. If a Borel pair $(T_1,B_1)$ of $G$ is defined over $F$, then $\tx{brd}(T_1,B_1)$ carries a natural $\Gamma$-action. If $(T_2,B_2)$ is a second such pair, then the natural isomorphism $\tx{brd}(T_1,B_1) \rw \tx{brd}(T_2,B_2)$ is $\Gamma$-equivariant. This gives a natural $\Gamma$-action on $\tx{brd}(G)$.

A dual group for $G$ is the datum of a connected reductive group $\hat G$ defined over some field, which we will take to be $\C$, together with an action of $\Gamma$ on $\hat G$ by algebraic automorphisms preserving a splitting of $\hat G$, and a fixed $\Gamma$-equivariant isomorphism $\tx{brd}(\hat G) \cong \tx{brd}(G)^\vee$. Here the $\Gamma$-action on $\tx{brd}(\hat G)$ is derived in the same way as that on $\tx{brd}(G)$, namely using $\Gamma$-invariant Borel pairs $(\hat T,\hat B)$ which exist by assumption.

Given $\hat G$, we have the Galois form of the $L$-group given by $^LG = \hat G \rtimes \Gamma$. If $F$ is a local or global field, it may also be useful to consider the Weil-form of the $L$-group, given by $^LG = \hat G \rtimes W_F$, where $W_F$ is the Weil group of $F$.

\subsubsection{Quasi-split unitary groups $U_{E/F}(N)$}\label{subsub:qsuni}

  In this paragraph and the next we set up the basic notation for quasi-split unitary groups as well as the relevant general linear groups which possess unitary groups as twisted endoscopic groups. The $L$-groups of these groups will be made explicit below. In the next subsection they will be related to each other via $L$-morphisms and endoscopic data.

Let $E$ be a quadratic $F$-algebra so that $E$ is either a quadratic field extension or isomorphic to $F\times F$. In the latter case fix the isomorphism and identify $E=F\times F$. Define $c\in \Aut_F(E)$ to be the unique nontrivial automorphism in the former case. In the latter $c(x,y):=(y,x)$ for $x,y\in F$. Following Rogawski \cite[\S1.9]{Rog90}, we define the quasi-split unitary group in $N$-variables defined over $F$ and split over $E$ as
\[ U_{E/F}(N)(\ol{F}) = \tx{GL}_N(\ol{F}) \]
with the Galois action is given by
\[ \sigma^N(g) = \theta_{N,\sigma}(\sigma(g)) \]
for any $\sigma \in \Gamma$, $g \in \tx{GL}_N(\ol{F})$. Here,
\[ \theta_{N,\sigma}(g) = \begin{cases} g&,\sigma \in \Gamma_E\\ \tx{Ad}(J_N)g^{-t}&,\sigma \in \Gamma_F \sm \Gamma_E \end{cases} \]
where
\[ J_N = \begin{bmatrix} &&&&1\\&&&-1\\&&1\\&\iddots\\(-1)^{N-1} \end{bmatrix} \]
and $g^{-t}$ denotes the transpose of the inverse of $g$. This is the definition used in \cite[\S1.9]{Rog90} (the element $\xi \in E^\times$ of trace zero used there is irrelevant for the definition of the group and we have thus omitted it). The standard splitting $(T^*,B^*,\{X_\alpha\})$ of $\tx{GL}_N$, consisting of the group $T^*$ of diagonal matrices, the group $B^*$ of upper triangular matrices, and the set $\{E_{i,i+1}\}_{i=1,\dots,N-1}$, is invariant under $\sigma^{N}$ and provides a standard $F$-splitting for $U_{E/F}(N)$.

As a dual group for $U_{E/F}(N)$, we fix $\hat U_{E/F}(N) = \tx{GL}_N(\C)$ endowed with the standard splitting $(\hat T,\hat B,\{\hat X_\alpha\})$, where $\hat T$ is the subgroup of diagonal matrices, $\hat B$ is the subgroup of upper triangular matrices, and $\{\hat X_\alpha\}$ is the subset of root vectors given by $E_{i,i+1}$ for $i=1,\dots,N-1$. The action of $\Gamma$ factors through $\Gamma_{E/F}$ and the non-trivial element $\tau \in \Gamma_{E/F}$ acts by $g \mapsto \tx{Ad}(J_N)g^{-t}$. The adjoint group $U_{E/F}(N)_\tx{ad}=U_{E/F}(N)/Z(U_{E/F}(N))$ has as its dual group $[\hat U_{E/F}(N)]_\tx{sc}=\tx{SL}_N(\C)$. The centers of the dual groups of $U_{E/F}(N)$ and $U_{E/F}(N)_\tx{ad}$ are given by $\C^\times$ and $\mu_N$, respectively, with $\tau$ acts by inversion on both of these subgroups. We form the $L$-group $^L  U_{E/F}(N)$ to be the semi-direct product $\tx{GL}_N(\C)\rtimes \Gamma$ using the above $\Gamma$-action.

So far we allowed the case when $E=F\times F$. It is worth noting some simplifications in this case. It is easy to see that the projection from $F\times F$ to the first (resp. second) copy of $F$ induces an isomorphism $U_{E/F}(N)\simeq \GL(N)_F$, to be denoted $\iota_1$ (resp. $\iota_2$). An easy computation shows that $\iota_2\iota_1^{-1}(g)=J_N {}^t g^{-1} J_N^{-1}$. Moreover $\iota_j$, $j\in\{1,2\}$, induces isomorphisms $\hat \iota_j:\hat{\GL}(N)_F\simeq \hat U_{E/F}(N)$ and $^L \iota_j:{}^L \GL(N)_F\simeq {}^L U_{E/F}(N)$.

When $E/F$ is a quadratic extension of number fields and $v$ is a place of $F$ split as $w\ol w$ in $E$, we may assume that $w$ is a distinguished place, cf. \S\ref{sub:notation}. In this case we identify $E_v:=E\otimes_F F_v=F_w\times F_{\ol w}$ and write $\iota_w$ for $\iota_1$ and $\iota_{\ol w}$ for $\iota_2$. Similarly for $\hat\iota_w$, $^L \iota_w$, $\hat\iota_{\ol w}$, and $^L \iota_{\ol w}$.

\subsubsection{The groups $G_{E/F}(N)$ and $\tilde G_{E/F}(N)$}\label{subsub:G(N)-defn}

  Let $E$, $F$, and $c$ be as above. Denote by $\Res_{E/F}$ the Weil restriction of scalars. By applying $\Res_{E/F}$ to $\GL(N)$ over $E$, we obtain an $F$-algebraic group
  $$G_{E/F}(N):=\Res_{E/F}\GL(N)_E.$$
  We sometimes simplify $G_{E/F}(N)$ as $G(N)$ when there is no danger of confusion. There is an $F$-automorphism of order two on $G_{E/F}(N)$ given by
  $$\theta(g):=J_N {}^t c(g)^{-1}J_N^{-1},\quad g\in G_{E/F}(N).$$
  The dual group and $L$-group of $G_{E/F}(N)$ are given as
  $$\hat G_{E/F}(N) = \GL(N,C)\times \GL(N,\C),\quad ^L G_{E/F}(N) = \hat G_{E/F}(N)\rtimes W_F,$$
  where $W_F$ acts through $\Gal(E/F)$ on $\hat G_{E/F}(N)$ by permuting the two factors (so that the nontrivial Galois element sends $(g_1,g_2)$ to $(g_2,g_1)$.) Here we are using the standard $F$-splitting of $G_{E/F}(N)$ consisting of the diagonal maximal torus, the upper triangular Borel, and $\{E_{i,i+1}\}_{i=1}^{N-1}$ on each copy of $ \GL(N,C)$.
  We have an automorphism $\hat \theta$ of $\hat G_{E/F}(N)$ such that
  $$\hat \theta (g_1,g_2)=(J_N {}^t g_2^{-1} J_N^{-1},J_N {}^t g_1^{-1} J_N^{-1}).$$
  (There is a general recipe for $\hat \theta$ when $\theta$ is given, cf. \S\ref{subsub:endoscopic-triples} below.)
  Define a $G_{E/F}(N)$-bitorsor $$\tilde G_{E/F}(N):=G_{E/F}(N)\rtimes \theta$$
  so that if we identify $\tilde G_{E/F}(N)\simeq G_{E/F}(N)$ via $g\rtimes \theta\mapsto g$ then the action of $g\in G_{E/F}(N)$ is multiplication by $g$ (resp. $\theta(g)$) on the left (resp. right).

\subsection{Inner twists} \label{sub:inner}

  Our goal in this volume is to give an endoscopic classification of representations for unitary groups which are not quasi-split. This can be done by realizing the unitary group as an inner twist of a quasi-split unitary group and comparing stable trace formulas. In order to carry out this comparison however, it turns out that the notion of an inner twist is insufficient. For a general discussion of the problem, we refer the reader to \cite{Kal14} and the introduction of \cite{Kal13}. In this section we will introduce two modifications of the notion of inner twist that will be important for us. The first was defined by Vogan in \cite{Vog93} and is called a \emph{pure inner twist}. This notion is sufficient when dealing with unitary groups associated to Hermitian spaces defined over quadratic extensions of the ground field. In particular, it is sufficient for all unitary groups over local fields. The second notion stems from the work of Kottwitz on isocrystals with additional structure \cite{Kot85}, \cite{Kot97}, \cite{Kot14} and is called an \emph{extended pure inner twist}. It generalizes the notion of a pure inner twist and is needed when one wants to work with a unitary group associated to a non-split central simple algebra endowed with an involution of the second kind. The additional datum provided by an (extended) pure inner twist will play a role not only in the technical parts of the paper, but already in the statement of our local theorems. The global theorems will be less dependent on this additional datum.

\subsubsection{Definition of inner twists and their variants}\label{subsub:inner}

Let $F$ be a field of characteristic zero, $\ol{F}$ an algebraic closure, and $\Gamma$ the Galois group of $\ol{F}/F$. Let $G$ be a connected reductive group defined over $F$. An inner form of $G$ is a connected reductive group $G_1$ defined over $F$ for which there exists an isomorphism $\xi : G \times \ol{F} \rw G_1 \times \ol{F}$ with the property that for all $\sigma \in \Gamma$, the automorphism $\xi^{-1}\sigma(\xi) = \xi^{-1}\circ\sigma\circ\xi\circ\sigma^{-1}$ of $G$ is inner. The isomorphism $\xi : G \rw G_1$ itself is called an inner twists. If $\xi_1 : G \rw G_1$ and $\xi_2 : G \rw G_2$ are two inner twists, then an isomorphism $\xi_1 \rw \xi_2$ consists of an isomorphism $f : G_1 \rw G_2$ defined over $F$ and having the property that $\xi_2^{-1}\circ f \circ \xi_1$ is an inner automorphism of $G$. The map $\xi \mapsto \xi^{-1}\sigma(\xi)$ sets up a bijection from the set of isomorphism classes of inner twists of $G$ to the set set $H^1(\Gamma,G_\tx{ad})$

For a fixed inner form $G_1$ of $G$, the set of all inner twists $G \rw G_1$ carries an equivalence relation, with two twists being equivalent if and only if $\tx{id}_{G_1}$ is an isomorphism between them. We will oftentimes work with the classes of this equivalence relation.

A \textbf{pure inner twist} $(\xi,z) : G \rw G_1$ is a pair consisting of an inner twist $\xi : G \rw G_1$ and an element $z \in Z^1(\Gamma,G)$ having the property that $\xi^{-1}\sigma(\xi) = \tx{Ad}(z_\sigma)$. Given two pure inner twists $(\xi_1,z_1) : G \rw G_1$ and $(\xi_2,z_2) : G \rw G_2$, an isomorphism $(\xi_1,z_1) \rw (\xi_2,z_2)$ consists of a pair $(f,g)$, with $f : \xi_1 \rw \xi_2$ an isomorphism of inner twists, and $g \in G$ an element satisfying the conditions $\xi_2^{-1}\circ f\circ\xi_1 = \tx{Ad}(g)$ and $z_2(\sigma) = gz_1(\sigma)\sigma(g)^{-1}$. The map $(\xi,z) \mapsto z$ sets up a bijection from the set of isomorphism classes of pure inner twists to the set $H^1(\Gamma,G)$.

Before we can introduce the notion of an extended pure inner twist, we need to recall some recent results of Kottwitz \cite{Kot14}. Let $F$ be a local or global field. In \cite[\S10]{Kot14}, Kottwitz has constructed a certain cohomology set $B(F,G)$ for each affine algebraic group $G$ defined over $F$. This set is constructed as the colimit
\[ B(F,G) = \varinjlim_K H^1_\tx{alg}(\mc{E}(K/F),G(K)) \]
where $K$ runs over the set of finite Galois extensions of $F$ in a fixed algebraic closure $\ol{F}$ of $F$, $\mc{E}(K/F)$ is a certain extension
\[ 1 \rw \mb{D}_{K/F}(K) \rw \mc{E}(K/F) \rw \Gamma_{K/F} \rw 1 \]
of the Galois group of the extension $K/F$ by group of $K$-points of a certain group $\mb{D}_{K/F}$ of multiplicaitve type defined over $F$, and $H^1_\tx{alg}(\mc{E}(K/F),G(K))$ is a set of algebraic cohomology classes of $\mc{E}(K/F)$ with values in $G(K)$, as defined in \cite[\S2]{Kot14}.

The set $B(F,G)$ comes equipped with a Newton map \cite[\S10.7]{Kot14}
\[ \nu : B(F,G) \rw [\tx{Hom}_{\ol{F}}(\mb{D}_F,G)/G(\ol{F})]^\Gamma, \]
where $\mb{D}_F = \varprojlim_K \mb{D}_{K/F}$. The preimage of $\tx{Hom}_F(\mb{D}_F,Z(G))$ under the Newton map is called the set of \emph{basic} elements, denoted by $B(F,G)_\tx{bsc}$. There is a natural inclusion $H^1(\Gamma,G) \rw B(F,G)_\tx{bsc}$. When the center $Z(G)$ of $G$ is trivial, this inclusion is also surjective. Using this fact, we obtain a natural map $B(F,G)_\tx{bsc} \rw B(F,G/Z(G))_\tx{bsc} \cong H^1(\Gamma,G/Z(G))$.

The notion of basic can already be defined on the level of cocycles, so we can speak of $Z^1_\tx{bsc}(\mc{E}(K/F),G(K))$. Moreover, if $H$ is another group for which there is a natural inclusion $Z(H) \subset Z(G)$, we can speak of $H$-basic elements, which are those whose Newton point belongs to the subgroup $Z(H)$ of $Z(G)$. We will use the corresponding notation $B(F,G)_{H-\tx{bsc}}$ and $Z^1_{H-\tx{bsc}}(\mc{E}(K/F),G(K))$. From an element of $Z^1_\tx{bsc}(\mc{E}(K/F),G(K))$ we obtain an element of $Z^1(\Gamma_{K/F},G_\tx{ad}(K))$. In other words, basic algebraic 1-cocycles of $\mc{E}(K/F)$ with values in $G(K)$ lead to inner forms of $G$.

When $F$ is local, Kottwitz has constructed \cite[\S11]{Kot14} a canonical map of pointed sets
\begin{equation} \label{eq:kotisoloc} \kappa_G:B(F,G)_\tx{bsc} \rw X^*(Z(\hat G)^\Gamma), \end{equation}
which is a bijection if $F$ is $p$-adic, and which extends the map
\[ H^1(\Gamma,G) \rw X^*(\pi_0(Z(\hat G)^\Gamma)) \]
defined in \cite{Kot86}.

Let now $F$ be global. For each place $v$ of $\ol{F}$ there is a localization map $B(F,G) \rw B(F_v,G)$ which preserves the property of being basic \cite[\S10.9]{Kot14}. If we choose one place of $\ol{F}$ lying over each place of $F$ and combine these localization maps, we obtain a map
\begin{equation} \label{eq:h1algloc} B(F,G)_\tx{bsc} \rw \coprod_v B(F_v,G)_\tx{bsc}, \end{equation}
where $\coprod$ denotes the subset of the direct product consisting of elements whose components at almost all $v$ are equal to the neutral element in $B(F_v,G)_\tx{bsc}$. Kottwitz has shown \cite[\S15.7]{Kot14} that the kernel of \eqref{eq:h1algloc} is in bijection with $\tx{ker}^1(F,G)$. The usual twisting argument tells us that the fiber over $b \in B(F,G)_\tx{bsc}$ is in bijection with $\tx{ker}^1(F,G^b)$, where $G^b$ is the inner form of $G$ corresponding to $b$. According to \cite[(4.2.2)]{Kot84} (in which the assumption of not having $E_8$ factors can be dropped by the work of Chernousov \cite{Cher89}), each set $\tx{ker}^1(F,G^b)$ is in bijection with the finite abelian dual to $\tx{ker}^1(\Gamma,Z(\hat G))$. Moreover, according to \cite[Prop. 15.6]{Kot14}, the image of \eqref{eq:h1algloc} is equal to the kernel of the composition
\begin{equation} \label{eq:kotisolocglo} \coprod_v B(F_v,G)_\tx{bsc} \stackrel{\eqref{eq:kotisoloc}}{\lrw} \bigoplus_v X^*(Z(\hat G)^{\Gamma_v}) \stackrel{\sum}{\lrw} X^*(Z(\hat G)^\Gamma). \end{equation}

We will be most interested in the case when $G$ is a quasi-split connected reductive group. Using the cohomology sets $B(F,G)$ for local and global fields we define an \textbf{extended pure inner twist} $(\xi,z) : G \rw G_1$ to consist of an inner twist $\xi : G \rw G_1$ and an element $z \in Z^1_\tx{bsc}(\mc{E}(K/F),G(K))$, for some finite Galois extension $K/F$, whose image in $Z^1_\tx{bsc}(\mc{E}(K/F),G_\tx{ad}(K))=Z^1(\Gamma_{K/F},G_\tx{ad}(K))$ equals $\xi^{-1}\sigma(\xi)$. Since $K$ is arbitrary, it will be convenient to use the symbol $Z^1_\tx{bsc}(\mc{E},G)$ instead of $Z^1_\tx{bsc}(\mc{E}(K/F),G(K))$. An isomorphism of extended pure inner twists is defined in the same way as for pure inner twists. The isomorphism classes of extended pure inner twists are in bijection with the set $B(F,G)_\tx{bsc}$. When $G$ has connected center, Kottwitz has shown \cite[Prop. 10.4]{Kot14} that the natural map $B(F,G)_\tx{bsc} \rw B(F,G_\tx{ad})_\tx{bsc} = H^1(\Gamma,G_\tx{ad})$ is surjective. This implies that, when $G$ has connected center, any inner twist $\xi : G \rw G_1$ can be made into an extended pure inner twist $(\xi,z) : G \rw G_1$.

Given two extended pure inner twists $(\xi_1,z_1) : G \rw G_1$ and $(\xi_2,z_2) : G_1 \rw G_2$, we define composition and inverse as $(\xi_2,z_2) \circ (\xi_1,z_1) = (\xi_2\circ\xi_2,\xi_1^{-1}(z_2)\cdot z_1)$ and $(\xi_1,z_1)^{-1}=(\xi_1^{-1},\xi_1(z_1)^{-1})$. The operation of composition allows us to talk about commutative diagrams of pure inner twists.

\subsubsection{Extended pure inner twists of general linear groups} \label{subsub:inner-GL}

Let us describe concretely the discussion of Section \ref{subsub:inner} for the group $\tx{GL}_N$. It is well-known that the inner forms of $\tx{GL}_N$ are the groups $\tx{Res}_{D/F}\tx{GL}_M$, where $D$ is any division algebra over $F$ of degree $s^2$, where $s|N$ and $M=N/s$.

Let first $F$ be local. We have the commutative diagram
\[ \xymatrix{
B(F,\tx{GL}_N)_\tx{bsc}\ar[d]\ar[r]&H^1(\Gamma,\tx{PGL}_N)\ar[d]\\
\Z\ar[r]&\Z/N\Z
} \]
where the left map is \eqref{eq:kotisoloc}, the right map is the classical Kottwitz map \cite[\S1]{Kot86}, and the horizontal maps are the natural maps.

When $F$ is a $p$-adic field, both vertical maps are bijections. Moreover, if $D$ is a division algebra of degree $s^2$ and invariant $r/s$, then the equivalence classes of extended inner twists $\tx{GL}_N \rw \tx{Res}_{D/F}\tx{GL}_M$ correspond to the elements $x \in \Z$ whose class in $\Z/N\Z$ is represented by $r\frac{N}{s}$.

When $F$ is the field of real numbers, both vertical maps are injective, but the description of their images depends on whether $N$ is even or odd. When $N$ is odd, the image of the left map is $N\Z$ and the image of the right map is $0$. When $N$ is even, the image of the left map is $\frac{N}{2}\Z$ and the image of the right map is $\frac{N}{2}\Z/N\Z$. Moreover, the non-trivial element $x \in \frac{N}{2}\Z/N\Z$ corresponds to the inner form $\tx{Res}_{\H/\R}\tx{GL}_\frac{N}{2}$, where $\H$ denotes the quaternion algebra.
When $F$ is the field of complex numbers, both vertical maps are injective, the image of the left map is $N\Z$ and the image of the right map is $0$.

Let now $F$ be global. Then the localization map \eqref{eq:h1algloc} is injective due to the vanishing of $\tx{ker}^1(\Gamma,\C^\times)$. Thus we can identify $B(F,\tx{GL}_N)_\tx{bsc}$ with the subset of $\bigoplus_v \Z$ consisting of tuples that sum to zero and whose entries for places $v|\infty$ belong to $\frac{N}{2}\Z$ when $v$ is real and $N$ is even and to $N\Z$ otherwise.

\subsubsection{Extended pure inner twists of unitary groups} \label{subsub:inner-U}

Let us now describe concretely the discussion of Section \ref{subsub:inner} in the case of the quasi-split group $U_{E/F}(N)$ defined in Section \ref{subsub:qsuni}.

We begin with $F$ being a local field. In that case, the injection $H^1(\Gamma,U_{E/F}(N)) \rw B(F,U_{E/F}(N))_\tx{bsc}$ is in fact a bijection, so it suffices to recall the classical computations of Galois cohomology. We have the diagram

\[ \xymatrix{
H^1(\Gamma,U_{E/F}(N))\ar[r]\ar[d]&H^1(\Gamma,U_{E/F}(N)_\tx{ad})\ar[d]\\
\Z/2\Z\ar[r]&\Z/\delta\Z
} \]
where $\delta=2$ if $N$ is even and $\delta=1$ if $N$ is odd. The top map is surjective. When $F$ is $p$-adic, the vertical maps are bijective. When $F$ is real, the vertical maps are surjective, but not injective. In fact, we have $H^1(\Gamma,U_{E/F}(N))=\{(p,q)|0 \leq p,q \leq N, p+q=N\}$ and $H^1(\Gamma,U_{E/F}(N)_\tx{ad})$ is the quotient of this set by the symmetry relation $(p,q) \sim (q,p)$. The left map is given by $(p,q) \mapsto \lfloor \frac{N}{2}\rfloor + q \textrm{ mod } 2$.

We now turn to $F$ global. The localization map \eqref{eq:h1algloc} is injective due to the vanishing of $\tx{ker}^1(F,\C^\times_{(-1)})$, where $\C^\times_{(-1)}$ denotes the abelian group $\C^\times$ on which $\Gamma_E$ acts trivially and each element of $\Gamma \setminus \Gamma_E$ acts by inversion. Thus we can identify $B(F,U_{E/F}(N))_\tx{bsc}$ with the kernel of \eqref{eq:kotisolocglo}. From the discussion of local linear and unitary groups, we can now extract the following: For each place $v$ of $F$, let $\Xi_v$ be an equivalence class of extended inner twists of $U_{E/F}(N)$, and let $a_v$ be its invariant, which for $v$ is finite and split in $E$ belongs to $\Z$, for $v$ is finite and inert in $E$ belongs to $\Z/2\Z$, for $v$ is real and split in $E$ belongs to $\frac{N}{2}\Z \subset \Z$ if $N$ is even and to $N\Z \subset \Z$ if $N$ is odd, and for $v$ real and inert belongs to $\Z/2\Z$. Then the collection $(\Xi_v)$ is the localization of the a global extended inner twist if and only if almost all $a_v$ are equal to zero and, after mapping each $a_v$ into $\Z/2\Z$ using the natural projection $\Z \rw \Z/2\Z$, the sum of all $a_v$ is zero in $\Z/2\Z$.

It is also instructive to describe which collections of local inner twists come from a global inner twist. If $v$ is a finite place split in $E$, such a twist gives the group $\tx{Res}_{D_v/F_v}\tx{GL}_{M_v}$ for some division algebra $D_v/F_v$, which provides an element $a_v := N\cdot \tx{inv}(D_v) \in \Z/N\Z \rw \Z/2\Z$. If $v$ is finite and inert in $E$ we obtain either the quasi-split group $U_{E_v/F_v}(N)$, which leads to $a_v = 0 \in \Z/2\Z$, or the unique non-quasi-split inner twist of that group (only available for even $N$), which leads to $a_v = 1 \in \Z/2\Z$. If $v$ is real and split in $E$, we obtain either the group $\tx{GL}_N$, which leads to $a_v = 0 \in \Z/2\Z$, or its unique non-quasi-split inner twist $\tx{Res}_{\H/\R}\tx{GL}_\frac{N}{2}$ (available only for even $N$), which leads to $a_v = \frac{N}{2} \in \Z/2\Z$. When $v$ is real and inert in $E$, we obtain the unitary group $U(p,q)$ for some $0\leq p,q \leq N$ with $p+q=N$, and it leads to $a_v = 0 \in \Z/2\Z$ if $N$ is odd or to $a_v= \frac{N}{2}+q \in \Z/2\Z$ if $N$ is even. In order for this collection of inner twists to come from a global inner twist it is necessary that $a_v=0$ for almost all $v$. This is also sufficient if $N$ is odd, and if $N$ is even a necessary and sufficient additional condition is that the sum of all $a_v$ is zero in $\Z/2\Z$.

Finally, let us note which inner forms of unitary groups can be realized as pure inner forms. This discussion works over any field $F$. For any separable quadratic extension of fields $E/F$ we have the exact sequence
\[ H^1(\Gamma,U_{E/F}(N)) \rw H^1(\Gamma,U_{E/F}(N)/U_{E/F}(1)) \rw H^1(\Gamma,\tx{Res}_{E/F}(\tx{PGL}(N))). \]
By the generalized Hilbert 90 theorem \cite[Theorem 29.2]{BookInvol}, we know that the connecting homomorphism
\[ H^1(\Gamma,\tx{Res}_{E/F}(\tx{PGL}(N))) \rw H^2(\Gamma,\tx{Res}_{E/F}(\mb{G}_m)) = H^2(\Gamma_E,\mb{G}_m)=\tx{Br}(E) \]
is injective. The composed homomorphism $H^1(\Gamma,U_{E/F}(N)/U_{E/F}(1)) \rw H^1(\Gamma,\tx{Res}_{E/F}(\tx{PGL}(N))) = \tx{Br}(E)$ sends a unitary group $G$ to the class of the division algebra over $E$ which is the base algebra of the Hermitian space defining $G$. This shows that the inner forms of unitary groups that can be realized as pure inner forms are precisely those unitary groups constructed from Hermitian spaces over the field $E$, rather than over a division algebra over $E$.

\subsection{Parabolic subgroups and Levi subgroups}

  Here we collect some general facts about parabolic and Levi subgroups of a reductive group and its $L$-group and give a classification of Levi subgroups for inner forms of unitary groups. We work with the Galois form of the $L$-group in this section for the obvious reason that the Weil group for a general field $F$ does not make sense. When $F$ is a local or global field, it is straightforward to adapt the results to the Weil form of the $L$-group.

\subsubsection{Relationship between $G$ and $^LG$} \label{subsub:levis}

Let $G$ be a connected reductive algebraic group, defined and quasi-split over a field $F$. Let $\mc{P}(G)$ denote the set of $G$-conjugacy classes of parabolic subgroups of $G$. The inclusion relation among parabolic subgroups imposes a partial ordering $\leq$ on $\mc{P}(G)$. There is a canonical order-preserving bijection between $\mc{P}(G)$ and the powerset of of $\Delta$. Fixing a particular Borel pair $(T,B)$ of $G$, each class in $\mc{P}$ has a unique element $P$ that contains $B$. Moreover, $P$ has a unique Levi factor that contains $T$.

Now let $\Gamma$ be a group acting on $G$. Then $\Gamma$ acts on $\mc{P}$ and on $\Delta$, and the bijection $\mc{P} \leftrightarrow \mf{P}(\Delta)$ is $\Gamma$-equivariant. For any class $[P] \in \mc{P}$, a necessary condition for this class to have a $\Gamma$-invariant element is that it be $\Gamma$-invariant itself. When $G$ possesses a $\Gamma$-invariant Borel subgroup $B$, then this condition is also sufficient, for then the unique element of $[P]$ containing $B$ is necessarily $\Gamma$-invariant. If $G$ does not contain a $\Gamma$-invariant Borel subgroup, this condition is not sufficient.

Let $\mc{M}$ be the set of $G$-conjugacy classes of Levi components of parabolic subgroups. Recalling that all Levi components of a given parabolic subgroup are conjugate under that parabolic subgroup, we see that we have a $\Gamma$-equivariant surjection $\mc{P} \rw \mc{M}$. Two elements of $\mc{P}$ lying in the same fiber of this map are called associate. By the same argument as above, the existence of a $\Gamma$-fixed Borel pair of $G$ ensures that each $\Gamma$-invariant class in $\mc{M}$ has a $\Gamma$-invariant element.

Now consider the semi-direct product $\ts{G} = G \rtimes \Gamma$. A subgroup $\ts{P}$ of this product will be called \textbf{full} if the restriction of the projection $G \rtimes \Gamma \rw \Gamma$ to $\ts{P}$ is surjective. It will be called a parabolic subgroup, if it is full and $\ts{G} \cap \ts{P}$ is a parabolic subgroup of $G$. The maps
\[ \ts{P} \mapsto \ts{P} \cap G \qquad \tx{and} \qquad P \mapsto N(P,\ts{G}) \]
are mutually inverse bijections between the set of parabolic subgroups of $\ts{G}$ and the set of those parabolic subgroups of $G$ whose conjugacy class is $\Gamma$-invariant.

Given a parabolic subgroup $\ts{P} \subset \ts{G}$, we will call a subgroup $\ts{M} \subset \ts{P}$ a Levi-factor of $\ts{P}$ if $\ts{M}$ is a full subgroup of $\ts{P}$ and moreover $\ts{M} \cap G$ is a Levi factor of $\ts{P} \cap G$. The maps
\[ \ts{M} \mapsto \ts{M} \cap G \qquad \tx{and} \qquad M \mapsto N(M,\ts{P}) \]
are mutually inverse bijections between the set of Levi factors of $\ts{P}$ and the set of Levi factors of $P$. All these bijections are evidently equivariant under conjugation by $\ts{G}$.

A Levi subgroup $\ts{M}$ of a parabolic subgroup of $\ts{G}$ will be called a Levi subgroup of $\ts{G}$ for short. If we put $A_\ts{M}=\tx{Cent}(\ts{M},G)^\circ$, then we have $\ts{M} = \tx{Cent}(A_\ts{M},\ts{G})$.

It is elementary to check the validity of the above statements. In doing so, the following elementary observation is useful: Two subgroups $\ts{A},\ts{B}$ of $\ts{G}$ are equal as soon as $\ts{A} \subset \ts{B}$, $\ts{A} \cap G = \ts{B} \cap G$, and the images of $\ts{A},\ts{B}$ in $\Gamma$ coincide.

\subsubsection{Transfer of Levis and relevance of parabolics}\label{subsub:transfer-Levi-parabolic}

We continue with $F$ a field of characteristic zero and $G^*$ a quasi-split group defined over $F$. Set ${^LG^*} = \hat G^* \rtimes \Gamma$, the Galois form of the $L$-group of $G^*$. The discussion in Section \ref{subsub:levis} provides a bijection between the conjugacy classes defined over $F$ of parabolic subgroups of $G ^*\times \ol{F}$ and the $^LG^*$-conjugacy classes of parabolic subgroups of $^LG^*$. Since $G^*$ is quasi-split, and $\hat G^*$ has a $\Gamma$-fixed pinning, every conjugacy class defined over $F$ of parabolic subgroups of $G^*$ has an element defined over $F$, and every ${^LG^*}$-conjugacy class of parabolic subgroups of ${^LG^*}$ contains an element whose intersection with $\hat G^*$ is $\Gamma$-invariant.

Let $\Xi : G^* \rw G$ be an equivalence class of inner twists. It provides  a $\Gamma$-equivariant bijection $\mc{P}(G^*) \rw \mc{P}(G)$, hence a bijection $\mc{P}(G^*)^\Gamma \rw \mc{P}(G)^\Gamma$. If $G$ is not quasi-split, then an element of $\mc{P}(G)^\Gamma$ may not have a $\Gamma$-invariant member. This necessitates the following definition.

\begin{dfn} An element of $\mc{P}(G)^\Gamma$ is called \textbf{relevant} if it has a $\Gamma$-invariant representative. We will write $\mc{P}(G)_\tx{rel}$ for the set of relevant elements of $\mc{P}(G)^\Gamma$.
\end{dfn}

The study of what elements of $\mc{P}(G)^\Gamma$ are relevant leads naturally to the question of transfer of Levi subgroups. Let $M^*$ be a Levi subgroup of $G^*$, and let $A_{M^*}$ be the maximal split torus in $Z(M^*)$. It is known that $M=\tx{Cent}(A_{M^*},G^*)$.

\begin{lem} \label{lem:levitran} The following are equivalent.
\begin{enumerate}
\item There exists $\xi \in \Xi$ whose restriction to $A_{M^*}$ is defined over $F$.
\item There exists a Levi subgroup $M \subset G$ and $\xi \in \Xi$ which restricts to an inner twist $\xi : M^* \rw M$.
\item The class of $H^1(F,G^*_\tx{ad})$ determined by $\Xi$ belongs to the image of the injection $H^1(F,M^*/Z(G^*)) \rw H^1(F,G^*_\tx{ad})$.
\end{enumerate}
\end{lem}
\begin{proof}
Take $\xi \in \Xi$ with $\xi|_{A_{M^*}}$ defined over $F$ and put $A_{M}=\xi(A_{M^*})$. Then $A_{M}$ is a split torus in $G$, and $M= \tx{Cent}(A_{M},G)$ is a Levi subgroup of $G$, and $A_{M}$ is the maximal split torus in $Z(M)$. Moreover, for any $g \in G^*$ with $\xi^{-1}\sigma(\xi)=\tx{Ad}(g)$, we have $g \in \tx{Cent}(A_{M^*},G^*) = M^*$. Thus $\xi : M^* \rw M$ is an inner twist. It follows right away that the cocycle $\sigma \mapsto \xi^{-1}\sigma(\xi)$ is inflated from $M^*/Z(G^*)$. Finally, if we assume that $\xi^{-1}\sigma(\xi)$ takes values in $M^*/Z(G^*)$, then we see right away that $\xi|_{A_{M^*}}$ is defined over $F$.
\end{proof}

\begin{dfn}\label{d:Levi-transfer} We say that $M^*$ \textbf{transfers} to $G$ (with respect to $\Xi$) if the equivalent conditions of the proposition hold. \end{dfn}

The following question will also be relevant: Given an equivalence class of inner twists $\Xi_M : M^* \rw M$, does there exist an equivalence class of inner twists $\Xi : G^* \rw G$ and an embedding of $M$ as a Levi-subgroup of $G$ such that $\Xi_M$ is the restriction of $\Xi$? For this we consider the following diagram
\begin{equation} \label{eq:itex} \xymatrix{
1\ar[r]&H^1(\Gamma,\frac{M^*}{Z(G^*)})\ar[r]&H^1(\Gamma,\frac{M^*}{Z(M^*)})\ar[r]\ar[dr]&H^2(\Gamma,\frac{Z(M^*)}{Z(G^*)})\\
&&&H^2(\Gamma,Z(M^*))\ar[u]\\
&&&H^2(\Gamma,Z(G^*))\ar[u]\\
&&&1\ar[u]
} \end{equation}
To justify the two ``1''s in the diagram, recall that an induced torus is an algebraic torus whose character module (equivalently co-character module) has a $\Gamma$-invariant basis. For such a torus $S$ Shapiro's lemma and Hilbert's theorem 90 imply $H^1(\Gamma,S)=0$. Furthermore, if $\hat S$ is the complex dual torus, then $\hat S^\Gamma$ is connected, because the $X^*(\hat S^\Gamma/\hat S^{\Gamma,\circ})=X_*(S)_{\Gamma,\tx{tor}}=0$.
\begin{lem} \label{lem:zmzgind} If $M^* \subset G^*$ is a Levi subgroup, then the algebraic group $Z(M^*)/Z(G^*)$ is an induced torus. If $\ts{M} \subset {^LG^*}$ is a Levi subgroup, then $Z(\ts{M} \cap \hat G^*)/Z(\hat G^*)$ is the dual of an induced torus.
\end{lem}
\begin{proof}
Let $T^* \subset G^*$ be a minimal Levi. Conjugating $M^*$ within $G^*$ we may assume that $T^* \subset M^*$. Then $Z(M^*)/Z(G^*)$ is the subgroup of $T^*/Z(G^*)$ given as the intersection of the kernels of a $\Gamma$-invariant set of simple roots (for some choice of a Borel subgroup $B^*$ containing $T^*$). The complement of that set in the set of all simple roots projects under $X^*(T^*/Z(G^*)) \rw X^*(Z(M^*)/Z(G^*))$ to a basis. This proves the first statement. The second is proved the same way, by taking $\hat T^*$ to be part of a $\Gamma$-invariant splitting of $\hat G^*$.
\end{proof}

With the above diagram at hand, we can now answer the question as follows.
\begin{cor} \label{cor:itex}
There exists an embedding of $M$ as a Levi subgroup of $G$ such that the equivalence class $\Xi_M : M^* \rw M$ of inner twists is the restriction of an equivalence class of inner twists $\Xi : G^* \rw G$ if and only if the class of $\Xi_M$ in $H^1(\Gamma,M^*/Z(M^*))$ is mapped via the connecting homomorphism into the subgroup $H^2(\Gamma,Z(G^*))$ of $H^2(\Gamma,Z(M^*))$.
\end{cor}

When $F$ is a number field, it is useful to know a local-global principle for transfers.

\begin{lem}\label{lem:transfer-Levi-locglo}
  Assume that $F$ is a number field. Then $M^*$ transfers to $G$ over $F$ if and only if $M^*$ transfers to $G$ over $F_v$ at every place $v$ of $F$.
\end{lem}

\begin{proof}
We reproduce here the proof that one of us learned from Robert Kottwitz while being a graduate student. We are thankful to Kottwitz for allowing us to reproduce this proof.

If $M^*$ transfers to $G$ over $F$, then it tautologically does so over all completions $F_v$. Assume conversely that $M^*$ transfers to $G$ over each $F_v$. Let $x \in H^1(\Gamma,G_\tx{ad}^*)$ be the class of $\Xi$. Our task is to show that $x$ is the image of $x_M \in H^1(\Gamma,M^*/Z(G^*))$ under the injection $H^1(\Gamma,M^*/Z(G^*)) \rw H^1(\Gamma,G_\tx{ad}^*)$. Our assumption implies that the restriction $x_v \in H^1(\Gamma_v,G_\tx{ad}^*)$ is the image of $x_{M,v} \in H^1(\Gamma_v,M^*/Z(G^*))$ under the injection $H^1(\Gamma_v,M^*/Z(G^*)) \rw H^1(\Gamma_v,G_\tx{ad}^*)$, for each place $v$.

We claim first that there exists $x'_M \in H^1(\Gamma,M^*/Z(G^*))$ whose restriction at each $v$ is equal to $x_{M,v}$. According to \cite[Prop. 2.6]{Kot86}, this is equivalent to showing that if $\chi_{M,v}$ is the character of $\pi_0(Z(\hat M^*_\tx{sc})^\Gamma)$ corresponding to $x_{M,v}$ under the map \cite[Theorem 1.2]{Kot86}, then the sum of all $\chi_{M,v}$ is the trivial character. But we know that this is true for the pull-backs $\chi_{M,v}|_{\pi_0(Z(\hat G^*_\tx{sc})^\Gamma)}$, because they correspond to the classes $(x_v)$ under the analogous map for the group $G^*_\tx{ad}$. We consider the exact sequence
\[ 1 \rw Z(\hat G^*_\tx{sc}) \rw Z(\hat M^*_\tx{sc}) \rw Z(\hat M^*_\tx{sc})/Z(\hat G^*_\tx{sc}) \rw 1. \]
According to Lemma \ref{lem:zmzgind}, the third term is an induced torus, so its $\Gamma$-fixed points are connected, which implies that the map
\[ \pi_0(Z(\hat G^*_\tx{sc})) \rw \pi_0(Z(\hat M^*_\tx{sc})) \]
is surjective. We conclude that the sum of the characters $\chi_{M,v}$ is indeed trivial and obtain the existence of a class $x'_M \in H^1(\Gamma,M^*/Z(G^*))$ as required.

Let $x' \in H^1(\Gamma,G^*_\tx{ad})$ be the image of $x'_M$. We have by construction $x'_v=x_v$ for all $v$, but this does not yet mean that $x'=x$. Let $G'$ be the twist of $G$ by $x'$ and let $M'$ be the twist of $M^*$ by $x_M$. By construction $M'$ is a Levi subgroup of $G'$ and $x'=1$ in $H^1(\Gamma,G')$. Thus $x \in \tx{ker}^1(\Gamma,G'_\tx{ad})$.

We claim that the natural map $\tx{ker}^1(\Gamma,M'/Z(G')) \rw \tx{ker}^1(\Gamma,G'_\tx{ad})$ is surjective (in fact it is bijective, but we will not need this). By \cite[(4.2.2)]{Kot84}, there is a natural duality between $\tx{ker}^1(\Gamma,H)$ and $\tx{ker}^1(\Gamma,Z(\hat H))$ for any reductive group $H$. Thus it is enough to show that $\tx{ker}^1(\Gamma,Z(\hat G^*_\tx{sc})) \rw \tx{ker}^1(\Gamma,Z(\hat M^*_\tx{sc}))$ is injective. This however follows from the surjectivity of $\pi_0(Z(\hat G^*_\tx{sc})) \rw \pi_0(Z(\hat M^*_\tx{sc}))$ by \cite[Cor. 2.3]{Kot84}.

The surjectivity of $\tx{ker}^1(\Gamma,M'/Z(G')) \rw \tx{ker}^1(\Gamma,G'_\tx{ad})$ allows us to choose $x_M \in \tx{ker}^1(\Gamma,M'/Z(G'))$ lifting $x$. Then $x_M$, viewed as an element of $H^1(\Gamma,M^*/Z(G^*))$ also lifts $x$ and the proof is complete.
\end{proof}

Having discussed the transfer of Levi subgroups, we return to the relevance of parabolic subgroups. Let $F$ be an arbitrary field of characteristic zero again. If a Levi subgroup $M^* \subset G^*$ transfers to $M \subset G$, then a $\xi$ as in part 2 of Lemma \ref{lem:levitran} provides a $\Gamma$-equivariant isomorphism $\mf{a}_{M^*} \rw \mf{a}_{M}$ and thus also a $\Gamma$-equivariant bijection between the set of parabolic subgroups of $G^*$ with Levi factor $M^*$ to the set of parabolic subgroups of $G$ with Levi factor $M$. This bijection depends only on $\Xi$ and not on $\xi$.

\begin{lem} \label{lem:tranrel} An element element $[P] \in \mc{P}(G^*)^\Gamma$ maps to a relevant element $[P] \in \mc{P}(G)^\Gamma$ if and only if one (hence every) Levi component of one (hence every) $\Gamma$-invariant element $P^*$ of $[P^*]$ transfers to $G$.
\end{lem}
\begin{proof}
Let $P^*$ be a parabolic subgroup of $G^*$ defined over $F$. It has a $\Gamma$-invariant Levi component $M^*$ \cite[\S20.5]{Brl91}. If $M^*$ transfers to $G$, then fixing an element $\xi \in \Xi$ as above and looking at the map $\mf{a}_{M^*} \rw \mf{a}_{M}$ it provides, we see that $P=\xi(P^*)$ is $\Gamma$-invariant. Thus $\Xi([P^*])$ is relevant.

Conversely suppose that $\Xi([P^*])$ is relevant, and let $\xi \in \Xi$ be such that $P=\xi(P^*)$ is $\Gamma$-invariant. Then $\xi^{-1}\sigma(\xi) \in N(P^*,G^*_\tx{ad}) = P^*/Z(G^*)$. Fix Levi factors $M^*$ and $M$ of $P^*$ and $P$ which are $\Gamma$-invariant. Modifying $\xi$ by an element of $P^*$ we achieve that $\xi(M^*)=M$. But then $\xi^{-1}\sigma(\xi) \in N(M^*,P^*/Z(G^*))=M^*/Z(G^*)$ and we see that $\xi : M^* \rw M$ is an inner twist. In particular $\xi|_{A_{M^*}}$ is defined over $F$.
\end{proof}

\begin{cor} \label{cor:relinc} If $[P^*] \in \mc{P}(G)_\tx{rel}$ and $[P] \in \mc{P}(G)^\Gamma$ with $[P^*] \leq [P]$, then $[P] \in \mc{P}(G)_\tx{rel}$.
\end{cor}
\begin{proof} We may choose $P^*$ and $P$ within the corresponding classes as well as Levi factors $M^*$ and $M$ in such a way that $M^* \subset M$. Then $A_{M} \subset A_{M^*}$ and the statement follows from the preceding lemma.
\end{proof}

From now on we assume that $F$ is local. In that setting, Kottwitz's map
\[ H^1(\Gamma,G^*_\tx{ad}) \rw X^*(Z(\hat G^*_\tx{sc})^\Gamma) \]
assigns to $\Xi$ a character $\zeta_\Xi : Z(\hat G^*_\tx{sc})^\Gamma \rw \C^\times$. A ${^LG^*}$-conjugacy class of parabolic subgroups will be called \textbf{$\Xi$-relevant} (or \emph{$(G,\xi)$-relevant} or \emph{$\xi$-relevant}, where $\xi$ is any element of $\Xi$), if the corresponding element of $\mc{P}(G)^\Gamma$ is relevant. Given a ${^LG^*}$-conjugacy class of Levi subgroups of ${^LG^*}$, let $\ts{M}$ be an element of it. We will write $\hat M^*$ for $\ts{M} \cap \hat G^*$ and $\hat M^*_\tx{sc}$ for the inverse image of $\hat M^*$ in $\hat G^*_\tx{sc}$. Note that this is different from the simply-connected cover of $\hat M^*$. The latter is equal to $\hat M^*_{\tx{sc},\tx{der}}$, the derived group of $\hat M^*_\tx{sc}$, which is easily seen to be simply-connected. Consider
\[ Z(\hat G^*_\tx{sc})^\Gamma \cap Z(\hat M^*_\tx{sc})^{\Gamma,\circ}. \]
This is a subgroup of $Z(\hat G^*_\tx{sc})^\Gamma$ which is evidently independent of the choice of $\ts{M}$ within its ${^LG^*}$-conjugacy class.

\begin{lem}\label{lem:triv-on-ZGcapZM} If the parabolic subgroup $\ts{P}$ of ${^LG^*}$ is $\Xi$-relevant, then for one (hence any) Levi component $\ts{M}$ of $\ts{P}$, the character $\zeta_\Xi$ is trivial on $Z(\hat G^*_\tx{sc})^\Gamma \cap Z(\hat M^*_\tx{sc})^{\Gamma,\circ}$. Conversely, if the character $\zeta_\Xi$ is trivial on $Z(\hat G^*_\tx{sc})^\Gamma \cap Z(\hat M^*_\tx{sc})^{\Gamma,\circ}$, then any parabolic subgroup $\ts{P}$ of ${^LG^*}$ with Levi factor $\ts{M}$ is $\Xi$-relevant if $F$ is $p$-adic, and is $\Xi'$-relevant for some $\Xi'$ belonging to the same $K$-group as $\Xi$, if $F$ is real.
\end{lem}
\begin{proof}
Fix a Levi component $\ts{M}$ of $\ts{P}$, and a parabolic subgroup $P^*$ of $G^*$ corresponding to $\ts{P}$. Then $\hat M^* = \ts{M}\cap \hat G^*$ is a Levi of $\hat G^*$ and is the dual group of a Levi component $M^*$ of $P^*$. According to Lemma \ref{lem:tranrel}, $P^*$ is relevant if and only if $M^*$ transfers to $G$, which by Lemma \ref{lem:levitran} is equivalent to $\Xi$ belonging to the image of $H^1(F,M^*/Z(G^*)) \rw H^1(F,G^*_\tx{ad})$. Then $\hat M^*_\tx{sc}$ is a dual group for $M^*/Z(G^*)$. The map $H^1(F,M^*/Z(G^*)) \rw H^1(F,G^*_\tx{ad})$ translates under the Kottwitz homomorphism \cite[\S1.2]{Kot86} to the map
\[ X^*(\pi_0(Z(\hat M^*_\tx{sc})^\Gamma)) \rw X^*(Z(\hat G^*_\tx{sc})^\Gamma). \]
According to \cite[Lemma 1.1]{Art99}, we see that $\zeta_\Xi$ belongs to the image of the last map if and only if it satisfies the assumption of the Lemma.
\end{proof}

\begin{lem} \label{lem:cart} The square
\[ \xymatrix{
Z(\hat G^*_\tx{sc})^\Gamma\ar[r]\ar[d]&Z(\hat G^*)^\Gamma\ar[d]\\
Z(\hat M^*_\tx{sc})^\Gamma\ar[r]&Z(\hat M^*)^\Gamma
} \]
is both cartesian and cocartesian.
\end{lem}
\begin{proof}
The fact that the square is cartesian is easily checked. To show that it is cocartesian, assume first that the $\Gamma$-action is trivial. The group $\hat M^*_\tx{sc}$ is a Levi subgroup of $\hat G^*_\tx{sc}$.
Moreover, we have $Z(\hat M^*)/Z(\hat G^*)=Z(\hat M^*/Z(\hat G^*))$
as well as the analogous statement with $\hat M^*$ replaced by $\hat M^*_\tx{sc}$. Since $\hat M^*/Z(\hat G^*) = \hat M^*_\tx{sc}/Z(\hat G^*_\tx{sc})$, we see that the square is indeed cocartesian assuming $\Gamma$ acts trivially. Putting together what we have proved, we see that the sequence
\[ 1 \rw Z(\hat G^*_\tx{sc}) \rw Z(\hat G^*) \times Z(\hat M^*_\tx{sc}) \rw Z(\hat M^*) \rw 1, \]
where the first map is given by $z \mapsto (z^{-1},z)$, is exact. We now want to show that
\[ Z(\hat G^*)^\Gamma \times Z(\hat M^*_\tx{sc})^\Gamma \rw Z(\hat M^*)^\Gamma \]
is surjective, and for this it is enough to show that
\[ H^1(\Gamma,Z(\hat G^*_\tx{sc})) \rw H^1(\Gamma,Z(\hat M^*_\tx{sc})) \]
is injective. This follows from Lemma \ref{lem:zmzgind} and the remark preceding it, which imply that $[Z(\hat M^*_\tx{sc})/Z(\hat G^*_\tx{sc})]^\Gamma$ is connected and consequently that the connecting homomorphism
\[ (Z(\hat M^*_\tx{sc})/Z(\hat G^*_\tx{sc}))^\Gamma \rw H^1(\Gamma,Z(\hat G^*_\tx{sc})), \]
which has to factors though $\pi_0$, is trivial.
\end{proof}

We close this discussion on the relationship between extended inner twists of $G^*$ and those of $M^*$. Note first that in general there is no map $B(F,M^*)_\tx{bsc} \rw B(F,G^*)_\tx{bsc}$. There is however a natural map $B(F,M^*)_{G^*\tx{-bsc}} \rw B(F,G^*)_\tx{bsc}$.

\begin{lem} \label{lem:gbasl}
\begin{enumerate}
\item The map $B(F,M^*)_{G^*-\tx{bsc}} \rw B(F,G^*)_\tx{bsc}$ is injective and corresponds, via the natural transformations $\kappa_{M^*}$ and $\kappa_{G^*}$ of \cite[\S11]{Kot14} to the natural inclusion $Z(\hat G^*) \rw Z(\hat M^*)$.
\item Let $\chi \in X^*(Z(\hat M^*)^\Gamma)$ be the image under $\kappa_{M^*}$ of an element of $x \in B(F,M^*)_\tx{bsc}$. Then $x \in B(F,M^*)_{G^*-\tx{bsc}}$ if and only if for some finite Galois extension $L/F$ containing splitting $G^*$, $N_{L/F}(\chi)$ annihilates $Z(\hat M^*) \cap \hat G^*_\tx{der}$.
\end{enumerate}
\end{lem}

\begin{proof}
The injectivity of $B(F,M^*)_{G^*-\tx{bsc}} \rw B(F,G^*)_\tx{bsc}$ follows directly from the injectivity of $H^1(\Gamma,M^*) \rw H^1(\Gamma,G^*)$ and the injectivity of $\tx{Hom}_F(\mb{D}_F,Z(G^*)) \rw \tx{Hom}_F(\mb{D}_F,Z(M^*))$. The compatibility of this map with $Z(\hat G^*) \rw Z(\hat M^*)$ follows from the fact that the latter is dual to the map $\Lambda_{M^*} \rw \Lambda_{G^*}$ of algebraic fundamental groups coming from $M^* \subset G^*$.

For the second claim, we use Diagram (1.6) in \cite{Kot14} for the Galois extension $L/F$. Recalling that $X(L)=\Z$ has trivial $\Gamma$-action and is identified with $[L:F]^{-1}\Z \subset X^*(\mb{D}_F)$, we see that the Newton point of $x$ belongs to $X^*(\hat M^*/\hat M^*_\tx{der})^\Gamma \otimes [L:F]^{-1}\Z$ and is the unique element $y \otimes [L:F]^{-1}$ thereof such that the restriction of $y \in X^*(\hat M^*/\hat M^*_\tx{der})^\Gamma$ to $Z(\hat M^*)$ is equal to $N_{L/F}(\chi)$.
We conclude that the Newton point $y \otimes [L:F]^{-1}$ belongs to $X^*(\hat G^*/\hat G^*_\tx{der})^\Gamma \otimes [L:F]^{-1}\Z$ if and only if $N_{L/F}(\chi)$ annihilates $Z(\hat M^*) \cap \hat G^*_\tx{der}$.
\end{proof}

\subsubsection{The group $W(M)$} \label{subsub:wm}

As before, $F$ is a field of characteristic zero and $G^*$ is a quasi-split group defined over $F$. Let $\Xi : G^* \rw G$ be an equivalence class of inner twists. For a Levi subgroup $M \subset G$, we have the relative Weyl-group $W(M)=N(M,G)/M$. It is a finite group scheme defined over $F$. It is well-known that the map $H^1(\Gamma,M) \rw H^1(\Gamma,G)$ is injective and this implies that $N(M,G)(F)/M(F) \rw W(M)(F)$ is an isomorphism.

Consider now the dual side. Let $\ts{M} \subset {^LG^*}$ be a Levi-subgroup of $^LG^*$. Then we have the finite group
\[ N(\ts{M},\hat G^*)/(\hat G^* \cap \ts{M}) = N(A_\ts{M},\hat G^*)/\tx{Cent}(A_\ts{M},\hat G^*). \]
Let $\ts{P} \subset {^LG^*}$ be a parabolic subgroup with Levi-component $\ts{M}$, and assume that the conjugacy classes of $\ts{P}$ and $P$ match. Then we claim that the data of $(\ts{M},\ts{P})$, $(M,P)$, and $\Xi$, provides an isomorphism
\[  N(\ts{M},\hat G^*)/(\hat G^* \cap \ts{M}) \rw W(M)(F). \]

To see this, fix $\Gamma$-invariant splittings of $G^*$ and $\hat G^*$. There exists a unique standard parabolic pair $(M^*,P^*)$ of $G^*$ for which there exists $\xi \in \Xi$ that caries it over to $(M,P)$. Then this $\xi$ is unique up to composition by an inner automorphism of $M^*$ and is an inner twist $M^* \rw M$. Thus we obtain a canonical isomorphism $W(M^*) \cong W(M)$ defined over $F$. There also exists a unique standard parabolic pair $(\hat M^*,\hat P^*)$ of $\hat G^*$ such that $(\hat M^*\rtimes \Gamma,\hat P^*\rtimes\Gamma)$ is $\hat G^*$-conjugate to $(\ts{M},\ts{P})$. An element of $\hat G^*$ that conjugates $(\hat M^*\rtimes\Gamma,\hat P^*\rtimes\Gamma)$ to $(\ts{M},\ts{P})$ is unique up to multiplication by an element of $\hat M^*$ and so we have a canonical isomorphism
\[ N(\ts{M},\hat G^*)/(\hat G^* \cap \ts{M}) \rw N(\hat M^* \rtimes \Gamma,\hat G^*)/\hat M^*. \]
Next we claim that the inclusion
\[ N(\hat M^*,\hat G^*)^\Gamma/(\hat M^*)^\Gamma \rw N(\hat M^*\rtimes\Gamma,\hat G^*)/\hat M^* \]
is an isomorphism. It is clearly injective. For surjectivity, take $g \in N(\hat M^* \rtimes\Gamma,\hat G^*) \subset N(\hat M^*,\hat G^*)$. Then
\[ g\hat M^*\sigma(g^{-1}) = \hat M^*\Leftrightarrow g\sigma(g^{-1}) \in \hat M^*\]
together with the injectivity of $H^1(\Gamma,\hat M^*) \rw H^1(\Gamma,\hat G^*)$ imply $g \in N(\hat M^*,\hat G^*)^\Gamma \cdot \hat M^*$ as claimed.
Furthermore, we have the isomorphism
\[ N(\hat M^*,\hat G^*)^\Gamma/\hat M^{*,\Gamma} \rw [N(\hat M^*,\hat G^*)/\hat M^*]^\Gamma. \]
It remains to show that we have a canonical $\Gamma$-equivariant isomorphism
\[ W(M^*) \cong N(\hat M^*,\hat G^*)/\hat M^*. \]
This follows from the fact that $W(M^*)$ can be realized as a subquotient of the absolute Weyl group $W(T,G)$. If we let $W(T^*,M^*,G^*)$ be the  stabilizer in $W(T^*,G^*)$ of the root system of $M^*$, then we have
\[ W(T^*,M^*) \subset W(T^*,M^*,G^*) \subset W(T^*,G^*) \]
and $W(M^*)=W(T^*,M^*,G^*)/W(T^*,M^*)$.
The $\Gamma$-equivariant isomorphism $W(T^*,G^*) \cong W(\hat T^*,\hat G^*)$ restricts to isomorphisms $W(T^*,M^*,G^*) \cong W(\hat T^*,\hat M^*,\hat G^*)$ and $W(T^*,M^*) \cong W(\hat T^*,\hat M^*)$. Composing the isomorphisms described above we obtain the isomorphism
\[  N(\ts{M},\hat G^*)/(\hat G^* \cap \ts{M}) \rw W(M)(F). \]
It is easy to see that this isomorphism does not depend on the chosen splittings of $G^*$ and $\hat G^*$.

\subsubsection{Centralizers of parameters and Levi subgroups} \label{subsub:centlevi}

For this subsection, the field $F$ may be \emph{arbitrary}. Keep assuming that $G^*$ is quasi-split over $F$. Let $L$ be an abstract group equipped with a surjective map to $\Gamma$, and let $\phi : L \rw {^LG^*}$ be an L-homomorphism, i.e. a group homomorphism commuting with the projections of $L$ and $^L G^*$ onto $\Gamma$. (Later we will be mainly concerned with the case where $F$ is a local or global field and $\phi$ is an $L$-parameter or $A$-parameter. In that case $\phi$ will be a map over the Weil group $W_F$ but it is easy to adapt our discussion below to that setup.) We introduce the following notation
\[ S_\phi = \tx{Cent}(\phi(L),\hat G^*),\qquad S_\phi^\tx{rad} = \tx{Cent}(\phi(L),\hat G^*_\tx{der})^\circ,\qquad S_\phi^\natural = S_\phi/S_\phi^\tx{rad}. \]
We will assume that $S_\phi^\circ$ is reductive. This is automatically true if $L$ is an algebraic group, or a locally compact group, and $\phi$ respects this structure. Then one observes that
\[ [S_\phi^\circ]_\tx{der} \subset S_\phi^\tx{rad} \]
as the left hand side is a connected subgroup of $S_\phi$ generated by commutators. In particular, $S_\phi^\circ/S_\phi^\tx{rad}$ is a torus, and the map $Z(S_\phi^\circ)^\circ \rw S_\phi^\circ/S_\phi^\tx{rad}$ is surjective.

It is clear that $Z(\hat G^*)^\Gamma$ is a central subgroup of $S_\phi$. Restricting the central character of irreducible representations thus provides a map $\tx{cc} : \tx{Irr}(S_\phi) \rw X^*(Z(\hat G^*)^\Gamma)$.

\begin{lem} \label{lem:sphilevi} Consider the maps
\begin{eqnarray*}
\left\{\begin{array}{c}\textrm{Levi subgroups of }^L G^*\\ \textrm{containing }\phi(L)\end{array}\right\}&\rw&\left\{\begin{array}{c}\textrm{Levi subgroups}\\ \textrm{of }S_\phi^\circ\end{array}\right\}\\
\tx{cntr} : \ts{M}&\mapsto&S_\phi^\circ \cap \ts{M}\\
\tx{infl} : \tx{Cent}(Z(M')^\circ,{^LG^*})&\mapsfrom&M'
\end{eqnarray*}
These maps are inclusion preserving, and satisfy
\[ \tx{cntr}(\tx{infl}(M'))=M'\qquad\tx{and}\qquad\tx{infl}(\tx{cntr}(\ts{M})) \subset \ts{M}. \]
In particular, $\tx{cntr}$ is surjective and $\tx{infl}$ is injective.
\end{lem}
\begin{proof}
Given $\ts{M}$, put $A_\ts{M} = \tx{Cent}(\ts{M},\hat G^*)^\circ$. Then $\ts{M}=\tx{Cent}(A_\ts{M},{^LG^*})$, and
\[ S_\phi^\circ \cap \ts{M} = \tx{Cent}(A_\ts{M},S_\phi^\circ). \]
By assumption, $\phi(L) \subset \ts{M}$, thus $A_\ts{M} \subset S_\phi^\circ$, and it follows that $\tx{Cent}(A_\ts{M},S_\phi^\circ)$ is a Levi subgroup of $S_\phi^\circ$.

Conversely, let $M'$ be a Levi subgroup of $S_\phi^\circ$. Then $Z(M')^\circ$ is a torus in $\hat G^*$, hence $\tx{Cent}(Z(M')^\circ,\hat G^*)$ is a Levi subgroup of $\hat G^*$. Moreover, $\ts{M}:=\tx{Cent}(Z(M')^\circ,{^LG^*})$ contains $\phi(L)$ by definition, and the surjectivity of $L \rw \Gamma$ ensures that $\ts{M}$ is full.

From $S_\phi^\circ \cap \tx{Cent}(Z(M')^\circ,{^LG^*}) = \tx{Cent}(Z(M')^\circ,S_\phi^\circ) = M'$ we conclude that $\tx{cntr}(\tx{infl}(M'))=M'$. Conversely, $A_\ts{M} \subset Z(S_\phi^\circ\cap\ts{M})^\circ$, thus $\tx{infl}(\tx{cntr}(\ts{M}))\subset \ts{M}$.
\end{proof}

In particular, we see that the minimal Levi subgroups of $^LG^*$ containing $\phi(L)$ are precisely the centralizers in $^LG^*$ of the maximal tori of $S_\phi^\circ$ and are thus conjugate under $S_\phi^\circ$.

\begin{lem} \label{lem:squot} We have
\[ S_\phi/S_\phi^\tx{rad}Z(\hat G^*)^\Gamma = S_\phi/S_\phi^\circ Z(\hat G^*)^\Gamma. \]
\end{lem}
\begin{proof}
We need to show that $S_\phi^\circ Z(\hat G^*)^\Gamma = S_\phi^\tx{rad}Z(\hat G^*)^\Gamma$, and this is equivalent to the surjectivity of
\[ Z(\hat G^*)^\Gamma \rw \frac{S_\phi^\circ Z(\hat G^*)^\Gamma}{S_\phi^\tx{rad}}. \]
Let $\ts{M}$ be a minimal Levi subgroup of $^LG^*$ containing $\phi(L)$. Then $A_\ts{M}$ is a maximal torus of $S_\phi^\circ$, and since $S_\phi^\tx{rad}$ contains the derived group of $S_\phi^\circ$, we see that
\[ A_\ts{M} \rw \frac{S_\phi^\circ}{S_\phi^\tx{rad}} \]
is surjective. Notice that $A_\ts{M}=Z(\hat M^*)^{\Gamma,\circ}$. (Recall that $\hat M^*:=\ts{M}\cap \hat G^*$.) Applying Lemma \ref{lem:cart}, we see that
$A_\ts{M} \subset Z(\hat M^*_\tx{sc})^\Gamma \cdot Z(\hat G^*)^\Gamma,$
and thus
\[ Z(\hat M^*_\tx{sc})^\Gamma \cdot Z(\hat G^*)^\Gamma \rw \frac{S_\phi^\circ Z(\hat G^*)^\Gamma}{S_\phi^\tx{rad}} \]
is surjective. Applying \cite[Lemma 1.1]{Art99} to $\hat M^*_\tx{sc}$, we see
\[ Z(\hat M^*_\tx{sc})^\Gamma \cdot Z(\hat G^*)^\Gamma = Z(\hat M^*_\tx{sc})^{\Gamma,\circ} Z(\hat G^*)^\Gamma. \]
The image of $Z(\hat M^*_\tx{sc})^{\Gamma,\circ}$ in $\hat G^*$ belongs to both $S_\phi$ and $\hat G^*_\tx{der}$. Being connected, it thus belongs to $S_\phi^\tx{rad}$, and this proves the surjectivity of $Z(\hat G^*)^\Gamma \rw S_\phi^\circ Z(\hat G^*)^\Gamma/S_\phi^\tx{rad}$.
\end{proof}

\begin{dfn}\label{def:Xi-relevant-general} Let $\Xi : G^* \rw G$ be an inner twist and $\zeta_\Xi \in X^*(Z(\hat G^*_\tx{sc})^\Gamma)$ the corresponding character. We will say that $\phi : L \rw {^LG^*}$ is \textbf{$\Xi$-relevant}, if any Levi subgroup of $^LG^*$ containing $\phi(L)$ is the Levi-component of a $\Xi$-relevant parabolic subgroup of $^LG^*$.
\end{dfn}
According to Corollary \ref{cor:relinc} and the above remark, it is enough to check this for one smallest Levi subgroup of $^LG^*$ containing $\phi(L)$.

Now let $\ts{M}$ be any Levi subgroup of $^LG^*$ containing $\phi(L)$. Given $\zeta \in X^*(Z(\hat G^*)^\Gamma)$ mapping to $\zeta_\Xi$, we are going to construct a canonical element $\zeta_M \in X^*(Z(\hat M^*)^\Gamma)$ which maps to $\zeta$. For this, consider the diagram
\[ \xymatrix{
X^*(Z(\hat M^*)^\Gamma)\ar[d]\ar[r]&X^*(Z(\hat G^*)^\Gamma)\ar[d]\\
X^*(Z(\hat M^*_\tx{sc})^\Gamma)\ar[r]&X^*(Z(\hat G^*_\tx{sc})^\Gamma)\\
X^*(\pi_0(Z(\hat M^*_\tx{sc})^\Gamma))\ar@{^{(}->}[u]\ar@{^{(}->}[ur]
} \]
By assumption, $\zeta_{G^*}$ belongs to the image of the diagonal injection. Thus there exists a canonical element of $X^*(Z(\hat M^*_\tx{sc})^\Gamma)$ mapping to $\zeta_\Xi$. Applying Lemma \ref{lem:cart}, we see that the top square is cartesian, hence there exists a unique element $\zeta_{M^*} \in X^*(Z(\hat M^*)^\Gamma)$ which maps to $\zeta$ in $X^*(Z(\hat G^*)^\Gamma)$ and the the canonical lift of $\zeta_\Xi$ in $X^*(Z(\hat M^*_\tx{sc})^\Gamma)$.

It is clear form the construction that if $\ts{M} \subset \ts{M'}$ are two Levi subgroups of $^LG^*$ containing $\phi(L)$, then the image of $\zeta_{M^*}$ under $X^*(Z(\hat M^*)^\Gamma) \rw X^*(Z(\hat M^*)^\Gamma)$ is equal to $\zeta_{M}$.

\begin{lem} \label{lem:rhofact} Let $\ts{M}$ be a Levi subgroup of $^LG^*$ containing $\phi(L)$, and let $S_\phi(\ts{M}) = \tx{Cent}(\phi(L),\hat M^*)$. If $\rho \in \tx{Irr}(S_\phi(M^*)^\natural)$ satisfies $\tx{cc}(\rho)=\zeta_{M^*}$, then $\rho$ belongs to
\[ \tx{Irr}(S_\phi(\ts{M})/(S_\phi(\ts{M}) \cap S_\phi^\tx{rad})). \]
\end{lem}
\begin{proof}
We know that $S_\phi(\ts{M})=S_\phi \cap \ts{M}$ and that $S_\phi^\circ \cap \ts{M}$ is a Levi subgroup of $S_\phi^\circ$, in particular connected. Then $S_\phi(\ts{M}) \cap S_\phi^\tx{rad}$ is also connected, as it is a Levi subgroup of $S_\phi^\tx{rad}$. Then $S_\phi(\ts{M}) \cap S_\phi^\tx{rad} = S_\phi(\ts{M})^\circ \cap S_\phi^\tx{rad}$. We know from Lemma \ref{lem:squot} that $S_\phi(\ts{M})^\circ Z(\hat M^*)^\Gamma=S_\phi(\ts{M})^\tx{rad} Z(\hat M^*)^\Gamma$. Thus we get
\[ S_\phi(\ts{M}) \cap S_\phi^\tx{rad} \subset (S_\phi(\ts{M})^\tx{rad} Z(\hat M^*)^\Gamma) \cap S_\phi^\tx{rad} = S_\phi(\ts{M})^\tx{rad} (Z(\hat M^*)^\Gamma \cap S_\phi^\tx{rad}). \]
We already know that $\rho$ is trivial on $S_\phi(\ts{M})^\tx{rad}$, and is $\zeta_{M^*}$-isotypic on $Z(\hat M^*)^\Gamma$. Now,
\[ Z(\hat M^*)^\Gamma = Z(\hat M^*_\tx{sc})^\Gamma Z(\hat G^*)^\Gamma = Z(\hat M^*_\tx{sc})^{\Gamma,\circ} Z(\hat G^*)^\Gamma \]
by Lemma \ref{lem:cart}, as well as \cite[Lemma 1.1]{Art99} applied to $\hat M^*_\tx{sc}$. Since $Z(\hat M^*_\tx{sc})^{\Gamma,\circ} \subset S_\phi^\tx{rad}$, we see that
\[ (Z(\hat M^*_\tx{sc})^{\Gamma,\circ} Z(\hat G^*)^\Gamma) \cap S_\phi^\tx{rad} = Z(\hat M^*_\tx{sc})^{\Gamma,\circ} (Z(\hat G^*)^\Gamma \cap S_\phi^\tx{rad}). \]
By construction, $\zeta_{M^*}$ is trivial on $Z(\hat M^*_\tx{sc})^{\Gamma,\circ}$ and restricts to $\zeta$ on $Z(\hat G^*)^\Gamma$. Since we are assuming that $\tx{Irr}(S_\phi^\natural,\zeta) \neq \emptyset$, we know that $\zeta$ restricts trivially to $Z(\hat G^*)^\Gamma \cap S_\phi^\tx{rad}$.
\end{proof}

\subsubsection{Levi subgroups of linear groups}\label{subsub:Levi-GL}

Let $F$ be a local or a global field, and $G^*=\tx{GL}_N/F$. A Levi subgroup $M^* \subset G^*$ has the form
\[ M^* = \prod_{i=1}^k \tx{GL}_{N_i},\qquad \sum_i N_i = N. \]
An equivalence class $\Xi_M : M^* \rw M$ of inner twists breaks up into a product $(\Xi_i)$ with
\[ \Xi_i : \tx{GL}_{N_i} \rw \tx{Res}_{D_i/F} \tx{GL}_{M_i} \]
for some division algebras $D_i/F$ satisfying $[D_i:F]M_i^2=N_i^2$. According to Corollary \ref{cor:itex}, in order for $\Xi_M$ to come from an equivalence class of inner twists $\Xi : G^* \rw G$ with $M \subset G$ a Levi subgroup, it is necessary and sufficient that the division algebras $D_i$ are all the same.

\subsubsection{Levi subgroups of unitary groups}\label{subsub:Levi-U}

Let $F$ be a local or a global field, $E/F$ a quadratic extension, and $G^*=\tx{U}_{E/F}(N)$. A Levi subgroup $M^* \subset G^*$ has the form
\[ M^* = U_{E/F}(N_\frac{k+1}{2}) \times \prod_{i=1}^{\lfloor\frac{k}{2}\rfloor} \tx{Res}_{E/F} \tx{GL}_{N_i}, \qquad N_\frac{k+1}{2}+2\sum_{i=1}^{\lfloor\frac{k}{2}\rfloor} N_i=N. \]
Here the terms with index $\frac{k+1}{2}$ are understood to appear only if $k$ is odd. %

An equivalence class of inner twists $\Xi_M : M^* \rw M$ again breaks up into a product $(\Xi_i)$ with
\[ \Xi_\frac{k+1}{2} : U_{E/F}(N_\frac{k+1}{2}) \rw U,\qquad \Xi_i : \tx{Res}_{E/F}\tx{GL}_{N_i} \rw \tx{Res}_{D_i/F}\tx{GL}_{M_i}, \]
where $U$ is a not necessarily quasi-split unitary group, and $D_i$ is a division algebra over $E$ satisfying $[D_i:E]M_i^2=N_i^2$. We apply again Corollary \ref{cor:itex} to see when $\Xi_M$ comes from an equivalence class of inner twists $\Xi : G^* \rw G$ with $M \subset G$ a Levi subgroup. The first condition is that the classes of all $D_i$ in $H^2(\Gamma_E,\mb{G}_m)$ coincide, thus all $D_i$ are equal to a fixed division algebra $D/E$. The second condition is that the class of $D$ belongs to the image of
\begin{equation} \label{eq:h2ug} H^2(\Gamma,U_{E/F}(1)) \rw H^2(\Gamma,\tx{Res}_{E/F}\mb{G}_m) \rw H^2(\Gamma_E,\mb{G}_m), \end{equation}
where the first map is induced by the inclusion $U_{E/F}(1) \rw \tx{Res}_{E/F}\mb{G}_m$ and the second map is the Shapiro isomorphism. The third and final condition is that, in case $k$ is odd, the composition of the above map with the connecting homomorphism $H^1(\Gamma,U_{E/F}(N)_\tx{ad}) \rw H^2(\Gamma,U_{E/F}(1))$ sends $\Xi_\frac{k+1}{2}$ to the class of $D$.

If $F$ is local, then $H^2(\Gamma,U_{E/F}(1))=0$ and hence $D=E$. If $F$ is global then the map \eqref{eq:h2ug} is injective and its image corresponds to those division algebras $D$ which are split at all places of $E$ fixed under $\Gamma_{E/F}$, and which satisfy $\tx{inv}(D_w)=-\tx{inv}(D_{w'})$ for all pairs $(w,w')$ of places of $E$ which are conjugate under $\Gamma_{E/F}$.

\subsubsection{The group $W(M^*,G^*)$ for unitary groups and lifts of its elements} \label{subsub:wmgu}

Let $F$ be a local or global field of characteristic zero, $E/F$ a quadratic algebra, and $G^*=U_{E/F}(N)$. Thus $G^*$ is the quasi-split unitary group in $N$ variables when $E/F$ is a field extension, or $G^*$ is isomorphic to the general linear group in $N$ variables over $F$ when $E/F$ is the split quadratic algebra. Recall that $G^*$ has a distinguished pinning $(T^*,B^*,\{X_\alpha\}_{\alpha \in \Delta(T^*,G^*)})$.

Let $(M^*,P^*)$ be a standard parabolic pair -- $P^*$ is a parabolic subgroup of $G^*$ containing $B^*$, $M^*$ is a Levi subgroup of $P^*$ containing $T^*$, and both $M^*$ and $P^*$ are defined over $F$. We have the decomposition $M^*=M^*_+ \times M^*_-$, $M^*_+=M^*_1 \times \dots \times M^*_k$, where $M^*_i = \tx{Res}_{E/F}(\tx{GL}(N_i))$, $M^*_-=U_{E/F}(N_-)$ and $2\sum N_i+N_-=N$. The pinning of $G^*$ induces a pinning of $M^*$, namely $(T^*,B^* \cap M^*,\{X_\alpha\}_{\alpha \in \Delta(T^*,M^*)})$.

The relative Weyl group $W(M^*,G^*)$ can be identified with a subquotient of $W(T^*,G^*)$, namely
\[ N(A_{M^*},W(T^*,G^*))/Z(A_{M^*},W(T^*,G^*)), \]
where $A_{M^*}$ is the maximal split torus in the center of $M^*$. Note that the subgroup $Z(A_{M^*},W(T^*,G^*))$ of $W(T^*,G^*)$ is naturally identified with $W(T^*,M^*)$, the absolute Weyl group of $M^*$. This subquotient comes with a splitting
\begin{equation} \label{eq:wsp} W(M^*,G^*) \rw N(A_{M^*},W(T^*,G^*)) \end{equation}
defined to map a given coset of $Z(A_{M^*},W(T^*,G^*))$ in $N(A_{M^*},W(T^*,G^*))$ to its unique member that preserves the set of positive roots $R(T^*,M^* \cap B^*)$. Via this splitting we will regard elements of $W(M^*,G^*)$ as elements of $W(T^*,G^*)$ without change in notation.

Furthermore, there are two splittings of the quotient $W(T^*,G^*) = N(T^*,G^*)/T^*$. The first one is the traditional splitting for the group $G^*(\ol{F}) \cong \tx{GL}_N(\ol{F})$, given by permutation matrices. We denote it by $w \mapsto \hat w$. It is multiplicative, but not $\Gamma$-equivariant unless $E/F$ is split. The second one is the one studied by Langlands-Shelstad in \cite[\S2.1]{LS87}. We denote it by $w \mapsto \tilde w$. It is $\Gamma$-equivariant, but not multiplicative. The relationship between $\hat w$ and $\tilde w$ is easily described.

\begin{lem} \label{lem:lsl} If $w \in W(T^*,G^*)$ belongs to the image of \eqref{eq:wsp}, the element $\hat w\cdot \tilde w^{-1}$ belongs to $Z(M^*)$ and is of order 2.
\end{lem}

\begin{proof}
Certainly $\hat w \cdot \tilde w^{-1}$ belongs to $T^*$. According to \cite[Prop 6.2.1]{Kal12a}, it is given by
\[ t_w := \prod_{\substack{\alpha \in R(T^*,G^*)\\ \alpha > 0\\w^{-1}\alpha<0}} y_\alpha(-1) \]
where $x_\alpha,y_\alpha \in X_*(T^*)$ are the unique members of the standard basis such that $\alpha^\vee=x_\alpha-y_\alpha$. If $w$ belongs to the image of \eqref{eq:wsp} and if $x_\alpha-y_\alpha$ is the coroot of a root belonging to the index set of the above product, then so is $x_\alpha-uy_\alpha$ for all $u \in W(T^*,M^*)$. Hence $t \in T^*$ is fixed under the action of $W(T^*,M^*)$ and, $M^* \times \ol{F}$ being the product of general linear groups, this is equivalent to $t \in Z(M^*)$.
\end{proof}

We see in particular that the automorphisms of $M^*$ given by conjugation by $\hat w$ and $\tilde w$ coincide. The group $W(M^*,G^*)$ is a finite algebraic group defined over $F$ and via the map $w \mapsto \hat w$, the group $W(M^*,G^*)(F)$ acts on $M^*$ by automorphisms over $F$. This action can be described as follows. There is a natural isomorphism $W(M^*,G^*)(F) \cong (\Z/2\Z)^k \rtimes S$ with $S \subset S_k$, corresponding to the decomposition $M^*=M_1^* \times \dots \times M_k^* \times M^*_-$. The subgroup $S \subset S_k$ is defined by $S=\{\sigma \in S_k| \forall i=1,\dots,k: \dim M_i^*=\dim M_{\sigma(i)}^* \}$. The $i$-th basis vector $(0,\dots,0,1,0,\dots,0)$ of $(\Z/2\Z)^k$ acts on $M^*$ by the pinned outer automorphism on $M^*_i$ and the identity on the other factors, and $\sigma \in S$ acts by permutation of the factors $M_1^*,\dots,M_k^*$.

The element $t_w = \hat w \cdot \tilde w^{-1}$ is not multiplicative in $w$. However, it does exhibit certain multiplicativity that will be useful for later. Namely, decompose $t_w = t_{w,+} \times t_{w,-}$ according to the decomposition $M = M_+ \times M_-$.

\begin{lem} \label{lem:tw-} The map $W(M^*,G^*)(F) \rw Z(M^*_-)$ given by $w \mapsto t_{w,-}$ is a group homomorphism.
\end{lem}
\begin{proof}
Let us write $M = M_1 \times \dots\times M_k \times M_-$ and identify $W(M^*,G^*)(F)$ with $(\Z/2\Z)^k \rtimes S$, with $S \subset S_k$. The formula for $t_w$ given in the proof of Lemma \ref{lem:lsl} implies that $t_{e_i,-} = (-1)^{N_i}$ for $e_i \in (\Z/2\Z)^k$ the $i$-th standard basis vector, and $t_{s,-}=1$ for $s \in S$. The claim follows.
\end{proof}
We remark that this lemma can be used equally well for the dual group $\hat G^*$ and its Levi subgroup $\hat M^*$, with $W(M^*,G^*)(F)$ replaced by $W(\hat M^*,\hat G^*)^\Gamma$. The proof is the same.

\subsubsection{Extended pure inner twists for Levi subgroups of unitary groups}

If $\Xi : G^* \rw G$ is an equivalence class of extended pure inner twists and $(M^*,P^*)$ and $(M,P)$ are parabolic pairs whose conjugacy classes are identified by $\Xi$, then the set $\Xi_M = \{(\xi,z) \in \Xi| \xi(M^*,P^*)=(M,P) \}$ is an equivalence class of extended pure inner twists $M^* \rw M$.

\begin{lem} \label{lem:c1} Assume that $F$ is local, $\Xi : G^* \rw G$ is an equivalence class of extended pure inner twists, and $(M,P)$ is a parabolic pair of $G$ corresponding to $(M^*,P^*)$. There exists an element $(\xi,z) \in \Xi$ such that $\xi(M^*,P^*) = (M,P)$ and $z$ commutes with $\tilde w$ for all $w \in W(M^*,G^*)$.
\end{lem}
\begin{proof}
Consider first the case when $G^*$ is linear. Then we have the product decomposition $M^* \cong M^*_1 \times \dots \times M^*_k$ with $M^*_i=\tx{GL}_{N_i}$ and $N=\sum N_i$. Let $h \in H^1_\tx{G^*-\tx{bsc}}(\mc{E},M^*)$ be the class of $\Xi_M$ and decompose it $h=h_1 \times \dots \times h_k$ accordingly. The Newton point of all $h_i$ is the same.  The usual twisting argument, coupled with the generalized Hilbert theorem 90 \cite[Theorem 29.2]{BookInvol} implies that all fibers of the Newton map $H^1_\tx{bsc}(\mc{E},M^*_i) \rw \tx{Hom}_F(\mb{D}_F,\mb{G}_m)$ are trivial. Thus, for $i,j$ with $N_i=N_j$ the classes $h_i$ and $h_j$ in $H^1_\tx{bsc}(\mc{E},M^*_i)=H^1_\tx{bsc}(\mc{E},M^*_j)$ are equal. This allows us to choose a representative $z=z_1 \times \dots \times z_k$ of $h$ such that for any $i,j$ with $N_i=N_j$ we have $z_i=z_j$. Since $\tilde w$ acts by permuting the factors $M^*_i$, the statement follows.

Now consider the case when $G^*$ is unitary and let again $h \in H^1_\tx{G^*-\tx{bsc}}(\mc{E},M^*)$ be the class of $\Xi_M$. Since $\mb{D}_F$ is split and $Z(G^*)$ is anisotropic, the Newton point of $h$ is trivial, and thus $h \in H^1(\Gamma,M^*)$. We have the product decomposition $M^* \cong M^*_1 \times \dots \times M^*_k \times M^*_-$, with $M^*_i=\tx{Res}_{E/F}(\tx{GL}_{N_i})$, $M^*_-=U_{E/F}(N_-)$, and $N=2\sum N_i + N_-$. Decomposing $h=h_1 \times \dots \times h_k \times h_-$ accordingly, we see that all $h_i$ are trivial. Thus we may choose a representative $z=z_1 \times \dots \times z_k \times z_-$ with $z_i=1$. Since $\tilde w$ fixes $M^*_-$ pointwise, the statement follows.
\end{proof}

\begin{lem} \label{lem:c2} Assume that $F$ is global, $\Xi : G^* \rw G$ is an equivalence class of inner twists, and $(M,P)$ is a parabolic pair of $G$ corresponding to $(M^*,P^*)$. There exists an element $\xi \in \Xi$ such that $\xi(M^*,P^*) = (M,P)$ and an element $z \in Z^1_\tx{G^*-\tx{bsc}}(\mc{E},M^*)$ that commutes with $\tilde w$ for all $w \in W(M^*,G^*)$, such that $(\xi,z)$ is an extended pure inner twist.
\end{lem}
\begin{proof}
Let $\xi \in \Xi$ be an arbitrary inner twist satisfying $\xi(M^*,P^*)=(M,P)$, let $h \in H^1(\Gamma,M^*/Z(G^*))$ be the class of the 1-cocycle $z_\sigma = \xi^{-1}\sigma(\xi)$. Decompose $M^* \cong M^*_1 \times \dots \times M^*_k \times M^*_-$, with $M^*_i=\tx{Res}_{E/F}(\tx{GL}_{N_i})$, $M^*_-=U_{E/F}(N_-)$, and $N=2\sum N_i + N_-$. We will produce a class $H \in H^1_\tx{G^*-\tx{bsc}}(\mc{E},U_{E/F}(N_1) \times \dots \times U_{E/F}(N_k) \times U_{E/F}(N_-))$  whose image in $H^1_\tx{G^*-\tx{bsc}}(\mc{E},M^*)$ lifts $h$ and which, when written as $H=H_1 \times \dots \times H_k \times H_-$, satisfies $H_i=H_j$ whenever $N_i=N_j$. Any 1-cocycle representing $H$ will then have the property of being fixed by all Langlands-Shelstad lifts of elements of $W(M^*,G^*)^\Gamma$.

We begin the construction. Let $\bar h$ be the image of $h$ in $H^1(\Gamma,M^*/Z(M^*))$, which we can write as $\bar h = \bar h_1 \times \dots \times \bar h_k \times \bar h_{-}$ according to the decomposition of $M^*$. We have $Z(M^*_i)=\tx{Res}_{E/F}\mb{G}_m$ and $Z(M^*_-)=U_{E/F}(1)$, and the embedding of $Z(G^*)=U_{E/F}(1)$ into $Z(M^*)=(\tx{Res}_{E/F}\mb{G}_m)^k \times U_{E/F}(1)$ is the diagonal embedding. According to Equation \eqref{eq:itex} and Corollary \ref{cor:itex}, the images of $\bar h_i$ in $H^2(\Gamma,\tx{Res}_{E/F}(\mb{G}_m))$ coincide and belong to the subset $H^2(\Gamma,U_{E/F}(1))$. Since the map $H^1(\Gamma,M^*_i/Z(M^*_i)) \rw H^2(\Gamma,Z(M^*_i))$ is injective by the generalized Hilbert theorem 90 \cite[Theorem 29.2]{BookInvol}, this means in particular that for any $i,j$ with $N_i=N_j$ the classes $\bar h_i$ and $\bar h_j$ are equal. Applying Lemma \ref{lem:c3} below we obtain a lift $\bar h_i' \in H^1(\Gamma,U_{E/F}(N_i)/U_{E/F}(1))$ for each $\bar h_i$ which we may moreover choose so that $\bar h_i'=\bar h_j'$ whenever $N_i=N_j$. By construction, all classes $\bar h_i'$ have the same image in $H^2(\Gamma,U_{E/F}(1))$ and this image also equals the image of $\bar h_-$. Thus, the element $\bar h_1' \times \dots \times \bar h_k' \times \bar h_-$ of
\[ H^1\left(\Gamma,\frac{U_{E/F}(N_1)}{U_{E/F}(1)} \times \dots \times \frac{U_{E/F}(N_k)}{U_{E/F}(1)} \times \frac{U_{E/F}(N_-)}{U_{E/F}(1)}\right) \]
lifts to an element $h'$ of
\[ H^1\left(\Gamma,\frac{U_{E/F}(N_1) \times \dots \times U_{E/F}(N_k) \times U_{E/F}(N_-)}{U_{E/F}(1)}\right).\]
The image of $h'$ in $H^2(\Gamma,U_{E/F}(1))=H^2(\Gamma,Z(G^*))$ equals the image of $h$ there. This, together with the injectivity of the maps $H^1(\Gamma,M^*/Z(G^*)) \rw H^1(\Gamma,M^*/Z(M^*))$ and $H^1(\Gamma,M^*_i/Z(M^*_i)) \rw H^2(\Gamma,Z(M^*_i))$, which is part of diagram \eqref{eq:itex} and follows from Lemma \ref{lem:zmzgind}, implies that the image of $h'$ under \[H^1\left(\Gamma,\frac{U_{E/F}(N_1) \times \dots \times U_{E/F}(N_k) \times U_{E/F}(N_-)}{U_{E/F}(1)}\right) \rw H^1\left(\Gamma,\frac{M^*}{Z(G^*)}\right) \]
equals $h$. Using \cite[Prop. 10.4]{Kot14} we now find a class $H' \in H^1_\tx{bsc}(\mc{E},U_{E/F}(N_1) \times \dots \times U_{E/F}(N_k) \times U_{E/F}(N_-))$ that lifts $h'$. The Newton point of this element necessarily belongs to the diagonally embedded copy of $U_{E/F}(1)$, hence the image of $H'$ in $H^1_\tx{bsc}(\mc{E},M^*)$ actually belongs to $H^1_\tx{G^*-\tx{bsc}}(\mc{E},M^*)$ and will be called $H$. By construction, the image of $H$ in $H^1(\Gamma,M^*/Z(G^*))$ equals $h$. Moreover, $H$ has a representative $z$ which is valued in the subgroup $U_{E/F}(N_1) \times \dots \times U_{E/F}(N_k) \times U_{E/F}(N_-)$ of $M^*$ and if we write it as $z=(z_1,\dots,z_k,z_-)$, then for all $N_i=N_j$ we have $z_i=z_j$. But this implies that $z$ is fixed by all Langlands-Shelstad lifts of elements of $W(M^*,G^*)^\Gamma$.
\end{proof}

\begin{lem} \label{lem:c3} Let $E/F$ be a separable extension of fields and let $R=\tx{Res}_{E/F}(\tx{GL}(N))$. Let $\theta$ be the pinned automorphism of $R$ whose group of fixed points is $G^*=U_{E/F}(N)$. The image of $H^1(\Gamma,G^*/Z(G^*))$ in $H^1(\Gamma,R/Z(R))$ coincides with the preimage of $H^2(\Gamma,Z(G^*))$ seen as a subset of $H^2(\Gamma,Z(R))$.
\end{lem}
\begin{proof}
The image of $H^1(\Gamma,G^*/Z(G^*))$ in $H^1(\Gamma,R/Z(R))$ clearly maps into the image of $H^2(\Gamma,Z(G^*))$ in $H^2(\Gamma,Z(R))$ by functoriality of the long exact cohomology sequence. We turn to the converse inclusion.

According to the generalized Hilbert theorem 90 \cite[Theorem 29.2]{BookInvol}, the maps $H^1(\Gamma,R/Z(R)) \rw H^2(\Gamma,Z(R))$ and $H^2(\Gamma,Z(G^*)) \rw H^2(\Gamma,Z(R))$ are injective. Let $h \in H^1(\Gamma,R/Z(R))$ have image $x \in H^2(\Gamma,Z(R))$ and assume $x \in H^2(\Gamma,Z(G^*))$. We want to find a lift $h' \in H^1(\Gamma,G^*/Z(G^*))$ of $h$. The norm map $Z(R) \rw \mb{G}_m$ induces a map $H^2(\Gamma,Z(R)) \rw H^2(\Gamma,\mb{G}_m)$, which, under the Shapiro isomorphism $H^2(\Gamma,Z(R)) \cong H^2(\Gamma_E,\mb{G}_m)$ becomes identified with the corestriction map $H^2(\Gamma_E,\mb{G}_m) \rw H^2(\Gamma,\mb{G}_m)$. The element $x$ lies in the kernel of this map, which implies via \cite[Theorem 3.1]{BookInvol} that the central simple algebra $A/E$ corresponding to $x$ admits an involution $\tau$ of the second kind. The corresponding unitary group $U(A,\tau)$ (see \cite[\S23.A]{BookInvol}) is an inner form of $G^*$ and corresponds to a class $h' \in H^1(\Gamma,G^*/Z(G^*))$. The image of $h'$ in $H^2(\Gamma,Z(R)) \cong H^2(\Gamma_E,\mb{G}_m)$ corresponds to the central simple algebra $A$ and thus equals $x$. We conclude that the image of $h'$ in $H^1(\Gamma,R/Z(R))$ equals $h$.
\end{proof}

\section{Chapter 1: Parameters and the Main Theorems}\label{s:param-main-thm}

\subsection{Endoscopic data} \label{sub:endoscopic-data}

  In this section we introduce endoscopic data that are relevant for studying the representations of unitary groups. There are two kinds of endoscopic data: one for (not necessarily quasi-split) unitary groups and the other for a twisted form of $\GL(N)$. Though the latter already enters the very definition of parameters for unitary groups in a way, we will be mainly concerned with the former in this paper. One reason is that we never have to work directly with the twisted trace formula for $\GL(N)$ as long as we are willing to accept various results in the quasi-split case (in which the twisted formula is indispensable).

\subsubsection{Endoscopic triples} \label{subsub:endoscopic-triples}

  We recall some general definitions of endoscopic data from \cite[1.2]{LS87} and \cite[2.1]{KS99} taking into a simplification that the group $\cH$ can be taken to be the $L$-group of $H$ in all cases of our concern. Let $F$ be a local or global field. Consider a pair $(G^*,\theta^*)$ consisting of a connected quasi-split reductive group $G^*$ over $F$, equipped with a fixed $\Gamma$-invariant pinning, and a pinned automorphism $\theta^*$ of $G^*$. %
  As explained in \cite[1.2]{KS99} we have an automorphism $\hat \theta$ of $\hat G^*$ and may assume that $\hat \theta$ preserves a $\Gamma$-pinning for $\hat G^*$, the latter being fixed once and for all. Set $^L \theta:=\hat \theta\rtimes \id_{W_F}$, an automorphism of ${}^L G^*$.
  An (extended) \textbf{endoscopic triple} is a triple $(H,s,\eta)$, where $H$ is a connected quasi-split reductive group over $F$, $s\in \hat G^*$, and $\eta:{}^L H \ra {}^L G^*$ is an $L$-morphism such that
\begin{itemize}
  \item $\Int(s)\circ \hat \theta$ preserves a pair of a Borel subgroup and a maximal torus therein in $\hat G^*$, and $\Int(s)\circ {}^L \theta\circ \eta=\eta$.
  \item $\eta(\hat H)$ is the connected component of the subgroup of $\Int(s)\circ \hat\theta$-fixed elements in $\hat G^*$.
\end{itemize}
  We say that $(H,s,\eta)$ is \textbf{elliptic} if $\eta(Z(\hat H)^\Gamma)^0\subset Z(\hat G^*)$. In case $\theta=\id$ an endoscopic triple is often said to be \textbf{ordinary}. If $(H,s,\eta)$ is an endoscopic triple then by considering $(H,{}^L H,s,\eta)$, it can be viewed as an endoscopic datum in the sense of \cite[2.1]{KS99}. In general, not all endoscopic data arise in this way. In the cases needed for this volume however, all endoscopic data do arise in this way and we find it more convenient to work with the notion of (extended) endoscopic triples.

  We also need to define three notions of isomorphism between endoscopic triples. A \textbf{strict isomorphism} (resp. \textbf{weak isomorphism}, resp. \textbf{isomorphism}) is $(H,s,\eta)\ra (H',s',\eta')$ is an element $g\in \hat G^*$ such that $g \eta({}^L H)g^{-1}=\eta'({}^L H')$ and $gs\hat \theta(g)^{-1}=s'$ (resp. $gs\hat \theta(g)^{-1}=s'$ modulo $Z(\hat G^*)^\Gamma$, resp. $gs\hat \theta(g)^{-1}=s'$ modulo $Z(\hat G^*)$). The definition of isomorphism is the same as in \cite[(2.1.5), (2.1.6)]{KS99} but sometimes too loose for our purposes; that is why we introduce stricter versions. Clearly $\eta(h)$ for each $h\in \hat H$ is an automorphism of $(H,s,\eta)$. Define the outer automorphism group
  $$\Out_{G^*\rtimes \theta^*}((H,s,\eta)):=\Aut((H,s,\eta))/\eta(\hat H),$$
  using strict automorphisms. Likewise $\Out^{\mathrm{w}}_{G^*\rtimes \theta^*}((H,s,\eta))$ and $\ol{\Out}_{G^*\rtimes \theta^*}((H,s,\eta))$ are defined by means of weak automorphisms and automorphisms, respectively.\footnote{As usual the overline notation indicates that something is taken modulo the center of the dual group.}

  Finally when $(G,\theta)$ is an inner form of $(G^*,\theta^*)$ then the notion of endoscopic triple for $G$ is the same as that for $G^*$. In this paper we prefer to use the notation
  $$\fke=(G^\fke,s^\fke,\eta^\fke) $$
  to denote an endoscopic triple or a twisted endoscopic triple, instead of $(H,s,\eta)$.
  We write $\cE(G\rtimes \theta)$ (resp. $\cE^{\mathrm{w}}(G\rtimes \theta)$, resp. $\ol{\cE}(G\rtimes \theta)$) for the set of strict isomorphism classes (resp. weak isomorphism classes, resp. isomorphism classes) of all endoscopic triples.
  The corresponding subsets of elliptic endoscopic triples will be denoted $\cE(G\rtimes \theta)$, $\cE^{\mathrm{w}}_\el(G\rtimes \theta)$, and $\ol{\cE}_\el(G\rtimes \theta)$.
  It will be convenient to work with a set of representatives for the isomorphism classes, cf. \S\ref{subsub:endo-unitary} below. If $\theta=\id$ then simply write  $\cE(G)$, $\cE_{\el}(G)$, etc. In general we should have used endoscopic data in the definition, but this does no harm to us since all endoscopic data are represented by endoscopic triples in all cases we consider.

\subsubsection{Normalization of transfer factors} \label{subsub:normtf} \label{sec:normtfs}
An important ingredient in the stabilization of the Arthur-Selberg trace formula is the transfer factor, defined for ordinary endoscopy in \cite{LS87} and for twisted endoscopy in \cite{KS99}. Given an endoscopic triple $\mf{e}$ for a connected reductive (possibly twisted) group $G$ defined over a local field $F$, these references provide a canonical \emph{relative} transfer factor. For the purposes of the stabilization, we need an \emph{absolute} transfer factor. While in some cases one can work with an arbitrary choice of absolute transfer factor, in order to extract the necessary information from the spectral side of the trace formula, one needs to fix a specific normalization that has the right properties. We will do this now in the case of ordinary endoscopy, as well as in a simple case of twisted endoscopy, which will be used for the normalization of intertwining operators. In fact, the simple case of twisted endoscopy that we need is a generalization of ordinary endoscopy, so we will present our arguments in this case.

Let $F$ be a local field. Before we begin with the construction, we remind the reader that there are two different normalizations of the relative transfer factor for twisted endoscopy. These are explained in \cite[\S5]{KS12}, where they are called $\Delta_D$ and $\Delta'$. The factor $\Delta_D$ is compatible with the normalization of the local Artin reciprocity map $F^\times \rw W_F^\tx{ab}$ used by Deligne, which sends a uniformizing element to the inverse of the Frobenius automorphism, while the factor $\Delta'$ is compatible with the normalization which sends a uniformizing element to the Frobenius automorphism. The absolute transfer factor we will define will correspond to the relative transfer factor $\Delta'$ and will thus be compatible with the classical normalization of the reciprocity map, and so also with the classical Langlands correspondence for tori \cite{Lan97}. We will however use the symbol $\Delta$ to denote this absolute transfer factor, and not $\Delta'$. The reader should be warned that the symbol $\Delta$ is used in \cite{KS12} to denote yet another normalization of the relative transfer factor, which is available for ordinary, but not for twisted, endoscopy, and which follows the conventions of \cite{LS87}. The difference between the relative factor $\Delta$ of \cite{LS87} and the relative factor $\Delta'$ of \cite{KS12} is easy to explain: One obtains $\Delta$ from $\Delta'$ by replacing the endoscopic element $s$ with its inverse $s^{-1}$. Despite our notation, our absolute transfer factor $\Delta$ will be compatible with the relative factor $\Delta'$ of \cite{KS12}, and not with the relative factor $\Delta$. This choice of normalization was made for multiple reasons, an important one being that the internal structure of $L$-packets it leads to is compatible with the formulations of the local Langlands conjectures of \cite{Vog93} and \cite{GGP12}. It will however lead to the occasional appearance of an inverse in some formulas.

The construction of the transfer factor will involve the cohomology of complexes of tori of length two. When dealing with pure inner twists, it will be enough to consider Galois cohomology and all we need has already been developed in the appendices to \cite{KS99}. When dealing with extended pure inner twists, we have to use the more general cohomology $B(F,-)_\tx{bsc}$ defined by Kottwitz in \cite{Kot14} and discussed in Section \ref{subsub:inner}. More precisely, we will need the following.
\begin{itemize}
\item Let $T \rw S$ be a complex of length 2 of tori defined over a local field $F$. We have the cohomology group $B(F,T \rw S)$. This group is functorial in $T \rw S$ and fits into the two long exact sequences of hypercohomology described in \cite[\S A.1]{KS99}. There is a natural embedding $H^1(F,T \rw S) \rw B(F,T \rw S)$. Each element of $B(F,T \rw S)$ provides a character of $H^1(W_F,\hat S \rw \hat T)$. If the element happens to belong to the subgroup $H^1(F,T \rw S)$, this character coincides with the one constructed in \cite[\S A]{KS99}.

\item Let $T \rw S$ be a complex of length 2 of tori defined over a global field $F$. We have the cohomology group $B(\A/F,T \rw S)$. This group is functorial in $T \rw S$ and fits into the two long exact sequences of hypercohomology described in \cite[\S A.1]{KS99}. There is a natural embedding $H^1(\A/F,T \rw S) \rw B(\A/F,T \rw S)$. Each element of $B(\A/F,T \rw S)$ provides a character of $H^1(W_F,\hat S \rw \hat T)$. If the element happens to belong to the subgroup $H^1(\A/F,T \rw S)$, this character coincides with the one constructed in \cite[\S C]{KS99}.

\item There is a natural map $B(F_v,T \rw S) \rw B(\A/F,T \rw S)$ for each place $v$ of $F$. This map is dual to the restriction map $H^1(W_F,\hat S \rw \hat T) \rw H^1(W_{F_v},\hat S \rw \hat T)$.
 \end{itemize}

When $F$ is a $p$-adic field, these results were obtained in \cite[\S9,10,11]{Kot97}. In general they will be established in \cite{KMS_B}. For now, we will take them for granted. We emphasize again that in the case of pure inner twists, the cohomology groups $H^1(F,T \rw S)$ and $H^1(\A/F,T \rw S)$ are sufficient and the results we need have already been established in \cite{KS99}.

We now proceed to construct the transfer factor. Let $G^*$ a quasi-split connected reductive group defined over $F$ and endowed with a pinning, and let $\theta^*$ be an automorphism preserving that pinning. We assume that $Z(G^*)$ is connected. Let further $\psi_F : F \rw \C^\times$ be a non-trivial additive character. It determines, together with the pinning of $G^*$, a Whittaker datum for $G^*$, see \cite[\S5.3]{KS99}. Let $(\xi,z) : G^* \rw G$ be an extended pure inner twist. We assume that $\theta^*(z)=z$ and set $\theta=\xi\circ\theta^*\circ\xi^{-1}$. This is an automorphism of $G$ defined over $F$. Not all twisted groups $(G,\theta)$ arise in this way, but those are the ones that will be relevant for us. Luckily, normalizing twisted transfer factors for these special twisted groups is very similar to normalizing non-twisted transfer factors.

Let $\mf{e}$ be an endoscopic triple for $(G,\theta)$. Let $\delta \in G(F)$ be $\theta$-strongly regular and $\theta$-semi-simple and let $\gamma \in G^\mf{e}(F)$. Assume that $\gamma$ is a norm of $\delta$. Following \cite[\S5.3]{KS99} and \cite[\S5.4]{KS12} we are going to define the absolute transfer factor
\[ \Delta[\mf{e},\xi,z](\gamma,\delta) \]
as a product
\[ \epsilon(\frac{1}{2},V,\psi_F)\Delta_I^\tx{new}(\gamma,\delta)^{-1} \Delta_{II}(\gamma,\delta) \Delta_{III}(\gamma,\delta)^{-1}\Delta_{IV}(\gamma,\delta). \]
The first factor is the Artin $\epsilon$-factor, normalized according to Langlands' notation, for the virtual representation $V=X^*(T^*)^{\theta^*}\otimes \C - X^*(T^\mf{e}) \otimes \C$, where $T^*$ and $T^\mf{e}$ are minimal Levi subgroups of $G^*$ and $G^\mf{e}$. The definitions of $\Delta_{II}$ and $\Delta_{IV}$ are given in \cite[\S4.3,\S4.5]{KS99}. The definition of $\Delta_I^\tx{new}$ is given in \cite[\S3.4]{KS12}. We will use these definitions without modification. It is the factor $\Delta_{III}$ that we must define. A definition of $\Delta_{III}(\gamma,\delta)$ is given in \cite[\S5.3]{KS99} under the assumption $(\xi,z)=(\tx{id},1)$, while a definition of a relative factor $\Delta_{III}(\gamma,\delta,\gamma',\delta')$ is given in \cite[\S4.4]{KS99} for general $\xi$. We will now define a factor $\Delta_{III}(\gamma,\delta)$ adapted to the extended pure inner twist $(\xi,z)$ by extending the arguments in \cite[\S4.4]{KS99}.

We let $S' \subset G^\mf{e}$ be the centralizer of $\gamma$, a maximal torus of $G^\mf{e}$. Choose an admissible isomorphism $S' \rw S^*_{\theta^*}$, where $(S^*,C^*)$ is a Borel pair of $G^*$ invariant under $\theta^*$ with $S^*$ defined over $F$. The assumption that $\gamma$ is a norm of $\delta$ ensures the existence of $g \in G^*$ and $\delta^* \in S^*$ such that the image of $\delta^*$ in $S^*_{\theta^*}$ equals the image of $\gamma$ under $S' \rw S^*_{\theta^*}$ and moreover $\delta=\xi(g^{-1}\delta^*\theta^*(g))$. Let $K/F$ be an finite Galois extension such that the element $z \in B(F,G^*)_\tx{bsc}$ has a representative in $Z^1_\tx{bsc}(\mc{E}(K/F),G^*(K))$ and such that $g \in G^*(K)$. We use the same letter $z$ to denote this representative. For each $e \in \mc{E}(K/F)$, let $v(e)=gz(e)\sigma_e(g^{-1})$, where $\sigma_e \in \Gamma_{K/F}$ is the image of $e$. The argument of \cite[Lemma 4.4.A]{KS99} shows that $(v(e)^{-1},\delta^*)$ provides an element of $Z^1_\tx{bsc}(\mc{E}(K/F),S^*(K) \stackrel{1-\theta^*}{\lrw} S^*(K))$. Notice that the restriction of $v$ to the subgroup $\mb{D}_{K/F} \subset \mc{E}(K/F)$ is equal to the restriction of $z$, both of them taking values in $Z(G^*)$. In particular, $v$ is $G^*$-basic.

On the other hand, the $\chi$-data chosen for the construction of $\Delta$ provides, as described in \cite[\S4.4]{KS99}, an $L$-embedding $^LS' \rw {^LG^\mf{e}}$, which composed with $\eta^{\mf{e}}$ gives an $L$-embedding $f : {^LS'} \rw {^LG}$. The same $\chi$-data provides an $L$-embedding $^L(S^*_{\theta^*}) \rw {^LG}$. We compose this embedding with the isomorphism $^LS' \rw {^LS^*_{\theta^*}}$ dual to the chosen admissible isomorphism $S' \rw S^*_{\theta^*}$ and obtain a second embedding $g : {^LS'} \rw {^LG}$. We then have $f(x \rtimes w) =  g(a_S(w)\cdot x \rtimes w)$ for any $x \rtimes w \in \hat S' \rtimes W_F = {^LS'}$, and suitable $a_S(w) \in \hat S$. The dual of the admissible isomorphism $S' \rw S^*_{\theta^*}$ is a composition of $\hat S' \rw \hat T^{\hat\theta}$ and $\hat T \rw \hat S^*$, and via the latter map we can view $s^\mf{e}$ as an element of $\hat S^*$. A direct calculation shows that $(a_S^{-1},s^\mf{e})$ is an element of $Z^1(W_F,\hat S^* \stackrel{1-\hat\theta^*}{\lrw} \hat S^*)$.

We define $\Delta_{III}(\gamma,\delta)$ to be the pairing of these two objects with respect to the pairing between $B(F,S^* \rw S^*)$ and $H^1(W_F,\hat S^* \rw \hat S^*)$ discussed above. This concludes the construction of $\Delta[\mf{e},\xi,z]$. As a first step, we need to check that this construction actually gives a transfer factor.

\begin{pro} \label{pro:tf} Let $\gamma_1,\gamma_2 \in G^\mf{e}(F)$ be norms of strongly $\theta$-regular $\theta$-semi-simple elements $\delta_1,\delta_2 \in G(F)$. Then
\[ \frac{\Delta[\mf{e},\xi,z](\gamma_1,\delta_1)}{\Delta[\mf{e},\xi,z](\gamma_2,\delta_2)} = \Delta'(\gamma_1,\delta_1;\gamma_2,\delta_2), \]
where the right hand side is the canonical relative transfer factor of \cite[\S5.4]{KS12}.
\end{pro}
\begin{proof} The proof of this proposition is very similar to that of Proposition 2.3.1 in \cite{Kal11}. We give the details for the sake of completeness. By construction one sees immediately that the equality to be proved is equivalent to
\[ \frac{\Delta_{III}(\gamma_1,\delta_1)}{\Delta_{III}(\gamma_2,\delta_2)} = \Delta_{III}(\gamma_1,\delta_1;\gamma_2,\delta_2), \]
where the factors on the left are the ones just constructed, while the factor on the right is the one constructed in \cite[\S4.4]{KS99}. Let $(v_i(\sigma)^{-1},\delta^*_i)$ for $i=1,2$ be the two elements of $Z^1_\tx{G^*-\tx{bsc}}(\mc{E}(K/F),S_i^*(K) \stackrel{1-\theta^*}{\lrw} S_i^*(K))$ constructed above for the two pairs $\gamma_i,\delta_i$, and let $(a_{S_i}^{-1},s^\mf{e}_i)$ be the two corresponding elements of $Z^1(W_F,\hat S_i^* \stackrel{1-\hat\theta}{\lrw} \hat S_i^*)$. We can interpret the left hand side of our equality as the pairing of
\[ V_{12} := ((v_1^{-1},\delta^*_1),(v_2^{-1},\delta^*_2)^{-1}) \in Z^1_\tx{alg}(\mc{E}(K/F),[S_1^* \times S_2^*](K) \rw [S_1^* \times S_2^*](K)) \]
with the element
\[ A_{12} := ((a_{S_1}^{-1},s^\mf{e}_1),(a_{S_2}^{-1},s^\mf{e}_2)) \in Z^1(W_F,\hat S_1^* \times \hat S_2^* \rw \hat S_1^* \times \hat S_2^*). \]
On the other hand, the right hand side is given by the pairing of an element
$\tb{V} \in H^1(\Gamma,U \rw S_{12}^*)$ with an element $\tb{A} \in H^1(W_F,\hat S_{12}^* \rw \hat U)$, both constructed in \cite[\S4.4]{KS99}. Our task is to show that the two pairings give the same result. While doing so, we will recall the necessary notation from \cite{KS99}.

We consider the torus $S_{12}^* := S_1^* \times S_2^* /Z(G^*)$, with $Z(G^*)$ embedded via the map $z \mapsto (z,z^{-1})$. The quotient map $S_1^* \times S_2^* \rw S_{12}^*$ extends to a map of complexes of tori, under which we may map the element $V_{12}$ to obtain an element $V_{12}' \in Z^1_\tx{alg}(\mc{E}(K/F),S_{12}^*(K) \rw S_{12}^*(K))$. Recalling that the restrictions of $v_1$ and $v_2$ to $\mb{D}_{K/F}$ are both equal to the restriction of $z$, we see that the restriction of $V_{12}$ to $\mb{D}_{K/F}$ factors through the embedding of $Z(G^*)$ into $S_1^* \times S_2^*$ whose cokernel is $S_{12}^*$. This shows that $V_{12}'$ is trivial on $\mb{D}_{K/F}$ and thus belongs to the subgroup $Z^1(\Gamma_{K/F},S_{12}^*(K) \rw S_{12}^*(K))$ of $Z^1_\tx{alg}(\mc{E}(K/F),S_{12}^*(K) \rw S_{12}^*(K))$. Now recall that $U=S^*_{1,\tx{sc}} \times S^*_{2,\tx{sc}}/Z(G_\tx{sc}^*)$. We claim that the image of $\tb{V}$ under the obvious map $[U \rw S_{12}^*] \rw [S_{12}^* \rw S_{12}^*]$ equals the class of $V_{12}'$. This will be obvious once we recall the construction of $\tb{V}$. Indeed, $\tb{V}$ is represented by the cocycle $(V(\sigma),D)$, where $V(\sigma)=(v_{1,KS}(\sigma)^{-1},v_{2,KS}(\sigma))$ and $D=(\delta_1^*,\delta_2^{*,-1})$. The 1-cochains $v_{i,KS} \in C^1(\Gamma,S_i^*)$ are given by $v_{i,KS}(\sigma)=g_iu(\sigma)\sigma(g_i^{-1})$, where $\delta_i=\xi(g_i^{-1}\delta_i^*\theta^*(g_i))$ and $u(\sigma) \in C^1(\Gamma_{K/F},G^*_\tx{sc}(K))$ lifts the 1-cocycle $\xi^{-1}\sigma(\xi)$. For each $e \in \mc{E}(K/F)$, the images of $u(\sigma_e)$ and $z(e)$ in $G^*_\tx{ad}(K)$ are equal, where again $\sigma_e \in \Gamma_{K/F}$ is the image of $e$ under the natural map $\mc{E}(K/F) \rw \Gamma_{K/F}$. Hence there exists a 1-cochain $x \in C^1(\mc{E}(K/F),Z(G^*))$ with $z(e)=\bar u(\sigma_e)x(e)$, where $\bar u$ is the composition of $u$ with the natural map $G^*_\tx{sc}(K) \rw G^*(K)$. From this one sees that the image of $(V(\sigma),D)$ in $Z^1(\Gamma_{K/F},S_{12}^*(K) \rw S_{12}^*(K))$ is equal to $V_{12}$.

To complete the proof we need to produce an element of $H^1(W_F,\hat S^*_{12} \rw \hat S^*_{12})$ which simultaneously maps to $\tb{A} \in H^1(W_F,\hat S^*_{12} \rw \hat U)$ and to the class of $A_{12}$ in $H^1(W_F,\hat S^*_1 \times \hat S^*_2 \rw \hat S^*_1 \times \hat S^*_2)$. This can be done as follows. For the construction of the transfer factor we have chosen admissible isomorphisms $S_i^* \rw \hat T$. We now use them to identify $\hat S^*_1$ and $\hat S^*_2$ with $\hat T$. Recalling that by assumption $Z(G^*)$ is connected, one checks that the torus $\hat S^*_{12}$ dual to $S^*_{12}$ is equal to the subgroup of $\hat S^*_1 \times \hat S^*_2$ consisting of those pairs $(a,b)$ for which $ab^{-1} \in \hat T_\tx{der}$. By construction, the elements $s^\mf{e}_i \in \hat S_i^*$ are both identified with the element $s^\mf{e} \in \hat T$, so the pair $(s_1^\mf{e},s_2^\mf{e})$ belongs to the subgroup $\hat S^*_{12}$ of $\hat S^*_1 \times \hat S^*_2$. On the other hand, it is argued in the proof of \cite[Lemma 4.4.B]{KS99} that for each $w \in W_F$ the quotient $a_{S_2}(w)/a_{S_1}(w)$ belongs to $\hat S_{2,\tx{der}}^*$. Thus again, the pair $(a_{S_1}^{-1},a_{S_2}^{-1})$ takes values in the subgroup $\hat S^*_{12}$ of $\hat S^*_1 \times \hat S^*_2$. We have thus seen that $A_{12} \in Z^1(W_F,\hat S^*_{12} \rw \hat S^*_{12})$. The fact that its image in $Z^1(W_F,\hat S^*_{12} \rw \hat U)$ represents $\tb{A}$ follows by inspecting the construction of $\tb{A}$ given just before the statement of \cite[Lemma 4.4.B]{KS99}.

\end{proof}

We will now study some basic equivariance properties of the factor $\Delta[\mf{e},\xi,z]$. On the one hand, given $x \in Z(\hat G^*)^\Gamma$, we can consider the endoscopic triple $x\mf{e}=(G^\mf{e},xs^\mf{e},\eta^\mf{e})$. On the other hand, given $y \in Z^1_\tx{alg}(\mc{E},Z(G^*)^{\theta^*,\circ})$, we may consider the inner twist $(\xi,yz)$, which is of the same special type as $(\xi,z)$. In order to study how $\Delta[\mf{e},\xi,z]$ would change if we replace $\mf{e}$ by $x\mf{e}$ or $(\xi,z)$ by $(\xi,yz)$, we recall that Kottwitz's map \eqref{eq:kotisoloc} provides a pairing between the set $B(F,G^*)_\tx{bsc}$ and the group $Z(\hat G^*)^\Gamma$, as well as between the set $B(F,Z(G^*)^{\theta^*,\circ})$ and the group $[\hat G^*/\hat G^*_\tx{der}]_{\hat\theta,\tx{free}}^\Gamma$, which is the set of $\Gamma$-fixed points in the torsion-free quotient of the $\hat\theta$-coinvariants of $\hat G^*/\hat G^*_\tx{der}$. We denote both of these pairings by $\<\cdot,\cdot\>$.

\begin{lem} \label{lem:tfequi}
Let $\bar s^\mf{e}$ denote the image of $s^\mf{e}$ in $[\hat G^*/\hat G^*_\tx{der}]_{\hat\theta,\tx{free}}$. Then $\bar s^\mf{e}$ is $\Gamma$-fixed and we have
\[ \Delta[x\mf{e},\xi,z]=\<z,x\>\Delta[\mf{e},\xi,z] \]
and
\[ \Delta[\mf{e},\xi,yz]= \<y,\bar s^\mf{e}\>\Delta[\mf{e},\xi,z]. \]
\end{lem}
\begin{proof}
That $\bar s^\mf{e}$ is $\Gamma$-fixed follows immediately from the definition of endoscopic triple given in Section \ref{subsub:endoscopic-triples}. Replacing $\mf{e}$ by $x\mf{e}$ multiplies the element
\[ (a_S^{-1},s^\mf{e}) \in Z^1(W_F,\hat S^* \stackrel{1-\hat\theta^*}{\lrw} \hat S^*) \]
by $(1,x)$ and hence the factor $\Delta_{III}$ gets multiplied by $\<(v(e)^{-1},\delta^*),(1,x)\>$. Since $(1,x)$ is the image of $x \in H^1(W_F,1 \rw Z(\hat G^*))$ under the map induced by
\[ [1 \rw Z(\hat G^*)] \rw [\hat S^* \stackrel{1-\hat\theta^*}{\lrw} \hat S^*] \]
we may map $(v(e)^{-1},\delta^*)$ under the dual map
\[ [S^* \stackrel{1-\theta^*}{\lrw} S^*] \rw [G^* \rw 1]. \]
The image of $(v(e)^{-1},\delta^*)$ in $B(F,G^*)_\tx{bsc}$ is by construction equal to $z^{-1}$ and thus $\Delta_{III}$ is multiplied by $\<z,x\>^{-1}$.

Replacing $z$ by $yz$ multiplies the element
\[ (v(e)^{-1},\delta^*) \in Z^1_\tx{alg}(\mc{E},S^* \stackrel{1-\theta^*}{\lrw} S^*) \]
by $(y^{-1},1)$ and hence the factor $\Delta_{III}$ is multiplied by $\<(y^{-1},1),(a_S^{-1},s^\mf{e})\>$. Since $(y^{-1},1)$ is the image of $y^{-1} \in Z^1_\tx{alg}(\mc{E},Z(G^*)^{\theta^*,\circ})$, we may map $(a_S^{-1},s^\mf{e})$ under the natural map induced by
\[ [\hat S^* \stackrel{1-\hat\theta^*}{\lrw} \hat S^*] \rw [1 \rw (\hat G^*/\hat G^*_\tx{der})_{\theta^*,\tx{free}}] \]
to obtain the element $\bar s^\mf{e} \in [\hat G^*/\hat G^*_\tx{der}]_{\theta^*,\tx{free}}^\Gamma$. Thus $\Delta_{III}$ gets multiplied by $\<y^{-1},s^\mf{e}\>$.

Since $\Delta_{III}$ contributes to $\Delta[\mf{e},\xi,z]$ via its inverse, the proof is complete.
\end{proof}

We now consider the global situation. Let $F$ be a global field and as before we chose an extension to $\ol{F}$ for each place $v$ of $F$, thereby identifying the absolute Galois group of $F_v$ with the decomposition group $\Gamma_v$. We assume that $(\xi,z) : G^* \rw G$ is an extended pure inner twist and $\mf{e}$ is an extended endoscopic triple, but now all defined over the global field $F$. Furthermore, we assume given a $\Gamma$-invariant pinning of $G^*$ and a non-trivial character $\psi_F : \A/\F \rw \C^\times$. At each place $v$ we obtain the corresponding local objects defined over $F_v$. Given $\gamma \in G^\mf{e}(\A)$ a norm of a $\theta$-strongly regular $\theta$-semi-simple element $\delta \in G(\A)$, we define
\[ \Delta_\A[\fke,\xi](\gamma,\delta) := \prod_v \Delta[\fke,\xi_v,z_v](\gamma_v,\delta_{v})  \]
Almost all terms in the product are equal to $1$, so the product is well-defined. Furthermore, the product is independent of $z$. This follows from Lemma \ref{lem:tfequi} and the exact sequence \eqref{eq:kotisolocglo}.

\begin{pro}\label{prop:trans-factor-coincide}
The factor $\Delta_\A[\mf{e},\xi]$ coincides with the inverse of the canonical adelic transfer factor defined in \cite[\S7.3]{KS99}.
\end{pro}

\begin{proof}
Let $\Delta_{\A,KS}$ denote the canonical adelic transfer factor of \cite[\S7.3]{KS99}. It is defined under the assumption that there exists $\gamma_0 \in G^\mf{e}(F)$ which is a norm of a $\theta$-strongly regular, $\theta$-semi-simple element $\delta_0 \in G(\A)$. Under this assumption, for any $\gamma_1 \in G^\mf{e}(\A)$ and $\delta_1 \in G(\A)$ with the same properties, the defining equation is
\[ \Delta_{\A,KS}(\gamma_1,\delta_1)=\Delta_\A'(\gamma_1,\delta_1;\gamma_0,\delta_0)^{-1}\<\tx{obs}(\delta_0),\kappa_0\>. \]
Here the first factor on the right is the canonical relative adelic transfer factor and the second factor will be recalled in a moment. Given Proposition \ref{pro:tf}, in order to prove
\[ \Delta_\A[\mf{e},\xi](\gamma_1,\delta_1) = \Delta_{\A,KS}(\gamma_1,\delta_1)^{-1} \]
it will be enough to show
\[ \Delta_\A[\mf{e},\xi](\gamma_0,\delta_0) = \<\tx{obs}(\delta_0),\kappa_0\>^{-1}. \]
For this, we let $S' \subset G^\mf{e}$ be the centralizer of $\gamma_0$, a maximal torus defined over $F$. We choose an admissible isomorphism $S' \rw S^*_{\theta^*}$, where $S^* \subset G^*$ is a $\theta^*$-stable maximal torus defined over $F$ and contained in a $\theta^*$-stable Borel subgroup defined over $\ol{F}$. Choosing $a$-data and $\chi$-data globally, the argument in the proof of \cite[Lemma 7.3.A]{KS99} shows that the adelic versions of $\Delta_I^\tx{new}$, $\Delta_{II}$ and $\Delta_{IV}$ are all equal to $1$. We thus have to show that the adelic version of our factor $\Delta_{III}(\gamma_0,\delta_0)$ is equal to $\<\tx{obs}(\delta_0),\kappa_0\>$. For this we must recall the definitions of $\tx{obs}(\delta_0)$ and of $\kappa_0$. The construction of $\tx{obs}(\delta_0)$ is the subject of \cite[\S6.3]{KS99}. By assumption there exists $\delta^* \in S^*(\ol{\A})$ and $g \in G_\tx{sc}(\ol{\A})$ such that the image of $\delta^*$ in $S^*_{\theta^*}$ is equal to the image of $\gamma$ under the admissible isomorphism $S' \rw S^*_{\theta^*}$, and such that $\delta^*=g\xi^{-1}(\delta)\theta^*(g^{-1})$. As in the local case, one defines $v_{KS}(\sigma)=gu(\sigma)\sigma(g^{-1})$. Recall that $u \in C^1(\Gamma,G^*_\tx{sc})$ lifts $\xi^{-1}\sigma(\xi) \in Z^1(\Gamma,G^*_\tx{ad})$. Then
\[ (v_{KS}(\sigma)^{-1},\delta^*) \in Z^1\left(\Gamma,\frac{S^*_\tx{sc}(\A)}{S^*_\tx{sc}(F)} \stackrel{1-\theta^*}{\lrw}\frac{S^*(\A)}{S^*(F)}\right). \]
The class of this cocycle is called $\tx{obs}(\delta_0)$. It is seen to belong to the subgroup $H^1(\A/F,S^*_\tx{sc} \rw V)$ of $H^1(\A/F,S^*_\tx{sc} \rw S^*)$, where $V$ is the kernel of the natural projection $S^* \rw S^*_{\theta^*}$.

Now let $K/F$ be a large Galois extension for which $g \in G^*_\tx{sc}(\A_K)$ and $z$ has a representative in $Z^1_\tx{alg}(\mc{E}(K/F),G^*(K))$. Consider $v(e)=gz(e)\sigma_e(g^{-1})$ for $e \in \mc{E}(K/F)$. This is an element of $S^*(\A_K)$ and for any $d \in \mb{D}_{K/F}$, we have $v(de)=gz(d)z(e)\sigma_e(g^{-1})$ with $z(d) \in Z(G^*)(K)$. It follows that when valued in $S^*(\A_K)/S^*(K)$, the function $v$ is invariant under $\mb{D}_{K/F}$ and hence descends to $\Gamma_{K/F}$. We obtain $v \in Z^1(\Gamma,S^*(\ol{\A})/S^*(\ol{F}))$. Moreover, the argument used in the proof of Proposition \ref{pro:tf} shows that $v$ is the image of $v_{KS}$ under the natural map $Z^1(\Gamma,S^*_\tx{sc}(\ol{\A})/S^*_\tx{sc}(\ol{F})) \rw Z^1(\Gamma,S^*(\A)/S^*(\ol{F}))$. We conclude that the image of $(v_{KS}(\sigma)^{-1},\delta^*)$ in $Z^1(\A/F,S^* \rw S^*)$ is equal to $(v(e)^{-1},\delta^*)$.

We now turn to the element $\kappa_0$. Choose an admissible embedding $\hat S' \rw \hat G^\mf{e}$. Using global $\chi$-data, extend it to an $L$-embedding $^LS' \rw {^LG^\mf{e}}$ and compose it with $\eta^\mf{e}$ to obtain an $L$-embedding $f: {^LS'} \rw {^LG}$. Using the same $\chi$-data we obtain another $L$-embedding $g: {^LS'} \rw {^LG}$. We may arrange that the two $L$-embeddings coincide on $\hat S'$ and take it into $\hat T$. Then we have $f(x \rtimes w) = g(a_S(w)\cdot x \rtimes w)$ for any $x \rtimes w \in \hat S' \rtimes W_F = {^LS'}$. The properties of the endoscopic triple $\mf{e}$ imply the equation
\[ \tx{Int}(s)\circ {^L\theta}\circ f = f \]
from which we get $(1-\hat\theta)(a_S^{-1})=\partial s$. This means that $(a_S^{-1},s) \in Z^1(W_F,\hat S^* \stackrel{1-\hat\theta}{\lrw} \hat S^*)$. The image of the class of $(a_S^{-1},s)$ in the group $H^1(W_F,\hat V \rw \hat S^*_\tx{ad})$ under the natural map $[\hat S^* \rw \hat S^*] \rw [\hat V \rw \hat S^*_\tx{ad}]$ is the element $\kappa_0$. We thus see that
\[ \<\tx{obs}(\delta_0),\kappa_0\> = \<(v(\sigma)^{-1},\delta^*),(a_S^{-1},s)\>. \]
By construction, the image of $(a_S^{-1},s)$ under the restriction map $H^1(W_F,\hat S^* \rw \hat S^*) \rw H^1(W_{F_v},\hat S^* \rw \hat S^*)$ for each place $v$ of $F$ is equal to the element $(a_S^{-1},s)$ used in the construction of the local factor $\Delta_{III}$. At the same time, the element $(v(\sigma)^{-1},\delta^*) \in H^1(\A/F,S^* \rw S^*)$ is equal to the sum over all places of the images in this group of the elements $(v(e)^{-1},\delta^*) \in B(F_v,S^* \rw S^*)$ used in the construction of the local factor $\Delta_{III}$. We conclude that
\[ \<(v(\sigma)^{-1},\delta^*),(a_S^{-1},s)\> = \prod_v \Delta_{III}[\mf{e},\xi_v,z_v](\gamma_{0,v},\delta_{0,v}). \]
\end{proof}

We will now return to the case when $F$ is a local field and gather some properties of the transfer factor $\Delta[\mf{w},\xi,z]$ under the assumption $\theta=1$. First, we observe that the construction of the term $\Delta_{III}$ simplifies. The map $S^*(K) \stackrel{1-\theta*}{\lrw} S^*(K)$ is the trivial map, so $H^1_\tx{bsc}(\mc{E}(K/F),S^*(K) \rw S^*(K)) \cong H^1_\tx{bsc}(\mc{E}(K/F),S^*(K)) \times H^0(\mc{E}(K/F),S^*(K))$. The second factor is equal to $H^0(\Gamma_{K/F},S^*(K))=S^*(F)$. In the same way, $H^1(W_F,\hat S \rw \hat S) \cong H^1(W_F,\hat S) \times H^0(\Gamma,\hat S)$. Finally, the pairing between these groups also breaks into the product of the Langlands duality pairing between $S^*(F)$ and $H^1(W_F,\hat S)$ and the Kottwitz-pairing between $H^1_\tx{bsc}(\mc{E}(K/F))$ and $\hat S^\Gamma$. Finally, the element $\delta^* \in S^*$ used in the construction of $\Delta_{III}$ now belongs to $S^*(F)$. With this, one sees that the factor $\Delta[\mf{e},\xi,z](\gamma,\delta)$ is given by
\begin{equation} \label{eq:tfuntwist} \epsilon(\frac{1}{2},V,\psi_F)\Delta_I^\tx{new}(\gamma,\delta)^{-1} \Delta_{II}(\gamma,\delta) \<v(e),s^\mf{e}\>_{Kot} \<a_S,\delta^*\>_{Lan}\Delta_{IV}(\gamma,\delta). \end{equation}

\begin{lem} \label{lem:tflevi} Let $M^* \subset G^*$ be a standard Levi subgroup. Assume that $\xi(M) \subset G$ is a Levi subgroup defined over $F$, so in particular $z \in B(F,M^*)_{G^*-\tx{bsc}}$. Let $\mf{e}_M=(M^\mf{e},s^\mf{e},\xi^\mf{e})$ be an endoscopic triple for $M^*$ and let $\mf{e}=(G^\mf{e},s^\mf{e},\xi^\mf{e})$ be the corresponding endoscopic triple for $G^*$, so that $M^\mf{e} \subset G^\mf{e}$ is a Levi subgroup.
\begin{enumerate}
\item For each $G^*$-regular related pair of elements $\gamma \in M^\mf{e}(F)$ and $\delta \in M(F)$ we have
\[ \Delta[\mf{e}_M,\xi,z](\gamma,\delta)=\Delta[\mf{e},\xi,z](\gamma,\delta)\cdot \left(\frac{|D^{G^\mf{e}}_{M^\mf{e}}(\gamma)|}{|D^G_M(\delta)|}\right)^\frac{1}{2}. \]
\item Let $f \in \mc{H}(G)$ and $f^\mf{e} \in \mc{H}(G^\mf{e})$ have $\Delta[\mf{e},\xi,z]$-matching orbital integrals. Then $f_M$ and $f^\mf{e}_{M^\mf{e}}$ have $\Delta[\mf{e}_M,\xi,z]$-matching orbital integrals.
\end{enumerate}
\end{lem}
\begin{proof}
We begin with the first statement. We will compare the two transfer factors using the expression \eqref{eq:tfuntwist}. In order for the individual factors to make sense, we need to fix $a$-data and $\chi$-data. Since $S^*$ is a maximal torus of $M^*$, no root of $S^*$ outside of $M^*$ can be symmetric. We choose the $a$-data and $\chi$-data so that it is trivial on the roots of $S^*$ outside of $M^*$. We first observe that the virtual representation $V$ is the same for the pairs $(G^\mf{e},G^*)$ and $(M^\mf{e},M^*)$, because $G^*$ and $M^*$ share the same minimal Levi subgroup, and so do $G^\mf{e}$ and $M^\mf{e}$. The terms $\Delta_I^\tx{new}$ are also equal. The definition of this term is given in \cite[\S3.4]{KS12} and involves besides $a$-data also an element $h \in G^*$ conjugating $S^*$ to the fixed minimal Levi $T^*$. We are free to choose $h \in M^*$ with this property. Then formula \cite[(2.1.4)]{KS12}, once executed within $M^*$ and once within $G^*$, gives the same element of $H^1(\Gamma,T^*)$, due to our choice of $a$-data and the fact that the pinning of $M^*$ is inherited from $G^*$.

The terms $\Delta_{II}$ are equal due to our choice of $a$-data and $\chi$-data. The terms $\<v(e),s^\mf{e}\>$ are equal because the element $g \in G^*$ with $\delta=\xi(g^{-1}\delta^*g)$ that defines $v(e)$ can be chosen within $M^*$. The terms $\<a_S,\delta^*\>$ are equal because the 1-cocycle $a_S$ is constructed for $G$ and $M$ using the same $\chi$-data (in particular it is trivial for roots of $G^*$ outside $M^*$). Finally the ratio of the terms $\Delta_{IV}$ coincides with the ration of the Weyl discriminants in the statement of the Lemma.

We now proceed to the second statement, which essentially follows from the first. Namely, let $\gamma \in M^\mf{e}(F)$ and $\delta \in M(F)$ be $G^*$-strongly regular, which we may assume as the reduction to $M^*$-regular elements follows by continuity. Then we have the basic identity for parabolic descent of orbital integrals
\[ O^{M}_\delta(f_{M}) = |D^G_{M}(\delta)|^\frac{1}{2}O^{G}_\delta(f). \]
The same identity holds on the side of $G^\mf{e}$ for usual orbital integrals. Since the map $H^1(\Gamma,M^\mf{e}) \rw H^1(\Gamma,G^\mf{e})$ is injective, the set of $M^\mf{e}(F)$-conjugacy classes within the $M^\mf{e}$-stable class of $\gamma$ is in bijection with the set of $G^\mf{e}(F)$-conjugacy classed within the $G^\mf{e}$-stable class of $\gamma$. This implies that the analog of the above identity holds for the stable orbital integrals of $f^\mf{e}_{M^\mf{e}}$ and $f^\mf{e}$, that is
\[ SO^{M^\mf{e}}_\gamma(f^\mf{e}_{M^\mf{e}}) = |D^{G^\mf{e}}_{M^\mf{e}}(\gamma)|^\frac{1}{2}O^{G^\mf{e}}_\gamma(f^\mf{e}). \]
These two equations together with the first statement of the lemma now imply the second.
\end{proof}

\subsubsection{Simple descent for twisted endoscopy} \label{sec:prelsimp}

This section contains a simple version of descent for twisted endoscopy that will be used in the normalization of intertwining operators. Let $G^*$ be a quasi-split connected reductive group defined over a local field $F$ and let $(\xi,z) : G^* \rw G$ be an extended pure inner twist. We fix a natural number $n$ and define $\tilde G^* = G^* \times \dots \times G^*$, where we have taken $n$ copies of $G^*$. We define the group $\tilde G$ in the same way and let $\tilde\xi = (\xi,\dots,\xi)$ and $\tilde z=(z,\dots,z)$. Then $(\tilde\xi,\tilde z): \tilde G^* \rw \tilde G$ is an extended pure inner twist.

We fix a pinning $(T^*,B^*,\{X_\alpha\})$ of $G^*$ and let $\theta^*$ be an automorphism of $G^*$ that preserves the pinning. We assume that $\theta^*(z)=z$ and then set $\theta=\xi\circ\theta^*\circ\xi^{-1}$, which is an automorphism of $G$ defined over $F$. Define the automorphism $\tilde\theta^*$ of $\tilde G^*$ by $\tilde\theta^*(g_1,\dots,g_n)=(\theta^*(g_n),g_1,\dots,g_{n-1})$. The group $\tilde G^*$ inherits a pinning $(\tilde T^*,\tilde B^*,\{X_{\tilde\alpha}\})$ from $G^*$. More precisely, $\tilde T^*$ and $\tilde B^*$ are obtained by taking the product of $n$ copies of $T^*$ and $B^*$, respectively. Any $\tilde\alpha \in R(\tilde T^*,\tilde G^*) = \sqcup_{i=1}^n R(T^*,G^*)$ can be identified with a pair $(i,\alpha)$, where $1 \leq i \leq n$ and $\alpha \in R(T^*,G^*)$, and then $X_{\tilde\alpha}$ is equal to $(0,\dots,0,X_\alpha,0,\dots,0)$, with $X_\alpha$ placed in the $i$-th slot. The automorphism $\tilde\theta^*$ preserves this pinning. We have $\tilde\theta^*(\tilde z)=\tilde z$. Let $\tilde\theta$ be the automorphism of $\tilde G$ given by $\tilde\theta(g_1,\dots,g_n)=(\theta(g_n),g_1,\dots,g_{n-1})$. Then $\tilde\theta$ is defined over $F$ and satisfies $\tilde\theta=\tilde\xi\circ\tilde\theta^*\circ\tilde\xi^{-1}$.

The purpose of this section is to compare twisted endoscopy for $(\tilde G,\tilde\theta)$ with twisted endoscopy for $(G,\theta)$. Fix a Langlands dual group $\hat G$ for $G^*$ with a pinning $(\hat T,\hat B,{\hat X_\alpha})$ dual to the pinning of $G^*$. Let $\hat\theta$ be the automorphism of $\hat G$ dual to $\theta^*$ and preserving the pinning. Thus the action of $\hat\theta$ on $X^*(\hat T)$ is identified with the action of $\theta^*$ on $X_*(T^*)$. We form the dual group $\hat{\tilde G}$ of $\tilde G^*$ by taking $\hat G \times \dots \times \hat G$ with the pinning $(\hat{\tilde T},\hat{\tilde B},\{X_{\hat{\tilde\alpha}}\})$ inherited from $\hat G$. The automorphism $\hat{\tilde\theta}$ dual to $\tilde\theta^*$ is given by $\hat{\tilde\theta}(g_1,\dots,g_n) = (g_2,\dots,g_n,\hat\theta(g_1))$.

We begin our comparison by constructing an endoscopic datum for $(G^*,\theta^*)$ from one for $(\tilde G^*,\tilde\theta^*)$. Let $\mf{\tilde e}=(G^\mf{\tilde e}, \mc{G}^\mf{\tilde e}, s^\mf{\tilde e},\eta^\mf{\tilde e})$ be an endoscopic datum for $(\tilde G^*,\tilde\theta^*)$. Write $s^\mf{\tilde e} = (s_1,\dots,s_n)$ and let $s^\mf{e}= s_1 \cdot s_2 \dots s_n$. It is easy to see that $s^\mf{e} \in \hat G$ is a $\hat\theta$-semi-simple element, we will give the argument further down. A routine computation shows that
\begin{eqnarray} \label{eq:prelsimpiso}
\hat{\tilde G}^{s^\mf{\tilde e}\circ\hat{\tilde\theta}}&\rw&\hat G^{s^\mf{e}\circ\hat\theta} \nonumber\\
(g_1,\dots,g_n)&\mapsto&g_1 \nonumber \\
(g,s_1^{-1}gs_1,(s_2s_1)^{-1}g(s_2s_1),\dots,(s_1\dots s_{n-1})^{-1}g(s_1\dots s_{n-1}))&\mapsfrom&g
\end{eqnarray}
are mutually inverse isomorphisms. Letting $G^\mf{e}=G^\mf{\tilde e}$, $\mc{G}^\mf{e}=\mc{G}^\mf{\tilde e}$ and taking $\eta^\mf{e} : \mc{G}^\mf{e} \rw {^LG}$ to be the composition of $\eta^\mf{\tilde e}$ with the map $^L{\tilde G} \rw {^LG}$ sending $(g_1,\dots,g_n) \rtimes w$ to $g_1 \rtimes w$, we obtain an endoscopic datum $\mf{e}=(G^\mf{e},\mc{G}^\mf{e},s^\mf{e},\eta^\mf{e})$ for $(G^*,\theta^*)$. The routine verification that $\mf{e}$ satisfies the axioms of an endoscopic datum is left to the reader.

Next, we compare twisted conjugacy in $\tilde G$ and $G$. Consider the map $\Phi^* : \tilde G^* \rw G^*$ given by $\Phi^*(g_1,\dots,g_n)=g_n \cdot g_{n-1} \dots g_1$. It is clear that $\Phi^*$ is a map of varieties that respects the $F$-structure. One checks furthermore that for any $\tilde h = (h_1,\dots,h_n) \in \tilde G^*$ we have
\begin{equation} \label{eq:prelsimpconj}
\Phi^*(\tilde h\cdot \tilde g\cdot \tilde\theta^*(\tilde h^{-1})) = h_n \cdot \Phi^*(\tilde g) \cdot \theta^*(h_n^{-1}).
\end{equation}
We can also define the map $\Phi : \tilde G \rw G$ by the same formula and obtain the same properties. Note that
\begin{equation} \label{eq:prelsimp1} \Phi\circ\tilde\xi = \xi\circ\Phi^*. \end{equation}
Furthermore, note that $\Phi$ defines a bijection between the $\tilde\theta$-twisted conjugacy classes of $\tilde G$ and the $\theta$-twisted conjugacy classes of $G$, as well as between the $\tilde\theta$-twisted conjugacy classes of $\tilde G(F)$ and the $\theta$-twisted conjugacy classes of $G(F)$. We will prove the second point, the proof of the first being analogous. First, according to \eqref{eq:prelsimpconj} the map
\[ \Phi : \{ \tilde\theta-\textrm{twisted classes in } \tilde G(F)\} \quad \rw\quad \{\theta-\textrm{twisted classes in } G(F)\}  \]
is well-defined and surjective. To show injectivity, take $\tilde g,\tilde h \in \tilde G(F)$ and assume that $\Phi(\tilde h) = x\Phi(\tilde g)\theta(x^{-1})$ for some $x \in G(F)$. Replacing $\tilde g$ by its $\tilde\theta$-conjugate $(1,\dots,1,x)\tilde g\tilde\theta((1,\dots,1,x)^{-1})$ we may assume that $\Phi(\tilde g)=\Phi(\tilde h)$. Next we replace $\tilde g=(g_1,\dots,g_n)$ by $\tilde x^{-1}\tilde g\tilde\theta(\tilde x)$, where $\tilde x=(g_1,(g_2g_1),\dots,(g_{n-1}\dots g_2g_1),1)$. This does not change $\Phi(\tilde g)$, according to \eqref{eq:prelsimpconj}, but allows us to assume $\tilde g = (1,\dots,1,g_n)$. We do the same with $\tilde h$. But then $\Phi(\tilde g)=\Phi(\tilde h)$ implies $\tilde g = \tilde h$ and this completes the proof of claimed bijectivity.

The maps $\Phi^*$ and $\Phi$ preserve the notion of twisted semi-simplicity. Indeed, let $\tilde g = (g_1,\dots,g_n) \in \tilde G$. A maximal torus $\tilde S \subset \tilde G$ is of the form $\tilde S = S_1 \times \dots \times S_n$ for maximal tori $S_i \subset G$. If $\tilde S$ is preserved by $\tx{Ad}(\tilde g)\circ\tilde\theta$, then $S_n$ is preserved by $\tx{Ad}(\Phi(\tilde g))\circ\theta$. Conversely, if $S_n$ is preserved by $\tx{Ad}(\Phi(\tilde g))\circ\theta$, then $\tilde S = S_1 \times \dots \times S_n$ with $S_i = g_i S_{i-1}g_i^{-1}$ is preserved by $\tx{Ad}(\tilde g)\circ\tilde\theta$. The same argument works for Borel subgroups. We conclude that $\tilde g \in \tilde G$ is $\tilde\theta$-semi-simple if and only if $\Phi(\tilde g)$ is $\theta$-semi-simple. We will see below that $\Phi^*$ and $\Phi$ also preserve the notions of twisted strong regularity. Note that the analogous argument applied to $s^\mf{\tilde e} \in \hat{\tilde G}$ shows that $s^\mf{e} \in \hat G$ is $\hat\theta$-semi-simple.

We now study the notion of norms and transfer factors. These notions are unaffected if we replace the endoscopic datum $\mf{\tilde e}$ by an equivalent one. It will be convenient to do so and to thereby assume that $s^\mf{\tilde e} \in \hat{\tilde T}$. Then $s_1,\dots,s_n \in \hat T$ and $s^\mf{e} \in \hat T$. We claim that $\gamma \in G^\mf{\tilde e}$ is a norm of $\tilde\delta^* \in \tilde G^*$ (or of $\tilde\delta \in \tilde G$) if and only if $\gamma$, seen now as an element of $G^\mf{e}$, is a norm of $\delta^* = \Phi^*(\tilde\delta^*)$ (or of $\delta = \Phi(\tilde\delta)$). To see this, we note that the following two commutative diagrams are dual to each other
\begin{equation} \label{eq:prelsimp2} \xymatrix{
[\hat{\tilde T}]^{s^\mf{\tilde e}\hat{\tilde\theta}}\ar@{^(->}[r]&\hat{\tilde T}&&\tilde T^*_{\tilde\theta^*}\ar[d]^\cong&\tilde T^*\ar[l]\ar[d]^m\\
[\hat{T}]^{s^\mf{e}\hat{\theta}}\ar[u]^\cong\ar@{^(->}[r]&\hat T\ar[u]^\Delta&&T^*_{\theta^*}&T^*\ar[l]\\
} \end{equation}
Here the left-most vertical isomorphism comes from the isomorphism \eqref{eq:prelsimpiso}, the map $\Delta$ sends $t \in \hat T$ to $(t,\dots,t) \in \hat{\tilde T}$, and  $m$ is the multiplication map. This proves the claim for elements $\tilde g^* \in \tilde T^*$
and the general case follows from \eqref{eq:prelsimp1}, \eqref{eq:prelsimpconj}, and the fact that $\Phi$ and $\Phi^*$ induce bijections on the level of twisted conjugacy classes.

We now assume that $\gamma \in G^\mf{\tilde e}(F)$ is a norm of $\tilde\delta \in \tilde G(F)$ and set $\delta=\Phi(\tilde\delta)$. For simplicity we assume that there exists an isomorphism $^LG^\mf{e} \rw \mc{G}^\mf{e}$ and upgrade the endoscopic data $\mf{e}$ and $\mf{\tilde e}$ to endoscopic triples using this isomorphism, without change in notation. We obtain the transfer factors $\Delta[\mf{\tilde e},\tilde\xi,\tilde z]$ and $\Delta[\mf{e},\xi,z]$ as in Section \ref{sec:normtfs} and claim
\begin{equation} \label{eq:prelsimptf} \Delta[\mf{\tilde e},\tilde\xi,\tilde z](\gamma,\tilde\delta) = \Delta[\mf{e},\xi,z](\gamma,\delta). \end{equation}
Before we can begin the proof of this equality, we must choose various objects and fix notation. Let $S^\mf{e} \subset G^\mf{e}$ be the centralizer of $\gamma$. There exists a maximal torus $\tilde S^* \subset \tilde G^*$ defined over $F$ and a Borel subgroup $\tilde C^* \subset G^*$ defined over $\ol{F}$ and containing $\tilde S^*$, such that both $\tilde S^*$ and $\tilde C^*$ are $\tilde\theta^*$ invariant. Furthermore, there exists an admissible isomorphism $S^\mf{e} \rw \tilde S^*_{\tilde\theta^*}$ and an element $\tilde \delta^* \in \tilde S^*$ whose image in $\tilde S^*_{\tilde\theta^*}(F)$ equals the image of $\gamma$ under the admissible isomorphism. Finally, there exists $\tilde g \in \tilde G^*$ such that $\tilde\delta = \tilde g^{-1}\tilde\xi(\tilde\delta^*)\tilde\theta(\tilde g)$. All of this comes from the assumption that $\gamma$ is a norm of $\tilde\delta$.

The maximal torus $\tilde S^*$ is of the form $S^* \times \dots \times S^*$ for some $\theta^*$-stable maximal torus $S^* \subset G^*$. In the same way, $\tilde C^*=C^* \times \dots \times C^*$ for some $\theta^*$-stable Borel subgroup $C^* \subset G^*$ containing $S^*$.
Let $\delta^* = \Phi^*(\tilde\delta^*)$. According to \eqref{eq:prelsimpconj} and \eqref{eq:prelsimp1} we have $\delta=g_n^{-1}\xi(\delta^*)\theta(g_n)$, where we have written $\tilde g = (g_1,\dots,g_n)$.

We further have $R(\tilde S^*,\tilde G^*) = R(S^*,G^*) \sqcup \dots \sqcup R(S^*,G^*)$. We fix $\theta^*$-admissible $a$-data for $R(S^*,G^*)$ \cite[\S2.2]{KS12}. Placing it in each copy of the disjoint union, we obtain $\tilde\theta^*$-admissible $a$-data for $R(\tilde S^*,\tilde G^*)$.

We have the isomorphism
\[ [\tilde S^*]^{\tilde\theta^*} = \{(s,\dots,s)| s \in [S^*]^{\theta^*} \} \cong [S^*]^{\theta^*}. \]
This isomorphism induces a bijection between the sets of restricted roots
\begin{equation} \label{eq:prelsimp3} R_\tx{res}(\tilde S^*,\tilde G^*) \cong R_\tx{res}(S^*,G^*) \end{equation}
and this bijection preserves the type $R_1,R_2,R_3$ of each root \cite[\S1.3]{KS99}. We fix $\chi$-data for these two sets of restricted roots so that it is compatible with this bijection.

Having fixed the necessary objects for the construction of $\Delta[\mf{\tilde e},\tilde\xi,\tilde z](\gamma,\tilde\delta)$ and $\Delta[\mf{e},\xi,z](\gamma,\delta)$, we begin the comparison. Consider first $\Delta_I^\tx{new}$. For this we choose $h \in G^*$ such that $hB^*h^{-1}=C^*$. Then we obtain $t \in Z^1(\Gamma,S^*)$ by formula \cite[(2.1.4)]{KS12}. If we let $\tilde h = (h,\dots,h) \in \tilde G^*$, then $\tilde h\tilde B^*\tilde h^{-1}=\tilde C^*$. Using $\tilde h$ to obtain $\tilde t \in Z^1(\Gamma,\tilde S^*)$ by the same formula, we see that $\tilde t = (t,\dots,t)$. Pairing $\tilde t$ with $s^\mf{\tilde e}=(s_1,\dots,s_n)$ under Tate-Nakayama duality gives the same result as pairing $t$ with $s_1 \dots s_n = s^\mf{e}$. This proves the equality
\[ \Delta_I^\tx{new}[\mf{\tilde e},\tilde\xi,\tilde z](\gamma,\tilde\delta) = \Delta_I^\tx{new}[\mf{e},\xi,z](\gamma,\delta). \]
Next we consider the factors $\Delta_{II}$ and $\Delta_{IV}$. The $\mf{\tilde e}$-versions of these involve terms of the form $[N_{\tilde\theta^*}\tilde\alpha](\tilde\delta^*)$, where $N_{\tilde\theta^*}\tilde\alpha$ is the sum of the members of the $\tilde\theta^*$-orbit of $\tilde\alpha$. But $[N_{\tilde\theta^*}\tilde\alpha](\tilde\delta^*) = \tilde\alpha_\tx{res}(N_{\tilde\theta^*}\tilde\delta^*)$, where now $N_{\tilde\theta^*}\tilde\delta^* \in [\tilde S^*]^{\tilde\theta^*}$ is the product of the members of the $\tilde\theta^*$-orbit of $\tilde\delta^*$, and $\tilde\alpha_\tx{res}$ is the restriction of $\tilde\alpha$ to $[\tilde S^*]^{\tilde\theta^*}$. One sees quickly that
\[ N_{\tilde\theta^*}(\tilde\delta^*) = (N_{\theta^*}(\delta^*),\dots,N_{\theta^*}(\delta^*)) \]
and it thus follows that if $\tilde\alpha_\tx{res}$ corresponds to $\alpha_\tx{res}$ under the bijection \eqref{eq:prelsimp3}, then $\tilde\alpha_\tx{res}(N_{\tilde\theta^*}(\tilde\delta^*)) = \alpha_\tx{res}(N_{\theta^*}(\delta^*))$. Recalling that this bijection preserves the types of restricted roots and using Lemmas 4.3.A and 4.5.A of \cite{KS99} we obtain the equalities
\[ \Delta_{II}[\mf{\tilde e},\tilde\xi,\tilde z](\gamma,\tilde\delta) = \Delta_{II}[\mf{e},\xi,z](\gamma,\delta), \]
and
\[ \Delta_{IV}[\mf{\tilde e},\tilde\xi,\tilde z](\gamma,\tilde\delta) = \Delta_{IV}[\mf{e},\xi,z](\gamma,\delta). \]

Finally, we consider the factor $\Delta_{III}$ constructed in Section \ref{sec:normtfs}. For this, we consider the following two dual commutative diagrams
\[ \xymatrix{
\tilde S^*\ar[d]_{1-\tilde\theta^*}\ar[r]^{p_n}&S^*\ar[d]^{1-\theta^*}&&&&\hat{\tilde S}&\hat S\ar[l]_{i_n}\\
\tilde S^*\ar[r]_m&S^*&&&&\hat{\tilde S}\ar[u]^{1-\hat{\tilde\theta}}&\hat S\ar[l]^\Delta\ar[u]_{1-\hat\theta}
} \]
Here $m$ is the multiplication map, $\Delta$ is the diagonal inclusion, $p_n$ is the projection of the $n$-th coordinate, and $i_n$ is the inclusion into the $n$-th coordinate. The left commutative diagram can be interpreted as a morphism of complexes of tori of length $2$, the complexes being given by the vertical maps, and the morphism being given by the horizontal maps. One checks that this morphism is a quasi-isomorphism. It follows that this morphism induces an isomorphism
\begin{equation} \label{eq:prelsimp5} B(F,\tilde S^* \stackrel{1-\tilde\theta^*}{\lrw} \tilde S^* ) \rw B(F,S^*\stackrel{1-\theta^*}{\lrw}S^*). \end{equation}
The element $\tx{inv}(\gamma,\tilde\delta) \in B(F,\tilde S^* \stackrel{1-\tilde\theta^*}{\lrw} \tilde S^* )$ is the class of $((\tilde g\tilde z(w)w(\tilde g)^{-1})^{-1},\tilde\delta^*)$ and this class is translated by the above isomorphism to $((g_nz(w)w(g_n)^{-1})^{-1},\delta^*)$, which is in fact a representative of $\tx{inv}(\gamma,\delta)$.

In order to prove the equality of the $\Delta_{III}$ factors for $\mf{\tilde e}$ and $\mf{e}$ it will be enough to show the following. Let $\tilde a_S : W_F \rw \hat{\tilde S}$ and $a_S : W_F \rw \hat S$ be the 1-cochains constructed at the end of Section \ref{sec:normtfs}. Then the map
\begin{equation} \label{eq:prelsimp4} H^1(W_F,\hat S \stackrel{1-\hat\theta}{\lrw} \hat S) \rw H^1(W_F,\hat{\tilde S} \stackrel{1-\hat{\tilde\theta}}{\lrw} \hat{\tilde S}) \end{equation}
induced by the right commutative diagram above, sends the class of $(a_S^{-1},s^\mf{e})$ to the class of $(\tilde a_S^{-1},s^\mf{\tilde e})$. We will in fact show that this is true already at the level of cocycles. For this we consider the following diagram
\[ \xymatrix{
&{^L\tilde G}&\\
{^LH}\ar[ur]^{\eta^\mf{\tilde e}}\ar[dr]_{\eta^\mf{e}}&&{^LG^1}\ar[lu]\ar[ld]\\
&{^LG}&\\
{^LS^H}\ar[uu]\ar@{=}[rr]&&^LS_\theta\ar[uu]
} \]
Here $\hat G^1$ is the group of fixed points of $\hat\theta$ in $\hat G$, or equivalently that group of fixed points of $\hat{\tilde\theta}$ in $\hat{\tilde G}$. The diagonal arrows on the right are given by
\begin{eqnarray*}
^LG^1&=&\{g \rtimes w| g \in \hat G^{\hat\theta}, w \in W_F \} \subset {^LG}\\
&=&\{(g,\dots,g) \rtimes w| g \in \hat G^{\hat\theta},w \in W_F \} \subset {^L\tilde G}.
\end{eqnarray*}
We recall that the diagonal arrows on the left are given by the condition that $\eta^\mf{e}(h \rtimes w)=g \rtimes w \in {^LG}$ if and only if
\[ \eta^\mf{\tilde e}(h \rtimes w) = (g,s_1^{-1}\cdot g\cdot ws_1,(s_1s_2)^{-1} \cdot g\cdot w(s_1s_2),\dots), \]
where $s^\mf{\tilde e}=(s_1,\dots,s_n)$. Writing $\sigma_S(w)$ for the action of $W_F$ on $\hat S$, we then compute
\[ \tilde a_S(w) = (a_S(w),(s_1^{-1}\sigma_S(w)s_1)\cdot a_S(w),((s_1s_2)^{-1}\sigma_S(w)(s_1s_2))\cdot a_S(w),\dots). \]
The element $(\tilde a_S^{-1},s^\mf{\tilde e})$ of $Z^1(W_F,\hat{\tilde S} \stackrel{1-\hat{\tilde\theta}}{\lrw} \hat{\tilde S})$ is thus given by
\[ [(a_S(w),(s_1^{-1}\sigma_S(w)s_1)\cdot a_S(w),((s_1s_2)^{-1}\sigma_S(w)(s_1s_2))\cdot a_S(w),\dots)^{-1},(s_1,\dots,s_n)]. \]
On the other hand, the image of the element $(a_S^{-1},s^\mf{e})$ under \eqref{eq:prelsimp4} is given by
\[ [(a_S(w),a_S(w),\dots,a_S(w))^{-1},(1,1,\dots,1,s_1\cdot s_2 \dots s_n)]. \]
These two cocycles are cohomologous via the coboundary $(1,s_1,s_1s_2,\dots) \in \hat{\tilde S}$. We have thus proved the following.
\begin{pro} \label{pro:simptfs} Given an endoscopic triple $\mf{\tilde e}$ for $(\tilde G,\tilde\theta)$ we obtain an endoscopic triple $\mf{e}$ for $(G,\theta)$ with the property $G^\mf{\tilde e}=G^\mf{e}$. Then $\gamma \in G^\mf{\tilde e}(F)$ is a norm of $\tilde\delta \in \tilde G(F)$ if and only if $\gamma \in G^\mf{e}$ is a norm of $\delta = \Phi(\tilde\delta) \in G(F)$. Furthermore,
\[ \Delta[\mf{\tilde e},\tilde\xi,\tilde z](\gamma,\tilde\delta) = \Delta[\mf{e},\xi,z](\gamma,\delta). \]
\end{pro}

We will now compare twisted orbital integrals. Let $\tilde\delta \in \tilde G(F)$ be  $\tilde\theta$-semi-simple and let $\delta = \Phi(\tilde\delta) \in G(F)$. Write $I_{\tilde \delta}$ for the centralizer $\tx{Cent}(\tilde \delta\tilde\theta,\tilde G)$ and use the analogous notation $I_\delta$ for $\delta$. Just as in \eqref{eq:prelsimpiso} we have the mutually inverse isomorphisms
\begin{eqnarray} \label{eq:prelsimpiso2}
I_{\tilde\delta}&\rw&I_\delta\\
(h_1,\dots,h_n)&\mapsto&h_n \nonumber \\
(\delta_1\theta(h_n)\delta_1^{-1},\delta_2\delta_1\theta(h_n)(\delta_2\delta_1)^{-1},\dots,h_n)&\mapsfrom&h_n. \nonumber
\end{eqnarray}
In particular we see that $\delta$ is strongly $\theta$-regular if and only if $\tilde\delta$ is strongly $\tilde\theta$-regular

We fix a Haar measure $dg$ on $G(F)$ and endow $\tilde G(F)$ with the product Haar measure $d\tilde g = dg \times \dots \times dg$. For two functions $f_1,f_2$ on $G(F)$, we define their convolution as
\[ (f_1 * f_2)(g) = \int f_1(x)f_2(gx^{-1})dx. \]
Given functions $f_1,\dots,f_n$ on $G(F)$ we then have
\[ (f_1* \dots * f_n)(x) = \int f_1(x_1)f_2(x_2x_1^{-1})\dots f_{n-1}(x_{n-1}x_{n-2}^{-1})f_n(gx_{n-1}^{-1}) dx_1\dots dx_{n-1}. \]
Now let $\tilde f = f_1 \otimes \dots \otimes f_n$ be a function on $\tilde G(F)$ and consider the orbital integral
\[ \int_{\tilde G(F)/I_{\tilde\delta}} \tilde f(\tilde h\tilde\delta\tilde\theta(\tilde h^{-1})) d\tilde g/di \]
for a fixed Haar measure $di$ on $I_{\tilde\delta}(F)$. A direct computation shows that this orbital integral is equal to the orbital integral
\[ \int_{G(F)/I_\delta(F)} (f_1*\dots*f_n)(h_n\delta\theta(h_n^{-1})) dh_n/di \]
where $di$ has been transported to $I_\delta(F)$ via the isomorphism \eqref{eq:prelsimpiso2}. Combining this equation with Proposition \ref{pro:simptfs} we arrive at
\begin{cor} If the functions $f' \in \mc{H}(G^\mf{\tilde e})$ and $\tilde f = f_1 \times \dots \times f_n \in \mc{H}(\tilde G(F))$ have $\Delta[\mf{\tilde e},\tilde\xi,\tilde z]$-matching orbital integrals, then the functions $f' \in \mc{H}(G^\mf{e})$ and $f_1 * \dots *f_n \in \mc{H}(G(F))$ have $\Delta[\mf{e},\xi,z]$-matching orbital integrals.
\end{cor}

Our final task in this section is to compare twisted characters. An irreducible admissible representation $\tilde\pi$ of $\tilde G(F)$ is $\tilde\theta$-stable if and only if it is of the form $\tilde\pi = \pi \otimes \dots \otimes \pi$ for an irreducible admissible $\theta$-stable representation $\pi$ of $G(F)$. We write $V_\pi$ for the space on which $\pi$ acts and let $\tilde\pi$ act on $V_\pi^{\otimes n}$. Fix an isomorphism $s : (\pi\circ\theta^{-1},V_\pi) \rw (\pi,V_\pi)$ and define $\tilde s : (\tilde\pi \circ\tilde\theta^{-1},V_\pi^{\otimes n}) \rw (\tilde\pi,V_\pi^{\otimes n})$ by $\tilde s(v_1\otimes \dots \otimes v_n) = (s(v_n) \otimes v_1 \otimes \dots \otimes v_{n-1})$.

\begin{lem} \label{lem:traceprod} Let $f_1,\dots,f_n$ be smooth compactly supported functions of $G(F)$ and let $\tilde f = f_1 \otimes\dots\otimes f_n$. Then \[ \tx{tr}(\tilde\pi(\tilde f) \circ \tilde s) = \tx{tr}(\pi(f_1 * \dots *f_n)\circ s). \]
\end{lem}
To prove this, let $\phi_i = \pi_i(f_i) \in \tx{End}(V_\pi)$. Then $\tilde\pi(\tilde f) = \phi_1 \otimes \dots \otimes \phi_n \in \tx{End}(V_\pi^{\otimes n})$. On the other hand, a direct computation reveals $\pi(f_1* \dots *f_n)=\phi_n \circ \dots \circ\phi_1 \in \tx{End}(V_\pi)$. We are thus showing that given any Hilbert space $V$, any automorphism $s \in \tx{Aut}(V)$ and any trace class operators $\phi_1,\dots,\phi_n \in \tx{End}(V)$, the equality
\[ \tx{tr}((\phi_1 \otimes \dots \otimes \phi_n) \circ \tilde s|V^{\otimes n}) = \tx{tr}(\phi_n \circ\dots\circ\phi_1\circ s|V) \]
holds, where $\tilde s \in \tx{Aut}(V^{\otimes n})$ is defined as above. We first reduce to the case where each $\phi_i$ has finite rank, because a general trace class operator is a limit of finite rank operators. Now that $\phi_i$ is of finite rank, it is given by a finite sum $\sum_j \lambda_i^j \otimes v_i^j$ for $\lambda_i^j \in \tx{Hom}(V,\C)$ and $v_i^j \in V$. Since the equation we are proving is $n$-linear in $(\phi_1,\dots,\phi_n)$, we reduce to the case $\phi_i = \lambda_i \otimes v_i$. But then $\phi_n \circ \dots \circ\phi_1 \circ s = \lambda_2(v_1)\lambda_3(v_2)\dots\lambda_n(v_{n-1}) \cdot (\lambda_1\circ s \otimes v_n)$ and its trace is equal to $\lambda_2(v_1)\dots\lambda_n(v_{n-1}) \lambda_1(s(v_n))$. On the other hand, $(\phi_1 \otimes \dots \otimes \phi_n) \circ \tilde s = (\lambda_2 \otimes \dots \otimes \lambda_n \otimes \lambda_1 \circ s) \otimes (v_1 \otimes \dots \otimes  v_n)$ and its trace is again $\lambda_2(v_1)\dots \lambda_n(v_{n-1}) \lambda_1(s(v_n))$. This completes the proof of the lemma.

\subsubsection{Endoscopic data in the case of unitary groups} \label{subsub:endo-unitary}

  Here we will explicate the sets equivalence classes $\cE_\el(G^*)$ and $\cE^\tx{w}_\el(G^*)$ of elliptic endoscopic triples (\S\ref{subsub:endoscopic-triples}) in three cases below, cf. \cite[4.6,4.7]{Rog90}.
 \begin{itemize}
   \item $G^*=U_{E/F}(N)$ and $E/F$ is an extension of fields,
   \item $G^*=U_{E/F}(N)$ and $E=F\times F$,
   \item $G^*=\tilde G_{E/F}(N)$ and $E/F$ is an extension of fields. %
 \end{itemize}

We begin with a short discussion on characters. Let $E/F$ be a quadratic extension of local or global fields of characteristic zero. Set $C_E:=E^\times$ and $C_F:=F^\times$ in the local case and $C_E:=\A_E^\times/E^\times$ and $C_F:=\A_F^\times/F^\times$ in the global case. Note that $C_F\subset C_E$ and that there is a norm map $N_{E/F}:C_E\ra C_F$. We shall denote by $\omega_{E/F}:C_F/N_{E/F}C_E\ra \{\pm 1\}$
the character associated to the extension $E/F$ via class field theory.
Define the set of continuous characters
\[
	\mathcal{Z}_E = \{ \chi :  C_E \rightarrow \C^\times \textrm{ unitary}, \chi \circ c = \chi^{-1} \}.
\]
We have a partition
\[
	\mathcal{Z}_E = \mathcal{Z}^+_E \sqcup \mathcal{Z}_E^-
\]
where
\[
	\mathcal{Z}_E^+ = \{ \chi \in \mathcal{Z}_E : \chi|_{C_F} = 1 \} \textrm{ and }
	\mathcal{Z}_E^- = \{ \chi \in \mathcal{Z}_E : \chi|_{C_F} = \omega_{E/F} \}.
\]
Let  $\chi_{\kappa} \in \mathcal{Z}_E^\kappa$ for $\kappa = \pm1$.
We shall often implicitly view these characters as characters of the Weil group $W_E$ via class field theory.

We now discuss the endoscopic groups of $G^*=U_{E/F}(N)$, where $E$ is a field. Then $\cE^\tx{w}_{\el}(G^*)$ is identified with the following set of triples:
  $$(G^\fke,s^\fke,\eta^\fke)=\left(U_{E/F}(N_1)\times U_{E/F}(N_2),\left(
                                       \begin{array}{cc}
                                         I_{N_1} & 0 \\
                                         0 & -I_{N_2} \\
                                       \end{array}
                                     \right)
  ,\zeta_{(\chi_1,\chi_2)}\right),$$
  as the pair $(N_1,N_2)$ varies subject to the condition that $N_1\ge N_2\ge 0$, $N=N_1+N_2$. For each $(N_1,N_2)$, we choose any $\chi_i\in \cZ_E^{\kappa_i}$ with $\kappa_i=(-1)^{N-N_i}$ for $i=1,2$ and define the $L$-homomorphism

\begin{eqnarray*}
\zeta_{(\chi_1,\chi_2)} : {}^L (U_{E/F}(N_1) \times U_{E/F}(N_2)) &\rightarrow& {}^L U_{E/F}(N)		\\
				(g_1,g_2) \rtimes 1 &\mapsto& \mathrm{diag}(g_1,g_2) \rtimes 1	\\
				(I_{N_1},I_{N_2}) \rtimes w &\mapsto&
							 \mathrm{diag}(\chi_{\kappa_1}(w) I_{N_1}, \chi_{\kappa_2}(w) I_{N_2}) \rtimes w
								\text{ if } w \in W_E	\\
				(I_{N_1},I_{N_2}) \rtimes w_c &\mapsto&	\mathrm{diag}(\kappa_1 \Phi_{N_1},\kappa_2 \Phi_{N_2}) \cdot \Phi_N^{-1} \rtimes w_c,
\end{eqnarray*}
where $W_F = W_E \cup W_E w_c$. We note that the $\widehat{U}_{E/F}(N)$-conjugacy class of $\eta_{\chi_\kappa}$ is independent of the choice of $w_c$.

The choice of $\chi_i$ does not change the isomorphism class of the endoscopic triple. The set $\cE^\tx{w}_\el(G^*)$ is in bijection with the set $\{N_1\ge N_2\ge 0,~N=N_1+N_2\}$. The outer automorphism group $\Out_{G^*}(\fke)$ %
is trivial if $N_1\neq N_2$ and isomorphic to $\Z/2$ if $N_1=N_2$, in which case the nontrivial automorphism swaps $U_{E/F}(N_1)$ and $U_{E/F}(N_2)$. The set $\cE_\el(G^*)$ has a similar description. It is a double cover of $\cE^\tx{w}_{\el}(G^*)$, with distinct elements $(G^\mf{e},\pm s^\mf{e},\eta^\mf{e})$ mapping to the same element in $\cE^\tx{w}_{\el}(G^*)$.

The second case to consider is $G^*=U_{E/F}(N)$ with $E=F\times F$ so that $G^*$ is isomorphic to $\GL(N,F)$. Then it is elementary to verify that $\cE^\tx{w}_{\el}(G^*)=\{(G^*,1,\id)\}$ and $\cE_{\el}(G^*)=\{(G^*,z,\id)|z \in \C^\times\}$.

Finally let $G^*=\tilde G_{E/F}(N)$ where $E$ is a field. We first describe the \emph{simple} endoscopic triples. They are given by $(U_{E/F}(N),1,\eta_\chi)$, where $\eta_\chi$ is the following $L$-embedding
\begin{eqnarray*}
\eta_{\chi_\kappa} : {}^L U_{E/F}(N) &\rightarrow&  {}^L G_{E/F}(N)		\\
				g \rtimes 1 &\mapsto&	(g , J_N{^tg^{-1}}J_N^{-1}) \rtimes 1			\\
				I_N \rtimes w &\mapsto& (\chi_\kappa(w) I_N, \chi^{-1}_\kappa(w) I_N) \rtimes w \textrm{ if } w \in W_E \\
				I_N \rtimes w_c &\mapsto& (I_N, \kappa I_N) \rtimes w_c
 \end{eqnarray*}
if $W_F = W_E \cup W_E w_c$.  We note that the
$\widehat{G}_{E/F}(N)$-conjugacy class of $\eta_{\chi_\kappa}$ is independent of the choice of $w_c$.

We note that when $\kappa = 1$ (resp. $-1$), this $L$-homomorphism is referred to as \emph{base change} (resp. $\emph{twisted base change}$).  If $\chi$ is chosen to be the identity character, then this $L$-homomorphism is referred to as \emph{standard base change}. More generally, we will say that the \emph{parity} of the simple triple $(U_{E/F}(N),1,\eta_\chi)$ is $(-1)^{N-1}\kappa$.

We now turn to the set of elliptic endoscopic triples for $\widehat{G}_{E/F}(N)$. A set of representatives for it is given by
  $$\left(U_{E/F}(N_1)\times U_{E/F}(N_2),\left(
                                       \begin{array}{cc}
                                         I_{N_1} & 0 \\
                                         0 & -I_{N_2} \\
                                       \end{array}
                                     \right)
  ,\eta_{(\chi_1,\chi_2)}\right),$$
  where $(N_1,N_2)$ varies over the set satisfying $N_1\ge N_2\ge 0$ and $N=N_1+N_2$, and for each $(N_1,N_2)$, we choose any $\chi_i\in \cZ_E^{\kappa_i}$, $i=1,2$. The $L$-homomorphism
\[
	\eta_{\underline{\chi}} : {}^L (U_{E/F}(N_1) \times U_{E/F}(N_2)) \rightarrow {}^L  G_{E/F}(N)
\]
is given by composing the product of the $L$-homomorphisms
\[
	\eta_{\kappa_1} \times \eta_{\kappa_2} : {}^L (U_{E/F}(N_1) \times U_{E/F}(N_2)) \rightarrow
	 {}^L  (G_{E/F}(N_1)\times G_{E/F}(N_2))
\]
with the natural diagonal $L$-embedding
\begin{eqnarray*}
	 {}^L  (G_{E/F}(N_1)\times G_{E/F}(N_2)) &\rightarrow& {}^L G_{E/F}(N)	\\
    (g_1,g_2) \times (h_1,h_2) \rtimes w&\mapsto&(\mathrm{diag}(g_1,h_1) \times \mathrm{diag}(g_2,h_2)) \rtimes w.
\end{eqnarray*}
There are generally two choices of $(\kappa_1,\kappa_2)$:
  $$(\kappa_1,\kappa_2)=\left\{ \begin{array}{ll}
                                        (1,-1)~\mbox{or}~(-1,1) & \mbox{if}~ N_1\equiv N_2~\mbox{mod}~2, \\
                                         (1,1)~\mbox{or}~(-1,-1) & \mbox{if}~ N_1 \,\not{\!\!\equiv}\, N_2~\mbox{mod}~2,
                                       \end{array}\right.$$
  The only nontrivial isomorphism occurs between $(U(N_1)\times U(N_1),(1,-1))$ and $(U(N_1)\times U(N_1),(-1,1))$. In other words, $\cE^\tx{w}_{\el}(G^*)$ is in bijection with the set $\{(N_1,N_2,\kappa_1,\kappa_2)\}$, where $N_1\ge N_2\ge 0$, $N=N_1+N_2$, and $(\kappa_1,\kappa_2)$ is as above except that only one of the two choices of $(\kappa_1,\kappa_2)$ is taken if $N_1=N_2$. The set $\cE_{\el}(G^*)$ is again a double cover of $\cE^\tx{w}_{\el}(G^*)$.

  For later discussions it is useful to fix a choice of characters $\chi_+\in \cZ^+_E$ and $\chi_-\in \cZ^-_E$ and also to fix representatives for endoscopic triples such that each character $\chi_i\in \cZ_E^{\kappa_i}$ above is $\chi_+$ if $\kappa_i=1$ and $\chi_-$ if $\kappa_i=-1$.

\subsection{Local parameters}\label{sub:local-param}

  In this subsection we will recall some of the well-known definitions about local parameters and representations to set up some notation.

\subsubsection{L-parameters and A-parameters}\label{subsub:L-param-A-param}

   Let $F$ be a local field. Define the local Langlands group
  $$L_F:=\left\{ \begin{array}{ll} W_F, & \mbox{if~$F$~is~archimedean},\\
W_F\times \SU(2), & \mbox{if~$F$~is~non-archimedean},\end{array} \right.$$
  as a topological group. The $L$-group $^L G=\hat G\rtimes W_F$ is also a topological group with complex topology on $\hat G$ and the usual topology on $W_F$.
  A (local) \textbf{$L$-parameter}, or a Langlands parameter, for a connected reductive group $G$ over $F$ is a continuous homomorphism
  $$\phi:L_F\ra {}^L G$$
  commuting with the canonical projections of $L_F$ and $^L G$ onto $W_F$ such that $\phi$ maps semisimple elements to semisimple elements, cf. \cite[\S8.2]{Bor79}. Two $L$-parameters are considered equivalent if they are conjugate under an element of $\hat G$. Henceforth $\Phi(G)$ (or $\Phi(G,F)$ if the base field is to be emphasized) will denote the set of equivalence classes of $L$-parameters. By abuse of notation we write a parameter to mean an equivalence class thereof. When $F$ is a completion of a global field $\dot F$ at a place $v$, then we often write $\Phi_v(G)$ for $\Phi(G,F)=\Phi(G,\dot F_v)$. We say that $\phi\in \Phi(G)$ is \textbf{bounded} if the image of $\phi$ in ${}^L G$ projects onto a relatively compact subset of $\hat G$ and \textbf{discrete} (or \textbf{square-integrable}) if the image does not lie in any proper parabolic subgroup of $^L G$. Define $\Phi_2(G)$ (resp. $\Phi_{\bdd}(G)$) to be the subset of discrete (resp. bounded) parameters, and put $\Phi_{2,\bdd}(G):=\Phi_2(G)\cap \Phi_{\bdd}(G)$.
  We associate some complex Lie groups to $\phi\in \Phi(G)$. The general formalism of \S\ref{subsub:centlevi} applies to $L=L_F$ (with $L$-groups in the Weil form), yielding the definition of $S_\phi$, which is the centralizer group of $\phi$ in $\hat G$, as well as $\srad_\phi$ and $S^\natural_\phi$. We also define
   $\ol{S}_\phi:=S_\phi/Z(\hat G)^\Gamma$, $\cS_\phi:=\pi_0(\ol{S}_\phi)$, and $\ol{\cS}_\phi:=\pi_0(\ol{S}_\phi)=S_\phi/S^0_\phi Z(\hat G)^\Gamma$.
  In particular $S_\phi=S_\phi(G)$ and $\ol{S}_\phi$ are (possibly disconnected) complex reductive groups \cite{Kot84} and
  since two equivalent parameters differ by an inner automorphism of $\hat G$, we see that an element in each of the groups $S_\phi$, $\srad_\phi$, $S^\natural_\phi$, $\ol S_\phi$, $\cS_\phi$, and $\ol \cS_\phi$ is well-defined up to an inner automorphism of $S_\phi$.

  If $G$ is not quasi-split it is expected that the representations of $G$ are classified by relevant parameters, which were introduced in a little more general setting, cf. Definition \ref{def:Xi-relevant-general}, but are recalled below. Suppose that $(G,\xi)$ is an inner twist of a quasi-split group $G^*$. Then we have an isomorphism ${}^L G \simeq {}^L G^*$ canonical up to an inner automorphism. In particular there is a canonical identification $\Phi(G)=\Phi(G^*)$.\footnote{Arthur's convention \cite[1.3]{Arthur} is a little different in that his $\Phi(G)$ is our $\Phi(G^*)_{(G,\xi)\rel}$. This should not cause confusion as we will never mention $\Phi(G)$ for non quasi-split groups $G$ later.} So a parameter for $G^*$ may be viewed as a parameter for $G$ for any inner twist $(G,\xi)$ (as we only care about the equivalence class thereof after all).
\begin{dfn}
  An $L$-parameter $\phi$ for $G^*$ is $(G,\xi)$-\textbf{relevant} if every Levi subgroup $^L M$ of $^L G$ such that $\phi(L_F)\subset {}^L M$ is relevant, i.e. if such an $^L M$ is a Levi component of a $(G,\xi)$-relevant parabolic subgroup of $^L G$, \S\ref{subsub:transfer-Levi-parabolic}.
\end{dfn}
 We simply call $\phi$ relevant if the inner twist is clear from the context. Obviously the relevance of a parameter is invariant under equivalence. Define $\Phi(G^*)_{(G,\xi)\rel}$ to be the subset of $(G,\xi)$-relevant equivalence classes. We define the sets $\Phi_2(G^*)_{(G,\xi)\rel}$ by intersecting $\Phi_2(G^*)$ with $\Phi(G^*)_{(G,\xi)\rel}$, and similarly $\Phi_{\bdd}(G^*)_{(G,\xi)\rel}$. There is no need to introduce $\Phi_{2,\bdd}(G^*)_{(G,\xi)\rel}$ as the relevance condition is vacuously satisfied for parameters in $\Phi_{2,\bdd}(G^*)$.

  The above parameter sets are in parallel with various sets consisting of representations of $G$. Write $\Pi(G)$ (or $\Pi(G(F))$) for the set of isomorphism classes of irreducible smooth representations of $G(F)$. Again we often omit the expression ``isomorphism classes of'' by abuse of terminology if the context is clear. Let $\Pi_\temp(G)$ (resp. $\Pi_2(G)$) denote the subset of tempered (resp. essentially square-integrable) representations of $G$ in $\Pi(G)$. Similarly as before $\Pi_{2,\temp}(G):=\Pi_2(G)\cap \Pi_\temp(G)$, which therefore consists of square-integrable representations. In addition we introduce $\Pi_{\unit}(G)$,  the subset of unitary representations. Then we have the following horizontal inclusions
  $$\xymatrix{  \Pi_{2,\temp}(G) \ar@{-->>}[d] & \subset &  \Pi_\temp(G) \ar@{-->>}[d] & \subset & \Pi(G) \ar@{-->>}[d] \\
   \Phi_{2,\bdd}(G^*) & \subset & \Phi_{\bdd}(G^*)_{(G,\xi)\rel} & \subset &  \Phi(G^*)_{(G,\xi)\rel}}$$
   in which we have inserted the conjectural local Langlands classification as dotted vertical arrows, which should each be a surjective finite-to-one map. In other words, the hypothetical finite-to-one surjection $\Pi(G)\ra \Phi(G^*)_{(G,\xi)\rel}$ giving the Langlands classification %
   should carry $\Pi_{2,\temp}(G)$ and $\Pi_\temp(G)$ surjectively onto $\Phi_{2,\bdd}(G^*)$ and $\Phi_{\bdd}(G^*)_{(G,\xi)\rel}$. The construction of the general surjection $\Pi(G)\ra \Phi(G^*)_{(G,\xi)\rel}$ is reduced by the Langlands quotient construction to the latter surjection.

   In this paper (cf. Theorem \ref{thm:locclass-single} below) we construct a map $\Pi_\temp(G)\ra \Phi_{\bdd}(G^*)_{(G,\xi)\rel}$ with such properties, uniquely determined by a family of character identities naturally occurring in the theory of endoscopy. Moreover the finite fiber over each $\phi\in \Phi(G^*)_{(G,\xi)\rel}$, referred to as the $L$-packet $\Pi_\phi(G,\xi)$, will be described in terms of certain characters on the group $S^\natural_\phi$.

   We now introduce (local) $A$-parameters. One main motivation is that they are useful to describe the local components of automorphic representations in the discrete spectrum. As the local components are unitary but not necessarily tempered, we need parameters to accommodate a little more than tempered representations, but not too much more to be practical. %
   This little more room is created by an extra $\SU(2)$-factor in the parameter.
   Thus a local \textbf{$A$-parameter}, or an Arthur parameter, is a continuous homomorphism
   $$\psi:L_F\times \SU(2) \ra {}^L G^*$$
   such that $\psi|_{L_F}$ is a bounded $L$-parameter. Equivalent to the latter condition is that $\psi$ has a relatively compact image in $\hat G^*$ and commutes with the natural projections of $L_F\times \SU(2)$ and ${}^L G^*$ onto $W_F$. Two $A$-parameters are equivalent if they are conjugate by an element of $\hat G^*$.
   The set of equivalence classes of $A$-parameters for $G^*$ is denoted $\Psi(G^*)$. In fact it will be convenient to introduce a set $\Psi^+(G^*)$ consisting of continuous homomorphisms $\psi:L_F\times \SU(2) \ra {}^L G^*$ without requiring that $\psi|_{L_F}$ be bounded. An $A$-parameter $\psi$, or generally an element of $\Psi^+(G^*)$, is said to be \textbf{generic} (resp. \textbf{non-generic}) if $\psi|_{\SU(2)}$ is trivial (resp. nontrivial). We define the groups $S_\psi$, $\srad_\psi$, $S^\natural_\psi$, $\ol{S}_\psi$, $\cS_\psi$, and $\ol{\cS}_\psi$ as before by replacing $\phi$ with $\psi\in \Psi^+(G^*)$. Likewise the parameter sets $\Psi(G^*)_{(G,\xi)\rel}$ and $\Psi^+(G^*)_{(G,\xi)\rel}$ are defined. %
   Finally $\psi\in \Psi^+(G^*)$ determines a canonical element $s_\psi\in S_\psi$ by
   $$s_\psi:=\psi\left( 1,\left(
                            \begin{array}{cc}
                              -1 & 0 \\
                              0 & -1 \\
                            \end{array}
                          \right)                          \right).$$

  Let us discuss the $A$-packet classification for inner forms with precise normalizations coming from the extended pure inner twist data. Let $G^*$ and $(G,\xi)$ be as above. Now suppose that $(G,\xi,z)$ is an extended pure inner twist. Write $\chi_z\in X^*(Z(\hat G^*)^\Gamma)$ for the image of $z$ under the Kottwitz map \eqref{eq:kotisoloc}. To each $\psi$ should be associated an $A$-packet, a finite subset $\Pi_\psi(G,\xi)$ of $\Pi(G)$ (depending only on the inner twist $(G,\xi)$ and not on $z$). Moreover there should be a map (which does depend on $z$)
  \begin{equation}\label{eq:Pi-to-Irr(S)}
    \Pi_\psi(G,\xi)\ra \Irr(S^\natural_\psi,\chi_z),
  \end{equation}
  where the latter denotes the set of irreducible characters on $S^\natural_\psi=S_\phi/\srad_\psi$ which transform under $Z(\hat G^*)^\Gamma$ as $\chi_z$. (Our work on the inner forms of unitary groups builds the pairing \eqref{eq:Pi-to-Irr(S)}. See \S\ref{subsub:statement-local-thm} below for precise statements.) If $\psi$ is not relevant then $\Pi_\phi(G,\xi)$ should be empty so \eqref{eq:Pi-to-Irr(S)} is vacuous.  When $\psi$ is  $(G,\xi)$-relevant then the following lemma tells us that $\Irr(S^\natural_\psi,\chi_z)$ is always nonempty.\footnote{However if $\psi$ is non-generic and if $G$ is not quasi-split, it occasionally happens that $\Pi_\psi(G,\xi)$ is empty, already for $G$ an inner form of a general linear group.}

\begin{lem}\label{lem:chi_z-triv-on-intersection}
  Suppose that the image of $\phi\in \Phi(G^*)_{(G,\xi)\rel}$ factors through a Levi embedding ${}^L M^*\hra {}^L G^*$. Then $\chi_z$ is trivial on $Z(\hat G^*)^\Gamma \cap \srad_\phi$. The same holds true for $\psi\in \Psi(G^*)_{(G,\xi)\rel}$.
\end{lem}

\begin{proof}
  We may assume that the Levi subgroup ${}^L M^*$ is minimal such that it contains the image of $\phi$. By the assumption ${}^L M^*$ is relevant for $\G$. Further we may conjugate $\phi$ so that $^L M^*$ is standard, i.e. of the form $\hat M^* \rtimes W_F$ as this does not change the intersection in the lemma. Then $(Z(\hat M^*)^{\Gamma})^0$ is a maximal torus of $S_\phi^\circ$, and hence $(Z(\hat M^*) \cap \srad_\phi)^0 = (Z(\hat M^*_{\sc})^{\Gamma})^0$ is a maximal torus of $\srad_\phi$. The intersection $Z(\hat G^*)^\Gamma \cap \srad_\phi$ is central in $\srad_\phi$ and thus contained in that maximal torus. Now we conclude by Lemma \ref{lem:triv-on-ZGcapZM}. The proof for $\psi$ is the same.
\end{proof}

  It is useful to make a simple observation about the parameters for groups obtained by restriction of scalars.

\begin{lem}\label{l:param-for-Res-groups}
  Let $E/F$ be a finite extension of local fields. Let $G^*$ be a connected quasi-split reductive group over $F$. Then there are canonical bijections
  $$\Phi(\Res_{E/F}G^*,F)\isom \Phi(G^*,E),\quad \Psi(\Res_{E/F}G^*,F)\isom \Psi(G^*,E).$$
\end{lem}

\begin{proof}
  Note that $\Phi(G^*,E)$ (resp. $\Phi(\Res_{E/F}G^*,F)$) is a subset of $H^1_{\cont}(L_E,\hat G)$ (resp. $H^1_{\cont}(L_F,\hat {\Res_{E/F}G^*})$. Then Shapiro's lemma provides a canonical bijection from $H^1_{\cont}(L_E,\hat G^*)$ to $H^1_{\cont}(L_F,\hat {\Res_{E/F}G^*})$ which carries $\Phi(G^*,E)$ onto $\Phi(\Res_{E/F}G^*,F)$ since $\hat {\Res_{E/F}G^*}$ is an induced group of $\hat G^*$ relative to the extension $E/F$. (See \cite[Prop 8.4]{Bor79} for details.) The argument for $A$-parameters is analogous.
\end{proof}

  There is a map
  \begin{eqnarray} \label{eq:artlanpar}
    \Psi^+(G^*) & \longrightarrow& \Phi(G^*) \nonumber \\
    \psi & \mapsto &  \phi_\psi:~ w~\mapsto~ \psi\left( w,\left(
                                                                           \begin{array}{cc}
                                                                             |w|^{1/2} & 0 \\
                                                                             0 &  |w|^{-1/2} \\
                                                                           \end{array}
                                                                         \right) \right).
  \end{eqnarray}
  In particular each $A$-parameter $\psi$ has an associated $L$-parameter $\phi_\psi$.

\subsubsection{Local parameters for general linear groups}\label{subsub:local-param-GL}

When $G$ is a general linear group we are on a firm ground thanks to the fact that the local Langlands correspondence is known by Harris-Taylor and Henniart in the non-archimedean case and Langlands in the archimedean case. It is fair to say that our local classification theorem for inner forms of $U_{E/F}(N)$ ultimately hinges on that for $\GL(N,E)$ via endoscopy, just like in the case of $U_{E/F}(N)$ itself.

Let $N\ge1$ be an integer. We shall write
 $$\Phi(N) := \Phi(\GL(N)),\quad \Phi_2(N):=\Phi_2(\GL(N)),\quad \Phi_{\bdd}(N) = \Phi_{\bdd}(\GL(N)),$$
and so on. We set $\Phi_{\mathrm{sim}}(N):=\Phi_2(N) $. Define $\Phi_{\mathrm{scusp}}(N)$ to be the subset of irreducible representations of $L_F$ which factor through $W_F$ if $F$ is non-archimedean. For an archimedean local field $F$ we take $\Phi_{\mathrm{scusp}}(1):=\Phi(1)$ and $\Phi_{\mathrm{scusp}}(N):=\emptyset$ if $N>1$. As before we define $\Phi_{\diamondsuit,\heartsuit}(N):=\Phi_{\diamondsuit}(N)\cap \Phi_{\heartsuit}(N)$ for any two subscript words $\diamondsuit$ and $\heartsuit$. Exactly the same convention will be adopted for $\Pi$ in place of $\Phi$ except that the convention for supercuspidal representations should be explained. We write $\Pi_{\scusp}(N)$ for the subset of supercuspidal representations in $\Pi(N)$, which has the usual meaning in the non-archimedean case while the archimedean convention is that $\Pi_{\mathrm{scusp}}(N)$ is just $\Pi(N)$ if $N=1$ and empty if $N>1$. The meaning of $\Pi(N)$, $\Pi_2(N)$, $\Pi_\temp(N)$, $\Pi_{\scusp,\temp}(N)$, $\Pi_{2,\temp}(N)$, etc, should be clear by now.

The local classification for $\GL(N,F)$ can now be formulated below. It is due to Harris-Taylor and Henniart (independently) in the non archimedean case and Langlands in the archimedean case.
We fix $\psi_F$ a non-trivial additive character of $F$. Recall that there is a canonical isomorphism $W_F^{\mathrm{ab}}\simeq F^\times$ for every local field $F$. Indeed if $F$ is non-archimedean it is given by local class field theory (and normalized to send a lift of the geometric Frobenius element to a uniformizer of $F$). If $F=\R$, it is induced by the map $W_\R=\C\coprod \C j\ra \R^\times$ such that $z\in \C\mapsto z\ol{z}$ and $j\mapsto -1$. If $F=\C$, we have $W_\C=W_\C^{\mathrm{ab}}=\C^\times$.
\begin{thm}\label{thm:LLC-GL}
There is a unique bijective correspondence $\phi \mapsto \pi$ from
$\Phi(N)$ onto $\Pi(N)$ such that the following hold.
\begin{enumerate}
  \item $\Phi(1) = \Pi(1)$ via the canonical isomorphism $W_F^{\mathrm{ab}}\simeq F^\times$ above.
  \item $\phi \otimes \chi \mapsto \pi \otimes (\chi \circ \det)$ for every character $\chi\in \Phi(1)=\Pi(1)$.
  \item $\det \circ \phi \mapsto \eta_\pi$ for the central character $\eta_\pi$ of $\pi$.
  \item $\phi^\vee \mapsto \pi^\vee$, where the superscripts designate dual representations.
  \item If $\phi_i \mapsto \pi_i, \ \ \phi_i \in \Phi(N_i)$ for $i=1,2$, then
  \[
	L(s, \pi_1 \times \pi_2) = L(s, \phi_1 \times \phi_2)\quad\mbox{and}\quad
	\epsilon(s, \pi_1 \times \pi_2, \psi_F) = \epsilon(s, \phi_1 \times \phi_2, \psi_F).
\]
\end{enumerate}
Furthermore, the bijection is compatible with the two chains
\[
	\Phi_{\mathrm{scusp}, \mathrm{bdd}}(N) \subset
	 	\Phi_{\mathrm{sim}, \mathrm{bdd}}(N)
			\subset \Phi_{\mathrm{bdd}}(N) \subset \Phi(N)
\]
and
\[
	 \Pi_{\scusp,\temp}(N) \subset  \Pi_{2,\temp}(N) \subset \Pi_{\temp}(N) \subset \Pi(N)
\]
in the sense that it maps each subset in the first chain onto its counterpart in the second chain.
\end{thm}

  So far we discussed subsets of $L$-parameters in $\Phi(N)$. We also introduce a chain of sets in the context of $A$-parameters:
  \begin{equation}\label{eq:local-Psi(N)-chain}
  \Psi_{\cusp}(N)\subset \Psi_{\simp}(N)\subset \Psi(N)\subset\Psi^+_{\unit}(N)\subset\Psi^+(N).
  \end{equation}
  The first set consists of (isomorphism classes of) irreducible $N$-dimensional unitary representations of $L_F$. The second set is the collection of $N$-dimensional representations of $L_F\times \SU(2)$ of the form $\mu\otimes \nu$, where $\mu\in \Psi_{\cusp}(m)$, $\nu=\Sym^{n-1}$, as $m,n\in \Z_{\ge 1}$ vary subject to $mn=N$. The set $\Psi(N)$ (resp. $\Psi^+(N)$), which was already defined above, consists of unitary (resp. possibly non-unitary) $N$-dimensional representations of $L_F\times \SU(2)$. The remaining set $\Psi^+_{\unit}(N)$ is explained below.

  For each $\psi\in \Psi^+(N)$ one attaches an admissible representation $\pi_\psi$ of $\GL(N,F)$ as follows. Observe that there is a decomposition
  $\psi=\oplus_{i\in I} (\psi_i\otimes \chi_i)$, where $\psi_i=\mu_i\otimes \Sym^{n_i-1}\in \Psi_{\simp}(N_i)$, $\chi_i:L_F\ra \C^\times$ is a quasi-character, and $N=\sum_{i\in I} N_i$. Define $\pi_{\psi_i}$ to be the representation of $\GL(N_i,F)$ whose L-parameter is $$\phi_{\psi_i}=\mu_i|\det|^{\frac{n_i-1}{2}}\oplus \mu_i|\det|^{\frac{n_i-3}{2}}\oplus \cdots \oplus \mu_i|\det|^{\frac{1-n_i}{2}},$$
  where each $\det$ is defined on a diagonal $\GL(N_i/n_i)$-block in $\GL(N_i)$.
  Viewing $\chi_i$ as a character of $\GL(1,F)$, and also of $\GL(N_i,F)$ via the determinant map, one can now define (a possibly reducible representation)
  \begin{equation}\label{eq:pi_psi-def}
  \pi_\psi:=\cI\left(\bigoplus_{i\in I} (\pi_{\psi_i}\otimes \chi_i)\right).
  \end{equation}
  Now $\Psi^+_{\unit}(N)$ is defined to be the set of $\psi\in \Psi^+(N)$ such that $\pi_\psi$ is irreducible and unitary.
  If $\psi\in \Psi(N)$ then all $\chi_i$ may be chosen to be trivial, and the above induction is irreducible and unitary by Bernstein's theorem (since each $\pi_{\psi_i}$ is unitary). Hence $\Psi(N)\subset \Psi^+_{\unit}(N)$, justifying the third inclusion in \eqref{eq:local-Psi(N)-chain}.

  Any $\psi\in \Psi^+(N)$ can be written as a direct sum of irreducible representations $\psi=\oplus_{j\in J_\psi} \ell_j \psi_j$, where $\ell_j\in \Z_{\ge 1}$ and $\psi_j$ are mutually non-isomorphic. Then $S_\psi\simeq \prod_{j\in J_\psi} \GL(\ell_j,\C)$, $\cS_\psi=\ol{\cS}_\psi=\{1\}$ and $S_\psi^\natural\simeq \C^\times$ (the latter isomorphism is induced by determinant on $\GL(N,\C)$).

\subsubsection{Local conjugate self-dual parameters}\label{subsub:local-conj-self-dual}

  Here we recall the definition and basic properties of conjugate self-dual parameters on $\GL(N)$ which will then be related to parameters on unitary groups below. Let $E$ be a quadratic field extension of $F$, and $c\in \Gal(E/F)$ be the nontrivial automorphism. Choose $\tilde c\in W_F$ lifting $c$.
  Let $\cL_E$ and $\cL_F$ designate either $L_E$ and $L_F$ or $L_E\times \SU(2)$ and $L_F\times \SU(2)$. We embed $W_F$ in $\cL_F$ via the canonical injection $W_F\hra L_F$ (isomorphism if $F$ is archimedean; $w\mapsto (w,1)$ if $F$ is non-archimedean). Then $\tilde c$ acts on $\cL_E$ by $g\mapsto \tilde c g \tilde c^{-1}$ (conjugation in $\cL_F$). Given a continuous representation $\rho:\cL_E\ra \GL(N,\C)$ define $\rho^\star:=(\rho^{\tilde c})^\vee$, where $\rho^{\tilde c}(g):=\rho(\tilde c g \tilde c^{-1})$ (the same underlying vector space with twisted action). Clearly the isomorphism class of $\rho^\star$ is well-defined, independently of the choice of $\tilde c$. We say that $\rho$ is \textbf{conjugate self-dual} if $\rho^\star$ is isomorphic to $\rho$.

   We introduce the notion of parity for conjugate self-dual parameters following section 7 of \cite{GGP12}. Let $V$ be an $N$-dimensional vector space over $\C$ with tautological action by $\GL(V)$. Let $H:=(\GL(V)\times \GL(V))\rtimes \Gal(E/F)$ be a semi-direct product where $\Gal(E/F)$ acts on $\GL(V)\times \GL(V)$ by permuting the two factors. Let $H^0:=\GL(V)\times \GL(V)$ denote the index 2 subgroup in $H$. There is a decomposition as $H$-representations
  $$\Ind^H_{H^0}(V\boxtimes V)=\As^+(V)\oplus \As^-(V),$$
  where $\As^+(V)$ and $\As^-(V)$ are irreducible and $N^2$-dimensional, characterized by the property that $\tr (w_c|\As^{\pm})=\pm N$.
  Note that $\As^+(V)$ (resp. $\As^-(V)$) is isomorphic to $V\boxtimes V\simeq \Hom_{\C}(V^\vee,V)$ as an $H^0$-representation. An element of $\As^+(V)$ (resp. $\As^-(V)$) is said to be nondegenerate if it corresponds to an isomorphism $V^\vee \simeq V$.

    A continuous representation $\rho:\cL_E\ra \GL(V)$ gives rise to a map $$\tilde\rho:\cL_F\ra H,$$
    given by $\tilde\rho(g):=(\rho(g),\rho(\tilde c g \tilde c^{-1}))\in H^0$ for $g\in \cL_E$ and $\tilde \rho(\tilde c):=(1,\rho(\tilde c^2))\in H\bs H^0$. The $\cL_F$-action through $\tilde \rho$ stabilizes $\As^+(V)$ and $\As^-(V)$ so we obtain
    $$\cL_F\ra \GL(\As^+(V)),\quad \cL_F\ra \GL(\As^-(V)).$$
     We call $\rho$ \textbf{conjugate-orthogonal} (resp. \textbf{conjugate-symplectic}) if $\cL_F$ fixes a nondegenerate vector in $\As^+(V)$ (resp. $\As^-(V)$). We say that $\rho$ has \textbf{parity} $\sigma$ if $\sigma$ is the sign such that $\As^\sigma(V)$ has a nondegenerate $\cL_F$-fixed vector. It is not hard to verify that these definitions do not depend on the choice of $\tilde c$. If $\rho$ is conjugate self-dual and irreducible then it follows from Schur's lemma that the parity of $\rho$ is unique. So such a $\rho$ is either conjugate orthogonal or conjugate symplectic but cannot be both. However the parity of $\rho$ is not unique in general in the reducible case (for a simple example consider the trivial representation $\rho$ when $N>1$).
  We remark that an alternative definition of conjugate orthogonal/symplectic parameters is given in the section 3 of \cite{GGP12} in terms of the existence of a certain nondegenerate bilinear pairing on $V$. The two definitions are shown to be equivalent in \cite[Prop 7.5.1]{GGP12}.

  The decomposition of $\rho$ into irreducibles may be put in the following form (cf. \cite[\S4]{GGP12})
   \begin{equation}\label{eq:decompose-rho}
   \rho=\left(\bigoplus_{i\in I_\rho} \ell_i \rho_i\right)\oplus \left(\bigoplus_{j\in J_\rho} \ell_j ( \rho_j\oplus \rho_j^\star)\right),
   \end{equation}
  where $\ell_i,\ell_j\ge 1$, $\sum_{i\in I_\rho} \dim \rho_i+2\sum_{j\in J_\rho}\dim \rho_j=N$, and $\rho_i\simeq \rho_i^\star$ and $\rho_j\ncong \rho_j^\star$ for every $i\in I_\rho$ and $j\in J_\rho$. We say that $\rho$ is \textbf{elliptic} if $\ell_i=1$ for all $i\in I_\rho$ and if $J_\rho$ is empty.

  The discussion so far applies when $\rho$ is either an element of $\Phi(\GL(N),E)$ or  $\Psi(\GL(N),E)$. The subset of conjugate self-dual parameters is to be denoted $\tilde\Phi(\GL(N),E)$ and $\tilde\Psi(\GL(N),E)$. The parity of such parameters is defined as above.
  The chain \eqref{eq:local-Psi(N)-chain} gives rise to a chain of subsets
 \begin{equation}\label{eq:local-Psi(N)-chain-csd}
  \tilde \Psi_{\simp}(N)\subset \tilde\Psi_{\el}(N)\subset\tilde\Psi(N)\subset\tilde\Psi^+_{\unit}(N)\subset\tilde\Psi^+(N),
  \end{equation}
  which are obtained by collecting conjugate self-dual parameters. Observe that we omitted $\tilde\Psi_{\cusp}(N)$ in the chain but inserted $\tilde\Psi_{\el}(N)$, the subset of $\tilde\Psi(N)$ consisting of elliptic parameters. It is clear from the definition that $\tilde\Psi_{\simp}(N)\subset \tilde\Psi_{\el}(N)$.

\subsubsection{Local parameters for unitary groups}\label{subsub:local-param-U}

 It is desirable to understand local $L$-parameters and $A$-parameters for $U_{E/F}(N)$ in relation to those for $G_{E/F}(N)$, which amounts to understanding (a suitable twist of) the local base change relative to $E/F$. A careful analysis is not only important in the purely local context but also motivates the somewhat indirect construction of global parameters on unitary groups via those on general linear groups, where the global analogue of $L_F$ is not available.

  Assume that $E$ is a quadratic field extension of $F$. (The case $E=F\times F$ is treated at the end of this subsection.)
  Let $\kappa\in \{\pm 1\}$ and choose $\chi_\kappa\in \cZ^\kappa_E$. The composition with the $L$-morphism $\eta_{\chi_\kappa}$ induces maps that we denote by the same notation:
  $$\eta_{\chi_\kappa,*}: \Phi(U_{E/F}(N)) \ra \Phi(G_{E/F}(N)),$$
    $$\eta_{\chi_\kappa,*}: \Psi(U_{E/F}(N)) \ra \Psi(G_{E/F}(N)).$$
  To describe the image of these maps, note that we can identify $\Phi(G_{E/F}(N))= \Phi(\GL(N),E)$ and $\Psi(G_{E/F}(N))= \Psi(\GL(N),E)$ via Lemma \ref{l:param-for-Res-groups}, cf. \cite[(2.2.2)]{Mok}. So the question is when a parameter for $\GL(N)$ over $E$ comes from a parameter for $U_{E/F}(N)$ via $\eta_{\chi_\kappa,*}$. It is not hard to see that the image of $\eta_{\chi_\kappa,*}$ consists of conjugate self-dual parameters. The subtle part is to determine whether a conjugate self-dual parameter lies in the image of $\eta_{\chi_\kappa,*}$ for $\kappa=1$ or $\kappa=-1$ (or both or neither). %

\begin{lem}\label{lem:image-xi_chi}
  The map $\eta_{\chi_\kappa,*}$ is injective on $\Phi(U_{E/F}(N))$ (resp. $\Psi(U_{E/F}(N))$) and restricts to a bijection from the preimage of $\Phi_\simp(G_{E/F}(N))$ (resp. the preimage of $\Psi_\simp(G_{E/F}(N))$) onto the subset of conjugate self-dual parameters with parity $(-1)^{N-1}\kappa$ in $\Phi_\simp(G_{E/F}(N))$ (resp. $\Psi_\simp(G_{E/F}(N))$).
\end{lem}

\begin{rem}
  The global counterpart of the lemma, Proposition \ref{prop:1st-seed-thm} below, is a much deeper result.
\end{rem}

\begin{proof}
  The lemma for $L$-parameters is proved in \cite[Thm 8.1, Cor 8.2]{GGP12}, cf. \cite[Lem 2.2.1]{Mok}. (The proof is given in \cite{GGP12} when $\kappa=1$ but it is easy to extend the result to the case $\kappa=-1$ by twisting by a character in $\cZ_E^-$.) The same argument works in the case of $A$-parameters, replacing $L_F$ and $L_E$ with $L_F\times \SU(2)$ and $L_E\times \SU(2)$.
\end{proof}

 Let $(G^{\tilde\fke},s^{\tilde\fke},\eta^{\tilde\fke})\in \cE_{\simp}(\tilde G (N))$ so that $G^{\tilde\fke}\simeq U_{E/F}(N)$. In comparison with \eqref{eq:local-Psi(N)-chain-csd} we produce another chain %
   \begin{equation}\label{eq:local-Psi(G)-chain}
 \Psi_{\simp}(G^{\tilde\fke})\subset \Psi_{2}(G^{\tilde\fke}) \subset \Psi(G^{\tilde\fke})\subset\Psi^+_{\unit}(G^{\tilde\fke})\subset \Psi^+(G^{\tilde\fke}),
  \end{equation}
 defined as follows. The set $\Psi^+(G^{\tilde\fke})$ was defined in \S\ref{subsub:L-param-A-param}. Take $\Psi_{\simp}(G^{\tilde\fke})$ (resp. $ \Psi_{2}(G^{\tilde\fke})$) to be the set of $\psi\in \Psi^+(G^{\tilde\fke})$ such that $\eta^{\tilde\fke}\psi$ belongs to $\tilde \Psi_{\simp}(N)$ (resp. $\tilde\Psi_{\el}(N)$). This set does not change if the group $G^{\tilde\fke}$ extends to another endoscopic triple, since $\eta^{\tilde\fke}\psi$ will only change by a character of $\GL(N)$. We can now rephrase the preceding lemma as the decomposition
 $$ \Phi_\simp(G(N))=\eta_{\chi_+,*}\Phi_\simp(U(N))\coprod\eta_{\chi_-,*}\Phi_\simp(U(N)) $$
 according to the parity of parameters, and likewise for $\Psi_\simp(G(N))$.
 The other sets $\Psi(G^{\tilde\fke})$ and $\Psi^+_{\unit}(G^{\tilde\fke})$ are given similarly in parallel with \eqref{eq:local-Psi(N)-chain-csd}. By construction $\psi\mapsto \eta^{\tilde\fke}\psi$ is an injective chain-preserving map from \eqref{eq:local-Psi(G)-chain} to \eqref{eq:local-Psi(N)-chain-csd}.

  Each parameter $\psi\in \tilde \Psi^+_{\simp}(\GL(N),E)$ can be written as $\psi=\phi\otimes \nu$, where $\phi\in \tilde\Phi_{\simp}(\GL(m),E)$, $\nu$ is the $\Sym^{n-1}$-representation of $\SU(2)$, and $mn=N$. Here we view $\phi$ (resp. $\nu$) as an irreducible conjugate self-dual representation of $L_E\times \SU(2)$ via its projection onto $L_E$ (resp. $\SU(2)$). We would like to compare the sign of the $L$-embedding that $\psi$ and $\phi$ each comes from. As irreducible conjugate self-dual representations, the parity of each of $\psi$, $\phi$, and $\nu$ is uniquely determined. Write $b(\psi)$, $b(\phi)$, and $b(\nu)$ for the parities. Since $\nu$ is orthogonal (resp. symplectic) as a representation of $\SU(2)$ when $n$ is even (resp. odd), the $b(\nu)=(-1)^{n-1}$. (To see this one can apply the criterion in \cite[\S3]{GGP12} for conjugate orthogonality/symplecticity, which reduces to the usual orthogonality/symplecticity in this case.) Then the lemma 3.2 of \cite{GGP12} tells us that $$b(\psi)=b(\phi)b(\nu)=(-1)^{n-1}b(\phi).$$
  Lemma \ref{lem:image-xi_chi} implies that $\psi\in \eta_{\chi(\psi),*}\Psi(U(N))$ and $\phi\in \eta_{\chi(\phi),*}\Psi(U(N))$ for $\chi(\psi)\in \cZ_E^{\kappa(\psi)}$ and $\chi(\phi)\in \cZ_E^{\kappa(\phi)}$ with $\kappa(\psi)=(-1)^{N-1}b(\psi)$ and $\kappa(\phi)=(-1)^{m-1}b(\phi)$. Thus
  \begin{equation}\label{eq:kappa(psi)-kappa(phi)}
    \kappa(\psi)=(-1)^{N-m-n-1}\kappa(\phi).
  \end{equation}

  We can make $S_\psi$ explicit for each $\psi\in \Psi^+(U_{E/F}(N))$. %
  Let $\kappa\in \{\pm 1\}$, $\chi\in \cZ_E^\kappa$, and set $\psi^N:=\eta_{\chi}\psi$. We can write
  $$\psi^N=\left(\bigoplus_{i\in I_{\psi^{N}}} \ell_i \psi^{N_i}_i\right)\oplus \left(\bigoplus_{j\in J_{\psi^N}} \ell_j ( \psi^{N_j}_j\oplus (\psi^{N_j}_j)^\star)\right),$$
  where the notation is as in \eqref{eq:decompose-rho}. Each $\psi^N_i\in \tilde \Phi (\GL(N_i))$ determines $\kappa_i\in \{\pm 1\}$ such that the parity of $\psi^N_i$ is $\kappa_i(-1)^{N_i-1}$. Consider a partition $I_{\psi^N}=I_{\psi^N}^+\coprod I_{\psi^N}^-$ such that $i\in I_{\psi^N}$ belongs to $I_{\psi^N}^+$ (resp. $I_{\psi^N}^-$) if $\kappa_i=\kappa (-1)^{N-N_i}$ (resp. otherwise). Then it follows that $\ell_i$ is even for each $i\in I_{\psi^N}^-$ and that
  $$S_\psi\simeq \prod_{i\in I_{\psi^N}^+} \O(\ell_i,\C) \times \prod_{i\in I_{\psi^N}^-} \Sp(\ell_i,\C)\times \prod_{j\in J_{\psi^N}} \GL(\ell_j,\C).$$
  (See \cite[\S4]{GGP12}.) The group $Z(\hat G)^\Gamma=\{\pm 1\}$ is embedded diagonally into the right hand side. It is easy to see that
  $$\cS_\psi\simeq \cS_\psi^\natural\simeq (\Z/2\Z)^{|I_{\psi^N}^+|},\quad \ol{\cS}_\psi\simeq \left\{
  \begin{array}{cc}
    (\Z/2\Z)^{|I_{\psi^N}^+|-1}, & \forall_{i\in I_{\psi^N}^+},~ 2|\ell_i,\\
    (\Z/2\Z)^{|I_{\psi^N}^+|}, & \mbox{otherwise}.
  \end{array}\right.$$

  Now we briefly discuss the easy case when $E=F\times F$. For each $\psi\in \Psi^+(U_{E/F}(N))$ one can write $\psi^N=\eta_\chi\psi$ in the form
  $\psi^N=\oplus_{j\in J_{\psi^N}} \ell_j \psi_j^{N_j}$. Then $S_\psi\simeq \prod_{j\in J_{\psi^N}} \GL(\ell_j,\C)$ and $\cS_\psi=\ol\cS_\psi=\{1\}$.

\subsection{Global parameters}\label{sub:global-param}

In this subsection we recall the formalism of global parameters for (quasi-split) unitary groups according to \cite{Arthur} and \cite{Mok}. Since the global analogue of $L_F$ of the last subsection is highly hypothetical, a different approach should be taken to define the analogous global $L$-parameters in an unconditional way. The idea is to substitute cuspidal automorphic representations of $\GL(N)$ for irreducible $N$-dimensional representations of $L_F$. Then, loosely speaking, a parameter for a unitary group can be constructed as a formal sum of cuspidal automorphic representations of general linear groups subject to a suitable conjugate self-duality condition.

\subsubsection{Conjugacy class data of automorphic representations}\label{subsub:conj-class-data}

We explain how to associate a string of local conjugacy classes in the $L$-group to an automorphic representation via Satake isomorphism.
This conjugacy class data is used to state a sign dichotomy when descending a conjugate self-dual automorphic representation of $\GL(N)$, cf. Proposition \ref{prop:1st-seed-thm} below. Another role is played in the decomposition of the discrete part of the trace formula in \S\ref{sub:I_disc} below.

Let $F$ be a global field, and $G$ a connected reductive group over $F$.
We have the sets of automorphic representations of $G(\A_F)$, cf. \cite[p.19]{Arthur}:
\[
	\mathcal{A}_{\mathrm{cusp}}(G) \subset \mathcal{A}_2(G) \subset \mathcal{A}(G),
\]
where each set consists of irreducible unitary representations of $G(\A_F)$ whose restrictions to $G(\A_F)^1$ are constituents of $L^2_{\diamondsuit}(G(F)\bs G(\A_F)^1)$ with $\diamondsuit\in \{\cusp,\disc,\emptyset\}$ in the same order.

For $S$ any finite set of primes outside of which $G$ is unramified, we define $\mathcal{C}^S_{\mathbb{A}_F}(G)$
to be the set of adelic families of semisimple $\widehat{G}$-conjugacy classes $c^S$ of the form
\[
	c^S = (c_v)_{v \not\in S}
\]
where $c_v$ is a $\widehat{G}$-conjugacy class in
${}^L G_v = \widehat{G} \rtimes W_{F_v} \subset {}^L G$ represented
by an element of the form $c_v = t_v \rtimes \mathrm{Frob}_v$ with $t_v$ a semi-simple element of $\widehat{G}$.
We define $\mathcal{C}_{\mathbb{A}_F}(G)$ to be the set of equivalence classes of such families,
$c^S$ and $(c')^{S'}$ being equivalent if $c_v$ equals $c'_v$ for almost all $v$.

Let $\pi = \otimes_v \pi_v$ be an automorphic representation of $G(\mathbb{A}_F)$ unramified outside some finite set $S$.
For all places $v \not\in S$, the Satake isomorphism associates to the unramified representation
$\pi_v \mapsto c(\pi_v) \in {}^L G_v$
a semisimple $\widehat{G}$-conjugacy class in ${}^L G_v$.   This allows us to associate to $\pi$
\[
	c^S(\pi) = \{ c(\pi_v) : v \not \in S\} \in \mathcal{C}_{\mathbb{A}_F}^S(G).
\]
We then have the mapping
\[
	\pi \mapsto c(\pi) = c^S(\pi)
\]
from $\mathcal{A}(G) \rightarrow  \mathcal{C}_{\mathbb{A}_F}(G)$.  We define the image of this map
to be $\mathcal{C}_{\mathrm{aut}}(G)$. When $G=\GL(N)$
we will see that the map $\pi \mapsto c(\pi)$ is particularly well-behaved.

\subsubsection{Global parameters for $\GL(N)$}\label{subsub:global-param-GL}

 Let $\cA^+_2(N)$ (resp. $\cA^+_{\cusp}(N)$) be the set of irreducible admissible representations of $\GL(N,\A_F)$ whose restriction to $\GL(N,\A_F)^1$ appears as a direct summand of $L^2_{\disc}(\GL(N,F)\bs \GL(N,\A_F)^1)$ (resp. $L^2_{\cusp}(\GL(N,F)\bs \GL(N,\A_F)^1)$). The subset of unitary representations in $\cA^+_2(N)$ (resp. $\cA^+_{\cusp}(N)$) is denoted $\cA_2(N)$ (resp. $\cA_{\cusp}(N)$).
  Let $\cA^+_{\iso}(N)$ denote the set of $\pi=\otimes_v\pi_v$ which is an isobaric sum
  \begin{equation}\label{eq:isobaric}
  \pi=\pi_1\boxplus \cdots \boxplus \pi_r
  \end{equation}
  for some $r\in \Z_{\ge 1}$, $N_1,...,N_r\in \Z_{\ge 1}$ $(N_1+\cdots+N_r=N)$, $\pi_i\in \cA^+_\cusp(N_i)$ for $1\le i\le r$. The isobaric sum \eqref{eq:isobaric} means that $\pi_v$ is the Langlands quotient of the normalized parabolic induction from $\pi_{1,v}\oplus \cdots \oplus \pi_{r,v}$ at every place $v$. The interpretation via the local Langlands correspondence is that the $L$-parameter for $\pi_v$ is the direct sum of the $L$-parameter for $\pi_{1,v},...,\pi_{r,v}$ at every $v$.

   Moeglin and Waldspurger identified $\cA_2(N)$ with a precise subset of $\cA^+_{\iso}(N)$. Namely $\cA_2(N)$ consists of representations of the following form:
  \begin{equation}\label{eq:MW-rep}
    \mu|\det|^{\frac{n-1}{2}} \boxplus \mu|\det|^{\frac{n-3}{2}} \boxplus \cdots \boxplus \mu|\det|^{\frac{1-n}{2}},
  \end{equation}
  where $N=mn$, $m,n\in \Z_{\ge 1}$, and $\mu\in \cA_{\cusp}(m)$.
  (No two representations of such form are isomorphic unless $m$, $n$ and $\mu$ are the same.) Define $\cA(N)$ to be the set of \eqref{eq:isobaric} such that $\pi_i\in \cA_2(N_i)$ instead of $\pi_i\in \cA^+_\cusp(N_i)$. Then we have a chain
   \begin{equation}\label{eq:chain-of-A(N)}
   \cA_{\cusp}(N)\subset \cA_2(N)\subset \cA(N)\subset \cA^+_{\iso}(N).
    \end{equation}
  Moreover an element of $\cA^+_{\iso}(N)$ is uniquely determined by the data at almost all (unramified) places by the Jacquet-Shalika theorem, which tells us that the map $$\cA^+_{\iso}(N)\lra \cC_{\A_F}(N),\quad \pi\mapsto c(\pi)$$
  is a bijection onto its image (which is much smaller than $\cC_{\A_F}(N)$). Here we write $\cC_{\A_F}(N)$ for $\cC_{\A_F}(\GL(N))$ as usual. The image in $\cC_{\A_F}(N)$ will be denoted $\cC_{\aut}(N)$.%

We shall introduce some formal parameter sets for $\GL(N)$ in the global setting
  $$ \Psi_{\cusp}(N)\subset \Psi_{\simp}(N) \subset \Psi(N)$$
  in parallel with the first three sets in \eqref{eq:chain-of-A(N)}.
First off, put $\Psi_{\mathrm{cusp}}(N): =  \mathcal{A}_{\mathrm{cusp}}(N)$.
Next, let $\Psi_{\mathrm{sim}}(N) = \Psi_{\mathrm{sim}}(\GL(N))$ denote the set of formal tensor products
\[
	\psi = \mu \boxtimes \nu
\]
where $\mu \in \mathcal{A}_{\mathrm{cusp}}(m)$ and  $\nu$ is the  irreducible representation of  $\SU(2)$ of dimension $n$ such that $N = m n$.  We shall associate to such a $\psi$ the automorphic representation $\pi_\psi$ given by \eqref{eq:MW-rep} so that $\psi\mapsto \pi_\psi$ is a bijection from $\Psi_{\mathrm{sim}}(N)$ onto $\mathcal{A}_2(N)$.
Finally $\Psi(N)$ is going to be the set of formal unordered direct sums
\begin{equation}\label{eq:formal-global-GL(N)}
	\psi = \ell_1 \psi_1 \boxplus \cdots \boxplus \ell_r \psi_r
\end{equation}
for positive integers $\ell_k$ and distinct elements $\psi_k = \mu_k \boxtimes \nu_k$ in $\Psi_{\mathrm{sim}}(N_k)$.
The ranks here are positive integers $N_k = m_k n_k$ such that
\[
	N = \ell_1 N_1 + \cdots \ell_r N_r = \ell_1 m_1 n_1 + \cdots + \ell_r m_r n_r.
\]
To such a $\psi$, we associate
 $$\pi_\psi:= \boxplus_{i=1}^r (\underbrace{\pi_{\psi_i}\boxplus \cdots \boxplus \pi_{\psi_i}}_{\ell_i})$$
 so that the map $\psi\mapsto \pi_\psi$ is a bijection from $\Psi(N)$ onto $\cA(N)$.
 By Bernstein's theorem that the normalized induction of irreducible unitary representations is irreducible, we see that $\pi_\psi$ may also be described as the normalized induction $\cI(\otimes_{i=1}^r (\pi_{\psi_i}\otimes\cdots \otimes \pi_{\psi_i}))$.\footnote{We could have introduced the analogues of $\Psi^+(N)$ and $\Psi^+_{\unit}(N)$ in \cite{Arthur} \S1.5 (in the global setting) but we do not need them.}

  Finally let $E$ be a finite field extension of $F$. In view of the functorial isomorphism $\Res_{E/F}\GL(N,R)=\GL(N,R\otimes_F E)$ for $F$-algebras $R$, we define $\Psi(\Res_{E/F}\GL(N))$ to be $\Psi(\GL(N)_E)$, namely the set of formal parameters for $\GL(N)$ over $E$. Likewise $\Psi_\diamondsuit(\Res_{E/F}\GL(N))$ is defined for various constraints $\diamondsuit$. In particular this applies to the case $[E:F]=2$, where the group is $G_{E/F}(N)$.

\subsubsection{Localizing global parameters for $\GL(N)$}\label{subsub:localizations-GL}

Consider a cuspidal  automorphic representation
$\mu \in \mathcal{A}_{\mathrm{cusp}}(N)$
and  $v$  a place of $F$.
  By
the local Langlands correspondence, we can associate to  $\mu_v \in \Pi(\GL(N, F_v))$ an $L$-parameter
$\phi_{\mu_v} \in \Phi(\GL(N, F_v))$.  The $L$-parameter $\phi_{\mu_v}$ can be viewed
as a generic unbounded $A$-parameter $\psi_{\mu_v} \in \Psi^+(\GL(N, F_v))$.

This allows us to define for every $N\in\Z_{\ge1}$
the localization map
\begin{eqnarray*}
	\Psi_{\mathrm{cusp}}(N) &\rightarrow& \Psi^+_{v}(N) \\
		\mu &\mapsto& \psi_{\mu_v}
\end{eqnarray*}
This can be extended to define the localization map
\begin{eqnarray*}
\Psi_{\mathrm{sim}}(N) &\rightarrow& \Psi^+_{v}(N)	\\
	\mu \boxtimes \nu &\mapsto& \phi_{\mu_v} \otimes \nu.
\end{eqnarray*}
This map can then be uniquely extended by requiring that $\boxplus$ be carried over to $\oplus$ to produce the localization map
\[
	\Psi(N) \rightarrow \Psi^+_v(N)		.
\]

  Similarly there is a localization map $\Psi(G_{E/F}(N))\ra \Psi^+_v(G_{E/F}(N))$. Concretely it is identified (via Lemma \ref{l:param-for-Res-groups}) with the localization map $\Psi_E(N)\ra \Psi^+_w(N)\times \Psi^+_{\ol w}(N)$ if $v$ splits as $w\ol w$ in $E$, or the localization map $\Psi_E(N)\ra \Psi^+_w(N)$ if $w$ is the only place of $E$ above $v$.

\subsubsection{Global parameters for $U_{E/F}(N)$}\label{subsub:global-param-U}

We begin with the discussion of conjugate self-dual parameters on $\GL(N)$ as the global analogue of \S\ref{subsub:local-conj-self-dual}. A careful consideration of parity will allow us to define formal global parameters for unitary groups, depending on the sign of an $L$-embedding $^L U_{E/F}(N)\hra {}^L G_{E/F}(N)$. A definition which is independent of choices will be discussed in \S\ref{subsub:canonical-def} below.

If $\pi$ is a cuspidal automorphic representation of $\GL(N, \A_E)$, we shall denote by
$\pi^\star = \pi^{\vee, c}$ the conjugate dual representation of $\pi$.
This notation shall also be used in the context of local $L$-parameters and $A$-parameters to refer to
the conjugate dual representation.

Consider a
partition $N_1 + \cdots + N_r = N$ and a formal
global parameter $\psi =   \ell_1 \psi_1 \boxplus \cdots \boxplus \ell_r \psi_r \in \Psi(N)$
where  $\psi_i = \mu_i \boxtimes \nu_i \in \Psi_{\mathrm{sim}}(N_i)$ and $N_i = m_i n_i$ for $i=1,\ldots, r$.
We define the conjugate dual parameter
\[
\psi^\star =  \ell_1 \psi^\star_1 \boxplus \cdots \boxplus \ell_r \psi^\star_r
\]
where $\psi^\star_i = \mu_i^\star \boxtimes \nu_i$ for $i=1,\ldots, r$.   We remark that
$\pi_\psi^{\star} \simeq \pi_{\psi^\star}$.
Such a parameter $\psi$ is said to be \textbf{conjugate self-dual} if $\psi=\psi^\star$ up to reordering, or more precisely if there exists an involution $i \leftrightarrow i^\star$ of the indexing set
$\{1,\ldots,r\}$ such that
\[
	\psi_i^\star = \psi_{i^\star} \textrm{ and } \ell_i = \ell_{i^\star} \text{ for } i =1,\ldots, r.
\]
The subset of parameters $\psi \in \Psi(N)$ that are conjugate self-dual is denoted by $\widetilde{\Psi}(N)$.
If the parameter $\psi$ satisfies in addition the condition that
\[
i^\star = i \textrm{ and } \ell_i=1 \textrm{ for } i =1,\ldots, r,
\]
 then $\psi$ is said to be \textbf{elliptic}.
The subset of elliptic parameters is denoted  by $\widetilde{\Psi}_{\mathrm{ell}}(N)$.
We shall also have need of the subset of conjugate self-dual simple parameters that shall be denoted by
$\widetilde{\Psi}_{\mathrm{sim}}(N)$.    We shall denote by $\widetilde{\Phi}(N) \subset \widetilde{\Psi}(N)$ the subset of conjugate self-dual generic parameters. (Recall that $\psi$ is called generic if $\nu_i$ is the trivial representation of $\SU(2)$ for all $i=1,\ldots,r$.) A generic parameter is often written as $\phi$. Put $\widetilde{\Phi}_{\mathrm{sim}}(N):=\widetilde{\Phi}(N)\cap \widetilde{\Psi}_{\simp}(N)$.

A conjugate self-dual cuspidal automorphic representation $\phi$ is
said to be \textbf{conjugate-orthogonal} (resp. \textbf{conjugate-symplectic}) if
the Asai $L$-function
\[
		L(s, \phi, \mathrm{As}^+)\quad \textrm{ (resp. $L(s, \phi, \mathrm{As}^{-})$) }
\]
has a pole at $s=1$.  Then $\phi \in \widetilde{\Phi}_{\mathrm{sim}}(N)$ is
either conjugate-orthogonal or conjugate-symplectic, and this is mutually exclusive. Indeed, there is a decomposition of the Rankin-Selberg $L$-function $$L(s,\phi\times \phi^c)=L(s, \phi, \mathrm{As}^+)L(s, \phi, \mathrm{As}^-),$$
which has a simple pole at $s=1$ since $\phi$ is cuspidal and $\phi^\star\simeq \phi$. Each factor on the right hand side has no zero at $s=1$ by a result of Shahidi (\cite[Thm 5.1]{Shahidi81}). Hence exactly one factor on the right has a pole at $s=1$. (See the paragraph preceding the theorem 2.5.4 in \cite{Mok} for more details and references.)

We are almost ready to define the global parameters for $U_{E/F}(N)$ by means of conjugate self-dual parameters. Let us fix
a pair of characters
 $\chi_{+} \in \mathcal{Z}_E^+ $ and $\chi_{-} \in \mathcal{Z}_E^-$.
(We can make the definition more natural by considering all choices of $\chi_+$ and $\chi_-$ at once. See \S\ref{subsub:canonical-def} below.)

\begin{prop}(1st seed theorem)\label{prop:1st-seed-thm}
  For each $\phi\in \tilde{\Phi}_{\simp}(N)$ there exists a unique $\fke_\phi=(G_\phi,s_\phi,\eta_\phi)\in \tilde\cE_{\el}(N)$ such that $c(\phi)=\eta_\phi(c(\pi))$ for some $\pi\in \cA_2(G)$. The datum $\fke_\phi$ is simple and may be represented by
  \[
	(U_{E/F}(N),1, \eta_{\chi_\kappa}),
\]
  where $\kappa \in \{\pm1\}$ is equal to $(-1)^{N-1}$ (resp. $(-1)^N$) if $\phi$ is conjugate-orthogonal
(resp. conjugate-symplectic).
\end{prop}

\begin{proof}
  This follows from the theorems 2.4.2 and 2.5.4 of \cite{Mok}.
\end{proof}

 Given a $\psi^N \in \widetilde{\Psi}(N)$, which admits a decomposition \eqref{eq:formal-global-GL(N)}, write $K_{\psi^N}$ for the indexing set $\{1,\ldots, r\}$. We have the decomposition $K_{\psi^N} = I_{\psi^N} \sqcup J_{\psi^N} \sqcup (J_{\psi^N})^\star$, where $I_{\psi^N}$ consists of the set of indices that are fixed under the involution $i\leftrightarrow i^\star$ whilst
$J_{\psi^N}$ is a set of representatives for the orbits of size two of that involution.  We can then write
the parameter $\psi$ in the following form.
\begin{equation}\label{eq:psi-global-dec}
  	\psi^N = (\boxplus_{i \in I_{\psi^N}} \ell_i \psi^{N_i}_i) \boxplus (\boxplus_{j \in J_{\psi^N}} \ell_j (\psi^{N_j}_j \boxplus \psi^{N_j}_{j^\star})),
\end{equation}
where  $\psi_i = \mu_i \boxtimes \nu_i \in \Psi_{\mathrm{sim}}(N_i)$, $N_i = m_i n_i$ for $i=1,\ldots, r$ as before, and where $\psi^{N_i}_i$ are conjugate self-dual and $\psi^{N_j}_j$ are not.
To each $i \in I_{\psi}$, we can associate a pair  $(U_{E/F}(m_i), \eta_{\chi_{\kappa(\mu_i)}})$
for a unique $\kappa(\mu_i)\in \{\pm1\}$ by considering the pair associated to the simple generic factor
$\mu_i \in \widetilde{\Phi}(m_i)$.  We shall let $H_i := U_{E/F}(m_i)$.
To each $j \in J_{\psi}$, we shall let $H_j := G_{E/F}(m_j)$.

We shall denote by $\{K_{\psi}\}$ the set of orbits of $K_\psi$ under the involution $k\leftrightarrow k^\star$,
which can be identified with $I_{\psi} \sqcup J_\psi$.
To each $k \in \{K_{\psi}\}$, we have associated  $H_k$, a connected reductive group defined over $F$.
We shall form
 the fibre product over $W_F$
\[
	\mathcal{L}_{\psi} = \prod_{k \in \{K_\psi\}} ({}^L H_k \rightarrow W_F).
\]
The group $\mathcal{L}_{\psi}$
shall serve as our substitute for the part of the global Langlands group that accounts for the parameter $\psi$.
If $k=i \in I_{\psi}$, we have the embedding
\[
	\widetilde{\mu}_i := \eta_{\chi_{\kappa(\mu_i)}} : {}^L U_{E/F}(m_i) \rightarrow {}^L G_{E/F}(m_i)
\]
If $k=j \in J_{\psi}$,  then
we define the embedding
\begin{eqnarray*}
\widetilde{\mu}_j:  {}^L G_{E/F}(m_j) &\rightarrow& {}^L G_{E/F}(2 m_j)	\\
					g_1 \times g_2 \times w	&\mapsto&
                    \left(
                      \begin{array}{cc}
                        g_1 & 0 \\
                        0 & J_n {}^t g_2^{-1} J_n^{-1} \\
                      \end{array}
                    \right)
                   \times
                    \left(
                      \begin{array}{cc}
                        g_2 & 0 \\
                        0 & J_n {}^t g_1^{-1} J_n^{-1} \\
                      \end{array}
                    \right)\times w,
\end{eqnarray*}
for $w\in W_F$.
Putting everything together,  we associate to a parameter $\psi^N \in \widetilde{\Psi}(N)$ the
$L$-homomorphism
\[
	\widetilde{\psi}^N : \mathcal{L}_{\psi} \times \SL(2, \C) \rightarrow {}^L G_{E/F}(N)
\]
which is defined as the direct sum
\[
	\widetilde{\psi}^N = \left(\bigoplus_{i \in I_\psi} \ell_i(\tilde \mu_i \boxtimes \nu_i) \right)
		\oplus
		\left(\bigoplus_{j \in J_\psi} \ell_j(\tilde\mu_j \boxtimes \nu_j) \right)
\]
where we have identified an $n$-dimensional representation $\nu : \SL(2, \C) \rightarrow \GL(n,\C)$
as the homomorphism
\begin{eqnarray*}
	\overline{\nu} : \SL(2, \C) &\rightarrow&  \widehat{G}_{E/F}(n) = \GL(n,\C) \times \GL(n,\C)	\\
		g &\mapsto& \nu(g) \times \nu(g).
\end{eqnarray*}
We remind the reader that any finite dimensional representation of $\SL(2,\C)$ is self-dual.
Consequently, we shall  write $\nu$ in place of $\widetilde{\nu}$.

We shall now define the parameter set
$\Psi(U_{E/F}(N), \eta_{\chi_\kappa})$ to be the set %
 consisting of pairs
$\psi = (\psi^N, \widetilde{\psi})$ where $\psi^N \in \widetilde{\Psi}(N)$ and
\[
	\widetilde{\psi} : \mathcal{L}_{\psi^N} \times \SL(2,\C) \rightarrow {}^L U_{E/F}(N)
\]
is an $L$-homomorphism (considered up to $\widehat{U}$-conjugacy) such that
\[
	\widetilde{\psi}^N = \eta_{\chi_\kappa} \circ \widetilde{\psi}
\]
Denote by $\Psi_2(U_{E/F}(N), \eta_{\chi_\kappa})$ the subset of $\psi=(\psi^N,\tilde\psi)$ such that
$\psi^N$ is elliptic (i.e. all $\ell_i$ are 1 and $J_{\psi^N}$ is empty in \eqref{eq:psi-global-dec}).

  The group $\cL_\psi$ is a substitute for the conjectural global Langlands group so that the datum $\tilde{\psi}$ allows to define various invariants associated with $\psi = (\psi^N, \tilde{\psi})\in \Psi(U_{E/F}(N), \eta_{\chi_\kappa})$, which is only a formal parameter.
    If $\psi_1=\psi_2$ in the parameter set then $\psi_1^N=\psi_2^N$ in $\tilde \Psi(N)$, so in particular there is an isomorphism $\cL_{\psi_1}\simeq \cL_{\psi_2}$ which is canonical (since there is a unique bijection between $\{K_{\psi_1}\}$ and $\{K_{\psi_2}\}$ matching the $\star$-orbits of simple parameters).

   The centralizer group for $\psi$ and its variants may be defined as follows, where we put $G^*:=U_{E/F}(N)$.%
   $$S_\psi:=\Cent(\Im \tilde \psi,\hat G^*),\quad \ol S_\psi:=S_\psi/Z(\hat G^*)^\Gamma,$$
   $$\cS_\psi:=\pi_0(S_\psi),\quad \ol \cS_\psi:=\pi_0(\ol S_\psi),$$
  $$\srad_\psi:=(S_\psi\cap \hat G_{\der}^*)^0,\quad S^\natural_\psi:=S_\psi/\srad_\psi.$$
  We also write $S_\psi(G^*)$ and so on if the group $G^*$ is to be emphasized.
  If two parameters $\psi_1=(\psi_1^N,\tilde\psi_1)$ and $\psi_2=(\psi_2^N,\tilde\psi_2)$ are equal then $\tilde\psi_1$ and $\tilde\psi_2$ are $\hat G^*$-conjugate by definition. So there is an isomorphism $S_{\psi_1}\simeq S_{\psi_2}$, which is canonical up to $S_{\psi_1}$-conjugacy. Similarly the elements of $\ol S_\psi$, $\cS_\psi$, $\ol \cS_\psi$, $\srad_\psi$, and $S^\natural_\psi$ are well-defined up to inner automorphisms of $S_\psi$.

  To describe the group $S_\psi$ explicitly, consider
   the decomposition of $\psi^N$ into irreducibles as in \eqref{eq:psi-global-dec} (with $\psi^N$ in place of $\psi$). Define (cf. \eqref{eq:kappa(psi)-kappa(phi)}) $$\kappa(\psi_i):=(-1)^{N_i-m_i-n_i-1}\kappa(\mu_i).$$
    Take a partition $I_{\psi^N}=I_{\psi^N}^+\coprod I_{\psi^N}^-$ such that $i\in I_{\psi^N}$ belongs to $I_{\psi^N}^+$ (resp. $I_{\psi^N}^-$) if $\kappa(\psi_i)=\kappa (-1)^{N-N_i}$ (resp. otherwise). It follows from the analogue of Lemma \ref{lem:image-xi_chi} for $\cL_\psi$ in place of $\cL_E$, using an analogous argument for obtaining \eqref{eq:kappa(psi)-kappa(phi)}, that the parameter $\tilde \psi_i$ factors through $\eta_{\kappa}(\psi_i)$. As in the local case
     we have that $\ell_i$ is even for each $i\in I_{\psi^N}^-$ and that
   \begin{equation}\label{e:global-centralizer}
  S_\psi\simeq \prod_{i\in I_{\psi^N}^+} \O(\ell_i,\C) \times \prod_{i\in I_{\psi^N}^-} \Sp(\ell_i,\C)\times \prod_{j\in J_{\psi^N}} \GL(\ell_j,\C).
   \end{equation}
  Again the group $Z(\hat G^*)^\Gamma=\{\pm 1\}$ is embedded diagonally into the right hand side and
  $$\cS_\psi\stackrel{\tx{can}}{\simeq} S^\natural_\psi\simeq (\Z/2\Z)^{|I_{\psi^N}^+|},\quad \ol{\cS}_\psi\simeq \left\{
  \begin{array}{cc}
    (\Z/2\Z)^{|I_{\psi^N}^+|-1}, & \forall_{i\in I_{\psi^N}^+},~ 2|\ell_i,\\
    (\Z/2\Z)^{|I_{\psi^N}^+|}, & \mbox{otherwise}.
  \end{array}\right.$$

  It is useful to consider a finer chain of parameter sets in $\Psi(G^*,\eta_{\chi})$ for the later trace formula argument. We write
  \begin{equation}\label{e:Psi-chain-global-U}
    \Psi_{\simp}(G^*,\eta_{\chi})\subset \Psi_2(G^*,\eta_{\chi})\subset \Psi_{\el}(G^*,\eta_{\chi})\subset \Psi_{\disc}(G^*,\eta_{\chi})\subset \Psi(G^*,\eta_{\chi}),
  \end{equation}
  which are defined in terms of the group $\ol{S}_\psi$ as follows, cf. \cite[\S4.1]{Arthur}.
  \begin{eqnarray}
    \Psi_{\simp}(G^*,\eta_{\chi})&:=&\{\psi\in \Psi(G^*,\eta_{\chi})~:~ |\ol S_\psi|=1\}, \nonumber \\
    \Psi_{\el}(G^*,\eta_{\chi})&:=&\{\psi\in \Psi(G^*,\eta_{\chi})~:~ \exists s\in \ol{S}_{\psi,\sspl},~|\ol S_{\psi,s}|<\infty\},\nonumber \\
    \Psi_{\disc}(G^*,\eta_{\chi})&:=&\{\psi\in \Psi(G^*,\eta_{\chi})~:~ |Z( S_\psi)|<\infty\}. \nonumber
  \end{eqnarray}
  Here $S_{\psi,s}$ denotes the centralizer of $s$ in $S_\psi$. The set $\Psi_2(G^*,\eta_{\chi})$ was already defined and can also be characterized as the subset of $\psi\in \Psi(G^*,\eta_{\chi})$ such that $\ol S_\psi$ is finite.

  We have not talked about the trace formula yet but motivate the above definition by means of the trace formula. A parameter $\psi$ in $\Psi_{\simp}(G^*,\eta_{\chi})$ (resp. $\Psi_2(G^*,\eta_{\chi})$) is supposed to contribute to the discrete spectrum of $\GL(N,\A_E)$ (resp. $G^*(\A_F)$). The set $\Psi_{\el}(G^*,\eta_{\chi})$ consists of $\psi$ which contributes to the discrete spectrum of an elliptic endoscopic group of $G^*$. The condition $\psi\in\Psi_{\disc}(G^*,\eta_{\chi})$ should mean that the $\psi$-part of the discrete part of the trace formula for $G^*$ does not vanish identically (i.e. $I^{G^*}_{\disc,\psi}\neq 0$ in later notation).
  The reader is cautioned that $\Psi_{2}(G^*,\eta_{\chi})$ (resp. $\Psi_{\el}(G^*,\eta_{\chi})$) is not the intersection of $\Psi_{2}(G(N))$ (resp. $\Psi_{\el}(G(N))$) with $\Psi(G^*,\eta_{\chi})$ though it is true that $\Psi_{\simp}(G^*,\eta_{\chi})=\Psi(G^*,\eta_{\chi})\cap \tilde \Psi_{\simp}(N)$.
\begin{lem}\label{lem:global-disc-param}
  Let $\psi=(\psi^N,\tilde\psi)\in \Psi(G^*,\eta_\chi)$. Suppose that $\ell_1,...,\ell_r\in \Z_{\ge1}$ are the multiplicities of simple factors in the decomposition of $\psi^N$. Then $\psi\in \Psi_2(G^*,\eta_\chi)$ if and only if $\ell_i=1$ for all $1\le i\le r$ and $I^-_{\psi^N}=J_{\psi^N}=\emptyset$.
\end{lem}

\begin{proof}
  From the explicit description of \eqref{lem:global-disc-param} and the finiteness of $Z(\hat G^*)^\Gamma$ it is clear that $\ol{S}_\psi$ is finite if and only if $S_\psi$ is finite if and only if all $\ell_i$ are equal to 1 and both $I^-_{\psi^N}$ and $J_{\psi^N}$ are empty.
\end{proof}

  We say that $\psi=(\psi^N,\tilde\psi)$ is \textbf{generic} if $\psi^N\in \tilde\Phi(N)$, i.e. if $\psi^N$ is generic. Passing to subsets of generic parameters in \eqref{e:Psi-chain-global-U}, we obtain a chain of sets
   $$\Phi_{\simp}(G^*,\eta_{\chi})\subset \Phi_2(G^*,\eta_{\chi})\subset \Phi_{\el}(G^*,\eta_{\chi})\subset \Phi_{\disc}(G^*,\eta_{\chi})\subset \Phi(G^*,\eta_{\chi}).$$

\subsubsection{Localizing global parameters for $U_{E/F}(N)$}\label{subsub:localizations-U}

Fix characters $\chi_{+} \in \mathcal{Z}_E^+ $ and $\chi_{-} \in \mathcal{Z}_E^-$ as above. Let $\kappa\in \{\pm 1\}$ and
 consider a parameter $\psi = (\psi^N, \widetilde{\psi}) \in \Psi(U_{E/F}(N),\eta_{\chi_\kappa})$.
If $v$ is a place of $F$,
we would like to define the localization $\psi_v \in \Psi^+(U_{E_v/F_v}(N))$ as the ($\widehat{U}_{E_v/F_v}(N)$-conjugacy class of)
$L$-homomorphism
\[
	\psi_v : L_{F_v} \times \SU(2) \rightarrow {}^L U_{E_v/F_v}(N)
\]
such that $\psi_v^N = \eta_{\chi_\kappa} \circ \psi_v$. If such a homomorphism exists, it is necessarily unique by Lemma \ref{lem:image-xi_chi}.

Consider the easier case where $v$ splits as $w\ol w$ in $E$ so that $E_v=F_w\times F_{\ol w}$. (Recall that $w$ is determined by the composite of the distinguished embeddings $E\hra \ol{F}$ and $\ol{F}\hra \ol{F}_v$.) Then $\psi^N\in \Psi(GL(N)_E)$ has localizations $\psi^N_w\in \Psi(\GL(N)_{E_w})$ and $\psi^N_{\ol w}\in \Psi(\GL(N)_{E_{\ol w}})$, cf. \S\ref{subsub:localizations-GL}. Let $^L \xi_w: {}^L \GL(N)_{E_w}\simeq {}^L U_{E_v/F_v}(N)$ be the isomorphism induced by the isomorphism $\xi_w: U_{E_v/F_v}(N)\simeq \GL(N)_{E_w}$. Then
$$\psi_v: L_{F_v} \times \SU(2) \stackrel{\psi^N_w}{\ra} {}^L \GL(N)_{E_w} \stackrel{^L \xi_w}{\rightarrow} {}^L U_{E_v/F_v}(N)$$
(or its suitable $\hat U_{E_v/F_v}(N)=\GL(N,\C)$-conjugate) is the desired localization.

When $v$ does not split in $E$ the existence of a localization $\psi_v$ turns out to be highly nontrivial and is only proved as a part of an inductive argument proving all the main theorems in \cite{Mok} (like Proposition \ref{prop:1st-seed-thm} above), as it was the case in \cite{Arthur} for symplectic and orthogonal groups. More specifically, the statement is contained in Theorem 2.4.10 and Corollary 2.4.11 in \cite{Mok}.

We summarize this discussion as follows.

\begin{prop}(2nd seed theorem)\label{prop:2nd-seed-thm}
  For each $\psi\in\Psi(U_{E/F}(N),\eta_\kappa)$ there exists $\psi_v \in \Psi^+(U_{E_v/F_v}(N))$ as above, i.e. such that $\psi_v^N = \eta_{\chi_\kappa} \circ \psi_v$. The isomorphism class of $\psi_v$ is uniquely determined.
\end{prop}

  Proposition \ref{prop:2nd-seed-thm} can be used to produce a map $L_{F_v}\ra \cL_\psi$ at each $v$ as well as a localization map for centralizer groups $S_\psi\ra S_{\psi_v}$ and its variants. We only sketch the idea, keeping the same notation as in \S\ref{subsub:global-param-U}. Before the proposition our definition and construction in \S\ref{subsub:global-param-U} yield the following diagram without the dotted arrows, in which every triangle and rectangle (not involving the dotted arrows) commute.
   $$\xymatrix{
  L_{F_v}\times \SU(2) \ar@{-->}[d] \ar@{-->}[r]^-{ \psi_v} \ar@/^2pc/[rr]^-{\psi^N_v} & {}^L U_{E_v/F_v}(N) \ar[d]\ar[r]^-{\eta_\chi}
  &  {}^L G_{E_v/F_v}(N) \ar[d]\ar[r] & W_{F_v} \ar[d]\\
  \cL_\psi \times \SL(2,\C)\ar[r]^-{\tilde\psi} \ar@/_2pc/[rr]_-{\tilde \psi^N} & {}^L U_{E/F}(N) \ar[r]^-{\eta_\chi} & {}^L G_{E/F}(N) \ar[r] &  W_F}$$
  Proposition \ref{prop:2nd-seed-thm} provides us with the top dotted arrow such that the top triangle commutes (i.e. $\psi^N_v=\eta_\chi\psi_v$) and the left dotted arrow such that the rectangle enclosed by the first and third vertical arrows commutes. Since $\eta_\chi$ is an injection (locally and globally), it follows that the leftmost rectangle also commutes. Hence the full diagram commutes.

  The diagram gives us localization maps for the centralizer groups. The second vertical arrow carries $\Im (\psi_v)$ into $\Im (\tilde\psi)$, inducing a map
  $$S_\psi\ra S_{\psi_v},\quad\mbox{thus also}\quad \cS_\psi\ra \cS_{\psi_v}~~\mbox{and}~~S^\natural_\psi\ra S^\natural_{\psi_v}.$$
  The map $S_\psi\ra S_{\psi_v}$ sends $Z(\hat U_{E/F}(N))^{\Gamma}$ into $Z(\hat U_{E/F}(N))^{\Gamma_v}$, so we also have maps
  $$\ol S_\psi\ra \ol S_{\psi_v}\quad\mbox{and}\quad \ol \cS_\psi\ra \ol \cS_{\psi_v}.$$
  Each localization map is canonical up to an inner automorphism of $S_{\psi_v}$.

\subsubsection{Parameters for endoscopic groups of $U_{E/F}(N)$}\label{subsub:param-end-U}

  So far we have discussed the parameters sets only for $G^*=U_{E/F}(N)$. We will introduce the analogues for endoscopic groups of $G^*$ (which include Levi subgroups) in both global and local settings.

  First we consider the global case.
  Let $(G^\fke,s^\fke,\eta^\fke)\in \cE(G^*)$.
  Define the parameter set $\Psi(G^\fke,\eta^\fke)$ to be
  \begin{equation}\label{eq:param-end-U-glo}\Psi(G^\fke,\eta^\fke):= \left\{(\psi^N,\tilde\psi^\fke), ~ \psi^N\in \Psi(N),~\tilde\psi^\fke:\cL_{\psi^N}\ra {}^L G^\fke,~\mbox{s.t.}~\eta^\fke\circ \tilde\psi^\fke=\tilde\psi^N\right\},\end{equation}
   where $\tilde\psi^\fke$ is viewed as a $\hat G^\fke$-conjugacy class of $L$-morphisms. In other words $(\psi^N_1,\tilde\psi^\fke_1)$ and $(\psi^N_2,\tilde\psi^\fke_2)$ are considered the same element if $\psi^N_1=\psi^N_2$ and if $\tilde\psi^\fke_1$ and $\tilde\psi^\fke_2$ are $\hat G^\fke$-conjugate. %
  For a more explicit description we fix $\chi_+\in \cZ^+_E$ and $\chi_-\in \cZ^-_E$ globally and locally. Let $\kappa\in \{\pm1\}$ so that we have $^L G^* \hra {}^L G_{E/F}(N)$. Notice that there is an isomorphism
  \begin{equation}\label{eq:Gfke-product}G^\fke\simeq \prod_{i=1}^a U_{E/F}(n_i)\times \prod_{j=1}^b G_{E/F}(m_j)\end{equation}
  for suitable $a,b\ge 0$,  $n_1,...,n_a\ge 1$, and $m_1,...,m_b\ge 1$ such that $N=\sum_i n_i + 2\sum_j m_j$. Let $\underline \kappa=(\kappa_i)_{i=1}^a$ be a collection of signs given by $\kappa_i=(-1)^{N-n_i}\kappa$.
   Without disturbing the weak isomorphism class of the endoscopic triple we can choose $\eta^\fke$ such that the following diagram commutes.
$$\xymatrix{^L\left( \prod_i U_{E/F}(n_i)\times \prod_j G_{E/F}(m_j)\right) \ar[rr]^-{( \eta_{\kappa_i},\chi_\kappa\tilde \mu_j)} \ar[d]^-{\eta^\fke} & &
    {}^L\left( \prod_i G_{E/F}(n_i)\times \prod_j G_{E/F}(2m_j)\right) \ar[d] \\ {}^L U_{E/F}(N) \ar[rr]^-{\eta_\kappa} & & {}^L G_{E/F}(N)}$$
   Here the right vertical map is given as in \S\ref{subsub:endo-unitary}, and $\chi_\kappa\tilde \mu_j$ is the twist of the $L$-morphism $\tilde \mu_j$ in \S\ref{subsub:global-param-U} by $\chi_\kappa$. With the aid of the above diagram we identify
   $$\Psi(G^\fke,\eta^\fke)=\prod_{i=1}^a \Psi(U_{E/F}(n_i),\eta_{\kappa_i}) \times \prod_{j=1}^b \Psi(G_{E/F}(m_j)).$$
   We define $\Psi_2(G^\fke,\eta^\fke)$ by the product of the subsets of discrete parameters on the right hand side. The subset of generic parameters is denoted $\Phi(G^\fke,\eta^\fke)$, $\Phi_2(G^\fke,\eta^\fke)$, etc. For $\psi\in \Psi(G^\fke,\eta^\fke)$ we define the centralizer group $S_\psi$ and its variants as in the case $G^\fke=U_{E/F}(N)$.

  Next we put ourselves in the local case, where $\Psi(G^\fke)$ and $\Psi_2(G^\fke)$ are defined by a general discussion of \S\ref{subsub:L-param-A-param} without reference to an $L$-morphism $\eta^\fke:{}^L G^\fke \ra {}^L G^*$. Still once we fix a choice of $\eta^\fke$, we can show that the two descriptions in the global setting are valid. Namely there is a natural bijection between $\Psi(G^\fke)$ and the right hand side of \eqref{eq:param-end-U-glo} with $\cL_{\psi^N}$ replaced with $L_F\times \SU(2)$. The explicit description of $G^\fke$ as in \eqref{eq:Gfke-product} is again available. Accordingly there is a decomposition $\Psi(G^\fke)=\prod_{i=1}^a \Psi(U_{E/F}(n_i))\times \prod_{j=1}^b \Psi(G_{E/F}(m_j))$ and similarly for discrete parameters.

  Just like when $G^\fke=U_{E/F}(N)$ (or a general linear group), we have a localization map $\Psi(G^\fke,\eta^\fke)\ra \Psi(G^\fke_v)$ compatibly with the product decompositions in the source and the target.

\subsubsection{Relevance of global parameters for $U_{E/F}(N)$}\label{subsub:relevance-global}

  So far our global discussion has been only about quasi-split groups. Now we consider the notion of relevant parameters for inner forms.
  As above let $\kappa\in \{\pm 1\}$ and $\psi = (\psi^N, \widetilde{\psi}) \in \Psi(U_{E/F}(N),\eta_\kappa)$.
 Now consider an inner twist $(G,\xi)$ of $G^*=U_{E/F}(N)$.
  The global parameter $\psi$ is said to be \textbf{$(G,\xi)$-relevant} if it is relevant locally everywhere, i.e. if $\psi_v$ is $(G_v,\xi_v)$-relevant at every place $v$.
  We write $\Psi(G^*,\eta_\kappa)_{(G,\xi)\rel}$, or $\Psi(G^*,\eta_\kappa)_{\Grel}$ if $\G$ is an extended pure inner twist extending $(G,\xi)$,  for the subset of $(G,\xi)$-relevant parameters in $\Psi(G^*,\eta_\kappa)$. Similarly the notation $\Psi_2(G^*,\eta_\kappa)_{(G,\xi)\rel}$ etc will have the obvious meaning.

\begin{lem}\label{l:relevance-and-Levi}
  Let $\psi\in \Psi(G^*,\eta_\kappa)$ and $M^*\subset G^*$ a Levi subgroup. Suppose that $\psi$ comes from a parameter on $M^*$ and that $\psi$ is relevant for an inner twist $(G,\xi)$ of $G^*$. Then $M^*$ transfers to $G$ in the sense of Definition \ref{d:Levi-transfer}.
\end{lem}

\begin{proof}
  By definition $M^*$ transfers to $G$ locally at every place of $F$. We conclude by Lemma \ref{lem:transfer-Levi-locglo}.
\end{proof}

\subsubsection{Canonical global parameters for $U_{E/F}(N)$}\label{subsub:canonical-def}

  Our definition of various global parameter sets has not been optimal in that it depends on the choice of $\chi_+$ and $\chi_-$, which has been fixed thus far. To remove the dependence we consider all possible choices of $\chi_+$ and $\chi_-$ simultaneously. Define $\Psi(G^*)$ to be the set of equivalence classes\footnote{In the global setup we refrain from using $\Psi(G^*)$ as an abbreviation for $\Phi(G^*,\eta_\chi)$ so as to avoid confusion.}
  \begin{equation}\label{eq:Psi(G^*)-canonical}
  \Psi(G^*):=\left(\coprod_{\chi\in \cZ_E^+\coprod \cZ_E^-} \Psi(G^*,\eta_\chi)\right) / \sim,
  \end{equation}
  where $\psi=(\psi^N,\tilde \psi)\in \Psi(G^*,\eta_\chi)$ and $\psi'=((\psi')^N,\tilde \psi')\in\Psi(G^*,\eta_{\chi'})$ are considered equivalent if
  \begin{itemize}
    \item $(\psi')^N=\psi^N\otimes \chi'\chi^{-1}$ and
    \item the $\hat G^*$-orbits of $\tilde \psi$ and $\tilde \psi'$ are the same via the canonical isomorphism $\cL_\psi\simeq \cL_{\psi'}$.
  \end{itemize}
  Let us elaborate on these conditions. When we write $\psi^N=\boxplus_{i=1}^r \ell_i \mu_i\boxtimes \nu_i$ as usual, the character twist $\psi^N\otimes \chi'\chi^{-1}$ is defined to be $$\boxplus_{i=1}^r \ell_i \left(\mu_i\otimes(\chi'\chi^{-1}\circ\det)\right)\boxtimes \nu_i.$$
  We explained in \S\ref{subsub:localizations-U} that $\cL_\psi\simeq \cL_{\psi'}$ canonically if $\psi^N=(\psi')^N$. Now it is not hard to see that the same remains true when they differ by $\chi'\chi^{-1}$. (Since $\chi'\chi^{-1}$ is conjugate self-dual, the group $^L H_k$ with  its projection onto $W_F$ for $\psi$ and $\psi'$ are identified for each $k$ up to reordering $k$'s.)

\begin{lem}\label{lem:localization-compat}
  Suppose that $\psi\in\Psi(G^*,\eta_\chi)$ and $\psi'\in\Psi(G^*,\eta_{\chi'})$ are equivalent. Let $v$ be a place of $F$. Then $\psi_v=\psi'_v$ in $\Psi^+_v(G^*)$. In particular the localization maps patch together to a map
  $$\Psi(G^*)\ra \Psi_v(G^*),\quad [\psi]\mapsto [\psi]_v.$$
\end{lem}

\begin{proof}
  The localizations $\eta_\chi\circ \psi_v$ and $\eta_{\chi'}\circ \psi'_v$ correspond to $\psi^N_v$ and $ (\psi')^N_v$ respectively, if we view the representations $\psi^N_v$ and $ (\psi')^N_v$ of $\GL(N,E_v)$ as $L$-parameters $L_{F_v}\ra {}^L G(N)$ via the local Langlands correspondence and Shapiro's lemma. Since $(\psi')_v^N\simeq \psi^N_v\otimes \chi'_v\chi_v^{-1}$ by assumption, the correspondence between $\eta_{\chi'}\circ \psi'_v$ and $ (\psi')^N_v$ implies that $\eta_\chi\circ \psi'_v$ corresponds to $\psi^N_v$ (noting that $\chi'\chi^{-1}$ corresponds to the parameter $W_F\ra {}^L G(1)$ such that $w\mapsto (\chi'\chi^{-1}(w),(\chi')^{-1}\chi(w))\rtimes w$ on $W_{E}$ and $w_c\mapsto (\kappa'\kappa,1)\rtimes w_c$, cf. \S\ref{subsub:endo-unitary}). We deduce from the injectivity in Lemma \ref{lem:image-xi_chi} that $\psi_v$ and $\psi'_v$ are isomorphic parameters.
\end{proof}

  Further the following are verified rather easily.

 \begin{itemize}
   \item If $\psi$ and $\psi'$ are equivalent parameters as above then we have isomorphisms $S_\psi\simeq S_{\psi'}$ and $\ol{S}_\psi\simeq \ol{S}_{\psi'}$ canonical up to an inner automorphism of $S_\psi$. %
       Hence we may associate to each $[\psi]\in \Psi(G^*)$ the groups
        \begin{equation}\label{eq:S_[psi]etc}
        S_{[\psi]},\quad \ol S_{[\psi]},\quad \cS_{[\psi]},\quad \ol\cS_{[\psi]}, \quad S^\natural_{[\psi]}.
        \end{equation}
   \item When $\diamondsuit\in \{\simp,2,\el,\disc\}$,
         the set $\Psi_{\diamondsuit}(G^*)$ is well-defined as the image of $\Psi_{\diamondsuit}(G^*,\eta_\chi)$ in $\Psi(G^*)$ independently of the choice of $\chi$, since the set is characterized by means of the group $\ol S_\psi$. The definition of $\Phi_{\diamondsuit}(G^*)$ is similar.
   \item The localization map induces $S_{[\psi]}\ra S_{[\psi_v]}$ (canonical up to $S_{[\psi_v]}$-conjugacy) and similarly for the other groups in \eqref{eq:S_[psi]etc}.
   \item For each inner twist $(G,\xi)$ of $G^*$, the $(G,\xi)$-relevance property is invariant under the equivalence relation. So $\Psi_{\diamondsuit}(G^*)_{(G,\xi)\rel}$ is well-defined.
 \end{itemize}

\subsection{A correspondence of endoscopic data}\label{sub:endo-correspondence}

  Here we recall a bijective correspondence between two kinds of endoscopic data from a discussion in \cite[\S1.4]{Arthur} in the case of ordinary endoscopy for $U_{E/F}(N)$.

   Some terminology and notation need to be set up. For the moment we focus on the local case and will remark below what to do in the global case. Let $G^*$ be a connected reductive quasi-split group over $F$. We assume that
   $G^*_{\der}$ is \emph{simply connected}, which suffices for our purpose. (See \cite[\S4.8]{Arthur} for the analogous discussion without the assumption.)
   Two ordinary endoscopic triples $\fke_1=(G^\fke_1,s^\fke_1,\eta^\fke_1)$ and $\fke_2=(G^\fke_2,s^\fke_2,\eta^\fke_2)$ for $G^*$ are considered strictly equivalent (resp. weakly equivalent, resp. equivalent) if $s^\fke_1=zs^\fke_2$ for $z=1$ (resp. some $z\in Z(\hat G^*)^\Gamma$, resp. some $z\in Z(\hat G^*)$) and $\eta^\fke_1({}^L G^\fke_1)=\eta^\fke_2({}^L G^\fke_2)$ as subgroups of $^L G^*$. The equivalence here is different from the isomorphism defined in \S\ref{subsub:endoscopic-triples} in that there is no conjugation by an element of $\hat G^*$.
   For simplicity we further assume that that weak equivalence is the same as equivalence and that weak isomorphism is the same as isomorphism. While this may not be true in general, in our special case this is true and can be checked explicitly. Henceforth we only distinguish between equivalence/isomorphism and strict equivalence/isomorphism.

   Define $E(G^*)$ (resp. $\ol{E}(G^*)$) to be the set of strict equivalence (resp. equivalence) classes of endoscopic data for $G^*$. Of course $E(G^*)$ is just the set of all endoscopic data but our terminology is chosen in parallel with the definitions of \S\ref{subsub:endoscopic-triples} as the set of $G^*$-orbits on $E(G^*)$ (resp. $\ol{E}(G^*)$) is none other than $\cE(G^*)$ (resp. $\ol{\cE}(G^*)$) there. Let $F(G^*)$ denote the set of $A$-parameters $\cL_F\ra {}^L G^*$. (Here we literally mean the parameters, not the isomorphism classes thereof.)%
   For each $\psi\in F(G^*)$ write $\ol{S}_{\psi,\sspl}$ for the set of semisimple elements in $\ol{S}_\psi$. Define two sets
   \begin{eqnarray}
    X(G^*) &:=& \{(\fke=(G^\fke,s^\fke,\eta^\fke),\psi^\fke):~\fke\in E(G^*),~\psi^\fke \in F(G^\fke) \}, \nonumber\\
    Y(G^*) &:=& \{(\psi,s):~\psi\in F(G^*),~s\in S_{\psi,\sspl} \}.\nonumber
   \end{eqnarray}
   Similarly define $\ol{X}(G^*)$, and $\ol{Y}(G^*)$ with $E(G^*)$ and $S_{\psi,\sspl}$ replaced by $\ol{E}(G^*)$ and $\ol{S}_{\psi,\sspl}$, respectively. Naturally $X(G^*)$ and $Y(G^*)$ are left $\hat G^*\times Z(\hat G^*)^\Gamma$-sets by the following formulas: for $g\in G^*$ and $z\in Z(\hat G^*)^\Gamma$,
   \begin{eqnarray}
    g\cdot (G^\fke,s^\fke,\eta^\fke),\psi^\fke) &=& (G^\fke,g s^\fke g^{-1},g\eta^\fke g^{-1}),\psi^\fke),\nonumber\\
    g\cdot (\psi,s)&=&(g\psi g^{-1},g s g^{-1}), \nonumber\\
    z\cdot (G^\fke,s^\fke,\eta^\fke),\psi^\fke) &=& (G^\fke,z s^\fke,\eta^\fke),\psi^\fke),\nonumber\\
    z\cdot (\psi,s)&=&(\psi ,zs ).\nonumber
   \end{eqnarray}
   The right hand side of the first formula is understood as the representative in $E(G^*)$ in its strict equivalence class fixed above.
   We see that $\ol{X}(G^*)$, and $\ol{Y}(G^*)$ are none other than the quotient sets of $X(G^*)$ and $Y(G^*)$ by the action of $Z(\hat G^*)^\Gamma$.   Put
   $$\cX(G^*):=\hat G^* \dblbs X(G^*),\quad \cY(G^*):=\hat G^* \dblbs Y(G^*),$$
   standing for the sets of $\hat G^*$-orbits in $X(G^*)$ and $Y(G^*)$, respectively. Likewise $\ol\cX(G^*):=\hat G^* \dblbs \ol X(G^*),\quad \ol\cY(G^*):=\hat G^* \dblbs \ol Y(G^*)$. Note that there is a canonical identification $\cE(G^*)=\hat G^*\dblbs E(G^*)$ and $\ol\cE(G^*)=\hat G^*\dblbs \ol E(G^*)$. It is easy to see that
      \begin{eqnarray}
    \cX(G^*) &=& \{(\fke=(G^\fke,s^\fke,\eta^\fke),\psi^\fke):~\fke\in \cE(G^*),~\psi^\fke \in \tilde\Psi(G^\fke) \}, \label{eq:cX(G^*)cY(G^*)-def}\\
    \cY(G^*) &=& \{(\psi,s):~\psi\in \Psi(G^*),~s\in S_{\psi,\sspl}/\sim \},\nonumber\\
    \ol\cX(G^*) &=& \{(\fke=(G^\fke,s^\fke,\eta^\fke),\psi^\fke):~\fke\in \ol\cE(G^*),~\psi^\fke \in \tilde\Psi(G^\fke) \}, \nonumber\\
    \ol\cY(G^*) &=& \{(\psi,s):~\psi\in \Psi(G^*),~s\in \ol{S}_{\psi,\sspl}/\sim \},\nonumber
   \end{eqnarray}
   where $S_{\psi,\sspl}/\sim$ and $\ol{S}_{\psi,\sspl}/\sim$ is the set of semisimple conjugacy classes in $S_{\psi}$ and $\ol{S}_{\psi}$, respectively, and $\tilde\Psi(G^\fke)$ is the quotient set of $\Psi(G^\fke)$ modulo the action of $\Out_G(G^\fke)$ (resp. $\ol{\Out}_G(G^\fke)$).

\begin{lem}\label{lem:local-endo-data-corr} Let $G^*$ be a connected reductive quasi-split group over a local field with simply connected derived subgroup. Then the natural map
\begin{equation}\label{eq:X(G^*)toY(G^*)}
     X(G^*)\ra Y(G^*),\quad (\fke,\psi^\fke)\mapsto (\eta^\fke\psi^\fke,s^\fke).
   \end{equation}
   is a $\hat G^*\times Z(\hat G^*)^\Gamma$-equivariant bijection, thus induces a $\hat G^*$-equivariant bijection $\ol X(G^*)\simeq \ol Y(G^*)$, a $Z(\hat G^*)^\Gamma$-equivariant bijection $\cX(G^*)\simeq \cY(G^*)$, and a bijection $\ol \cX(G^*)\simeq \ol \cY(G^*)$.
\end{lem}

\begin{rem}
  Two (weakly) equivalent pairs $(\fke_1,\psi_1^\fke)$ and $(\fke_2,\psi^\fke_2)$ do not have the same image in $Y(G^*)$ under \eqref{eq:X(G^*)toY(G^*)} in general. This is why we prefer to fix the set of representatives for $E(G^*)$.
\end{rem}

  \begin{proof}
    The equivariance is immediately verified.
    It suffices to show that \eqref{eq:X(G^*)toY(G^*)} is a bijection, which we do by
    constructing the inverse map. Let $(\psi,s)\in Y(G^*)$. %
    Define $\cG^\fke:=\hat G^\fke \cdot \psi(L_F\times \SL(2,\C))$, which is a subgroup of ${}^L G^*$. Then $\cG^\fke$ is a split extension of $W_F$ by $\hat G^\fke$, where a splitting $W_F\ra \cG^\fke$ is given by $w\mapsto \psi(w)$. The induced (finite) Galois action on $\hat G^\fke$ determines a quasi-split group $G^\fke$ over $F$. Moreover we have an isomorphism between $\eta^\fke:{}^L G^\fke \rw \cG^\fke$ since $G_{\der}$ is simply connected. (See \cite[Prop 1]{Lan79}; also see the second
    last paragraph above \cite[Lem 2.2.A]{KS99}.) On the other hand, the image of $\psi$ is contained in $\cG^\fke$ by construction, so $\psi=\eta^\fke \psi^\fke$ for some $\psi^\fke\in F(G^\fke)$. Replacing $(G^\fke,s^\fke,\eta^\fke)$ with its (quasi-)equivalent representative in $E(G^*)$, we get a map $Y(G^*)$ to $X(G^*)$. It is routine to verify that \eqref{eq:X(G^*)toY(G^*)} is inverse to the latter map.
  \end{proof}

   As we already remarked, the above discussion applies to unitary groups as the two assumptions at the start of this subsection are satisfied. For instance we can take $G^*=U_{E/F}(N)$, where $E$ is a quadratic algebra over a local field $F$. Now we assume that $E$ is a quadratic field extension over a \emph{global} field $F$. (Still $G^*=U_{E/F}(N)$.) The description of $\cX(G^*)$ and $\cY(G^*)$ continues to make sense, where the global parameter sets are understood as in \S\ref{subsub:global-param-U}. We take it as the definition of the sets $\cX(G^*)$ and $\cY(G^*)$. The global map $\cX(G^*)\ra \cY(G^*)$ can be analogously defined as follows. Let $(\fke,\psi^\fke)\in \cX(G^*)$, where we write $\psi^\fke=(\psi^N,\tilde\psi^\fke)$ following the convention of \S\ref{subsub:global-param-U}. Then we associate the element of $\cY(G^*)$ given by $(\psi^N,\eta^\fke\tilde\psi^\fke)\in \Psi(G^*)$ and (the conjugacy class of) $s^\fke\in \ol{S}_{\psi,\sspl}$. One checks that this map $\cX(G^*)\ra \cY(G^*)$ is well-defined and that there is an inverse map as in the proof of Lemma \ref{lem:local-endo-data-corr}. Since the arguments are very similar we omit the details and just record the result.

\begin{lem}\label{lem:global-endo-data-corr}
  Let  $G^*=U_{E/F}(N)$ be a global unitary group. Then the above map  $\cX(G^*)\ra \cY(G^*)$ is a bijection.
\end{lem}

  As Arthur remarks in \cite[\S1.4]{Arthur}, the bijection $\cX(G^*)\leftrightarrow \cY(G^*)$ and its variants reduce many questions on the transfer of characters to a study of the groups $S_\psi$. The two lemmas above will be used repeatedly without citing them every time. For instance we will simply say that the pairs $(\fke,\psi^\fke)$ and $(\psi,s)$ correspond or are associated.

\subsection{Results on quasi-split unitary groups}\label{sub:results-qsuni}

  Here is the list of the main results in the quasi-split case that we import from \cite{Mok}. The numbering for the lemmas and theorems below is as in that paper.
\begin{itemize}
  \item The two seed theorems (Theorems 2.4.2, 2.4.10);
  \item The main local theorems (Theorems 2.5.1, 3.2.1);
  \item The main global theorems (Theorems 2.5.2, 2.5.4);
  \item Well-definedness and multiplicativity of normalized intertwining operators (Proposition 3.3.1), c.f. our Lemma 2.2.3;
  \item The local intertwining relation (Theorem 3.4.3 and Corollary 7.4.7), cf. our Theorem \ref{thm:lir};
  \item The twisted local intertwining relation for the twisted group $\tilde G_{E/F}(N)$ (Corollary 3.5.2), c.f. our Proposition \ref{pro:tlir};
  \item The endoscopic expansion of the stable elliptic inner product (Proposition 7.5.3, c.f. \cite[Proposition 6.5.1]{Arthur}).
  \item The stable multiplicity formula (Theorem 5.1.2)
  \item The spectral and endoscopic sign lemmas (Lemmas 5.5.1, 5.6.1, respectively)
\end{itemize}
  In addition we need the analogue of Ban's results (\cite{Ban, Ban2}) for quasi-split and non-quasi-split unitary groups. In the quasi-split case this enters the proof of the main theorems and was stated as Proposition 8.2.5 in \cite{Mok} but without justification. In Appendix \ref{sec:appendix} we prove this in more generality than is needed in \cite{Mok} and our sequel paper on non-generic parameters.

  In the discussion of global parameters we already cited the two seed theorems and the theorem 2.5.4 above. In this subsection we focus on stating the main local theorems. The theorem 2.5.2 above is not logically necessary but its statement is subsumed in Theorem \ref{thm:main-global}. (Of course we are not reproving the theorem 2.5.2 in that that theorem is intertwined with the other main results when it comes to proof.) It should be noted that our main results in \S\ref{sub:main-local-thm} and \S\ref{sub:main-global-thm} are in parallel with the main local theorem above and the theorem 2.5.2. The inner form case of the local intertwining relation will be stated later in Theorem \ref{thm:lir} after setting up some additional definition and notation. The sign lemmas turn out to be pertinent only to the quasi-split groups and will be utilized in the trace formula comparison in \S\ref{sub:std-model} below.

  Before stating the main local theorem in the quasi-split case in full we single out the statement on the existence of stable linear forms.

\begin{prop}[cf. Theorem 3.2.1.(a), \cite{Mok}]\label{prop:local-stable-linear}
  There is a unique family of stable linear forms $f\mapsto f^{G^*}(\psi)$ on $G^*=U_{E/F}(N)$ (as $N$ and $\psi\in \Psi(G^*)$ vary) satisfying a character identity \cite[Thm 3.2.1.(a)]{Mok} relative to twisted endoscopic data in $\cE_\el(\tilde G_{E/F}(N))$. %
\end{prop}

  We informally explain the above character identity, leaving the details to Mok's paper as we will not need to use it explicitly in the arguments of our paper. The transfer of orbital integrals from $\tilde G(N)$ to $U(N)$ leads to a transfer of stable distributions from $\tilde G(N)$ to $U(N)$. The character identity says that for every endoscopic triple $(G^{\tilde\fke},s^{\tilde \fke},\eta^{\tilde \fke})$ with $G^{\tilde\fke}=U_{E/F}(N)$, the linear form $f\mapsto f^{G^{\tilde\fke}}(\psi)$ is the transfer of the (twisted) stable linear form $f^N \mapsto f^N(\eta^{\tilde \fke}\psi)$ (evaluating the twisted character of the conjugate self-dual representation $\pi_{\eta^{\tilde \fke}\psi}$ against $f^N$). It is worth pointing out that the stable linear form $f\mapsto f^{G^*}(\psi)$ is intrinsic to $G^*$, namely that it is independent of how $G^*$ extends to a twisted endoscopic datum. (There are two extensions up to isomorphism.) We will abbreviate $f^{G^*}(\psi)$ as $f(\psi)$ when there is no danger of confusion.

  The proposition is used to make sense of stable linear forms on all endoscopic groups of $U_{E/F}(N)$. Let $H=\Res_{F'/F}\GL(N)$ with $F'=F$ or $F'=E$. For $f\in \cH(H(F))=\cH(\GL(N,F'))$ and $\psi\in \Psi(H,F)=\Psi(\GL(N),F')$, we define $f(\psi):=f(\pi_\psi)$, where $\pi_\psi\in \Pi(\GL(N,F'))$ is as in \eqref{eq:pi_psi-def}. Hence for any endoscopic triple $(G^\fke,s^\fke,\eta^\fke)$ for $U_{E/F}(N)$, the stable linear form $f'(\psi')$ for $f'\in \cH(G^\fke)$ and $\psi'\in \Psi(G^\fke)$ is defined  in the obvious manner since either $G^\fke$ is a finite product of general linear groups (if $E=F\times F$) or a finite product of groups of the form $U_{E/F}(N_i)$ and $\Res_{E/F}\GL(N_j)$ for $E/F$ a quadratic field extension and positive integers $N_i$'s and $N_j$'s.

  We are ready to state the rest of the main local theorem.

\begin{prop}[cf. Theorems 2.5.1, 3.2.1, \cite{Mok}]\label{prop:main-local-qsplit}

\begin{enumerate}
\item Let $\psi \in \Psi(G^*)$. There exists a nonempty finite set $\Pi_\psi(G^*)$ with a map to $\Pi_\tx{unit}(G^*)$ as well as a map to $\tx{Irr}(\ol{\cS}_\psi)$.
\item For each $\pi \in \Pi_\tx{unit}(G^*)$ in the image of $\Pi_\psi(G^*)$, the central character $\omega_\pi : Z(G^*)(F) \rw \C^\times$ has a Langlands parameter given by the composition
\[\xymatrix{ L_F\ar[r]^-{\phi_\psi}&{^L G^*}\ar[r]^-{\tx{det}\rtimes\tx{id}}& \C^\times\rtimes W_F}. \]
\item Let $(G^\fke,s^\fke,\eta^\fke)$ be an endoscopic triple with $s^\fke \in S_\psi$, and let $\psi^\fke\in \Psi(G^\fke)$ be a parameter such that $\eta^\fke\circ\psi^\fke=\psi$. If $f^\fke(\psi^\fke)$ is the stable distribution on $G^\fke(F)$ associated to the parameter $\psi^\fke$, and $f^\fke \in \mc{H}(G^\fke)$ and $f \in \mc{H}(G^*)$ are two functions with $\Delta[\xi,z]$-matching orbital integrals, then we have
    \[ f^\fke(\psi^\fke) = \sum_{\pi \in \Pi_\psi(G^*)} \<\pi,s^\fke\cdot s_\psi\> f(\pi), \]
where $\<\pi,-\>$ denotes the character of the irreducible representation of $\ol{\cS}_\psi$ that $\pi$ corresponds to by 1.
\item If $\psi \in \Phi_{\bdd}(G^*)$, then the map $\Pi_\psi(G^*) \rw \Pi_\tx{unit}(G^*)$ is injective and its image belongs to $\Pi_\tx{temp}(G^*)$. The map $\Pi_\psi(G^*) \rw \tx{Irr}(\ol{\cS}_\psi)$ is also injective, and bijective if $F$ is nonarchimedean.
\item Suppose that $F$ is nonarchimedean, $G^*$ is unramified over $F$, and $\psi$ is unramified. Then the preimage of the trivial character on $\ol{\cS}_\psi$ is the unique unramified representation in $\Pi_\psi(G^*)$ (relative to the distinguished hyperspecial subgroup).
\item As $\psi$ runs over $\Phi_{\bdd}(G^*)$ the sets $\Pi_\psi(G^*)$ are disjoint and exhaust $\Pi_\tx{temp}(G^*)$.
\end{enumerate}

\end{prop}

In addition, for $\pi$ in the $L$-packet for a bounded parameter, we know that $\lg \cdot,\pi\rg=1$ if and only if $\pi$ has Whittaker model with respect to the given Whittaker datum. This is \cite[Cor 9.2.4]{Mok} (cf. \cite[Prop 8.3.2]{Arthur}).

The theorem above corresponds to the case of the extended pure inner twist $(G^*,\id,1)$ in Theorem \ref{thm:locclass-single} below. In that case $\chi_z=1$ so $\Irr(S^\natural_\psi,\chi_z)=\tx{Irr}(\ol{\cS}_\psi)$ in view of Lemma \ref{lem:squot}.

We remark that that the map $\Pi_\psi(G^*) \rw \tx{Irr}(\ol{\cS}_\psi)$ in part 1 depends on the additive character $\psi_F : F \rw \C^\times$. This character, together with the fixed pinning of $G^*$, gives rise to a Whittaker datum, which is used in the normalization of the transfer factors. These in turn influence the transfer $f^\mf{e}$ of $f$, and hence by part 3 of the theorem also the pairing $\<\pi,-\>$, which is just a different notation for the map in part 1.

\subsection{The main local theorem}\label{sub:main-local-thm}

  We are about to state our main local theorem on the local endoscopic classification for inner forms of the unitary group $U_{E/F}(N)$. The results are novel when $N>3$, $N$ is even, and $E/F$ are non-archimedean local fields. (The archimedean case is due to Langlands and Shelstad. In the non-archimedean setting, the case $N\le 3$ is due to Rogawski. In the quasi-split and non-quasi-split cases, Moeglin has obtained partial results in \cite{Moe07,Moe11} and later Mok completed the classification for quasi-split unitary groups (including the global case) in \cite{Mok} adapting Arthur's strategy. Note that $U_{E/F}(N)$ admits a non-quasi-split inner form exactly when $N$ is even.) Nevertheless our theorem includes the archimedean case as well as the split case (i.e. $E=F\times F$) in the statement. The point is to have a precise and consistent normalization for local results in order to provide input for the global theorem.

\subsubsection{Statement of the main local theorem}\label{subsub:statement-local-thm}

   Let $F$ be a local field. Fix a non-trivial character $\psi_F : F \rw \C^\times$. If $E/F$ is a quadratic extension we have the quasi-split unitary group $G^* := U_{E/F}(N)$ defined in Section \ref{subsub:qsuni}. It is endowed with a standard $F$-splitting which determines, together with $\psi_F$, a Whittaker datum as in \cite[\S5.3]{KS99}.
For $\psi \in \Psi(G^*)$ recall the algebraic groups $S_\psi = \tx{Cent}(\psi,\hat G^*)$, $S_\psi^\tx{rad} = [S_\psi \cap [\hat G^*]_\tx{der}]^\circ$ and $S_\psi^\natural = S_\psi / S_\psi^\tx{rad}$.
 Consider an extended inner twist $(\xi,z) : G^*\ra G$. The datum $(\xi,z)$ determines a character $\chi_z\in X^*(Z(\hat G)^\Gamma)$ via \eqref{eq:kotisoloc}. Write $\Irr(S^\natural_\psi,\chi_z)$ for the set of characters on $S^\natural_\psi$ whose pullbacks via $Z(\hat G)^\Gamma\hra S_\psi\twoheadrightarrow S^\natural_\psi$ are $\chi_z$. Recall from Lemma \ref{lem:chi_z-triv-on-intersection} that $\Irr(S^\natural_\psi,\chi_z)$ is nonempty when $\psi$ is $(G,\xi)$-relevant. In order to understand how the local classification depends on $z$, note that when $\xi$ is fixed $z$ can only be replaced by $zx$ for some $x \in Z^1_\tx{alg}(\mc{E},Z(G^*))$, and that $x$ determines a character $\chi_x \in X^*([\hat G/\hat G_\tx{sc}]^\Gamma)$. This character provides a character on $Z(\hat G)^\Gamma$ by pull-back along $Z(\hat G)^\Gamma \rw \hat G^\Gamma \rw [\hat G/\hat G_\tx{sc}]^\Gamma$. It furthermore provides a character on $S_\psi^\natural$ as follows: If we twist the $W_F$-action on the exact sequence
\[ 1 \rw \hat G_\tx{sc} \rw \hat G \rw \hat G/\hat G_\tx{sc} \rw 1 \]
by $\psi$, then the action on the third term remains unchanged, and taking $W_F$-invariants we obtain a map
\[ S_\psi \hrw H^0(\psi(W_F),\hat G) \rw [\hat G/\hat G_\tx{sc}]^\Gamma \]
via which $\chi_x$ provides a character on $S_\psi$ that descends to $S_\psi^\natural$.

  The following local classification theorem is our main local result. In this paper we establish it under the hypothesis that $\psi$ is generic. The theorem will be proven in full in the next paper \cite{KMS_A}.

\begin{thm*} \label{thm:locclass-single}

\begin{enumerate}
\item Let $\psi \in \Psi(G^*)$. There exists a finite set $\Pi_\psi(G,\xi)$ with a map to $\Pi_\tx{unit}(G)$. If $\psi$ is a generic parameter (i.e. $\psi\in \Phi_{\bdd}(G^*)$) then this set is empty if and only if $\psi$ is not $(G,\xi)$-relevant. In general the set is empty if  $\psi$ is not $(G,\xi)$-relevant. %

\item The set $\Pi_\psi(G,\xi)$ is independent of $z$. It is equipped with a map
\[ \Pi_\psi(G,\xi) \rw \tx{Irr}(S_\psi^\natural,\chi_z),\qquad \pi \mapsto \<\pi,-\>_{\xi,z}, \]
whose dependence on $z$ is expressed by
\[ \<\pi,-\>_{\xi,zx} = \<\pi,-\>_{\xi,z} \otimes \chi_x \]
for any $x \in Z^1_\tx{alg}(\mc{E},G^*)$. Both the set $\Pi_\psi(G,\xi)$ and the pairing $\<-,-\>_{\xi,z}$ depend only on the equivalence class $\Xi$ of $(\xi,z)$ and can therefore be denoted by $\Pi_\psi(G,\Xi)$ and $\<-,-\>_\Xi$.

\item For each $\pi \in \Pi_\tx{unit}(G)$ in the image of $\Pi_\psi(G,\xi)$, the central character $\omega_\pi : Z(G)(F) \rw \C^\times$ has a Langlands parameter given by the composition
\[\xymatrix{ L_F\ar[r]^-{\phi_\psi}&{^L G^*}\ar[r]^-{\tx{det}\rtimes\tx{id}}& \C^\times \rtimes W_F}. \]

\item Let $(G^\fke,s^\fke,\eta^\fke)$ be an endoscopic triple with $s^\fke \in S_\psi$, and let $\psi^\fke\in \Psi(G^\fke)$ be a parameter such that $\eta^\fke\circ\psi^\fke=\psi$. If $f^\fke \mapsto f^\fke(\psi^\fke)$ is the stable distribution on $G^\fke(F)$ associated to the parameter $\psi^\fke$, and $f^\fke \in \mc{H}(G^\fke)$ and $f \in \mc{H}(G)$ are two functions with $\Delta[\fke,\xi,z]$-matching orbital integrals, then we have
    \begin{equation} \label{eq:eci} f^\fke(\psi^\fke) = e(G) \sum_{\pi \in \Pi_\psi(G,\Xi)} \<\pi,s^\fke\cdot s_\psi\> f(\pi). \end{equation}

\item Suppose that $\psi \in \Phi_{\bdd}(G^*)$, i.e. $\psi$ is generic. Then the map $\Pi_\psi(G,\Xi) \rw \Pi_\tx{unit}(G)$ is injective and its image belongs to $\Pi_\tx{temp}(G)$. If $F$ is nonarchimedean the map $\Pi_\psi(G,\Xi) \rw \tx{Irr}(S_\psi^\natural,\chi_z)$ is bijective. If $F$ is archimedean, let $(\xi_i,z_i) : G \rw G_i$ be pure inner twists such that $\{z_i\}$ is a set of representatives for the image of $H^1(\Gamma,G_\tx{sc}) \rw H^1(\Gamma,G)$. Then the maps $\Pi_\psi(G_i,\xi_i\circ\xi) \rw \tx{Irr}(S_\psi^\natural,\chi_z)$ given by $\pi_i \mapsto \<\pi_i,-\>_{\xi_i\circ\xi,\xi^{-1}(z_i)z}$ lead to a bijection
    \[ \bigsqcup_i \Pi_\psi(G_i,\xi_i\circ\xi) \rw \tx{Irr}(S_\psi^\natural,\chi_z). \]

\item As $\psi$ runs over $\Phi_{\bdd}(G^*)$ (resp. $\Phi_{2,\bdd}(G^*)$) the sets $\Pi_\psi(G,\Xi)$ are disjoint and exhaust $\Pi_\tx{temp}(G)$ (resp. $\Pi_{2,\temp}(G)$).
\end{enumerate}

\end{thm*}

\subsubsection{Initial reductions}\label{subsub:initial-red}

  As before suppose that $E$ is a quadratic algebra over a local field $F$ (allowing $E=F\times F$). Let $(G,\xi)$ be an inner twist of $G^*=U_{E/F}(N)$.

\begin{pro}\label{p:local-thm-other-ext-pure-inner-tw}
  Let $z,z' \in Z^1_\tx{bsc}(\mc{E},G^*)$ be such that $(\xi,z)$ and $(\xi,z')$ are extended pure inner twists $G^* \rw G$. If Theorem \ref{thm:locclass-single} holds true for $(\xi,z)$ then it does so for  $(\xi,z')$.
\end{pro}
\begin{proof}
As we already remarked prior to stating Theorem \ref{thm:locclass-single}, $z'=zx$ for some $x \in Z^1_\tx{alg}(\mc{E},Z(G^*))$. Applying the Theorem to $(\xi,z)$ we obtain the set $\Pi_\psi(G,\xi)$ and its map to $\Pi_\tx{unit}(G)$. These serve $(\xi,z')$ as well. We also obtain the map $\pi \mapsto \<\pi,-\>_{\xi,z}$ and define the map $\pi \mapsto \<\pi,-\>_{\xi,z'}$ by the formula in part 2 of the Theorem. Parts 3 and 6 are independent of $z$ and thus remain valid, while part 5 depends on $z$ in an obvious way and also remains valid. To treat part 4, recall that according to Lemma \ref{lem:tfequi} the transfer factors $\Delta[\fke,\xi,z]$ and $\Delta[\fke,\xi,zx]$ are related by the equation
\[ \Delta[\fke,\xi,zx](\gamma,\delta) = \Delta[\fke,\xi,z](\gamma,\delta) \cdot \chi_x(s), \]
for any strongly regular semi-simple related elements $\gamma \in G^\fke(F)$ and $\delta \in G(F)$. It follows that if $f$ and $f^\fke$ have $\Delta[\fke,\xi,z]$-matching orbital integrals, then $f$ and $f^\fke \cdot \chi_x(s)$ have $\Delta[\fke,\xi,z']$-matching orbital integrals. Noting that $s_\psi$ lifts to $\hat G_\tx{sc}$ and thus belongs to the kernel of $\chi_x$, we see that the two versions of equation \eqref{eq:eci} for $(\xi,z)$ and $(\xi,z')$ are equivalent.
\end{proof}

\begin{cor}\label{c:local-thm-q-split}
  Theorem \ref{thm:locclass-single} holds true if $G=G^*$ is quasi-split (which includes the case $E=F\times F$).
\end{cor}

\begin{proof}
  Mok's main local theorem (Proposition \ref{prop:main-local-qsplit}) tells us that the theorem is true for the trivial extended pure inner twist $(G^*,\id,1)$. (If $E=F\times F$ then the theorem is no more than Theorem \ref{thm:LLC-GL}.) So the corollary is an immediate consequence of Proposition \ref{p:local-thm-other-ext-pure-inner-tw}.
\end{proof}

\subsubsection{The local classification for linear groups} \label{sec:loclin}

Let $F$ be a local field and $G^*=\tx{GL}_N$. We take $^LG=\hat G=\tx{GL}_N(\C)$ as the dual group. Recall from Section \ref{subsub:inner-GL} that an inner form $G$ of $G^*$ is of the form $\tx{Res}_{D/F}(\tx{GL}(M))$ where $D$ is a division algebra over $F$ of degree $d^2$ and $N=Md$. Every $\phi \in \Phi(G^*)$ can be written as $\phi=k_1\phi_1 \oplus \dots \oplus k_t \phi_t$ with $\phi_i \in \Phi_2(\tx{GL}_{n_i})$ and $k_1n_1 + \dots +k_t n_t =N$. Then $\phi$ is relevant for $G$ if $d$ divides $n_i$ for all $1 \leq i\leq t$. The Kottwitz sign $e(G)$ is explicitly given as $e(G)=(-1)^{N-M}$.

To each Arthur parameter $\psi \in \Psi(G^\ast)$ we associate a Langlands parameter $\phi_\psi$ by \eqref{eq:artlanpar} and to it, we can further associate a representation $\pi^*_\psi$ of $G^*$ via the local Langlands correspondence of Theorem \ref{thm:LLC-GL}.

Let $\psi=k_1\psi_1 \oplus \dots \oplus k_t \psi_t \in \Psi(N)$ be a decomposition of an Arthur parameter into simple constituents. We would like to associate to $\psi$ a sign $a_\psi$ but to do so we have to differentiate the real and $p$-adic cases.

Suppose first that $F$ is \textit{real}. Then $\psi_i$ is of the form $\nu_i \otimes \tx{Sym}^{m_i-1}$ where $\nu_i \in \Phi_2(f_i)$, $f_i \in \{1,2\}$ and and $\tx{Sym}^{m_i-1}$ is the irreducible representation of ${\rm SU}(2)$ of dimension $m_i$. We set $a_{\psi_i}=1$ if $f_i=2$ and  $a_{\psi_i}=(-1)^{m_i/2}$ if $f_i=1$. We then define
\begin{equation*}a_\psi= \underset{i=1}{\overset{t}{\prod}}a_{\psi_i}^{k_i}.
\end{equation*}

Suppose now that $F$ is {\textit{$p$-adic}. Then $\psi_i$ is of the form $\nu_i \otimes \tx{Sym}^{r_i-1}\otimes \tx{Sym}^{m_i-1}$  with $\nu_i \in \Phi_2(f_i)$. We set $a_{\psi_i}=1$ unless $d$ divides $f_im_i$. In this case , let $s$ be the smallest positive integer such that $d$ divides $f_is$ (thus $s|m_i)$; we set $a_{\psi_i}=1$ if $s$ is odd and $a_{\psi_i}=(-1)^{\frac{r_im_i}{s}}$ if $s$ is even. We then define
\begin{equation*} a_\psi= \underset{i=1}{\overset{t}{\prod}}a_{\psi_i}^{k_i}.
\end{equation*}

Finnaly, in the $p$-adic case denote, for any Arthur parameter $\psi \in \Psi(G^\ast)$,
$$\hat{\psi}(w,u_1,u_2)=\psi(w,u_2,u_1), \qquad w \in W_F, u_1,u_2 \in {\rm SU}(2)$$
for the dual parameter that interchanges the two ${\rm SU}(2)$-factors in $W_F \times {\rm SU}(2) \times {\rm SU}(2)$. With this notation we can now recall the results of \cite{DKV84}, \cite{Bad} and \cite{BadRen}. We remark that the ``Langlands-Jacquet correpondence'' ${\bf LJ}$ is normalized in different ways in the two references \cite{Bad} and \cite{BadRen}.

\begin{thm}\label{thm:LLC-GLmD}
Let $\psi=k_1\psi_1 \oplus \dots \oplus k_t \psi_t \in \Psi(N)$ be a decomposition of an Arthur parameter into simple consituents.
\begin{enumerate}
\item Suppose that for each $1 \leq i \leq t$ we have that at least one of $\psi_i|_{L_F}$ and $\hat\psi_i|_{L_F}$
is relevant (we disregard $\hat\psi_i|_{L_F}$ if $F$ is real). Then there exists a unique irreducible unitary representation $\pi_\psi$ of $G(F)$ such that
\begin{equation}\label{eq:tr-GLmD}\tx{tr}(\pi^*_\psi(f^*))=e(G) a_\psi \tx{tr}(\pi_\psi(f))
\end{equation}
for any $f^* \in \mc{H}(G^*)$, $f \in \mc{H}(G)$ such that $O_{\delta^*}(f^*)=O_\delta(f)$ whenever $\delta^* \in G^*(F)$ and $\delta \in G(F)$ are stably conjugate.
\item Suppose there exists $1 \leq i \leq t$ such that neither of $\psi_i|_{L_F}$ and $\hat\psi_i|_{L_F}$ is relevant (we disregard $\hat\psi_i|_{L_F}$ if $F$ is real). Then for any $f^* \in \mc{H}(G^*)$, $f \in \mc{H}(G)$ such that $O_{\delta^*}(f^*)=O_\delta(f)$ whenever $\delta^* \in G^*(F)$ and $\delta \in G(F)$ are stably conjugate, we have
\[\tx{tr}(\pi^*_\psi(f^*))=0.\]
\item As $\psi$ runs over $\Phi(G^*)$ the representations $\pi_\psi$ are different and exhaust $\Pi_\tx{temp}(G)$.

\end{enumerate}
\end{thm}

 Let $\psi \in \Psi(G)$ be an Arthur parameter. To deduce Theorem \ref{thm:locclass-single} from Theorem \ref{thm:LLC-GLmD}, we will show that the group $S_\psi^\natural$ is a complex torus of rank 1 and we will assign to each $\rho \in X^*(S_\psi^\natural)$ an equivalence class of extended pure inner twists $\Xi : G^* \rw G$, a representation $\pi$ of $G(F)$ (possibly empty) and use \eqref{eq:tr-GLmD} to prove  that the character identities of Theorem \ref{thm:locclass-single} hold.

Decompose $\psi = k_1\psi_1 \oplus \dots \oplus k_t\psi_t$ into irreducible representations with $\dim(\psi_i)=n_i$. The parameter $\psi$ is then discrete for the Levi subgroup $M^* \subset G^*$ dual to $\hat M = \tx{GL}_{n_1}(\C)^{k_1} \times \dots \times \tx{GL}_{n_t}(\C)^{k_t}$. We have $S_\psi \cong \tx{GL}_{k_1} \times \dots \times \tx{GL}_{k_t}$, with each entry $c$ of $\tx{GL}_{k_i}$ corresponding to an $n_i \times n_i$-block of the form $c\cdot \tb{1}_{n_i}$ inside of $\hat G$. We have
\[ S_\psi \cap \hat G_\tx{der} = \{(M_1,\dots,M_t)|\prod_i \tx{det}(M_i)^{n_i}=1\}. \]
Let $n'=(n_1,\dots,n_t)$ be the greatest common divisor, and let $n_i'=n_i/n'$. Then
\[ (S_\psi \cap \hat G_\tx{der})^0 = \{(M_1,\dots,M_t)|\prod_i \tx{det}(M_i)^{n_i'}=1\}. \]
It follows that we have the isomorphism
\[ S_\psi^\natural \cong \C^\times,\qquad (M_1,\dots,M_t) \mapsto \prod_i \tx{det}(M_i)^{n_i'}. \]
The inclusion $Z(\hat G) \rw S_\psi$ induces the surjective map
\[ \C^\times = Z(\hat G) \rw S_\psi^\natural = \C^\times,\qquad z \mapsto z^{n/n'}. \]
Dual to this map is an inclusion $X^*(S_\psi^\natural) \subset X^*(Z(\hat G))$ and we use it to view any $\rho \in X^*(S_\psi^\natural)$ as an element of $X^*(Z(\hat G))$. This element gives rise to an equivalence class of extended pure inner twists $(\xi,z) : G^* \rw G$. One checks using Lemma \ref{lem:triv-on-ZGcapZM} that the inner forms $G$ obtained in this way are precisely those to which the Levi subgroup $M^*$ of $G^*$ transfers.
This completes the construction.

Let $s \in S_\psi$ be a semi-simple element. Conjugating within $S_\psi$ we may assume that it belongs to the standard diagonal torus of $S_\psi \cong \tx{GL}_{k_1} \times \dots \times \tx{GL}_{k_t}$, hence also to the standard diagonal torus of $\hat G$. The group $\hat L = \hat G_s$ is connected and is a semi-standard Levi subgroup of $\hat G$ and contains $\hat M$, because $Z(\hat M)$ is the standard maximal torus of $S_\psi$. Let $L^*$ be the standard Levi subgroup of $G^*$ dual to $\hat L$. It contains $M^*$ and the parameter $\psi$ provides a representation $\sigma$ of $L^*(F)$. We may take $\hat L\times W_F$ as the L-group of $L^*$ and in this way $(L^*,s)$ becomes an endoscopic datum for $G^*$.

We now discuss $\Delta[\xi,z]$. Let $\gamma \in L^*(F)$ and $\delta \in G(F)$ be semi-simple and $G^*$-regular. These elements are related if and only if they are stably conjugate, with $\gamma$ viewed as an element of $G^*(F)$ via the inclusion $L^* \subset G^*$.

\begin{lem} \label{lem:tfgln} We have
\[ \Delta[\xi,z](\gamma,\delta) = D_{G^*/L^*}(\gamma) \cdot \rho(s). \]
\end{lem}
\begin{proof}
Let us adjust $(\xi,z)$ within its equivalence class so that $\xi(\gamma)=\delta$. Then $L=\xi(L^*)$ is a Levi subgroup of $G$ defined over $F$ and $(\xi,z) : L^* \rw L$ is an extended pure inner twist. Let $S^*=\tx{Cent}(\gamma,G^*)$. We have
\[ \Delta[\xi,z](\gamma,\delta) = \Delta(\gamma,\gamma) \cdot \<\tx{inv}(\gamma,\delta),s\>_{S^*}. \]
The term $\Delta(\gamma,\gamma)$ is computed with respect to the Whittaker normalization of the transfer factor for the endoscopic datum $(L^*,s)$ of $G^*$. A routine calculation shows that it equals $D_{G^*/L^*}(\gamma)=\prod_\alpha|\alpha(\gamma)-1|^\frac{1}{2}$, the product being taken over all roots of $S^*$ in $G^*$ outside of $L^*$. The term $\tx{inv}(\gamma,\delta)$ belongs to $H^1_\tx{G-\tx{bsc}}(\mc{E},S^*)$ (and is equal to the class of $z$ there), while the term $s$ belongs to $Z(\hat L^*) \subset [\hat S^*]^\Gamma$. The pairing is a-priori that between $H^1_\tx{bsc}(\mc{E},S^*)$ and $[\hat S^*]^\Gamma$, but since $s \in Z(\hat L^*)$ we may also map $\tx{inv}(\gamma,\delta)$ to $H^1_\tx{G-bsc}(\mc{E},L^*)$ and then pair it with $s \in Z(\hat L^*)$. We have seen
\[ \Delta[\xi,z](\gamma,\delta)=D_{G/M}(\gamma) \cdot \<z,s\>_{L^*}. \]
Recall that $M^*$ is contained in $L^*$ and transfers to $G$. Hence it also transfers to $L$. Adjusting $(\xi,z) : L^* \rw L$ within its equivalence class we achieve that $M=\xi(M^*)$ is defined over $F$ and $(\xi,z) : M^* \rw M$ is an extended pure inner twist, in particular $z \in Z^1_\tx{G-bsc}(\mc{E},M^*)$. We are thus free to replace the pairing $\<-,-\>_{L^*}$ with the pairing $\<-,-\>_{M^*}$ between $H^1_\tx{bsc}(\mc{E},M^*)$ and $Z(\hat M)$. Since $z$ is $G^*$-basic, Lemma \ref{lem:gbasl} tells us that $\<z,-\>_{M^*}$ annihilates $[Z(\hat M) \cap \hat G_\tx{der}]^0$, the latter being contained in $N_{K/F}(Z(\hat M) \cap \hat G_\tx{der})$ for any finite Galois extension $K/F$. But  $Z(\hat M)$ is a maximal torus of the connected reductive group $S_\psi$, thus $Z(\hat M) \cap S_\psi^\tx{rad}$ is a maximal torus of $S_\psi^\tx{rad}$ (in particular, it is connected) and equals $[Z(\hat M) \cap \hat G_\tx{der}]^0$. We conclude that the inclusion $Z(\hat M) \subset S_\psi$ induces an isomorphism $Z(\hat M)/[Z(\hat M) \cap \hat G_\tx{der}]^0 \rw S_\psi^\natural$. Using the surjection $Z(\hat G) \rw S_\psi^\natural$ we choose $\dot s \in Z(\hat G)$ mapping to the image of $s$ in $S_\psi^\natural$. With this, we see
\[ \Delta[\xi,z](\gamma,\delta) = D_{G/M}(\gamma) \cdot \<z,\dot s\>_{M^*}. \]
But since $\dot s$ now resides in $Z(\hat G)$, we may finally replace $\<-,-\>_{M^*}$ with $\<-,-\>_{G^*}$. Recall however that $\<z,-\>_{G^*}$ is the character on $Z(\hat G)$ determining the equivalence class of the extended pure inner twist $(\xi,z) : G^* \rw G$ and was by construction equal to the restriction of $\rho \in X^*(S_\psi^\natural)$ to $Z(\hat G)$. Thus $\<z,\dot s\>_{G^*}=\rho(\dot s)=\rho(s)$. We arrive at
\[ \Delta[\xi,z](\gamma,\delta) = D_{G^*/L^*}(\gamma) \cdot \rho(s) \]
and the proof is complete.
\end{proof}

\begin{proof}[Proof of Theorem \ref{thm:locclass-single} for $G$.]
Let $\psi \in \Psi(G^*)$ and $\rho \in X^*(S_\psi^\natural)$. Let $(\xi,z) : G^* \rw G$ be the equivalence class of extended pure inner twists constructed as above from $\psi$ and $\rho$. Let $\pi^*_\psi$ be the representation of $G^*(F)$ associated to $\psi$ via the local Langlands correspondence.

Decompose as before $\psi = k_1\psi_1 \oplus \dots \oplus k_t\psi_t$ into irreducible representations.
If, for all $1 \leq i \leq t$, either $\psi_i|_{L_F}$ or $\hat\psi_i|_{L_F}$ is relevant, set $\Pi_\psi(G,\xi)=\{ \pi_\psi\}$ where $\pi_\psi$ is the representation  associated to $\psi$ by Theorem \ref{thm:LLC-GLmD}. Otherwise, set $\Pi_\psi(G,\xi)=\emptyset$.

To prove  Theorem \ref{thm:locclass-single} for $G$ it is enough to show that, if $f \in \mc{H}(G)$ and $f' \in \mc{H}(L^*)$ have $\Delta[\xi,z]$-matching orbital integrals, then
\begin{enumerate}
\item $ \tx{tr}(\sigma(f')) = e(G)\rho(s)\rho(s_\psi)\tx{tr}(\pi_\psi(f))$, if $\Pi_\psi(G,\xi)=\{ \pi_\psi\}$;
\item $\tx{tr}(\pi^*_\psi(f'))=0$, if $\Pi_\psi(G,\xi)=\emptyset$.
\end{enumerate}

The discussion of $\Delta[\xi,z]$ implies that we can take $f' \in \mc{H}(L^*)$ to be the constant term along $L^*$ of $\rho(s)f^*$. Then, by Theorem \ref{thm:LLC-GLmD}, we have
\begin{eqnarray*}
\tr(\sigma(f'))&=&\tr(\pi_\psi^*(\rho(s)f^*))\\&=&\begin{cases} e(G)\rho(s)a_\psi\tr(\pi_\psi(f)) &\text{if $\Pi_\psi(G,\xi)=\{ \pi_\psi\}$;} \\ 0 & \text{if $\Pi_\psi(G,\xi)=\emptyset$},\end{cases}
\end{eqnarray*}
so Theorem \ref{thm:locclass-single} is reduced to show that in the first case
\[\rho(s_\psi)=a_\psi.
\]
Suppose  first $\psi$ is a discrete parameter.

If $F$ is real then, by Corollary \ref{c:local-thm-q-split}, we just need to consider the case where $N$ is even and $G=\GL_{N/2}(\mb{H})$, where $\mb{H}$ denotes the quaternion algebra.
Then $\psi$ is of the form $\nu \otimes \tx{Sym}^{m-1}$ where $\nu$ is a representation of $W_F$ of dimension $f\in \{1,2\}$ and $fm=N$, and  $s_\psi=(-1)^{m-1}\tx{Id}_N$.
By \ref{subsub:inner-GL}, $\rho(s)=s^k$, with $k \equiv N/2 \mod N$, so $\rho(s_\psi)=(-1)^{(m-1)k}  $. This is equal to $1$ unless $m$ is even and $N/2$ is odd, that is if and only if $f=1$ and $m/2$ is odd. Thus $\rho(s_\psi)=a_\psi.$

If $F$ is $p$-adic, then $\psi$ is of the form $\nu \otimes \tx{Sym}^{n-1} \otimes \tx{Sym}^{m-1}$ where $\nu$ is a representation of $W_F$ of dimension $f$ and $fnm=N$. Again,  $s_\psi$ can be computed explicitely and one has $s_\psi=(-1)^{m-1}\tx{Id}_N$. Then $\rho(s_\psi)=(-1)^{(m-1)k}  $, with $k \equiv r\frac{N}{d} \mod N$ where $\frac{r}{d}$ is the Hasse invariant of $G$ (see Paragraph \ref{subsub:inner-GL}).

Hence $\rho(s_\psi)=(-1)^{(m-1)r\frac{fnm}{d}}$. If $d$ divides $nf$ then $\rho(s_\psi)=1$. If $d$ divides $mf$ let $s$ be the smallest integer such that $d|sf$. Let $m'\in \Z$ such that $m=m's$. Thus, if $s$ is odd
we have that $\rho(s_\psi)=(-1)^{(m's-1)rnm'\frac{fs}{d}}=1$. If $s$ is even, then $m$ and $d$ are even, and
$\frac{fs}{d}$ and $r$ need ot be odd. We deduce that $\rho(s_\psi)=(-1)^{\frac{nm}{s}}$. Again we have $\rho(s_\psi)=a_\psi.$

Suppose now $\psi = k_1\psi_1 \oplus \dots \oplus k_t\psi_t$. Then the discussion above tells us that $\<\pi_\psi,\dot s_\psi\>_{G^*}= \underset{i=1}{\overset{t}{\prod}} \<\pi_{\psi_i},\dot s_{\psi_i}\>_{\GL_{n_i}^*}^{k_i}$. The multiplicativity definition of $a_\psi$ gives us $\rho(s_\psi)=a_\psi.$
\end{proof}

\begin{rem}
For $G^*=\GL_N(F)$ we have a necessary and sufficient condition for an Arthur packet to be non-empty. Namely if $\psi=k_1\psi_1 \oplus \dots \oplus k_t \psi_t \in \Psi(N)$ is a decomposition of an Arthur parameter $\psi \in \Psi(G^*)$ into simple consituents, then $\Pi_\psi(G,\xi)$ is non-empty if and only for all $1 \leq i \leq t$, either $\phi_{\psi_i}$ or $\phi_{\widehat{\psi_i}}$ is $(G,\xi)$-relevant. And all Arthur packets are either empty or singletons.
\end{rem}

\subsubsection{Local packets for parameters in $\Psi^+(U_{E/F}(N))$}\label{subsub:local-packet-Psi+}

  In the above theorem we considered packets of an extended pure inner twist $(G,\xi,z)$ associated to parameters $\psi$ only in $\Psi(U_{E/F}(N))$. As a preparation for the global theorem in the next subsection we introduce $\Pi_\psi=\Pi_\psi(G,\xi)$ when $\psi$ belongs to the larger set $\Psi^+(U_{E/F}(N))$. The necessity comes from the possibility that the generalized Ramanujan conjecture might fail, in which case the localization of the global parameter may not be in $\Psi(U_{E/F}(N))$ but only in $\Psi^+(U_{E/F}(N))$. The construction is basically the same as in \cite[2.5]{Mok}, cf. \cite[1.5]{Arthur}. The idea is to go down to a Levi subgroup $M$ of $G$ such that $\psi$ lies (not just in $\Psi^+(M^*)$ but) in $\Psi(M^*)$ up to a character twist, where the packet is already defined, and then induce up from $M$ to $G$. We will also define a pairing between $S^\natural_\psi$ and $\Pi_\psi$.

  We need a preliminary discussion of characters on Levi subgroups. Suppose that $P^*=M^*N_{P^*}$ is a standard parabolic subgroup of $U_{E/F}(N)$ and that $M^*$ transfers to $G$. Lemmas \ref{lem:levitran} and \ref{lem:tranrel} allow us to arrange that $P^*$ transfers to a standard parabolic subgroup $P=M N_P$ and that $\xi:G^*\ra G$ restricts to an inner twist $\xi:M^*\ra M$ by disturbing $(G,\xi,z)$ within its equivalence class if necessary. In particular the isomorphism $\xi:A_{M^*}\ra A_M$ is defined over $F$. Consider a point $\lambda$ in the open chamber of $P$ in $$\fka_M^*:=X(M^*)_F\otimes_\Z \R=X(A_{M^*})_F\otimes_\Z \R=X(A_M)_F\otimes_\Z \R=X(M)_F\otimes_\Z \R.$$
  The first and third identifications are induced by inclusions $A_{M^*}\hra M^*$ and $A_M\hra M$. Then $\lambda$ gives rise to an unramified quasi-character $\chi^*_\lambda:M^*(F)\ra \C^\times$ and $\chi_\lambda:M(F)\ra \C^\times$. Both $\chi^*_\lambda$ and $\chi_\lambda$ correspond to the same element $\phi_\lambda\in H^1(W_F,Z(\hat M^*))$ (which may also be viewed as an $L$-parameter for $M^*$), cf. \cite[10.2]{Bor79}. Every element, in particular $\phi_\lambda$, acts on the set $\Psi(M^*)$ by multiplication, cf. \cite[8.5]{Bor79}. (Since $\phi_\lambda$ has central image, there is no need to distinguish left and right multiplications.)

  Now we explain how to associate a packet to a given $\psi\in \Psi^+(U_{E/F}(N))$. There exist a standard parabolic subgroup $P^*=M^*N_{P^*}$ of $U_{E/F}(N)$, a parameter $\psi_{M}\in \Psi(M^*)$, and a point $\lambda\in\fka_M^*$ such that $\psi_{M,\lambda}:=\psi_{M}\cdot \psi_\lambda$ is carried to $\psi$ under ${}^L M^*\hra {}^L U_{E/F}(N)$, where $\chi_\lambda:W_F\ra {}^L M^*$ is the parameter given by $\lambda$, cf. \cite[1.5]{Arthur}. (In particular the image of $\chi_\lambda$ is contained in $Z(\hat M^*)\rtimes W_F$.)
  If $\psi$ is not $(G,\xi)$-relevant, simply define $\Pi_\psi(G,\xi)$ to be the empty set. Now assume that $\psi$ is $(G,\xi)$-relevant so that $M^*$ transfers to $G$ and the discussion in the preceding paragraph applies. Using Lemma \ref{lem:c1} we may replace $(\xi,z)$ by an equivalent pair such that $(\xi,z)$ equips (not only $G$ but also) $M$ with the structure of an extended pure inner twist. Let us write $(M,\xi_M,z_M)$ for the latter.
  Theorem \ref{thm:locclass-single} attaches a packet $\Pi_{\psi_{M}}=\Pi_{\psi_M}(M,\xi_M)$ and a pairing $\lg \cdot,\pi_M\rg_{\xi_M,z_M}$ for each $\pi_M\in \Pi_{\psi_{M}}$.
  Now we are ready to define the packet
  $$\Pi_\psi(G,\xi):=\{\pi=\cI_P(\pi_M\otimes\chi_\lambda):\pi_M\in \Pi_{\psi_M}\},$$
  which may consist of reducible or non-unitary representations, just like in the quasi-split case, cf. \cite[2.5]{Mok}. As usual we write $\Pi_\psi$ for $\Pi_\psi(G,\xi)$ if the inner twist is clear from the context.
  It is easy to see that different choices of $\psi_M$ lead to the same definition of $\Pi_\psi(G,\xi)$. Indeed it suffices to check that $\Pi_{\psi_M}$ is invariant under the Weyl group of $M$ in $G$. This is trivial when $G$ is linear. When $G$ is unitary, with the decomposition $M=M_+\times M_-$ as before, each Weyl element acts trivially on $M_-$ and permutes the linear factors of $M_+$, but the packets of linear groups are singletons so we are done.
  It is also worth pointing out that $\Pi_\psi(G,\xi)$ depends only on $(G,\xi)$ and not on the extended pure inner twist by virtue of Theorem \ref{thm:locclass-single}.

  Our next task is to give the pairing between $S^\natural_\psi$ and $\Pi_\psi$  for $\psi\in \Psi^+(U_{E/F}(N))$, which depends on the extended pure inner twist.
  Observe that the centralizer group $S_\psi$ is contained in $\hat M^*$ by the assumption on $\psi$ so that $S_\psi=S_{\psi_{M}}$ and that $Z(\hat M^*)^\Gamma \subset S_\psi$. So
  \begin{equation}\label{eq:SZ=SMZM}
    S_\psi^0 Z(\hat G^*)^\Gamma = S_{\psi_{M}}^0 Z(\hat M^*)^\Gamma.
  \end{equation} (The inclusion $\subset $ is clear. For the other inclusion it is enough to see that $\pi_0(Z(\hat G^*)^\Gamma)\ra\pi_0(Z(\hat M^*)^\Gamma)$ is onto, which is verified explicitly for the list of Levi subgroups in \S\ref{subsub:Levi-GL} and \S\ref{subsub:Levi-U}.) The assumption that $M^*$ transfers to a Levi subgroup $M\subset G$ shows that the character $\chi_z\in X^*(Z(\hat G^*)^\Gamma)$ coming from $z$ is trivial on $\srad_\psi \cap Z(\hat G^*)^\Gamma$ by Lemma \ref{lem:chi_z-triv-on-intersection}. Hence $\chi_z$ extends uniquely to a character $\tilde \chi_z$ on $\srad_\psi Z(\hat G^*)^\Gamma=S^0_\psi Z(\hat G^*)^\Gamma$ which is trivial on $\srad_\psi$. Let $\chi_{z,M}$ denote the restriction of $\tilde \chi_z$ to $Z(\hat M^*)^\Gamma$ via \eqref{eq:SZ=SMZM}. Then $\chi_{z,M}=\chi_{z_M}$ since $\chi_{z_M}$ coincides with $\chi_z$ on $Z(\hat G^*)^\Gamma$ and also extends uniquely to $S^0_\psi Z(\hat G^*)^\Gamma$ by the same reasoning.
 Lemma \ref{lem:squot} and \eqref{eq:SZ=SMZM} imply that $\srad_\psi Z(\hat G^*)^\Gamma = \srad_{\psi_{M}} Z(\hat M^*)^\Gamma$, and also that
   \begin{equation}\label{eq:Irr(S)=Irr(S_M)}
   \Irr(S^\natural_\psi,\chi_z)=\Irr(S^\natural_{\psi_{M}},\chi_{z,M})
   \end{equation} since both sets parametrize characters on $S_\psi$ whose restriction to $\srad_\psi Z(\hat G)^\Gamma$ is $\tilde \chi_z$. Thanks to the above identification we can now define for each $\pi=\cI_P(\pi_M\otimes\chi_\lambda)\in \Pi_\psi$ that
  $$\lg s,\pi\rg_{\xi,z} := \lg s,\pi_M \rg_{\xi_M,z_M},\quad s\in S^\natural_\psi,$$
  which makes sense because $\lg \cdot,\pi_M\rg_{\xi_M,z_M}$ can be viewed as a character in $\Irr(S^\natural_\psi,\chi_z)$ via \eqref{eq:Irr(S)=Irr(S_M)} and the equality $\chi_{z,M}=\chi_{z_M}$.

\subsection{Main global theorem}\label{sub:main-global-thm}

   In this subsection we state our main global theorem describing a spectral decomposition of the discrete automorphic spectrum for inner forms of unitary groups via $L$-packets. So let $(G,\xi)$ be an inner twist of $G^*=U_{E/F}(N)$ where $E$ is a quadratic extension of a number field $F$.
  We choose $z$ such that $(G,\xi,z)$ is an extended pure inner twist; this is possible thanks to Lemma \ref{lem:c2}. The role of $z$ is auxiliary in that the final theorem turns out to be independent of the choice of $z$.

   To state the decomposition we would like to attach certain global packets $\Pi_\psi(G,\xi)$ to $\psi\in \Psi(G^*,\eta_\chi)$.
      As discussed in \S\ref{subsub:localizations-U} there is a localization map $\psi\mapsto \psi_v$, from $\Psi(G^*,\eta_\chi)$ to $\Psi^+_{\unit}(G^*_v)$. Note that $\psi_v$ is unramified at all but finitely many places $v$ (since we have $\psi_v^N=\eta_\chi\circ \psi_v$ from \S\ref{subsub:localizations-U} and know that $\eta_{\chi}$ and $\psi_v^N$ are unramified at almost all $v$). At the end of the last subsection we attached a local packet $\Pi_{\psi_v}(G_v,\xi_v)$ equipped with a map
 $$\Pi_{\psi_v}(G_v,\xi_v)\ra \Irr(S^\natural_\psi,\chi_z),\quad \pi_v\mapsto \lg \cdot,\pi_v\rg_{\xi_v,z_v}.$$
  Define the global packet for $\psi$ by
 $$\Pi_\psi(G,\xi):=\left\{\bigotimes_v \pi_v~:~\pi_v\in \Pi_{\psi_v}(G_v,\xi_v),~~\lg\cdot,\pi_v\rg_{\xi_v,z_v}=1~\mbox{for~almost~all}~v\right\}.$$
 When $G$ is not quasi-split and $\psi$ is not generic, the packet $\Pi_\psi(G,\xi)$ may be empty because $\Pi_{\psi_v}(G_v,\xi_v)$ are possibly empty, for instance if $\psi$ is not $(G,\xi)$-relevant.
  To each $\pi=\otimes_v \pi_v\in \Pi_\psi(G,\xi)$ we attach a character on $S^\natural_\psi$ by
  $$\lg s,\pi\rg_\xi:=\prod_v \lg s,\pi_v\rg_{\xi_v,z_v},\quad s\in S^\natural_{\psi},$$
   where the latter $s$ denotes the image of $s$ under $S^\natural_{\psi}\ra S^\natural_{\psi_v}$. If $s\in Z(\hat G^*)^\Gamma$ then $\lg s,\pi\rg=\prod_v \lg s,\pi_v\rg=\prod_v\chi_{z_v}(s)=1$ since $\prod_v \chi_{\xi_v,z_v}$ is the trivial character by the product formula \eqref{eq:kotisolocglo}. Since the natural map $Z(\hat G^*)^\Gamma\ra S^\natural_{\psi}$ has cokernel $\ol\cS_{\psi}$ (Lemma \ref{lem:squot}), the global character $\lg \cdot,\pi\rg_\xi$ descends to a character on the finite abelian group $\ol\cS_{\psi}$. It can also be seen from the product formula that the global character $\lg \cdot,\pi\rg_\xi$ is independent of the choice of $z$, explaining the omission of $z$ from the notation.

   The final ingredient in the global theorem is the character $\epsilon_\psi: \ol \cS_{\psi}\ra \{\pm 1\}$. We refer the reader to \cite[2.5]{Mok}, cf. \cite[1.5]{Arthur}, for the precise definition and simply note that $\epsilon_\psi=1$ identically if $\psi$ is generic. %
   In particular the definition of $\epsilon_\psi$ is the same as in the quasi-split case and does not depend on $(G,\xi)$.
   Finally put
   $$\Pi_\psi(G,\xi,\epsilon_\psi):=\{\pi\in \Pi_\psi(G,\xi)~:~\lg\cdot,\pi\rg_\xi=\epsilon_\psi\}.$$

\begin{thm*}\label{thm:main-global}
  Let $E/F$ be a quadratic extension of global fields and $\kappa\in \{\pm 1\}$. Fix $\chi_\kappa\in \cZ^\kappa_E$. Let $(G,\xi)$ be an inner twist of $G^*=U_{E/F}(N)$. Then there is a $G(\A_F)$-module isomorphism
  $$L^2_{\disc}(G(F)\bs G(\A_F))\simeq \bigoplus_{\psi\in \Psi_2(G^*,\eta_{\chi_\kappa})} \bigoplus_{\pi\in \Pi_\psi(G,\xi,\epsilon_\psi)} \pi.$$
\end{thm*}

  We could have stated the theorem with $\kappa=1$ and $\chi_\kappa=1$ but it seems preferable not to do so since different choices of $\kappa$ and $\chi_\kappa$ naturally appear when studying parameters coming from endoscopic groups. A better statement would be to take the first sum over $\psi\in\Psi_2(G^*)$, the canonical set of discrete parameters defined in \S\ref{subsub:canonical-def} by considering all choices of $\kappa$ and $\chi_\kappa$ at once. Note that the localizations $\psi_v$ of $\psi$ are well defined according to Lemma \ref{lem:localization-compat}, so $\Pi_\psi(G,\xi)$ is defined independently of $\kappa$ and $\chi_\kappa$. One also checks that the same independence holds for $\epsilon_\psi$, hence also for $\Pi_\psi(G,\xi,\epsilon_\psi)$.

  Theorem \ref{thm:main-global} will be completely proved in \cite{KMS_A} if $(G,\xi)$ is realized as a pure inner twist and in \cite{KMS_B} in general. More precisely we prove the theorem in Chapter \ref{chapter5} under two hypotheses, which are resolved in \cite{KMS_A} and \cite{KMS_B} in the corresponding cases. The unconditional result of this paper towards the theorem,  is a natural decomposition of $L^2_{\disc}(G(F)\bs G(\A_F))$ according to the parameters $\psi\in \Psi(G^*,\eta_{\chi_\kappa})$ only when $\psi$ is generic and $(G,\xi)$ comes from a pure inner twist, in which case the two hypotheses are verified.

\begin{rem}
  We cannot say that the $\psi$-part of the spectrum is everywhere tempered for each generic parameter $\psi$ due to the possible failure of the Ramanujan conjecture for general linear groups, but it does follow from the result of this paper (and the result in the quasi-split case) that a discrete automorphic representation $\pi$ of $G(\A_F)$ which is everywhere tempered does occur in the $\psi$-part for some generic parameter $\psi$. Indeed we show in \S\ref{sub:stable-multiplicity} that such a $\pi$ should occur in the $\psi$-part for some $\psi\in \Psi_2(G^*,\eta_{\chi_\kappa})$. If $\psi$ were non-generic then $\pi_v$ are unramified and belong to the packet for the non-generic parameter $\psi_v$ at almost all finite places $\pi_v$, so cannot be tempered.
\end{rem}

\section{Chapter 2: The local intertwining relation}\label{chapter2}

  Here we define normalized local intertwining operators motivated by their global analogues (cf. \ref{sec:giop} below) and formulate the local intertwining relation (Theorem \ref{thm:lir} below), which plays an essential role in the interpretation of the spectral side of the stabilized trace formula and lies at the heart of the proof of the main theorems. Once the local intertwining relation is stated, we prove initial reduction steps by purely local methods based on the induction hypothesis and a few special cases for the real unitary group $U(3,1)$. The latter is not only an illustration but serves as a cornerstone for the inductive argument when completing the proof of the local intertwining relation in Chapter \ref{chapter4}.

\subsection{The basic diagram} \label{sec:diag}
We introduce a diagram of complex algebraic groups that will be useful in both the local and the global context. It can be stated in the general context where $G$ is a quasi-split connected reductive group defined over a local or global field $F$, $M \subset G$ is a proper Levi subgroup, and $\psi \in \Psi(M)$ is a parameter. Recall the possibly disconnected complex reductive group $S_\psi = \tx{Cent}(\psi,\hat G)$ and its connected subgroup $S_\psi^\tx{rad}=\tx{Cent}(\psi,[\hat G_\tx{der}])^0$ which contains the derived subgroup of $S_\psi^0$. When $F$ is global, one should be careful about the meaning of $S_\psi$. We will assume that $S_\psi$ has been properly defined. This is the case for the groups considered in this paper by virtue of Section \ref{sub:global-param}. The quotient $S_\psi^\natural=S_\psi/S_\psi^\tx{rad}$ is a complex diagonalizable group. We also have the analogs $S_\psi(M)=\tx{Cent}(\psi,\hat M)=\tx{Cent}(A_{\hat M},S_\psi)$ and $S_\psi^\natural(M)$ relative to $M$. Note that $S_\psi^\natural(M)$ is not a subgroup of $S_\psi^\natural$, because $S_\psi(M)^\tx{rad}$ is not equal to $S_\psi^\tx{rad} \cap S_\psi(M)$. We write $N_\psi(M,G)$ for the normalizer of $A_{\hat M}$ in $S_\psi$.

Writing $Z$ for centralizer and $N$ for normalizer, we recall the three finite groups
\[ W_\psi^0(M,G) = \frac{N(A_{\hat M},S_\psi^0)}{Z(A_{\hat M},S_\psi^0)}, \qquad W_\psi(M,G) = \frac{N(A_{\hat M},S_\psi)}{Z(A_{\hat M},S_\psi)} \]
and
\[ R_\psi(M,G) = \frac{N(A_{\hat M},S_\psi)}{N(A_{\hat M},S_\psi^0)\cdot Z(A_{\hat M},S_\psi)}. \]
Define moreover
\[ S^{\natural\natural}_\psi(M) = \frac{Z(A_{\hat M},S_\psi)}{Z(A_{\hat M},S_\psi^\tx{rad})},\quad S^\natural_\psi(M,G) = \frac{N(A_{\hat M},S_\psi)}{N(A_{\hat M},S_\psi^\tx{rad})}, \]
and
\[ N^\natural_\psi(M,G) = \frac{N(A_{\hat M},S_\psi)}{Z(A_{\hat M},S_\psi^\tx{rad})},\quad W_\psi^\tx{rad}(M,G) = \frac{N(A_{\hat M},S_\psi^\tx{rad})}{Z(A_{\hat M},S_\psi^\tx{rad})}. \]
Since $S_\psi^\tx{rad}$ contains the derived subgroup of $S_\psi^0$, we have the equality $N(A_{\hat M},S_\psi^0)=N(A_{\hat M},S_\psi^\tx{rad})\cdot Z(A_{\hat M},S_\psi^0)$.

We obtain the following commutative diagram with exact rows and columns.
\begin{equation} \label{eq:diag} \xymatrix{
&&1\ar[d]&1\ar[d]\\
&&W_\psi^\tx{rad}(M,G)\ar@{=}[r]\ar[d]&W_\psi^0(M,G)\ar[d]\\
1\ar[r]&S^{\natural\natural}_\psi(M)\ar[r]\ar@{=}[d]&N^\natural_\psi(M,G)\ar[r]\ar[d]&W_\psi(M,G)\ar[r]\ar[d]&1\\
1\ar[r]&S^{\natural\natural}_\psi(M)\ar[r]&S^\natural_\psi(M,G)\ar[r]\ar[d]&R_\psi(M,G)\ar[r]\ar[d]&1\\
&&1&1
} \end{equation}

\subsection{Local intertwining operator I} \label{sec:iop1}

Let $F$ be a local field of characteristic zero, $E/F$ a quadratic algebra, and $G^*=U_{E/F}(N)$. Let $M^* \subset G^*$ be a proper standard Levi subgroup. Let $\xi : G^* \rw G$ be an inner twist which restricts to an inner twist $\xi : M^* \rw M$. Let $P \subset G$ be a parabolic subgroup defined over $F$ with Levi factor $M$. Let $\psi \in \Psi(M^*)$ be a parameter. We assume the validity of Theorem \ref{thm:locclass-single} for $M$ (as part of the induction hypothesis) and obtain the packet $\Pi_\psi(M,\xi)$. We are concerned with the case that this packet is non-empty, which we now assume. Let $\pi \in \Pi_\psi(M,\xi)$.

For simplicity we often write $N^\natural_\psi(M,G)$, $S^\natural_\psi(M,G)$, etc for the groups $N^\natural_\psi(M^*,G^*)$, $S^\natural_\psi(M^*,G^*)$ etc introduced above (note that the group $G$ was assumed to be quasi-split in the preceding subsection) and similarly write $\hat M$ and $\hat G$ for $\hat M^*$ and $\hat G^*$ in Chapter \ref{chapter2}, except in \S\ref{sec:lir} and \S\ref{sub:prelim-local-intertwining} where we have the generality that $M^*$ may not transfer to $G$, in which case there is no $M$.

We will describe a normalization of the intertwining operator
\[ R_{P'|P}(\xi,\psi) : \mc{H}_{P}(\pi) \rw \mc{H}_{P'}(\pi) \]
for two parabolic subgroups $P$ and $P'$ of $G$ with common Levi factor $M$. The normalization will essentially be the standard one, as outlined for example in \cite[\S2.3]{Arthur}. We must however use the inner twist $\xi$ in order to specify the relevant Haar measures as well as the Galois representations on the dual side that will provide the normalizing factors. In doing so, we must ensure that the end result will transform well when we change $\xi$ within its $M^*$-equivalence class.

We begin by considering the
un-normalized intertwining operator. For this, recall \cite[\S1]{ArtIOR1} the spaces $\mf{a}_M=\tx{Hom}(X^*(M)^\Gamma,\R)$ and $\mf{a}_{M,\C}^*=X^*(M)^\Gamma \otimes \C$, as well as the function $H_M : M(F) \rw \mf{a}_M$ defined by $\exp(\<H_M(m),\chi\>)=|\chi(m)|_F$ for all $\chi \in X^*(M)^\Gamma$ and $m \in M(F)$. For each $\lambda \in \mf{a}_{M,\C}^*$ we have the character
\[ M(F) \rw \C^\times,\qquad m \mapsto \exp(\<H_M(m),\lambda\>). \]
Let $\pi_\lambda$ denote the tensor product of $\pi$ with this character. Consider the operator
$J_{P'|P}(\xi,\psi_F) : \mc{H}_{P}(\pi_\lambda) \rw \mc{H}_{P'}(\pi_\lambda)$ given by the integral formula
\[ [J_{P'|P}(\xi,\psi_F)f](g)=\int_{N(F) \cap N'(F)\lmod N'(F)} f(n'g)dn'. \]
Here $N$ and $N'$ are the unipotent radicals of $P$ and $P'$ and $\psi_F : F \rw \C^\times$ is a non-trivial additive character. This integral is absolutely convergent whenever the real part of $\lambda \in \mf{a}_{M,\C}^*$ belongs to a certain cone. As a function of $\lambda$, it has a meromorphic continuation to all of $\mf{a}_{M,\C}^*$. It is important to specify the measure $dn'$ precisely. The additive character $\psi_F : F \rw \C^\times$ specifies a Haar measure on the additive group $F$ that is self-dual with respect to $\psi_F$. To specify $dn'$ it is thus enough to give a top form on the vector space $\mf{n}(F) \cap \mf{n'}(F) \lmod \mf{n'}(F) \cong (\mf{\bar n} \cap \mf{n'})(F)$, where we use Gothic letters for the Lie algebras of $N$ and $N'$, and we denote by $\bar N$ the unipotent radical of the parabolic subgroup of $G$ that is $M$-opposite to $P$.

By assumption $\xi : M^* \rw M$ is an inner twist. Thus $P^*=\xi^{-1}(P)$ and $P'^*=\xi^{-1}(P')$ are parabolic subgroups of $G^*$ with Levi factor $M^*$ and defined over $F$. Then $\xi$ restricts to an isomorphism of $\ol{F}$-vector spaces $\xi : \mf{\bar n}^* \cap \mf{n}'^* \rw \mf{\bar n} \cap \mf{n}$. The standard pinning of $G^*$ can be used to obtain a Chevalley basis of $G^*$, i.e. a choice of a non-zero vector $X_\alpha \in \mf{g}^*_\alpha$ for each absolute root $\alpha \in R(T^*,G^*)$. Each vector $X_\alpha$ is determined up to multiplication by $\pm 1$. The vector space $\mf{\bar n}^* \cap \mf{n}'^*$ is a direct sum of root spaces $\mf{g}^*_\alpha$ and the corresponding $X_\alpha$ form a basis for it. Choose an arbitrary order of that basis and let $\eta^*$ be the top form with value $1$ on that ordered basis. Let $\eta$ be the transport of $\eta^*$ under $\xi$. It is a top form on the $\ol{F}$-vector space $\mf{\bar n} \cap \mf{n}'$. Up to multiplication by $\pm 1$, $\eta$ is independent of the choice of order as well as of the choice of $X_\alpha$. Let $a \in \ol{F}$ be such that $a\eta$ is defined over $F$. We define $dn'$ to be the Haar measure on $\bar N(F) \cap N'(F)$ given by $|a|_F^{-1}\cdot |d(a\eta)|_F$. It is independent of the choice of $a$. Having specified the measure $dn'$, the definition of the un-normalized operator $J_{P'|P}(\xi,\psi_F)$ is now complete.

Recall the modulus character $\delta_P : M(F) \rw \R_{>0}$. It is the restriction to $M(F)$ of the character
\[ M \rw \R_{>0},\qquad m \mapsto |\det(\tx{Ad}(m);\mf{n})|_F \]
which we will also denote by $\delta_P$. Here $|-|_F$ is the unique extension to $\ol{F}$ of the absolute value of $F$. It then follows immediately that for $m \in M^*$ the following equation holds
\begin{equation} \label{eq:iop1xc} J_{P'|P}(\xi\circ\tx{Ad}(m),\psi_F) = \delta_{P'}(\xi(m))^\frac{1}{2}\delta_P(\xi(m))^{-\frac{1}{2}} J_{P'|P}(\xi,\psi_F). \end{equation}

Next, we introduce the normalizing factor $r_{P'|P}(\psi_\lambda,\psi_F)$. Let $\hat M$ be the standard Levi subgroup of $\hat G$ corresponding to $M$ and let $\hat P$ and $\hat P'$ be the parabolic subgroups with Levi factor $\hat M$ dual to $P^*$ and $P'^*$. Let $A_{\hat M}=Z(\hat M)^{\Gamma,\circ}$.
The character $\exp(\<H_M(\cdot),\lambda\>)$ has the parameter $W_F \rw A_{\hat M}$ given by $w\mapsto |w|^\lambda$. The parameter of $\pi_\lambda$ is then $\psi_\lambda(w)=\psi(w)|w|^\lambda$. This is still an element of $\Psi(M^*)$. From it, we construct a Langlands parameter
\[ \phi_{\psi_\lambda} : L_F \rw {^LM},\qquad w \mapsto \psi_\lambda\left( w,\begin{bmatrix} |w|^\frac{1}{2}\\ &|w|^{-\frac{1}{2}}\end{bmatrix}\right). \]
Let $\rho_{P'|P}$ denote the adjoint representation of $^LM$ on $\mf{\hat n'} \cap \mf{\hat n} \lmod \mf{\hat n'}$. Following Arthur \cite[\S2.3]{Arthur} we set
\[ r_{P'|P}(\xi,\psi_\lambda,\psi_F) = \frac{L(0,\rho^\vee_{P'|P}\circ\phi_{\psi_\lambda})}{L(1,\rho^\vee_{P'|P}\circ\phi_{\psi_\lambda})} \frac{\epsilon(\frac{1}{2},\rho^\vee_{P'|P}\circ\phi_{\psi_\lambda},\psi_F)}{\epsilon(0,\rho^\vee_{P'|P}\circ\phi_{\psi_\lambda},\psi_F)}, \]
where we are using the Artin $L$- and $\epsilon$-factors of the given representation of $L_F$. We define the normalized intertwining operator $R_{P'|P}(\xi,\psi_\lambda)$ as the product
\[ R_{P'|P}(\xi,\psi_\lambda) = r_{P'|P}(\xi,\psi_\lambda,\psi_F)^{-1} J_{P'|P}(\xi,\psi_F). \]
It follows from  \cite[(3.6.6)]{TateCorvallis} that the dependence on $\psi_F$ of the two factors cancels out. It is known that the function $\lambda \mapsto R_{P'|P}(\xi,\psi_\lambda)$ extends meromorphically to all $\lambda \in \mf{a}_{M,\C}^*$. We will argue that it is defined and non-zero at $\lambda=0$ and will then set $R_{P'|P}(\xi,\psi)$ to be the value at $\lambda=0$.

The essential part of the argument involves the case where $\psi=\phi$ belongs to $\Phi(M^*)$, i.e. it is a tempered Langlands parameter. In that case we have $\phi_{\psi_\lambda}=\phi_\lambda$.

\begin{lem} \label{lem:iop1rm} Assume $F=\R$ and $\psi=\phi \in \Phi_{\bdd}(M^*)$. The function $\lambda \mapsto R_{P'|P}(\xi,\phi_\lambda)$ has neither a zero nor a pole at $\lambda=0$. If $P,P',P''$ are parabolic subgroups of $G$ defined over $F$ with Levi factor $M$, then
\[ R_{P''|P}(\xi,\phi) = R_{P''|P'}(\xi,\phi) \circ R_{P'|P}(\xi,\phi). \]
\end{lem}
\begin{proof}
The essential work has already been done by Arthur in \cite[\S3]{ArtIOR1}. In fact, the first statement follows directly from his work, as our operator $R_{P'|P}(\xi,\phi_\lambda)$ differs from the operator defined by Arthur only by multiplication by a positive real number coming from the different measures used in the definition of $J_{P'|P}$. To obtain the second statement, we need to compare the measures more carefully. First we notice that changing $\psi_F$ does not influence the validity of the second statement. We take $\psi_F$ to be the standard additive character $\psi_F(x)=e^{2\pi ix}$. Furthermore, notice that \eqref{eq:iop1xc} allows us to replace $\xi$ with an $M^*$-equivalent inner twist. We thus arrange the following: The maximal torus $S=\xi(T^*)$ of $G$ is defined over $F$, invariant under a Cartan involution $\theta$ of $G$, and for any choice of a Chevalley basis $X_\alpha$ extending the standard pinning of $G^*$, we have $\theta\sigma\xi(X_\alpha)=-\xi(Y_\alpha)$, where $Y_\alpha \in \mf{g}^*_{-\alpha}$ is chosen so that $[X_\alpha,Y_\alpha]=H_\alpha$.

Let $B$ be the symmetric bilinear form on $\mf{g}^*=\tx{Mat}(N,N;\C)$ given by $\tr(X\cdot Y)$. It is $G^* \rtimes \Gamma$-invariant. We use the same letter for the transport of $B$ to $\mf{g}$ via $\xi$, where it is again a $G \rtimes \Gamma$-invariant symmetric bilinear form. One checks immediately that each root $\alpha \in R(S,G) \subset \mf{s}^*$ is identified with its coroot $H_\alpha \in \mf{s}$ under the form $B$. The constant $\alpha_{P'|P}$ defined in \cite[\S3]{ArtIOR1} is thus equal to $1$. Moreover, the form $-B(X,\theta X)$ is positive definite on $\mf{g}(\R)$. To see this, write $\mf{g} = \mf{s} \oplus \mf{x}$, where $\mf{x}$ is the sum of all non-trivial root spaces for $S$. Both spaces $\mf{s}$ and $\mf{x}$ are defined over $\R$. Take first a non-zero $X \in \mf{s}(\R)$ and observe $-B(X,\theta X)=-B(X,\theta\sigma X)=-B(X,-\bar X)=\tx{tr}(X\bar X)>0$, where $\bar\ $ denotes complex conjugation of the entries of the matrix $X$. Next, focus on $\mf{x}$. It decomposes as direct sum of vector spaces $\mf{x}=\oplus_{a \in R(S,G)/\Gamma} \mf{g}_a$ and each factor $\mf{g}_a=\oplus_{\alpha \in a}\mf{g}_\alpha$ is defined over $\R$. Let $a \in R(S,G)/\Gamma$ and $\alpha \in a$. If $\alpha$ is a complex root, i.e. $\sigma\alpha \neq \pm\alpha$, then the $\R$-vector space $\mf{g}_a(\R)$ is two-dimensional and has basis $X,iX$, where $X=\xi(X_\alpha)+\sigma\xi(X_\alpha)$. Since $\theta\sigma\alpha=-\alpha$ we conclude that $\theta\alpha =-\sigma\alpha \neq \pm\alpha$ and thus $-B(X,\theta X)=2$. If $\alpha$ is a real root, i.e. $\sigma\alpha=\alpha$, the $\R$-vector space $\mf{g}_a(\R)$ is one-dimensional and has basis $X=\xi(X_\alpha)$ or $X=i\xi(X_\alpha)$. In either case we have $-B(X,\theta X)=1$. Finally, if $\alpha$ is an imaginary root, i.e. $\sigma\alpha=-\alpha$, then $\mf{g}_a(\R)$ is again one-dimensional and has basis $X=\xi(X_\alpha)+\sigma\xi(X_\alpha)$, but now $\sigma\alpha=-\alpha$ implies $\theta\alpha=\alpha$, which again leads to $-B(X,\theta X)=1$. Besides showing that the form $-B(X,\theta X)$ is positive definite on $\mf{g}(\R)$, this argument also exhibits an orthonormal basis for the vector space $\mf{n'}(\R) \cap \mf{\bar n}(\R)$. It follows that the measure on $N'(R) \cap \bar N(\R)$ defined by Arthur in loc. cit. can be characterized as follows: Fix the basis ${X,iX}$ with $X=\xi(X_\alpha)+\sigma\xi(X_\alpha)$ for $\mf{g}_a$ when $a$ consists of complex roots, and the basis $X=\xi(X_\alpha)$ or $X=i\xi(X_\alpha)$ when $a=\{\alpha\}$ for a real root $\alpha$. This gives a basis of $\mf{n'}(\R) \cap \mf{\bar n}(\R)$. Note that imaginary roots do not appear in that subspace of $\mf{x}$. Then Arthur's measure corresponds to a top form whose absolute value on that basis is equal to $2^k$, where $k$ is the number of summands $\mf{g}_a$ with $a$ consisting of complex roots. It is now straightforward to check that the top form $\eta$ used in our construction of $J_{P'|P}$ satisfies this property, bearing in mind the arrangement we have made for $\xi$. Thus, for this particular kind of $\xi$, our measure coincides with Arthur's and \cite[Theorem 2.1]{ArtIOR1} implies the claim.
\end{proof}

\begin{lem} \label{lem:iop1lg} Let $F$ be a global field, $E/F$ a quadratic algebra, $G^*=U_{E/F}(N)$ with standard Levi subgroup $M^*$, $\xi : G^* \rw G$ an inner twist restricting to an inner twist $\xi : M^* \rw M$, and $\pi$ a discrete automorphic representation of $M$ belonging to the global packet $\Pi_\psi(M,\xi)$. Then
\[ R_{P'|P}(\xi,\psi_\lambda) = \bigotimes_v R_{P'|P}(\xi_v,\psi_{\lambda,v}), \]
where the left-hand side is the global intertwining operator \eqref{eq:giop1r}.
\end{lem}
\begin{proof} Let $\psi_\A : \A/F \rw \C^\times$ be a non-trivial additive character, and let $\psi_{F_v}$ be its restriction to $F_v$ for all $v$. Henniart's result \cite{Hen10} implies that at each place $v$ the normalizing factor on the left-hand side, which involves quotients of automorphic $L$- and $\epsilon$-factors, is equal to the normalizing factor on the right-hand side, which involves quotients of Artin $L$- and $\epsilon$-factors. It remains to prove the analog of the claimed equation for the un-normalized operators $J_{P'|P}$ and this comes down to comparing the local measures $dn'$ used in the definition of $J_{P'|P}(\xi_v,\psi_{F_v})$ with the adelic measure $dn'$ used in the definition of $J_{P'|P}$ of equation \eqref{eq:giop1}. The top form $\eta^*$ of $\mf{n'}^* \cap \mf{\bar n}^*$ determined (up to $\{\pm 1\}$) by the distinguished pinning of $G^*$ is defined over $F$ and hence $\eta$ is defined over $\ol{F}$, which allows us to find $a \in \ol{F}$ so that $a\eta$ is defined over $F$. The product of the local measures $|a|_{F_v}^{-1}|d(a\eta)|_{F_v}$ is then equal to the adelic measure corresponding to the top form $a\eta$ and the measure on $\A$ that is self-dual with respect to $\psi_\A$. According to \cite[(3.3)]{Tate67}, the latter assigns $\A/F$ volume $1$ and we conclude that the product of the local measures involved in $J_{P'|P}(\xi_v,\psi_{F_v})$ is equal to the adelic measure involved in the global operator $J_{P'|P}$.
\end{proof}

\begin{lem*} \label{lem:iop1pm}
Let $F$ be a $p$-adic field and $\psi=\phi \in \Phi(M^*)$. Assume the validity of Theorem \ref{thm:locclass-single} and \ref{thm:main-global} for the proper Levi subgroups of $G$. The function $\lambda \mapsto R_{P'|P}(\xi,\phi_\lambda)$ has neither a zero nor a pole at $\lambda=0$.
If $P,P',P''$ are parabolic subgroups of $G$ defined over $F$ with Levi factor $M$, then
\[ R_{P''|P}(\xi,\phi) = R_{P''|P'}(\xi,\phi) \circ R_{P'|P}(\xi,\phi). \]
\end{lem*}
\begin{proof}
We assume that $E/F$ is an extension of fields. The case where $E/F$ is a split algebra and hence $G^*$ is a linear group will be treated in \cite{KMS_B}.

The case $G=G^*$ and $\xi=\tx{id}$ is part of our assumptions listed in Section \ref{sub:results-qsuni}. Equation \eqref{eq:iop1xc} extends this to the case that $G$ is quasi-split and $\xi$ is arbitrary (but necessarily cohomologically trivial).

We now consider the case when $G$ is not quasi-split. By the reduction steps described in \cite[\S2]{ArtIOR1}, we may assume that $\pi$ is a discrete series representation. To handle this case, we pass to a global situation by applying Lemma \ref{lem:globalize_for_iop1}. It provides a quadratic extension of global fields $\dot E/\dot F$ and two places $u,v$ of $\dot F$, both non-split in $E$, with $v$ archimedean, such that with $\dot F_u \cong F$ and $\dot E_u \cong E$. Moreover, if $\dot G^* = U_{\dot E/\dot F}(N)$ and $\dot M^* \subset \dot G^*$ is the standard Levi subgroup with $\dot M^*_{u}=M^*$, then we obtain an inner twist $\dot \xi : \dot G^* \rw \dot G$ such that $\dot G$ is quasi-split away from $u$ and $v$ and $\dot M = \dot\xi(\dot M^*)$ is defined over $\dot F$, and we obtain an isomorphism $\iota : \dot G_u \rw G$ restricting to $\dot M_u \rw M$. Finally, we obtain an irreducible discrete automorphic representation $\dot \pi$ of $G$ such that $\iota$ provides an isomorphism $\dot\pi_{u} \cong \pi$.

According to Langlands' theory of Eisenstein series, the operator on the left side of the equality in Lemma \ref{lem:iop1lg} is defined and non-zero at $\lambda=0$, and its un-normalized  version $J_{P'|P}$ satisfies the desired multiplicativity property. The normalizing factor $r_{P'|P}(\xi_v\phi_{\lambda,v},\psi_{F_v})$ at each place $v$ also satisfies the required multiplicativity, so it follows that the normalized global operator, i.e. the left-hand side of the equality in Lemma \ref{lem:iop1lg}, satisfies the required multiplicaitivity.

We now look at the right-hand side of Lemma \ref{lem:iop1lg}. At all places away from $u$ and $v$, the local normalized intertwining operator is defined and satisfies the required multiplicativity. This is due to the fact that $\dot G$ is quasi-split at these places and so these properties are part of  the assumptions listed in Section \ref{sub:results-qsuni}. The same is true at the real place $v$ by Lemma \ref{lem:iop1rm}. This forces the operator at the place $u$ to also be defined at $\lambda=0$ and multiplicative in $P',P$.
\end{proof}

To treat non-tempered representations, we need the analog of Lemmas \ref{lem:iop1rm} and \ref{lem:iop1pm} for general parameters $\psi \in \Psi(M^*)$. We formulate this lemma now, but postpone the proof to \cite{KMS_A}.

\begin{lem*} \label{lem:iop1m} Let $F$ be a local field of characteristic zero. Assume the validity of Theorem \ref{thm:locclass-single} and \ref{thm:main-global} for the proper Levi subgroups of $G$. The function $\lambda \mapsto R_{P'|P}(\xi,\phi_\lambda)$ has neither a zero nor a pole at $\lambda=0$.
If $P,P',P''$ are parabolic subgroups of $G$ defined over $F$ with Levi factor $M$, then
\[ R_{P''|P}(\xi,\psi) = R_{P''|P'}(\xi,\psi) \circ R_{P'|P}(\xi,\psi). \]
\end{lem*}

\begin{lem} \label{lem:iop123c} Let $n \in N(M,G)(F)$.
\begin{enumerate}
\item Given an isomorphism $\pi(n) : (V_\pi,\pi\circ\tx{Ad}(n)) \rw (V_\pi,\pi)$ we have an equality
$\mc{I}_{P'}^G(\pi(n))^{-1} \circ R_{P'|P}(\xi,\psi) \circ \mc{I}_P^G(\pi(n)) = R_{P'|P}(\xi,\psi)$
of intertwining operators $\mc{H}_P(\pi\circ\tx{Ad}(n)) \rw \mc{H}_{P'}(\pi\circ\tx{Ad}(n))$.
\item For a function $f$ defined on $G(F)$, let $[l(x)f](g) := f(x^{-1}g)$. Then we have an equality
$l(x)^{-1}\circ R_{P'|P}(\xi,\psi)\circ l(x) = R_{x^{-1}P'x|x^{-1}Px}(\tx{Ad}(x^{-1})\circ\xi,\psi)$
of intertwining operators $\mc{H}_{x^{-1}Px}(\pi) \rw \mc{H}_{x^{-1}P'x}(\pi)$.
\item For any $w \in W(M^*,G^*)(F)$ we have $R_{P'|P}(\tx{Ad}(\breve w)\circ\xi,\psi)=R_{P'|P}(\xi,w\psi)$.
\end{enumerate}
\end{lem}
\begin{proof}
For all points we may assume that $\pi$ is in general position, so that both the un-normalized operator $J_{P'|P}$ and the normalizing factor $r_{P'|P}$ are defined, as the general case follows from this by analytic continuation. The first point is clear, as the operators $\mc{I}_P^G(\pi(n))$ and $\mc{I}_{P'}^G(\pi(n))$ act on the values of the functions that comprise $\mc{H}_P(\pi)$ and $\mc{H}_{P'}(\pi)$, while the operator $R_{P'|P}$ acts on their variables.

For the second point we note that while $J_{P'|P}(\xi,\psi_F)$ is given by integration over $N(F) \cap N'(F) \lmod N'(F)$ with respect to a measure $dn'$, the operator $l(x)^{-1}\circ J_{P'|P}(\xi,\psi_F)\circ l(x)$ is given by integration over $x^{-1}Nx(F) \cap x^{-1}N'x(F) \lmod x^{-1}N'x(F)$ with respect to the measure $\tx{Ad}(x^{-1})_*dn'$. This is the same as the operator $J_{x^{-1}P'x|x^{-1}Px}(\tx{Ad}(x^{-1})\circ\xi,\psi_F)$. On the other hand, by definition we have $r_{P'|P}(\xi,\psi,\psi_F)=r_{x^{-1}P'x|x^{-1}Px}(\tx{Ad}(x^{-1})\circ\xi,\psi,\psi_F)$.

For the last point, we have $r_{P'|P}(\xi\circ\tx{Ad}(\tilde w),\psi,\psi_F) = r_{P'|P}(\xi,w\psi,\psi_F)$. At the same time, the difference between $J_{P'|P}(\xi\circ\tx{Ad}(\tilde w)^{-1},\psi_F)$ and $J_{P'|P}(\xi,\psi_F)$ comes from transporting the top form $\eta^*$ to $\mf{n'} \cap \mf{\bar n}$ either via $\xi$ or via $\xi\circ\tx{Ad}(\tilde w)$. But $\tx{Ad}(\tilde w)$ preserves the pinning of $G^*$ so the top forms $\eta^*$ and $\tx{Ad}(\tilde w)^*\eta^*$ are equal.
\end{proof}

\subsection{Local intertwining operator II} \label{sec:iop2}

We keep the assumptions made in the beginning of the previous section: $F$ is a local field, $E/F$ is a quadratic algebra, $G^*=U_{E/F}(N)$, $M^* \subset G^*$ is a proper standard Levi subgroup, and $\xi : G^* \rw G$ is an inner twist that restricts to an inner twist $\xi : M^* \rw M$. Let $P$ be a parabolic subgroup of $G$ defined over $F$ with Levi factor $M$. Let $\psi \in \Psi(M^*)$ and $\pi \in \Pi_\psi(M,\xi)$.

In addition, we now assume that $P^*=\xi^{-1}(P)$ is a standard parabolic subgroup of $G^*$. For an element $w \in W(M^*,G^*)^\Gamma$ we let $\tilde w \in N(T^*,G^*)(F)$ be its Langlands-Shelstad lift as described in Section \ref{subsub:wmgu}. We assume that for $\sigma \in \Gamma$ the inner automorphism $\xi^{-1}\sigma(\xi)$ of $G^*$ fixes the $\tilde w$ for any $w \in W(M^*,G^*)^\Gamma$. This implies that $\breve w := \xi(\tilde w)$ belongs to $M(M,G)(F)$. The twist $\xi$ induces a $\Gamma$-equivariant isomorphism $W(M^*,G^*) \cong W(M,G)$ and we use it to regard $w$ as an element of $W(M,G)$. Then $\breve w$ is a lift of $w$.

We will describe a normalized intertwining operator
\[ l_P(w,\xi,\psi,\psi_F) : \mc{H}_{w^{-1}Pw}(\pi) \rw \mc{H}_{P}(\breve w\pi), \]
where the representation $\breve w\pi$ of $M(F)$ is defined by $\breve w\pi = \pi\circ\tx{Ad}(\breve w^{-1})$. We begin with the un-normalized intertwining operator
\[ l(\breve w) : \mc{H}_{w^{-1}P}(\pi) \rw \mc{H}_{P}(\breve w\pi),\qquad [l(\breve w)f](g) = f(\breve w^{-1}g), \]
where $f \in \mc{H}_{w^{-1}P}(\pi)$ and $g \in G(F)$. The map $w \mapsto \tilde w$ is not multiplicative and this failure is inherited by the map $w \mapsto \breve w$. For $w,w' \in W(T^*,G^*)$, Langlands and Shelstad provide in \cite[\S2.1]{LS87} a formula \[ \tilde w \cdot \tilde w' = t(w,w') \cdot \widetilde{ww'}, \]
for the failure of multiplicaitivity, with $t \in Z^2(W(T^*,G^*),T^*)$. When $w,w'$ belong to $W(M^*,G^*)$ the argument used in the proof of \cite[Lemma 2.3.4]{Arthur}, which will also be recalled in the proof below, shows that we have $t(w',w) \in A_{M^*}(F)$, where $A_{M^*}$ is the maximal split central torus of $M^*$. The failure of multiplicativity of $w \mapsto \breve w$ is then measured by $\xi(t(w',w)) \in A_M$. This failure necessitates a renormalization of the operator $l(\breve w)$ in order to obtain an operator that has the desired multiplicative properties. For this we follow the strategy outlined in \cite[\S2.3]{Arthur}. Let $\hat M,\hat P,\hat P'$ be dual to $M^*,P^*,P'^*=w^{-1}P^*w$ and let $\rho_{w^{-1}Pw|P}$ denote the adjoint representation of $^LM$ on the vector space $\mf{\hat n'} \cap \mf{\hat n} \lmod \mf{\hat n'} \cong \mf{\hat n'} \cap \mf{\hat{\bar n}}$, where $\mf{\hat n}$, $\mf{\hat n'}$, and $\mf{\hat{\bar n}}$ denote the Lie algebras of the unipotent radicals of $\hat P$, $\hat P'$ and the parabolic subgroup $\hat{\bar P}$ that is $\hat M$-opposite to $\hat P$. We then have the Artin $\epsilon$-factor $\epsilon(\frac{1}{2},\rho^\vee_{w^{-1}Pw|P}\circ\phi_\psi,\psi_F)$. Ideally, we would like to use it for our normalization purposes, but for global reasons we need to use its representation-theoretic analog instead. For this, recall the Levi subgroup $\tilde M$ of $\tx{Res}_{E/F}\tx{GL}(N)$ associated to $M^*$ as described in Section \ref{sec:giop} and let $\pi_\psi$ be the representation of $\tilde M(F)$ associated to $\phi_\psi$. We let
\[ \epsilon_P(w,\psi,\psi_F) := \epsilon(\frac{1}{2},\pi_\psi,\rho^\vee_{w^{-1}Pw|P},\psi_F). \]
Henniart has shown in \cite{Hen10} that this factor is equal to its Artin analog up to multiplication by a root of unity. Work in progress of Cogdell-Shahidi-Tsai aims at showing that this root of unity is equal to $1$. Until this has been proven, we will need to use the representation theoretic $\epsilon$-factor instead of the Artin factor.

Besides the $\epsilon$-factor, the work of Keys-Shahidi has shown that another normalizing factor needs to be added, namely the constant $\lambda(w,\psi_F)$ defined in \cite[(4.1)]{KS88} with respect to the canonical lift of $w \in W(M^*,G^*)^\Gamma$ to an element of $W(T^*,G^*)^\Gamma$.

Following Arthur, we define
\begin{equation} \label{eq:iop2} l_P(w,\xi,\psi,\psi_F) := \epsilon_{P}(w,\psi,\psi_F) \cdot \lambda(w,\psi_F)^{-1} \cdot l(\breve w). \end{equation}

\begin{lem} \label{lem:iop2m} For any $w,w' \in W(M,G)(F)$ we have
\[ l_P(w'w,\xi,\psi,\psi_F) = l_P(w',\xi,w\psi,\psi_F) \circ l_P(w,\xi,\psi,\psi_F). \]
\end{lem}

Before beginning with the proof, we comment briefly on the statement, which employs a slight abuse of notation. The operator $l_P(w,\xi,\psi,\psi_F)$ was defined as a linear map $\mc{H}_{w^{-1}Pw}(\pi) \rw \mc{H}_P(\breve w \pi)$ for any $\pi \in \Pi_\psi(M,\Xi_M)$. However, the same formula gives an operator $\mc{H}_{w^{-1}w'^{-1}Pw'w}(\pi) \rw \mc{H}_{w'^{-1}Pw'}(\breve w\pi)$, and this is what we are using here. Moreover, the operator $l_P(w',\xi,w\psi,\psi_F)$ was defined as an operator $\mc{H}_{w'^{-1}Pw'}(\pi') \rw \mc{H}_P(\breve w'\pi')$ for any $\pi' \in \Pi_{w\psi}(M,\Xi_M)$ and we are applying it here to $\pi' = \breve w\pi$.

\begin{proof}
We follow closely the proof of Lemma 2.3.4 in \cite{Arthur}. The right hand side operator sends $f \in \mc{H}_P(\pi)$ to the function
\[ g \mapsto \epsilon_{P}(w',w\psi,\psi_F)\epsilon_{P}(w,\psi,\psi_F)\lambda(w',\psi_F)^{-1}\lambda(w,\psi_F)^{-1}f(\breve w^{-1}\breve w'^{-1}g), \]
while the left hand side operator sends $f$ to the function
\[ g \mapsto \epsilon_{P}(w'w,\psi,\psi_F)\lambda(w'w,\psi_F)^{-1}\eta_\pi(\xi((w'w)^{-1}t(w',w)w'w))f(\breve w^{-1} \breve w'^{-1}g), \]
where $\eta_\pi$ is the central character of the representation $\pi$. Our goal is to show
\begin{eqnarray} \label{eq:iop2eq}
\eta_\pi(\xi((w'w)^{-1}t(w',w)w'w))&=&\lambda(w'w,\psi_F)\lambda(w,\psi_F)^{-1}\lambda(w',\psi_F)^{-1} \\
&\cdot&\epsilon_P(w',w\psi,\psi_F)\epsilon_P(w,\psi,\psi_F)\epsilon_P(w'w,\psi,\psi_F)^{-1}. \nonumber
\end{eqnarray}
We study first the left hand side and begin by recalling the argument that $t(w',w) \in A_{M^*}$. According to \cite[Lemma 2.1.A]{LS87}, we have
\[ (w'w)^{-1}t(w',w)w'w = \prod_{\alpha>0,w\alpha<0,w'w\alpha>0} \alpha^\vee(-1) = (-1)^{\sum_{\alpha>0,w\alpha<0,w'w\alpha>0}\alpha^\vee} \]
for $\alpha \in R(T^*,G^*)$ and positivity taken with respect to $B^*$. The set $\{\alpha>0,w\alpha<0,w'w\alpha>0\}$ is preserved by $\Gamma$, because the elements $w',w \in W(T^*,G^*)$ and the Borel subgroup $B^*$ are. This set is also preserved by any element $u \in W(T^*,M^*)$. This is because both $w',w$ normalize $W(T^*,M^*)$ and because $u$ preserves the positive roots in $R(T^*,G^*) \sm R(T^*,M^*)$.
It follows that
\[ \lambda := \sum_{\alpha>0,w\alpha<0,w'w\alpha>0}\alpha^\vee \in X_*(A_{M^*}). \]
This proves $t(w',w) \in A_{M^*}$ and also tells us that the left-hand side of \eqref{eq:iop2eq} is given by $\eta_\pi(\xi((-1)^\lambda))$.

The dual torus to $Z(M^*)$ is $\hat M/\hat M_\tx{sc,der}$, and the dual torus to $A_{M^*}$ is $[\hat M/\hat M_\tx{sc,der}]_\Gamma$.\footnote{Recall that we are following Arthur's convention that $\hat M_\tx{sc}$ is the preimage of $\hat M$ in $\hat G_\tx{sc}$, and $\hat M_\tx{sc,der}$ is the derived subgroup of $\hat M_\tx{sc}$, which coincides with the simply-connected cover of the derived subgroup of $\hat M$.} This gives a canonical isomorphism $X_*(A_{M^*}) \cong X^*(\hat M)^\Gamma$. According to Theorem \ref{thm:locclass-single}, $\eta_\pi$ is the character of $Z(M^*)(F)$ with parameter $W_F \stackrel{\phi_\psi}{\lrw} \hat M \rtimes W_F \rw \hat M/\hat M_\tx{sc,der} \rtimes W_F$. Composing this parameter with the projection $\hat M/\hat M_\tx{sc,der} \rtimes W_F \rw [\hat M/\hat M_\tx{sc,der}]_\Gamma \times W_F$ gives the parameter of the restriction of $\eta_\pi$ to $A_{M^*}$. We conclude that if $x \in W_F$ is any element with image $-1 \in F^\times$ under the Artin reciprocity map and $\phi_\psi(x)=m(x) \rtimes x$, then the left hand side of \eqref{eq:iop2eq} is equal to $\lambda(m(x))$.

Turning to the right hand side of \eqref{eq:iop2eq}, we study first the product of the three $\epsilon$-factors. Decompose $\mf{\hat g}$ under the action of $A_{\hat M}$ as a direct sum of weight spaces $\hat g_\beta$ for $\beta \in R(A_{\hat M},\hat G)$ and denote by $\beta>0$ those weights for which $\mf{\hat g}_\beta \subset \mf{\hat n}$. The adjoint representation of $^LM$ on the space $\mf{\hat n}' \cap \mf{\hat{\bar n}}$ is then the direct sum of $\mf{\hat g}_\beta$ for $\beta<0$ with $w\beta<0$, and its contragredient is the direct sum over $\beta>0$ with $w\beta<0$. If we denote the adjoint action of $^LM$ on $\mf{\hat g}_\beta$ by $\rho_\beta$, then the factor $\epsilon_P(w,\phi,\psi_F)$ decomposes as the product
\[ \epsilon_P(w,\psi,\psi_F) = \prod_{\beta>0,w\beta<0} \epsilon(\frac{1}{2},\pi_\psi,\rho_\beta,\psi_F). \]
The product of the three $\epsilon$-factors on the right hand side of \eqref{eq:iop2eq} is thus equal to
\[ \prod_{w\beta>0,w'w\beta<0}f(\beta) \cdot \prod_{\beta>0,w\beta<0}f(\beta) \cdot \prod_{\beta>0,w'w\beta<0} f(\beta)^{-1} \]
with $f(\beta)=\epsilon(\frac{1}{2},\pi_\psi,\rho_\beta,\psi_F)$. An elementary calculation shows that this triple product equals
\[ \prod_{\beta>0,w\beta<0,w'w\beta>0} f(\beta)\cdot f(-\beta). \]
Applying Henniart's result \cite{Hen10}, we conclude that if we replace $f$ by the function $f(\beta)=\epsilon(\frac{1}{2},\rho_\beta\circ\phi_\psi,\psi_F)$ that uses the Artin $\epsilon$-factor instead of the representation-theoretic one, the above product is unchanged. Letting $x \in W_F$ is any element whose image under the reciprocity map $W_F \rw F^\times$ is equal to $-1$ and using \cite[(3.6.8)]{TateCorvallis}, the above product then equals to
\[ \prod_{\beta>0,w\beta<0,w'w\beta>0} \det(\rho_\beta\circ\phi_\psi(x)). \]

The indexing set of the last displayed product is a subset of $X^*(A_{\hat M})$ and equals the restriction to $A_{\hat M}$ of the subset of $X^*(\hat T)$ given by $\{\alpha^\vee|\alpha \in R(T^*,G^*), \alpha>0,w\alpha<0,w'w\alpha>0\}$. Writing again $\phi_\psi(x)=m(x) \rtimes x \in \hat M \rtimes W_F$ we see that the last displayed product equals
\[ \tx{det}\left(\tx{Ad}(m(x) \rtimes x) \left| \bigoplus_{\alpha>0,w\alpha<0,w'w\alpha>0}\mf{\hat g}_{\alpha^\vee}\right.\right). \]
The determinant of the action of the maximal torus $\hat T \subset \hat M$ on this direct sum is the character of $\hat T$ given by the restriction to $\hat T$ of $\lambda \in X^*(\hat M)$. On the other hand, the action of $\hat M_\tx{sc,der}$ on this direct sum has trivial determinant. We conclude that
\[ \tx{det}\left(\tx{Ad}(m(x) \rtimes 1) \left| \bigoplus_{\alpha>0,w\alpha<0,w'w\alpha>0}\mf{\hat g}_{\alpha^\vee}\right.\right) \]
is equal to $\lambda(m(x))$, i.e. to the left-hand side of \eqref{eq:iop2eq}.

To complete the proof of \eqref{eq:iop2eq} we must show that
\begin{equation} \label{eq:iop2eq2} \tx{det}\left(\tx{Ad}(1 \rtimes x) \left| \bigoplus_{\alpha>0,w\alpha<0,w'w\alpha>0}\mf{\hat g}_{\alpha^\vee}\right.\right) = \lambda(w'w,\psi_F)^{-1}\lambda(w,\psi_F)\lambda(w',\psi_F). \end{equation}

The definition of $\lambda(w,\psi_F)$ given in \cite[(4.1)]{KS88} and specialized to our setting is
\[ \lambda(w,\psi_F) = \prod_{a \in \Delta_1(w)} \lambda(E/F,\psi_F) \cdot \prod_{a \in \Delta_2(w)} \lambda(E/F,\psi_F)^2, \]
where $\Delta_1(w)$ and $\Delta_2(w)$ are the sets of reduced relative roots $a \in X^*(A_{T^*})$ with $a>0$ and $wa<0$ for which the corresponding rank-1 semi-simple group has simply connected cover isomorphic to $\tx{Res}_{E/F}\tx{SL}(2)$ and $SU_{E/F}(3)$, respectively. Here $\lambda(E/F,\psi_F)$ is the Langlands constant \cite[Thm 2.1]{LanArt}. It satisfies $\lambda(E/F,\psi_F)^2=\eta_{E/F}(-1)$, where $\eta_{E/F} : F^\times/N_{E/F}(E^\times) \rw \{\pm 1\}$ is the non-trivial sign character coming from local class field theory. We apply the argument that we used above to study the product of the three $\epsilon$-factors again, first to the set $\Delta_1(w)$ and the function $f(a)=\lambda(E/F,\psi_F)$, and then to the set $\Delta_2(w)$ and the function $f(a)=\eta_{E/F}(-1)$. Initially we obtain in both cases the product
\[ \prod_{a>0,w'a<0}f(a) \cdot \prod_{a>0,wa<0}f(a) \cdot \prod_{a>0,w'wa<0} f(a)^{-1} \]
but since $w$ is defined over $F$ the function $f$ is invariant under $w$ and the first product can be taken over the set $wa>0,w'wa<0$ instead. In this way we obtain
\[ \lambda(w'w,\psi_F)^{-1}\lambda(w,\psi_F)\lambda(w',\psi_F) = \eta_{E/F}(-1)^t, \]
where $t$ is the number of reduced relative roots $a \in X^*(A_{T^*})$ with $a>0,wa<0,w'wa>0$ for which the corresponding rank-1 semi-simple subgroup of $G^*$ has simply connected cover isomorphic to $\tx{Res}_{E/F}\tx{SL}(2)$. Now $\eta_{E/F}(-1)=1$ if and only if $x \in W_E$. If this is the case, then both sides of \eqref{eq:iop2eq2} are equal to $1$. Assume thus that this is not the case and consider the left-hand side of \eqref{eq:iop2eq2}. We will distinguish four types of roots $\alpha \in R(T^*,G^*)$ with $\alpha>0,w\alpha<0,w'w\alpha>0$. The first type are those $\alpha$ for which $x\alpha \neq \alpha$ and $x\alpha+\alpha \notin R(T^*,G^*)$. Those are precisely those roots whose restriction $a=\alpha|_{T^*} \in X^*(A_{T^*})$ is a reduced relative root contributing a copy of $\tx{Res}_{E/F}(\tx{SL}(2))$, and hence a factor of $\eta_{E/F}(-1)=-1$ to right-hand side of \eqref{eq:iop2eq2}. At the same time, the subspace $\mf{\hat g}_{\alpha^\vee} \oplus \mf{\hat g}_{x\alpha^\vee}$ is $x$-invariant and the action of $x$ on it has determinant $-1$. The second type of roots $\alpha \in R(T^*,G^*)$ are those for which $x\alpha \neq \alpha$ but $x\alpha+\alpha \in R(T^*,G^*)$. Those are precisely those roots whose restriction $a \in X^*(A_{T^*})$ is a reduced relative root contributing a copy of $SU_{E/F}(3)$. They have no contribution to the right-hand side of \eqref{eq:iop2eq2}. At the same time, the subspace $\mf{\hat g}_{\alpha^\vee} \oplus \mf{\hat g}_{x\alpha^\vee} \oplus \mf{\hat g}_{\alpha^\vee+x\alpha^\vee}$ is $x$-invariant, with $x$ swapping the first two factors and acting by $-1$ on the third, hence having determinant $+1$. Thus, this subspace also has no contribution to the left-hand side of \eqref{eq:iop2eq2}. The third type of roots $\alpha \in R(T^*,G^*)$ are those of the form $\beta+x\beta$ for $\beta \in R(T^*,G^*)$. Their reduction $a \in X^*(A_{T^*})$ is not a reduced root, so has no contribution to the right-hand side of \eqref{eq:iop2eq2}. On the other hand, the corresponding subspace $\mf{\hat g}_{\alpha^\vee}$ has already been subsumed into the treatment of the second type of roots and thus need not be considered again. The fourth type of roots are those $\alpha \in R(T^*,G^*)$ with $\alpha=x\alpha$ but not of type 3. They do not contribute to the right hand side of \eqref{eq:iop2eq2}, because the corresponding rank-1 simply connected group is isomorphic to $\tx{SL}(2)$. The action of $x$ preserves and in fact pointwise fixes the subspace $\mf{\hat g}_{\alpha^\vee}$, hence this space does not contribute to the left-hand side of \eqref{eq:iop2eq2} either.
\end{proof}

The following is immediate from the construction.

\begin{fct} \label{fct:iop23c}
Given $n \in N(M,G)(F)$ and an isomorphism $\pi(n) : (V_\pi,\pi\circ\tx{Ad}(n)) \rw (V_\pi,\pi)$, we have an equality
\[ \pi(n)^{-1} \circ l_P(w,\xi,\psi,\psi_F) \circ \pi(n) = l_P(w,\xi,\psi,\psi_F) \]
of intertwining operators $\mc{H}_{w^{-1}Pw}(\pi\circ\tx{Ad}(n)) \rw \mc{H}_{P}(\pi\circ\tx{Ad}(n))$.
\end{fct}

\subsection{Local intertwining operator III} \label{sec:iop3b}
We maintain the assumptions of the previous section: $F$ is a local field, $E/F$ is a quadratic algebra, $G^*=U_{E/F}(N)$, $(M^*,P^*)$ is a proper standard parabolic pair of $G^*$ defined over $F$ and its image $(M,P)$ under $\xi$ is a parabolic pair of $G$ defined over $F$. Let $\psi \in \Psi(M^*)$ and $\pi \in \Pi_\psi(M,\xi)$.

We now assume further that we are given $z \in Z^1_\tx{G^*-bsc}(\mc{E},M^*)$ such that $(\xi,z)$ is an extended pure inner twist. The element $z$ is assumed to commute with the Langlands-Shelstad lifts $\tilde w \in N(T^*,G^*)$ of all $w \in W(M^*,G^*)^\Gamma$. Furthermore, let $u \in N(\hat T,\hat G)$ be such that it commutes with $\psi$ and preserves the $\hat B$-positive roots in $R(\hat T,\hat M)$. Let $u^\natural \in N^\natural_\psi(M,G)$ and $w \in W(\hat M,\hat G)^\Gamma$ be the images of $u$. Via the $\Gamma$-equivariant isomorphisms $W(\hat M,\hat G) \cong W(M^*,G^*) \cong W(M,G)$ we may regard $w$ as an element in any of these groups.

Recall the element
\[ s_\psi = \psi(1,\begin{bmatrix}-1\\&-1\end{bmatrix}) \in \hat G. \]
It belongs to the center of the image of $\psi$ and hence also to the center of $S_\psi(M)$.

We are going to define an intertwining operator $\pi(u^\natural)_{\xi,z} : (\breve w\pi,V_\pi) \rw (\pi,V_\pi)$. We will first treat the cases of unitary groups and linear groups individually, where the definition can be made quite explicit under additional assumptions. Afterwards, we will provide a uniform construction that works under less assumptions and has better invariance properties, but is less explicit. This will be needed in both the local and global applications ahead.

Before we continue, we recall a discussion from \cite[\S2.2]{Arthur} about using Whittaker data to normalize intertwining operators. Let $M^*_+$ be a product of groups of the form $G_{E/F}(N)$ and let $\theta$ be an automorphism of $M^*_+$ that preserves the standard pinning of that group. Let $\psi_F : F \rw \C^\times$ be a non-trivial character. Let $\pi$ be an irreducible admissible representation of $M^*_+(F)$ whose isomorphism class is stable under $\theta$. In this situation there exists a natural choice for an isomorphism $\pi\circ\theta^{-1} \rw \pi$.

In the case when $\pi$ is tempered, we choose a Whittaker functional $\mu$ for the representation $\pi$ and the Whittaker datum corresponding to $\psi_F$ and the standard pinning of $M^*_+$. Then $\mu$ is also a Whittaker functional for the representation $\pi\circ\theta^{-1}$, since the pinning of $M^*_+$ is stable under $\theta$. We choose $\pi\circ\theta^{-1} \rw \pi$ to be the unique isomorphism which preserves the Whittaker functional $\mu$.

More generally, $\pi$ is the Langlands quotient of a standard representation $\rho=\tx{Ind}_{P_0^*}^{M^*_+}(\sigma_\lambda)$ for a tempered representation $\sigma$ of a standard Levi subgroup $M_0^*$ of $M^*_+$ and a positive $\lambda \in \mf{a}_{M_0^*}^*$. Then $\pi\circ\theta^{-1} \cong \pi$ implies that $M_0^*$ is stable under $\theta$ and $\sigma\circ\theta^{-1} \cong \theta$. Since $P_0^*$ is generated by $M_0^*$ and the standard Borel subgroup of $M^*_+$, it is also invariant under $\theta$. We fix an isomorphism $\sigma\circ\theta^{-1} \rw \theta$ as above and this leads to an isomorphism $\rho\circ\theta^{-1} \rw \rho$. This isomorphism then descends to an isomorphism $\pi\circ\theta^{-1} \rw \pi$.

\subsubsection{The case of unitary groups} \label{sec:iop3u}
We assume now that $E/F$ is a quadratic field extension, so that $G^*$ is a unitary group. Moreover, we place the following assumption on the cocycle $z \in Z^1_\tx{G^*-bsc}(\mc{E},M^*)$: If we decompose it as $z=z_+ \times z_-$ according to the decomposition $M^*=M^*_+ \times M^*_-$, then the 1-cocycle $z_+$ takes the constant value $1 \in M^*_+$. This can always be achieved by changing $(\xi,z)$ within its equivalence class. In this special case the operator $\pi(u^\natural)_{\xi,z}$ can be defined as the product $\<\pi,u^\natural\>_{\xi,z} \cdot \pi(\breve w)_\xi$, where $\pi(\breve w)_\xi : (\breve w\pi,V_\pi) \rw (\pi,V_\pi)$ is an operator that depends only on the image $w \in W_\psi(M,G)$ of $u^\natural \in N^\natural_\psi(M,G)$, and $\<\pi,u^\natural\>_{\xi,z} \in \C$.

The isomorphism $\pi(\breve w)_\xi$ is given as follows. Decompose $M^*= M^*_+ \times M^*_-$, $M = M_+ \times M_-$, and accordingly $\pi = \pi_+ \otimes \pi_-$. We will choose $\pi_-(\breve w)$ and $\pi_+(\breve w)$ separately and define $\pi(\breve w)_\xi = \pi_+(\breve w) \otimes \pi_-(\breve w)$. The automorphism $\tx{Ad}(\breve w)$ centralizes $M^*_-$, hence $\breve w\pi_- = \pi_-$ and we choose $\pi_-(\breve w)= \tx{id}$. The restriction of $\xi$ to $M^*_+$ provides an isomorphism $M^*_+ \rw M_+$ defined over $F$. We treat $\pi_+$ as a representation of $M^*_+$ via this isomorphism and take  $\pi_+(\breve w) : \breve w\pi_+ \rw \pi_+$ to be the isomorphism discussed in Section \ref{sec:iop3b} above. The map $w \mapsto \pi(\breve w)_\xi$ thus defined is multiplicative.

Next we construct $\<\pi,u^\natural\>_{\xi,z} \in \C$. Consider the middle row of diagram \eqref{eq:diag}. In the setting of unitary groups, this middle row has the form
\begin{equation} \label{eq:esuni} 1 \rw \pi_0(S_\psi(M)) \rw \pi_0(N_\psi(M,G)) \stackrel{p}{\lrw} W_\psi(M,G) \rw 1. \end{equation}
We will construct a splitting of this exact sequence. Notice that we have the decomposition $N_\psi(M,G) = S_{\psi_-}(M_-) \times N_{\psi_+}(M_+,G)$, which expresses $u \in N_\psi(M,G)$ as the product $u_- \times u_+$, with $u_-$ being the middle $N_- \times N_-$-block of the matrix $u$, and $u_+$ being obtained from $u$ by replacing the middle $N_- \times N_-$-block with the identity matrix. Since $\pi_0(S_\psi(M))=\pi_0(S_{\psi_-}(M_-))$, this decomposition would provide a splitting of the sequence \eqref{eq:esuni}. This is however \emph{not} the splitting that we want. The splitting that we are going to construct will have the property that an element $u \in N_\psi(M,G)$, whose image in $W_\psi(M,G)$ belongs $W^0_\psi(M,G)$, may be mapped to a non-trivial element of $\pi_0(S_\psi(M))$. It turns out that this more sophisticated splitting is needed in order to provide the correct normalization of the compound intertwining operator, as well as its invariance under equivalences of the extended pure inner twist $(\xi,z)$. We refer the reader to Section \ref{sec:lirexp} for an example of why the naive splitting does not provide the right normalization, as well as to the proofs of Lemmas \ref{lem:iopind} and \ref{lem:lirind} for how the splitting we are about to define ensures the invariance of the compound operator and of the Local Intertwining Relation.

To construct the correct splitting, we use the distinguished pinning of $\hat G$. For any element $w \in W_\psi(M,G)$, the process described in Section \ref{subsub:wmgu} lifts $w$ first to an element of $W(\hat T,\hat G)^\Gamma$, and then to an element of $N(\hat T,\hat G)^\Gamma$ via the construction of Langlands and Shelstad. We will call this element $\tilde w$. The element $\tilde w$ preserves the $\hat M$-conjugacy class of $\psi$, but may not centralize $\psi$. However, it follows from Lemma \ref{lem:lsl} that it does centralize $^LM_-$, hence $\psi_-$. Thus there exists an element $m \in \hat M_+$ such that $m\tilde w \in N_\psi(M,G)$. Since $S_{\psi_+}(M_+)$ is connected, the image of $m\tilde w$ in $\pi_0(N_\psi(M,G))$ does not depend on the choice of $m$. We will call it $s'(w)$. We claim that the map $s' : W_\psi(M,G) \rw \pi_0(N_\psi(M,G))$ thus obtained is multiplicative. For this, recall from Section \ref{subsub:wmgu} that for any $w \in W(\hat M,\hat G)^\Gamma$ we have $\tilde w = t_{w}\hat w$ with $t_{w} =\hat w\tilde w^{-1}\in Z(\hat M)$ an element of order $2$, as in Lemma \ref{lem:lsl}. Recall that the map $w \mapsto \hat w$ is multiplicative by construction. We further write $t_w = t_{w,-}\cdot t_{w,+}$ according to $Z(\hat M)=Z(\hat M_-) \times Z(\hat M_+)$ and know that the map $w \mapsto t_{w,-}$ is multiplicative by Lemma \ref{lem:tw-} and that $Z(\hat M_+)$ is preserved by $W(\hat M,\hat G)^\Gamma$. Thus, given $w_1,w_2 \in W(\hat M,\hat G)^\Gamma$, the difference between the elements $N(A_{\hat M},\hat G)$ given by $\tilde{w_1w_2}$ and $\tilde w_1 \cdot \tilde w_2$ is given by an element of $Z(\hat M_+)$. This proves the multiplicativity of $s'$.

Recalling the decomposition $N_\psi(M,G) = S_{\psi_-}(M_-) \times N_{\psi_+}(M_+,G)$ we see that $\pi_0(S_\psi(M))$ is a central subgroup of $\pi_0(N_\psi(M,G))$. Thus the splitting $s' : W_\psi(M,G) \rw \pi_0(N_\psi(M,G))$ of the second map in \eqref{eq:esuni} leads to a splitting $s : \pi_0(N_\psi(M,G)) \rw \pi_0(S_\psi(M))$ of the first map in \eqref{eq:esuni} in the usual way, i.e. $s(u)=u\cdot s'p(u)^{-1}$. Recall the character $\<\pi,-\>_{\xi,z}$ of $\pi_0(S_\psi(M))$ associated to the representation $\pi$ by Theorem \ref{thm:locclass-single}. We define
\[ \<\pi,u^\natural\>_{\xi,z} = \<\pi,s(u^\natural)^{-1}\>_{\xi,z}. \]
Note that it is multiplicative in $u^\natural$ by construction.

\subsubsection{The case of linear groups} \label{sec:iop3l}

Assume now that $E/F$ is the split quadratic algebra $F \oplus F$, so that $G^*=\tx{GL}(N)/F$. Decompose $M^*=M^*_1 \times \dots \times M^*_n$. Conjugation by the element $\tilde w$ acts by permuting these factors. Let $a$ be the permutation of $\{1,\dots,n\}$ such that $\tx{Ad}(\tilde w)M^*_i=M^*_{a(i)}$. We assume that the element $z \in Z^1_\tx{G-bsc}(\mc{E},M^*)$ satisfies the following assumption: if we decompose it as $z=z_1 \times \dots \times z_n$ with $z_i \in Z^1_\tx{bsc}(\mc{E},M^*_i)$, then $z_{a(i)}=z_i$ for all $1 \leq i \leq n$. This is equivalent to requiring that $z$ commute with $\tilde w$.

We will define the operator $\pi(u^\natural)_{\xi,z} : (\breve w\pi,V_\pi) \rw (\pi,V_\pi)$ under this assumption. It will again be given as $\<\pi,u^\natural\>_{\xi,z}^{-1}\cdot\pi(\breve w)_\xi$ with $\<\pi,u^\natural\>_{\xi,z} \in \C$ and $\pi(\breve w)_\xi : (\breve w\pi,V_\pi) \rw (\pi,V_\pi)$ depending only on $w$.  We decompose $M=M_1 \times \dots \times M_n$ and fix an isomorphism
$\iota: (\pi,V_\pi) \rw (\pi_1,V_1) \otimes \dots \otimes (\pi_n,V_n)$ whose target we choose so that for all $i$, the vector space $V_i$ and $V_{a(i)}$ are the same, and $\tx{Ad}(\breve w)$ intertwines the representations $\pi_i$ and $\pi_{a(i)}$ on the vector space $V_i=V_{a(i)}$.
Let $\theta \in \tx{Aut}(V_1 \otimes \dots \otimes V_{n})$ be given by $\theta(v_1 \otimes \dots \otimes v_n)=(v_{a(1)},\dots,v_{a(n)})$. We define $\pi(\breve w)_\xi : (\breve w\pi,V_\pi) \rw (\pi,V_\pi)$ as the pullback of $\theta$ under $\iota$. The resulting operator $\pi(\breve w)_\xi$ does not depend on the choice of $\iota$ and its target.
Furthermore, the map $w \mapsto \pi(\breve w)_\xi$ is multiplicative.

We now define a complex number $\<\pi,u^\natural\>_{\xi,z}$. We have the decomposition $\hat M = \hat M_1 \times\dots\times \hat M_n$ parallel to the decomposition of $M$. As in the previous section, we let $\tilde u \in N(\hat T,\hat G)$ be the Langlands-Shelstad lift of the image of $u$ in $W_\phi(M,G)$. Since $G^*$ is the general linear group, the actions of $\tilde u$ and of $\hat u$ on $\hat M$ coincide and are given by $\tx{Ad}(\tilde u)\hat M_i=\hat M_{a(i)}$. We conjugate $\phi$ within $\hat M$ if necessary to arrange $\phi = \phi_1 \times \dots \times \phi_n$ with $\phi_{a(i)}=\tx{Ad}(\tilde u)\circ\phi_i$.  Then $\tilde u \in N_\phi(M,G)$ and the equation $u=s\tilde u$ defines an element $s \in S_\phi(M) = S_{\phi_1}(M_1) \times \dots \times S_{\phi_n}(M_n)$. We define $\<\pi,u^\natural\>_{\xi,z}=\rho_{\xi,z}(s)$, where $\rho_{\xi,z}\in X^*(S^{\natural\natural}_\phi(M))$ is the character associated to $\pi$ as in Section \ref{sec:loclin}. The character $\rho_{\xi,z}$ satisfies $\rho_{\xi,z}\circ\tx{Ad}(\tilde u)=\rho_{\xi,z}$.
Moreover, an explicit computation reveals that the difference $\tilde{uv}\tilde v^{-1}\tilde u^{-1} \in A_{\hat M}$ is killed by any element of $X^*(S_\psi^{\natural\natural}(M))$. This implies the multiplicativity of $\<\pi,-\>_{\xi,z}$.

\subsubsection{Uniform treatment} \label{sec:iop3}
In the previous two sections we defined an operator $\pi(u^\natural)_{\xi,z} : (\breve w\pi,V_\pi) \rw (\pi,V_\pi)$ when $G$ was either a unitary group or a linear group, by placing additional assumptions. In principle, this allows us to construct the normalized self-intertwining operator $R_P(u^\natural,(\xi,z),\pi,\psi,\psi_F)$ in those cases. However, in order to be able to prove the necessary invariance properties of this operator, as well as to be able to prove its compatibility with the canonical global intertwining operator, we need a definition of the operator $\pi(u^\natural)_{\xi,z} : (\breve w\pi,V_\pi) \rw (\pi,V_\pi)$ that works uniformly for unitary and linear groups and does not require special assumptions on the extended pure inner twist $(\xi,z)$ beyond those which can be satisfied globally. We will now give such a definition and check that it specializes to the constructions of the previous two sections under the assumptions made there.

We consider the element $u'=u^{-1} \in N_\psi(M,G)$. We use the distinguished pinning of $\hat G$ and the process described in Section \ref{subsub:wmgu} to lift the image $w' \in W(\hat M,\hat G)^\Gamma$ of $u'$ first to an element of $W(\hat T,\hat G)^\Gamma$, and then to an element of $N(\hat T,\hat G)^\Gamma$ via the construction of Langlands and Shelstad. We will call this element $\tilde u'$. We have $u^{-1}=s\cdot \tilde u'$ for some $s \in \hat M$. In fact, since we have chosen $u$ to normalize $\hat T$ and preserve the $\hat B \cap \hat M$-positive roots in $R(\hat T,\hat M)$, and since the same is true for $\tilde u'$, and since $u'$ and $\tilde u'$ lie in the same $W(\hat T,\hat M)$-coset, we conclude that $s \in \hat T \subset \hat M$. The group $\hat M$ inherits a pinning from the group $\hat G$ and the automorphism $\hat\theta := \tx{Ad}(\tilde u')$ of $\hat M$ preserves that pinning. It is moreover dual to the automorphism $\theta^*:=\tx{Ad}(\tilde w)$ of $M^*$. We now create a twisted endoscopic datum for the group $M^*$ and its automorphism $\theta^*$, where we take $\hat M$ as the dual group of $M$ and $\hat\theta$ as the automorphism of $\hat M$ dual to $\theta^*$. Let $\hat M'$ be the group of fixed points of $\tx{Ad}(s_\psi s)\circ \hat \theta = \tx{Ad}(s_\psi u^{-1})$. This group is connected and has a connected center and a simply connected derived group.
It is furthermore normalized by the image of $\psi$. This allows us to form $\mc{M'}=\hat M' \cdot \psi(L_F)$. The arguments in \cite[\S2.2]{KS99} show that $\mc{M'}$ is a split extension of $\hat M'$ by $W_F$. This extension leads to a homomorphism $W_F \rw \tx{Out}(\hat M')$ which induces a homomorphism $\Gamma_F \rw \tx{Out}(\hat M')$. Let $M'$ be the unique quasi-split group defined over $F$ with dual group $\hat M'$ and with rational structure given by the homomorphism $\Gamma_F \rw \tx{Out}(\hat M') \cong \tx{Out}(M')$. Since $Z(\hat M')$ is connected, the arguments of loc. cit. show that there exists an isomorphism $\eta : \mc{M'} \rw {^LM'}$. Then $\mf{e}_{M,\psi}=(M',s_\psi s,\eta)$ is a twisted endoscopic triple for $(M^*,\theta^*)$.

The inner twist $\xi : M^* \rw M$ intertwines the operators $\theta^*$ and $\theta=\tx{Ad}(\breve w)$. From Section \ref{sec:normtfs} we have the normalized transfer factor $\Delta[\mf{e}_{M,\psi},\xi,z]$ for the twisted group $(M,\theta)$ and the endoscopic triple $\mf{e}_{M,\psi}$. We are now ready to define the operator $\pi(u^\natural) : (\breve w\pi,V_\pi) \rw (\pi,V_\pi)$. We have the parameter $\eta\circ\psi : L_F \rw {^LM'}$. Composing the stable linear form associated to this parameter by Proposition \ref{prop:local-stable-linear} with the transfer mapping $\cH(M) \rw \mc{SI}(M')$ provided by the transfer factor $\Delta[\mf{e}_{M,\psi},\xi,z]$, we obtain a linear form $f \mapsto f^{\mf{e}_{M,\psi}}(\eta\circ\psi)$.

\begin{lem} \label{lem:iop3} For each $\pi \in \Pi_\psi(M,\xi)$ there exists a unique isomorphism $\pi(u^\natural)_{\xi,z} : (\breve w\pi,V_\pi) \rw (\pi,V_\pi)$ such that for all $f \in \mc{H}(M)$ we have
\[ f^{\mf{e}_{M,\psi}}(\eta\circ\psi) = e(M^\theta)\sum_{\pi \in \Pi_\psi(M,\xi)} \tx{tr}(\pi(u^\natural)_{\xi,z}\circ\pi(f)). \]
Given $x \in S_\psi^{\natural\natural}(M)$, we have $\pi(xu^\natural)_{\xi,z} = \<\pi,x\>_{\xi,z}^{-1}\pi(u^\natural)_{\xi,z}$.
Moreover, in the settings of Sections \ref{sec:iop3u} and \ref{sec:iop3l}, the isomorphisms $\pi(u^\natural)_{\xi,z}$ constructed there have these properties.
\end{lem}
\begin{proof}
The isomorphism $\pi(u^\natural)_{\xi,z}$ is already unique up to a scalar, so the uniqueness statement follows form the linear independence of twisted characters. Its existence, together with the desired multiplicative property, will follow once we prove that the specific constructions of $\pi(u^\natural)_{\xi,z}$ made in Sections \ref{sec:iop3u} and \ref{sec:iop3l} satisfy the displayed equation, because these specific situations discussed there cover all possible cases up to isomorphism.

We consider first the situation treated in Section \ref{sec:iop3u}. Thus $E/F$ is a field extension and $G^*=U_{E/F}(N)$ is the quasi-split unitary group, $P^* \subset G^*$ is a standard parabolic subgroup, $M^*=M^*_+ \times M^*_-$ is a standard Levi subgroup of $P^*$ with $M^*_-=U_{E/F}(N_-)$ and $M^*_+=M^*_1 \times \dots \times M^*_n$ with $M^*_i=\tx{Res}_{E/F}(\tx{GL}(N_i))$, $(\xi,z) : G^* \rw G$ is an extended pure inner twist such that $P=\xi(P^*)$ and $M=\xi(M^*)$ are defined over $F$, thus $z \in Z^1_{G-\tx{bsc}}(\mc{E},M^*)$, and such that $z=z_+ \times z_-$ with $z_+=1$.

The automorphism $\theta^*$ preserves the decomposition  $M^* = M^*_+ \times M^*_-$ and acts trivially on $M^*_-$. The endoscopic triple $\mf{e}_{M,\psi}$ is therefore the product of a twisted endoscopic triple $\mf{e}_{M,+,\psi}=(M'_+,s_{\psi,+}s_+,\eta_+)$ for $(M^*_+,\theta^*)$ and an ordinary endoscopic triple $\mf{e}_{M,-,\psi}=(M'_-,s_{\psi,-}s_-,\eta_-)$ for $M^*_-$. Here we have written $s=s_+ \times s_-$ and $s_\psi=s_{\psi,+} \times s_{\psi,-}$ according to the decomposition $\hat M = \hat M_+ \times \hat M_-$. The linear form $f^{\mf{e}_{M,\psi}}(\eta\circ\psi)$ breaks up as the product $f^{\mf{e}_{M,\psi}}_+(\eta_+\circ\psi_+) \cdot f^{\mf{e}_{M,\psi}}_-(\eta_-\circ\psi_-)$. Noting that $M^\theta=M_- \times M_+^\theta$, the right hand side of the claimed equation also breaks up into a product
\begin{equation} \label{eq:iop3eq1} e(M_+^\theta)\!\!\!\sum_{\pi_+\in\Pi_{\psi_+}}\tx{tr}(\pi_+(u^\natural)_{\xi,z}\circ\pi_+(f_+)) \cdot e(M_-)\!\!\!\sum_{\pi_-\in\Pi_{\psi_-}}\tx{tr}(\pi_-(u^\natural)_{\xi,z}\circ\pi_-(f_-)). \end{equation}
We consider first the left factor in this product. Note that $\Pi_{\psi_+}$ consists of a single element $\pi_+$. The assumption $z_+=1$ implies that $\xi$ induces an isomorphism $M^*_+ \rw M_+$, thus $M_+$ is a product $M_1 \times \dots \times M_n$ of groups of the form $M_i=\tx{Res}_{E/F}(\tx{GL}(N_i))$, each endowed with its standard pinning. The automorphism $\theta$ preserves this product decomposition as well as the pinning of the group $M_+$. In particular, $M_+^\theta$ is quasi-split, so $e(M_+^\theta)=1$. We may further assume that the permutation that $\theta$ induces on the factors of $M$ is transitive, otherwise we study each orbit of this permutation separately and take the product of the results. Assuming transitivity, we find ourselves in the situation discussed in Section \ref{sec:prelsimp}. The results of that section reduce the problem to the setting where there is only one factor in the product, i.e. $M_+=\tx{Res}_{E/F}(\tx{GL}(N_+))$. Then there are two possibilities: Either $\theta$ is the trivial automorphism or the unique non-trivial pinned automorphism. In both cases, we obtain the character identity
\[ f^{\mf{e}_{M,\psi}}_+(\eta_+\circ\psi_+)=\tx{tr}(I_+\circ\pi_+(f)), \]
where $I_+ : \pi\circ\theta^{-1} \rw \pi$ is the natural isomorphism discussed in Section \ref{sec:iop3b}.
In the case $\theta=1$ this parabolic descent, while in the other case it follows from Proposition \ref{prop:local-stable-linear}.

Consider now the right factor in the product \eqref{eq:iop3eq1}. By Theorem \ref{thm:locclass-single} we have the equation
\[ f^{\mf{e}_{M,\psi}}_-(\eta_-\circ\psi_-) = e(M_-)\sum_{\pi_- \in \Pi_{\psi_-}} \<s_-,\pi_-\>_{\xi,z} \tr(\pi_-(f_-)). \]
It follows that the equation in the statement of the lemma holds with $\pi(u^\natural)_{\xi,z}$ defined as $\<\pi,s_-\>_{\xi,z}\cdot \tx{id}_{\pi_-} \otimes I_+$. But this is the construction of $\pi(u^\natural)_{\xi,z}$ given in Section \ref{sec:iop3u}, once we note that the element $s_-$ is equal to the image of $u^{-1}$ under the section $s:\pi_0(N_\psi(M,G)) \rw \pi_0(S_\psi(M))$ constructed there.

We now turn to the situation treated in Section \ref{sec:iop3l}. Thus $G^*=\tx{GL}(N)/F$, $P^* \subset G^*$ is a standard parabolic subgroup, $M^*=M^*_1 \times \dots \times M^*_n$ is a standard Levi subgroup of $P^*$ with $M^*_i=\tx{GL}(N_i)$, $(\xi,z) : G^* \rw G$ is an extended pure inner twist such that $P=\xi(P^*)$ and $M=\xi(M^*)$ are defined over $F$, thus $z \in Z^1_{G-\tx{bsc}}(\mc{E},M^*)$, and such that in fact $z \in Z^1_{G-\tx{bsc}}(\mc{E},M^{*,\theta^*})$. The Levi subgroup $M \subset G$ then breaks up accordingly as $M = M_1 \times \dots \times M_n$, with $\theta$ permuting the factors. We have $e(M^\theta)=e(M_{i_1})\cdot \dots \cdot e(M_{i_k})$, where $i_1,\dots,i_k$ are representatives for the orbits of the permutation of $\{1,\dots,n\}$ induces by $\theta$. Applying the simple descent for twisted endoscopy discussed in Section \ref{sec:prelsimp} reduces the problem to the setting where there is only one factor, i.e. $M=\tx{GL}(N/d)/D$, with $D/F$ division algebra of degree $d$, and $\theta$ acts trivially on $M$. Then $M'=\tx{GL}(N)$ and the discussion in Section \ref{sec:loclin} provides the character identity
\[ f^{\mf{e}_{M,\psi}}(\eta\circ\psi) = e(M)\<\pi,s\>_{\xi,z}\tx{tr}(\pi(f)). \]
This completes the proof.
\end{proof}
\subsection{Local intertwining operator -- the compound operator} \label{sec:iop}
In this section we are going to define a normalized self-intertwining operator on a parabolically induced representation by putting together the three operators defined in the previous sections.

We take $F$ to be a local field, $E/F$ a quadratic algebra, and $G^*=U_{E/F}(N)$. We take $\Xi : G^* \rw G$ to be an equivalence class of extended pure inner twists. Let $(M,P)$ be a proper parabolic pair for $G$. There is a unique standard parabolic pair $(M^*,P^*)$ of $G^*$ corresponding to $(M,P)$, and this determines an equivalence class $\Xi_M : M^* \rw M$ of extended pure inner twists: $\Xi_M$ consists of those pairs $(\xi,z) \in \Xi$ such that $\xi(M^*,P^*)=(M,P)$.

Let $\psi \in \Psi(M^*)$ and assume that $\Pi_\psi(M,\xi) \neq \emptyset$. Let $\pi \in \Pi_\psi(M,\xi)$ and write $V_\pi$ for the underlying vector space. Let $(\hat M,\hat P)$ be the standard parabolic pair of $\hat G$ dual to $(M^*,P^*)$. Up to equivalence $\psi$ can be chosen so that the minimal Levi subgroup of $^LM$ through which it factors is of the form $\hat M_0 \rtimes W_F$ for a standard Levi subgroup of $\hat M$, hence of $\hat G$. This implies that $\hat T_\psi^\tx{rad} := \hat T \cap S_\psi^\tx{rad}$ is a maximal torus of $S_\psi^\tx{rad}$ and $\hat M_\psi^\tx{rad} := \hat M \cap S_\psi^\tx{rad}=Z(A_{\hat M},S_\psi^\tx{rad})$ is a Levi subgroup of $S_\psi^\tx{rad}$ containing $\hat T_\psi^\tx{rad}$. We refer the reader to Section \ref{sec:diag} for the relevant notation.

Let $u^\natural \in N_\psi^\natural(M,G)$. For an arbitrary lift $u \in N(A_{\hat M},S_\psi)$ of $u^\natural$, both $\hat T_\psi^\tx{rad}$ and $\tx{Ad}(u)\hat T_\psi^\tx{rad}$ are maximal tori of $\hat M_\psi^\tx{rad}$, hence conjugate under $\hat M_\psi^\tx{rad}$. Thus we may choose the lift $u$ so that it normalizes $\hat T_\psi^\tx{rad}$. Then it also normalizes the centralizer of $\hat T_\psi^\tx{rad}$ in $\hat G$, which equals to $\hat T$. In particular, $u$ is semi-simple. We can say a bit more by exploiting the specificity of the group $G^*$. Indeed, one checks that $u$ can be chosen so that its image in $W(\hat T,\hat G)$ preserves the $\hat B \cap \hat M$-positive roots in $R(\hat T,\hat M)$.

We have $\Gamma$-equivariant isomorphisms $W(\hat M,\hat G) \cong W(M^*,G^*) \cong W(M,G)$, the first one induced by the duality of $\hat G$ and $G^*$ and the second one by the equivalence class $\Xi_M$. Let $w$ denote the image of $u$ in any of these isomorphic groups.

Choose an arbitrary lift $\dot w \in N(M,G)(F)$ of $w$ and set $\dot w\pi(m) = \pi(\dot w^{-1}m\dot w)$, acting on the vector space $V_\pi$. The isomorphism class of $\dot w\pi$ is independent of the choice of $\dot w$. We claim that $\dot w\pi \cong \pi$. Indeed, $\pi$ decomposes as $\pi_+ \otimes \pi_-$ according to the decomposition $M=M_+ \times M_-$. The restriction of $\tx{Ad}(\dot w)$ to $M_-$ is an inner automorphism and we have $\dot w\pi_- \cong \pi_-$. On the other hand, $w$ preserves the $\hat M$-equivalence class of $\psi$, hence also the $L$-packet of representations of $M_+(F)$ containing $\pi_+$. But this $L$-packet is a singleton, so $\dot w\pi_+ \cong \pi_+$. This shows that $\dot w\pi \cong \pi$ and hence the element $w$ corresponds to a self-intertwining operator on the representation $\mc{H}_P(\pi)$, well-defined up to scalar. We will give a specific normalization of this operator, called $R_P(u^\natural,\Xi,\pi,\psi,\psi_F)$.

Fix $(\xi,z) \in \Xi_M$ as in Lemma \ref{lem:c1}. We will use this choice to give a normalization $R_P(u^\natural,(\xi,z),\pi,\psi,\psi_F)$ of the local self-intertwining operator for the Weyl element $w$ acting on the representation $\mc{H}_P(\pi)$. We will then show that this normalization does not depend on the choice of $(\xi,z)$ and hence can be referred to as $R_P(u^\natural,\Xi,\pi,\psi,\psi_F)$. It will however in general depend on the element $u^\natural \in N^\natural_\psi(M,G)$, not just its image $w$ in $W_\psi(M,G)$. As we saw in Sections \ref{sec:iop3u} and \ref{sec:iop3l}, if we make a special choice of $(\xi,z)$ then we can break this dependency and define an operator $R_P(w,(\xi,z),\pi,\psi,\psi_F)$ with the property that the map $w \mapsto R_P(w,(\xi,z),\pi,\psi,\psi_F)$ is a homomorphism $W_\psi(M,G) \rw \tx{End}(\mc{H}_P(\pi))$. The operator $R_P(u^\natural,(\xi,z),\pi,\psi,\psi_F)$ is then a product $\<\pi,u^\natural\>_{\xi,z}\cdot R_P(w,(\xi,z),\pi,\psi,\psi_F)$, with $\<\pi,u^\natural\>_{\xi,z} \in \C$. While the operator $R_P(w,(\xi,z),\pi,\psi,\psi_F)$ has the appeal of only depending on $w$ and not on $u^\natural$, the operator $R_P(u^\natural,\Xi,\pi,\psi,\psi_F)$ is the more invariant object and the one we ultimately use.

We now come to the definition of the normalized local self-intertwining operator
\[ R_P(u^\natural,(\xi,z),\pi,\psi,\psi_F) : \mc{H}_P(\pi) \rw \mc{H}_P(\pi), \]
as the composition
\begin{equation} \label{eq:iop} R_P(u^\natural,(\xi,z),\pi,\psi,\psi_F) = \mc{I}_P(\pi(u^\natural)_{\xi,z}) \circ l_P(w,\xi,\psi,\psi_F) \circ R_{w^{-1}P|P}(\xi,\psi). \end{equation}

\begin{lem} \label{lem:iopind}
The operator $R_P(u^\natural,(\xi,z),\pi,\psi,\psi_F)$ does not depend on the choice of $(\xi,z) \in \Xi$.
\end{lem}
\begin{proof}
We had chosen $(\xi,z) : G^* \rw G$ to satisfy the conditions $\xi(M^*)=M$, $\xi(P^*)=P$, and $z \in Z^1_{G-bsc}(\mc{E},M^{*,\tilde w})$. An equivalent $(\xi',z')$ is of the form $\xi'=\xi\circ\tx{Ad}(b)$ and $z'(\sigma)=b^{-1}z(\sigma)\sigma(b)$ for some $b \in G^*$. If $(\xi',z')$ satisfies the same conditions as $(\xi,z)$, then we see first that $b \in M^*$, and second that $\theta^*(b^{-1}z(\sigma)\sigma(b))=b^{-1}z(\sigma)\sigma(b)$, where $\theta^*=\tx{Ad}(\tilde w)$. If we set $d:=\xi(b\theta^*(b^{-1}))$, then the latter equation implies  $d \in M(F)$. As before, we set $\breve w := \xi(\tilde w)$ and we have $\breve w \in G(F)$. We now also set $\invbreve{w}=\xi'(\tilde w) \in G(F)$. One checks immediately that $\invbreve{w}=d\breve w$.

We claim that $\pi(u^\natural)_{\xi',z'} = \pi(d)\circ\pi(u^\natural)_{\xi,z}$. Indeed, set $\breve \theta=\tx{Ad}(\breve w)$ (this automorphism was denoted by $\theta$ earlier), and $\invbreve{\theta}=\tx{Ad}(\invbreve{w})$. A direct computation shows that the map
\[ M(F) \rw M(F),\qquad \breve \delta \mapsto \invbreve{\delta} := \breve\delta \cdot d^{-1} \]
is a bijection which translates $\breve\theta$-twisted conjugacy on its source to $\invbreve{\theta}$-twisted conjugacy on its target, and moreover
\[ \Delta[\mf{e}_{M,\psi},\xi,z](\gamma,\breve\delta) = \Delta[\mf{e}_{M,\psi},\xi',z'](\gamma,\invbreve{\delta}), \]
where the transfer factors are the one constructed in Section \ref{sec:iop3}. Given a function $\breve f \in \mc{H}(M)$, let $\invbreve{f} \in \mc{H}(M)$ be defined by $\invbreve{f}(g)=\breve f(gd)$. The equality of transfer factors and the equivariance of $\breve\delta \mapsto \invbreve{\delta}$ imply that for any $\gamma \in M'(F)$ strongly $G$-regular the unstable twisted orbital integral
\[ \sum_{\breve\delta}\Delta[\mf{e}_{M,\psi},\xi,z](\gamma,\breve\delta)\int_{M_{\breve\delta\breve w}(F)\lmod M(F)}\breve f(g^{-1}\breve\delta\breve\theta(g))dg \]
is equal to
\[ \sum_{\invbreve{\delta}}\Delta[\mf{e}_{M,\psi},\xi',z'](\gamma,\invbreve{\delta})\int_{M_{\invbreve{\delta}\invbreve{w}}(F)\lmod M(F)}\invbreve{f}(g^{-1}\invbreve{\delta}\invbreve{\theta}(g))dg, \]
where the first sum runs over the $\breve\theta$-twisted conjugacy classes of strongly $\breve\theta$-regular $\breve\theta$-semi-simple elements of $M(F)$, and the second sum is the analogous sum with $\breve\ $ replaced by $\invbreve{\ }$. We conclude that, given a function $f' \in \mc{H}(M')$, the functions $f'$ and $\breve f$ are $\Delta[\mf{e}_{M,\psi},\xi,z]$-matching if and only if the functions $f'$ and $\invbreve{f}$ are $\Delta[\mf{e}_{M,\psi},\xi',z']$-matching. A glance at the statement of Lemma \ref{lem:iop3} completes the proof of the claim that $\pi(u^\natural)_{\xi',z'} = \pi(d)\circ\pi(u^\natural)_{\xi,z}$.

Turning to $l_P(w,\xi',\psi,\psi_F)$, we see directly from the definition that
\[ l_P(w,\xi',\psi,\psi_F) = \delta_P^\frac{1}{2}(d^{-1})\cdot \mc{I}_P(\breve w\pi(d^{-1}))\circ l_P(w,\xi,\psi,\psi_F). \]
Recalling that $d=\xi(b\theta^*(b^{-1}))$ and using the intertwining property of $\pi(u^\natural)_{\xi,z}$, we conclude that
\[ \mc{I}_P(\pi(u^\natural)_{\xi',z'}) \circ l(w,\xi',\psi,\psi_F) = \frac{\delta_P(\xi(b))^\frac{1}{2}}{\delta_{w^{-1}Pw}(\xi(b))^{\frac{1}{2}}}\mc{I}_P(\pi(u^\natural)_{\xi,z}) \circ l_P(w,\xi,\psi,\psi_F). \]
Here we have used the extension to $M(\ol{F})$ of the modulus characters for $P$ and $w^{-1}Pw$ that was already used in Section \ref{sec:iop1}. A glance at Equation \eqref{eq:iop1xc} completes the proof.

\end{proof}

Before stating the next lemma, recall that Kottwitz's map \eqref{eq:kotisoloc} provides a character of $Z(\hat G)^\Gamma$ from any element of $B(F,G^*)_\tx{bsc}$, as well as a character of $[\hat G/\hat G_\tx{der}]^\Gamma$ from any element of $B(F,Z(G^*))$. In particular, the class $\Xi$, being an element of $B(F,G^*)_\tx{bsc}$, provides a character of $Z(\hat G)^\Gamma$, which we denote by $\<\Xi,-\>$. Moreover, the element $u^\natural \in N_\psi^\natural(M,G)$ can be mapped to the quotient $\hat G/\hat G_\tx{der}$, where it is $\Gamma$-fixed, and for any $y \in B(F,Z(G^*))$ we can evaluate the corresponding character $\<y,-\>$ of $[\hat G/\hat G_\tx{der}]^\Gamma$ at the point $u^\natural$.

\begin{lem}\ \\[-20pt] \label{lem:rpequiv}
\begin{enumerate}
\item Let $x \in Z(\hat G)^\Gamma$. Then $R_P(xu^\natural,\Xi,\pi,\psi,\psi_F)=\<\Xi,x\>^{-1}R_P(u^\natural,\Xi,\pi,\psi,\psi_F)$.
\item Let $c \in B(F,Z(G^*))_\tx{bsc}$. Then $R_P(u^\natural,c\Xi,\pi,\psi,\psi_F)=\<c,u^\natural\>^{-1}R_P(u^\natural,\Xi,\pi,\psi,\psi_F)$.
\item Let $y \in S_\psi^{\natural\natural}(M)$. Then $R_P(yu^\natural,\Xi,\pi,\psi,\psi_F)=\<\pi,y\>_{\xi,z}^{-1}R_P(u^\natural,\Xi,\pi,\psi,\psi_F)$.
\end{enumerate}
\end{lem}
\begin{proof}
The first two points follow essentially from Lemma \ref{lem:tfequi}. Indeed, notice that the right and middle terms in \eqref{eq:iop} depend only on the image of $u^\natural$ in $W_\psi(M,G)$ and only on the inner twist underlying $\Xi$. Thus they do not change when we pass from $u^\natural$ to $xu^\natural$, or when we pass from $\Xi$ to $c\Xi$. Only the term $\pi(u^\natural)_{\xi,z}$ changes. The reason it does is because in the defining equation in Lemma \ref{lem:iop3}, the linear form $f^{\mf{e}_{M,\psi}}(\eta\circ\psi)$ depends on the transfer factor $\Delta[\mf{e}_{M,\psi},\xi,z]$. Changing $u^\natural$ to $xu^\natural$ changes $\mf{e}_{M,\psi}$ to $x^{-1}\mf{e}_{M,\psi}$. According to Lemma \ref{lem:tfequi}, we have
\[ \Delta[x\mf{e}_{M,\psi},\xi,z] = \<z,x\>^{-1} \Delta[\mf{e}_{M,\psi},\xi,z], \]
but $\<z,x\>$ is just a different notation for $\<\Xi,x\>$. On the other hand, again by Lemma \ref{lem:tfequi}, we have
\[ \Delta[\mf{e}_{M,\psi},\xi,cz] = \<c,\bar s_\psi \bar s\> \Delta[\mf{e}_{M,\psi},\xi,z]. \]
Recall that $s$ was constructed in Section \ref{sec:iop3} to be $u'\cdot\tilde u'^{-1}$, with $u' \in \hat G$ an arbitrary lift of $u^{-1,\natural}$ and $\tilde u' \in \hat G$ a specific lift of the image of $u'$ in $W(\hat M,\hat G)$. In the statement of Lemma \ref{lem:tfequi}, $\bar s$ is the image of $s$ in $[\hat M/\hat M_\tx{der}]_{\hat\theta,\tx{free}}$ (recall notation from Lemma \ref{lem:tfequi}). However, since $c$ takes values in $Z(G^*) \subset Z(M^*)$, we may further map $s$ down to $\hat G/\hat G_\tx{der}$ before pairing it with $c$. By construction, $\tilde u' \in \hat G_\tx{der}$, so the images of $u'$ and $s$ in $\hat G/\hat G_\tx{der}$ are equal. Moreover, the image $\bar s_\psi$ of $s_\psi$ in $\hat G/\hat G_\tx{der}$ is trivial.

For the third point, we use again the fact that the images of $yu^\natural$ and $u^\natural$ in $W_\psi(M,G)$ are equal, so the only factor in \eqref{eq:iop} that is affected is $\pi(u^\natural)_{\xi,z}$, and according to Lemma \ref{lem:iop3} we have $\pi(yu^\natural)_{\xi,z}=\<\pi,y\>_{\xi,z}^{-1}\pi(u^\natural)_{\xi,z}$.
\end{proof}

\begin{lem} \label{lem:iopm} The operator $R_P(u^\natural,\Xi,\pi,\psi,\psi_F)$ is multiplicative in $u^\natural$.
\end{lem}
\begin{proof}
Let $(\xi,z) \in \Xi$ be chosen as in Section \ref{sec:iop3u} or \ref{sec:iop3l}, so that we know by construction that $\pi(u^\natural)_{\xi,z}$ is multiplicative. The claim then follows from Lemmas \ref{lem:iop1m}, \ref{lem:iop123c}, \ref{lem:iop2m}, and Fact \ref{fct:iop23c}.
\end{proof}

\subsection{The local intertwining relation} \label{sec:lir}

Let $F$ be a local field, $E/F$ a quadratic algebra, and $G^*=U_{E/F}(N)$. Let $\Xi : G^* \rw G$ be an equivalence class of extended pure inner twists and $(M^*,P^*)$ be a proper standard parabolic pair of $G^*$.

Given $\psi \in \Psi(M^*)$ and $u \in N_\psi(M^*,G^*)$ we introduce a linear form $f \mapsto f_{G,\Xi}(\psi, u^\natural)$ where $ u^\natural$ is the image of $u$ in $N^\natural_\psi(M^*,G^*)$. If $\psi$ is not relevant for $\Xi : G^* \rw G$ then $f_{G,\Xi}(\psi, u^\natural)$ is defined to be zero (since we think of the underlying induced representation and the intertwining operator to be zero). If $\psi$ is relevant for $\Xi$, then the Levi subgroup $M^*$ of $G^*$ transfers to $G$ and we let $(M,P)$ be a parabolic pair for $G$ corresponding to $(M^*,P^*)$ and write $\Xi_M : M^* \rw M$ for the associated equivalence class of extended pure inner twists.
Let $\Pi_\psi(M,\Xi_M)$ be the packet of representations of $M(F)$ corresponding to $\psi$ via Theorem \ref{thm:locclass-single}. If $\Pi_\psi(M,\Xi_M)$ happens to be empty (which may only happen when $\psi$ is non-generic), we define $f_{G,\Xi}(\psi, u^\natural)$ to be zero again. Now assume $\Pi_\psi(M,\Xi_M)\neq \emptyset$. Using the self-intertwining operators $R_P(u^\natural,\Xi,\pi,\psi,\psi_F)$ constructed in Section \ref{sec:iop} for all members $\pi \in \Pi_\psi(M,\Xi_M)$, we define a linear form $f \mapsto f_{G,\Xi}(\psi, u^\natural)$ on $\mc{H}(G)$ as follows
\[ f_{G,\Xi}(\psi,u^\natural) = \sum_{\pi \in \Pi_\psi(M,\Xi_M)} \tx{tr}(R_P(u^\natural,\Xi,\pi,\psi,\psi_F)\mc{I}_P(\pi)(f)). \]

We construct a second linear form $f \mapsto f_{G,\Xi}'(\psi,s_\psi u^{-1})$ on $\mc{H}(G)$ in the following way. Fix $f \in \mc{H}(G)$. For any semi-simple element $s \in S_\psi$, let $\psi^{\fke}\in \Psi(G^\fke)$ and $\mf{e}\in \cE(G^*)$ be associated to $\psi$ and $s$ as in \ref{sub:endo-correspondence}. (There is no need to take $\psi^{\fke}$ as an orbit under the strict outer automorphism group since the latter group is trivial in our case.) From $\psi^{\fke}$ we obtain the stable linear form on $\mc{H}(G^\mf{e})$ in part 4 of Theorem \ref{thm:locclass-single}.
 The value of this stable linear form on all functions $f^\mf{e} \in \mc{H}(G^\mf{e})$ whose orbital integrals match those of $f$ with respect to the transfer factor $\Delta[\mf{e},\Xi]$ is the same. We call this value $f_{G,\Xi}'(\psi,s)$.

\begin{lem}\label{lem:local-indep-f'}
  For any $s,s_0\in S_{\psi,\sspl}(G^*)$ such that their images in $S^\natural_\psi(G^*)$ belong to $S^\natural_\psi(M^*,G^*)$ and are equal we have
  $f'_{G,\Xi}(\psi,s)=f'_{G,\Xi}(\psi,s_0)$.
\end{lem}

\begin{proof}

 Observe that $f'_{G,\Xi}(\psi,s)$ does not change when $s$ is conjugated by an element of $S_{\psi}(G^*)^0$.
 By conjugating $s$ if necessary, we may assume that $\Int(s)$ stabilizes $T_\psi$ (and a Borel subgroup $B_\psi\supset T_\psi$). Define $T_{\psi,s}:=Z_{T_\psi}(s)^0$. Arguing as in the paragraphs below \cite[(4.5.8)]{Arthur}, it suffices to show that $f'_{G,\Xi}(\psi,s)$ is unchanged if $s$ is multiplied by any element of $T_{\psi,s}\cap \srad_\psi$.

 One checks that $\hat{M}_s:=Z_{\hat G^*}(T_{\psi,s})$ is a $\Gamma$-stable Levi subgroup of $\hat G^*$ containing that $\hat M^*$. Let $M^*_s$ denote a Levi subgroup of $G^*$ which is dual to $\hat M_s$.
  Observe that $(\psi,s)$ is the image of a pair $$(\psi_{M_s},s_{M_s})$$ where $\psi_{M_s}\in \Psi(M_s)$ and $s_{M_s}\in S_{\psi_{M_s},\sspl}$. There exists an elliptic endoscopic datum $M^\fke_s$ and $\psi^\fke_{M_s}\in \Psi(M^\fke_s)$ whose image is $\psi_{M_s}$.
 We have
  \begin{equation}\label{eq:LIR-induction-f'}
    f'_{G,\Xi}(\psi,s)=f^\fke(\psi^\fke)=f^\fke_{M^\fke_s}(\psi^\fke_{M_s})%
  \end{equation}
  where $f^\fke_{M^\fke_s}$ is the constant term of $f^\fke$ along (a parabolic subgroup whose Levi subgroup is) $M^\fke_s$. If $M_s^*$ does not transfer to $G$ then $f^\fke_{M^\fke_s}(\psi^\fke_{M_s})=0$ so in particular the lemma is verified. Indeed no stable strongly regular semisimple conjugacy classes of $M^\fke_s$ have matching conjugacy classes on $G$ in that case, so one deduces that $f^\fke_{M^\fke_s}$ defines a trivial stable linear form from the fact that $f^\fke_{M^\fke_s}$ is a transfer of $f$. (For the latter fact, see the discussion in the preamble and proof of \cite[Prop 2.1.1]{Arthur} and the references therein. This can also be thought of as the degenerate case of Lemma \ref{lem:tflevi} where $M^*$ does not transfer to $G$, in which case $f_M$ should be thought of as zero, hence the stable linear form given by $f^\fke_{M^\fke_s}$ should be identically zero.)

  From now we assume that $M_s^*$ transfers to a Levi subgroup $M_s$ of $G$. Then there is an extended pure inner twist $(M_s,\xi_{M_s},z_{M_s})$ given by Lemma \ref{lem:c1}. Lemma \ref{lem:tflevi} tells us that $f^\fke_{M^\fke_s}$ is a transfer of $f_{M_s}$ with respect to $(M_s,\xi_{M_s},z_{M_s})$.
   Now if $s$ is translated by $t\in T_{\psi,s}$, then $s_{M_s}$ is translated by $t$. The effect of the translation, in light of Lemma \ref{lem:tfequi} with $M_s$ in place of $G$, is that $$f'_{M_s,\xi_{M_s},z_{M_s}}(\psi_{M_s},s_{M_s})=f^\fke_{M^\fke_s}(\psi^\fke_{M_s})$$ is multiplied by $\<z_{M_s},t\>$. According to Lemmas \ref{lem:gbasl} and \ref{lem:rhofact} (we apply the latter lemma with $\zeta=\lg z,\cdot\rg$ and $\zeta_{M^*}=\lg z_M,\cdot\rg$, which are the images of $z$ and $z_M$ under the Kottwitz maps; the former lemma ensures that $\zeta$ is simply the restriction of $\zeta_{M^*}$ to $Z(\hat G^*)^\Gamma$) we have that $\<z_{M_s},t\>=1$ if $t\in T_{\psi,s}\cap \srad_\psi$. Hence the translation of $s$ by $t\in T_{\psi,s}\cap \srad_\psi$ does not change $f'_{G,\Xi}(\psi,s)$ as we wanted to verify.

\end{proof}

We can apply the construction of $f'_{G,\Xi}(\psi,s)$ in particular to the element $s=s_\psi u^{-1}$ and obtain $f_{G,\Xi}'(\psi,s_\psi u^{-1})$. Note that $s_\psi u^{-1}$ is semi-simple, because both $s_\psi$ and $u^{-1}$ are semi-simple and commute.

\begin{thm*}[Local intertwining relation] \label{thm:lir}
Let $\psi \in \Psi(M^*)$, $u \in N_\psi(M^*,G^*)$, and $f \in \mc{H}(G)$. Then
\begin{enumerate}
\item Suppose that $\Pi_\psi(M,\Xi_M)$ is nonempty (so $\psi$ is relevant in particular). If $u^\natural$ belongs to $W_\psi(M^*,G^*)^\tx{rad}$, then $R_P(u^\natural,\Xi,\pi,\psi,\psi_F)=1$ for all representations $\pi \in \Pi_\psi(M,\Xi_M)$.
\item The complex number $f_{G,\Xi}(\psi, u^\natural)$ depends only on the image of $u$ in  $S_\psi^\natural$.
\item We have an equality
\begin{equation} \label{eq:lir} f_{G,\Xi}'(\psi,s_\psi u^{-1})=e(G)f_{G,\Xi}(\psi, u^\natural). \end{equation}
\end{enumerate}
\end{thm*}

Note that the multiplicativity of $R_P(u^\natural,\Xi,\pi,\psi,\psi_F)$ in $u^\natural$ formulated in Lemma \ref{lem:iopm} implies that the second point follows from the first. However, we will only be able to deduce the first point once we have proved the second by different means. Note also that the last point breaks up into two cases, namely
\begin{itemize}
  \item If $\psi$ is relevant for $\Xi$ then $f_{G,\Xi}'(\psi,s_\psi u^{-1})=e(G)f_{G,\Xi}(\psi, u^\natural)$.
  \item If $\psi$ is not relevant for $\Xi$ then $ f_G'(\psi,s_\psi u^{-1})=0$.
\end{itemize}

A further immediate consequence of the theorem is that the linear form $f_{G,\Xi}(\psi,u^\natural)$ depends only on the $\hat G^*$-equivalence class of the parameter $\psi$, and not just on the $\hat M^*$-equivalence class. We will now formulate a lemma that contains a weaker statement. This statement will be enough to conclude said independence in the case where $\psi \in \Psi_2(M^*)$ and will also be used in the global arguments that lead to the proof of Theorem \ref{thm:lir}. In the next section we will show that the form $f_{G,\Xi}(\psi,u^\natural)$ for an arbitrary $\psi \in \Psi(M^*)$ is equal to $f_{G,\Xi}(\psi_0,u^\natural)$ for a suitable $\psi_0 \in \Psi_2(M^*_0)$, so in principle we would then have said independence for all $\psi \in \Psi(M^*)$, but we will not need this stronger result in the proof of Theorem \ref{thm:lir}.

\begin{lem}\label{lem:linear-form-ind-of-M} Let $\psi \in \Psi(M^*)$ and $v \in N(A_{\hat M^*},\hat G^*)$. Given $u \in N_\psi(M^*,G^*)$ we have $f_{G,\Xi}(\psi,u^\natural)=f_{G,\Xi}(v\psi,(vuv^{-1})^\natural)$.
\end{lem}

\begin{proof}
We assume that $\psi$ is relevant, otherwise both sides are zero. Let $(\xi,z) \in \Xi_M$ be as in Lemma \ref{lem:c1}. The left-hand side is a sum over $\pi \in \Pi_\psi(M,\Xi_M)$, while the right-hand side is a sum over $\pi' \in \Pi_{v\psi}(M,\Xi_M)$. Both index sets are in bijection via $\pi' = \breve v\pi$. To see this, decompose $M=M_+ \times M_-$ as usual and observe that $\Pi_\psi(M,\Xi_M) = \{\pi_+\} \otimes \Pi_\psi(M_-,\Xi_{M_-})$. Now $v$ acts trivially on $M_-$ and sends $\psi_+$ to $v\psi_+$, hence $\pi_+$ to $v\pi_+$. The claim follows.

Thus the right hand side is the sum over $\pi \in \Pi_\psi(M,\Xi_M)$ of the traces of the operators
\[ R_P(vuv^{-1},(\xi,z),v\pi,v\psi,\psi_F) \circ \mc{I}_P(v\pi,f), \]
each acting on the vector space $\mc{H}_P(v\pi)$. We conjugate each such operator by $l_P(v,\xi,\psi,\psi_F)\circ R_{v^{-1}P|P}(\xi,\psi)$ and obtain an operator $R$ on $\mc{H}_P(\pi)$ with the same trace. To compute that operator, we first expand
$R_P(vuv^{-1},(\xi,z),v\pi,v\psi,\psi_F)$ according to its definition \eqref{eq:iop} and obtain
\[ \mc{I}_P((v\pi)(vuv^{-1})_{\xi,z}) \circ l_P(vuv^{-1},\xi,v\psi,\psi_F) \circ R_{(vuv^{-1})^{-1}P|P}(\xi,v\psi). \]
Using Lemma \ref{lem:iop2m} one checks that
\[ l_P(vuv^{-1},\xi,v\psi,\psi_F) = l_P(v,\xi,\psi,\psi_F)\circ l_P(u,\xi,\psi,\psi_F) \circ l_P(v,\xi,\psi,\psi_F)^{-1}. \]
The intertwining property of $l_P(v,\xi,\psi,\psi_F)\circ R_{v^{-1}P|P}(\xi,\psi)$ together with Lemma \ref{lem:iop123c} and Fact \ref{fct:iop23c} imply that the operator $R$ we are computing is equal to
\[ \mc{I}_P((v\pi)(vuv^{-1})_{\xi,z}) R_{v^{-1}P|P}(\xi,\psi)^{-1}l_P(u,\xi,\psi,\psi_F)R_{u^{-1}v^{-1}P|v^{-1}P}(\xi,\psi)R_{v^{-1}P|P}(\xi,\psi)\mc{I}_P(\pi,f), \]
which by Lemmas \ref{lem:iop123c} and \ref{lem:iop1m} works out to
\[ \mc{I}_P((v\pi)(vuv^{-1})_{\xi,z}) l_P(u,\xi,\psi,\psi_F)R_{u^{-1}P|P}(\xi,\psi)\mc{I}_P(\pi,f). \]
Our final task is to show that the operators $(v\pi)(vuv^{-1})_{\xi,z}$ and $\pi(u)_{\xi,z}$ on the vector space $V_\pi$ are equal. This can be seen by considering the defining equation in Lemma \ref{lem:iop3}. We view $\tx{Ad}(\breve v)$ as an isomorphism from the twisted group $(M,\tx{Ad}(\breve u))$ to the twisted group $(M,\tx{Ad}(\breve v\breve u\breve v^{-1}))$. It allows us to compare the right hand sides of Lemma \ref{lem:iop3} for the two cases $\pi(u)$ and $(v\pi)(vuv^{-1})$ by identifying the two index sets in the sum, as well as $\pi(f)$ with $v\pi(f)$. This isomorphism is further translated under the inner twist $\xi$ to the isomorphism $\tx{Ad}(\tilde v) : (M^*,\tx{Ad}(\tilde u)) \rw (M^*,\tx{Ad}(\tilde v\tilde u\tilde v^{-1}))$, whose dual is the isomorphism $\tx{Ad}(\tilde v) : (\hat M,\tx{Ad}(\tilde u')) \rw (\hat M,\tx{Ad}(\tilde v\tilde u'\tilde v^{-1}))$, where the $\tilde\ $ are now taken on the dual side. This isomorphism sends the element $u'$ to the element $vu'v^{-1}$ and the parameter $\psi$ to the parameter $v\psi$, thereby identifying the two twisted endoscopic data used to define the left-hand side of Lemma \ref{lem:iop3} in the two cases at hand. We conclude that the two left-hand sides are equal, hence the two right-hand sides are equal. Since the summation sets are already identified with each other and $\pi(f)$ is identified with $v\pi(f)$, the claim follows.
\end{proof}

\begin{lem} \label{lem:lirind} Let $\Xi$ and $\Xi'$ be two equivalence classes of extended pure inner twists that determine the same equivalence class of inner twists. If Theorem \ref{thm:lir} holds for $\Xi$, then it also holds for $\Xi'$.
\end{lem}
\begin{proof}
By assumption we may choose $(\xi,z) \in \Xi$ and $(\xi',z') \in \Xi'$ such that the maps $\xi$ and $\xi'$ are equal. Then $z'=cz$ with $c \in Z^1_\tx{alg}(\mc{E},Z(G^*))$. By Lemma \ref{lem:tfequi} we have
\[ f'_{G,({\xi',z'})}(\psi,s_\psi u^{-1}) = \<c,s_\psi u^{-1}\>\cdot f'_{G,(\xi,z)}(\psi,u^{-1}). \]
Here we have paired $c$ with the image of $s_\psi u^{-1}$ in $[\hat G^*/\hat G^*_\tx{der}]^\Gamma$. Notice that $s_\psi \in \hat G^*_\tx{der}$, because it belongs to the image of a homomorphism $\tx{SL}_2 \rw \hat G$. Thus $\<c,s_\psi u^{-1}\>=\<c,u^{-1}\>$. According to Lemma \ref{lem:rpequiv}, we also have
\[ f_{G,(\xi',z)}(\psi,u^\natural) = \<c,u^\natural\>^{-1}\cdot f_{G,(\xi,z)}(\psi,u^\natural). \]
This deals with part 3 of Theorem \ref{thm:lir}. To see that part 1 of the theorem propagates from $(\xi,z)$ to $(\xi',z')$ we appeal to part 2 of Lemma \ref{lem:rpequiv}. Then it suffices to observe that $\lg c,u^\natural \rg=1$ when $u^\natural\in W_\psi(M^*,G^*)^{\rad}$, but this is clear since $S_\psi^{\rad}\subset \hat G^*_{\der}$ by definition (see \S\ref{sec:diag}). Since part 2 of Theorem \ref{thm:lir} follows from part 1 as remarked below the theorem, we are done.
\end{proof}

\begin{lem}\label{lem:central-action} In the same setup as in the preceding theorem,
  for any $y\in Z(\hat G^*)^\Gamma$, $f'_{G,\Xi}(\psi,s_\psi ys)=\<z,y\>f'_{G,\Xi}(\psi,s_\psi s)$ and $f_{G,\Xi}(\psi,yu)=\<z,y\>^{-1}f_{G,\Xi}(\psi,u)$, where the pairing is given Kottwitz's isomorphism \eqref{eq:kotisoloc} for the torus $G^*/G^*_\tx{der}$ and its dual $Z(\hat G^*)$.
\end{lem}

\begin{rem}\label{rem:central-action}
  We are assuming $M^*\neq G^*$ as in the theorem but even when $M^*=G^*$, the former assertion of the lemma remains true by the same argument. The latter also follows in the same way if the local classification theorem (Theorem \ref{thm:locclass-single}) is assumed.
\end{rem}

\begin{proof}
The former follows from Lemma \ref{lem:tfequi} and the latter from Lemma \ref{lem:rpequiv}.
\end{proof}

\begin{cor}\label{c:local-indep-f'}
 Fix $\ol{x}\in \ol{\cS}_\psi(G^*)$. Suppose that there exists a lift $x\in S^\natural_\psi(G^*)$ of $\ol x$ such that the local intertwining relation holds for all $u$ mapping to $x$. Then for every lift $x\in S^\natural_\psi(G^*)$ of $\ol x$, it holds true for all $u$ mapping to $x$.
\end{cor}

\begin{proof}
This follows from Lemmas \ref{lem:squot} and \ref{lem:central-action}.
\end{proof}

As a first application of Theorem \ref{thm:lir}, we will now construct the $L$-packet $\Pi_\psi$ of representations of $G(F)$ under the assumption that $\psi \in \Psi_2(M^*)$ and $M^* \neq G^*$. We assume that $M^*$ transfers to $G$, i.e. that $\psi$ is relevant for $\Xi$, otherwise we simply set $\Pi_\psi(G,\Xi) = \emptyset$. We further assume inductively Theorem \ref{thm:locclass-single} for the Levi subgroup $M \subset G$. Finally, we assume Theorem \ref{thm:lir} for the group $G$ and its Levi subgroup $M$.

Recall that $\psi \in \Psi_2(M^*)$ implies that the map $N_\psi(M^*,G^*) \rw S_\psi^\natural$ is surjective. For any $\pi \in \Pi_\psi(M)$, the map
\[ N^\natural_\psi(M^*,G^*) \times G(F) \rw \tx{Aut}(\mc{H}_P(\pi)),\qquad (\bar u,g) \mapsto R_P(\bar u,\Xi,\pi,\psi,\psi_F)\mc{I}_P(\pi,g) \]
is a representation of $N^\natural_\psi(M^*,G^*) \times G(F)$. Let $\Pi_\psi^1(G,\Xi)$ be the direct sum of these representations, as $\pi$ runs over $\Pi_\psi(M,\Xi_M)$. The value at $\bar u \in N^\natural_\psi(M^*,G^*)$ and $f \in \mc{H}(G)$ of the character of $\Pi_\psi^1(G,\Xi)$ is equal to $f_{G,\Xi}(\psi,\bar u)$. According to Theorem \ref{thm:lir}, the representation $\Pi_\psi^1(G,\Xi)$ is inflated from $S_\psi^\natural \times G(F)$. According to Lemma \ref{lem:central-action}, this representation transforms under $Z(\hat G^*)^\Gamma \times \{1\}$ by the character $\chi_\Xi^{-1}=\<\Xi,-\>^{-1}$. Decomposing this representation into irreducible constituents, we obtain
\[ \Pi_\psi^1(G,\Xi) = \bigoplus_\pi\left( \<\pi,-\>_\Xi^{-1} \otimes \pi\right) \]
where $\pi$ now runs over representations of $G(F)$ and $\<\pi,-\>_\Xi$ are elements of $X^*(S_\psi^\natural)$ whose restriction to $Z(\hat G^*)^\Gamma$ is equal to $\chi_\Xi$. We define the packet $\Pi_\psi(G,\Xi)$ to be the disjoint union of the occurring $\pi$ and the map $\Pi_\psi(G,\Xi) \rw \tx{Irr}(S_\psi^\natural,\chi_\Xi)$ to be given by $\pi \mapsto \<\pi,-\>_\Xi$. The character identity \eqref{eq:eci} of Theorem \ref{thm:locclass-single} now follows directly from Theorem \ref{thm:lir}. It is now not hard to see that the map $\Pi_\psi(G,\Xi) \rw \tx{Irr}(S_\psi^\natural,\chi_\Xi)$ is bijective when $F$ is $p$-adic. We will postpone this discussion until Section \ref{sec:lpackns}.

\subsection{A preliminary result on the local intertwining relation I}\label{sub:prelim-local-intertwining1}

In this section we are going to prove some results on the relationship between parabolic induction and the normalized intertwining operators as well as the local intertwining relation. One consequence of these results will be the reduction of the proof of the local intertwining relation to the case of discrete parameters. This essential case will be handled in later chapters using global methods and finally completed in Section \ref{sub:LIR-proof}. As we said at the start of \S\ref{sec:iop1}, we simplify notation in this subsection by omitting $*$ in the superscript when referring to the objects on the dual group side if there is no danger of confusion.

The general situation we will discuss is that of two nested standard proper Levi subgroups $M^*_0 \subset M^*$ of $G^*$ and a parameter $\psi_0 \in \Psi(M^*_0)$, as well as an element $u \in S_\psi \cap N(A_{\hat M_0},\hat G) \cap N(A_{\hat M},\hat G)$. In this situation we have the linear forms $f_{G,\Xi}(\psi_0,u^\natural)$ and $f_{G,\Xi}(\psi,u^\natural)$ on $\mc{H}(G)$, where $\psi$ is the image in $\Psi(M^*)$ of $\psi_0$. In the special case $u \in S_\psi(M^*) \cap N(A_{\hat M_0},\hat G)$ we also have the linear form $f_{M,\Xi_M}(\psi_0,u^\natural)$ on $\mc{H}(M)$. All these linear forms are zero unless $\psi$ is relevant for $G$, so we assume now that $M_0^*$ transfers to $G$ and fix corresponding Levi subgroups $M_0 \subset M \subset G$. We further fix a parabolic subgroup $P_0 \in \mc{P}^G(M_0)$. It gives rise to $P \in \mc{P}^G(M)$ and $Q \in \mc{P}^M(M_0)$. The following two lemmas express the relationships between these three linear forms.

\begin{lem*} \label{lem:lirdesc1} For any $u \in S_\psi \cap N(A_{\hat M_0},\hat G) \cap N(A_{\hat M},\hat G)$ we have the equality $f_{G,\Xi}(\psi_0,u^\natural) = f_{G,\Xi}(\psi,u^\natural)$.
\end{lem*}

\begin{lem*} \label{lem:lirdesc2} For any $u \in S_\psi(M) \cap N(A_{\hat M_0},\hat G)$ and $\pi \in \Pi_{\psi_0}(M_0,\Xi_{M_0})$ we have the equality of intertwining operators
\[ R^G_{P_0}(u^\natural,\Xi,\pi,\psi_0,\psi_F) = \mc{I}_P^G(R^M_Q(u^\natural,\Xi_M,\pi,\psi_0,\psi_F)). \]
Moreover, we have the equality $f_{G,\Xi}(\psi_0,u^\natural)=f_{M,\Xi_M}(\psi_0,u^\natural)$.
\end{lem*}

The proof of these two lemmas will occupy most of this section, but before we get to it, we will first extract an important consequence of Lemma \ref{lem:lirdesc1}.

\begin{pro} \label{pro:lirreddisc}
Assume that parts 2 and 3 of Theorem \ref{thm:lir} hold for all standard parabolic pairs $(M^*,P^*)$ of $G^*$ and all parameters $\psi \in \Psi_2(M^*)$. Then they hold for all standard parabolic pairs $(M^*,P^*)$ of $G^*$ and all parameters $\psi \in \Psi(M^*)$.
\end{pro}

\begin{proof}
As in the beginning of Section \ref{sec:iop} we choose $\psi$ within its equivalence class so that the minimal Levi subgroup of $^LM$ through which it factors is of the form $\hat M_0 \rtimes W_F$ for a standard Levi subgroup $\hat M_0$ of $\hat M$, hence of $\hat G$. Then $A_{\hat M_0}$ is a maximal torus of $S_\psi^0$. Let $M_0^*$ be the standard Levi subgroup of $G^*$ dual to $\hat M_0$ and let $P_0^*$ be the standard parabolic subgroup of $G^*$ with Levi factor $M_0^*$. We have $M_0^* \subset M^*$ and $P_0^* \subset P^*$. Moreover, $\psi \in \Psi(M^*)$ is the image of $\psi_0 \in \Psi_2(M_0^*)$ under the natural map $\Psi(M_0^*) \rw \Psi(M^*)$.

We have the linear forms $f_{G,\Xi}(\psi,u)$ and $f'_{G,\Xi}(\psi,s_\psi u^{-1})$ associated to the parabolic pair $(M^*,P^*)$ and the parameter $\psi$. They only depend on the image $u^\natural \in N_\psi^\natural(M,G)$ of $u$. We have $A_{\hat M_0} \cap S_\psi^\tx{rad}=\hat T \cap S_\psi^\tx{rad}=\hat T_\psi^\tx{rad}$ and the argument used in the beginning of Section \ref{sec:iop} allows us to assume that $u$ normalizes $\hat T_\psi^\tx{rad}$. Then it also normalizes $A_{\hat M_0}$ and thus belongs to $N_{\psi_0}(M_0,G) \cap N_\psi(M,G)$. This allows us to consider the linear forms $f_{G,\Xi}(\psi_0,u)$ and $f'_{G,\Xi}(\psi_0,s_\psi u^{-1})$. We have $f'_{G,\Xi}(\psi_0,s_\psi u^{-1})=f'_{G,\Xi}(\psi,s_\psi u^{-1})$: Indeed, both forms depend only on the images of $\psi_0$ and $\psi$ in $\Psi(G^*)$ and these are equal. On the other hand, Lemma \ref{lem:lirdesc1} asserts that $f_{G,\Xi}(\psi_0,u^\natural) = f_{G,\Xi}(\psi,u^\natural)$ and the proof is complete.
\end{proof}

The proof of Lemmas \ref{lem:lirdesc1} and \ref{lem:lirdesc2} will require the validity of the following piece of the quasi-split theory: the twisted local intertwining relation for a product of groups of the form $G_{E/F}(N)$ and an automorphism of this product that permutes the factors transitively. We will now review its statement.

Let $N_1$ and $k$ be positive integers and $H=G_{E/F}(N_1)^k$. The standard pinning of $G_{E/F}(N_1)$ gives rise to a pinning of $H$, whose maximal torus and Borel subgroup we call $T_H$ and $B_H$. Let $\theta_H$ be the automorphism of $H$ given by $\theta_H(h_1,\dots,h_k)=(\theta(h_k),h_1,\dots,h_{k-1})$, where $\theta$ is either the identity automorphism of $G_{E/F}(N_1)$, or the automorphism $\theta$ described in Section \ref{subsub:G(N)-defn}, whose fixed subgroup is $U_{E/F}(N_1)$. The Borel pair $(T_H,B_H)$ is invariant under $\theta_H$. Let $(M_H,P_H)$ be a standard parabolic pair for $H$. We are not assuming that it is invariant under $\theta_H$. Let $\psi \in \Psi(M_H)$ be a parameter and let $\pi$ be the unique representation of $M_H$ corresponding to $\psi$. For our purposes we may assume that $\psi$ is discrete. Let $S_\psi(H,\theta_H)$ denote the subset of the coset $\hat H \rtimes \hat \theta_H$ consisting of elements that centralize (the image of) $\psi$. Let $u \in N(A_{\hat M_H},S_\psi(H,\theta_H^{-1}))$ and let $w$ be its image in the Weyl set $W(\hat M_H,\hat H,\hat \theta_H^{-1})=N(A_{\hat M_H},\hat H \rtimes \hat \theta_H^{-1})/\hat M_H$, as well as in the isomorphic Weyl set $W(M_H,H,\theta_H)=N(A_{M_H},H \rtimes \theta_H)/M_H$.

Associated to the element $w$ is a self-intertwining operator $R_{P_H}(w,\psi,\psi_F)$ of the induced representation $\mc{H}_{P_H}^H(\pi)$. It is again a composition of three operators. The first operator is
\[ R_{w^{-1}P_H|P_H}(\pi) : \mc{H}_{P_H}^H(\pi) \rw \mc{H}_{w^{-1}P_H}^H(\pi). \]
This operator is given again as the evaluation at $\lambda=0$ of the product of the usual un-normalized intertwining operator $J_{w^{-1}P_H|P_H}(\pi_{\psi,\lambda},\psi_F)$ and the normalizing factor
\[ r_{w^{-1}P_H|P_H}(\psi_\lambda,\psi_F) =  \frac{L(0,\rho^\vee_{w^{-1}P_H|P_H}\circ\phi_{\psi_\lambda})}{L(1,\rho^\vee_{w^{-1}P_H|P_H}\circ\phi_{\psi_\lambda})} \frac{\epsilon(\frac{1}{2},\rho^\vee_{w^{-1}P_H|P_H}\circ\phi_{\psi_\lambda},\psi_F)}{\epsilon(0,\rho^\vee_{w^{-1}P_H|P_H}\circ\phi_{\psi_\lambda},\psi_F)}. \]

The second operator is
\[ l(w,\pi) : \mc{H}_{w^{-1}P_H}^H(\pi) \rw \mc{H}_{P_H}^H(\tilde w\pi). \]
To define it, we first let $\dot w \in N(T_H,H \rtimes \theta_H)/T_H$ be the unique element mapping to $w$ and preserving the Borel pair $(T_H,B_H \cap M_H)$ of $M_H$. Since $\theta_H$ already preserves $T_H$ we have $\dot w = \dot w_0 \rtimes \theta_H$ for some $\dot w_0 \in N(T_H,H)/T_H = W(T_H,H)$. Let $\tilde w_0 \in N(T_H,H)$ be the Langlands-Shelstad lift of $\dot w_0$ and put $\tilde w = \tilde w_0 \rtimes \theta_H$. We set $\tilde w\pi=\pi \circ \tilde w^{-1}$, where we are using the notation $\tilde w$ to denote both the element $\tilde w \in H \rtimes \theta_H$, as well as the automorphism $\tx{Ad}(\tilde w_0) \circ \theta_H$ of $H$ that it induces. Define the operator $l(w,\pi)$ by
\[ [l(w,\pi)\phi](h) = \epsilon(\frac{1}{2},\pi,\rho^\vee_{w^{-1}P_H|P_H},\psi_F)\lambda(\dot w,\psi_F)^{-1}\phi(\theta_H^{-1}(\tilde w_0^{-1}h)). \]

The third operator comes from an intertwining operator $\pi(\tilde w) : \tilde w\pi \rw \pi$. Write $\tilde w = (\tilde w_1,\dots,\tilde w_k) \rtimes \theta_H$, $M_H=M_1 \times \dots \times M_k$ and $\pi=\pi_1 \otimes \dots \otimes \pi_k$. Then we have $\tilde w_1\theta M_k=M_1$ and $\tilde w_iM_{i-1}=M_i$ for $i>1$. Furthermore we have $\pi_k\circ(\tilde w_1\theta)^{-1} \cong \pi_1$ and $\pi_{i-1}\circ \tilde w_i^{-1} \cong \pi_i$ for $i>1$. We fix arbitrary isomorphisms $\iota_1 : (V_{\pi_k},\pi_k\circ(\tilde w_1\theta)^{-1}) \rw (V_1,\pi_1)$ and $\iota_i : (V_{\pi_{i-1}},\pi_{i-1}\circ \tilde w_i^{-1}) \rw (V_{\pi_i},\pi_i)$ for $i>1$. The composition $\iota_k\dots\iota_1$ is an isomorphism $(V_{\pi_k},\pi_k\circ(\tilde w_k\tilde w_{k-1}\dots \tilde w_1\theta)^{-1}) \rw (V_{\pi_k},\pi_k)$. The automorphism $(\tilde w_k\tilde w_{k-1}\dots \tilde w_1\theta)$ preserves $M_k$ as well as the pinning of it induced from the pinning of $G_{E/F}(N_1)$. As discussed in the beginning of Section \ref{sec:iop3b}, there exists a canonical choice for the isomorphism $(V_{\pi_k},\pi_k\circ(\tilde w_k\tilde w_{k-1}\dots \tilde w_1\theta)^{-1}) \rw (V_{\pi_k},\pi_k)$ and we require that our choices of $\iota_1,\dots,\iota_k$ are such that the composition $\iota_k\dots\iota_1$ is equal to this canonical isomorphism. Then we obtain the isomorphism
\[ \pi(\tilde w) : (V_\pi,\pi\circ\tilde w^{-1}) \rw (V_\pi,\pi),\quad (v_1 \otimes \dots \otimes v_k) \mapsto (\iota_1(v_k) \otimes \iota_2(v_1) \otimes \dots \otimes \iota_k(v_{k-1})) \]
and it is independent of the choices of $\iota_1,\dots,\iota_k$.

We now let $R_{P_H}(w,\psi,\psi_F) = \mc{I}_{P_H}^H(\pi(\tilde w))\circ l(w,\pi) \circ R_{w^{-1}P_H|P_H}(\pi)$. This is an automorphism of the space $\mc{H}_{P_H}^H(\pi)$ and is furthermore an intertwining operator
\[ \mc{I}_{P_H}^H(\pi) \rw \mc{I}_{P_H}^H(\pi) \circ \theta_H. \]
For any $f \in \mc{H}(M_H)$ we also have the trace-class automorphism $\mc{I}_{P_H}^H(\pi,f)$ of $\mc{H}_{P_H}^H(\pi)$. We define the linear form $f \mapsto f_H(\psi,w)$ on $\mc{H}(H)$ by
\[ f_H(\psi,w) = \tx{tr}(R_{P_H}(w,\psi,\psi_F)\circ\mc{I}_{P_H}^H(\pi,f)). \]
On the other hand, the element $u^{-1} \in N(A_{\hat M_H},S_\psi(H,\theta_H))$, which gave rise to $w \in W(M_H,H,\theta_H)$, can also be used together with the parameter $\psi$ to produce an endoscopic triple $\mf{\tilde e}$ for the twisted group $(H,\theta_H)$. Just as in Section \ref{sec:lir}, this endoscopic triple leads to a second linear form $f \mapsto f_H'(\psi,u^{-1})$ on $\mc{H}(H)$: We let $f^\mf{\tilde e} \in \mc{H}(H^\mf{\tilde e})$ have orbital integrals matching those of $f$ with respect to the transfer factor normalized using Whittaker datum, this datum coming from the fixed pinning of $H$ and the additive character $\psi_F$. Then we evaluate at $f^\mf{e}$ the stable linear form on $\mc{H}(H^\mf{\tilde e})$ associated to the parameter $(\eta^\mf{\tilde e})^{-1}\circ\psi$.

The twisted intertwining relation we need is the following.
\begin{pro} \label{pro:tlir}
We have
\[ f_H(\psi,w) = f'_H(\psi, u^{-1}). \]
\end{pro}
Note that we could have equivalently stated this equality as $f_H(\psi,w)=f'_H(\psi,s_\psi u^{-1})$, since both $u$ and $s_\psi u$ are elements of $N_\psi(M,G)$ that map to $w$. This is a special feature of the group $G_{E/F}(N)$.

Using the arguments of Section \ref{sec:prelsimp} one can reduce this proposition to the case $k=1$, which is part of our assumptions stated in \ref{sub:results-qsuni}. This reduction however requires a more flexible definition of the intertwining operators for the group $G_{E/F}(N)$ than we currently have at our disposal. We will first give the new definition of the intertwining operator and then proceed to the reduction of Proposition \ref{pro:tlir} to the case $k=1$.

Consider two standard parabolic pairs $(M_1,P_1)$ and $(M_2,P_2)$ of $G_1=G_{E/F}(N_1)$. Let $\psi_1 \in \Psi(M_1)$ and let $\pi_1$ be the unique representation of $M_1(F)$ corresponding to $\psi_1$. We consider the set $\tx{Trans}_{G_1\rtimes\theta}(M_1,M_2)=\{g \in G_1|g\theta(M_1)g^{-1}=M_2\}$. This set may be empty. For any $\Gamma$-fixed $w \in M_2\lmod \tx{Trans}_{G_1\rtimes\theta}(M_1,M_2)$ we are going to define a representation $\tilde w\pi_1$ of $M_2(F)$ and an intertwining operator
\[ R_{P_2|P_1}(w,\psi_1,\psi_F) : \mc{H}_{P_1}^{G_1}(\pi_1) \rw \mc{H}_{P_2}^{G_1}(\tilde w\pi_1). \]
We begin by defining $\tilde w\pi_1$. Recall that $G_1$ is equipped with a standard pinning, whose Borel pair we will denote by $(T_1,B_1)$. Let $\dot w_1$ be the unique element of $N(T_1,G_1 \rtimes \theta)/T_1$ which maps to $w$ and which transports the Borel pair $(T_1,B_1 \cap M_1)$ of $M_1$ to the Borel pair $(T_1,B_1 \cap M_2)$ of $M_2$. Since $\theta$ already preserves $T_1$, we have $\dot w=\dot w_0 \theta$ for some $\dot w_0 \in N(T_1,G_1)/T_1$. Let $\tilde w_0 \in N(T_1,G_1)$ be the Langlands-Shelstad lift of $\dot w_0$ and let $\tilde w = \tilde w_0\theta$. The representation $\tilde w\pi$ of $M_2(F)$ acts on the vector space $V_{\pi_1}$ underlying $\pi_1$ and is defined by $\tilde w\pi = \pi\circ\tilde w^{-1}$, where we identify the element $\tilde w \in G_1 \rtimes \theta$ with the automorphism $\tx{Ad}(\tilde w_0)\theta$ of $G_1$.

We define the intertwining operator $R_{P_2|P_1}(w,\psi_1,\psi_F)$ as the composition $l(w,\psi_1,\psi_F) \circ R_{w^{-1}P_2|P_1}(\psi_1)$, where $R_{w^{-1}P_2|P_1}(\psi_1)$ is the usual intertwining operator, as defined in Section \ref{sec:iop1}, and $l(w,\psi_1,\psi_F)$ is an intertwining operator
\[ \mc{H}_{w^{-1}P_2}^{G_1}(\pi_1) \rw \mc{H}_{P_2}^{G_2}(\tilde w\pi_1) \]
that sends a function $\phi$ to the function
\[ g \mapsto \epsilon(\frac{1}{2},\rho_{w^{-1}P_2|P_1}^\vee,\pi_1,\psi_F)\lambda(\dot w)^{-1}\phi(\theta^{-1}(\tilde w_0^{-1}g)). \]
This completes the definition of $R_{P_2|P_1}(w,\psi_1,\psi_F)$. Note that when $(M_2,P_2)=(M_1,P_1)$, we recover the definition of the usual intertwining operator, denoted by $R_P(w,\pi,\psi)$ in \cite[(2.3.25)]{Arthur} in the untwisted case.

\begin{lem} \label{lem:iopxm} Let $(M_3,P_3)$ be a third standard parabolic pair for $G_1$ and let $w' \in \tx{Trans}_{G_1\rtimes\theta'}(M_2,M_3)/M_3$. Then we have
\[ R_{P_3|P_1}(w'w,\psi_1,\psi_F) = R_{P_3|P_2}(w',w\psi_1,\psi_F) \circ R_{P_2|P_1}(w,\psi_1,\psi_F). \]
\end{lem}
\begin{proof}
We have
\begin{eqnarray*}
&&R_{P_3|P_2}(w',w\psi_1,\psi_F)R_{P_2|P_1}(w,\psi_1,\psi_F)\\
&=&l(w',w\psi_1,\psi_F) \circ R_{w'^{-1}P_3|P_2}(w\psi_1)\circ l(w,\psi_1,\psi_F) \circ R_{w^{-1}P_2|P_1}(\psi_1)\\
&=&l(w',w\psi_1,\psi_F)l(w,\psi_1,\psi_F)R_{(w'w)^{-1}P_3|w^{-1}P_2}(\psi_1)\circ R_{w^{-1}P_2|P_1}(\psi_1)\\
&=&l(w',w\psi_1,\psi_F)l(w,\psi_1,\psi_F)R_{(w'w)^{-1}P_3|P_1}(\psi_1)
\end{eqnarray*}
It is thus enough to prove $l(w',w\psi_1,\psi_F)l(w,\psi_1,\psi_F)=l(w'w,\psi_1,\psi_F)$. The proof is very similar to that of Lemma \ref{lem:iop2m} so we will only give a sketch. The equality we would like to prove is
\begin{eqnarray*}
&&\phi((\tilde{w'w})^{-1}g)\epsilon(\frac{1}{2},\rho_{(w'w)^{-1}P_3|P_1}^\vee,\pi_1,\psi_F)\lambda(\dot w'\dot w)^{-1}\\
&=&\phi((\tilde w'\tilde w)^{-1}g)\epsilon(\frac{1}{2},\rho_{w'^{-1}P_3|P_2}^\vee,w\pi_1,\psi_F)\epsilon(\frac{1}{2},\rho_{w^{-1}P_2|P_1}^\vee,\pi_1,\psi_F)\lambda(\dot w')^{-1}\lambda(\dot w)^{-1}
\end{eqnarray*}
for any $\phi \in \mc{H}_{P_1}^{G_1}(\pi_1)$ and any $g \in G_1(F)$. According to \cite[Lemma 2.1.A]{LS87}, which is stated in sufficient generality so as to apply to the ``twisted'' elements $\dot w$ and $\dot w'$, we have
\[ \tilde{w'w}^{-1} = \prod_{\alpha>0,\dot w\alpha<0,\dot w'\dot w\alpha>0}\alpha^\vee(-1)\cdot (\tilde w'\tilde w)^{-1} \]
where the product is taken over the set of absolute roots of $T_1$ in $G_1$. The set $\{\alpha>0,\dot w\alpha<0,\dot w'\dot w\alpha>0\}$ is once again invariant under $\Gamma$, because $\dot w$, $\dot w'$, and $B_1$ are, as well as invariant under $W(T_1,M_1)$. For the latter, note that since $\dot w$ sends $B_1 \cap M_1$ to $B_2 \cap M_2$, the set under consideration contains only roots $\alpha$ which lie outside of $M_1$. But $W(T_1,M_1)$ preserves the set of positive roots outside of $M_1$, and furthermore $W(T_1,M_1)$ is transported to $W(T_1,M_2)$ by $\dot w$ and to $W(T_1,M_3)$ by $\dot w'\dot w$. We conclude as before that
\[ \phi((\tilde{w'w})^{-1}g)\phi((\tilde w'\tilde w)^{-1}g) =\lambda(m(x))(\tilde w'\tilde w)^{-1}, \]
where $\lambda(m(x)) \in A_{M_1}(F)$ is equal to the image under $\lambda = \sum_{\alpha>0,\dot w\alpha<0,\dot w'\dot w\alpha>0} \alpha^\vee \in X_*(A_{M_1})=X^*(\hat M_1)^\Gamma$ of the element $m(x)$ determined by $m(x) \rtimes x = \phi_\psi(x)$ for $x \in W_F$ any preimage of $-1 \in F^\times$ under the Artin reciprocity map. The equality we are proving thus becomes
\[ \lambda(m(x)) = \frac{\epsilon(\frac{1}{2},\rho_{w'^{-1}P_3|P_2}^\vee,w\pi_1,\psi_F)\epsilon(\frac{1}{2},\rho_{w^{-1}P_2|P_1}^\vee,\pi_1,\psi_F)}{\epsilon(\frac{1}{2},\rho_{(w'w)^{-1}P_3|P_1}^\vee,\pi_1,\psi_F)}
\frac{\lambda(\dot w'\dot w)}{\lambda(\dot w')\lambda(\dot w)}, \]
which is the direct analog of \eqref{eq:iop2eq}. To study the $\epsilon$-factors, we again decompose
\[ \epsilon(\frac{1}{2},\rho_{w^{-1}P_2|P_1}^\vee,\pi_1,\psi_F) = \prod_{\beta>0,w\beta<0} \epsilon(\frac{1}{2},\pi_1,\rho_\beta,\psi_F). \]
Since $w$ does not preserve $M_1$, but rather sends it to $M_2$, we need to interpret the index set of the product as going over those weights $\beta \in R(A_{\hat M_1},\hat G_1)$ which are positive, i.e. $\hat g_\beta \subset \mf{\hat n_1}$, and for which $w\beta \in R(A_{\hat M_2},\hat G_1)$ is negative, i.e. $\hat g_{w\beta} \subset \mf{\hat n_2}$. Since both $\hat P_1$ and $\hat P_2$ are standard, the notions of positivity can of course be interpreted in both cases as $\hat g_{\beta} \subset \hat b_1$ and $\hat g_{w\beta} \subset \hat b_1$, with $\hat b_1$ being the Lie algebra of the Borel subgroup $\hat B_1$. With this interpretation, the argument given in the proof of Lemma \ref{lem:iop2m} goes through and shows that
\[ \frac{\epsilon(\frac{1}{2},\rho_{w'^{-1}P_3|P_2}^\vee,w\pi_1,\psi_F)\epsilon(\frac{1}{2},\rho_{w^{-1}P_2|P_1}^\vee,\pi_1,\psi_F)}{\epsilon(\frac{1}{2},\rho_{(w'w)^{-1}P_3|P_1}^\vee,\pi_1,\psi_F)} = \tx{det}\left(\tx{Ad}(m(x) \rtimes x) \left| \bigoplus_{\alpha>0,w\alpha<0,w'w\alpha>0}\mf{\hat g}_{\alpha^\vee}\right.\right). \]
This again leaves us with having to show the analog of \eqref{eq:iop2eq2}, i.e.
\[ \tx{det}\left(\tx{Ad}(1 \rtimes x) \left| \bigoplus_{\alpha>0,w\alpha<0,w'w\alpha>0}\mf{\hat g}_{\alpha^\vee}\right.\right) = \lambda(\dot w'\dot w,\psi_F)^{-1}\lambda(\dot w,\psi_F)\lambda(\dot w',\psi_F). \]
The argument in the proof of Lemma \ref{lem:iop2m} applies verbatim in this case.
\end{proof}

Having established the more flexible notion of intertwining operators, we are now ready to reduce Proposition \ref{pro:tlir} to the case $k=1$.

\begin{lem} \label{lem:tlirred} Assume that Proposition \ref{pro:tlir} holds in the case $k=1$. Then it holds for any $k$.
\end{lem}
\begin{proof}
We begin with the discussion of $f_H(\psi,w)$. As in the construction of this linear form, we write $M_H=M_1 \times \dots \times M_k$, $\tilde w = (\tilde w_1,\dots,\tilde w_k) \rtimes \theta_H$ and $\pi=\pi_1 \otimes \dots \otimes \pi_k$. Write further $P_H=P_1 \times \dots \times P_k$ and $G_1=G_{E/F}(N_1)$. We have the isomorphism
\begin{equation} \label{eq:tlirred1} \mc{H}_{P_1}^{G_1}(\pi_1) \otimes \dots \otimes \mc{H}_{P_k}^{G_1}(\pi_k) \rw \mc{H}_{P_H}^H(\pi) \end{equation}
of $H$-representations, sending a simple tensor $\phi_1 \otimes \dots \otimes \phi_k$ to the function $H \rw V_\pi = \bigotimes_i V_{\pi_i}$ whose value at $(g_1,\dots,g_k) \in H$ is given by $\phi_1(g_1)\otimes\dots\otimes\phi_k(g_k)$. The operator $R_{P_H}(w,\psi,\psi_F)$ translates under this isomorphism to the composition of the following operators. First, the operator
\[ \mc{H}_{P_1}^G(\pi_1) \otimes \dots \otimes \mc{H}_{P_k}^G(\pi_k) \rw \mc{H}_{w_2^{-1}P_2}^G(\pi_1) \otimes \dots \otimes \mc{H}_{w_k^{-1}P_k}^G(\pi_{k-1}) \otimes \mc{H}_{(w_1\theta)^{-1}P_1}^G(\pi_k) \]
given by $R_{w_2^{-1}P_2|P_1}(\psi_1) \otimes \dots \otimes R_{w_k^{-1}P_k|P_{k-1}}(\psi_{k-1}) \otimes R_{(w_1\theta)^{-1}P_1|P_k}(\psi_k)$. Second, the operator from
\[ \mc{H}_{w_2^{-1}P_2}^G(\pi_1) \otimes \dots \otimes \mc{H}_{w_k^{-1}P_k}^G(\pi_{k-1}) \otimes \mc{H}_{(w_1\theta)^{-1}P_1}^G(\pi_k) \]
to
\[\mc{H}_{P_2}^G(\tilde w_2\pi_1) \otimes \dots \otimes \mc{H}_{P_k}^G(\tilde w_k\pi_{k-1}) \otimes \mc{H}_{P_1}^G(\tilde{w_1\theta} \pi_k) \]
given by $l(w_2,\psi_1,\psi_F) \otimes \dots \otimes l(w_k,\psi_{k-1},\psi_F) \otimes l(w_1\theta,\psi_k,\psi_F)$. And third, the operator
\[ \mc{H}_{P_2}^G(\tilde w_2\pi_1) \otimes \dots \otimes \mc{H}_{P_k}^G(\tilde w_k\pi_{k-1}) \otimes \mc{H}_{P_1}^G(\tilde{w_1\theta} \pi_k) \rw \mc{H}_{P_1}^G(\pi_1) \otimes \dots \otimes \mc{H}_{P_k}^G(\pi_k) \]
sending $\phi_1 \otimes \dots \otimes \phi_k$ to $\iota_1\circ\phi_k\otimes\iota_2\circ\phi_1 \otimes\dots\otimes\iota_k\circ\phi_{k-1}$.

Define for $i=1,\dots,k$ the operators
\[ \Lambda_i : \mc{H}_{P_{i-1}}^G(\pi_{i-1}) \rw \mc{H}_{P_i}^G(\pi_i) \]
by $\Lambda_1 = I_{P_1}^G(\iota_1)\circ R_{P_1|P_k}(w_1\theta,\psi_k,\psi_F)$ and $\Lambda_i=I_{P_i}(\iota_i)\circ R_{P_i|P_{i-1}}(w_i,\psi_{i-1},\psi_F)$ for $i>1$. Then we see that the operator $R_{P_H}(w,\psi,\psi_F)$ translates under the isomorphism \eqref{eq:tlirred1} to the composition of $\Lambda_2 \otimes \dots \otimes \Lambda_k \otimes \Lambda_1$ with the shift operator $\phi_1 \otimes \dots \otimes \phi_k \mapsto \phi_k \otimes \phi_1 \otimes \dots \otimes \phi_{k-1}$.

Consider now the isomorphism
\begin{equation} \label{eq:tlirred2} \mc{H}_{P_k}^G(\pi_k)^{\otimes k} \rw \mc{H}_{P_1}^G(\pi_1) \otimes \dots \otimes \mc{H}_{P_k}^G(\pi_k) \end{equation}
given by
\[ \Lambda_1 \otimes (\Lambda_2\circ\Lambda_1) \otimes \dots\otimes(\Lambda_k\circ\dots\circ\Lambda_1). \]
A direct computation shows that the operator $R_{P_H}(w,\psi,\psi_F)$ translates under the composition of \eqref{eq:tlirred2} and \eqref{eq:tlirred1} to the operator on $\mc{H}_{P_k}^G(\pi_k)^{\otimes k}$ that sends a simple tensor $\phi_1 \otimes\dots\otimes \phi_k$ to the simple tensor $(\Lambda_k\dots\Lambda_1\phi_k)\otimes\phi_1\otimes\dots\otimes\phi_{k-1}$. At the same time, given $f_1 \otimes \dots \otimes f_k = f \in \mc{H}(G)^{\otimes k} \cong \mc{H}(H)$, the operator $\mc{I}_{P_k}^{G_1}(\pi_k,f_1)\otimes\dots\otimes\mc{I}_{P_k}^{G_1}(\pi_k,f_k)$ translates under the composition of \eqref{eq:tlirred2} and \eqref{eq:tlirred1} to the operator $\mc{I}_{P_H}^H(\pi,f_1\otimes\dots\otimes f_k)$. Applying Lemma \ref{lem:traceprod} we conclude that
\[ f_H(\psi,w) = \tx{tr}(\mc{I}_{P_k}^{G_1}(\pi_k,f_1*\dots*f_k)\circ(\Lambda_k\circ\dots\circ\Lambda_1)). \]
Recall now that $w$ was the image of an element $u \in N(A_{\hat M_H},S_\psi(H,\theta_H))$, thus $w\psi=\psi$. From this and Lemma \ref{lem:iopxm} it follows that
\[ \Lambda_k\circ \dots\circ \Lambda_1 = \mc{I}_{P_k}^{G_1}(\iota_k\dots\iota_1)\circ R_{P_k|P_k}(w_k\dots w_1\theta,\psi_k,\psi_F). \]
The right hand side is the canonical intertwining operator on $\mc{H}_{P_k}^{G_1}(\pi_k)$. Thus
\[ f_H(\psi,w) = (f_1*\dots *f_k)_G(w_k\dots w_1\theta,\psi_k). \]
We now turn to the linear form $f'_H(\psi,u^{-1})$. Write $u=(u_1,\dots,u_k) \rtimes \hat\theta_H^{-1}$. Then $u^{-1}=(u_2^{-1},\dots,u_k^{-1},\hat\theta(u_1^{-1})) \rtimes\hat\theta_H$. According to Proposition \ref{pro:simptfs} we have
\[ f'_H(\psi,u^{-1})=(f_1*\dots*f_k)_{G_1}'(\psi_1,u_2^{-1}\dots u_k^{-1}\hat\theta(u_1^{-1})\rtimes\hat\theta). \]
From $u\psi=\psi$ we know $u_1^{-1}\psi_1=\hat\theta^{-1}(\psi_k)$, so conjugating by $\hat\theta \cdot u_1^{-1}$ the right hand side becomes
\[ (f_1*\dots*f_k)'_{G_1}(\psi_k,(u_k\dots u_1 \rtimes \hat\theta)^{-1}). \]
According to the $k=1$ case of Proposition \ref{pro:tlir}, which is our assumption, this equals the expression we obtained for $f_H(\psi,w)$ above.
\end{proof}

Combining Lemma \ref{lem:tlirred} with the assumption made in Section \ref{sub:results-qsuni} that Proposition \ref{pro:tlir} holds in the case $k=1$, we can now assume Proposition \ref{pro:tlir} for any $k$. We are now almost ready to commence with the proof of Lemmas \ref{lem:lirdesc1} and \ref{lem:lirdesc2}. For this, we return to the notation used in their statements. Thus, $G^*=U_{E/F}(N)$, $\Xi : G^* \rw G$ is an equivalence class of extended pure inner twists, $M^*_0 \subset M^*$ are standard Levi subgroup which transfer to $M_0 \subset M \subset G$, $\psi_0 \in \Psi(M^*_0)$ is a parameter whose image in $\Psi(M^*)$ we denote by $\psi$, and $u \in S_\psi \cap N(A_{\hat M_0},\hat G) \cap N(A_{\hat M},\hat G)$. In this paper, we will give the proof under the additional assumption $E/F$ is an extension of fields, leaving the case of a split quadratic algebra for \cite{KMS_B}.

By construction we have
\[ f_{G,\Xi}(\psi_0,u^\natural) = \sum_{\pi_0 \in \Pi_{\psi_0}(M_0,\Xi)} \tx{tr}(R_{P_0}^G(u^\natural,\Xi,\pi_0,\psi_0,\psi_F)\mc{I}_{P_0}^G(\pi_0)(f)). \]
We have inserted the superscript $G$ to keep track which ground we are inducing to. We view this as the character at $(u^\natural,f)$ of the representation of $N_{\psi_0}^\natural(M_0,G) \times \mc{H}(G)$ given by
\[ \bigoplus_{\pi_0 \in \Pi_{\psi_0}(M_0,\Xi)} R_{P_0}^G(-,\Xi,\pi_0,\psi_0,\psi_F)\circ\mc{I}_{P_0}^G(\pi_0,-). \]
We will compare this representation with the representation of $N_\psi^\natural(M,G) \times \mc{H}(G)$ whose character is  $f_{G,\Xi}(\psi,u^\natural)$. Setting $Q=P_0 \cap M$, we consider the induction in stages isomorphism
\[ \mc{H}_{P_0}^G(\pi_0) \rw \mc{H}_P^G(\mc{H}_Q^M(\pi_0)), x \mapsto y, x(g)=y(g,1), y(g,m)=\delta_P^{-\frac{1}{2}}(m)x(mg) \]
and study how the operator $R_{P_0}^G(-,\Xi,\pi_0,\psi_0,\psi_F)$ translates under this isomorphism and how this translated operator compares with $R_{P}^G(-,\Xi,\pi,\psi,\psi_F)$. Recall that $\Xi$ is assumed to be an equivalence class of pure inner twists. We choose $(\xi,z) \in \Xi$ with the property that if we decompose $M^*_0=M^*_{0,+} \times M^*_{0,-}$ with $M^*_{0,+}$ a product of groups of the form $G_{E/F}(N_1)$ and $M^*_{0,-}$ a group of the form $U_{E/F}(N_2)$, then in the corresponding decomposition $z=z_+ \times z_-$ we have $z_+=1$. In particular, $z$ is fixed by the Langlands-Shelstad lifts of all elements of $W(M^*_0,G^*)^\Gamma$. Furthermore, $P_0=\xi(P^*_0)$ is a parabolic subgroup of $G$ defined over $F$ with Levi factor $M_0$, and $M=\xi(M^*)$ and $P=\xi(P^*)$ form a parabolic pair of $G$ defined over $F$. Recall the definition \eqref{eq:iop}
\[ R_{P_0}^G(u^\natural,\Xi,\pi_0,\psi_0,\psi_F) = \mc{I}_{P_0}^G(\pi_0(u^\natural)_{\xi,z}) \circ l^G_{P_0}(w,\xi,\psi_0,\psi_F) \circ R^G_{w^{-1}P_0|P_0}(\xi,\psi_0), \]
where we have again added the superscript $G$ in order to keep track of the relevant groups. We have furthermore written $w$ for the image of $u^\natural$ in $W(M_0^*,G^*)$. Note that the automorphism of $A_{M_0^*}$ induced by $w$ preserves the subgroup $A_M$ and hence gives rise to a well-defined element of $W(M^*,G^*)$, which we will also call $w$. It coincides with the element corresponding to $u^\natural$.

The construction of the operator $R_{P_0}^G$ depends on lifting $w \in W(M_0^*,G^*)$ to an element of $N(T^*,G^*)(F)$, while the construction of the operator $R_P^G$ also depends on lifting $w$ to an element of $N(T^*,G^*)(F)$, but with the important difference that now $w$ is seen as an element of $W(M^*,G^*)$. The results of these two lifting procedures are in general different. More precisely, let $w_0$ be the unique element of $W(T^*,G^*)$ which normalizes $A_{M_0^*}$, preserves the Borel subgroup $B^* \cap M_0^*$, and whose image in $N(A_{M_0^*},W(T^*,G^*))/W(T^*,M_0^*) = W(M_0^*,G^*)$ is equal to $w$. Let $w_1$ be the unique element of $W(T^*,G^*)$ which normalizes $A_{M^*}$, preserves the Borel subgroup $B^* \cap M^*$, and whose image in $N(A_{M^*},W(T^*,G^*))/W(T^*,M^*) = W(M^*,G^*)$ is equal to $w$. The elements $w_0$ and $w_1$ of $W(T^*,G^*)$ will in general be different. Indeed, $w_0$ normalizes both $A_{M^*_0}$ and $A_{M^*}$ and preserves the Borel subgroup $B^* \cap M^*_0$ of $M^*_0$, while $w_1$ normalizes only $A_{M^*}$ and preserves the Borel subgroup $B^* \cap M^*$ of $M^*$. The images of $w_0$ and $w_1$ in $N(A_{M^*},W(T^*,G^*))/W(T^*,M^*)=W(M^*,G^*)$ being equal, we have $w_0=w_{01}w_1$ for some $w_{01} \in W(T^*,M^*)$. Since both elements $w_1$ and $w_0$ are fixed by $\Gamma$, so is $w_{01}$.

While it may appear tempting to use the decomposition $w_0=w_{01}w_1$ in combination with Lemma \ref{lem:iopm}, which in turn rests in particular on Lemmas \ref{lem:iop1m} and \ref{lem:iop2m}, to study $R^G_{P_0}(u^\natural,\Xi,\psi_0,\psi_F)$, we are not in a position to do so, because these lemmas only apply to the product of two elements of $N_\psi^\natural(M,G)$ or of $W(M^*,G^*)^\Gamma$, while the equation $w_0=w_{01}w_1$ expresses a relationship between two different lifts to $W(T^*,G^*)^\Gamma$ of the same element $w \in W(M^*,G^*)^\Gamma$. Instead, we need to go back to the individual constituents of the operator $R^G_{P_0}(u^\natural,\Xi,\psi_0,\psi_F)$ and study how each of them behaves.

Consider the element $w \in W(M^*_0,G^*)$ and view it as an element of $W(M_0,G)$ via the isomorphism determined by $\xi$. Since it preserves $M$ we have well-defined parabolic subgroups $w^{-1}P_0 \in \mc{P}^G(M_0$, $w^{-1}P \in \mc{P}^G(M)$ and $w^{-1}Q \in \mc{P}^{M}(M_0)$.

\begin{lem} \label{lem:l1} Under the induction in stages isomorphism we have the identification
\[ R^G_{w^{-1}P_0|P_0}(\xi,\psi_0) = R^G_{w^{-1}P|P}(\xi,\psi) \circ \mc{I}_P^G(R^M_{w^{-1}Q|Q}(\xi,\psi_0)). \]
\end{lem}
\begin{proof}
We may assume that $\psi_0$ is in general position, so that
\[ R^G_{w^{-1}P_0|P_0}(\xi,\psi_0)x(g) = r^G_{w^{-1}P_0|P_0}(\xi,\psi_0,\psi_F)^{-1}J^G_{w^{-1}P_0|P_0}(\xi,\psi_F), \]
with the general case following by analytic continuation. Let $N_P$, $N_{P_0}$ and $N_Q$ be the unipotent radicals of the parabolic subgroups $P$, $P_0$, $Q$. Then we have a direct product decomposition of affine varieties $N_{P_0} = N_P \times N_Q$ and this leads to a decomposition of domain of integration in the definition of $J^G_{w^{-1}P_0|P_0}$ as
\[ w^{-1}N_{P_0}w \cap \bar N_{P_0} = (w^{-1}N_Pw \cap \bar N_P) \times (w^{-1}N_Qw \cap \bar N_Q). \]
From this one sees directly that under the induction in stages isomorphism the operator $J^G_{w^{-1}P_0|P_0}$ is identified with the composition $J^G_{w^{-1}P|P}\circ\mc{I}_P(J^M_{w^{-1}Q|Q})$. The complex vector space $w^{-1}\mf{\hat n}_{P_0}w \cap \mf{\hat{\bar n}}_{P_0}$ has a corresponding decomposition
\begin{equation} \label{eq:l1eq1} w^{-1}\mf{\hat n}_{P_0}w \cap \mf{\hat{\bar n}}_{P_0} = w^{-1}\mf{\hat n}_{P}w \cap \mf{\hat{\bar n}}_{P} \times w^{-1}\mf{\hat n}_{Q}w \cap \mf{\hat{\bar n}}_{Q}\end{equation}
and this decomposition is stable under the adjoint action of $\psi_0(L_F)$. This leads to a decomposition
\[ r^G_{w^{-1}P_0|P_0}(\xi,\psi_0,\psi_F) = r^G_{w^{-1}P|P}(\xi,\psi,\psi_F) \cdot r^M_{w^{-1}Q|Q}(\xi,\psi_0,\psi_F). \]
\end{proof}

Next we study the operator $l_{P_0}^G(w,\xi,\psi_0,\psi_F)$ and how it relates to the operators $l_P^G(w,\xi,\psi,\psi_F)$ and $l_Q^M(w,\xi,\psi_0,\psi_F)$. The operators $l_{P_0}^G$ and $l_P^G$ are those defined in Section \ref{sec:iop2}. Their definition involves the images $\breve w_0$ and $\breve w_1$ under $\xi$ of the Langlands-Shelstad lifts $\tilde w_0$ and $\tilde w_1$ of the elements $w_0$ and $w_1$ respectively. The operator $l_Q^M$ will be described now. The automorphism $\theta_1^* := \tx{Ad}(\tilde w_1)$ of $M^*$ preserves the splitting of this group inherited from $G^*$. The element $w \in W(M^*_0,G^*)$ preserves $M^*$ and provides an element of $W(M^*_0,M^*,\theta_1^*)=N(A_{M^*_0},M^* \rtimes \theta_1^*) / M^*_0$. Its lift to $W(T^*,M^*,\theta_1^*)=N(T^*,M^* \rtimes \theta_1^*)/T^*$ is equal to $w_{01}\theta_1^*$ and its Langlands-Shelstad lift is equal to $\tilde w_{01}\theta_1^*$. Both $\theta_1^*$ and $\tilde w_{01}$ commute with $z_-$, hence with $z$, so $\breve w_{01} = \xi(\tilde w_{01}) \in M(F)$ and $\theta_1=\xi\theta_1^*\xi^{-1}$ is an automorphism of $M$ defined over $F$. Just as we did earlier in this section for the group $H$, we define
\[ [l_Q^M(w,\xi,\psi_0,\psi_F)\phi](m)=\epsilon(\frac{1}{2},\pi_{\psi_0},\rho_{w^{-1}Q|Q}^\vee,\psi_F)\lambda_{M^*}(w_{01}\theta_1^*,\psi_F)^{-1}\phi(\theta_1^{-1}(\breve w_{01}^{-1}m)). \]
We have added the subscript $M^*$ to the Keys-Shahidi constant to emphasize that the ambient group is now $M^*$ and not $G^*$. The representation $\pi_\psi$ is the same representation of $\tilde M_0(F)$ that is used in the normalization of the operator $l_{P_0}^G$.

\begin{lem} \label{lem:l2} Under the induction in stages isomorphism we have the identification
\[ l^G_{P_0}(w,\xi,\psi_0,\psi_F) = \mc{I}_P^G(l_Q^M(w,\xi,\psi_0,\psi_F)) \circ l_P^G(w,\xi,\psi,\psi_F). \]
\end{lem}
\begin{proof}
We claim that $\tilde w_0=\tilde w_{01}\tilde w_1$. Indeed, according to \cite[Lemma 2.1.A]{LS87}, we have
$\tilde w_0 = \lambda(-1)\tilde w_{01}\tilde w_1$, where $\lambda$ is the sum of all elements of the set subset of $R(T^*,G^*)^\vee$ given by $\{\alpha>0,w_{01}^{-1}\alpha<0,(w_{01}w_1)^{-1}\alpha>0\}$. But this set is empty. Indeed, since $w_{01} \in W(T^*,M^*)$ we know that the two conditions $\alpha>0$ and $w_{01}^{-1}\alpha<0$ imply that $\alpha \in R(T^*,M^*)^\vee$. But then $w_1^{-1}w_{01}^{-1}\alpha>0$ is impossible, because $w_1$ preserves the Borel subgroup $M^* \cap B^*$ of $M^*$.

Let $x \in \mc{H}_{P_0}^G(\pi_0)$ and let $y \in \mc{H}_P^G(\mc{H}_Q^M(\pi_0))$ correspond to $x$ under the induction in stages isomorphism. Beginning with the defining equation \eqref{eq:iop2} we see
\begin{eqnarray*}
&&[l^G_{P_0}(w,\xi,\psi_0,\psi_F)x](g)\\
&=&\epsilon_{P_0}(w,\psi_0,\psi_F) \cdot \lambda_{G^*}(w_0,\psi_F)^{-1} \cdot x(\breve w_0^{-1}g) \\
&=&\epsilon_{P_0}(w,\psi_0,\psi_F) \cdot \lambda_{G^*}(w_{01}w_1,\psi_F)^{-1} \cdot x(\breve w_{01}^{-1}\breve w_1^{-1}g)\\
&=&\epsilon_{P_0}(w,\psi_0,\psi_F) \cdot \lambda_{G^*}(w_{01}w_1,\psi_F)^{-1} \cdot y(\breve w_1^{-1}g,\breve w_1^{-1}\breve w_{01}^{-1}\breve w_1)
\end{eqnarray*}
Note that $\tx{Ad}(\breve w_1)=\theta_1$. The decomposition \eqref{eq:l1eq1} implies a corresponding decomposition
\[ \epsilon_{P_0}(w,\psi_0,\psi_F) = \epsilon_{P}(w,\psi,\psi_F)\epsilon_{Q}(w,\psi_0,\psi_F). \]
We claim that $\lambda_{G^*}(w_{01}w_1,\psi_F)=\lambda_{M^*}(w_{01}w_1,\psi_F)\lambda_{G^*}(w_1,\psi_F)$. Indeed, the left-hand side is a product indexed by the subset of $R(A_{T^*},G^*)$ given by $\{\alpha>0,w_{01}w_1\alpha<0\}$. The factors corresponding to the intersection of this subset with $R(A_{T^*},M^*)$ comprise $\lambda_{M^*}(w_{01}w_1,\psi_F)$. On the other hand,
\[ \{ \alpha \in R(A_{T^*},G^*) \sm R(A_{T^*},M^*)| \alpha > 0, w_{01}w_1\alpha<0\} = \{ \alpha \in R(A_{T^*},G^*) | \alpha > 0, w_1\alpha<0\}. \]
The inclusion of the left side into the right side follows from the fact that $w_{01}$ preserves the positive elements of $R(A_{T^*},G^*) \sm R(A_{T^*},M^*)$, while the opposite inclusion follows from the fact that $w_1$ preserves the positive elements of $R(A_{T^*},M^*)$. Now the right side is precisely the index set in the definition of $\lambda_{G^*}(w_1,\psi_F)$.
\end{proof}
With the previous two lemmas at hand we conclude that the representation
\[ \bigoplus_{\pi_0 \in \Pi_{\psi_0}(M_0,\Xi)} R_{P_0}^G(-,\Xi,\pi_0,\psi_0,\psi_F)\circ\mc{I}_{P_0}^G(\pi_0,-) \]
of $N_{\psi_0}^\natural(M_0,G) \times \mc{H}(G)$ is isomorphic to the representation sending $(u^\natural,f)$ to
\begin{eqnarray*}
&&\bigoplus_{\pi_0 \in \Pi_{\psi_0}(M_0,\Xi)} l_M^G(w,\xi,\psi,\psi_F) \circ R_{w^{-1}P|P}^G(\Xi,\pi_0,\psi_0,\psi_F)\\
&\circ&\mc{I}_{P}^G( \mc{I}_{P_0}^M(\pi_0(u^\natural))\circ l_Q^M(w,\xi,\psi_0,\psi_F)\circ R^M_{w^{-1}Q|Q}(\xi,\psi_0)\circ \mc{I}_Q^M(\pi_0,f)).
\end{eqnarray*}
The first two operators in this composition are defined in terms independent of the particular representation $\pi_0$ and can be extended, using the same formulas, to the direct sum of all $\pi_0$. This allows us to move the direct sum past them. On the other hand, the first three of the four factors inside of $\mc{I}_P^G$ comprise the canonical twisted self-intertwining operator $R_{Q}^M(u^\natural,(\xi,z),\pi_0,\psi_0,\psi_F)$. The above representation becomes
\begin{eqnarray*}
&&l_P^G(w,\xi,\psi,\psi_F) \circ R_{w^{-1}P|P}^G(\Xi,\pi_0,\psi_0,\psi_F)\\
&\circ&\mc{I}_{P}^G( \bigoplus_{\pi_0 \in \Pi_{\psi_0}(M_0,\Xi)}R_Q^M(u^\natural,(\xi,z),\pi_0,\psi_0,\psi_F)\circ \mc{I}_Q^M(\pi_0,f)).
\end{eqnarray*}

We can now apply the twisted local intertwining relation to the argument of $\mc{I}_P^G$. Indeed, we have the decomposition $M^*=M^*_+ \times M^*_-$, with $M^*_+$ being a product of groups of the form $G_{E/F}(N_1)$ and $M^*_-$ being equal to $U_{E/F}(N_-)$. Write $\xi=\xi_+ \times \xi_-$ accordingly. Our assumption $z_+=1$ made earlier ensures that $\xi_+ : M^*_+ \rw M_+$ is an isomorphism over $F$, while $(\xi_-,z_-) : M^*_- \rw M_-$ is a pure inner twist. The automorphism $\tx{Ad}(\tilde w_1)$ of $M^*$ preserves the splitting of this group inherited from $G^*$. It acts trivially on $M^*_-$ and so the twisted group $(M^*,\tx{Ad}(\tilde w_1))$ decomposes as the product of the twisted group $(M^*_+,\tx{Ad}(\tilde w_1))$ and the untwisted group $M^*_-$. The automorphism $\tx{Ad}(\tilde w_1)$ acts on $M^*_+$ by permuting the individual factors of the form $G_{E/F}(N_1)$. The twisted group $(M^*_+,\tx{Ad}(\tilde w_1))$ thus decomposes as a product of twisted groups $(M^*_{+,i},\theta_i)$, where each twisted group $(M^*_{+,i},\theta_i)$ is of the form $(H,\theta_H)$ discussed earlier in this section. The twisted intertwining relation for $M$ thus follows from the usual intertwining relation for $M_-$ stated as Theorem \ref{thm:lir}, whose validity is being assumed by induction, as well as the twisted local intertwining relation for each of the twisted groups $(M^*_{+,i},\theta_i)$, which is Proposition \ref{pro:tlir}. This relation, coupled with Theorem \ref{thm:locclass-single} and Proposition \ref{prop:local-stable-linear} tell us that
\[ \bigoplus_{\pi_0 \in \Pi_{\psi_0}(M_0,\Xi)}R_Q^M(u^\natural,(\xi,z),\pi_0,\psi_0,\psi_F)\circ \mc{I}_Q^M(\pi_0,f) \cong \bigoplus_{\pi \in \Pi_{\psi}(M,\Xi_M)} \pi(u^\natural)\pi(f) \]
as representations of $(N_{\psi_0}(M_0,G) \cap N_\psi(M,G)) \times \mc{H}(M)$. With this, the above representation becomes
\begin{eqnarray*}
l_P^G(w,\xi,\psi,\psi_F) \circ R_{w^{-1}P|P}^G(\Xi,\pi,\psi,\psi_F)\circ\mc{I}_{P}^G\left( \bigoplus_{\pi \in \Pi_{\psi}(M,\Xi_M)} \pi(u^\natural)\pi(f)\right).
\end{eqnarray*}
The trace at $(u^\natural,f)$ of the latter is by definition $f_{G,\Xi}(\psi,u^\natural)$.

We have thus proved $f_{G,\Xi}(\psi_0,u^\natural)=f_{G,\Xi}(\psi,u^\natural)$, i.e. Lemma \ref{lem:lirdesc1}. In the course of the proof we have also collected all the information we need for the proof of Lemma \ref{lem:lirdesc2}. Indeed, assume now that $u \in S_\psi(M) \cap N(A_{\hat M_0},\hat G)$. For $\pi \in \Pi_{\psi_0}(M_0,\Xi_{M_0})$ we consider the intertwining operator
\[ R_{P_0}^G(u^\natural,\Xi,\pi_0,\psi_0,\psi_F) = \mc{I}_{P_0}^G(\pi_0(u^\natural)_{\xi,z}) \circ l^G_{P_0}(w,\xi,\psi_0,\psi_F) \circ R^G_{w^{-1}P_0|P_0}(\xi,\psi_0). \]
We apply Lemmas \ref{lem:l1} and \ref{lem:l2} and note that both operators $R_{w^{-1}P|P}^G(\xi,\psi)$ and $l_P^G(w,\xi,\psi,\psi_F)$ are equal to the identity, because the image of $u$ in $W(\hat M,\hat G)$ is trivial, and so $w$ is the trivial element of $W(M,G)$. Thus
\begin{eqnarray*}
R_{P_0}^G(u^\natural,\Xi,\pi_0,\psi_0,\psi_F)&=&\mc{I}_P^G(\mc{I}_Q^M(\pi_0(u^\natural)_{\xi,z}) \circ l^M_{Q}(w,\xi,\psi_0,\psi_F) \circ R^M_{w^{-1}Q|Q}(\xi,\psi_0))\\
&=&\mc{I}_P^G(R_Q^M(u^\natural,\Xi_M,\pi_0,\psi_0,\psi_F)).
\end{eqnarray*}
This is the first statement of Lemma \ref{lem:lirdesc2}. The second statement follows immediately from the first. Indeed, for $f \in \mc{H}(G)$ we have
\begin{eqnarray*}
f_{G,\Xi}(\psi_0,u^\natural)&=& \sum_{\pi_0 \in \Pi_{\psi_0}(M_0,\Xi)} \tx{tr}(R_{P_0}^G(u^\natural,\Xi,\pi_0,\psi_0,\psi_F)\circ\mc{I}_{P_0}^G(\pi_0)(f))\\
&=&\sum_{\pi_0 \in \Pi_{\psi_0}(M_0,\Xi)} \tx{tr}(\mc{I}_P^G(R_Q^M(u^\natural,\Xi_M,\pi_0,\psi_0,\psi_F))\circ\mc{I}_P^G(\mc{I}_Q^M(\pi_0))(f))\\
&=&\sum_{\pi_0 \in \Pi_{\psi_0}(M_0,\Xi)} \tx{tr}(R_Q^M(u^\natural,\Xi_M,\pi_0,\psi_0,\psi_F)\circ\mc{I}_Q^M(\pi_0)(f_M))\\
&=&f_{M,\Xi_M}(\psi_0,u^\natural),
\end{eqnarray*}
where $f_M$ is the constant term of $f$ along $P$.

\subsection{A preliminary result on the local intertwining relation II}\label{sub:prelim-local-intertwining}

In the setup of the local intertwining relation there are a proper Levi subgroup $M^*$ of $G^*$ and a parameter $\psi_{M^*}\in \Psi(M^*)$. Throughout \S\ref{sub:prelim-local-intertwining} we will write $\psi$ for the image of $\psi_{M^*}$ in $\Psi(G^*)$ and distinguish it from $\psi_{M^*}$ whenever it helps to avoid confusion. We apologize for this change of convention.
  As a further reduction step, we reduce the proof of LIR to the case when the parameter $\psi\in \Psi(G^*)$ is either elliptic or non-elliptic but belongs to one of the two very special (``exceptional'') cases. Again the argument is based on the inductive hypotheses and purely local arguments.

  Throughout this subsection $G^*$ is assumed to be a unitary group over a local field $F$.\footnote{The ``linear case'', namely the case of $\GL(m,D)$, is treated separately when it comes to the local intertwining relation; one problem is that not all Levis of $\GL(m,D)$ globalize to Levis of global unitary groups.}
  Write a local parameter $\psi\in \Psi(G^*)$ as a direct sum of simple parameters:
  $$\psi=\ell_1\psi_1 \oplus \cdots \oplus \ell_r \psi_r,\quad r\ge1,~\ell_1\ge\cdots\ge \ell_r\ge 1,$$
  where $\psi_i$ are mutually inequivalent. The parameter $\psi$ is said to be elliptic if there exists a semisimple element in $\ol{S}_{\psi}$ whose centralizer in $\ol{S}_{\psi}$ is finite. The set $\Psi_2(G^*)$ (resp. $\Psi_{\el}(G^*)$, resp. $\Psi^2_{\el}(G^*)$, resp. $\Psi^{\el}(G^*)$) denotes the subset of $\Psi(G^*)$ consisting of discrete (resp. elliptic, resp. non-discrete and elliptic, resp. non-elliptic) members. We have inclusions
  $$\Psi_2(G^*)\subset \Psi_{\el}(G^*)\subset \Psi(G^*).$$
  Recall that the local intertwining relation is concerned with non-discrete parameters $\psi\in \Psi^2(G^*)=\Psi^{\el}(G^*)\coprod\Psi^2_{\el}(G^*)$. Elliptic non-discrete parameters allow an explicit description. See \cite[(6.4.1)]{Arthur} and also \cite[(7.4.1)]{Mok}.

\begin{lem}\label{lem:local-ell-nondisc-param} The set $\Psi^2_{\el}(G^*)$ exactly consists of $\psi$ of the form
    $$\psi=2\psi_1\oplus \cdots \oplus 2\psi_q \oplus \psi_{q+1} \oplus \cdots \oplus \psi_r,\quad q\ge 1$$
    such that $S_\psi\simeq \O(2,\C)^q \times \O(1,\C)^{r-q}$. Moreover $W^0_\psi=\{1\}$ for every $\psi\in \Psi^2_{\el}(G^*)$.
\end{lem}

\begin{rem}
  In case $G^*$ is a general linear group (which is excluded in this subsection), $\Psi^2_{\el}(G^*)$ is empty.
\end{rem}

\begin{proof}
  The characterization of $\Psi^2_{\el}(G^*)$ is easy from its definition and an explicit description of $S_\psi$ for a general parameter $\psi$. The triviality of $W^0_\psi$ is immediate from the fact that $S^0_\psi$ is abelian.
\end{proof}

  The following two cases are readily checked to be non-elliptic parameters and turn out to be more difficult to treat than the other non-elliptic parameters.
    \begin{enumerate}
\item[(exc1)] $\psi=2\psi_1\oplus\psi_2\oplus \cdots \oplus \psi_r$, $S_\psi\simeq \SL(2,\C)\times \O(1,\C)^{r-1}$,
\item[(exc2)] $\psi=3\psi_1\oplus\psi_2\oplus \cdots \oplus \psi_r$, $S_\psi\simeq \O(3,\C)\times \O(1,\C)^{r-1}$,
\end{enumerate}
  Define $\Psi_{\EXC1}(G^*)$ (resp. $\Psi_{\EXC2}(G^*)$) to be the subset of $\Psi(G^*)$ consisting of non-elliptic parameters of type (exc1) (resp. (exc2)) as above. Set $\Psi_{\EXC}(G^*):=\Psi_{\EXC1}(G^*)\coprod\Psi_{\EXC2}(G^*)$. Loosely speaking, we will tackle the local intertwining relation in the following order of increasing difficulty:
\begin{itemize}
  \item when $\psi$ is non-elliptic and not of form (exc1) or (exc2),
  \item when $\psi$ is non-discrete but elliptic, or of form (exc1) or (exc2).
\end{itemize}

Here we are going to show that in the former case, the induction hypothesis implies the local intertwining relation for $\psi$ by a local method.  The local intertwining relation in the latter case will be handled in \S\ref{sub:LIR-proof} only after developing enough global machineries relying on the trace formula.

 In view of Proposition \ref{pro:lirreddisc} we are reduced to the case where the following hypothesis holds true, which will remain in effect until the end of this subsection:
\begin{itemize}
  \item $\psi_{M^*}$ belongs to $\Psi_2(M^*)$ (rather than just $\Psi(M^*)$).
\end{itemize}
  Then the torus $\ol{T}_\psi$ of $\ol{S}_\psi:=S_\psi/Z(\hat G^*)^\Gamma$ given by
   $$\ol{T}_\psi:=Z(\hat M^*)^\Gamma/Z(\hat G^*)^\Gamma=A_{\hat M^*}/(A_{\hat M^*}\cap Z(\hat G^*)^\Gamma)$$
   is a maximal torus. Similarly $T_\psi:=A_{\hat M^*}$ is a maximal torus of $S_\psi$. Observe that $W^0_\psi$ and $W_\psi$ may be thought of as the Weyl groups for $\ol{T}_\psi$ in $\ol{S}^0_\psi$ and $\ol{S}_\psi$ (as well as the Weyl groups for $T_\psi$ in $S^0_\psi$ and $S_\psi$), respectively. For $\ol{s}\in \ol{S}_\psi$ denote by $\ol{T}_{\psi,\ol{s}}$ the centralizer of $\ol{s}$ in $\ol{T}_\psi$.

\begin{lem}\label{lem:S(M,G)-S(G)}
  Under the assumption $\psi_{M^*}\in \Psi_2(M^*)$ we have canonical isomorphisms
  $S^\natural_{\psi_{M^*}}(M^*,G^*)=S^\natural_\psi(G^*)$. %
\end{lem}

\begin{proof}
  Consider the natural maps
  $$S^\natural_{\psi_{M^*}}(M^*,G^*)=\frac{N(A_{\hat M^*},S_\psi)}{N(A_{\hat M^*},S^{\rad}_\psi)} \ra S^\natural_\psi(G^*)=\frac{S_\psi}{S^{\rad}_\psi}
\ra \ol{\cS}_\psi(G^*)=\frac{S_\psi}{S^0_\psi Z(\hat G^*)}.$$
  It is easy to see that $N(A_{\hat M^*},S_\psi)$ meets every connected component of $S_\psi$, so the composition map is onto. The second map is the quotient map by the image of $Z(\hat G^*)$, so surjective. Since $Z(\hat G^*)\subset N(A_{\hat M^*},S_\psi)$, we conclude that the first map above is onto.
\end{proof}

  Following \cite[(4.1.7)]{Arthur} we define $S_{\psi,\el}$ (resp. $\ol{S}_{\psi,\el}$) to be the set of $s\in S_{\psi,\sspl}$ (resp. $s\in \ol{S}_{\psi,\sspl}$ such that $Z(\Cent(s,S_\psi^0))$ (resp. $Z(\Cent(s,\ol{S}_\psi^0))$ has finite cardinality.
  It is easy to see that the natural surjection $S_\psi\ra \ol{S}_\psi$ carries $S_{\psi,\el}$ onto $\ol{S}_{\psi,\el}$.
  Considering the natural surjections $S_{\psi}\ra S^\natural_{\psi}$ and $\ol{S}_\psi\ra \ol{\cS}_\psi$, we define $$S^\natural_{\psi,\el}\quad\mbox{and}\quad \ol{\cS}_{\psi,\el}$$ to be the images of $S_{\psi,\el}$ and $\ol{S}_{\psi,\el}$, respectively. (Our $\ol{S}_{\psi,\el}$ is the same as Arthur's defined on \cite[p.331]{Arthur}.) Since the two natural composition maps $S_\psi\ra \ol{S}_\psi\ra \ol{\cS}_\psi$ and $S_\psi\ra S^\natural_\psi\ra \ol{\cS}_\psi$ agree, we see that $S^\natural_{\psi,\el}$ has the same image in $\ol{\cS}_\psi$ as $\ol{S}_{\psi,\el}$, namely $\ol{\cS}_{\psi,\el}$.

\begin{lem}\label{lem:local-exc}
  Suppose that $\psi\in \Psi^{\el}(G^*)$ does not fall into (exc1) or (exc2). Then
  \begin{enumerate}
  \item  every simple reflection $w\in W^\rad_\psi$ centralizes a torus of positive dimension in $\ol{T}_\psi$ and
  \item  $\dim\ol{T}_{\psi,\ol{s}}\ge 1,~\forall \ol{s}\in \ol{S}_\psi$.
  \end{enumerate}
\end{lem}

\begin{rem}
  Note that condition 1 is satisfied $\dim \ol{T}_\psi\ge 2$. If $\dim Z(\ol{S}_\psi)\ge 1$ then both conditions 1 and 2 hold.
\end{rem}

\begin{proof}
  We are supposing that $\psi$ is non-elliptic. It suffices to show that $\psi$ is of the form (exc1) or (exc2) under the assumption that at least one of conditions 1 or 2 is false. As we closely follow the discussion in the middle of the proof of \cite[Prop 4.5.1]{Arthur}, some details will be omitted. The starting point is the description
  $$ S_\psi = \left( \prod_{i\in I^+_\psi(G^*)} \O(\ell_i,\C)\right) \times \left(\prod_{i\in I^-_\psi(G^*)} \Sp(\ell_i,\C)\right) \times \left(\prod_{j\in J_\psi(G^*)} \GL(\ell_j,\C)\right)$$
  along with suitable integers $N_i$ for $i\in  I^+_\psi(G^*)\coprod I^-_\psi(G^*)$ and $N_j$ for $j\in J_\psi(G^*)$ such that
  $\sum_{i\in I^+_\psi(G^*)} N_i + \sum_{i\in I^-_\psi(G^*)} N_i+ 2 \sum_{j\in J_\psi(G^*)}N_j=N$ and $|I^+_\psi(G^*)|+ | I^-_\psi(G^*)|+2|J_\psi(G^*)|=r$.

  We may assume that $ J_\psi(G^*)$ is empty; otherwise $S_\psi$ contains a central torus of positive dimension, so clearly conditions 1 and 2 are satisfied. Likewise we may assume that either (i) $I^-_\psi(G^*)=\emptyset$ or (ii) $|I^-_\psi(G^*)|=1$ with $\Sp(\ell_1,\C)=\Sp(2,\C)$. Now we examine what happens over the parameter set $I^+_\psi(G^*)$. In case (i) an easy explicit computation shows that either $\ell_i\le 2$ for every $i\in I^+_\psi(G^*)$ or exactly one of the $\ell_i$'s is 3 and all the others are 1 in $I^+_\psi(G^*)$. The latter case is (exc2). In the former case a direct computation shows that either $\psi$ is elliptic or $\dim Z(\ol{S}_\psi)\ge 1$, but these are excluded by our assumption. In the remaining case (ii), the assumption on $\psi$ forces that $\O(\ell_i,\C)=\O(1,\C)$ for every $i\in I^+_\psi(G^*)$, resulting in (exc1).

\end{proof}

\begin{lem}\label{lem:S-elliptic} Let $\psi\in \Psi(G^*)$.
\begin{enumerate}
  \item
  The natural map $S^\natural_{\psi}\ra \ol{\cS}_{\psi}$ carries $S^\natural_{\psi,\el}$ onto $\ol{\cS}_{\psi,\el}$.
\item
  If $\psi\in \Psi_{\EXC}(G^*)$ then $S^\natural_{\psi,\el}=S^\natural_{\psi}$ and $\ol{\cS}_{\psi,\el}=\ol{\cS}_{\psi}$.
\end{enumerate}

\end{lem}

\begin{proof}
  The first part was already proved in the paragraph above Lemma \ref{lem:local-exc}. For part 2 we only need to check that $S^\natural_{\psi,\el}=S^\natural_{\psi}$ by part 1. This amounts to the assertion that the natural map $S_{\psi,\el}\ra S^\natural_\psi$ is onto. To see this recall that when $\psi$ is exceptional, $S_\psi$ is isomorphic to either $\Sp(2)\times \O(1)^{r-1}$ or $\SO(3)\times \O(1)^{r}$ for some $r\ge 1$. In either case the surjection $S_\psi\ra S^\natural_\psi$ is just the projection onto $\O(1)^{r-1}$ or $\O(1)^r$. Since every element of $\O(1)^{r-1}$ or $\O(1)^r$ (with trivial entry in $\Sp(2)$ or $\SO(3)$) belongs to $S_{\psi,\el}$, we get the desired surjectivity.

\end{proof}

  Let $(G,\Xi)$ be an equivalence class of extended pure inner twists of $G^*$ as before. The lemma below is proved by the same argument as on page 204 of \cite{Arthur}.

\begin{lem}\label{lem:local-indep-f} %
  Assume that either
\begin{enumerate}
  \item $\psi$ is elliptic, or
  \item every simple reflection $w\in W^\rad_\psi$ centralizes a torus of positive dimension in $\ol{T}_\psi$.
\end{enumerate}
  Then parts 1 and 2 of Theorem \ref{thm:lir} hold for $\psi$ (and for all $u$ and $f$).

\end{lem}

\begin{proof}

  We assume that $M^*$ transfers to a Levi subgroup $M$ of $G$ since $f(\psi,u^\natural)=0$ otherwise.
  If $\psi$ is elliptic then $W^\rad_\psi$ is trivial by Lemma \ref{lem:local-ell-nondisc-param}. So there exists a unique $u^\natural$ mapping to a given $x$. The desired assertion is obvious.
  From now we may work under the hypotheses that $\psi$ is non-elliptic and that condition 2 holds. In view of Lemma \ref{lem:iopm} it is enough to show that $$R_P(u^\natural,\Xi,\pi,\psi,\psi_F)=1,\quad u^\natural\in W^\rad_{\psi}$$ when $u^\natural$ corresponds to a simple reflection $w\in W^0_\psi$ via $W^\rad_\psi\simeq W^0_\psi$. We adopt the notation of the proof of Lemma \ref{lem:local-indep-f'} for any lift $s$ of $x$, and apply Lemma \ref{lem:lirdesc2} with $M_0=M$, $\psi_0=\psi$, and $M=M_s$, noticing that the assumption of the current lemma shows $M_s\subsetneq G$. Thus it is enough to verify that $R_{P\cap M_s}(u^\natural,\Xi_M,\pi,\psi,\psi_F)=1$, which holds true since $f_{M_s,\Xi_M}(\psi,u^\natural)$ is known to depend only on $u^\natural$ modulo $W^\rad_\psi$ by the induction hypothesis.

\end{proof}

\begin{lem}\label{lem:local-intertwining-non-exceptional} Let $x\in S_\psi^\natural(M^*,G^*)$. %
 Assume that either
\begin{enumerate}
  \item $\psi$ is elliptic and $x\notin S^\natural_{\psi,\el}$, %
   or
  \item $\psi$ is non-elliptic, every simple reflection $w\in W^\rad_\psi$ centralizes a torus of positive dimension in $\ol{T}_\psi$ (e.g. $\dim \ol{T}_\psi \ge 2$), and $\dim\ol{T}_{\psi,\ol{s}}\ge 1$ for all $\ol{s}\in \ol{S}_{\psi,\sspl}$.
\end{enumerate}
  Then $f'_{G,\Xi}(\psi,s_{\psi} s^{-1})=e(G) f_{G,\Xi}(\psi,u^\natural)$ whenever $u^\natural$ belongs to $N_\psi^\natural(M^*,G^*)$ and $s\in S_{\psi,\sspl}$ maps to $x$.
\end{lem}

\begin{proof}

  Lemmas \ref{lem:local-indep-f'} and \ref{lem:local-indep-f} tell us that $f_{G,\Xi}(\psi,u^\natural)=f_{G,\Xi}(\psi,x)$ and $f'_{G,\Xi}(\psi,s_\psi s^{-1})=f'_{G,\Xi}(\psi,x^{-1})$. We need to show $f'_{G,\Xi}(\psi,s_\psi x^{-1})= f_{G,\Xi}(\psi,x)$.
 We argue as in the proof of Lemma \ref{lem:local-indep-f'} up to \eqref{eq:LIR-induction-f'}, adopting the same notation.
 In particular $M^*_s$ is a Levi subgroup of $G^*$ corresponding to $\hat M_s:=Z_{\hat G^*}(T_{\psi,s})$ and that there are $\psi_{M_s}\in \Psi(M_s^*)$ and $s_{M_s}\in S_{\psi_{M_s},\sspl}$ whose image is $(\psi,s)$. The condition that
  $\dim\ol{T}_{\psi,\ol{s}}\ge 1$ implies that $M_s$ is a \emph{proper} Levi subgroup of $G^*$ (since then $T_{\psi,s}$ is not contained in the center of $Z_{\hat G}$, so $\hat M_s$ is smaller than $\hat G$). If we instead have the condition that $x\notin S^\natural_{\psi,\el}$ then every lift of $x$ in $\cE_{\psi}$ lies outside $\cE_{\psi,\el}$, so we deduce that $M_s$ is again a proper Levi subgroup. If $M^*_s$ does not transfer to $G$ then $\psi$ is irrelevant so $f_{G,\Xi}(\psi,u^\natural)=0$ while the argument of Lemma \ref{lem:local-indep-f'} shows that $f'_{G,\Xi}(\psi,s_\psi s^{-1})=0$, so we are done.

 Now we assume that $M^*_s$ does transfer to a Levi subgroup $M_s$ of $G$, equipped with an extended pure inner twist $(M_s,\xi_{M_s},z_{M_s})$ given by $\Xi$.
  When $s$ is replaced with $s_\psi s$, there is no change in $M_s$ and we get (again as in the proof of Lemma \ref{lem:local-indep-f'})
   $$f'_{G,\Xi}(\psi,s_\psi x^{-1})=f'_{M_s,\xi_{M_s},z_{M_s}}(\psi_{M_s},s_{\psi_{M_s}}x^{-1}_{M_s}).$$
   On the other hand, by taking $M=M_s$, $M_0=M$ in Lemma \ref{lem:lirdesc2}, noticing that $x$ can be lifted to $u^\natural$ as in that lemma, we obtain
   \begin{equation}\label{eq:LIR-induction-f}
  f_{G,\Xi}(\psi,x)=f_{M_s,\xi_{M_s},z_{M_s}}(\psi_{M_s},x).
  \end{equation}
  It follows from the induction hypothesis that
  $$f'_{M_s,\xi_{M_s},z_{M_s}}(\psi_{M_s},s_{\psi_{M_s}}x^{-1}_{M_s})=e(M_s) f_{M_s,\xi_{M_s},z_{M_s}}(\psi_{M_s},x_{M_s}).$$
  Hence we conclude that $f'_{G,\Xi}(\psi,s_\psi x^{-1})= e(G) f_{G,\Xi}(\psi,x)$ since $e(G)=e(M_s)$ by \cite[Cor (6)]{KotSign}.

\end{proof}

\begin{cor}\label{cor:local-intertwining-step1}
  Let $\psi$ be as above and $x\in S_\psi^\natural(M^*,G^*)$. Then Theorem \ref{thm:lir} holds for all $u^\natural\in N_\psi^\natural(M^*,G^*)$ mapping to $x$ unless either
  \begin{itemize}
    \item  $\psi$ is elliptic and $x\in S^\natural_{\psi,\el}(M^*,G^*)$, or
    \item  $\psi \in \Psi_{\EXC}(G^*)$ (then automatically $x\in S^\natural_{\psi,\el}(M^*,G^*)$ by Lemma \ref{lem:S-elliptic}).
  \end{itemize}
\end{cor}

\begin{proof}
  Immediate from Lemmas \ref{lem:local-exc}, \ref{lem:local-indep-f}, and \ref{lem:local-intertwining-non-exceptional}.
\end{proof}

  Define $W_{\psi,\reg}(M^*,G^*)$ to be the subset of $w\in W_{\psi}(M^*,G^*)$ such that $w$ has finitely many fixed points on $\ol T_\psi$. (This is consistent with the definition of $W_{\reg}(S)$ in \S\ref{sub:stable-multiplicity} below.)
  We write $N^\natural_{\psi,\reg}(M^*,G^*)$ for the preimage of $W_{\psi,\reg}(M^*,G^*)$ in $N^\natural_{\psi}(M^*,G^*)$.

\begin{lem}\label{lem:w_u-regular-enough}

  Assume that $\psi \in \Psi_{\EXC}(G^*)$ and $u\notin N^\natural_{\psi,\reg}(M^*,G^*)$ (i.e. $u\in N^\natural_\psi(M^*,G^*)$ has the property that $w_u\notin W_{\psi,\reg}(M^*,G^*)$). Then $f'_{G,\Xi}(\psi,s_\psi s^{-1})=e(G) f_{G,\Xi}(\psi,u)$ provided that $u$ and $s\in S_{\psi,\sspl}$ have the same image in $S_\psi^\natural(M^*,G^*)$.

\end{lem}

\begin{proof}
  For each $u$ it is enough to prove for any one $s$ which has the same image in $S_\psi^\natural(M^*,G^*)$
   by Lemma \ref{lem:local-indep-f'}.
  It is seen from the explicit form of the parameter in (exc1) and (exc2) that $|W^0_\psi|=|W_\psi|=2$ and $\dim \ol{T}_\psi=1$ in either case, and exactly one element of $W_\psi$ is regular. So the hypothesis implies that $w_u=1$, which in turn tells us in view of \eqref{eq:diag} that $u$ is the image of some element $s\in Z_{S_\psi}(T_\psi)$ in $N^\natural_\psi$. Then obviously $\dim\ol{T}_{\psi,s}=\dim\ol{T}_\psi=1$. In this case the descent argument exactly as in the proof of Lemma \ref{lem:local-intertwining-non-exceptional} applies to show that $f'_{G,\Xi}(\psi,s_\psi s^{-1})=e(G) f_{G,\Xi}(\psi,u)$.

\end{proof}

  Unless we are in case (exc1) or (exc2), Lemmas \ref{lem:local-exc} and \ref{lem:local-indep-f} tell us that $f_{G,\Xi}(\psi,u)$ depends only on the image of $u$ in $S^\natural_\psi(M^*,G^*)$ so that $f_{G,\Xi}(\psi,x)$ is well-defined for $x\in S^\natural_\psi(M^*,G^*)$. If $\psi\in \Psi_{\EXC}(G^*)$ then we may only consider
  \begin{itemize}
    \item $u\in N^\natural_{\psi}$ such that $w_u\in W_{\psi,\reg}(M^*,G^*)$
  \end{itemize}
  thanks to Lemma \ref{lem:w_u-regular-enough}. Then one checks from the explicit form of (exc1) and (exc2) that there is a unique element $u\in N^\natural_{\psi,\reg}$ in the fiber over $x$ such that $w_u\in W_{\psi,\reg}(M^*,G^*)$. (In either case $|W^0_\psi|=2$ so there are two elements in the fiber over $x$. One of them maps to a regular element in $W_{\psi}(M^*,G^*)$ but the other has trivial image, which is thus not regular.) Hence we may and will choose the convention that whenever $\psi\in \Psi_{\EXC}(G^*)$ and $x\in S^\natural_\psi(M^*,G^*)$, we set $f_{G,\Xi}(\psi,x)$ to be $f_{G,\Xi}(\psi,u)$ for the unique $u$ just mentioned.

  The harder cases of the local intertwining relation, left over from Corollary \ref{cor:local-intertwining-step1} and Lemma \ref{lem:w_u-regular-enough}, will be treated in \S\ref{chapter4} after the method of comparing the trace formula is developed.
  The discussion in this subsection has a close analogue in the global situation (see \S\ref{sub:global-intertwining}).

\subsection{A proof of the local intertwining relation in a special case} \label{sec:lirexp}

In this section we are going to prove a special case of Theorem \ref{thm:lir} by a direct computation. This special case will serve as a basis for the inductive proof of the general case of Theorem \ref{thm:lir}, which will use global techniques and will be interwoven with the proofs of the main local and global theorems. This section will also provide a good illustration of the objects involved in Theorem \ref{thm:lir} and their interplay.

The special case we are concerned with here is the following. We take $E/F=\C/\R$ and $N=4$ and thus the quasi-split group is $G^*=\U_{\C/\R}(4)$. Fix the standard additive character $\psi_\R : \R \rw \C^\times$ given by $\psi_\R(x)=\exp(2\pi i x)$. The inner form we are interested in is the real reductive group $G$ with $G(\C)=\tx{GL}_4(\C)$ and Galois action given by $\sigma(g)=\tx{Ad}(\delta_{3,1}J_4)\bar g^{-t}$, where $\sigma \in \Gamma_{\C/\R}$ is the non-trivial element, $\bar\ $ denotes complex conjugation of complex numbers, as well as complex conjugation of the entries of a complex matrix, and $\delta_{3,1}$ is the matrix
\[ \delta_{3,1} = \begin{bmatrix} 1\\ &0&i\\&-i&0\\ &&&1 \end{bmatrix}. \]
The group $G$ is isomorphic to the unitary group $\U(3,1)$ and $(\tx{id},\delta_{3,1}) : G^* \rw G$ is a pure inner twist. We are using the short-hand notation here and throughout this section which identifies a 1-cocycle of $\Gamma_{\C/\R}$ with its value at $\sigma$. We write $(\xi,z)=(\tx{id},\delta_{3,1})$.

We consider the group $M^*=\tx{Res}_{\C/\R}(\mb{G}_m) \times \U_{\C/\R}(2)$. If we represent $\tx{Res}_{\C/\R}(\mb{G}_m)$ by viewing its $\C$-points as $\C^\times \times \C^\times$ with $\sigma(a,b) = (\bar b,\bar a)$, then $M^*$ embeds as a standard Levi subgroup of $G^*$ by the map
\[ m^* : M^* \rw G^*,\qquad (a,b,B) \mapsto \begin{bmatrix} a\\ &B\\ &&b^{-1} \end{bmatrix}. \]
We take $P^*$ to be the unique standard parabolic subgroup of $G^*$ with Levi factor $M^*$. It is clear that $(\xi,z)$ restricts to a pure inner twist $M^* \rw M$, where $M$ is the group $\tx{Res}_{\C/\R}(\mb{G}_m) \times U(2,0)$ and is embedded as a Levi subgroup of $G$ by the map $m : M \rw G$ given by the same formula as for $m^*$. We let $P=\xi(P^*)$.

Now consider the standard parabolic pair $(\hat M,\hat P)$ of $\hat G$ dual to $(M^*,P^*)$. The Weil-group $W_\R$ is a non-split extension of $\C^\times$ by $\Gamma_{\C/\R}$. We fix a lift $j \in W_\R$ of $\sigma \in \Gamma_{\C/\R}$ that satisfies $j^2=-1$ and denote elements of $\C^\times \subset W_\R$ by $z$. We then take, for $x \in \frac{1}{2}\Z$, the Langlands parameter $\phi : W_\R \rw {^LG}$ given by $\phi(w)=\phi^0(w) \rtimes w$ with
\[ \phi^0(z) = \begin{bmatrix} z^x\bar z^{-x}\\& z^\frac{1}{2}\bar z^{-\frac{1}{2}}\\&& z^{-\frac{1}{2}}\bar z^{\frac{1}{2}}\\ &&& z^x\bar z^{-x} \end{bmatrix},\quad \phi^0(j) = \begin{bmatrix} 1\\ &&1\\ &-1\\ &&& (-1)^{2x} \end{bmatrix}. \]
It is clear that the image of $\phi$ belongs to $^LM$.

Consider the restriction of $\phi$ to $W_\C$. On the one hand, according to \cite[Thm. 8.1]{GGP12}, this is a conjugate-symplectic representation of $W_\C$. On the other hand, it decomposes as a product of characters
\[ 2(x,-x)\oplus (\frac{1}{2},-\frac{1}{2}) \oplus (-\frac{1}{2},\frac{1}{2}) \]
where we have denoted by $(a,b)$ the character of $W_\C = \C^\times$ given by $z \mapsto z^a\bar z^b$. Since the character $(a,b)$ is conjugate-symplectic if $a-b \in \Z \smallsetminus 2\Z$ and conjugate-orthogonal if $a-b \in 2\Z$, we infer from \cite[\S4]{GGP12} that $S_\phi^\natural=\pi_0(S_\phi)$ is given by the following table
\[ \begin{tabular}{|>{$}c<{$}|>{$}c<{$}|}
\hline
x&S_\phi^\natural\\
\hline
\Z&\tx{Sp}(2) \times \O(1) \times \O(1)\\
\hline
\frac{1}{2}&\O(3) \times \O(1)\\
\hline
-\frac{1}{2}&\O(1) \times \O(3)\\
\hline
\frac{1}{2}\Z \smallsetminus \Z \cup \{-\frac{1}{2},\frac{1}{2}\}& \O(2) \times \O(1) \times \O(1)\\
\hline
\end{tabular} \]
On the other hand, $S_\phi(M)=\O(1) \times \O(1)$ for all $x$. We see that for $x \in \Z \cup \{-\frac{1}{2},\frac{1}{2}\}$ the endoscopic $R$-group $R_\phi(M,G)$ is trivial. We see furthermore that for $x \in \Z$ the parameter $\phi$ is of type (exc1), while for $x \in \{-\frac{1}{2},\frac{1}{2}\}$ it is of type (exc2), where (exc1) and (exc2) where the exceptional parameter types discussed in Section \ref{sub:prelim-local-intertwining}.

\begin{pro} \label{pro:u31lir}
Theorem \ref{thm:lir} is valid for the parabolic pair $(M,P)$ of $G$ and the parameter $\phi$ whenever $x \in \{-2,-1,-\frac{1}{2},0,\frac{1}{2},1,2\}$.
\end{pro}

The rest of this section is devoted to the proof of this proposition. We begin by making explicit Diagram \eqref{eq:diag}. Since we are working with unitary groups, $S_\phi^\tx{rad}=S_\phi^0$. Furthermore, $S_\phi^0 \cap S_\phi(M)=S_\phi(M)^0$. Diagram \eqref{eq:diag} specialized to this case has the form (with $C_2=\Z/2\Z$)
\[ \xymatrix{
&&0\ar[d]&\\
&&\{0\} \times \{0\} \times C_2\ar@{=}[r]\ar[d]&\{0\} \times \{0\} \times C_2\ar@{=}[d]\\
0\ar[r]&C_2 \times C_2\times \{0\}\ar@{=}[d]\ar[r]&C_2 \times C_2 \times C_2 \ar[r]\ar[d]&\{0\} \times \{0\} \times C_2\ar[r]&0\\
0\ar[r]&C_2 \times C_2\times \{0\}\ar@{=}[r]&C_2 \times C_2\times \{0\}\ar[d]\\
&&0
} \]
For any value of $x$, the non-trivial elements in the first two copies of $C_2$ can be represented by the diagonal matrices with diagonal entries $(1,-1,1,1)$ and $(1,1,-1,1)$. These matrices belong to $S_\phi(M) \subset S_\phi(G)$ and we will call them $a_1$ and $a_2$. A matrix representing the non-trivial element in the third copy of $C_2$ however depends on the value of $x$. We have the three cases
\[ x\in \Z: \begin{bmatrix}&&&1 \\ &1\\ &&1\\ -1 \end{bmatrix},\quad x=\frac{1}{2}: \begin{bmatrix}&&&1\\ &-1\\ &&1\\ 1\end{bmatrix},\quad x=-\frac{1}{2}:\begin{bmatrix}&&&1\\ &1\\ &&-1\\ 1 \end{bmatrix} \]
This matrix belongs to $S_\phi^0 \sm S_\phi(M)$ and normalizes $A_{\hat M}$. We will call it $a_3$.

We now take $u = a_1^{\epsilon_1} \cdot a_2^{\epsilon_2} \cdot a_3^{\epsilon_3} \in N_\phi(M,G)$ for $\epsilon_i\in \{0,1\}$ and consider the linear form $f(\phi,u)$ that is the essential ingredient of the right-hand side of Theorem \ref{thm:lir}. The $L$-packet $\Pi_\phi(M,\xi)$ contains a single element $\pi$. It is the representation of $M(\R) = \C^\times \times U(2,0)(\R)$ which is the tensor product of the character $z \mapsto z^x\bar z^{-x}$ of $\C^\times$ and the trivial representation of $U(2,0)(\R)$. The equality of the representation-theoretic and endoscopic $R$-groups is known for real groups \cite[\S3]{KZ79} and implies that $\mc{H}_P(\pi)$ is irreducible and thus that $R_P(u,(\xi,z),\pi,\phi,\psi_\R)$ is a scalar. Using the multiplicativity of $R_P(u,(\xi,z),\pi,\phi,\psi_\R)$ in $u$ guaranteed by Lemma \ref{lem:iopm}, we can write the right-hand side of equation \eqref{eq:lir} as
\[ e(G)R_P(a_3^{\epsilon_3},(\xi,z),\pi,\phi,\psi_\R)R_P(a_1^{\epsilon_1}a_2^{\epsilon_2},(\xi,z),\pi,\phi,\psi_\R)\tx{tr}(\mc{I}_P(\pi)(f)). \]
Consider the scalar operator $R_P(a_1^{\epsilon_1}a_2^{\epsilon_2},(\xi,z),\pi,\phi,\psi_\R)$. Since $a_1^{\epsilon_1}a_2^{\epsilon_2}$ has trivial image in $W_\phi(M,G)$, the right and middle factors in \eqref{eq:iop} are easily seen to be trivial. For the construction of $\pi(a_1^{\epsilon_1}a_2^{\epsilon_2})_{\xi,z}$ we can use Section \ref{sec:iop3u}, because the 1-cocycle $z$ satisfies the assumption $z_+=1$ under which the constructions of that section are valid. We then see that $\pi(a_1^{\epsilon_1}a_2^{\epsilon_2})_{\xi,z}=\<a_1^{\epsilon_1}a_2^{\epsilon_2},\pi\>_{\xi,z}^{-1}$, where the pairing $\<-,-\>_{\xi,z}$ is the pairing between the $L$-packet $\Pi_\phi(M)$ and the centralizer $S_\phi(M)$ associated to the pure inner twist $(\xi,z) : M^* \rw M$ by Theorem \ref{thm:locclass-single}. For this we just need to note that $a_1,a_2 \in S_\phi(M)$ and thus the section $s : \pi_0(N_\phi(M,G)) \rw \pi_0(S_\phi(M))$ employed in the construction of $\pi(\bar u)_{\xi,z}$ fixes them. It follows that the right-hand side of equation \eqref{eq:lir} has the form
\[ e(G)R_P(a_3^{\epsilon_3},(\xi,z),\pi,\phi,\psi_\R)\<a_1^{\epsilon_1}a_2^{\epsilon_2},\pi\>_{\xi,z}^{-1}\tx{tr}(\mc{I}_P(\pi)(f)). \]
We now turn to the left-hand side of Theorem \ref{thm:lir}. The triviality of the endoscopic $R$-group implies that the pairing $\<-,-\>_{\xi,z}$ between the $L$-packet $\Pi_\phi$ on $G$ and $S_\phi$ is compatible with the pairing $\<-,-\>_{\xi,z}$ between the $L$-packet $\Pi_\phi(M)$ and $S_\phi(M)$ in the sense that $\<\mc{I}_P(\pi),s\>_{\xi,z} = \<\pi,s\>_{\xi,z}$ for any $s \in \pi_0(S_\phi(M))=\pi_0(S_\phi)$. An exposition of this can be found in \cite[\S5.6]{Kal13}. Noting that the images of $a_1^{\epsilon_1}\cdot a_2^{\epsilon_2}\cdot a_3^{\epsilon_3}$ and $a_1^{\epsilon_1} \cdot a_2^{\epsilon_2}$ in $\pi_0(S_\phi)$ coincide, the endoscopic character identities for real groups \cite{She82}, \cite{SheTE2}, \cite{SheTE3}, (see also \cite[\S5.6]{Kal13} for an exposition in the case of pure inner forms) imply that the left-hand side of equation \eqref{eq:lir}, given by $f'(\phi,a_1^{-\epsilon_1} \cdot a_2^{-\epsilon_2} \cdot a_3^{-\epsilon_3})$, is equal to
\[
e(G)\<\mc{I}_P(\pi),a_1^{-\epsilon_1} \cdot a_2^{-\epsilon_2} \cdot a_3^{-\epsilon_3}\>_{\xi,z}\tx{tr}(\mc{I}_P(\pi)(f))=
e(G)\<\pi,a_1^{\epsilon_1} \cdot a_2^{\epsilon_2}\>_{\xi,z}^{-1}\tx{tr}(\mc{I}_P(\pi)(f)). \]
Comparing this with the expression we obtained for $f(\phi,a_1^{\epsilon_1} \cdot a_2^{\epsilon_2} \cdot a_3^{\epsilon_3})$ above, we see that equation \eqref{eq:lir} will be proved once we show
\begin{equation} \label{eq:lirexp1} R_P(a_3^{\epsilon_3},(\xi,z),\pi,\phi,\psi_\R) = 1.\end{equation}
The remaining parts of Theorem \ref{thm:lir} follow as well from this. We may of course assume $\epsilon_3=1$, for otherwise the statement is trivial. Denote by $w$ the image of $a_3$ in any of the groups $W(\hat M,\hat G) \cong W(M^*,G^*) \cong W(M,G)$, where the isomorphisms are given by the parabolic pairs $(\hat M,\hat P)$, $(M^*,P^*)$, and $(M,P)$. The element $w$ belongs to the subgroup $W_\phi(M,G)$ of $W(\hat M,\hat G)$ and is thus $\Gamma$-fixed. According to \eqref{eq:iop} we are to show
\[ \pi(a_3)_{\xi,z} \circ l(w,\xi,\phi,\psi_F) \circ R_{w^{-1}Pw|P}(\xi,\phi) = 1. \]
For this we may choose an arbitrary $y \in G(\R)$ and an arbitrary smooth function $f \in \mc{H}_P(\pi)$ normalized so that $f(y)=1$. The above operator preserves the smooth vectors in this representation and sends $f$ to another smooth function. We must then show that
\[ [\pi(a_3)_{\xi,z} \circ l_P(w,\xi,\phi,\psi_F) \circ R_{w^{-1}Pw|P}(\xi,\phi)f](y) = 1. \]

We focus first on $\pi(a_3)_{\xi,z}$. The representation $\pi$ acts on the $1$-dimensional vector space $\C$, with $z \in \C^\times = \tx{Res}_{\C/\R}(\R)$ acting by the character $z^x\bar z^{-x}$, and $B \in U(2,0)(\R)$ acting trivially. The operator $\pi(w)$ constructed in Section \ref{sec:iop3u} is equal to the identity. On the other hand, the scalar $\<\pi,a_3\>_{\xi,z} \in \C^\times$ could be non-trivial. Indeed, it is constructed as the value $\<\pi,s(a_3)\>_{\xi,z}$, where $\<\pi,-\>_{\xi,z}$ is the character of $\pi_0(S_\phi(M))$ corresponding to $\pi$ \ref{thm:locclass-single} and $s(a_3)$ is an element of $\pi_0(S_\phi(M))$ whose construction we now review. Let $s'(w) \in N(\hat T,\hat G)^\Gamma$ be the Langlands-Shelstad lift of $w$, which in our case is
\[ s'(w) = \begin{bmatrix} &&&1 \\ &-1 \\ &&-1 \\ -1 \end{bmatrix}. \]
We choose $m \in \hat M_+$ such that $m\cdot a_3 \cdot s'(w) \in S_\phi(M)$ and then let $s(u)$ be the image of this element in $\pi_0(S_\phi(M))$. We have argued in Section \ref{sec:iop3u} that $s(u)$ is independent of the choice of $m$.
Comparing with the possible values of $a_3$ above we see that $s(a_3)$ is equal to the following
\[ x \in \Z:\begin{bmatrix}1\\ &-1 \\ && -1 \\ &&& 1 \end{bmatrix}, x = \frac{1}{2}:\begin{bmatrix}1\\ &1\\ &&-1 \\ &&&1 \end{bmatrix}, x=-\frac{1}{2}:\begin{bmatrix}1\\ &-1\\ &&1 \\ &&&1 \end{bmatrix} \]
and we have chosen $m$ to be the identity matrix in the first case and the diagonal matrix with entries $(1,1,1,-1)$ in the latter two cases. Working through the parameterization of discrete series $L$-packets of real groups described in \cite[\S5.6]{Kal13}, one sees that the character $\<\pi,-\>_{\xi,z} : \pi_0(S_\phi(M)) \rw \{ \pm 1\}$ is given by $a_1^{\epsilon_1}a_2^{\epsilon_2} \mapsto (-1)^{\epsilon_1}$. We arrive at the following table
\[ \begin{tabular}{|>{$}c<{$}||>{$}c<{$}|>{$}c<{$}|>{$}c<{$}|}
\hline
x&\Z&\frac{1}{2}&-\frac{1}{2}\\
\hline
\pi(a_3)_{\xi,z}&-1&1&-1\\
\hline
\end{tabular} \]
We note at this point, in reference to the remarks made in Section \ref{sec:iop3u}, that if we had used the naive splitting of the exact sequence \eqref{eq:esuni}, then the value of $\pi(a_3)_{\xi,z}$ would always be equal to $1$. As we shall see below, this would lead to a violation of \eqref{eq:lirexp1}.

We turn now to the scalar $[l(w,\xi,\phi,\psi_F)\circ R_{w^{-1}Pw|P}(\xi,\phi)f](y)$. It is defined by analytic continuation at $v=0$ of the family of scalars obtained by replacing $\phi$ by its twist $\phi_v$ for $v \in \mf{a}_{M,\C}^*$. We will evaluate it as follows. We have $X^*(M)^\Gamma \cong \Z$, with $n \in \Z$ acting as the character $(a,b,B)^n=(ab)^n$. Hence $\mf{a}_{M,\C}^* \cong \C$ and for $v \in \C$ the representation $\pi_v$ acts on the 1-dimensional vector space $\C$ by the character which sends $(a,\bar a,B) \in M(\R)$ to the scalar $\pi_v(a,\bar a,B) = a^x \bar a^{-x} (a\bar a)^v$. The representation $\mc{I}_P(\pi_v)$ acts on the space of functions
\[ \{f_v : G(\R) \rw \C| f_v(n\cdot m(a,\bar a,B)\cdot g) = a^x\bar a^{-x}(a\bar a)^{\frac{3}{2}+v}f(g) \} \]
whose restriction to a maximal compact subgroup $K \subset G(\R)$ is square\-integrable. Choosing a continuous in $v$ family of smooth functions $f_v \in \mc{H}(\pi_v)$ with $f_v(y)=1$, we seek to compute
\[ \lim_{v \rw 0} [l(w,\xi,\phi_v,\psi_F)\circ R_{w^{-1}Pw|P}(\xi,\phi_v)f_v](y). \]
For $v \in \R_{>0}$ the above scalar is equal to
\[ \lambda(w,\psi_\R)^{-1}\epsilon(0,\rho_{\bar P|P}^\vee\circ\phi_v,\psi_\R) \frac{L(1,\rho_{\bar P|P}^\vee\circ\phi_v)}{L(0,\rho_{\bar P|P}^\vee\circ\phi_v)} \int_{\bar N(\R)} f_v(n' \breve w^{-1}y) dn', \]
where we have written $\bar P$ for $w^{-1}Pw$, because it happens to be the parabolic subgroup that is $M$-opposite to $P$.

Let us first evaluate all constants in front of the integral. The relative roots of $A_{T^*}$ are all reduced, and the positive ones which $w$ maps to negative are precisely the one occurring in $\mf{n^*}$. There are three such, one contributing a 1-parameter group with simply connected cover $\tx{SL}_2$, and the other two contributing a 1-parameter group with simply connected cover $\tx{Res}_{\C/\R}\tx{SL}_2$. Thus $\lambda(w,\psi_\R) = \lambda(\C/\R,\psi_\R)^2=-1$. Next, the representation $\rho_{\bar P|P}^\vee\circ\phi_v$ decomposes as the direct sum
\[ \tx{Ind}_{\C^\times}^{W_\R}(z^{x+v-\frac{1}{2}}\bar z^{-x+v+\frac{1}{2}}) \oplus \tx{Ind}_{\C^\times}^{W_\R}(z^{x+v+\frac{1}{2}}\bar z^{-x+v-\frac{1}{2}}) \oplus
\tx{sgn}^{|2x|} \]
and according to \cite{LanRice} we have the following table:
\[ \begin{tabular}{|>{$}c<{$}||>{$}c<{$}||>{$}c<{$}|}
\hline
x&\lambda(w,\psi_\R)^{-1}\epsilon(0,\rho_{\bar P|P}^\vee\circ\phi_v,\psi_\R)& \frac{L(0,\rho_{\bar P|P}^\vee\circ\phi_v)}{L(1,\rho_{\bar P|P}^\vee\circ\phi_v)}\\
\hline
-2&+1&4\pi^\frac{5}{2}\Gamma(v)\Gamma(v+\frac{3}{2})\Gamma(v+\frac{1}{2})^{-1}\Gamma(v+\frac{7}{2})^{-1}\\
\hline
-1&+1&4\pi^\frac{5}{2}\Gamma(v)\Gamma(v+\frac{5}{2})^{-1}\\
\hline
-\frac{1}{2}&-i&4\pi^\frac{5}{2} \Gamma(v)\Gamma(v+\frac{1}{2})\Gamma(v+1)^{-1}\Gamma(v+2)^{-1}\\
\hline
0&-1&4\pi^\frac{5}{2}\Gamma(v)\Gamma(v+\frac{1}{2})\Gamma(v+\frac{3}{2})^{-2}\\
\hline
\frac{1}{2}&-i&4\pi^\frac{5}{2} \Gamma(v)\Gamma(v+\frac{1}{2})\Gamma(v+1)^{-1}\Gamma(v+2)^{-1}\\
\hline
1&+1&4\pi^\frac{5}{2}\Gamma(v)\Gamma(v+\frac{5}{2})^{-1}\\
\hline
2&+1&4\pi^\frac{5}{2}\Gamma(v)\Gamma(v+\frac{3}{2})\Gamma(v+\frac{1}{2})^{-1}\Gamma(v+\frac{7}{2})^{-1}\\
\hline
\end{tabular} \]

We now turn to the integral itself. One checks that we have isomorphisms of $\R$-vector spaces
\[ n : \C^2 \oplus \R \rw \mf{n}(\R),\qquad (Y_1,Y_2,Y_3) \mapsto \left\{\begin{bmatrix}&Y_1&Y_2&Y_3\\&&&i\bar Y_1\\&&&i \bar Y_2\\&&& \end{bmatrix}\right\} \]
and
\[ \bar n : \C^2 \oplus \R \rw \mf{\bar n}(\R),\qquad (X_1,X_2,X_3) \mapsto \left\{\begin{bmatrix}\\X_1&\\X_2&&\\X_3&-i\bar X_1&-i \bar X_2& \end{bmatrix}\right\}. \]
Let $\eta^*$ be the top form on $\mf{\bar n^*}$ whose value on the basis given by elementary matrices (all entries $0$ except a single entry equal to $1$) is equal to $1$. This form is defined over $\R$.
The action of $\tx{Ad}(\delta_{3,1})$ on $\mf{\bar n}^*$ has determinant $1$, so if $\eta$ denotes the transport of $\eta^*$ from $\mf{\bar n^*}$ to $\mf{\bar n}$ via $\xi=\tx{id}$, then $\eta$ is still defined over $\R$. Its pull-back via $\bar n$ sends the standard basis $((1,0,0),(i,0,0),(0,1,0),(0,i,0),(0,0,1))$ of $\C^2 \oplus \R$ to $-4$. We conclude that the measure $dn'$ is given by
\[ \int_{N'(\R)} \phi(n')dn' = 4\int_{\C \times \C \times \R} \phi(\exp(\bar n(X_1,X_2,X_3))) dX_1dX_2dX_3. \]
In order to fix a good function $f_v$, we use the relative Bruhat decomposition
\[ G(\R) = N(\R) \sqcup N(\R)M(\R) \breve w N(\R). \]
We have
\[ \breve w = \begin{bmatrix} &&&1 \\ &-1\\ &&-1 \\ -1 \end{bmatrix}. \]
We will define a function $f_v$ on $G(\R)$ which will be supported on the open set $N(\R)M(\R)\breve wN(\R)$ and will vanish rapidly towards its complement in $G(\R)$. To integrate this function over $\bar N(\R)$, we need to know how the element $\exp(\bar n(X_1,X_2,X_3)) \in \bar N(\R)$ decomposes as a product
\[ \exp(n(V_1,V_2,V_3))m(a,\bar a,B)\breve w\exp(n(Y_1,Y_2,Y_3)). \]
A direct computation reveals the following relations
\begin{eqnarray} \label{eq:lirexxy}
a&=&-(X_3+\frac{1}{2}i(X_1\bar X_1+X_2 \bar X_2))^{-1}\\
Y_1&=&i\bar a \bar X_1 \nonumber\\
Y_2&=&i\bar a \bar X_2 \nonumber\\
Y_3&=&a \bar a X_3 \nonumber
\end{eqnarray}
The following Lemma sheds some light on these formulas.
\begin{lem} \label{lem:lirexinvo} For $X_1,X_2 \in \C$, $X_3 \in \R$, define
\[ c(X_1,X_2,X_3)=X_3-i\frac{1}{2}(X_1\bar X_1+X_2\bar X_2). \]
Then then the map
\[ \C\oplus\C\oplus\R - \{0\} \rw \C\oplus\C\oplus\R - \{0\},\qquad (X_1,X_2,X_3) \mapsto (Y_1,Y_2,Y_3) \]
defined by
\[ Y_1=-i\bar X_1/c,\qquad Y_2=-i\bar X_2/c,\qquad Y_3=X_3/(c\bar c), \]
is a well-defined involution. Furthermore
\[ c(Y_1,Y_2,Y_3) = \ol{c(X_1,X_2,X_3)}^{-1}. \]
\end{lem}
The proof of this lemma is elementary and is left to the reader.

We can now define a good test function $f_v \in \mc{H}_P(\pi_v)$ as follows. It will be supported on the open set $N(\R)M(\R)wN(\R)$ of $G(\R)$ and will be given there by the formula
\[ f_v(\exp(n(V_1,V_2,V_3))m(a,\bar a,B)\breve w\exp(n(Y_1,Y_2,Y_3))) = a^x\bar a^{-x}(a\bar a)^{\frac{3}{2}+v}\cdot e^{-|c(Y_1,Y_2,Y_3)|}. \]
The uniqueness of the decomposition implies that the formula on the right provides a well-defined function on the given open set. It is smooth and satisfies $f_v(\breve w)=1$. The argument of the function approaches $N(\R)$ precisely when $|c(Y_1,Y_2,Y_3)|$ approaches $\infty$, and then the value of the function approaches $0$ very fast. This allows us to extend $f_v$ smoothly to $G(\R)$ by setting it to be equal to $0$ on $N(\R)$. Upon restriction to $K$ this function attains a maximum which depends continuously on the parameter $v$, so we obtain a continuous family $f_v|_K$. This family is in particular continuous at $v=0$.

We can now evaluate the integral
\begin{equation} \label{eq:lirexint}
\int_{\bar N(\R)} f_v(n' \breve w^{-1}y) dn'
\end{equation}
at $y = \breve w$. It is given by
\[ \int_{\bar N(\R)} f_v(n')dn' = 4\int_{\C \times \C \times \R} f_v(\exp(\bar n(X_1,X_2,X_3)))dX_1dX_2dX_3. \]
An elementary manipulation involving \eqref{eq:lirexxy} and Lemma \ref{lem:lirexinvo} shows that it is equal to
\[ 4(-1)^{2x}\int_{\C\times\C\times\R} c^{2x} \frac{1}{|c|^{3+2x+2v}}e^{-\frac{1}{|c|}}dX_1dX_2dX_3, \]
where we have abbreviated $c=c(X_1,X_2,X_3)$. We write $X_1=x_1+ix_2$, $X_2=x_3+ix_4$, $X_3=x_5$, with $(x_1,\dots,x_5) \in \R^5$, and then we have
$c = x_5-\frac{1}{2}i(x_1^2+x_2^2+x_3^2+x_4^2)$. We replace $(x_1,\dots,x_4)$ with hyperspherical coordinates
\begin{eqnarray*}
x_1&=&r\cos(\phi_1),\\
x_2&=&r\sin(\phi_1)\cos(\phi_2),\\
x_3&=&r\sin(\phi_1)\sin(\phi_2)\cos(\phi_3),\\
x_4&=&r\sin(\phi_1)\sin(\phi_2)\sin(\phi_3),
\end{eqnarray*}
and rename $x_5=a$, thereby obtaining $c=a-\frac{1}{2}ir^2$. The integral now has the form
\[ 4C(-1)^{2x}\int_{a=-\infty}^{a=+\infty}\int_{r=0}^{r=\infty}r^3\frac{(a-\frac{1}{2}ir^2)^{2x}}{(a^2+\frac{1}{4}r^4)^{\frac{3}{2}+x+v}}
\exp(-(a^2+\frac{1}{4}r^4)^\frac{-1}{2})dadr, \]
where $r^3$ comes from the Jacobian of the hypershperical transformation and $C$ is the positive constant
\[ C = 2\pi\int_0^\pi\int_0^\pi \sin^2(\phi_1)\sin(\phi_2)d\phi_1d\phi_2 = 2\pi^2. \]
We substitute $r$ by $\sqrt{2}r$, split the domain of integration along $a$ into $(-\infty,0)$ and $(0,\infty)$ and obtain
\[ \int_{\bar N(\R)} f_v(n') dn' = I(x,v)+(-1)^{2x}\ol{I(x,v)} \]
where
\[ I(x,v) = 16C(-1)^{2x}\int_{a=0}^{a=+\infty}\int_{r=0}^{r=\infty}r^3\frac{(a-ir^2)^{2x}}{(a^2+r^4)^{\frac{3}{2}+x+v}}
\exp(-(a^2+r^4)^\frac{-1}{2})dadr.   \]
To evaluate $I(x,v)$, we make the substitutions $r \mapsto r^4$, $r \mapsto a^2r$ and $a \mapsto a\sqrt{1+r}$. This gives
\[ I(x,v) = 4C(-1)^{2x}\Gamma(2v)\int_0^\infty \frac{(1-ir^\frac{1}{2})^{2x}}{\sqrt{1+r}^{3+2x}} dr. \]
We arrive at the formula
\[ \int_{\bar N(\R)} f_v(n') dn' = 8\pi^2(-1)^{2x}\Gamma(2v)\int_0^\infty \frac{(1-ir^\frac{1}{2})^{2x}+(-1-ir^\frac{1}{2})^{2x}}{\sqrt{1+r}^{3+2x}} dr \]
whose evaluation leads to the following table
\[ \begin{tabular}{|>{$}c<{$}||>{$}c<{$}|}
\hline
x&\int_{\bar N(\R)} f_v(n') dn'\\
\hline
-2&8\pi^2\Gamma(2v)\cdot (-\frac{4}{15})\\
-1&8\pi^2\Gamma(2v)\cdot (-\frac{4}{3})\\
-\frac{1}{2}&8\pi^3\Gamma(2v)\cdot (-i)\\
0&8\pi^2\Gamma(2v)\cdot 4\\
\frac{1}{2}&8\pi^3\Gamma(2v)\cdot i\\
1&8\pi^2\Gamma(2v)\cdot (-\frac{4}{3})\\
2&8\pi^2\Gamma(2v)\cdot (-\frac{4}{15})\\
\hline
\end{tabular} \]
Combining these values with those for the normalizing factors and taking the limit as $v$ is a positive real number approaching $0$, we obtain the following table
\[ \begin{tabular}{|>{$}c<{$}||>{$}c<{$}|>{$}c<{$}|>{$}c<{$}|>{$}c<{$}|>{$}c<{$}|>{$}c<{$}|>{$}c<{$}|}
\hline
x&-2&-1&-\frac{1}{2}&0&\frac{1}{2}&1&2\\
\hline
\lim_{v \rw 0} [l(w,\xi,\phi_v,\psi_F)\circ R_{w^{-1}Pw|P}(\xi,\phi_v)f_v](\breve w)&-1&-1&-1&-1&1&-1&-1\\
\hline
\end{tabular} \]
This matches the table of values for $\pi(a_3)_{\xi,z}$ and the proof is complete.

\section{Chapter 3: The trace formula}\label{chapter3}

\subsection{The discrete part of the untwisted trace formula}\label{sub:I_disc} %

  Our first task is to recall the definition of the discrete part of the Arthur-Selberg trace formula in the untwisted case, given in \eqref{eq:I_disc-defn} below. Only in this subsection we will treat arbitrary reductive groups, and then return to unitary groups. Let us start by setting up some notation. Our explanation will be brief; the reader is referred to section 3.1 of \cite{Arthur} for complete details.

   Let $G$ be a connected reductive group over a number field $F$. Fix a maximal compact subgroup $K\subset G(\A)$ and a minimal Levi subgroup $M_0\subset G$ in a good position relative to $K$. Write $\cL=\cL(M_0)$ for the finite set of Levi subgroups of $G$ containing $M_0$. Define $\fka_G:=\Hom_\Z(X(G)_F,\R)$ where $X(G)_F$ is the group of $F$-rational characters on $G$, and similarly $\fka_M$ for $M\in \cL(M_0)$. Put $\fka^G_M$ to be the canonical complement of $\fka_G$ in $\fka_M$. The maximal $F$-split torus in $Z(M)$ is denoted $A_M$. In the relative Weyl group $W^G(M) :=N_G(A_M)/M$, introduce the subset of regular elements
  $$W^G(M)_{\reg}:=\{w\in W^G(M): \det(w-1)_{\fka^G_M}\neq 0\}.$$
  We agree to write $W^M_0$ and $W^G_0$ for $W^M(M_0)$ and $W^G(M_0)$, respectively.

  We write $\A=\A_F$ for the adele ring over $F$. The subgroup $G(\A)^1\subset G(\A)$ is defined as usual. Let $\fkX_G$ be a closed subgroup of $Z(G(\A))$ such that $\fkX_GZ(G(F))$ is closed and cocompact in $Z(G(\A))$, and fix a character $\omega:\fkX_G\ra \C^\times$ trivial on $Z(G(F))\cap \fkX_G$. The pair $(\fkX_G,\omega)$ is called a central character datum for $G$. Let $P=N_PM$ be an $F$-rational parabolic subgroup of $G$ with Levi component $M\in \cL(M_0)$ and unipotent radical $N_P$. The datum $(\fkX_G,\omega)$ gives rise to a central character datum $(\fkX_M,\omega)$ for $M$, where $\omega$ in the latter designates by abuse of notation a pullback of $\omega$ in the former via $\fkX_M\ra \fkX_G$. Let $L^2_{\disc}(M(F)\bs M(\A),\omega)$ denote the space of functions on $M(\A)$ which are $\omega$-equivariant under $\fkX_M$ and square-integrable modulo $\fkX_M$. Its parabolic induction from $P$ to $G$ is written as $\cI_P(\omega)=\cI^G_P(\omega)$ with underlying Hilbert space $\cH_P(\omega)$. Let us fix a Weyl-invariant Hermitian metric $\|\cdot \|$ on the dual of the Cartan subalgebra $\fkh$ in the Lie algebra of $(\Res_{F/\Q}G)(\C)$. Given an irreducible representation $\pi$ of $G(\A)$, the infinitesimal character of its archimedean component gives a linear form $\mu_{\pi}:\fkh\ra \C$ with imaginary part $\Im(\mu_{\pi})$. This leads to a decomposition
  $$\cI_P(\omega)=\bigoplus_{t\ge 0} \cI_{P,t}(\omega),\quad \cH_P(\omega)=\bigoplus_{t\ge 0} \cH_{P,t}(\omega)$$
  where $\cI_{P,t}(\omega)$ is the direct summand whose irreducible constituents $\pi$ satisfy $\| \Im(\mu_{\pi})\|=t$. Let $\cH(G,\omega)$ be the Hecke algebra of smooth $K$-finite complex-valued functions on $G(\A)$ which are $\omega^{-1}$-equivariant under $\fkX_G$ and compactly supported modulo $\fkX_G$. For $f\in \cH(G,\omega)$ we define $\cI_{P,t}(\omega,f):\cH_{P,t}\ra \cH_{P,t}$ by
  $$\cI_{P,t}(\omega,f):=\int_{G(\A)/\fkX_G} f(x)\cI_{P,t}(\omega,x)dx.$$
  The intertwining operator $M_P(w,\omega):\cH_P(\omega)\ra \cH_P(\omega)$ for $w\in W^G(M)_{\reg}$ is defined via meromorphic continuation of the standard intertwining integral. It stabilizes $(\cI_{P,t}(\omega),\cH_{P,t}(\omega))$ for each $t\ge 0$, and the induced operator on the subspace is denoted $M_{P,t}(w,\omega)$.

  The discrete part of the trace formula $I^G_{\disc,t}(f)$ for $f\in \cH(G,\omega)$ is defined as

\begin{equation}\label{eq:I_disc-defn} %
  I^G_{\disc,t}(f):=\sum_{M\in \cL} \frac{|W^M_0|}{|W^G_0|}\sum_{w\in W^G(M)_{\reg}} |\det(w-1)_{\fka^G_M}|^{-1} \tr(M_{P,t}(w,\omega)\cI_{P,t}(\omega,f)).
\end{equation}
  The convergence of this linear form for every $f$ is a result of M\"uller (\cite{Mul89}).

  Now we describe the stabilization of the expression \eqref{eq:I_disc-defn} as envisioned by Langlands, Shelstad and Kottwitz and established by Arthur (\cite{ArtSTF1}, \cite{ArtSTF2}, \cite{ArtSTF3}). At the time Arthur's result had relied on the validity of the ordinary and weighted fundamental lemmas, which were verified recently by Ng\^o, Waldspurger and Chaudouard-Laumon (\cite{Ngo10}, \cite{Wal06}, \cite{Wal09}, \cite{CLWFL1}, \cite{CLWFL2}).\footnote{At the time of writing, Chaudouard and Laumon have not completed their series of papers on the proof of the weighted fundamental lemma in positive characteristic.} We remark that the fundamental lemma for unitary groups was proven earlier by Laumon and Ng\^o (\cite{LN08}).

  Let us introduce some notation.
  Write $\ol\cE_{\el}(G)$ (resp. $\ol\cE(G)$) for the set of isomorphism classes of elliptic (resp. all) endoscopic data $(G^\fke,\cG^\fke,s^\fke,\eta^\fke)$ over $F$.\footnote{Arthur restricts the definition to those elliptic endoscopic data such that for every place $v$ of $F$, $G^\fke(F_v)$ contains an element that is a norm from $G(F_v)$ in the sense of \cite[p.29]{KS99} but we need not do it.} In general $\ol\cE_{\el}(G)$ can be infinite.

  We may and will assume for simplicity that $G$ has simply connected derived subgroup for the rest of this chapter. There is no loss of generality for our purpose since unitary groups (quasi-split or not) do have this property. Under the assumption, we may find a representative $\fke=(G^\fke,\cG^\fke,s^\fke,\eta^\fke)$ in each isomorphism class of endoscopic data such that $\cG^\fke={}^L G^\fke$ and an $L$-embedding $\eta^\fke:{}^L G^\fke \hra {}^L G$, cf. \cite{Lan79}. (In general $\cG^\fke$ is a split extension of $W_F$ by $\hat{G}^\fke$ which may not be isomorphic to ${}^L G^\fke$.) By abuse of notation we often write $G^\fke$ rather than $\fke$ to denote a member of $\ol\cE_{\el}(G)$.

  To each $G^\fke\in \ol\cE_{\el}(G)$ is associated a nonzero rational number (cf. \cite[(3.2.4)]{Arthur} but note that $|\pi_0(\kappa_G)|=1$ in the untwisted trace formula)
  \begin{equation}\label{eq:iota(G,G')}
    \iota(G,G^\fke):=k(G,G^\fke)|\ol{Z}(\hat{G}^\fke)^{\Gamma}|^{-1}|\Out_G(G^\fke)|^{-1},
    \end{equation}
    where $k(G,G^\fke):=|\ker^1(F,Z(\hat{G}^0))|^{-1}|\ker^1(F,Z(\hat{G}^\fke))|$ but we will see that $k(G,G^\fke)=1$ for the groups this paper is mainly concerned with. As explained in section 3.2 of \cite{Arthur} (our case is simpler since $\tilde{G}^\fke=G^\fke$), the central character datum $(\fkX_G,\omega)$ for $G$ gives rise to the analogous datum $(\fkX_{G^\fke},\omega^\fke)$ for $G^\fke$. Moreover there is a transfer mapping   $$\cH(G,\omega)\ra \cS(G^\fke,\omega^\fke),\quad f\mapsto f^\fke=f^{G^\fke},$$
  characterized by an identity of orbital integrals and conjectured to exist by Langlands and Shelstad. The existence is now a consequence of Ng\^o's proof of the fundamental lemma for Lie algebras in positive characteristic via Waldspurger's results (\cite{Wal97}, \cite{Wal06}) on the change of characteristic, the reduction of the (original) fundamental lemma to the Lie algebra version, and the assertion that the fundamental lemma implies the transfer conjecture. More generally Arthur conjectured that the weighted fundamental lemma as he formulated is true. This is now a theorem by Chaudouard-Laumon's proof in positive characteristic combined with another result of Waldspurger (\cite{Wal09}) on going from positive characteristic to characteristic zero. As long as we accept Arthur's stabilization of the trace formula, we need not and will not recall here the precise statements of the fundamental lemma, transfer mapping, and their weighted variants.

  The stabilization is a decomposition

\begin{equation}\label{eq:stabilization} %
   I^G_{\disc,t}(f)=  \sum_{G^\fke\in \ol\cE_{\el}(G)} \iota(G,G^\fke)\hat{S}^\fke_{\disc,t}(f^\fke), \quad f\in \cH(G,\omega)
\end{equation}
where $\hat{S}^\fke_{\disc,t}=\hat{S}^{G^\fke}_{\disc,t}:\cH(G^\fke,\omega^\fke)\ra \C$ is a stable linear form (i.e. linear form factoring through $\cS(G^\fke,\omega^\fke)$) depending only on $G^\fke$ and not on $G$. Note that for any given $f$, the sum has finitely many nonzero summands (even if $\ol\cE_{\el}(G)$ is infinite). When $G$ is quasi-split, the main point of the stabilization is that
\begin{equation}\label{eq:define-S-form}
  I^G_{\disc,t} -  \sum_{G^\fke\in\ol \cE_{\el}(G)\smallsetminus\{G\}} \iota(G,G^\fke)\hat{S}^\fke_{\disc,t}(f^\fke)
\end{equation}
is a stable linear form on $\cH(G,\omega)$, which is then taken as the definition of the stable linear form $\hat{S}_{\disc,t}(f)$. In the non quasi-split case, which is of main interest in this paper, the two sides of \eqref{eq:stabilization} are defined independently. The assertion of \eqref{eq:stabilization} is that the two are equal.

  Let $(\fkX_G,\omega)$ be a central character datum for $G$ as above. Let $S$ be a finite set of places of $F$ outside of which $G$ is unramified. Let $\cC^S_{\A}(G)$ be the set consisting of families of semisimple conjugacy classes $\{c_v:v\notin S\}$ in ${}^L G$ such that the image of each $c_v$ under the natural projections ${}^L G\ra W_{F_v}\ra W_{F_v}/I_{F_v}$ is the Frobenius element. Recall that $\fkX_G$ is a subgroup of $Z(G(\A))$. Each $c_v$ corresponds via the Satake transform to an irreducible unramified representation of $G(F_v)$. Let $\zeta_v$ denote the restriction to $Z(G(F_v))$ of the central character of the latter representation. By definition the subset $\cC^S_{\A}(G,\omega)$ comprises $\{c_v:v\notin S\}$ with the property that there exists an extension of $\omega$ to $Z(G(\A))$ whose $v$-component is $\zeta_v$ at $v\notin S$. Define
  $$\cC_{\A}(G,\omega):=\dirlim{S} \cC^S_{\A}(G,\omega).$$
  Then we have a decomposition of the automorphic spectrum
  $$L^2_\disc(G(F)\bs G(\A),\omega)=\bigoplus_{c\in \cC_{\A}(G,\omega)\atop t\ge 0} L^2_{\disc,t,c}(G(F)\bs G(\A),\omega)$$
  such that $L^2_{\disc,t,c}(G(F)\bs G(\A),\omega)$ is the direct sum of $\pi$, where the central character of $\pi$ on $\mathfrak{X}_G$ is $\omega$, $\|\Im(\mu_\pi)\|=t$, and $c_v$ corresponds to $\pi_v$ via the Satake transform away from a sufficiently large finite set $S$.
  Hence the regular representation $R_{\disc}$ acting on the left hand side also decomposes as the sum of $R_{\disc,\psi}$. In particular
  $$\tr R_{\disc}(f) = \sum_{c\in \cC_{\A}(G,\omega)\atop t\ge 0} \tr R_{\disc,t,c}(f),\qquad f\in \cH(G).$$

  For an endoscopic datum $G^\fke\in \ol\cE(G)$, let $(\fkX_{G^\fke},\omega^\fke)$ be as above. Then the $L$-morphism $\eta^\fke:{}^L G^\fke\ra {}^L G$ induces a transfer mapping $\cC_{\A}(G^\fke,\omega^\fke)\ra \cC_{\A}(G,\omega)$. For a quasi-split group $G^*$ and $c\in \cC_{\A}(G^*,\omega)$ we define a stable linear form $S^{G^*}_{\disc,t,c}$ exactly as in Lemma 3.3.1 of \cite{Arthur}. (The basic idea is to make sense of the $c$-part of \eqref{eq:define-S-form} and take it as the definition.) Now let $c\in \cC_{\A}(G,\omega)$ and allow $G$ to be any connected reductive group. In the same lemma Arthur shows the decomposition
\begin{equation}\label{eq:stabilization-t,c} %
   I^G_{\disc,t,c}(f)=  \sum_{G^\fke\in \ol \cE_{\el}(G)} \iota(G,G^\fke)\hat{S}^\fke_{\disc,t,c}(f^\fke), \quad f\in \cH(G,\omega).
\end{equation}
  Here the Weyl-invariant metric $\|\cdot\|$ on the linear dual of $\fkh$ induces analogous metrics for endoscopic groups as explained in \cite[3.2]{Arthur}.

\subsection{The vanishing of coefficients}\label{sub:vashing-coefficients}

  We continue to be in the general setup of \S\ref{sub:I_disc}. We prove an analogue of the corollary 3.5.3 of \cite{Arthur} on the vanishing of coefficients in a certain linear relation involving $G$ (and its Levi subgroups). Later the result will be applied to $G$ which is an inner form of a unitary group. Our case is simpler than Arthur's in that in his situation there may appear several groups which are twisted endoscopic groups of a general linear group simultaneously. %

  So far our consideration has been global but let us introduce some local definitions and notation following \cite{Arthur}. Temporarily let $F$ be a local field and $G$ be a connected reductive group over $F$ with a minimal Levi subgroup $M_0$. We define $\cL=\cL^G(M_0)$ and $W^G_0$ as in the global case. For $M\in \cL$, denote by $\Pi_2(M)$ the set of isomorphism classes of square-integrable representations of $M(F)$. The $R$-group of $\sigma\in \Pi_2(M)$ in $G$ is written as $R(\sigma)$. After choosing a finite central extension
  $$1\ra Z_\sigma \ra \tilde R(\sigma) \ra R(\sigma) \ra 1$$
  it is possible to define a homomorphism $r\mapsto R_P(r,\sigma)$ from $\tilde R(\sigma)$ to $\End_G(\cI_P(\sigma))$ such that $R_P(zr,\sigma)=\omega_\sigma(z)^{-1}R_P(r,\sigma)$ where $\omega_\sigma:Z_\sigma\ra \C^\times$ is a fixed character.
  Define $T(G)$ (resp. $\tilde T(G)$) to be the set of $W^G_0$-orbits of triples
  $$ \tau=\tau_r=(M,\sigma,r)$$
  where $M\in \cL$, $\sigma\in \Pi_2(M)$, and $r\in R(\sigma)$ (resp. $r\in \tilde R(\sigma)$). For $\tau\in \tilde T(G)$ as above, put
  \begin{equation}\label{eq:f_G(tau)}
    f_{G}(\tau)=f_{G}(\tau_r):=\tr(R_P(r,\sigma)\cI_P(\sigma,f)),\quad f\in \cH(G).
  \end{equation}
  If $M=G$ then $R(\sigma)$ is trivial and $\tau$ is identified with an element of $\Pi_2(G)$. So the notation of \eqref{eq:f_G(tau)} is consistent with our convention to write $f_{G}(\pi):=\tr \pi(f_G)$ for any irreducible admissible representation $\pi$ of $G(F)$.
   For $\pi\in \Pi(G)$, one has a ``change of basis''
   $$f_G(\pi)=\sum_{\tau\in T(G)}n(\pi,\tau)f_G(\tau),\quad f\in \cH(G),$$
   cf. \cite[(3.5.6)]{Arthur}. Even though $n(\pi,\tau)$ and $f_G(\tau)$ depend on the choice of a lift of $\tau$ to $\tilde T(G)$, their product depends only on $\tau$ since $n(\pi,z\tau)=\omega_\sigma(z)n(\pi,\tau)$ (following from the definition in \S3.5 of \cite{Arthur}) and $f_G(z\tau)=\omega_\sigma^{-1}(z)f_G(\tau)$.

   In this paragraph and the following lemma we recall some notation and fact from the section 3.5 of \cite{Arthur}, omitting details. Fix a minimal parabolic subgroup $P_0$ whose Levi factor is $M_0$. Let $\fka_0=\fka_{M_0}$ be as in \S\ref{sub:I_disc}, and write $(\ol{\fka^*_0})^+$ for the closed dual chamber equipped with partial ordering $\le$ corresponding to $P_0$. Let $\pi\in \Pi(G)$ and $\tau\in T_{\el}(M)$ for a Levi subgroup $M$ of $G$. Arthur defines linear forms $\Lambda_\pi$ and $\Lambda_\tau$ which belong to $(\ol{\fka^*_0})^+$ and whose deviation from $0$ measures the non-temperedness of the representation, loosely speaking. To $\pi$ is associated a standard representation which is induced from a character twist of a discrete series $\sigma_\pi$ on a Levi subgroup $M_\pi$ of $G$. (This data is well defined up to conjugacy.) Let $\tau_\pi$ denote the element in $T(G)$ or $\tilde T(G)$ given by $(M_\pi,\sigma_\pi,1)$. The following result serves as a key ingredient in the desired vanishing result. See the discussion between the statements of the proposition 3.5.1 and the lemma 3.5.2 in \cite{Arthur}.

  \begin{lem}\label{lem:n(pi,tau)} Use the local notation as above.
    \begin{enumerate}
      \item $\Lambda_{\tau_\pi}=\Lambda_\pi$ and $n(\pi,\tau_\pi)>0$.
      \item Let $\pi$ and $\tau$ be as above and satisfy that $n(\pi,\tau)\neq 0$. Then $\Lambda_\tau\le \Lambda_\pi$. If $\Lambda_\tau=\Lambda_\pi$ then $M_\tau$ contains (a suitable $W^G_0$-translate of) $M_\pi$. If $\Lambda_\tau=\Lambda_\pi$ and $M_\tau=M_\pi$ then $\tau=\tau_\pi$. and $n(\tau,\pi)>0$.
    \end{enumerate}

  \end{lem}

  Now revert back to the global case. The notation in the preceding paragraph over $F_v$ will be written with subscript $v$ suitably inserted.

\begin{lem}\label{lem:vanishing-coeff} Let $c_G(\pi)\in \R_{\ge 0}$ for each $\pi\in \Pi(G)$.
  Suppose that
$$\sum_{\pi\in \Pi(G)} c_G(\pi)f_G(\pi)=\sum_{\tau_v\in T(G_v)} d(\tau_v,f^v)f_{v,G}(\tau_v),\quad f=f^v f_v\in \cH(G)$$
where the coefficient $d(\tau_v,f^v)\in \C$, as a function of $\tau_v$, is supported on a finite set that depends only on a choice of Hecke type for $f_v$, and equal to 0 for any $\tau_v$ of the form $(M_v,\sigma_v,1)$. Then for all  $\pi\in \Pi(G)$, $\tau_v\in T(G_v)$, and $f^v\in \cH(G(\A^v))$,
$$c_G(\pi)=0,\quad d(\tau_v,f^v)=0.$$
\end{lem}

\begin{proof}

  Assume that $c_G(\pi)$ are all zero. Arthur's version of the trace Paley-Wiener theorem (\cite[\S4]{ArtLCR}) then allows us to find an $f_v$ such that $f_{v,G}$ is nonzero at one $\tau_v$ and zero at all the other elements of $T(G_v)$. So all $d(\tau_v,f^v)$ must vanish.

  So it suffices to show that $c_G(\pi)$ are all zero. The argument proceeds as in the proof of \cite[Cor 3.5.3]{Arthur}, with simplifications thanks to the fact that there are no groups other than $G$ involved. Fix a finite set of places $S$ containing $v$ and all infinite places of $F$. Fix a Hecke type $(S,\{(\tau_\infty,\kappa^\infty)\})$ for $f$ in the sense of Arthur. This means that $f=f^\infty f_\infty$, $\kappa^\infty$ is an open compact subgroup of $G(\A^\infty)$ which is hyperspecial away from $S$, $\tau_\infty$ is a finite set of irreducible representations of a maximal compact subgroup $K_\infty$ of $G(\A_\infty)$ such that $f^\infty$ is bi-invariant under $\kappa^\infty$ and $f_\infty$ transforms under $K_\infty$ on the left and right according to representations in $\tau_\infty$. Define
$$
    c_{G_S}(\pi_S):=\sum_{\pi'} c_G(\pi'),
$$
  where $\pi'$ runs over the finite set of irreducible representations of $G(\A)$ which satisfy that $\pi'_S\simeq \pi_S$ and have nonzero trace against some function of the given Hecke type (in particular $\pi$ are unramified outside $S$). Clearly it is enough to check that
    \begin{equation}\label{eq:c(pi)=0} c_{G_S}(\pi_S)=0.  \end{equation}
  Following the procedure in the proof of the proposition 3.5.1 (or the corollary 3.5.3) in \cite{Arthur}, we can rewrite the left hand side as
  $$\sum_{M_S} \sum_{\tau_S\in T_{\el}(M_S)} \sum_{\pi_S\in\Pi(G_S)} \frac{|W_0^G(\tau_S)|}{|W_0^G|} c_{G_S}(\pi_S) n(\pi_S,\tau_S)f_{M_S}(\tau_S),$$
  where the first sum runs over the Levi subgroups of $G_S=G(F_S)$ containing the minimal Levi subgroup. By writing subscript $S$ in the notation, as usual, we have naturally extended the definition of a local object at a single place to an analogous object at the set $S$ of finitely many places. Recall that $\cI(G_S)=\bigoplus_{M_S} \cI_{\cusp}(M_S)^{W(M_S)}$, cf. \cite[\S6]{ArtLCR}. Combining this with Arthur's trace Paley-Wiener theorem, as used earlier, we may find an element $f_S=f_vf^v_S$ such that the summand above does not vanish for exactly one $M_S$ and exactly one $\tau_S\in T_{\el}(M_S)$ for any given $M_S$ and $\tau_S$. If $\tau_S$ has $v$-component of the form $(M_v,\sigma_v,1)$ then the initial assumption implies that the right hand side vanishes. Hence
  \begin{equation}\label{eq:tau=(M,sigma,1)-vanishing}
    \sum_{\pi_S\in\Pi(G_S)}  c_{G_S}(\pi_S) n(\pi_S,\tau_S)=0,\quad \mbox{if}~\tau_v~\mbox{has form}~(M_v,\sigma_v,1).
  \end{equation}

    Our plan is to prove \eqref{eq:c(pi)=0} by contradicting \eqref{eq:tau=(M,sigma,1)-vanishing}. To this end, assume the existence of $\pi$ such that $c_{G_S}(\pi_S)\neq 0$. Introduce a $\tau$-equivalence relation $\sim$ on $\Pi(G_S)$ such that $\pi_S\sim \pi'_S$ if and only if there exists a pair $(M_S,\tau_S)$ such that $n(\pi_S,\tau_S)\neq0 $ and $n(\pi'_S,\tau_S)\neq 0$. By assumption there exists a $\tau$-equivalence class $\mC_S$ such that
    $$\mC'_S:=\{\pi_S\in \mC_S: c_{G_S}(\pi_S)\neq 0\}$$
    is non-empty. Equip $\R_{\ge 0}\times \Z_{\ge 0}$ with a partial ordering $\preceq$ such that $(\lambda_1,\mu_1)\preceq (\lambda_2,\mu_2)$ if and only if either $\lambda_1\le \lambda_2$ or $\lambda_1=\lambda_2$ and $\mu_1\ge \mu_2$. There is a map
    $$\mC'_S\ra \R_{\ge 0}\times \Z_{\ge 0},\quad \pi_S~\mapsto~ (\|\Lambda_{\pi_S}\|,\dim M_{\pi_S}),$$
    where $\|\cdot\|$ is the Hermitian norm on $\fka_0^*$ as in \S\ref{sub:I_disc}. Choose a maximal element $(\lambda',\mu')$ in the image of $\mC'_S$ and also a $\pi'_S\in \mC'_S$ in the preimage of $(\lambda',\mu')$. For every $\pi_S\in \Pi(G_S)$ such that $c_{G_S}(\pi_S)n(\pi_S,\tau_{\pi'_S})\neq 0$, we have $\pi_S\in \mC'_S$ by definition, so $$\|\Lambda_{\tau_{\pi'_S}}\|\le \|\Lambda_{\pi_S}\|,\quad  M_{\tau_{\pi'_S}} \supset M_{\pi_S} $$
    by Lemma \ref{lem:n(pi,tau)}. But the maximality of $\pi'_S$ implies that $\|\Lambda_{\tau_{\pi'_S}}\|= \|\Lambda_{\pi_S}\|$ and that $M_{\tau_{\pi'_S}}= M_{\pi_S}$. Again by Lemma \ref{lem:n(pi,tau)}, we see that $\tau_{\pi'_S}=\tau_{\pi_S}$ and that $n(\pi_S,\tau_{\pi_S})>0$.
    Finally let us apply \eqref{eq:tau=(M,sigma,1)-vanishing} to $\tau_{\pi'_S}$, which has the required form by definition. We have just seen that every nonzero summand has to be positive, and clearly $\pi'_S$ provides a nonzero summand (by the same lemma). Thus we are led to contradiction, proving that every $c_{G_S}(\pi_S)$ must vanish for each fixed Hecke type.

\end{proof}

\subsection{Stable multiplicity formula for unitary groups}\label{sub:stable-multiplicity}

  From here on we restrict ourselves to unitary groups. Our goal is to state the stable multiplicity formula for them after discussing the decomposition of the trace formula according to parameters and introducing the yet undefined players in the formula.

  Let $G^*=U_{E/F}(N)$ be the quasi-split unitary group in $N$ variables associated with a quadratic extension $E/F$ of number fields. \emph{Fix a character $\chi\in \mathcal{Z}_E$ once and for all.} Denote by $\eta_{\chi}: {}^L G^* \ra {}^L \GL(N)$ the associated $L$-morphism so that $(G^*,\eta_\chi)$ constitutes a twisted endoscopic datum for $G_{E/F}(N)=\Res_{E/F} \GL(N)$ with respect to a unitary involution. %

  Let $(G,\xi,z)$ be an extended pure inner twist data for $G^*$.
  We would like to have the analogue \eqref{eq:stabilization-t,c} with parameters in place of $(t,c)$.
  Let $\psi^N\in \tilde\Psi(N)$ and $G^\fke\in \ol\cE(G)$. To the former is associated $t(\psi^N)$ and $c(\psi^N)$. For $L$-group embeddings $\eta_\chi:{}^L G\hra {}^L \GL(N)$ and $\eta^\fke_\chi:{}^L G^\fke\hra {}^L \GL(N)$, define
  $$I^G_{\disc,\psi^N,\eta_\chi}:=\sum_{c\mapsto c(\psi^N)} I^G_{\disc,t(\psi^N),c}$$
   $$S^{G^\fke}_{\disc,\psi^N,\eta^\fke_\chi}:=\sum_{c\mapsto c(\psi^N)} S^{G^\fke}_{\disc,t(\psi^N),c}$$
  where the first (resp. second) sum runs over $c\in \cC_{\A}(G)$ which maps to $c(\psi^N)$ via $\eta_\chi$ (resp. $\eta^\fke_\chi$).

  We are ready to refine or regroup the two expansions \eqref{eq:I_disc-defn} and \eqref{eq:stabilization-t,c} according to $\psi^N$ and $\eta_\chi$. We are taking $\mathfrak {X}_G=\{1\}$ and $\chi=1$ in \S\ref{sub:I_disc}. The same procedure as above produces a direct summand $\cI_{P,\psi^N,\eta_\chi}(1)$ of the induced representation $\cI_{P,t}(1)$ and accordingly $ M_{P,\psi^N,\eta_\chi}(w,1)$ acting on $\cI_{P,\psi^N,\eta_\chi}(1)$. From now on we omit $1$ from the notation in favor of simplicity. We arrive at the following refinement of \eqref{eq:I_disc-defn}:
\begin{equation}\label{eq:I_disc-psi-spec} %
  I^G_{\disc,\psi^N,\eta_\chi}(f):=\sum_{M\in \cL} \frac{|W^M_0|}{|W^G_0|}\sum_{w\in W^G(M)_{\reg}} |\det(w-1)_{\fka^G_M}|^{-1} \tr(M_{P,\psi^N,\eta_\chi}(w)\cI_{P,\psi^N,\eta_\chi}(f)).
\end{equation}
   Taking the sum of \eqref{eq:stabilization-t,c} over $(t,c)$ such that $c$ maps to $c(\psi^N)$, we obtain
\begin{equation}\label{eq:I_disc-stabilized} %
   I^G_{\disc,\psi^N,\eta_\chi}=  \sum_{(G^\fke,\zeta^\fke)\in \ol\cE_{\el}(G)} \iota(G,G^\fke)\hat{S}^\fke_{\disc,\psi^N,\eta_\chi}(f^\fke)
\end{equation}
  Analogously we define $L^2_{\disc,\psi^N,\eta_\chi}(G(F)\bs G(\A_F))$ and the regular representation $R_{\disc,\psi^N,\eta_\chi}$ on it.
When $\psi=(\psi^N,\tilde\psi)\in \Psi(G^*,\eta_\chi)$, the above objects are going to be denoted simply
$$I^G_{\disc,\psi},~S^G_{\disc,\psi},~L^2_{\disc,\psi}(G(F)\bs G(\A_F)),~\mbox{and}~R_{\disc,\psi}.$$

  The stable multiplicity formula computes each $\hat{S}^\fke_{\disc,\psi,\eta_\chi\zeta^\fke}(f^\fke)$, thereby provides a crucial input for $I^G_{\disc,\psi,\eta_\chi}$. To explain it we need more notation.
  Let us introduce a global stable linear form $f^{G^*}(\psi)$ for $f\in \cH(G^*)$ and $\psi\in \Psi(G^*,\eta_\chi)$. It is enough to consider $f=\prod_v f_v$. The local theorem in the quasi-split case, cf. Proposition \ref{prop:local-stable-linear}, provides a stable linear form $f^{G^*}_v(\psi_v)$. We simply take the product $$f^{G^*}(\psi):=\prod_v f^{G^*}_v(\psi_v).$$

  Recall that $\cS_{\psi}$, $ \epsilon^{G^*}_\psi(\cdot)$ and $s_{\psi}$ were defined in Chapter \ref{s:param-main-thm}. To introduce a few more invariants attached to $\psi$ in desired generality, consider a possibly non-connected complex reductive group $S^+$ over $F$, whose neutral component is denoted $S^0=(S^+)^0$. Let $S$ be a union of some connected components of $S^+$. When $s\in S$ is semisimple, write $S_s$ for the centralizer of $s$ in $S^0$. Choose a maximal torus $T$ of $S^0$. Write $W(S)$ and $W(S^0)$ respectively for the Weyl sets $N_S(T)/T$ and $N_{S^0}(T)/T$. Denote by $W_{\reg}(S)$ the set of $w\in W(S)$ which has only finitely many fixed points in $T$. The sign $\sgn^0(w)\in \{\pm 1\}$ is defined to be the parity of the number of positive roots of $(S^0,T)$ mapped by $w$ to negative roots. Finally $\det(w-1)$ for $w\in W_{\reg}(S)$ will designate the determinant of $w-1$ on $\Hom(X^*(T),\R)$ as a real vector space. In our specific setup where a parameter $\psi\in\Psi(G^*,\eta_\chi)$ is given, we may take for example $S^+=\ol{S}_\psi$, in which case $W(S^0)=W^0_\psi$.

    Arthur defined real numbers $i(S)$, $e(S)$ and $\sigma(S)$. The first one is explicitly defined\footnote{In \cite[(4.1.5)]{Arthur} he writes $|W(S)|^{-1}$ in place of $|W(S^0)|^{-1}$. We believe the latter is correct as in (8.1) of \cite{ArtUAR2}.}
  \begin{equation}\label{eq:i(S)}
    i(S):=|W(S^0)|^{-1} \sum_{w\in W_{\reg}(S)} s^0(w)|\det(w-1)|^{-1},
  \end{equation}
  Let $S_{\sspl}$ denote the set of semisimple elements in $S$. Define $S_{\el}$ to be the subset consisting of $s\in S_{\sspl}$ such that $Z(S_s^0)$ is finite. Given a subset $\Sigma$ of $S$ invariant under $S^0$-conjugation, define $\cE(\Sigma)$ to be the set of equivalence classes on $\Sigma\cap S_{\sspl}$, where $\sigma_1$ and $\sigma_2$ are considered equivalent if there exist $s\in S^0$ and $z\in Z(S_{\sigma_1}^0)^0$ such that $\sigma_2=s z\sigma_1 s^{-1}$. Set $\cE_{\el}(S):=\cE(S_{\el})$.
 Then he showed (\cite[\S8]{ArtUAR2}, cf. \cite[Prop 4.1.1]{Arthur}) that there is a unique way to assign real numbers to $\sigma(S)$ for all such $S$ in order that $i(S)=e(S)$, where
  \begin{equation}\label{eq:e(S)}
  e(S):=\sum_{s\in \cE_{\el}(S)} |\pi_0(S_s)|^{-1}\sigma(S^0_s).
   \end{equation}
   Let $x\in \ol\cS_\psi$. We define $i_\psi(x)=i^{G^*}_\psi(x)$ and $e_\psi(x)=e^{G^*}_\psi(x)$ to be $i(S)$ and $e(S)$, respectively, in the special case when $S$ is the union of connected components of $\ol S_\psi$ which map to $x$, cf. \eqref{eq:i(x)} and \eqref{eq:e(x)} below. We have
     \begin{equation}\label{eq:i(S)=e(S)}
  i_\psi(x)=e_\psi(x),\qquad x\in \ol\cS_\psi.
  \end{equation}

  One of the most important results in the classification for quasi-split unitary groups is the following. As explained in \S\ref{sub:results-qsuni} we take this for granted. Note that on the right hand side the dependence on $\eta_\chi$ is in the transfer $f\mapsto f^{G^*}$ for the twisted endoscopic data arising from $(G^*,\eta_\chi)$.

\begin{prop}(Stable multiplicity formula, \cite[Thm 5.1.2]{Mok})\label{p:stable-multiplicity} %
   Let $G^*$ and $\eta_\chi$ be as above and $\psi\in \Psi(G^*,\eta_\chi)$. Then
\begin{equation}\label{eq:stable-mult}
   S^{G^*}_{\disc,\psi}(f)=|\cS_{\psi}|^{-1} \epsilon^{G^*}_\psi(s_{\psi}) \sigma(\ol{S}_{\psi}^0)f^{G^*}(\psi),\quad f\in \cH(G).
\end{equation}
If $\psi^N\in \tilde\Psi(N)$ does not belong to $\eta_{\chi,*}\Psi(G^*,\eta_\chi)$ then
 \begin{equation}\label{eq:stable-mult-0}
   S^{G^*}_{\disc,\psi^N,\eta_\chi}(f)=0,\quad f\in \cH(G).
\end{equation}
\end{prop}

  As explained in \cite[3.2]{Arthur}, we may and do choose the Weyl-invariant Hermitian metric on the linear dual of the Cartan subalgebra of the archimedean Lie algebra compatibly for $G$ (as well as $G^*$) and $G(N)$. (Otherwise the condition $t=t(\psi^N)$ needs to be modified.)

\begin{cor}\label{cor:outside-Psi(G)}
  Let $f\in \cH(G)$, $t\ge 0$, and $c\in \cC_{\A}(G)$. Then $I^G_{\disc,t,c}(f)=0$ and $L^2_{\disc,t,c}(G(F)\bs G(\A_F))=0$ unless $t=t(\psi)$ and $c\mapsto c(\psi)$ for some $\psi^N\in \tilde{\Psi}(N)$.  If $\psi^N\in \tilde\Psi(N)$ does not lie in $\eta_{\chi,*}\Psi(G^*,\eta_\chi)$ then $I^G_{\disc,\psi^N,\eta_\chi}(f)=0$ and $L^2_{\disc,\psi^N,\eta_\chi}(G(F)\bs G(\A_F))=0$.

\end{cor}

\begin{proof}

  The vanishing of $I^G_{\disc,t,c}(f)$ and $I^G_{\disc,\psi^N,\eta_\chi}(f)$ is immediate from \eqref{eq:stabilization-t,c} and the stable multiplicity formula. The assertions on the $L^2$-spaces are deduced from the vanishing of $I^G_{\disc,t,c}(f)$, resp. $I^G_{\disc,\psi^N,\eta_\chi}(f)$, and the inductive hypotheses, by arguing exactly as in the proof of \cite[Cor 3.4.3]{Arthur}.
\end{proof}

As a consequence (using the fact that $\psi=(\psi^N,\tilde\psi)\mapsto \psi^N$ is an injection) we obtain decompositions
$$L^2_\disc(G(F)\bs G(\A_F))=\bigoplus_{\psi\in \Psi(G^*,\eta_\chi)} L^2_{\disc,\psi}(G(F)\bs G(\A_F)),$$
$$ \tr R_\disc(f)=\sum_{\psi\in \Psi(G^*,\eta_\chi)} \tr R_{\disc,\psi}(f),\quad f\in \cH(G).$$

\subsection{The global intertwining operator} \label{sec:giop}
Let $E/F$ be a quadratic extension of number fields and let $G^*=U_{E/F}(N)$ be the quasi-split unitary group in $N$ variables. Let $\Xi : G^* \rw G$ be an equivalence class of inner twists. Let $(M,P)$ be a parabolic pair for $G$, i.e. a pair consisting of a parabolic subgroup $P$ of $G$ and a Levi subgroup $M$ of a $P$, both defined over $F$. There exists a unique standard parabolic pair $(M^*,P^*)$ of $G^*$ with the property that for some $\xi \in \Xi$ we have $\xi(M^*,P^*)=(M,P)$. The set of such $\xi$ forms an equivalence class of inner twists $\Xi_M : M^* \rw M$. Let $(\hat M,\hat P)$ be the standard parabolic pair of $\hat G$ dual to $(M^*,P^*)$.

We assume that $M\neq G$. Let $\pi_M$ be an irreducible constituent of $L^2_\tx{disc}(A_M(\R)^0 M(F)\lmod M(\A_F))$. Let $\psi_{M^*} \in \Psi_2(M^*,\eta_\chi)$ be the corresponding global parameter given by the induction hypothesis. Here we should really have written $\Psi_2(M^*,i_{M^*}\eta_\chi)$ in place $\Psi_2(M^*,\eta_\chi)$, where $i_{M^*}:{}^L M^*\hra {}^L G^*$ is a Levi embedding, but we will continue this abuse of language as there is no danger of confusion.
 Let $u \in N_{\psi_{M^*}}(M,G) = \tx{Cent}(\psi_{M^*},\hat G) \cap \tx{Norm}(A_{\hat M},\hat G)$ and write $w=w_u \in W(\hat G,\hat M)^\Gamma$ for the image of $u$. There are canonical $\Gamma$-equivariant isomorphisms $W(\hat M,\hat G) \cong W(M^*,G^*) \cong W(M,G)$ via which we view $w$ as an element of $W(M,G)^\Gamma$.

Consider the induced representation $\mc{I}_P(\pi_M)$ acting by the right regular representation on the Hilbert space of measurable functions
\[ \mc{H}_P(\pi_M) = \{ f : G(\A_F) \rw V_{\pi_M}| f(nmg) = \delta_P^\frac{1}{2}(m)\pi(m)f(g) \} \]
whose restriction to an open compact subgroup $K \subset G(\A_F)$ is square-integrable.
Here $V_{\pi_M}$ is the Hilbert space on which $\pi_M$ acts and is a subspace of $L^2_\tx{disc}(A_M(\R)^0M(F)\lmod M(\A_F))$. Given a second parabolic subgroup $P'$ with Levi factor $M$, Langlands' theory of Eisenstein series provides an intertwining operator
\[ J_{P'|P} : H_P(\pi_M) \rw H_{P'}(\pi_M). \]
If we replace $\pi_M$ by a twist $\pi_{M,\lambda}$ for $\lambda \in \mf{a}_{M,\C}^*$, the operator $J_{P'|P}$ is defined by the integral formula
\begin{equation} \label{eq:giop1} [J_{P'|P}f](g) =  \int_{N(\A_F) \cap N'(\A_F) \lmod N'(\A_F)} f(n'g) dn' \end{equation}
which converges absolutely whenever the real part of $\lambda$ lies in a certain open cone. The measure $dn'$ is taken with respect to an arbitrary top form on the vector space $\mf{n} \cap \mf{n'} \lmod \mf{n'}$ defined over $F$, as well as the Haar measure on $\A_F$ which assigns the quotient $\A_F/F$ volume $1$. Here $\mf{n}$ denotes the Lie-algebra of $N$, and the measure is independent of the choice of top form by the product formula. Langlands has shown that, as a function of $\lambda$, the operator $J_{P'|P}$ has meromorphic continuation and is defined and unitary at $\lambda=0$. This defines $J_{P'|P}$ for the representation $\pi_M=\pi_{M,0}$.

We will be particularly interested in the case $P'=w^{-1}Pw$. There are two more intertwining operators, defined in elementary terms as follows. Let $\breve{w} \in N(M,G)(F)$ be a lift of $w$. Define the representation $\breve w\pi_M$ on the Hilbert space $V_{\pi_M}$ by $\breve w\pi_M(m) = \pi_M(\breve w^{-1}m\breve w)$. The two intertwining operators are given by
\[ l(\breve w) : \mc{H}_{w^{-1}Pw}(\pi_M) \rw \mc{H}_P(\breve w\pi_M),\qquad [l(\breve w)f](g) = f(\breve w^{-1}g) \]
and
\[ C_{\breve w} : (\breve w\pi_M,V_{\pi_M}) \rw (\pi_M,V_{\pi_M}),\qquad [C_{\breve w}f](m) = f(\breve w^{-1}m\breve w). \]
The last operator extends pointwise to an intertwining operator $C_{\breve w} : \mc{H}_P(\breve w\pi_M) \rw \mc{H}_P(\pi_M)$. The composition $$M_P(w,\pi_M)=C_{\breve w}\circ l(\breve w)\circ J_{w^{-1}Pw|P}$$ is then a self\-intertwining operator of the representation $\mc{I}_P(\pi_M)$ on the space $\mc{H}_P(\pi_M)$. A simple calculation, using the automorphy of the functions comprising $V_{\pi_M}$, shows that $M_P(w,\pi_M)$ does not depend on the choice of $\breve w$.

The un-normalized operators $J_{P'|P}$ and $M_P(w,\pi_M)$ can be normalized to obtain operators $R_{P'|P}(\pi_M,\psi_{M^*})$ and $R_P(w,\pi_M,\psi_{M^*})$ by setting

\begin{equation} \label{eq:giop1r}
  R_{P'|P}(\pi_M,\psi_{M^*}) = r_{P'|P}(\psi_{M^*})^{-1} \cdot J_{P'|P},
\end{equation}
and
\begin{equation}\label{e:R=rM}
  R_P(w,\pi_M,\psi_{M^*}) = r_P(w,\psi_{M^*})^{-1} \cdot M_P(w,\pi_M),
\end{equation}
where $r_{P'|P}(\psi_{M^*})$ and $r_P(w,\psi_{M^*})$ are the global normalizing factors determined as follows. The Levi subgroup $M^*$ of $G^*=U_{E/F}(N)$ is of the form $G_{E/F}(N_1)\times \cdots G_{E/F}(N_k) \times U_{E/F}(N_-)$. %
 Concretely the parameter $\psi_{M^*}$ of $\pi_M$ consists of automorphic representations of $G_{E/F}(N_i,\A_F)$ for $1\le i\le k$ and a parameter for $U_{E/F}(N-)$ arising from a conjugate self-dual automorphic representation of $G_{E/F}(N-)$. Let $\rho_{P'|P}$ be the adjoint representation of $^LM$ on the vector space $\mf{\hat n} \cap \mf{\hat n'} \lmod \mf{\hat n'}$, where $\mf{\hat n}$ and $\mf{\hat n'}$ are the Lie algebras of the unipotent radicals of $\hat P$ and $\hat P'$. Associated to $\psi_{M^*}$ and the contragredient representation $\rho_{P'|P}^\vee$ we have automorphic $L$- and $\epsilon$-factors, and we set

\[ r_{P'|P}(\psi_{M^*}) := \frac{L(0,\psi_{M^*},\rho^\vee_{P'|P})}{L(1,\psi_{M^*},\rho^\vee_{P'|P})}\frac{\epsilon(\frac{1}{2},\psi_{M^*},\rho^\vee_{P'|P})}{\epsilon(0,\psi_{M^*},\rho^\vee_{P'|P})} \]
and
\[ r_P(w,\psi_{M^*}) := r_{w^{-1}P|P}^{-1} \cdot \epsilon(\frac{1}{2},\psi_{M^*},\rho^\vee_{w^{-1}Pw|P})^{-1}. \]

  The importance of the global intertwining operator is clear from the discrete part of the trace formula. In the next subsection we will see that the normalized operator $R_P(w,\pi_M,\psi_{M^*})$ admits a product decomposition and enters the global intertwining relation.

\subsection{The global intertwining relation}\label{sub:global-intertwining}

  In this subsection we introduce the two global linear forms $f_G(\psi_{M^*},u)$ and $f'_G(\psi_{M^*},s)$, whose local versions were introduced in  \S\ref{sec:lir}. Note the change of notation from the local setting that we are now denoting a global parameter in $\Psi(M^*,\eta_\chi)$ by $\psi_{M^*}$ rather than $\psi$, where $M^*$ is a Levi subgroup of $G^*$.
  The definition requires the localization of the parameter $\psi_{M^*}$, which we recall relies on a highly nontrivial result (Proposition \ref{prop:2nd-seed-thm}) established in the quasi-split case. The global linear forms will appear in the spectral and endoscopic expansions after we go through the first step (``standard model'') in comparing formulas \eqref{eq:I_disc-defn} and \eqref{eq:I_disc-stabilized}, and as such play a key role in later arguments. The two linear forms are closely related via the global intertwining relation to be stated below. It will follow as a corollary of the local intertwining relation to be proved later in Section \ref{chapter4}.

   Let $\psi_{M^*}\in \Psi_2(M^*,\eta_\chi)$. We may write $\psi_{M^*}=(\psi_{M^*}^N,\tilde\psi_{M^*})$ as in \eqref{eq:param-end-U-glo} (taking $M^*$ to be the endoscopic group) and define $\psi=(\psi^N,\tilde\psi)$ by $\psi^N:=\psi_{M^*}^N$ and $\tilde\psi:=i_{M^*}\tilde\psi_{M^*})$, where $i_{M^*}:{}^L M^*\hra {}^L G^*$. In particular the image of $\tilde\psi$ is centralized by the image of $A_{\hat M^*}:=(Z(\hat M^*)^\Gamma)^0$ under $i_{M^*}$. Since the latter property is destroyed if $\tilde\psi$ is replaced by an arbitrary $\hat G^*$-conjugate (as the image of ${}^ g \tilde\psi$ is centralized not by $i_{M^*}(A_{\hat M^*})$ but by the image of $A_{\hat M^*}$ under $g i_{M^*} g^{-1}$), we do not consider $\psi$ as an equivalence class in $\Psi(G^*,\eta_\chi)$ for the moment.  So when we want to be precise we use $\psi_{M^*}$ rather than $\psi$ in the definition of various groups below. Put $S_{\psi}:=S_\psi(G^*)$ and $\ol{S}_{\psi}:=\ol S_\psi(G^*)$. Recall from \S\ref{subsub:global-param-U} the definition of $\ol{S}_\psi:=S_\psi/Z(\hat G^*)^\Gamma$, $\cS_\psi:=\pi_0(S_\psi)$, $\ol{\cS}_\psi:=\pi_0(\ol{S}_\psi)$, $\srad_\psi:=(S^0_\psi\cap \hat{G}^*_{\der})^0$, and $S^\natural_\psi:=S_\psi/\srad_\psi$ in the global setting. In addition,
   write $Z_{S_{\psi}}$, $Z_{S^0_{\psi}}$, and $Z_{\srad_{\psi}}$ for the centralizers of $A_{\hat M^*}:=(Z(\hat M^*)^\Gamma)^0$ in $S_\psi$, $S^0_\psi$, and $\srad_\psi$, respectively, and $N_{S_{\psi}}$ and $N_{S^0_{\psi}}$ for the normalizers of $A_{\hat M^*}$ in $S_\psi$ and $S^0_\psi$, respectively.

   We need two global diagrams, where the first is the same as in \S\ref{sec:diag} (but recalled here for the reader's convenience). For simplicity we omitted $(M^*,G^*)$ from the notation in the second and third columns. For instance $N^\natural_{\psi_{M^*}}$ and $\ol{\cN}_{\psi_{M^*}}$ should be $N^\natural_{\psi_{M^*}}(M^*,G^*)$ and $\ol{\cN}_{\psi_{M^*}}(M^*,G^*)$ to be more precise.

   \begin{equation}\label{eq:diagram-extended-pure-inner-global}
  \xymatrix{
     & W^{\rad}_{\psi_{M^*}}=\frac{N_{\srad_{\psi}}}{Z_{\srad_{\psi}}} \ar[d] \ar@{=}[r] & W^0_{\psi_{M^*}}:=\frac{N_{S_{\psi}^0}}{Z_{S_{\psi}^0}} \ar[d] \\
    S^{\natural\natural}_{\psi_{M^*}}(M^*):=\frac{Z_{S_{\psi}}}{Z_{\srad_{\psi}}} \ar[r] \ar@{=}[d] & N^\natural_{\psi_{M^*}}:=\frac{N_{S_{\psi}}}{Z_{\srad_{\psi}}} \ar[r] \ar[d] & W_{\psi_{M^*}}:=\frac{N_{S_{\psi}}}{Z_{S_{\psi}}}  \ar[d] \\
    S^{\natural\natural}_{\psi_{M^*}}(M^*) \ar[r]  & S^\natural_{\psi_{M^*}}:=\frac{N_{S_{\psi}}}{N_{\srad_{\psi}}} \ar[r]  & R_{\psi_{M^*}}  :=\frac{N_{S_{\psi}}}{N_{S^0_{\psi}} Z_{S_{\psi}}}
  }
\end{equation}

\begin{equation}\label{eq:diagram-mod-central-global}
  \xymatrix{
     & W^0_{\psi_{M^*}}=\frac{N_{S_{\psi}^0}}{Z_{S_{\psi}^0}} \ar[d] \ar@{=}[r] & W^0_{\psi_{M^*}}=\frac{N_{S_{\psi}^0}}{Z_{S_{\psi}^0}} \ar[d] \\
    \ol{\cS}_{\psi_{M^*}}(M^*):=\frac{Z_{S_{\psi}}}{Z^0_{S_{\psi}} Z(\hat G)^\Gamma} \ar[r] \ar@{=}[d] & \ol{\cN}_{\psi_{M^*}}:=\frac{N_{S_{\psi}}}{Z^0_{S_{\psi}}Z(\hat G)^\Gamma} \ar[r] \ar[d] & W_{\psi_{M^*}}=\frac{N_{S_{\psi}}}{Z_{S_{\psi}}}  \ar[d] \\
    \ol{\cS}_{\psi_{M^*}}(M^*) \ar[r]  & \ol{\cS}_{\psi_{M^*}}:=\frac{N_{S_{\psi}}}{N_{S^0_{\psi}}Z(\hat G)^\Gamma} \ar[r]  & R_{\psi_{M^*}}  =\frac{N_{S_{\psi}}}{N_{S^0_{\psi}} Z_{S_{\psi}}}
  }
\end{equation}

Now if $\psi=(\psi^N,\tilde\psi)$ is conjugated by an element $g\in \hat G^*$ then both diagrams are conjugated by $g$. So $\psi$ represents an element of $\Psi(G^*,\eta_\chi)$, which is a $\hat G^*$-conjugacy orbit, then not elements but only conjugacy classes in the groups $W^0_{\psi_{M^*}}$, $N^\natural_{\psi_{M^*}}$, etc are well defined. However various quantities are still well defined. For instance the cardinality $|W_{\psi_{M^*}}|$ is well defined, and a sum over a subset of $\ol{\cN}_{\psi_{M^*}}$ is well defined as long as the sum is seen to be invariant under the $\hat G^*$-conjugation.

The diagram \eqref{eq:diagram-extended-pure-inner-global} is the global version of the local diagram in \S\ref{sec:diag}. The localization maps in \S\ref{subsub:localizations-U} induce functorial localization maps from \eqref{eq:diagram-extended-pure-inner-global} to the local diagram at each place $v$ (i.e. the diagram in \S\ref{sec:diag} with $\psi_v$ in place of $\psi$ there) in the sense that the induced maps commute with all maps in the global and local diagrams.

   The second diagram above is exactly the same as in the quasi-split case, cf. \cite[(4.2.3)]{Arthur}. Our notation is slightly different from theirs in that we insert the bar notation to emphasize that the construction of the groups involves dividing out by $Z(\hat G^*)^\Gamma$. For instance $\ol{S}_\psi:=S_\psi/Z(\hat G^*)^\Gamma$ and $\ol{\cS}_\psi:=\pi_0(\ol{S}_\psi)$ whereas $\cS_\psi:=\pi_0(S_\psi)$. This is to be consistent with our notation in the local setting, where it is essential not to divide out by a central subgroup. However when it comes to \emph{global} inner forms, it turns out that one can still work with the groups as in \eqref{eq:diagram-mod-central-global}. Indeed a typical observation is going to be that a global linear form defined in terms of an element of $N_{\psi_{M^*}}^\natural$ or $S_{\psi_{M^*}}^\natural$ descends to a linear form in $\ol{\cN}_{\psi_{M^*}}$ or $\ol{\cS}_{\psi_{M^*}}$. For this it is useful to know that the second diagram is the first diagram modulo (the image of) $Z(\hat G^*)^\Gamma$ in the lower left square, cf. the lemma below.

   \begin{lem}\label{lem:quot-by-central-subgroup} The following are true.
  \begin{enumerate}
    \item $Z^0_{S_\psi}=Z_{S^0_\psi}$.
    \item $Z_{\srad_{\psi}} Z(\hat G^*)^\Gamma = Z_{S^0_\psi} Z(\hat G^*)^\Gamma$ and $N_{\srad_{\psi}} Z(\hat G^*)^\Gamma = N^0_{S_\psi} Z(\hat G^*)^\Gamma$.
    \item The quotient of $\cS^{\natural\natural}_{\psi_{M^*}}(M^*)$ (resp. $\cN^\natural_{\psi_{M^*}}$, resp. $\cS^\natural_{\psi_{M^*}}$) by the image of $Z(\hat G^*)^\Gamma$ is canonically isomorphic to $\ol{\cS}_{\psi_{M^*}}(M^*)$ (resp. $\ol{\cN}_{\psi_{M^*}}$, resp. $\ol{\cS}_{\psi_{M^*}}$). The maps in \eqref{eq:diagram-mod-central-global} are induced by those in \eqref{eq:diagram-extended-pure-inner-global} by passing to quotients.
  \end{enumerate}
\end{lem}

\begin{proof}
  Since $Z_{S_\psi}/Z_{S^0_\psi}\hra S_\psi/S^0_\psi$, we have $[Z_{S_\psi}:Z_{S^0_\psi}]<\infty$. As the centralizer of a torus in a connected reductive group, $Z_{S_\psi^0}$ is connected. Hence $ Z_{S^0_\psi}$ is a finite index subgroup of $Z^0_{S_\psi}$, and they are equal since the latter is connected. This verifies part 1.
  Recall from Lemma \ref{lem:squot} that $\srad_{\psi} Z(\hat G^*)^\Gamma=S^0_{\psi} Z(\hat G^*)^\Gamma$. Since $Z(\hat G^*)^\Gamma$ is in the center of $S_\psi$, we have
  $$Z_{\srad_{\psi}} Z(\hat G^*)^\Gamma = Z_{\srad_{\psi} Z(\hat G^*)^\Gamma}(A_{\hat M^*})=
  Z_{S^0_{\psi} Z(\hat G^*)^\Gamma}(A_{\hat M^*})=Z_{S^0_\psi} Z(\hat G^*)^\Gamma.$$
  In the same manner one proves $N_{\srad_{\psi}} Z(\hat G^*)^\Gamma = N^0_{S_\psi} Z(\hat G^*)^\Gamma$, completing the proof of part 2.
  Part 3 follows from the earlier parts and the definition of the groups.

\end{proof}

  Recalling that $\psi_{M^*}\in \Psi_2(M^*)$, we have the following.

\begin{lem}\label{lem:S(M,G)-S(G)-global}
  We have canonical isomorphisms $S^\natural_{\psi_{M^*}}(M^*,G^*)=S^\natural_\psi(G^*)$ and $\ol{\cS}_{\psi_{M^*}}(M^*,G^*)=\ol{\cS}_{\psi}(G^*)$.
\end{lem}

\begin{proof}
  The first assertion is proved in the same way as Lemma \ref{lem:S(M,G)-S(G)}. The second follows from the first by taking quotients by $Z(\hat G^*)$.
\end{proof}

  In order to introduce the first global linear form we need some preparation regarding intertwining operators. Now assume that the Levi subgroup $M^*$ of $G^*$ is proper. %
  Choose any equivalence class $\Xi$ of extended pure inner twists $G^*\ra G$.
  Let $(M,P)$ and $(M^*,P^*)$ be parabolic pairs for $G$ and $G^*$ exactly as at the start of \S\ref{sec:giop}. In particular we have a representative $(G,\xi,z)\in\Xi$ such that $\xi(M^*,P^*)=(M,P)$. %
   Consider a parameter $\psi_{M^*}\in\Psi_2(M^*,\eta_\chi)$. Let $u\in \cN^\natural_{\psi_{M^*}}(M^*,G^*)$, and $\psi_F : \A_F/F \rw \C^\times$ be a nontrivial additive character.
  Denote the localization of $\Xi$, $\psi_{M^*}$, $u$, and $\psi_F$ at each place $v$ by $\Xi_v$, $\psi_{M^*,v}$, $u_v$, and $\psi_{F,v}$, respectively. Write $w_u$ for the image of $u$ in $W_{\psi_{M^*}}(M^*,G^*)$.
We have constructed the local normalized intertwining operators $R_P( u,\Xi_v,\pi_{M,v},\psi_{M^*,v},\psi_{F_v})$ at all places $v$ of $F$. We would like to compare the canonical global intertwining operator $R_P(w,\pi_M,\psi_{M^*})$ of section \ref{sec:giop} with
$$R_P( u,\Xi,\pi_M,\psi_{M^*},\psi_{F}):=\bigotimes_v R_P( u,\Xi_v,\pi_{M,v},\psi_{M^*,v},\psi_{F_v}).$$

\begin{prop*} \label{pro:iop3lg} Let $G^*=U_{E/F}(N)$, $(\xi,z) : G^* \rw G$, $\psi_{M^*} \in \Psi_2(M^*)$ be as above. Assume that $\pi_M\in \Pi_{\psi_{M^*}}(M,\xi)$ is automorphic in that $\lg \pi_M,\cdot\rg_{\xi}=\epsilon_{\psi_{M^*}}$. (See \S\ref{sub:main-global-thm} for the definition of $\Pi_{\psi_{M^*}}(M,\xi)$. Note that $\pi_M$ may not be irreducible.) Let $C_{\breve w}$ be as in section \ref{sec:giop}. For each $u \in N^\natural_{\psi_{M^*}}(M^*,G^*)$, we have an equality of isomorphisms $(\breve w\pi,V_{\pi_M})\ra (\pi_M,V_{\pi_M})$:
$$\bigotimes_v \pi_{M,v}(u)_{\xi_v,z_v}=\epsilon_{\psi_{M^*}}(u)\cdot C_{\breve w},$$
where the local components on the left are the operators defined in Subsection \ref{sec:iop3}.
\end{prop*}
\begin{proof}
We postpone the general proof of this proposition to \cite{KMS_B}. Presently, we will give the proof in the special case of global pure inner twists of unitary groups. More precisely, let $M^*=M^*_+ \times M^*_-$ and decompose $z=z_+ \times z_-$ accordingly. We assume that $z_+=1$ and $z_- \in Z^1(\Gamma,M^*_-)$. This is the global analog of the setting of Section \ref{sec:iop3u}. Moreover we temporarily assume that $\psi_{M,+}$, namely the $M^*_+$-part of $\psi_{M^*}$, is a generic parameter so that $\pi_{M,+}$ admits a nonzero global Whittaker functional.

In that case, we can use the description of $\pi_{M,v}(u)_{\xi_v,z_v}$ given in that section. Namely, write $\pi_{M,v}=\pi_{M,v,+} \otimes \pi_{M,v,-}$. Then $\pi_{M,v}(u)_{\xi_v,z_v} = \<\pi_{M,v},u\>_{\xi_v,z_v} \cdot \pi_{M,v,+}(\breve w)_{\xi_v} \otimes \tx{id}_{\pi_{M,v,-}}$. Since $\pi_M$ is automorphic, we have $\prod_v \<\pi_{M,v},u\>_{\xi_v,z_v}=\epsilon_{\psi_{M^*}}(u)$. %
 Furthermore, $\bigotimes_v \pi_{M,v,+}(\breve w)_{\xi_v}$ is the unique intertwining operator $\breve w\pi_{M,+} \rw \pi_{M,+}$ that preserves a global Whittaker functional.

On the other hand, we can also decompose $C_{\breve w}$ as $C_{\breve w,+} \otimes C_{\breve w,-}$ according to $\pi_M = \pi_{M,+} \otimes \pi_{M,-}$. As $\tx{Ad}(\breve w)$ acts trivially on $M_-$, we see $C_{\breve w,-} =\tx{id}_{\pi_{M,-}}$. To complete the proof, it is enough to show that  $C_{\breve w,+}$ is the unique intertwining operator $\breve w\pi_{M,+} \rw \pi_{M,+}$ that preserves a global Whittaker functional. For this we may assume that $M=M^*=M^*_+$ so that $\pi_M=\pi_{M,+}$. The proof is quickly reduced to the case where $M=M^*_+=\GL(r)\times \cdots \GL(r)$ ($k$ times), $w$ acts as a (transitive) permutation of the $k$ factors, and the Whittaker datum is taken with respect to a $w$-invariant Borel subgroup $B_M\subset M$. Write $N_M$ for the unipotent radical of $B_M$. By slight abuse of notation, the nondegenerate additive character of $N_M(\A_F)$ induced by $\psi_F$ is again denoted $\psi_F$. Then we have to check that for $f\in V_{\pi_M}$,
$$\int_{N_M(F)\bs N_M(\A_F)} (\breve w\pi(n) f)(m)\psi_F(n)dn = \int_{N_M(F)\bs N_M(\A_F)} (\pi(n) f)(m)\psi_F(n)dn.$$
The equality holds since the left hand side is computed as
$$\int_{N_M(F)\bs N_M(\A_F)} f(m\breve w^{-1} n\breve w)\psi_F(n)dn $$
$$= \int_{N_M(F)\bs N_M(\A_F)} f(m n)\psi_F(\breve w^{-1} n\breve w)dn
=  \int_{N_M(F)\bs N_M(\A_F)} f(m n)\psi_F(n)dn,$$
using the fact that the transformation $n\mapsto \breve w^{-1} n\breve w$ preserves the measure on $N_M(F)\bs N_M(\A_F)$.

Now we drop the assumption that $\psi_{M,+}$ is generic. Then $\pi_{M_+}$ is a discrete automorphic representation of $M^*_+(\A_F)$ so appears as an irreducible quotient of an induced representation from a cuspidal automorphic representation on a Levi subgroup of $M^*_+$. Then the above argument still goes through if we replace $\pi_{M_+}$ by the induced representation and use a nonzero Whittaker functional on the induced representation. (Then we obtain that $\bigotimes_v \pi_{M,v,+}(\breve w)_{\xi_v}=C_{\breve w,+}$ by observing that they are the unique intertwining operator preserving the nonzero Whittaker functional.)

\end{proof}

\begin{pro} \label{pro:ioplg}
 Let $\psi_{M^*}\in \psi_2(M^*)$ and $u \in N^\natural_{\psi_{M^*}}(M^*,G^*)$. Let $\pi_M\in \Pi_{\psi_{M^*}}(M,\xi)$.
\begin{enumerate}
  \item $R_P( yu,\Xi,\pi_M,\psi_{M^*},\psi_{F})=\lg \pi_M,\ol y\rg_{\xi} R_P( u,\Xi,\pi_M,\psi_{M^*},\psi_{F})$, where $y\in S^{\natural\natural}_\psi(M^*)$, its image in $ \ol\cS_{\psi_{M^*}}(M^*)$ is denoted $\ol y$, and $\lg \cdot,\pi_M\rg_\xi$ is the character of $\ol\cS_{\psi_{M^*}}$ introduced in \S\ref{sub:main-global-thm}.
  \item $R_P( u,\Xi_1,\pi_M,\psi_{M^*},\psi_{F})=R_P( u,\Xi_2,\pi_M,\psi_{M^*},\psi_{F})$ if $\Xi_1$ and $\Xi_2$ give rise to the same inner twist.
  \item Suppose %
  that $\pi_M$ is automorphic. Then $$R_P(w_u,\pi_M,\psi_{M^*})=\epsilon_{\psi_{M^*}}(u) R_P(u,\Xi,\pi_{M},\psi_{M},\psi_F).$$
\end{enumerate}

\end{pro}
\begin{proof}
Part 1 follows from part 3 of Proposition \ref{lem:rpequiv} and the fact that $\lg \pi_M,\cdot\rg_{\xi}=\prod_v \lg \pi_{M,v},\cdot\rg_{\xi_v,z_v}$ is trivial on $Z(\hat G^*)^\Gamma$, cf. \S\ref{sub:main-global-thm}. (The dependence only on $\ol y$ is also implied by part~1 of the same proposition.) Part~2 is proved similarly using part~2 of Proposition \ref{lem:rpequiv}. (Compare with the proof of Lemma \ref{lem:lirind}.) To prove part~3,
choose a global extended pure inner twist $(\xi,z)$ as in Lemma \ref{lem:c2}. Use the localizations $(\xi_v,z_v)$ to construct the operators $R_P(u^\natural,(\xi_v,z_v),\pi_{M,v},\psi_{M^*,v},\psi_F)$. The claim now follows from Lemma \ref{lem:iop1lg} and Proposition \ref{pro:iop3lg}, noting that $\breve w$ belongs to $N(M,G)(F)$ and the product of all $\lambda_v(w,\psi_F)$ is equal to $1$.
\end{proof}

  We are ready to define the first global linear form $$f\in \cH(G)\quad \mapsto \quad f_{G,\Xi}(\psi_{M^*},u). $$ It suffices to consider the case $f=\prod_v f_v$. By slight abuse we still write $u$ for its image in $\cN^\natural_{\psi_{M^*,v}}(M^*,G^*)$. Define
  \begin{equation}
    \label{eq:f_G(psi,u)-definition}\begin{aligned}
    f_{G,\Xi}(\psi_{M^*},u) & :=\prod_v f_{v,G,\Xi_v}(\psi_{M^*,v},u) \\ & = \sum_{\pi_M\in \Pi_{\psi_{M^*}}}  \tr (R_P(u,\Xi,\pi_M,\psi_{M^*},\psi_F) \cI_P(\pi_{M,v},f_v)).
  \end{aligned}
  \end{equation}
  (The global packet $\Pi_{\psi_{M^*}}$ is defined as in \S\ref{sub:main-global-thm} via induction hypothesis.) There may be several choices of $\psi_{M^*}$ but $f_{G,\Xi}(\psi_{M^*},u)$ is well defined independently of the choice thanks to Lemma \ref{lem:linear-form-ind-of-M}. Since $M^*$ transfers to $G$ over $F$ if and only if $M^*$ transfers to $G$ over $F_v$ for all places $v$ of $F$ (Lemma \ref{lem:transfer-Levi-locglo}), we see that $f_{G,\Xi}(\psi_{M^*},u)=0$ for all $f$ unless $M^*$ transfers to $G$ globally in view of the definition of $f_{v,G,\Xi}(\psi_{M^*,v},u)$ in \S\ref{sec:lir}.

\begin{lem}\label{lem:independence-f_G}
  The linear form $f_{G,\Xi}(\psi_{M^*},u)$ depends only on the underlying inner twist for $\Xi$ and the image of $u$ in $ \ol{\cN}_\psi(M^*,G^*)$. (Namely if $\Xi_1$ and $\Xi_2$ give rise to isomorphic inner twists and if $u_1,u_2$ have the same image in $ \ol{\cN}_\psi(M^*,G^*)$ then $f_{G,\Xi_1}(\psi_{M^*},u_1)=f_{G,\Xi_2}(\psi_{M^*},u_2)$.)   %
\end{lem}

  In light of the lemma we have a well-defined linear form $f\mapsto f_{G}(\psi_{M^*},\ol u)$ for $\ol u\in \ol{\cN}_{\psi_{M^*}}(M^*,G^*)$ without reference to $\Xi$. Compare this with the paragraph following Lemma \ref{lem:independence-end} below.

\begin{proof}

  This is clear from parts 1 and 2 of Proposition \ref{pro:ioplg}.

\end{proof}

  Next we define the second global linear form. Let $\psi\in \Psi(G^*,\eta_\chi)$ and $s\in S_{\psi,\sspl}$. To $(\psi,s)$ corresponds a pair $(\fke,\psi^\fke)$ (see \S\ref{sub:endo-correspondence}), where $\fke\in \ol\cE(G^*)$ and $\psi^\fke$ is an $\Out_G(G^\fke)$-orbit in $\Psi(G^\fke,\eta_\chi\eta^\fke)$.%
  Write $s_v\in S_{\psi_v,\sspl}$ for the image of $s$ at each place $v$. Define a linear form
  $$f\in \cH(G)\quad\mapsto\quad f'_{G,\Xi}(\psi,s)$$
  given, whenever $f=\prod_v f_v$, by formula
  $$\begin{aligned}
    f'_{G,\Xi}(\psi,s) & := f^\fke(\psi^\fke)=\prod_v f^\fke_v(\psi^\fke_v) =\prod_v f'_{v,G,\Xi}(\psi_v,s_v).
  \end{aligned}$$
  Unlike in the local setting $\fke$ is an isomorphism class, not a strict isomorphism class, but $f^\fke$ and $f^\fke(\psi^\fke)$ are well-defined. This is due to Proposition \ref{prop:trans-factor-coincide} and the fact that the adelic Kottwitz-Shelstad transfer factor depends only on the isomorphism class.
  Notice that the definition does not involve any choice of $M^*$ or $\psi_{M^*}$, as already observed in Lemma \ref{lem:linear-form-ind-of-M} in the local setup. However we also write $f'_{G,\Xi}(\psi_{M^*},s)$ for $f'_{G,\Xi}(\psi,s)$ if $\psi$ is the image of $\psi_{M^*}\in \Psi_2(M^*,\eta_\chi)$.

\begin{lem}\label{lem:independence-end} For $\psi\in \Psi(G^*,\eta_\chi)$,
  $f'_{G,\Xi}(\psi,s)$ is dependent only on the underlying inner twist of $\Xi$ and the image $\ol{s}$ of $s$ in $\ol{\cS}_{\psi}$ (in the same sense as in the preceding lemma).

\end{lem}

  As a consequence, the linear form $f\mapsto f'_{G}(\psi,x)$ is unambiguously defined on $\cH(G)$ for every $\psi\in \Psi(G^*,\eta_\chi)$ and $x\in \ol{\cS}_{\psi}$ regardless of the choice of $\Xi$. The analogue for $f_{G,\Xi}(\psi,x)$ is more complicated; a version of it is stated at the end of this subsection.

\begin{proof}
  The independence results from the fact that the adelic transfer factor is independent of the choice of $\Xi$, cf. Proposition \ref{prop:trans-factor-coincide}, implying that the global transfer $f\ra f^\fke$ is independent of the choice of $\Xi$. To see the latter assertion, let $s_0\in \srad_\psi$. Since $(s_0s)_v$ and $s_v$ have the same image in $S^\natural_{\psi_v}$, Lemma \ref{lem:local-indep-f'} gives us $f'_{v,G,\Xi}(\psi,(s_0 s)_v)=f'_{v,G,\Xi}(\psi,s_v)$, hence $f'_{G,\Xi}(\psi,s_0 s)=f'_{G,\Xi}(\psi,s)$. On the other hand for each $y\in Z(\hat G^*)^\Gamma$, we have $f'_{G,\Xi}(\psi,y s)=f'_{G,\Xi}(\psi,s)$ by Lemma \ref{lem:central-action} and the exact sequence \eqref{eq:kotisolocglo}. This completes the proof as $\ol{\cS}_\psi=S_{\psi}/\srad_\psi Z(\hat G^*)^\Gamma$ by Lemma \ref{lem:squot}.

\end{proof}

  Once the local intertwining relation (Theorem \ref{thm:lir}) is established for localizations of $G$, the following should be an immediate corollary by taking product over all places.

\begin{thm}(Global intertwining relation)\label{thm:global-intertwining} Let $M^*$ be a proper Levi subgroup of $G^*$ and $\psi_{M^*}\in \Psi_2(M^*,\eta_\chi)$. Write $\psi$ for the image of $\psi_{M^*}$ in $\Psi(G^*)$. Assume that the local intertwining relation holds for the pair $(M^*,G^*)$ and $\psi_{M^*,v}$ for each place $v$ of $F$.
  Then for every pair of $\ol{u}\in \ol{\cN}_{\psi_{M^*}}(M^*,G^*)$ and $s\in S_{\psi_{M^*},\sspl}(G)$ having the same image in $\ol{\cS}_{\psi_{M^*}}(G^*)$,
  \begin{equation}\label{eq:GIR}f'_{G}(\psi,s_\psi s^{-1})=f_{G}(\psi_{M^*},\ol{u}),\quad f\in \cH(G).\end{equation}
  In particular $f_{G}(\psi_{M^*},\ol{u})$ depends only on $\psi$ rather than $\psi_{M^*}$ itself, and
    if $M^*$ does not transfer to a Levi subgroup of $G$ then $f'_{G}(\psi,s_\psi s)=0$.
\end{thm}

\begin{rem}
  Eventually \eqref{eq:GIR} will be valid even when $M^*=G^*$ and thus $\psi\in \Psi_2(G^*,\eta_\chi)$, but can be stated only after the local classification theorem is known for the definition of $f_{G}(\psi,\ol{u})$ is conditional on it; see the next subsection. The case $M^*\neq G^*$ is treated before the local classification theorem and indeed serves as an input for the proof.
\end{rem}

\begin{proof}
  The last assertion follows from Lemma \ref{l:relevance-and-Levi} and (the paragraph below) Theorem \ref{thm:lir}. It remains to verify \eqref{eq:GIR}.
  Since passing from diagram \eqref{eq:diagram-extended-pure-inner-global} to diagram \eqref{eq:diagram-mod-central-global} is functorial, we have a commutative diagram where both rows are short exact sequences.
$$\xymatrix{
  W^0_{\psi_{M^*}} \ar[r] \ar@{=}[d] & N^\natural_{\psi_{M^*}} \ar[r] \ar@{->>}[d] & S^\natural_{\psi_{M^*}} \ar@{->>}[d]\\
  W^0_{\psi_{M^*}} \ar[r]& \ol{\cN}_{\psi_{M^*}} \ar[r]  & \ol{\cS}_{\psi_{M^*}}
}$$
    A simple diagram chase enables us to choose a lift $u\in N^\natural_{\psi_{M^*}}(M^*,G^*)$ of $\ol{u}$ such that the image of $u$ in $S^\natural_{\psi_{M^*}}(M^*,G^*)$ is equal to that of $s$. Then the localization $u_v\in N^\natural_{\psi_{M^*,v}}(M^*,G^*)$ of $u$ has the same image as $s_v$ in $S^\natural_{\psi_{M^*,v}}(M^*,G^*)$. So the local intertwining relation is applicable and yields $f'_{v,G,\Xi}(\psi_{M^*,v},s_{\psi_{M^*}} s_v^{-1})=e(G_v)f_{v,G,\Xi}(\psi_{M^*,v},u_v)$. We conclude by taking the product over all $v$.
\end{proof}

  We have a preliminary result on the global intertwining relation analogous to \S\ref{sub:prelim-local-intertwining}. As in the local case we define elliptic parameters by the existence of a semisimple element $\ol{s}\in \ol{S}_{\psi_{M^*}}$ whose centralizer in $\ol{S}_{\psi_{M^*}}$ is finite. There is again a chain of inclusions for the sets of discrete, elliptic, and all parameters:
    $$\Psi_2(G^*,\eta_\chi)\subset \Psi_{\el}(G^*,\eta_\chi)\subset \Psi(G^*,\eta_\chi).$$
 Decompose $\psi_{M^*}$ as a formal sum of simple parameters:
  $$\psi_{M^*}=\ell_1\psi_1 \boxplus \cdots \boxplus \ell_r \psi_r,\quad r\ge1,~\ell_1\ge\cdots\ge \ell_r\ge 1.$$
  In parallel with the analogous local setup, our basic strategy is to study $\psi$ in the increasing order of difficulty: non-elliptic parameters, elliptic non-discrete parameters, and then discrete parameters. Elliptic non-discrete parameters have the form
\begin{equation}\label{eq:globla-ell-param}\psi=2\psi_1\boxplus \cdots\boxplus 2\psi_q \boxplus \psi_{q+1} \boxplus \cdots\boxplus \psi_r, ~~ S_\psi\simeq \O(2,\C)^q\times \O(1,\C)^{r-q}, ~~q\ge 1.\end{equation}
   Among non-elliptic parameters the following two exceptional cases, cf. (5.7.12) and (5.7.13) of \cite{Mok} (also see (4.5.11) and (4.5.12) of \cite{Arthur}), turn out to be the most difficult.
\begin{enumerate}
\item[(exc1)] $\psi=2\psi_1\boxplus\psi_2\boxplus \cdots \boxplus \psi_r$, $S_\psi\simeq Sp(2,\C)\times O(1)^{r-1}$,
\item[(exc2)] $\psi=3\psi_1\boxplus\psi_2\boxplus \cdots \boxplus \psi_r$, $S_\psi\simeq O(3,\C)\times O(1)^{r-1}$.
\end{enumerate}

 Define $\Psi_{\EXC1}(G^*,\eta_\chi)$ and $\Psi_{\EXC2}(G^*,\eta_\chi)$ to be the subsets of non-elliptic parameters of $\Psi(G^*,\eta_\chi)$ which have the form (exc1) and (exc2), respectively. Write $\Psi_{\EXC}(G^*,\eta_\chi):=\Psi_{\EXC1}(G^*,\eta_\chi)\coprod \Psi_{\EXC2}(G^*,\eta_\chi)$.
 We record some basic lemmas on exceptional and non-exceptional parameters.

 \begin{lem}\label{lem:global-exc}
  Suppose that $\psi$ is a non-elliptic non-exceptional parameter, i.e. $\psi\in \Psi^{\el}(G^*,\eta_\chi)\bs \Psi_{\EXC}(G^*,\eta_\chi)$. Then
  \begin{enumerate}
  \item  every simple reflection $w\in W^0_\psi$ centralizes a torus of positive dimension in $\ol{T}_\psi$ and
  \item  $\dim\ol{T}_{\psi,\ol{s}}\ge 1,~\forall \ol{s}\in \ol{S}_\psi$.
  \end{enumerate}
\end{lem}

\begin{proof}
  Except that the local setup is replaced with the global one, the proof of Lemma \ref{lem:local-exc} carries over word by word.
\end{proof}

  Let us define a subset $\ol{\cS}_{\psi,\el}\subset \ol{\cS}_\psi$ exactly as in the local case, i.e. as in the paragraph below Lemma \ref{lem:local-exc}.

\begin{lem}\label{lem:global-S-elliptic}
  If $\psi\in \Psi_{\EXC}(G^*,\eta_\chi)$ then $\ol{\cS}_{\psi,\el}=\ol{\cS}_\psi$.
\end{lem}

\begin{proof}
  The proof is identical to that of Lemma \ref{lem:S-elliptic}.
\end{proof}

 In the non-elliptic non-exceptional case the global intertwining relation follows essentially from the induction hypothesis. More precisely we have the following results analogous to those in the local case, cf. \S\ref{sub:prelim-local-intertwining}. The proofs are omitted as the arguments in that subsection carry over with no essential change. (Also the reader may compare with the details in the global setup as in \cite[\S4.5]{Arthur}.) We keep the same notation as in the global intertwining relation.

\begin{lem}\label{lem:exceptional-global-param} Let $M$ be a proper Levi subgroup of $G$. Let $\psi_{M^*}\in \Psi(M^*,\eta_\chi)$, and $\psi\in \Psi(G^*,\eta_\chi)$ be the image of $\psi_{M^*}$. Let $x\in \ol{\cS}_{\psi_{M^*}}(M^*,G^*)$. Assume that either
\begin{enumerate}
  \item $\psi$ is elliptic, or
  \item every simple reflection $w\in W^0_\psi$ centralizes a torus of positive dimension in $\ol{T}_{\psi_{M^*}}$.
\end{enumerate}
  Then $f_{G}(\psi_{M^*},u)$ is the same for every $u\in \ol{\cN}_{\psi_{M^*}}(M^*,G^*)$ mapping to $x$.

\end{lem}

\begin{proof}
  The same proof for Lemma \ref{lem:local-indep-f} works globally.
\end{proof}

\begin{lem}\label{lem:global-intertwining-step1}
  Let $M$, $\psi_{M^*}$, $\psi$ and $x$ be as in the preceding lemma. Then $f_{G,\Xi}(\psi_{M^*},u)=f'_{G,\Xi}(\psi,s_\psi s^{-1})$ whenever $u\in \ol{\cN}_{\psi_{M^*}}(M^*,G^*)$ and $s\in S_{\psi_{M^*},\sspl}$ map to $x$ unless
  \begin{itemize}
    \item  $\psi$ is elliptic and $x\in \ol{\cS}_{\psi_{M^*},\el}$, or
    \item  $\psi\in \Psi_{\EXC}(G^*,\eta_\chi)$.
  \end{itemize}
\end{lem}

\begin{proof}
  Just like Corollary \ref{cor:local-intertwining-step1}, the lemma follows from Lemmas \ref{lem:exceptional-global-param} and \ref{lem:global-intertwining-step1} as well as the global analogue of Lemma \ref{lem:local-intertwining-non-exceptional}, the latter being verified by the same argument as in the local case.
\end{proof}

  We define the subsets $W_{\psi,\reg}(M^*,G^*)\subset W_{\psi}(M^*,G^*)$ and $N^\natural_{\psi,\reg}(M^*,G^*)\subset N^\natural_{\psi}(M^*,G^*)$ as in the local case, cf. the paragraph below Corollary \ref{cor:local-intertwining-step1}. So $W_{\psi,\reg}(M^*,G^*)$ is the subset of $w$ with finitely many fixed points on $\ol T_\psi$, and $N^\natural_{\psi,\reg}(M^*,G^*)$ is the preimage of $W_{\psi,\reg}(M^*,G^*)$.

   Unless $\psi\in \Psi_{\EXC}(G^*,\eta_\chi)$, the above lemmas imply that $f_{G}(\psi_{M^*},u)$ depends only on the image of $u$ in $\ol{\cS}_{\psi_{M^*}}(M^*,G^*)$ so that $f_{G}(\psi_{M^*},x)$ is well-defined for $x\in \ol{\cS}_{\psi_{M^*}}(M^*,G^*)$. In the situation that $\psi\in \Psi_{\EXC}(G^*,\eta_\chi)$, there is no harm in restricting our attention to $u\in N^\natural_{\psi_{M^*}}$ such that $w_u\in W_{\psi_{M^*},\reg}(M^*,G^*)$, since only elements in $W_{\psi_{M^*},\reg}(M^*,G^*)$ (rather than the larger $W_{\psi_{M^*}}(M^*,G^*)$) contribute to the trace formula. (So the global analogue of Lemma \ref{lem:w_u-regular-enough} is not needed.) Similarly to the local setting, one checks the uniqueness of $u\in N^\natural_{\psi_{M^*}}$ in the fiber over $x$ such that $w_u\in W_{\psi_{M^*},\reg}(M^*,G^*)$. (As in the local case, the fiber over $x$ has two elements, only one of which lands in $W_{\psi_{M^*},\reg}(M^*,G^*)$ inside $W_{\psi_{M^*}}(M^*,G^*)$.) This allows us to take $f_{G}(\psi_{M^*},x)$ to be $f_{G}(\psi_{M^*},u)$ for the unique $u$ just mentioned, when analyzing the spectral side of the standard model below. The upshot is that the linear form
   $$f_{G}(\psi_{M^*},x) ~\mbox{has unequivocal meaning}$$
   in all cases appearing in the trace formula.

  We remarked that the global intertwining relation follows from a complete proof of the local intertwining relation. The strategy to obtain the latter is to embed the local problem in some simple global situation where the global intertwining relation is already known (as in Lemma \ref{lem:global-intertwining-step1}) or where a good approximation to the global intertwining relation can be derived from a comparison of the trace formulas, cf. \S\ref{sub:elliptic-parameters} and \S\ref{subsection_elliptic_parameters_2} below.

\subsection{The standard model}\label{sub:std-model}

  Here we initiate the comparison of the trace formulas by a procedure named the standard model by Arthur. This major global input constitutes the backbone of our argument.
  Recall that we have the spectral expansion \eqref{eq:I_disc-defn} and the endoscopic decomposition \eqref{eq:I_disc-stabilized} for $I_{\disc,\psi}(f)$ with $f\in \cH(G)$. To compare the two we need to decompose each of them further and reorganize the terms into a more amenable form. Let us start the investigation of \eqref{eq:I_disc-stabilized}. Thanks to the stable multiplicity formula \eqref{eq:stable-mult}, we may write \eqref{eq:I_disc-stabilized} as
\begin{equation}\label{eq:std-model-end1}
   I^G_{\disc,\psi,\eta_\chi}(f)=  \sum_{(G^\fke,\zeta^\fke)\in \ol\cE_{\el}(G)\atop \psi\in (\eta\zeta^\fke)_*\Psi(G^\fke,\eta\zeta^\fke)} \iota(G,G^\fke) |\cS_{\psi}|^{-1} \epsilon^{G^\fke}_\psi(s_{\psi}) \sigma(\ol{S}_{\psi}^0)f^{G^\fke}(\psi).
\end{equation}

  Recall from Lemma \ref{lem:independence-end} that the linear form $f'_{G,\Xi}(\psi,x)$ is well-defined for all $\psi\in \Psi(G^*,\eta_\chi)$ and $x\in \ol{\cS}_\psi$. Set $\ol\cE_{\psi,\el}:=\cE_{\el}(\ol{S}_\psi)$ and write $\cE_{\psi,\el}(x)$ for the fiber in $\cE_{\el}(S_\psi)$ over $x\in \ol{\cS}_\psi$. The definition of $e_\psi(x)$ in \S\ref{sub:stable-multiplicity} may be rephrased as
  \begin{equation}\label{eq:e(x)}
    e^{G^*}_\psi(x):=\sum_{s\in \ol\cE_{\psi,\el}(x)} |\pi_0(\ol{S}_{\psi,s})|^{-1} \sigma(\ol{S}^0_{\psi,s}).
  \end{equation}

\begin{prop}\label{p:std-model-end} For $\psi\in \Psi(G^*,\eta_\chi)$,
  $$I^G_{\disc,\psi}(f)=  |\ol{\cS}_\psi|^{-1}\sum_{x\in \ol{\cS}_\psi}e^{G^*}_\psi(x)\epsilon^{G^*}_\psi(x)f'_{G,\Xi}(\psi,s_\psi x^{-1}).$$
\end{prop}

\begin{proof}
  This follows exactly as in the case of quasi-split classical groups \cite[Cor 4.4.3]{Arthur}, noting that $f'_{G,\Xi}(\psi,x)$ depends only on $x$ as explained in \S\ref{sub:global-intertwining} and that the linear form $^0 s^G_{\disc,\psi^N}$ in loc. cit. vanishes by the stable multiplicity formula thanks to Proposition \ref{p:stable-multiplicity}. %
  The reason for the inverse on $x$ in the summand is explained by the fact that $f'_{G,\Xi}(\psi,s_\psi x^{-1})$ is equal to what Arthur wrote $f'_{G}(\psi,s_\psi x)$ in his book in general. This traces back to the difference in the normalization of the transfer factors: \cite{Arthur} adopts the Langlands-Shelstad definition \cite{LS87} but we specialize the definition in \cite{KS12} to untwisted endoscopy. Finally we remark that the endoscopic sign lemma for quasi-split unitary groups, cf. \S\ref{sub:results-qsuni}, is used in the proof.

\end{proof}

  The next step in the standard model proceeds in parallel with the endoscopic expansion above and aims to turn the spectral expansion \eqref{eq:I_disc-defn} into an expression amenable for comparison. We begin by relating $\tr (M_{P,\psi}\cI_{P,\psi}(\pi,f))$ to the global linear forms $f_{G}(\psi,x)$. As usual $\Psi(M^*,\psi,w)$ denotes the set of $\psi_{M^*}\in \Psi(M^*,\eta_\chi)$ such that $\psi_{M^*}$ maps to $\psi$. Define $\Psi(M^*,\psi,w)$ to be the subset of $\Psi(M^*,\psi,w)$ given by the condition that $w\psi_{M^*}=\psi_{M^*}$.

\begin{lem}\label{lem:M(w)-expansion-prelim} We have an equality
  $\tr (M_{P,\psi}(w)\cI_{P,\psi}(f))$
  $$ =\sum_{\psi_{M^*}\in \Psi(M^*,\psi,w)}|\ol\cS_{\psi_{M^*}}|^{-1} \sum_{u\in \ol\cN_{\psi_{M^*}}(w)}
  r_P(w,\psi_{M^*})\epsilon_{\psi_{M^*}}(x_u)f_G(\psi_{M^*},u).$$
\end{lem}

\begin{proof} We will only sketch the argument as it is essentially the same as that of \cite[Cor 4.2.4]{Arthur}. Following \cite[pp.180-181]{Arthur}, we expand $\tr (M_{P,\psi}(w)\cI_{P,\psi}(f))$ as a sum
$$\sum_{\psi_{M^*}\in \Psi(M^*,\psi,w)} r_P(w,\psi_{M^*}) \sum_{\pi_M\in \Pi_{\psi_{M^*}}} m_{\psi_{M^*}}(\pi_M)\tr(R_P(w,\pi_M,\psi_{M^*})\cI_P(\pi_M,f)),$$
where $m_{\psi_{M^*}}(\pi_M)\in \Z_{\ge 0}$ denotes the multiplicity of $\pi_M$ in $L^2_{\disc,\psi_{M^*}}(M(F)\bs M(\A_F))$ (defined as on p.180 of loc. cit). Theorem \ref{thm:main-global} for $M$, available as part of the induction hypothesis, tells us that $\lg \pi_M,\cdot\rg_{\xi}=\epsilon_{\psi_{M^*}}$ as the character of the group $\ol{\cS}_{\psi_{M^*}}(M^*)$ if and only if $\pi_M$ is automorphic, in which case $m_{\psi_{M^*}}(\pi_M)=1$. Note that there is a canonical isomorphism $\ol{\cS}_{\psi_{M^*}}(M^*)\simeq \ol{\cS}_{\psi}(M^*)$, the latter defined as in \S\ref{sub:global-intertwining}. (It suffices to observe that $Z_{S_\psi}=S_{\psi_{M^*}}(M^*)$ and that $S^0_{\psi_{M^*}}(M^*)Z(\hat G^*)^\Gamma=S^0_{\psi_{M^*}}(M^*)Z(\hat M^*)^\Gamma$. The latter follows from $Z(\hat M^*)^\Gamma=(Z(\hat M^*)^\Gamma)^0 Z(\hat G^*)^\Gamma$, which is in turn deduced from the second display in the proof of Lemma \ref{lem:rhofact}, for instance.) In view of the definition of $f_G(\psi,u)$, the proof boils down to the claim that
\begin{equation}\label{eq:M(w)-expansion}m_{\psi_{M^*}}(\pi_M)\tr(R_P(w,\pi_M,\psi_{M^*})\cI_P(\pi_M,f))\end{equation}
$$=|\ol{\cS}_{\psi_{M^*}}(M^*)|^{-1}\sum_{u\in \ol{\cN}_{\psi_{M^*}}(w)} \epsilon_{\psi_{M^*}}(u)\tr (R_P(u,\Xi,\pi_M,\psi_{M^*},\psi_F)\cI_P(\pi_M,f)).$$
This is immediate from Part 2 of Proposition \ref{pro:ioplg} if $\pi_M$ is automorphic. Otherwise the left hand side vanishes, so we need to show that the right hand side is also zero. Fixing $u_0\in \ol{\cN}_{\psi_{M^*}}(w)$ and writing $u=yu_0$ for elements of $\ol{\cN}_{\psi_{M^*}}(w)$, we can rewrite the right hand side as $|\ol{\cS}_{\psi_{M^*}}(M^*)|^{-1}\epsilon_{\psi_{M^*}}(u_0)$ times
$$\tr(R_P(u_0,\Xi,\pi_M,\psi_{M^*},\psi_F)\cI_P(\pi_M,f)) \left(\sum_{y\in \ol{\cS}_{\psi_{M^*}}(M^*)} \epsilon_{\psi_{M^*}}(y)\lg \pi_M,y\rg_{\xi}\right)$$
by part 1 of Proposition \ref{pro:ioplg}. Since the bracketed term vanishes, we are done.

\end{proof}

  Let $\psi\in \Psi(G^*,\eta_\chi)$. It is convenient to introduce a linear form
  $$^0 r^G_{\disc,\psi}(f),\quad \psi\in \Psi(G^*,\eta_\chi),~f\in \cH(G),$$
  which will be shown to be zero eventually. If $\psi\notin \Psi_2(G^*,\eta_\chi)$ then set
  $$^0 r^G_{\disc,\psi}(f):=\tr R^G_{\disc,\psi}(f).$$
  In this case recall from \S\ref{sub:global-intertwining} (especially the last paragraph there) that the linear form $f\mapsto f_{G}(\psi,x)$ on $\cH(G)$ is well-defined for all $x\in \ol{\cS}_\psi$.
  Now consider $\psi\in \Psi_2(G^*,\eta_\chi)$. Let us introduce a hypothesis on $\psi$ which will be verified before a complete proof of global theorems:
  \begin{hypo}\label{hypo:Hyp(psi)}
    The local classification theorem for $\psi_v$ is proven for every $v$ (so that $\Pi_\psi=\Pi_\psi(G,\xi)$ is defined, cf. \S\ref{sub:main-global-thm}, and also that $\lg x,\pi \rg$ is defined for all $x\in \ol{\cS}_\psi$ and $\pi\in \Pi_\psi$).%
  \end{hypo}
    Choose an extended pure inner twist $\Xi$ giving rise to $(G,\xi)$. Under the above hypothesis one makes sense of the definition (which is the $M=G$ case of \eqref{eq:f_G(psi,u)-definition})
   $$f_{G}(\psi,x):=\sum_{\pi\in \Pi_\psi} \lg x,\pi\rg f_{G}(\pi)=\sum_{\pi\in \Pi_\psi}\left(\prod_v \lg x,\pi_v\rg f_{G}(\pi_v)\right).$$
   Even though we assumed $\psi\notin \Psi_2(G^*,\eta_\chi)$ and $M^*\neq G^*$ in \S\ref{sub:global-intertwining}, the argument for Lemma \ref{lem:independence-f_G} still shows (appealing to Lemma \ref{lem:central-action} in view of Remark \ref{lem:central-action} and noting that $ \ol{\cN}_\psi= \ol{\cS}_\psi$ when $M^*=G^*$) that $f_{G}(\psi,x)$ is well-defined independently of the choice of $\Xi$.
    Put\footnote{Compare with \cite[(5.5.16)]{Mok} and \cite[(4.3.7)]{Arthur}. In their notation we always have $\kappa_G=m_\psi=1$ and $C_\psi$ there is equal to our $|\cS_\psi|^{-1}$.}
     \begin{equation}\label{eq:0rG}
       ^0 r^G_{\disc,\psi}(f):=\tr R^G_{\disc,\psi}(f)- |\ol{\cS}_\psi|^{-1} \sum_{x\in \ol{\cS}_\psi} \epsilon^{G^*}_\psi(x) f_{G}(\psi,x).
     \end{equation}
 Write $W_{\psi}(x)$ for the image in $W_\psi$ of the subset of $\ol{\cN}$ which is the fiber over $x\in \ol{\cS}_\psi$. Set $W_{\psi,\reg}:=W_{\psi}\cap W_{\reg}(M)$ and $W_{\psi,\reg}(x):=W_{\psi}(x)\cap W_{\psi,\reg}$. Then the definition of $i_\psi(x)$ in \S\ref{sub:stable-multiplicity} was that
  \begin{equation}\label{eq:i(x)}
    i_\psi(x):= |W^0_\psi|^{-1}\sum_{w\in W_{\psi,\reg}(x)} \sgn^0_\psi(w)|\det(w-1)|^{-1}.
  \end{equation}
  The spectral expansion in the non quasi-split case below looks identical to \cite[Cor 4.3.3]{Arthur} in the quasi-split case. The difference is that our $f_{G}(\psi,x)$ is zero by definition if $\psi$ is not a relevant parameter for $G$. (In contrast, every parameter is relevant for $G^*$.)

\begin{prop}\label{p:std-model-spec} Let $\psi\in \Psi(G^*,\eta_\chi)$. Assume Hypothesis \ref{hypo:Hyp(psi)} on $\psi$ if $\psi\in \Psi_2(G^*,\eta_\chi)$ (but no hypothesis if $\psi\notin \Psi_2(G^*,\eta_\chi)$). Then
  $$I^G_{\disc,\psi}(f)={}^0 r^G_{\disc,\psi}(f) + |\ol{\cS}_\psi|^{-1}\sum_{x\in \ol{\cS}_\psi}i^{G^*}_\psi(x)\epsilon^{G^*}_\psi(x)f_{G}(\psi,x),$$

\end{prop}

\begin{proof}

  We argue as on \cite[pp.186-188]{Arthur}, appealing to our Lemma \ref{lem:M(w)-expansion-prelim} in place of his Corollary 4.2.4.  Then we have an equality of $I_{\disc,\psi}(f)-\tr R_{\disc,\psi}(f)$ with (see the top paragraph on page 187, loc. cit.)
 $$\sum_{M\neq G}|W(M)|^{-1}\sum_{w\in W_{\reg}(M)}|\det(w-1)_{\fka^G_M}|^{-1}$$
  $$\times \sum_{\psi_{M^*}\in\Psi(M^*,\psi,w)}\sum_{u\in \ol\cN_\psi(w)} |\ol\cS_{\psi_{M^*}}|^{-1}r_P(w,\psi_{M^*}) \epsilon_{\psi_{M^*}}(u) f_G(\psi_{M^*},u).$$
  Of course $M^*$ in the inner sum denotes a (standard) Levi subgroup of $G^*$ which is a quasi-split inner form of $M$. We may replace the first sum with the sum over proper Levi subgroups $M^*$ of $G^*$. Indeed if $M^*$ does not transfer to $G$ locally everywhere then $f_G(\psi_{M^*},u)=0$ by definition so the sum does not change.

  To proceed we argue as on page 187 of \cite{Arthur}. Recall that his set $V_\psi$ is the set of pairs $(w,\psi_{M^*})\in W_{\reg}(M)\times \Psi(M^*,\psi)$ such that $w\psi_{M^*}=\psi_{M^*}$. The first projection induces a map from $V_\psi$ to $\{W_{\psi,\reg}\}$, the set of conjugacy classes in $W_{\psi,\reg}$. The key point in his argument is to see that the last summand in the display above is constant on each fiber of the map $V_\psi\ra\{W_{\psi,\reg}\}$. This point is rather easy to check from the definitions except for the factor $f_G(\psi_{M^*},u)$, where we appeal to Lemma \ref{lem:linear-form-ind-of-M}. The rest of Arthur's argument on that page remains the same and shows that the display above is equal to
  $$ |W_{\psi}|^{-1} \sum_{w\in W_{\psi_{M^*},\reg}}\sum_{u\in \ol{\cN}_{\psi_{M^*}}(w)} \frac{ r_P(w,\psi_{M^*}) \epsilon_{\psi_{M^*}}(u) f_G(\psi_{M^*},u)}{  |\det(w-1)_{\fka^G_M}||\ol\cS_{\psi_{M^*}}|},$$
    which is the same as \cite[(4.3.2)]{Arthur} (except that we keep $\psi_{M^*}$ in our notation), as an intermediate result.\footnote{We remind the reader that we cannot write $f_G(\psi_{M^*},u)$ as $f_G(\psi,u)$ until we know that $f_G(\psi_{M^*},u)$ has the same value for every $\psi_{M^*}\in \Psi(M^*,\psi)$, which would only follow from the intertwining relation.}
    From here on Arthur's argument applies without change. Thereby we arrive at the equality in the proposition. %
     We merely remark that the spectral sign lemma for quasi-split unitary groups, cf. \S\ref{sub:results-qsuni}, is used in the proof.

\end{proof}

  Now the comparison of the endoscopic and spectral expansions leads to the following formulas.

\begin{cor}\label{cor:std-model-result}
  For a non-discrete parameter $\psi\in \Psi^2(G^*,\eta_\chi)$ and $f\in \cH(G)$,
  $$\tr R^G_{\disc,\psi}(f)=|\ol{\cS}_\psi|^{-1}\sum_{x\in \ol{\cS}_\psi}i^{G^*}_\psi(x)\epsilon^{G^*}_\psi(x)(f'_{G}(\psi,s_\psi x^{-1})-f_{G}(\psi,x)).$$
  When $\psi\in \Psi_2(G^*,\eta_\chi)$, assuming Hypothesis \ref{hypo:Hyp(psi)} we have for $f\in \cH(G)$ that
   $$^0 r^G_{\disc,\psi}(f)=|\ol{\cS}_\psi|^{-1}\sum_{x\in \ol{\cS}_\psi}i^{G^*}_\psi(x)\epsilon^{G^*}_\psi(x)(f'_{G}(\psi,s_\psi x^{-1})-f_{G}(\psi,x)).$$
\end{cor}

\begin{proof}
   The corollary is an immediate consequence of \eqref{eq:i(S)=e(S)} as well as Propositions \ref{p:std-model-end} and \ref{p:std-model-spec}.
\end{proof}

\subsection{On global non-elliptic exceptional parameters}\label{sub:non-elliptic-exceptional}

  The global intertwining relation for a non-elliptic parameter $\psi\in \Psi^{\el}(G^*,\eta_\chi)$ was essentially deduced from the induction hypothesis in Lemma \ref{lem:global-intertwining-step1} except when $\psi\in \Psi_\EXC(G^*,\eta_\chi)$. In the exceptional case we will prove a weaker identity by the comparison of the trace formulas. This will provide a global input for the later proof of the local intertwining relation in (exc1) and (exc2).

\begin{lem}\label{lem:EXC1-EXC2}
  Let $\psi\in \Psi_{\EXC}(G^*,\eta_\chi)$, i.e. it is a non-elliptic parameter of form (exc1) or (exc2). Suppose that $\psi$ comes from a discrete parameter on a proper Levi subgroup $M^*$ of $G^*$. Let $x\in \ol{\cS}_\psi$. If $M^*$ does not transfer locally everywhere to $G^*$ then $$f'_{G}(\psi,s_\psi x^{-1})=f_{G}(\psi,x)=0,\quad f\in \cH(G).$$ If $M^*$ does transfer to a Levi subgroup $M$ of $G$ then
  $$ \sum_{x\in \ol{\cS}_{\psi}} \epsilon^{G^*}_\psi(x)(f'_{G}(\psi,s_\psi x^{-1})-f_{G}(\psi,x))=0,\quad f\in \cH(G)$$
  and $R_P(w,\pi_M,\psi_{M^*})=1$.
\end{lem}

\begin{proof}
  Our proof proceeds as in the proof of \cite[Cor 4.5.2]{Arthur} but is simpler in that no groups other than $G$ need to be considered. So we outline the argument with small details omitted. %

  Using the explicit form of (exc1) and (exc2), we find that $R_\psi$ is trivial, that $|W^0_\psi|=|W_\psi|=2$, and that $W_{\psi,\reg}$ has a unique element $w$. For the latter $w$ we have $\sgn^0_\psi(w)=-1$ and $|\det(w-1)|=2$. Hence $i_\psi(x)=-1/4$ in the first identity of Corollary \ref{cor:std-model-result}. The identity reads
    \begin{equation}\label{eq:std-model-EXC1-EXC2}
      \tr R^G_{\disc,\psi}(f)=-\frac{1}{4|\ol{\cS}_\psi|}\sum_{x\in \ol{\cS}_\psi}\epsilon^{G^*}_\psi(x)(f'_{G}(\psi,s_\psi x^{-1})-f_{G}(\psi,x)),\quad f\in \cH(G).
    \end{equation}
  Fix an element $\psi_{M^*}\in \Psi_2(M^*,\eta_\chi)$ which maps to $\psi$. We argue as in the proof of Lemma \ref{lem:local-indep-f'}, noting that $M_s^*=M^*$ in our special case. The argument there shows that if $M^*$ does not transfer to $G$ locally everywhere then $f'_{G}(\psi,s_\psi x^{-1})=0$, and if $M^*$ does transfer then $f'_{G}(\psi,s_\psi x^{-1})=f'_{M}(\psi_{M^*},s_\psi s^{-1})$. In the former case $f_G(\psi,x)=0$ by definition, so the proof is finished. From now on we assume that $M^*$ transfers to a Levi subgroup $M$ of $G$ by Lemma \ref{lem:transfer-Levi-locglo}. Then a unique inner twist $(M,\xi_M)$ is determined on $M$.
  Arguing as at the top of \cite[p.210]{Arthur}, we obtain
  $$|\ol{\cS}_\psi|^{-1}\sum_{x\in \ol{\cS}_\psi}\epsilon^{G^*}_\psi(x)f'_{G}(\psi,s_\psi x^{-1})
  = \sum_{\pi_M\in \Pi_{\psi_{M^*}}}  m_{\psi_{M^*}} (\pi_M) \tr (\cI_P(\pi_M,f)).$$
  We have $|\ol{\cS}_\psi|=|\ol{\cS}_{\psi_{M^*}}|$ from the diagram \eqref{eq:diagram-mod-central-global} and the fact that $\ol{\cS}_\psi=\ol{\cS}_{\psi}(M^*,G^*)$, $\ol{\cS}_{\psi_{M^*}}=\ol{\cS}_{\psi}$ (cf. proof of Lemma \ref{lem:M(w)-expansion-prelim}), and $|W^0_\psi|=|W_\psi|=2$. The surjection $\ol{\cN}_\psi\ra \ol{\cS}_\psi$ restricts to a bijection $\ol{\cN}_\psi(w)\ra \ol{\cS}_\psi$, and we have
  $$\sum_{x\in \ol{\cS}_\psi}\epsilon^{G^*}_\psi(x)f_{G}(\psi, x)=\sum_{u\in \ol{\cN}_\psi(w)}\sum_{\pi_M\in \Pi_{\psi_{M^*}}} \epsilon^{G^*}_\psi(u) \tr (R_P(u,\Xi,\pi_M,\psi_{M^*},\psi_F)\cI_P(\pi_M,f)).$$
  Via \eqref{eq:M(w)-expansion}, this equals
  $$ \sum_{\pi_M\in \Pi_{\psi_{M^*}}} m_{\psi_{M^*}} (\pi_M) \tr (R_P(w, \pi_M,\psi_{M^*})\cI_P(\pi_M,f)).$$
  Therefore
  $$\tr R^G_{\disc,\psi}(f)+ \frac{1}{4} \sum_{x\in \ol{\cS}_\psi} \tr ( (1-R_P(w,\pi_M,\psi_{M^*}))\cI_P(\pi_M,f))=0.$$
  Since $w$ is an element of the group $W_\psi$ of order 2, the square of $R_P(w, \pi_M,\psi_{M^*})$ is the identity map. This implies that $\tr ( (1-R_P(w,\pi_M,\psi_{M^*}))\cI_P(\pi_M,f))$ is a nonnegative linear combination of the traces of irreducible representations (as a linear form in $f$), therefore that the whole left hand side is such a nonnegative linear combination. Hence we find that $R^G_{\disc,\psi}=0$ and that the summand for each $x$ is zero. The lemma results from this vanishing and \eqref{eq:std-model-EXC1-EXC2}.

\end{proof}

\subsection{On global elliptic parameters I}\label{sub:elliptic-parameters}

  For global elliptic parameters we have made little progress toward the global intertwining relation due to the fact that the induction hypothesis was not of any immediate help. As in the previous section \ref{sub:non-elliptic-exceptional}, our plan is to settle for a weaker identity for now, which is still good enough for deriving the local intertwining relation for local elliptic parameters later.

  Any elliptic non-discrete parameter $\psi\in \Psi^2_{\el}(G^*,\eta_\chi)$ has the following form, where $\psi_i$ are mutually non-isomorphic simple parameters, cf. \cite[(5.2.4)]{Arthur}, \cite[(6.2.1)]{Mok}:
  $$\psi=2\psi_1\boxplus \cdots \boxplus 2\psi_q \boxplus \psi_{q+1} \boxplus \cdots \boxplus \psi_r,$$
  $$ S_\psi\simeq O(2,\C)^q \times O(1,\C)^{r-q},\quad q\ge 1.$$

\begin{lem} %
 Let $\psi\in \Psi^2_{\el}(G^*,\eta_\chi)$. For every $f\in \cH(G)$,
$$ \tr R^G_{\disc,\psi}(f) = 2^{-q}|\ol{\cS}_\psi|^{-1}\sum_{x\in \ol{\cS}_{\psi,\el}} \epsilon^{G^*}_\psi(x)(f'_{G}(\psi,s_\psi x^{-1})-f_{G}(\psi,x)).$$

\end{lem}

\begin{proof}
  An easy observation from the definition of $i_\psi$ is that $i_\psi(x)=0$ if $x\notin \ol{\cS}_{\psi,\el}$. In case $x\in \ol{\cS}_{\psi,\el}$ it is easily computed again from the definition that $i_\psi(x)=2^{-q}$ as appearing in the proof of Lemmas 5.2.1 and 5.2.2 of \cite{Arthur}. Now the lemma is just a special case of Corollary \ref{cor:std-model-result}.

\end{proof}

\section{Chapter 4: Globalizations and the local classification}\label{chapter4}

  The aim of this chapter is to establish the local intertwining relation (Theorem \ref{thm:lir}) and the local classification theorem (Theorem \ref{thm:locclass-single}) for generic parameters. In both proofs we utilize the powerful global method based on the trace formula. To this end we embed the local setup of each theorem in amenable global situations by exploiting a large degree of freedom in choosing the global data.
  After constructing globalizations up to \S\ref{sub:Globalization-param} we prove the local intertwining relation in \S\ref{subsection_elliptic_parameters_2} and \S\ref{sub:LIR-proof} and the local classification theorem in the last three subsections.

  We could have constructed global extended pure inner twists in the globalizations but the reader will notice that we work with pure inner twists instead. The main reason is to keep the results in this paper as unconditional as possible, since certain facts about the transfer factors needed in our work are available for pure inner twists but not for extended pure inner twists at the moment (see the discussion of \S\ref{subsub:normtf}).

\subsection{Globalization of the group}

When $\dot F$ is a totally real number field and $\dot E$ is a totally imaginary quadratic extension of $\dot F$, we say
that $\dot E/\dot F$ is a \textbf{CM extension} (of number fields) for simplicity.

\begin{lem}
\label{lem_globalization_field}
Let $E/F$ be a quadratic extension of local fields of characteristic $0$.
Let $r_0 \in \N$.
Then there exists a CM extension of number fields  $\dot{E}/\dot{F}$ and a place $u$ of $\dot{F}$ such that
\begin{itemize}
	\item $\dot{F}$ has at least $r_0$ real places,
	 and
	\item $\dot{E}_u/\dot{F}_u = E/F$.
\end{itemize}
\end{lem}
\begin{proof}
This is standard and easily deduced from \cite[Lem 6.2.1]{Arthur}.
\end{proof}

\begin{lem}
\label{lem_globalization_gp}
Let $E/F$, $\dot{E}/\dot{F}$ and $u$ be as in Lemma \ref{lem_globalization_field}.
Let $v_{2} \neq u$ be a non-split place of $\dot{F}$ in $\dot{E}$.
Let $(G,\xi)$ be an inner twist of $U_{E/F}(N)$.
Then there exists
$(\dot{G},\dot \xi)$ an inner twist of $\dot G^*=U_{\dot{E}/\dot{F}}(N)$ such that
\begin{enumerate}
	\item $(\dot{G}_u,\dot{\xi}_u) = (G,\xi)$, and
	\item $\dot{G}_v$ is quasi-split for all $v \not\in \{u, v_2\}$.
\item
if $N$ is odd, we can also assume that $\dot{G}_{v_2}$ is quasi-split, and
\item
if $N$ is even and $v_2$ is an archimedean place, then we can assume that
 $\dot{G}_{v_2}$ is isomorphic to either $U_{\dot{E}_{v_2}/\dot{F}_{v_2}}(N/2, N/2)$ or
$U_{\dot{E}_{v_2}/\dot{F}_{v_2}}(N/2 - 1, N/2 + 1)$.
\end{enumerate}
  Moreover if $(G,\xi,z)$ is a pure inner twist of $U_{E/F}(N)$ (in fact every $(G,\xi)$ can be promoted to a pure inner twist since the top horizontal arrow is onto in the diagram of \S\ref{subsub:inner-U}) then there exists a pure inner twist $(\dot G,\dot \xi,\dot z)$ of $U_{\dot E/\dot F}(N)$ such that
  \begin{enumerate}
    \item $(\dot{G}_u,\dot{\xi}_u,\dot z_u) = (G,\xi,z)$,
    \item $(\dot{G}_v,\dot\xi_v,\dot z_u)$ is the trivial pure inner twist (in particular $\dot G_v$ is quasi-split) for all $v \not\in \{u, v_2\}$, and
    \item condition 3 above holds true verbatim.
  \end{enumerate}

\end{lem}
\begin{proof}
  It suffices to check the assertion about pure inner twists. From \cite[Cor 2.5, Prop 2.6]{Kot86} we get an exact sequence
  $$H^1(F,\dot G^*)\ra \bigoplus_v H^1(F_v,\dot G^*)\stackrel{\sum_{\alpha_v}}{\ra} \pi_0(Z(\hat {\dot G}^*)^\Gamma)^D \ra 0,$$
  where $D$ denotes the Pontryagin dual. Note that $\pi_0(Z(\hat {\dot G}^*)^\Gamma)^D\simeq \Z/2\Z$ and the map $\alpha_v:H^1(F_v,\dot G^*)\ra \pi_0(Z(\hat {\dot G}^*)^\Gamma)^D $ is identified with the left vertical map in the diagram of \S\ref{subsub:inner-U}). Since $u$ and $v_2$ do not split in $\dot E$ we see that the map $H^1(F_v,\dot G^*)\ra \Z/2\Z$ is onto for $v\in \{u,v_2\}$. Now choose an element $(G_v,\xi_v,z_v)$ at each place $v$ such that $(G_u,\xi_u,z_u)=(G,\xi,z)$, $\alpha_u(G_u,\xi_u,z_u)=\alpha_{v_2}(G_{v_2},\xi_{v_2},z_{v_2})$, and $(G_v,\xi_v,z_v)$ is trivial for $v\not\in \{u,v_2\}$. If $N$ is even and $v_2$ is archimedean then we can impose condition 4 because the real unitary groups $U(N/2,N/2)$ and $U(N/2-1,N/2+1)$ have different images under $\alpha_{v_2}$. By the exact sequence there exists $(\dot G,\dot \xi,\dot z)\in H^1(F,\dot G^*)$ localising to $(G_v,\xi_v,z_v)$ at every $v$, thus satisfying all the conditions of the lemma.
\end{proof}

We shall also occasionally have need of the following globalization, which is easily obtained via the same process.
\begin{lem}
\label{lem_globalization_field_gp_double_places}
Let $E/F$ be a quadratic extension of local fields of characteristic $0$.
Let $r_0 \in \N$.  Let  $\chi \in \mathcal{Z}^\kappa_E$.  Let $(G,\xi)$ be
an inner twist of $U_{E/F}(N)$.
 Then  there exists a CM extension
$\dot{E}/\dot{F}$, two places $u_1$, $u_2$ of $\dot F$,
$\dot{\chi} \in \mathcal{Z}^\kappa_{\dot{E}}$, and
$(\dot{G},\dot{\xi})$ an inner twist of $U_{\dot{E}/\dot{F}}(N)$
such that
\begin{itemize}
\item $\dot{F}$ has at least $r_0$ real places,
\item $\dot{E}_u/\dot{F}_u = E/F$ for $u=u_1,u_2$,
\item $\dot{\chi}_u = \chi$ for $u = u_1, u_2$,
\item  $(\dot{G}_u,\dot\xi_u) = G$ for $u=u_1, u_2$, and
\item $\dot{G}_v$ is quasi-split for all $v \not\in \{u_1, u_2\}$.
\end{itemize}
Moreover we can promote $(\dot{G},\dot{\xi})$ to a pure inner twist $(\dot{G},\dot{\xi},\dot z)$.
\end{lem}

\subsection{Globalization of the representation}
In order to be able to globalize a  parameter, we must first globalize individual representations.
The globalization is carried out on quasi-split unitary groups.

\begin{lem}
\label{lemma_globalization_of_the_representation_1}
Let $\dot{E}/\dot{F}$ be a CM extension of number fields such that $\dot{F}$ has at least two archimedean places.
Let  $\dot{G}^* = U_{\dot{E}/\dot{F}}(N)$ denote the associated quasi-split unitary group.
Let $V$ be a finite set of places of $\dot{F}$ that do not split in $\dot{E}$ and assume that at least one archimedean place $v_\infty$ of $\dot F$ is not contained in $V$.
For all $v \in V$, let $\pi_v \in \Pi_{2, \mathrm{temp}}(\dot{G}^*_v)$ be a square integrable representation.
Then there exists a cuspidal automorphic representation $\dot{\pi}$ of $\dot{G}^*(\A_{\dot{F}})$ such that for all $v \in V$,
$\dot{\pi}_v = \pi_v$ and $\dot\pi_{v_\infty}$ has sufficiently regular infinitesimal character.
\end{lem}
\begin{proof}
The result will follow from an application of Theorem 4.8 of \cite{Shi12}. Choose a discrete series representation $\pi_{v_\infty}\in \Pi_{2, \mathrm{temp}}(\dot{G}^*_{v_\infty})$ whose infinitesimal character is sufficiently regular.
 In the notation of \cite{Shi12}, we take
$S=V\cup \{v_\infty\}$ and  $\hat{f}_S = \otimes_{v \in S} \hat{f}_v$ where $\hat f_v$ is taken to be the characteristic
function on the singleton set $\{ \pi_v\} \subset \widehat{\dot{G}^*(F_v)}$ (which has positive Plancherel measure i.e. $\widehat{\mu}_S^{\mathrm{pl}}(\hat{f}_S) >0$, since the center of each $\dot{G}^*(F_v)$ is compact).
Theorem 4.8 of \cite{Shi12}
implies that
\[
	\lim_{n\ra\infty} \widehat{\mu}_n(\hat{f}_S) = \widehat{\mu}_S^{\mathrm{pl}}(\hat{f}_S) > 0
\]
This shows that for some $n\in\Z_{\ge1}$, $\widehat{\mu}_n(\hat{f}_S) \neq 0$.
Unwinding the definition gives
the existence of $\dot{\pi}$ such that $\hat{f}_v(\dot\pi_v)\neq 0$ for each $v\in S$, so the lemma follows.
\end{proof}

We shall often have need of the following strengthening of  Lemma \ref{lemma_globalization_of_the_representation_1}.
\begin{lem}
\label{lemma_globalization_of_the_representation_2}
Let $\dot{E}/\dot{F}$, $\dot{G}^*$, $V$ be as in Lemma \ref{lemma_globalization_of_the_representation_1}.
For all $v \in V$, let
$M^*_v$ be a Levi subgroup of $\dot{G}^*_v$ such that $M^*_v=\dot G^*_v$ if $v$ is archimedean. For each $v\in V$ let
$\pi_{M^*_v} \in \Pi_{2, \mathrm{temp}}(M^*_v)$.
Then there exists a cuspidal automorphic representation $\dot{\pi}$ of $\dot{G}^*(\A_{\dot{F}})$ such that for all $v \in V$,
if $M^*_v=\dot G^*_v$ then $\dot\pi_v=\pi_{M^*_v}$ and if $M^*_v\neq \dot G^*_v$ then
$\dot{\pi}_v$ is an irreducible subquotient of the  induced representation
 $\cI_{M^*_v}^{\dot{G}^*_{v}}(\pi_{M^*_v} \otimes \chi)$
for some unramified unitary character $\chi \in \Psi_{u}(M^*_v)$ (cf. \cite[\S 2.3]{Shi12}). %
\end{lem}
\begin{proof}
The result follows from an application of \cite[Thm 4.8]{Shi12} in the same
 way as
Lemma \ref{lemma_globalization_of_the_representation_1} if we take
  $\hat{f}_v$ to be the characteristic function on the set of irreducible subquotients of the
induced representations
  $\cI_{M^*_v}^{\dot{G}^*_{v}}(\pi_{M^*,v} \otimes \chi)$
where $\chi$ runs through the unramified unitary characters $\chi \in \Psi_{u}(M^*_v)$ (cf. Example 5.6 of \cite{Shi12}).
\end{proof}

\begin{lem} \label{lem:globalize_for_iop1}
Let $E/F$ be a quadratic extension of local fields, $\xi : U_{E/F}(N) \rw G$ an inner twist, $M^* \subset U_{E/F}(N)$ a standard Levi subgroup such that $M:=\xi(M^*)$ is defined over $F$, and $\pi \in \Pi_2(M)$ a discrete series representation.

Then there exists a CM-extension $\dot E/\dot F$ and two places $u,v$ of $\dot F$ that do not split in $E$, with $v$ archimedean, and $\dot E_u/\dot F_u \cong E/F$. There exists furthermore an inner twist $\dot\xi : U_{\dot E/\dot F}(N) \rw \dot G$ with $\dot G$ quasi-split away from $u$ and $v$, as well as an isomorphism $\iota : \dot G_u \rw G$ such that $\xi = \iota \circ \dot\xi_u$. Moreover, if $\dot M^* \subset U_{\dot E/\dot F}(N)$ is the standard Levi subgroup with $\dot M^*_u=M^*$, then $\dot M := \dot\xi(\dot M^*)$ is defined over $F$ and $\iota$ provides an isomorphism $\dot M_u \rw M$.

Finally, there exists a discrete automorphic representation $\dot\pi$ of $\dot M$ such that $\iota$ identifies $\dot\pi_u$ with $\pi$.
\end{lem}

\subsection{Globalization of a simple parameter}
We shall globalize a simple parameter. It is worth remembering that a global parameter can be localized thanks to Proposition \ref{prop:2nd-seed-thm}. We shall perform two different globalizations both of which shall later be required.

 For every CM extension $\dot E/\dot F$ we fix characters $\chi_+\in \cZ^+_{\dot E}$ and $\chi_-\in \cZ^-_{\dot E}$. As usual $\dot\eta_\lambda:{}^L U_{\dot E/\dot F}(N)\hra {}^L G_{\dot E/\dot F}$ denote the $L$-morphism $\dot\eta_{\chi_\lambda}$ for each $N\ge 1$ and each $\lambda\in \{\pm 1\}$. Fix a sign $\kappa\in \{\pm1\}$ once and for all. For simplicity we often write just $\dot \eta$ for  $\dot\eta_\kappa$.

\begin{lem}
\label{lemma_globalization_of_simple_parameter_1}
Let $\dot{E}/\dot{F}$ be a CM extension such that $\dot{F}$ has at least two archimedean places.
Let $\dot{G}^{*}:=U_{\dot E/\dot F}(N)$ and $\dot\eta=\dot\eta_\kappa$ as above.
Let $V$ be a finite set of  places of $\dot{F}$ that do not split in $\dot{E}$.  For all $v \in V$, let
 $\phi_v \in \Phi_{2, \mathrm{bdd}}(\dot{G}^*_{v})$ be a square integrable parameter.
Furthermore, assume that for at least one $v \in V$, $ \phi_v \in {\Phi}_{\mathrm{sim}}(\dot{G}^*_v)$.
Then there exists a simple generic global parameter $\dot{\phi} \in \Phi_{\mathrm{sim}}(\dot{G}^*, \dot\eta)$ such that
for all $v \in V$,
\[
	\dot{\phi}_v = \phi_v.
\]
\end{lem}
\begin{proof}
For  each $v \in V$, we choose  $\pi_v \in \Pi_{\phi_v}$, a square integrable representation of
the quasi-split unitary group  $\dot{G}^{*}(F_v)$
in the $L$-packet associated to $\phi_v$.
Applying Lemma  \ref{lemma_globalization_of_the_representation_1},
we obtain a cuspidal automorphic representation $\dot{\pi}$ of the quasi-split unitary group $\dot{G}^{*}(\A_{\dot{F}})$ such that $\dot\pi_v=\pi_v$ and $\dot\pi_{v_\infty}$ has sufficiently regular infinitesimal character.
Let $\dot{\phi}$ be the global parameter such that $\dot{\pi}\in \Pi_{\dot \phi}$. The condition on $\dot\pi_{v_\infty}$ implies that $\dot\phi$ is a generic parameter.
By construction, we know that $\dot{\phi}_v = \phi_v$ for all $v \in V$.\footnote{Here we use the fact that if $\pi_v \in \Pi_{2,\mathrm{temp}}(\dot{G}^{*}(F_v))$ belongs to $\Pi_{\phi_{v,1}}$ and $\Pi_{\phi_{v,2}}$ for two generic parameters $\phi_{v,1}$ and $\phi_{v,2}$ then $\phi_{v,1}=\phi_{v,2}$. However this may fail if either is non-generic.}
It remains to verify that $\dot{\phi}$ is simple.
This follows from the fact
that for some $v \in V$, $ \dot{\phi}_v \in {\Phi}_{\mathrm{sim}}(\dot{G}^*_{v})$.
\end{proof}

\begin{lem}
\label{lemma_globalization_of_simple_parameter_2}
Let $N \geq 2$.
Let $\dot{E}/\dot{F}$, $\dot{F}$, $\dot{G}^{*}$, $\dot \eta$, $V$, $\{\phi_v\}_{v\in V}$ be as in Lemma \ref{lemma_globalization_of_simple_parameter_1}.
Assume that for at least one $v \in V$, $\phi_v \in {\Phi}_{\mathrm{sim}}(\dot{G}^*_v)$.
Let $v_2 \not \in V$ be a finite place that does not split in $\dot{E}$.
Then there exists a simple generic global parameter $\dot{\phi}=(\dot\phi^N,\tilde{\dot{\phi}}) \in \Phi_{\mathrm{sim}}(\dot{G}^*,\dot \eta)$ such that
\begin{itemize}
\item
$\dot{\phi}_v = \phi_v$ for all $v \in V$, and
\item the parameter $\dot{\phi}_{v_2}^N$ has the decomposition
\[
 \dot{\phi}^N_{v_2} =  \phi^{N-2}_{-}  \oplus \phi^1_+ \oplus (\phi^1_+)^\star
\]
where $\phi^{N-2}_{-} \in \widetilde{\Phi}_{\mathrm{sim}}(N - 2)$, and $\phi^1_{+} \in \Phi(1)$ is a character which is
not conjugate self-dual.

\end{itemize}
\end{lem}
\begin{proof}
The proof follows the same lines as the proof of Lemma \ref{lemma_globalization_of_simple_parameter_1}.
The difference is that we shall apply Lemma \ref{lemma_globalization_of_the_representation_2}
to the set $V' = V \cup \{v_2\}$
 instead of
Lemma \ref{lemma_globalization_of_the_representation_1} to the set $V$.

For  each $v \in V$, we choose  $\pi_v \in \Pi_{\phi_v}$, a square integrable representation of
the quasi-split unitary group  $\dot{G}^{*}(F_v)$
in the $L$-packet associated to $\phi_v$.

Let $M^*_{v_2}$ be a maximal proper Levi subgroup of $\dot{G}^*_{v_2}$ of the form
\[
	M^*_{v_2} = M^*_{v_2, -} \times G_{\dot{E}_{v_2}/\dot{F}_{v_2}}(1),
\]
where $M^*_{v_2, - } =  U_{\dot{E}_{v_2}/\dot{F}_{v_2}}(N-2)$.  We fix
a representation $\pi_{M^*_{v_2}} = \pi_{-} \times \pi_+ \in \Pi_{2, \mathrm{temp}}(M^*_{v_2})$
such that $(\pi_+)^c\pi_+^{-1}$ is not an unramified character (viewing $\pi_+$ as a character of $\GL(1,\dot E_{v_2})$; here $ c\in \Gal(E_{v_2}/F_{v_2})$ is the nontrivial element). Write $\dot\eta_{\dot M^*}$ for the composition of $\dot\eta$ with the Levi embedding $^L \dot M^*\hra {}^L \dot G^*$. Denote by $\phi_{\pi_{M^*_{v_2}}}$ the generic parameter associated to $\pi_{M^*_{v_2}}$. Then we have
\[
	\dot\eta_{M^*, v_2} \phi_{\pi_{M^*_{v_2}}} = \phi^{N-2}_{-} \oplus  \lambda_{+} \oplus  \lambda_{+}^\star,
\]
where  $ \phi^{N-2}_{-} \in \widetilde{\Phi}_{\mathrm{sim}}(N - 2)$, $\lambda_{+}\in \Phi(1)$, and
 the character $\lambda_{+}^{-1} \lambda_+^\star$ is not unramified.
This implies that the character
${\phi}^N_{\pi_1} \cdot \chi$ is not conjugate self-dual for all unramified unitary characters $\chi$.
We observe that if $\sigma$ is an irreducible constituent of the induced representation
$\cI_{M^*_{v_2}}^{\dot{G}^*_{v_2}}(\pi_{M^*_{v_2}} \otimes \chi)$
for some unramified unitary character $\chi \in \Psi_u(M^*_{v_2})$, then the associated parameter $\phi_\sigma$ satisfies
\[
	\dot\eta_{v_2} \phi_{\sigma} =  \phi^{N-2}_{-}  \oplus \phi^1_+ \oplus (\phi^1_+)^\star
\]
where $\phi^1_+\in \Phi(1)$ is an unramified twist of $\lambda_+$, which is therefore not conjugate self-dual.

The result follows by arguing as in the proof of Lemma \ref{lemma_globalization_of_simple_parameter_1}
and appealing to Lemma \ref{lemma_globalization_of_the_representation_2}
with the set $V \cup \{v_2\}$
instead of
Lemma \ref{lemma_globalization_of_the_representation_1} with the set $V$.
\end{proof}

\subsection{Globalization of a parameter}\label{sub:Globalization-param}
  As before let $E/F$ be a quadratic extension of local fields of characteristic $0$. Let $M^*$ be a Levi subgroup of $G^*:=U_{E/F}(N)$.
This subsection consists of various globalizations of groups and parameters to be needed for later proofs in the generic case. To explain this we introduce new terminology: a Levi subgroup $M^*$ of $G^*$ is called \textit{linear} if it is isomorphic to a product of Weil restriction of scalars of linear groups. The propositions are then organized according to the following cases.
\begin{enumerate}
\item $N$ is odd. (Proposition \ref{prop_globalize_data_N_odd})
\item $N$ is even.
\begin{enumerate}
\item $M^*=G^*$. (Proposition \ref{prop_globalize_data_N_even_M_eq_G})
\item $M^*\neq G^*$, $M^*$ is not linear, $N\ge 6$. (Proposition \ref{prop_globalize_data_N_even_and_M_neq_G_and_N_geq_6})
\item $M^*\neq G^*$, $M^*$ is not linear, $N=4$. (Propositions \ref{prop_globalize_data_N_4_reduction_exceptional_case}, \ref{prop_globalize_data_N_4_exceptional_case})
\item $M^*\neq G^*$, $M^*$ is linear. (Proposition \ref{prop_globalize_data_N_even_and_M_neq_G_split})
\end{enumerate}
\end{enumerate}
The complication in the even case is caused by the parity obstruction that an inner twist $(G,\xi)$ of $G^*$ does not always admit a global inner twist $(\dot G,\dot\xi)$ which localizes to $(G,\xi)$ at one place and is quasi-split at all the other places. The most subtle case of all turns out to be the case 2(c), in which case we have to check the local intertwining relation explicitly at archimedean places for enough cases. In fact it is fair to say that the case 2(c) is at the basis of the whole induction argument to prove LIR.

  Let $E/F$ be a quadratic extension of local fields and $G^*:=U_{E/F}(N)$. Temporarily choose a sign $\kappa\in\{\pm1\}$ and a character $\chi_\kappa\in \cZ_E^\kappa$. (In practice $\kappa$ and $\chi_\kappa$ will be determined by the global choices given below.)\footnote{When $E/F$ globalizes to a CM extension $\dot E/\dot F$, we can globalize $\chi_\kappa$ to $\dot\chi_\kappa\in \cZ_{\dot E}^\kappa$ as in \cite[Lem 7.2.2]{Mok} but we do not need such a result.}
Let $\phi \in \Phi_{\mathrm{bdd}}(G^*)$ be a generic bounded parameter.
As usual $\eta_\kappa\phi$ admits a decomposition into simple parameters
\[
\eta_\kappa \phi=  \oplus_{i=1}^r \ell_i\phi^{N_i}_i,\qquad \phi_i^{N_i}\in \Phi_{\simp}(N_i),~~\ell_i\in \Z_{\ge 1}.
\]
Suppose that $\phi\in \Phi_\el(G^*)\cup \Phi_{\EXC}(G^*)$ (namely it is not simultaneously non-elliptic and non-exceptional).
Such a parameter can be coarsely classified up to a reordering of factors as follows. Observe that the classification is independent of the choice of $\kappa$ and $\chi_\kappa$.
\begin{itemize}
	\item[$\Phi_2(G^*)$:] $\eta_\kappa \phi = \phi^{N_1}_1 \oplus  \cdots \oplus \phi^{N_r}_r$, where $S_\phi \simeq \O(1,\C)^r$.
	\item[$\Phi^2_{\mathrm{ell}}(G^*)$:] $\eta_\kappa \phi = 2 \phi^{N_1}_1 \oplus \cdots 2 \phi^{N_q}_q \oplus \phi^{N_{q+1}}_{q+1} \oplus \cdots \oplus \phi^{N_r}_{r}$, where
		$S_\phi \simeq \O(2,\C)^q \times \O(1,\C)^{r-q}$, $q\ge1$.
	\item[exc1:] $\eta_\kappa \phi = 2 \phi^{N_1}_1 \oplus \phi^{N_2}_2 \oplus \cdots \oplus \phi^{N_r}_r$, where
		$S_\phi \simeq \Sp(2, \C) \times \O(1,\C)^{r-1}$.
	\item[exc2:] $\eta_\kappa \phi = 3\phi^{N_1}_1 \oplus \phi^{N_2}_2 \oplus \cdots \oplus \phi^{N_r}_r$, where
		$S_\phi \simeq \O(3,\C) \times \O(1,\C)^{r-1}$.
\end{itemize}

There exists a Levi subgroup $M^*$ of $G^*$ unique up to conjugation such that $\Phi_2(M^*, \phi)$
is non-empty, that is, there exists a parameter
$\phi_{M^*} \in \Phi_2(M^*)$ which maps to $\phi$ via (any) Levi embedding $i_{M^*}:{}^L M^*\hra {}^L G^*$.
Explicitly $M^*$ has the form
\[
	M^* = M^*_{-} \times G_{E/F}(N_1)^{\ell'_1} \times \cdots \times G_{E/F} (N_r)^{\ell'_r},
\]
where $\ell'_i = \lfloor \ell_i/2\rfloor$, $N_{-} = \sum_{\ell_i \textrm{ odd}} N_i$, and
$M^*_{-} = U_{E/F}(N_{-})$. By restricting the $L$-embedding $\eta_{M^*}=\eta_\kappa i_{M^*}$ we obtain $\eta_-:{}^L M^*_-\hra {}^L G(N_-)$ and $\eta'_i:{}^L G(N_i)\hra {}^L G(2N_i)$ for each $1\le i\le r$.
In particular this allows us to view $(M^*_{-}, \eta_{-}) \in  \widetilde{\mathcal{E}}_{\simp}(N_{-})$.
The parameter $\phi_{M^*}=\phi_-\times \prod_{i=1}^r (\phi'_i)^{\ell'_i}$ is determined by the condition that $\eta_-\phi_-=\oplus_{2\nmid\ell_i} \phi_i^{N_i}$ and $\eta'_i\phi'_i=2\phi^{N_i}_i$. Moreover $\phi_-$ is bounded since $\phi$ is.

Let us set up some common notation from here throughout the end of \S\ref{sub:LIR-proof}.
As in the previous subsection we fix a sign $\kappa\in \{\pm1\}$ and $\dot\chi_\lambda\in \cZ^\lambda_{\dot E}$ for all CM extensions $\dot E/\dot F$ and all $\lambda\in \{\pm1\}$. By $\dot G^*$ we always denote the global unitary group $U_{\dot E/\dot F}(N)$, and often write $\dot\eta=\dot\eta_\kappa$ for the $L$-embedding $^L \dot G^* \hra {^L G}_{\dot E/\dot F}(N)$. Typically we will have a place $u$ of $\dot F$ such that $\dot E_u/\dot F_u=E/F$ and $\dot G_u=G$. In this case we write $\eta$ for $\dot \eta_u$, the restriction of $\dot\eta$ to the embedding of $L$-groups relative to the local field $\dot F_u$. In every lemma or proposition below we aim to construct global data with certain properties starting from (a subset of) the following local data. (We already discussed $G^*$, $\phi$, $M^*$, and $\phi_{M^*}$ above.)
\begin{itemize}
  \item $E/F$ is a quadratic extension of local fields,
  \item $G^*=U_{E/F}(N)$ and $M^*$ is a Levi subgroup of $G^*$,
  \item $(G,\xi)$ is an inner twist of $G^*$,
  \item $\phi\in \Phi_{\bdd}(G^*)$ is the image of $\phi_{M^*}\in \Phi_{2}(M^*)$.
\end{itemize}

\begin{lem}
\label{lem_globalize_data_N_odd}
Given $(E/F, G^*,   \phi, M^*, \phi_{M^*})$ as above, we can find the global data
$$(\dot{E}/\dot{F}, u, v_1,  \dot{G}^*, \dot{\phi}, \dot{M}^*, \dot{\phi}_{\dot{M}^*}),$$
where $\dot{E}/\dot{F}$ is a CM extension, $u$ is a place of $\dot{F}$ that does not split in $\dot{E}$,
$v_1$ is a finite place of $\dot{F}$ that does not split in $\dot{E}$,
$\dot{G}^*=U_{\dot E/\dot F}(N)$, $\dot{\phi} \in \Phi(\dot{G}^*,  \dot{\eta})$,
$\dot{M}^*$ is a Levi subgroup of $\dot{G}^*$, and $\dot{\phi}_{\dot{M}^*} \in \Phi_2(\dot{M}^*, \dot{\phi})$,
such that
\begin{enumerate}
\item $\dot E_u/\dot F_u=E/F$ and
   $
 	(\dot{G}^*_u, \dot{\phi}_u, \dot{M}^*_u, \dot{\phi}_{\dot{M}^*,u}) = (G^*, \phi, M^*, \phi_{M^*})
  $,
\item if $\phi \in \Phi_2(G^*)$ (resp. $\Phi^2_{\el}(G^*)$, resp. $\Phi_{\EXC1}(G^*)$, resp. $\Phi_{\EXC2}(G^*)$) then  $\dot\phi \in \Phi_2(\dot G^*)$ (resp. $ \Phi^2_{\el}(\dot G^*)$, resp. $ \Phi_{\EXC1}(\dot G^*)$, resp. $\Phi_{\EXC2}(\dot G^*)$),
\item $\dot{\phi}_{v_1} \in \Phi_{\mathrm{bdd}}(\dot{G}^*_{v_1})$ and
$\dot{\phi}_{M^*, v_1} \in \Phi_{\mathrm{2, bdd}}(\dot{M}^*_{v_1})$
and
\item
 the canonical maps
\begin{eqnarray*}
	{{S}}_{\dot{\phi}_{M^*}} &\rightarrow& {{S}}_{\dot{\phi}_{M^*,v}} 	\\
	{{S}}_{\dot{\phi}} &\rightarrow& {{S}}_{\dot{\phi}_v} 	\\
\end{eqnarray*}
are isomorphisms for $v \in \{u, v_1\}$. %
\end{enumerate}
\end{lem}
\begin{proof}
By Lemma  \ref{lem_globalization_field}, we obtain $(\dot{E}/\dot{F}, u)$ which globalizes the local extension $E/F$.
We fix a non-archimedean place $v_1$ of $\dot{F}$ that does not split in $\dot{E}$.
We can associate to each $\phi^{N_i}_i$ a pair
$$(G^*_i, \eta_{i})= (U_{\dot{E}/\dot{F}}(N_i),\dot\eta_{\chi_{\kappa_i}})\in \widetilde{\mathcal{E}}_{\mathrm{sim}}(N_i)$$ with unique sign $\kappa_i\in \{\pm 1\}$.
We also fix some simple parameters $ \phi_{v_1, i}  \in {\Phi}_{\mathrm{sim}}(\dot{G}^*_{i, v_1})$
such that $\phi_{v_1, i} \neq \phi_{v_1, j}$ for all $i \neq j$ (hence also $\dot{\eta}_{i,v_1}\phi_{v_1, i} \neq \dot{\eta}_{j,v_1}\phi_{v_1, j}$).
We  apply Lemma \ref{lemma_globalization_of_simple_parameter_1}
to obtain a simple global parameter
$\dot{\phi}_i \in \Phi_{\mathrm{sim}}(\dot{G}^*_i, \dot{\eta}_i)$ such that $\dot{\phi}_{i, u} = \phi_i$ and
$\dot{\phi}_{i, v_1} = \phi_{v_1, i}$.

We  define the global parameter
\[
	\dot{\phi}^N = \ell_1  \dot{\phi}_1^{N_1} \oplus \cdots \oplus \ell_r \dot{\phi}^{N_r}_r \in \widetilde{\Phi}(N).
\]
By construction, we observe that $\dot{\phi}^N_{u} = \eta_\kappa \phi$.
We must check that $\dot{\phi}^N$ is in the image of $\Phi(\dot{G}^*, \dot{\eta}_\kappa)$. If this were not true,
Proposition \ref{prop:1st-seed-thm} tells us that $\dot{\phi}^N$ comes from a parameter $\dot\phi$ of $\Phi(\dot{G}^*, \dot{\eta}_{-\kappa})$.
Hence $\dot{\phi}^N_{u}$ has sign $-(-1)^{N-1}\kappa$ in view of Proposition \ref{prop:2nd-seed-thm}. However $\eta_\kappa\phi$ has sign $(-1)^{N-1}\kappa$, cf. Lemma \ref{lem:image-xi_chi}, so this is a contradiction. So $\dot{\phi}^N$ comes from $\dot\phi\in\Phi(\dot{G}^*, \dot{\eta}_{\kappa})$.

Set $\dot{M}^*_{-} := U_{\dot{E}/\dot{F}}(N_{-})$ and define
\[
	\dot{M}^* := \dot{M}^*_{-} \times G_{\dot{E}/\dot{F}}(N_1)^{\ell'_1} \times \cdots
					\times G_{\dot{E}/\dot{F}} (N_r)^{\ell'_r}.
\]
as a Levi-subgroup of $\dot{G}^*$. By restricting $\dot\eta$ to ${}^L \dot M^*$, we obtain an $L$-embedding $\dot\eta_{\dot M^*}:^L \dot M^* \hra ^L G(N)$.
 We define a parameter in $\Phi_2(\dot{M}^*, \dot{\eta}_{\dot{M}^*})$
 $$\dot{\phi}_{\dot{M}^*} = \dot\phi_-\times (\dot\phi'_1)^{\ell'_1}\times \cdots \times (\dot\phi'_r)^{\ell'_r},\quad
\dot\phi_-\in \Phi(\cdot M^*_-,\dot\eta_-),~~\dot\phi'_i\in \Phi(G(N_i)),$$
by requiring that $\dot{\phi}_{\dot{M}^*}^N  = (\oplus_{2\nmid \ell_i} \dot\phi_i^{N_i})\oplus (\oplus_{i=1}^r 2\ell_i'\dot\phi_i^{N_i})$ in the way that $\dot\phi_-$ maps to $(\oplus_{2\nmid \ell_i} \dot\phi_i^{N_i})$ and each $\dot\phi'_i$ to $2\ell_i'\dot\phi_i^{N_i}$.

We have constructed the globalized data
 $(\dot{E}/\dot{F}, u, v_1,  \dot{G}^*, \dot{\eta}, \dot{\phi}, \dot{M}^*, \dot{\phi}_{\dot{M}^*})$.  It remains
to verify that conditions $1$ through $4$ are satisfied.  Conditions $1$ and $3$ are clearly satisfied.
Our construction is an adaptation of Arthur's construction in \cite[Prop 6.3.1]{Arthur}.
It follows form the construction and the definition of the local and global groups that condition 4 is satisfied. (To see this one utilizes the explicit description of $S_{\dot\phi}$, $S_{\dot\phi_v}$, etc in \S\ref{subsub:local-param-U} and \S\ref{subsub:global-param-U} and argues as on the last paragraph on page 324 of \cite{Arthur}.)
Finally, it follows from the construction and the isomorphism
${{S}}_{\dot{\phi}} \simeq {{S}}_{\phi}$ that condition $2$ is satisfied.
\end{proof}

 We would like to adapt the preceding lemma, which concerns the quasi-split case, to the setting of inner twists. The following globalization shall act as our analogue of \cite[Prop 6.3.1]{Arthur} when  $N$ is odd.
\begin{prop}
\label{prop_globalize_data_N_odd}
Let $(E/F, G^*,(G,\xi),\phi,M^*,\phi_{M^*})$ be as above. Assume that $N$ is odd.  %
There exists the global data
$$
(\dot{E}/\dot{F}, u, v_1,  \dot{G}^*,  (\dot{G}, \dot{\xi}, \dot{z}), \dot{\phi}, \dot{M}^*,  \dot{\phi}_{\dot{M}^*})
$$
where $\dot{E}/\dot{F}$ is a CM extension, $u$ is a place of $\dot{F}$ not split in $\dot{E}$,
$v_1$ is a finite place of $\dot{F}$ not split in $\dot{E}$,
$\dot{G}^*=U_{\dot E/\dot F}(N)$,
$ (\dot{G}, \dot{\xi}, \dot{z})$ is a pure inner twist of $\dot{G}^*$,
$\dot{\phi} \in \Phi(\dot{G}^*,  \dot{\eta})$,
$\dot{M}^*$ is a Levi subgroup of $\dot{G}^*$, and $\dot{\phi}_{\dot{M}^*} \in \Phi_2(\dot{M}^*, \dot{\phi})$,
such that
\begin{enumerate}
\item $\dot E_u/\dot F_u=E/F$ and $
	(\dot{G}_u, \dot\xi_u,\dot{\phi}_u, \dot{M}^*_u, \dot{\phi}_{\dot{M},u}) = (G,\xi, \phi, M^*, \phi_{M^*})
$,
\item
 $\dot{G}_v$ is quasi-split for all $v \neq u$,

\item if $\phi \in \Phi_2(G^*)$ (resp. $\phi \in \Phi^2_{\el}(G^*)$, resp. $\phi \in \Phi_{\EXC1}(G^*)$, resp. $\phi \in \Phi_{\EXC2}(G^*)$) then  $\dot\phi \in \Phi_2(\dot G^*)$ (resp. $\dot\phi \in \Phi^2_{\el}(\dot G^*)$, resp. $\dot\phi \in \Phi_{\EXC1}(\dot G^*)$, resp. $\dot\phi \in \Phi_{\EXC2}(\dot G^*)$),
\item $\dot{\phi}_{v_1} \in \Phi_{\mathrm{bdd}}(\dot{G}^*_{v_1})$ and
$\dot{\phi}_{\dot{M}^*, v_1} \in \Phi_{\mathrm{2, bdd}}(\dot{M}^*_{v_1})$,
\item
 the canonical maps
\begin{eqnarray*}
	{{S}}_{\dot{\phi}_{M^*}} &\rightarrow& {{S}}_{\dot{\phi}_{M^*,v}} 	\\
	{{S}}_{\dot{\phi}} &\rightarrow& {{S}}_{\dot{\phi}_v} 	\\
\end{eqnarray*}
are isomorphisms for $v \in \{u, v_1\}$.

\end{enumerate}

\end{prop}
\begin{proof}
By applying Lemma \ref{lem_globalize_data_N_odd},
we globalize the data
$(E/F, G^*,  \phi, M^*, \phi_{M^*})$ to $(\dot{E}/\dot{F}, u, v_1,  \dot{G}^*, \dot{\phi}, \dot{M}^*, \dot{\phi}_{\dot{M}^*})$.
We apply Lemma \ref{lem_globalization_gp}
to obtain a pure inner twist $(\dot{G},\dot\xi,\dot z)$ of $\dot{G}^*$ such that $(\dot{G}_u,\dot\xi_u) = (G,\xi)$ and
$\dot{G}_v$ is quasi-split for all $v \neq u$.
\end{proof}

We are left to deal with the case that $N$ is even.
In this case, the globalized  group $\dot{G}_{v}$ will be quasi-split at all places $v \not\in\{u, v_2\}$ for some auxiliary place $v_2$, however the group
$\dot{G}_{v_2}$ need not be quasi-split.  This causes additional complications as we wish that the globalized parameter
$\dot{\phi}$ be both relevant and of a type for which the main local theorems are known outside of the place $u$.

We begin by dealing with the case that $M^* = G^*$ below. Then we shall be left to deal with the case that $N$ is even and $M^* \neq G^*$. When $G_{v_2}$ is indeed not quasi-split, we have only complete results at $v_2$, via the induction hypothesis, when the local parameter is neither elliptic nor exceptional. In this case we carry out further case-by-case analysis as explained at the start of this subsection.

\begin{prop}
\label{prop_globalize_data_N_even_M_eq_G}
Given $(E/F,G^*,(G,\xi),\phi)$ as above, assume that $N$ is even and $M^* = G^*$. Then we can find the global data
$$(\dot{E}/\dot{F}, u, v_1,v_2,  \dot{G}^*,   (\dot{G}, \dot{\xi}, \dot{z}), \dot{\phi}),$$
where $\dot{E}/\dot{F}$ is a CM extension, $u$ is a place of $\dot{F}$ that does not split in $\dot{E}$,
$v_1$ is a finite place of $\dot{F}$ that does not split in $\dot{E}$,
$v_2$ is an archimedean place of $\dot{F}$,
$\dot{G}^*=U_{\dot E/\dot F}(N)$,
$ (\dot{G}, \dot{\xi}, \dot{z})$ is a pure inner twist of $\dot{G}^*$,
$\dot{\phi} \in \Phi_2(\dot{G}^*,  \dot{\eta})$ is relevant  for $\dot{G}$,
such that
\begin{enumerate}
\item $\dot E_u/\dot F_u=E/F$ and
$
	(\dot{G}_u, \dot\xi_u,\dot{\phi}_u) = (G, \xi,\phi)
$,
\item
 $\dot{G}_v$ is quasi-split and the character $\<\dot z_v,-\>$ of $Z(\hat{\dot G})^{\Gamma_v}$ is trivial for all $v \notin \{u, v_2\}$,
\item $\dot{\phi}_{v_1} \in \Phi_{2}(\dot{G}^*_{v_1})$,
\item $\dot{\phi}_{v_2} \in \Phi_{2}(\dot{G}^*_{v_2})$,
\item
 the canonical map
\begin{eqnarray*}
	{{S}}_{\dot{\phi}} &\rightarrow& {{S}}_{\dot{\phi}_v} 	\\
\end{eqnarray*}
is an isomorphism for $v \in \{u, v_1\}$.

\end{enumerate}
\end{prop}
\begin{proof}
This is proved in the same way as Proposition \ref{prop_globalize_data_N_odd} is deduced from Lemmas \ref{lem_globalization_gp} and \ref{lem_globalize_data_N_odd} except the following modifications in the statement and proof of Lemma \ref{lem_globalize_data_N_odd}: Firstly fix any archimedean place $v_2$ of the totally real field $\dot F$. Secondly impose the conditions that $\dot G$ is quasi-split at all places $v\notin \{u,v_2\}$ and that $\dot\phi_{v_2}\in \Phi_2(\dot G^*_{v_2})$. The latter condition at $v_2$ is achieved in the same way as at $v_1$.

To obtain $\dot z$, we argue as follows. We have for each place $v$ of $\dot F$
\[ Z(\hat{\dot G^*}_\tx{sc})^{\Gamma_v} = Z(\hat{\dot G^*})^{\Gamma_v} = \begin{cases} 1,& v \textrm{ splits in }\dot E\\
\{\pm 1\},&\textrm{ else}. \end{cases}, \]
as well as
\[ Z(\hat{\dot G^*}_\tx{sc})^{\Gamma} = Z(\hat{\dot G^*})^{\Gamma} = \{\pm 1\}. \]
By assumption $\dot\xi_v \in H^1(\Gamma_v,\dot G^*_\tx{ad})$ is trivial for all $v \notin \{u,v_2\}$. By \cite[Prop 2.6]{Kot86} we have an equality $\<\dot\xi_u,-\>=\<\dot\xi_{v_2},-\>^{-1}$ of characters of $Z(\hat{\dot G}_\tx{sc})^{\Gamma}$. Using the surjectivity of $H^1(\Gamma_v,\dot G^*) \rw H^1(\Gamma_v,\dot G^*_\tx{ad})$ we choose a collection of elements $\dot h_v \in H^1(\Gamma_v,\dot G^*)$ satisfying $\dot h_v=1$ for $v \notin \{u,v_2\}$ and lifting the collection $\dot\xi_v$. Then we must have an equality $\<\dot h_u,-\>=\<\dot h_{v_2},-\>^{-1}$ of characters of $Z(\hat{\dot G^*})^{\Gamma_u}=\{\pm 1\}=Z(\hat{\dot G}^*)^{\Gamma_{v_2}}$. According to \cite[Prop 2.6]{Kot86}, there exists $\dot h \in H^1(\Gamma,\dot G^*)$ whose localization at $v$ is equal to $\dot h_v$ for all $v$. Both elements $\dot h$ and $\dot\xi$ are uniquely determined by their localizations, from which we see that $\dot h$ lifts $\dot\xi$. Hence there exists $\dot z \in \dot h$ such that $(\dot\xi,\dot z)$ is a pure inner twist.

\end{proof}

\begin{prop}
\label{prop_globalize_data_N_even_and_M_neq_G_and_N_geq_6}
Let $(E/F, G^*,  (G, \xi), \phi, M^*, \phi_{M^*})$ be as above (right before Lemma \ref{lem_globalize_data_N_odd}).
Assume that $N$ is even, $N \geq 6$, and $M^* \neq G^*$ is not linear.
 We can find the global data $$(\dot{E}/\dot{F}, u, v_1, v_2,  \dot{G}^*,   (\dot{G}, \dot{\xi}, \dot{z}), \dot{\phi}, \dot{M}^*,  \dot{\phi}_{\dot{M}^*}),$$
where $\dot{E}/\dot{F}$ is a CM extension, $u$ is a place of $\dot{F}$ that does not split in $\dot{E}$,
$v_1$ and $v_2$  are finite places of $\dot{F}$ that do not split in $\dot{E}$,
$\dot{G}^*=U_{\dot E/\dot F}(N)$,
$ (\dot{G}, \dot{\xi}, \dot{z})$ is a pure inner twist of $\dot{G}^*$,
$\dot{\phi} \in \Phi(\dot{G}^*,  \dot{\eta})$,
$\dot{M}^*$ is a Levi subgroup of $\dot{G}^*$, and $\dot{\phi}_{\dot{M}^*} \in \Phi_2(\dot{M}^*, \dot{\phi})$,
such that
\begin{enumerate}
\item $\dot E_u/\dot F_u=E/F$ and
$
	(\dot{G}_u, \dot \xi_u, \dot{\phi}_u, \dot{M}^*_u, \dot{\phi}_{\dot{M}^*,u}) = (G, \xi, \phi, M^*, \phi_{M^*})
$,
\item
 $\dot{G}_v$ is quasi-split for all $v \not\in \{u,v_2\}$,
\item if $\phi \in \Phi^2_{\el}(G^*)$ (resp. $\phi \in \Phi_{\EXC1}(G^*)$, resp. $\phi \in \Phi_{\EXC2}(G^*)$) then  $\dot\phi \in \Phi^2_{\el}(\dot G^*)$ (resp. $\dot\phi \in \Phi_{\EXC1}(\dot G^*)$, resp. $\dot\phi \in \Phi_{\EXC2}(\dot G^*)$),
\item $\dot{\phi}_{v_1} \in \Phi_{\mathrm{bdd}}(\dot{G}^*_{v_1})$ and
$\dot{\phi}_{\dot{M}^*, v_1} \in \Phi_{\mathrm{2, bdd}}(\dot{M}^*_{v_1})$,
 \item the parameter $\dot{\phi}_{v_2} \in \Phi_{\mathrm{bdd}}(\dot{G}^*_{v_2})$ is relevant for $(\dot{G}_{v_2},\dot\xi_{v_2})$, and it is neither elliptic
nor exceptional,
\item
 the canonical maps
\begin{eqnarray*}
	{{S}}_{\dot{\phi}_{\dot{M}^*}} &\rightarrow& {{S}}_{\dot{\phi}_{\dot{M}^*,v}} 	\\
	{{S}}_{\dot{\phi}} &\rightarrow& {{S}}_{\dot{\phi}_v} 	\\
\end{eqnarray*}
are isomorphisms for $v \in \{u, v_1\}$.

\end{enumerate}

\end{prop}
\begin{proof}
The proposition is proved in the same way as Proposition \ref{prop_globalize_data_N_odd} except we appeal
here to a modification of Lemma \ref{lem_globalize_data_N_odd}. This modification, whose proof will be given momentarily, is the following. We choose any place $v_2$ of $\dot F$ different from $u$ and $v_1$ as in the proposition. Now we construct a globalization as in Lemma \ref{lem_globalize_data_N_odd} satisfying conditions 1-4 there but with the extra condition that
\begin{enumerate}
\item[5.] the parameter $\dot{\phi}_{v_2} \in \Phi_{\mathrm{bdd}}(\dot{G}^*_{v_2})$  is bounded and it is neither elliptic nor exceptional,
furthermore
$\Phi_{{2}}(M^*_{v_2}, \dot{\phi}_{v_2}) \neq \emptyset$ for some non-linear Levi subgroup $M^*_{v_2}$.
\end{enumerate}
Suppose for now that this is constructed. Then it only remains to check that the parameter $\dot{\phi}_{v_2}$ is relevant for $(\dot{G}_{v_2},\dot\xi_{v_2})$. Let $M^*_{v_2}$ be as in condition 5 above. Since $\dot{M}^*_{v_2}$ is a non-linear Levi subgroup (as $M^*=\dot M^*_u$ is non-linear), we see that $\dot{M}^*_{v_2}$ transfers to $\dot{G}_{v_2}$. Indeed, $\dot{G}_{v_2}$ is a unitary group over a non-archimedean field as $v_2$ is finite and non-split in $\dot E$, so any non-linear Levi subgroup of $\dot{G}^*_{v_2}$ transfers to $\dot{G}_{v_2}$. Hence $\dot{\phi}_{v_2}$ is relevant and the proof will be complete.

The rest of the proof is devoted to the construction of the desired globalization, based upon a slight adaptation of the construction appearing in Lemma \ref{lem_globalize_data_N_odd}. Given the construction, conditions 1-4 of Lemma \ref{lem_globalize_data_N_odd} will be satisfied in the same way, so the main issue will be to verify the new condition 5 above.
We shall divide the proof into the following two cases, which we shall  consider separately.
\begin{enumerate}
\item There exists $1 \leq s \leq r$ such that either $N_ s \geq 3$  or
	$N_s = 2$  and $\ell_t$ is odd for some $s\neq t$.
\item Either $N_i = 1$ for all $i=1,\ldots,r$, or there exists $1 \leq s \leq r$ such that $N_s= 2$ and $N_i \in \{1,2\}$ and $\ell_i$ even for all $i \neq s$.
\end{enumerate}

We shall begin by considering the first case.
We follow the construction appearing in Lemma \ref{lem_globalize_data_N_odd} with the following two changes.
Firstly, we begin by globalising
 the parameter $\phi_s$, however we appeal to Lemma \ref{lemma_globalization_of_simple_parameter_2}
instead of Lemma \ref{lemma_globalization_of_simple_parameter_1}.  The localization of the parameter
$\dot{\phi}_s$ can be written as
\[
  \dot{\phi}^{N_s}_{s, v_2} = {\phi}^{N_s-2}_{s, v_2, -}  \oplus {\phi}^1_{s, v_2,+} \oplus {{\phi}}^{1,\star}_{s, v_2, +},
\]
where ${\phi}^{N_s-2}_{s, v_2, -} \in \widetilde{\Phi}_{\mathrm{sim}}(N_s - 2)$ and ${\phi}^1_{s, v_2, +} \in \Phi(1)$ is a character which is
not conjugate self-dual.
Secondly, when globalising the parameters
$\phi_i$ for $i\neq s$, we fix some simple parameters $\phi_{i, v_2} \in \Phi_{\mathrm{sim}}(\dot{G}^*_{i,v_2})$
so that  $\dot{\eta}_{i, v_2} \phi_{i, v_2} \neq \dot{\phi}^{N_s-2}_{s, v_2, -}$ and  $\dot{\eta}_{i, v_2} \phi_{i, v_2} \neq \dot{\eta}_{j, v_2} \phi_{j, v_2}$ for all $i \neq j$.

Proceeding as in  Lemma \ref{lem_globalize_data_N_odd}, we obtain
the data $(\dot{E}/\dot{F}, u, v_1, v_2,  \dot{G}^*, \dot{\eta}, \dot{\phi}, \dot{M}^*, \dot{\phi}_{\dot{M}^*})$.
The first four conditions are satisfied, and it remains to check the fifth condition.
By construction, we have that
\[
	 \dot{\phi}_{v_2} = \ell_s ({\phi}^{N_s-2}_{s, v_2, -}  \oplus  \phi^1_{s, v_2, +} \oplus {\phi}^{1,\star}_{s, v_2, +})
				\oplus \bigoplus_{i\neq s} \ell_i \dot{\eta}_{i, v_2} {\phi}_{i, v_2}.
\]
Such a parameter is bounded, and it is also neither elliptic nor exceptional due to the
appearance of the character $\phi^1_{s, v_2, +}$ which is not conjugate self-dual.
The Levi subgroup $M^*_{v_2}$ for which
 $\Phi_{{2}}(M^*_{v_2}, \dot{\phi}_{v_2}) \neq \emptyset$ can be seen to be
\[
	M^*_{v_2} = M^*_{v_2, -} \times
	 G_{\dot{E}_{v_2}/\dot{F}_{v_2}}(1)^{\ell_s} \times
	 G_{\dot{E}_{v_2}/\dot{F}_{v_2}} (N_s - 2)^{\ell'_s}
	 \times \prod_{i \neq s}  G_{\dot{E}_{v_2}/\dot{F}_{v_2}} (N_i)^{\ell'_i}
\]
where $\ell'_i = \lfloor \ell_i/2 \rfloor$ and
$ M^*_{v_2, -}$ is the quasi-split unitary group in $N - 2 \ell_s -  2 \ell'_s (N_s -2) - 2 \sum_{i \neq s} \ell'_i N_i $ variables.
It follows that the Levi $M^*_{v_2}$ is not linear.
This completes the proof of the result for the first case.

Consider now the second case.
Let us deal firstly with the parameters of type (exc1).  In this case,
\[
\eta_\kappa\phi = 2 \phi^{N_1}_1 \oplus \phi^{N_2}_2 \oplus \cdots \oplus \phi^{N_r}_r
\]
 and $S_\phi \simeq \Sp(2, \C) \times \O(1,\C)^{r-1}$.
As $N \geq 6$ by considering the possible parameters,  we see that
$N_1 = N_2 = \cdots = N_r = 1$ and $r \geq 5$.
We follow the   construction appearing in Lemma \ref{lem_globalize_data_N_odd}
but impose the
following
additional  conditions at the $v_2$-place.
\begin{itemize}
\item	$\dot{\phi}^{N_1}_{1, v_2} = \chi_1$ for some character $\chi_1 \in {\Phi}_{\mathrm{bdd}}(\dot{G}_{1, v_2})$.
\item $\dot{\phi}^{N_2}_{2, v_2} = \chi_2$ for some character $\chi_2 \in {\Phi}_{\mathrm{bdd}}(\dot{G}_{2, v_2})$  such that
	$\chi_2 \neq \chi_1$.
\item For $i=3,\ldots,r$, $\dot{\phi}^{N_i}_{i, v_2} = \chi$ for some character $\chi \in {\Phi}_{\mathrm{bdd}}(\dot{G}_{2, v_2}) = {\Phi}_{\mathrm{bdd}}(\dot{G}_{i, v_2})$
such that $\chi \neq \chi_2$ and $\chi \neq \chi_1$.
\end{itemize}
Under this globalization, we observe that the group $S_{\dot{\phi}_{v_2}} \simeq \Sp(2,\C) \times \O(1,\C) \times \O(N-3,\C)$.
As $N \geq 6$, such a parameter is neither elliptic nor exceptional.
The Levi subgroup $M^*_{v_2}$ such that $\Phi_2(M^*_{v_2},\dot\phi_{v_2})$ is nonempty has the form
$M^*_{v_2} = U_{\dot{E}_{v_2}/\dot{F}_{v_2}}(2) \times  G_{\dot{E}_{v_2}/\dot{F}_{v_2}} (1)^{\frac{N-2}{2}}$
so the last condition of the lemma is satisfied. Hence the proof is complete in this case.

Let us  now deal with the parameters of type (exc2).
In this case,
\[
\eta_\kappa\phi = 3\phi^{N_1}_1 \oplus \phi^{N_2}_2 \oplus \cdots \oplus \phi^{N_r}_r
\]
 and $S_\phi \simeq \O(3,\C) \times \O(1,\C)^{r-1}$.  As $N\geq 6$ by considering the possible parameters, we see that either
\begin{itemize}
\item[(i)] $N_1 = N_2 = \cdots = N_r = 1$ and $r \geq 4$ or
\item[(ii)] $N=6$ and $\phi = 3 \phi_1$.
\end{itemize}

Consider case (i).
We shall globalize in a similar fashion to the case of parameters of type (exc1).  However, here we impose
the following conditions at the $v_2$-place.
\begin{itemize}
\item	$\dot{\phi}^{N_1}_{1, v_2} = \chi_1$ for some character $\chi_1 \in {\Phi}_{\mathrm{bdd}}(\dot{G}_{1,v_2})$.
\item For $i=2,\ldots,r$, $\dot{\phi}^{N_i}_{i, v_2} = \chi_1$ for some character $\chi \in {\Phi}_{\mathrm{bdd}}(\dot{G}_{1,v_2}) = {\Phi}_{\mathrm{bdd}}(\dot{G}_{i,v_2})$  		such that $\chi \neq \chi_1$.
\end{itemize}
Under this globalization, we observe that the group $S_{\dot{\phi}_{v_2}} \simeq \O(3,\C) \times \O(N-3,\C)$.
As $N \geq 6$, such a parameter is neither elliptic nor exceptional.
The Levi subgroup $M^*_{v_2}$ as in the last condition of the lemma is not linear as it is isomorphic to
$U_{\dot{E}_{v_2}/\dot{F}_{v_2}}(2) \times  G_{\dot{E}_{v_2}/\dot{F}_{v_2}} (1)^{\frac{N-2}{2}}$.
It follows that our globalization is as required.

Consider case (ii).
We shall globalize in a similar fashion to the case of parameters of type (exc1).  However, here we impose
the following conditions at the $v_2$-place: firstly
${\eta}_{1, v_2} \dot{\phi}_{1,v_2} = \chi_1 \oplus \chi_2$ for some pair of characters $\chi_1$ and $\chi_2$, and secondly
 	$\dot{\phi}_{1,v_2} \in \Phi_{2, \mathrm{bdd}}(\dot{G}_{1, v_2})$ (so $\chi_1$ and $\chi_2$ are conjugate self-dual characters which are mutually distinct). Under this globalization, we observe that the group $S_{\dot{\phi}_{v_2}} \simeq \O(3,\C) \times \O(3,\C)$.
	Such a parameter is neither elliptic nor exceptional.
	One can see that the Levi subgroup $M^*_{v_2}$ in the last condition
 is isomorphic to $ U_{\dot{E}_{v_2}/\dot{F}_{v_2}}(2)\times  G_{\dot{E}_{v_2}/\dot{F}_{v_2}} (1)^2$, which is again not linear.
It follows that our globalization is as required.

Let us deal with the remaining case of elliptic non-square-integrable parameters.
In this case
\[
\eta_\kappa\phi = 2 \phi^{N_1}_1 \oplus \cdots \oplus 2 \phi^{N_q}_q \oplus \phi^{N_{q+1}}_{q+1} \oplus \cdots\oplus \phi^{N_r}_{r}, \ \ q \geq 1,
\]
and
$S_\phi \simeq \O(2,\C)^q \times \O(1,\C)^{r-q}$.
We see that either
\begin{itemize}
\item[(i)] $N_1 = N_2 = \ldots = N_r = 1$, or
\item[(ii)]  $N_1 =N_2  = \cdots = N_{r-1} = 1$, $N_r = 2$  and $q = {r-1}$.
\end{itemize}

Consider case (i).  We have that $q \neq r$ as $M^*$ is a non-linear Levi of $G^*$.
We shall globalize in a similar fashion to the case of parameters of type (exc1).  However, here we impose
the following conditions at the $v_2$-place.
\begin{itemize}
\item For $i=1,\ldots,r-1$, $\dot{\phi}_{i, v_2} = \chi_1$ for some character $\chi_1 \in {\Phi}_{\mathrm{bdd}}(\dot{G}_{1,v_2}) = {\Phi}_{\mathrm{bdd}}(\dot{G}_{i,v_2})$.
\item	$\dot{\phi}_{r, v_2} = \chi$ for some character $\chi \in {\Phi}_{\mathrm{bdd}}(\dot{G}_{r,v_2})$
	such that $\chi \neq \chi_1$.
\end{itemize}
Under this globalization, we observe that the group $S_{\dot{\phi}_{v_2}} \simeq \O(N-1,\C) \times \O(1,\C)$.
As $N \geq 6$, such a parameter is neither elliptic nor exceptional.
One verifies the last condition of the lemma as $M^*_{v_2}$ there is isomorphic to
 $U_{\dot{E}_{v_2}/\dot{F}_{v_2}}(2) \times G_{\dot{E}_{v_2}/\dot{F}_{v_2}} (1)^{\frac{N-2}{2}}$.
It follows that our globalization is as required.

Consider case (ii).
We shall globalize in a similar fashion to the case of parameters of type (exc1).  However, here we impose
the following conditions at the $v_2$-place.
\begin{itemize}
\item
$\dot{\phi}_{r,v_2} = \sigma$ for some simple parameter $\sigma \in \Phi_{\mathrm{sim}}(\dot{G}_{r,v_2})$.
\item For $i=1,\ldots,r-1$, $\dot{\phi}_{i, v_2} = \chi$ for some character $\chi \in {\Phi}_{\mathrm{bdd}}(\dot{G}_{1,v_2}) = {\Phi}_{\mathrm{bdd}}(\dot{G}_{i,v_2})$.
\end{itemize}
Under this globalization, we observe that the group $S_{\dot{\phi}_{v_2}} \simeq \O(1,\C) \times \O(N-2,\C)$.
As $N \geq 6$, such a parameter is neither elliptic nor exceptional.
The last condition of the lemma is satisfied as $M^*_{v_2}$ is non-linear and rather isomorphic to $U_{\dot{E}_{v_2}/\dot{F}_{v_2}}(2) \times  G_{\dot{E}_{v_2}/\dot{F}_{v_2}} (1)^{\frac{N-2}{2}}$.
We conclude the proof in this final case.

\end{proof}

We shall now deal with the case that $N=4$ and $M^* \neq G^*$ is the unique non-linear Levi subgroup.
This case is more difficult than the previous cases.
The problem is that we can no longer ensure that the globalized parameter is relevant and neither elliptic nor exceptional at the place $v_2$.   We shall reduce the problem to some examples of exceptional parameters at an archimedean place.

\begin{prop}
\label{prop_globalize_data_N_4_reduction_exceptional_case}
Let $(E/F, G^*, (G, \xi),\phi, M^*, \phi_{M^*})$ be as above. Assume that $N=4$, that $M^* \neq G^*$ is not linear, and
that the parameter  $\phi \in \Phi_{\mathrm{ell}}^2(G^*)$ is elliptic.
We can find the global data
$$(\dot{E}/\dot{F}, u, v_1, v_2,  \dot{G}^*,  (\dot{G}, \dot{\xi}, \dot{z}), \dot{\phi}, \dot{M}^*,  \dot{\phi}_{\dot{M}}),$$
where $\dot{E}/\dot{F}$ is a CM extension, $u$ is a place of $\dot{F}$ that does not split in $\dot{E}$,
$v_1$ and $v_2$  are finite places of $\dot{F}$ that do not split in $\dot{E}$,
$\dot{G}^*=U_{\dot E/\dot F}(N)$,
$ (\dot{G}, \dot{\xi}, \dot{z})$ is a pure inner twist of $\dot{G}^*$,
$\dot{\phi} \in \Phi(\dot{G}^*,  \dot{\eta})$,
$\dot{M}^*$ is a Levi subgroup of $\dot{G}^*$, and $\dot{\phi}_{\dot{M}} \in \Phi_2(\dot{M}^*, \dot{\phi})$,
such that
\begin{enumerate}
\item $\dot E_u/\dot F_u=E/F$ and
$
	(\dot{G}_u, \dot\xi_u,\dot{\phi}_u, \dot{M}^*_u, \dot{\phi}_{\dot{M},u}) = (G, \xi,\phi, M^*, \phi_{M^*})
$,
\item
 $\dot{G}_v$ is quasi-split for all $v \not\in \{u,v_2\}$,
\item   $ \dot{\phi} \in\Phi^2_{\mathrm{ell}}(\dot{G}^*)$,
\item $\dot{\phi}_{v_1} \in \Phi_{\mathrm{bdd}}(\dot{G}^*_{v_1})$ and
$\dot{\phi}_{M^*, v_1} \in \Phi_{\mathrm{2, bdd}}(\dot{M}^*_{v_1})$,
\item $\dot{\phi}_{v_2}$ belongs to $\Phi_{\mathrm{bdd},\EXC2}(\dot{G}^*_{v_2})$ and is $(\dot{G}^*_{v_2},\dot\xi_{v_2})$-relevant,
\item
 the canonical maps
\begin{eqnarray*}
	\overline{{S}}_{\dot{\phi}_{M}} &\rightarrow& \overline{{S}}_{\dot{\phi}_{M,v}} 	\\
	\overline{{S}}_{\dot{\phi}} &\rightarrow& \overline{{S}}_{\dot{\phi}_v} 	\\
\end{eqnarray*}
are isomorphisms for $v \in \{u, v_1\}$.

\end{enumerate}

\end{prop}
\begin{proof}
The proposition is proved in the same way as Proposition \ref{prop_globalize_data_N_even_and_M_neq_G_and_N_geq_6} except we appeal
here to the following variant of Lemma \ref{lem_globalize_data_N_odd}. When globalising, we will choose a place $v_2$ of $\dot F$ which does not split in $\dot E$ such that $v_2\notin \{u,v_1\}$, and impose not only conditions 1-4 of the lemma but also the extra condition at $v_2$ that
\begin{enumerate}
  \item[5.] $\dot{\phi}_{v_2}$ belongs to $\Phi_{\mathrm{bdd},\EXC2}(\dot{G}^*_{v_2})$ and is $(\dot{G}^*_{v_2},\dot\xi_{v_2})$-relevant;
furthermore $\Phi_{{2}}(M^*_{v_2}, \dot{\phi}_{v_2}) \neq \emptyset$ for some non-linear Levi $M^*_{v_2}$.
\end{enumerate}
Note that condition 2 simply says that $\dot\phi\in \Phi^2_{\el}(\dot G^*,\dot \eta)$ as we are assuming $\phi\in \Phi^2_{\el}(G^*)$.
To complete the proof it is enough to construct this variant globalization.

As the parameter $\phi$ is elliptic and $M^* \neq G^*$ is not linear, it must be of the form
\[
\phi^N = 2 \phi^{N_1}_1 \oplus \phi^{N_2}_2 \oplus \phi^{N_3}_3\]
where either
 $N_1 = N_2 = N_3 = 1$ and
 $S_\phi \simeq \O(2,\C) \times \O(1,\C)^2$; or $N_1 = 1$, $N_2 = 2$, $N_3 = 0$ and
 $S_\phi \simeq \O(2,\C) \times \O(1,\C)$.

We follow the   construction appearing in Lemma \ref{lem_globalize_data_N_odd}
but impose the
following
additional  conditions at the $v_2$-place.
If $N_1 = N_2 = N_3 = 1$, then
\begin{itemize}
\item	$\dot{\phi}_{1, v_2} = \chi_1$ for some character $\chi_1 \in {\Phi}_{\mathrm{bdd}}(\dot{G}_{1, v_2})$,
\item	$\dot{\phi}_{2, v_2} = \chi_1$, and
\item $\dot{\phi}_{3, v_2} = \chi_2$ for some character $\chi_2 \in {\Phi}_{\mathrm{bdd}}(\dot{G}_{2, v_2})$  such that
	$\chi_2 \neq \chi_1$.
\end{itemize}
Otherwise if $N_1 =1$, $N_2 = 2$ and $N_3 = 0$, then
\begin{itemize}
\item	$\dot{\phi}_{1, v_2} = \chi_1$ for some character $\chi_1 \in {\Phi}_{\mathrm{bdd}}(\dot{G}_{1, v_2})$, and
\item	$\dot{\phi}_{2, v_2} = \chi_1 \oplus \chi_2$ for some character $\chi_2 \in {\Phi}_{\mathrm{bdd}}(\dot{G}_{1, v_2})$  such 	that 	$\chi_2 \neq \chi_1$.
\end{itemize}

Under this globalization, we observe that the group $S_{\dot{\phi}_{v_2}} \simeq \O(3,\C) \times \O(1,\C)$.
Such a parameter is of the type (exc2).
One can see that the Levi subgroup
$M^*_{v_2} = M^*_{v_2, -} \times  G_{\dot{E}_{v_2}/\dot{F}_{v_2}} (1)$
where $M^*_{v_2, -} = U_{\dot{E}_{v_2}/\dot{F}_{v_2}}(2)$.
It follows that the data $(\dot{E}/\dot{F}, u, v_1, v_2,  \dot{G}^*, \dot{\phi}, \dot{M}^*, \dot{\phi}_{\dot{M}^*})$
meets all the requirements.

\end{proof}

When $N=4$ and $M^*$ is a non-linear proper Levi subgroup, the proof of LIR for elliptic parameters will be reduced to the case for exceptional parameter by the previous proposition. In order to address the problem for exceptional parameters, we construct a globalization which will allow us to further reduce to some special instances of exceptional parameters at archimedean places, where we already verified LIR in \S\ref{sec:lirexp} by an explicit computation.

\begin{prop}
\label{prop_globalize_data_N_4_exceptional_case}
Let $(E/F, G^*, (G, \xi),\phi, M^*, \phi_{M^*})$ be as above. Let $\heartsuit\in \{\EXC1,\EXC2\}$.
Assume that $N=4$,
that $G$ is not quasi-split, and that $M^* \neq G^*$ is not linear, and
 that $\phi \in \Phi_{\mathrm{bdd},\heartsuit}(G^*)$.
 There exists the data
$$(\dot{E}/\dot{F}, u, v_1, v_2,  \dot{G}^*,  (\dot{G}, \dot{\xi}, \dot{z}), \dot{\phi}, \dot{M}^*,  \dot{\phi}_{\dot{M}})$$
where $\dot{E}/\dot{F}$ is a CM extension, $u$ is a place of $\dot{F}$ that does not split in $\dot{E}$,
$v_1$ is a finite places of $\dot{F}$ that does not split in $\dot{E}$, $v_2$ is an archimedean place
of $\dot{F}$ that does not split in $\dot{E}$,
$\dot{G}^*=U_{\dot E/\dot F}(N)$,
$ (\dot{G}, \dot{\xi}, \dot{z})$ is a pure inner twist of $\dot{G}^*$,
$\dot{\phi} \in \Phi(\dot{G}^*,  \dot{\eta})$,
$\dot{M}^*$ is a Levi subgroup of $\dot{G}^*$, and $\dot{\phi}_{\dot{M}} \in \Phi_2(\dot{M}^*, \dot{\phi})$,
such that
\begin{enumerate}
\item $\dot E_u/\dot F_u=E/F$ and
$
	(\dot{G}_u,\dot\xi_u, \dot{\phi}_u, \dot{M}^*_u, \dot{\phi}_{\dot{M},u}) = (G,\xi, \phi, M^*, \phi_{M^*})
$,
\item
 $\dot{G}_v$ is quasi-split for all $v \not\in \{u,v_2\}$,

\item   $ \dot{\phi} \in\Phi_\heartsuit(\dot{G}^*)$,

\item $\dot{\phi}_{v_1} \in \Phi_{\mathrm{bdd}}(\dot{G}^*_{v_1})$ and
$\dot{\phi}_{M^*, v_1} \in \Phi_{\mathrm{2, bdd}}(\dot{M}^*_{v_1})$,

\item the parameter $\dot{\phi}_{v_2}  \in \Phi_{\mathrm{bdd},\heartsuit}(\dot{G}^*_{v_2})$ is $(\dot G_{v_2},\dot\xi_{v_2})$-relevant and satisfies Theorem \ref{thm:lir} relative to the Levi subgroup $\dot M^*_{v_2}$.
\item
 the canonical maps
\begin{eqnarray*}
	{{S}}_{\dot{\phi}_{M}} &\rightarrow& {{S}}_{\dot{\phi}_{M,v}} 	\\
	{{S}}_{\dot{\phi}} &\rightarrow& {{S}}_{\dot{\phi}_v} 	\\
\end{eqnarray*}
are isomorphisms for $v \in \{u, v_1\}$.

\end{enumerate}

\end{prop}
\begin{proof}
The proposition is proved in a similar way to Proposition \ref{prop_globalize_data_N_even_and_M_neq_G_and_N_geq_6} but
instead of Lemma \ref{lem_globalize_data_N_odd}, we appeal to the following variant. When globalising we consider a real place $v_2$ of the totally real field $\dot F$ different from $u$ and $v_1$, and impose not only conditions 1-4 of that lemma but also that
\begin{enumerate}
  \item[5.] the parameter $\dot{\phi}_{v_2} \in \Phi_{\mathrm{bdd}}(\dot{G}^*_{v_2})$ is equal to the parameter $\phi$ of \S\ref{sec:lirexp} with $x\in \{0,\pm 1,\pm 2\}$ (resp. $x\in \{\pm \frac{1}{2}\}$) if $\heartsuit$ is $\EXC1$ (resp. $\EXC2$).
\end{enumerate}
Note that condition 2 just amounts to $\dot\phi\in \Phi_{\heartsuit}(\dot G^*)$ here. Condition 5 just stated implies condition 5 of the proposition by Proposition \ref{pro:u31lir}. (The relevance is clear from the construction and Proposition \ref{pro:u31lir} is not needed.) So we will be done once we establish this variant of Lemma \ref{lem_globalize_data_N_odd}.

We remind the reader that the subgroup $M^* = U_{E/F}(2) \times G_{E/F} (1)$ and the parameter $\phi$ is of
the following form.
\begin{itemize}
	\item if $\phi$ is (exc1), then either $\eta \phi = 2 \phi^{N_1}_1 \oplus \phi^{N_2}_2$ and $S_\phi = \Sp(2,\C) \times \O(1,\C)$, or
			$\eta\phi = 2 \phi^{N_1}_1 \oplus \phi^{N_2}_2 \oplus \phi^{N_3}_3$ and $S_\phi = \Sp(2,\C) \times \O(1,\C)^2$, and
	\item  If $\phi$ is (exc2), then $\phi^N = 3 \phi^{N_1}_1 \oplus \phi^{N_2}_2$ and $S_\phi = \O(3,\C) \times \O(1,\C)$.
\end{itemize}
The result is proved in the same way as Lemma \ref{lem_globalize_data_N_odd}.
The only difference is that when globalising each $\phi^{N_i}_i$ to obtain $\dot\phi_i$, we impose an appropriate extra local condition at $v_2$ such that $\dot\phi_{v_2}$ satisfies condition 5 above. To see that this is possible, it is enough to remark that the archimedean parameter of \S\ref{sec:lirexp} is of type (exc1) if $x\in \{0,\pm 1,\pm 2\}$ and of type (exc2) if $x\in \{\pm \frac{1}{2}\}$.
\end{proof}

We are left to deal with the case that $N$ is even and $M^* \neq G^*$ is a linear Levi subgroup.
In our applications to the proof of the local intertwining relation, the group $G$ will not be quasi-split (the quasi-split case having already been treated
by Mok).  In this case, the Levi $M^*$ will not transfer to $G$, that is the parameters in question will not be relevant and the
desired local intertwining relation reduces to a vanishing statement. The following globalization shall suffice. Although it will be applied only when $M^*$ is linear, the proposition is true without such an assumption.

\begin{prop}
\label{prop_globalize_data_N_even_and_M_neq_G_split}
Let $(E/F,G^*,(G,\xi),\phi,M^*,\phi_{M^*})$ be as above.
Assume that $N$ is even and that $M^* \neq G^*$.
There exists the global data
$$(\dot{E}/\dot{F}, u_1, u_2, v_1,  \dot{G}^*,  (\dot{G}, \dot{\xi}, \dot{z}), \dot{\phi}, \dot{M}^*,  \dot{\phi}_{\dot{M}}),$$
where $\dot{E}/\dot{F}$ is a CM extension, $u_1, u_2$  are places of $\dot{F}$ that do not split in $\dot{E}$,
$v_1$ is a finite place of $\dot{F}$ that does not split in $\dot{E}$,
$\dot{G}^*=U_{\dot E/\dot F}(N)$,
$ (\dot{G}, \dot{\xi}, \dot{z})$ is a  pure inner twist of $\dot{G}^*$,
$\dot{\phi} \in \Phi(\dot{G}^*,  \dot{\eta})$,
$\dot{M}^*$ is a Levi subgroup of $\dot{G}^*$, and $\dot{\phi}_{\dot{M}} \in \Phi_2(\dot{M}^*, \dot{\phi})$,
such that
\begin{enumerate}
\item $\dot E_u/\dot F_u=E/F$ and
$(\dot{G}_u, \dot \xi_u,\dot{\phi}_u, \dot{M}^*_u, \dot{\phi}_{\dot{M},u}) = (G, \xi,\phi, M^*, \phi_{M^*})$
	for $u\in \{u_1,u_2\}$,

\item
 $\dot{G}_v$ is quasi-split for all $v \not\in\{u_1,u_2\}$,
\item if $\phi \in \Phi_2(G^*)$ (resp. $\Phi^2_{\mathrm{ell}}(G^*)$, resp. $\Phi_{\EXC1}(G^*)$, resp. $\Phi_{\EXC2}(G^*)$)
then  $ \dot{\phi} \in \Phi_2(\dot{G}^*)$ (resp. $\Phi^2_{\mathrm{ell}}(\dot{G}^*)$, resp. $\Phi_{\EXC1}(\dot G^*)$, resp. $\Phi_{\EXC2}(\dot G^*)$),
\item $\dot{\phi}_{v_1} \in \Phi_{\mathrm{bdd}}(\dot{G}^*_{v_1})$ and
$\dot{\phi}_{M^*, v_1} \in \Phi_{\mathrm{2, bdd}}(\dot{M}^*_{v_1})$,
\item
 the canonical maps
\begin{eqnarray*}
	{{S}}_{\dot{\phi}_{M}} &\rightarrow& {{S}}_{\dot{\phi}_{M,v}} 	\\
	{{S}}_{\dot{\phi}} &\rightarrow& {{S}}_{\dot{\phi}_v} 	\\
\end{eqnarray*}
are isomorphisms for $v \in \{u_1, u_2, v_1\}$.

\end{enumerate}

\end{prop}
\begin{proof}
This is proved in the same way as Proposition \ref{prop_globalize_data_N_odd} is deduced from Lemmas \ref{lem_globalization_gp} and \ref{lem_globalize_data_N_odd} except that we use Lemma \ref{lem_globalization_field_gp_double_places} in place of Lemma \ref{lem_globalization_gp} and that Lemma \ref{lem_globalize_data_N_odd} is modified as follows: Instead of the place $u$, we consider two places $u_1$ and $u_2$ of $\dot F$ which do not split in $\dot E$, and impose the same condition at $u_1$ and $u_2$ as we did at $u$ in Lemma \ref{lem_globalize_data_N_odd} when globalising the data (so that condition 1 of that lemma holds for $u\in \{u_1,u_2\}$ and condition 4 for $v\in \{u_1,u_2,v_1\}$).
\end{proof}

\subsection{On global elliptic parameters, II}\label{subsection_elliptic_parameters_2}

The following result shall act as our analogue of Lemmas 5.4.3 and 5.4.4 of \cite{Arthur}.
\begin{lem}
\label{lem_analgoue_of_arthur_lem_543_544}
Let $\dot{E}/\dot{F}$ be a quadratic extension of number fields.
Let $\dot{G}^* = U_{\dot{E}/\dot{F}}(N)$ and let $(\dot G,\dot\xi)$ be an inner twist of $\dot G^*$.
Let $\dot{M}^* $ be a proper Levi subgroup of $\dot G^*$.
Let $\dot{\phi} \in \Phi^2_{\mathrm{ell}}(\dot{G}^*, \dot{\eta})$ and suppose that $\dot{\phi}_{\dot{M}^*} \in \Phi_2(\dot{M}^*, \dot{\phi})$.
Assume that there exists a place $v_1$ of $\dot{F}$ that does not split in $\dot{E}$ such that the following conditions hold.
\begin{itemize}
\item $\dot{G}_{v_1}$ is quasi-split,
\item $\dot{\phi}_{v_1} \in \Phi_{\mathrm{bdd}}(\dot{G}^*_{v_1})$ and
	$\dot{\phi}_{\dot{M}^*, {v_1}} \in \Phi_2(\dot{M}^*_{v_1}, \dot{\phi}_{v_1})$, and
	\item
	 the canonical maps
\begin{eqnarray*}
	{{S}}_{\dot{\phi}_{\dot{M}^*}} &\rightarrow& {{S}}_{\dot{\phi}_{\dot{M}^*,v_1}} 	\\
	{{S}}_{\dot{\phi}} &\rightarrow& {{S}}_{\dot{\phi}_{v_1}} 	\\
\end{eqnarray*}
are isomorphisms.
 \end{itemize}
Then for all $\dot{f} \in \mathcal{H}(\dot{G})$,
\[
	\mathrm{tr} R^{\dot{G}}_{\mathrm{disc}, \dot{\phi}}(\dot{f}) =
	 \sum_{\overline{x} \in \overline{\cS}_{\dot{\phi}, \mathrm{ell}}}
		(\dot{f}'_{\dot{G}}(\dot{\phi}, \overline{x}) - \dot{f}_{\dot{G}}(\dot{\phi}, \overline{x})) = 0.
\]
\end{lem}

\begin{proof}
Fix an equivalence class of pure inner twists $\dot \Xi$ as in Lemma \ref{lem_globalization_gp} giving rise to $(\dot G,\dot \xi)$.
Let $\dot{f} = \otimes_v \dot f_v \in \mathcal{H}(\dot{G})$ be a decomposable function.
Let $\dot{u}_{\overline{x}} \in N^\natural_{\dot{\phi}}$ be an element that maps to $\overline{x}$.
Then we have that
\[
\dot{f}_{\dot{G}}(\dot{\phi}, \overline{x}) =  \dot{f}_{\dot{G}, \dot{\Xi}}(\dot{\phi}, \dot{u}_{\overline{x}}) =
	 \dot{f}_{\dot{G}, \dot{\Xi}, v_1}(\dot{\phi}_{v_1}, \dot{u}_{\overline{x}, v_1})
	 \dot{f}_{\dot{G}, \dot{\Xi}}^{v_1} (\dot{\phi}^{v_1}, \dot{u}_{\overline{x}}^{v_1}).
\]
Let $\dot{s}_{\overline{x}} \in S_{\dot{\phi}, \mathrm{ss}}$ be an element that has the same image as $\dot{u}_{\overline{x}}$ in
$S^\natural_{\dot{\phi}}$, and consequently maps to $\overline{x}$.  We have that
\[
\dot{f}'_{\dot{G}}(\dot{\phi}, \overline{x}) =  \dot{f}'_{\dot{G}, \dot{\Xi}}(\dot{\phi}, \dot{s}_{\overline{x}})
=
\dot{f}'_{\dot{G}, \dot{\Xi}, v_1}(\dot{\phi}_{v_1}, \dot{s}_{\overline{x},v_1})
(\dot{f}')_{\dot{G}, \dot{\Xi}}^{v_1}(\dot{\phi}^{v_1}, \dot{s}_{\overline{x}}^{v_1}).
\]

Since $\dot{\phi}$ is generic, we have that $s_{\dot{\phi}} = 1$ and $\epsilon^{\dot{G}}_{\dot{\phi}}(\overline{x}) = 1$.
We also know the local intertwining relation on $\dot{G}_{v_1}$, which implies that
\[
\dot{f}'_{\dot{G}, \dot{\Xi}, v_1} (\dot{\phi}_{v_1}, \dot{s}_{\overline{x}, v_1}) = e(\dot{G}_{v_1}) \dot{f}_{\dot{G}, \dot{\Xi}, v_1} (\dot{\phi}_{v_1}, \dot{u}_{\overline{x}, v_1}).
\]
 The expression in Section \ref{sub:elliptic-parameters} reduces to
\begin{equation}
\label{equation_lem_analgoue_of_arthur_lem_543_544}
\mathrm{tr} R^{\dot{G}}_{\mathrm{disc}, \dot{\phi}}(\dot{f}) = c \sum_{\overline{x} \in \overline{\cS}_{\dot{\phi}, \mathrm{ell}}}
\dot{f}_{\dot{G}, \dot{\Xi}, v_1}(\dot{\phi}_{v_1}, \dot{u}_{\overline{x}, v_1})
\left((\dot{f}')^{v_1}_{\dot{G}, \dot{\Xi}}(\dot{\phi}^{v_1}, \dot{s}_{\overline{x}}^{v_1})
-   \dot{f}_{\dot{G}, \dot{\Xi}}^{v_1} (\dot{\phi}^{v_1}, \dot{u}_{\overline{x}}^{v_1})\right)
\end{equation}
for some constant $c>0$.

We recall the definition
\[
	\dot{f}_{\dot{G}, \dot{\Xi}, v_1} (\dot{\phi}_{v_1}, \dot{u}_{\overline{x}, v_1})
		= \sum_{\pi_{v_1} \in \Pi_{\dot{\phi}_{M^*,v_1}}(M^*_{v_1})}
		\mathrm{tr}(R_P(\dot{u}_{\overline{x}, v_1}, \dot{\Xi}_{v_1},  \pi_{v_1}, \dot{\phi}_{v_1}, \dot{\phi}_{\dot{M}^*, v_1}))
		 \mathcal{I}_P(\pi_{v_1}) (\dot{f}_{v_1}).
\]
As in the proof of \cite[Lem 5.4.3]{Arthur},
the local intertwining relation and the local classification theorem for the quasi-split group  $\dot{G}_{v_1}$
 imply that there are natural isomorphisms from the
$R$-group $R_{\dot{\phi}_{v_1}}$
onto the representation theoretic $R$-groups $R(\pi_{v_1})$ for each $\pi_{v_1}$.
This enables  a $ W^0_{\dot{\phi}, {v_1}}$-conjugacy class
\[
	({M}_{v_1}, \pi_{v_1}, w_{\overline{x}, v_1})
\]
to be identified with an element in the basis
$T(\dot{G}_{v_1})$
where $w_{\overline{x}, v_1}$  denotes the image of $\dot{u}_{\overline{x}, v_1}$ in $W_{\dot{\phi}_{v_1}}$
 (see \S\ref{sub:vashing-coefficients} for the appropriate definitions).
It follows that we can write
\begin{eqnarray*}
&&\langle \pi_{v_1}, k_{{M}_{v_1}, \dot{G}_{v_1}}(x_{\overline{x}, v_1}) \rangle_{\dot{\xi}_{{M}_{v_1}}}
		\mathrm{tr}(R_P(w_{\overline{x}, v_1}, \dot{\xi}_{{M}_{v_1}}, \pi_{v_1}, \dot{\phi}_{v_1}) \mathcal{I}_P(\pi_{v_1}, \dot{f}_{v_1}))	\\
		&&= a(\pi_{v_1}, \dot{u}_{\overline{x}, v_1}) f_{\dot{G}, \dot{\Xi}, v_1}({M}_{v_1}, \pi_{v_1}, w_{\overline{x}, v_1})
\end{eqnarray*}
for complex coefficients $ a(\pi_{v_1}, \dot{u}_{\overline{x}, v_1})$.   Consequently, we can write
\begin{equation}\label{eq:proof-4.5}
	\dot{f}_{\dot{G}, \dot{\Xi}, v_1} (\dot{\phi}_{v_1}, \dot{u}_{\overline{x}, v_1})
		= \sum_{\pi_{v_1} \in \Pi_{\dot{\phi}_{M^*,v_1}}(M_{v_1})}
		a(\pi_{v_1}, \dot{u}_{\overline{x}, v_1}) \dot{f}_{\dot{G}, \dot{\Xi}, v_1}({M}_{v_1}, \pi_{v_1}, w_{\overline{x}, v_1}).
\end{equation}

  The expression (\ref{equation_lem_analgoue_of_arthur_lem_543_544})
can now be rearranged to the sum
\[
	\sum_{\tau_{v_1} \in T(\dot{G}_{v_1})} d(\tau_{v_1}, \dot{f}^{v_1}) \dot{f}_{\dot{G}, \dot{\Xi}, v_1}(\tau_{v_1}),
\]
where $d(\tau_{v_1}, \dot{f}^{v_1})$ equals to the sum
\[
	c \sum_{\overline{x}} \sum_{\pi_{v_1}}  a(\pi_{v_1}, \dot{u}_{\overline{x}, v_1}) \left(
	(\dot{f}')^{v_1}_{\dot{G}, \dot{\Xi}}(\dot{\phi}^{v_1}, \dot{s}_{\overline{x}}^{v_1})
-   \dot{f}_{\dot{G}, \dot{\Xi}}^{v_1} (\dot{\phi}^{v_1}, \dot{u}_{\overline{x}}^{v_1})\right)
\]
over elements $\overline{x}\in \overline{\cS}_{\dot{\phi}, \mathrm{ell}}$
and $ \pi_{v_1} \in \Pi_{\dot{\phi}_{v_1}}(M_{v_1})$
such that the triplet $({M}_{v_1}, \pi_{v_1}, w_{\overline{x}, v_1})$ belongs in the  $ W^0_{\dot{\phi}_{v_1}}$-conjugacy class
represented by $\tau_{v_1}$. We are now in a position to apply Lemma \ref{lem:vanishing-coeff}. Thereby the proof is reduced to showing that
$d(\tau_{v_1}, \dot{f}^{v_1})=0$ if $\tau_{v_1}$ is represented by a triple of the form
$({M}_{v_1}, \pi_{v_1}, 1)$.

  We have the following commutative diagram
\[
  \xymatrix{
    \ol{\cS}_{\dot{\phi}_{\dot{M}^*}} \ar[r] \ar[d] & \ol{\cS}_{\dot{\phi}}  \ar[r] \ar[d] & R_{\dot{\phi}} \ar[d]  \\
    \ol{\cS}_{\dot{\phi}_{\dot{M}^*, v_1}}  \ar[r] & \ol{\cS}_{\dot{\phi}_{v_1}} \ar[r] & R_{\dot{\phi}_{v_1}}   \\
  }
\]
where the two rows are short exact sequences and the vertical maps are the canonical morphisms.
  By the third condition of the lemma (and the fact that $v_1$ does not split in $\dot E$ so that $Z(\hat G^*)^\Gamma=Z(\hat G^*)^{\Gamma_v}$ and $Z(\hat M^*)^\Gamma=Z(\hat M^*)^{\Gamma_v}$), all vertical maps are isomorphisms.
Since $\dot\phi\in \Phi^2_{\el}(\dot G^*,\dot \eta)$, any element $\overline{x} \in \overline{\cS}_{\dot{\phi}, \mathrm{ell}}$ maps to a non-trivial element in the global $R$-group $R_{\dot{\phi}}$. (If $\dot\phi$ is written in the form \eqref{eq:globla-ell-param} then $\overline{\cS}_{\dot{\phi}, \mathrm{ell}}$ consists of $(\ol{s}_i)\in \{\pm 1\}^r$ such that $\ol{s}_i\neq 1$ for $1\le i\le q$ whereas $R_{\dot{\phi}}\simeq \{\pm 1\}^q$. The map $\overline{\cS}_{\dot{\phi}}\ra R_{\dot{\phi}}$ is the projection onto the first $q$ components.) By the previous diagram, we see that
 $\overline{x}$ must also map to a non-trivial element in the local $R$-group
$R_{\dot{\phi}_{v_1}}$.
It follows that $ w_{\overline{x}, v_1} \neq 1$. Hence $w_{\overline{x}, v_1}=1$ never occurs in the expansion \eqref{eq:proof-4.5}. This forces
 the coefficients $d(\tau_{v_1}, \dot{f}^{v_1})$ to vanish if $\tau_{v_1}$ has the form
$({M}_{v_1}, \pi_{v_1}, 1)$.

\end{proof}

\subsection{Proof of the local intertwining relation}\label{sub:LIR-proof}
Our main task in this section is to prove Theorem \ref{thm:lir}, i.e. the local intertwining relation, for unitary groups. The case of inner forms of linear groups, corresponding to $E=F\times F$, will be treated in \cite{KMS_B}. The main case to treat is that of a discrete parameter of a Levi subgroup. More precisely, let $E/F$ be a quadratic extension of local fields. %
Let $G^*=U_{E/F}(N)$ and let $(M^*,P^*)$ be a standard parabolic pair for $G^*$. Let $\phi_{M^*} \in \Phi_{2,\tx{bdd}}(M^*)$ and denote its image in $\Phi_\tx{bdd}(G^*)$ by $\phi$. Let $(G,\xi)$ be an inner twist of $G^*$. Our convention for $\dot\eta$ and $\eta$ will be the same as in \S\ref{sub:Globalization-param}.

\begin{lem}
\label{lem_lir_equality_first_lem}
Assume that either
\begin{itemize}
	\item $\phi \in \Phi^2_{\mathrm{ell}}(G^*)$ is elliptic non-square integrable, or
	\item $\phi$ is of the form (exc1) or (exc2).
\end{itemize}
Then there exists an equivalence class of pure inner twists $\Xi : G^* \rightarrow G$ extending $\xi$ such that
 for every $\overline{x} \in \overline{\cS}_{\phi, \mathrm{ell}}$, there exists a lift
$x \in S^{\natural}_{\phi}$ of $\overline{x}$ for which
\[
f'_{G,\Xi}(\phi, x^{-1}) = e(G) f_{G, \Xi} (\phi, x) \textrm{ for all } f \in \mathcal{H}(G).
\]
\end{lem}
\begin{proof}
\underline{Case $N$ odd:} We begin by applying Proposition \ref{prop_globalize_data_N_odd} to globalize
the data
$(E/F, G^*, \phi, M^*, \phi_{M^*})$
to
$(\dot{E}/\dot{F}, u, v_1,  \dot{G}^*, \dot{\Xi} = (\dot{G}, \dot{\xi}, \dot{z}), \dot{\phi}, \dot{M}^*,  \dot{\phi}_{\dot{M}^*})$.

If $\phi \in \Phi^2_{\mathrm{ell}}(G^*)$
(resp. $\phi\in\Phi_\EXC(G^*)$)
 then we can apply Lemma \ref{lem_analgoue_of_arthur_lem_543_544}
(resp. Lemma \ref{lem:EXC1-EXC2} and use the fact that $\overline{\cS}_{\dot{\phi}, \mathrm{ell}} = \overline{\cS}_{\dot{\phi}}$ by Lemma \ref{lem:global-S-elliptic})
 to deduce that  for all  $\dot{f}  = \otimes_v \dot f_v \in \mathcal{H}(\dot{G})$,
\begin{equation}
\label{equation_1_lem_lir_equality_first_lem_N_odd}
\sum_{\overline{x} \in \overline{\cS}_{\dot{\phi}, \mathrm{ell}}}(\dot{f}'_{\dot{G}, \dot{\Xi}}(\dot{\phi}, \overline{x}?) - \dot{f}_{\dot{G}, \dot{\Xi}}(\dot{\phi}, \overline{x})) = 0
\end{equation}

 Let $\overline{u}_{\overline{x}} \in \overline{\cN}_{\dot{\phi}}$  be a lift of $\overline{x}$
which is
\begin{itemize}
  \item the unique element whose image lies in $W_{\dot{\phi}, \mathrm{reg}}$ if $\dot{\phi}$ is exceptional,
  \item arbitrary if $\dot\phi$ is elliptic.
\end{itemize}
 Then by definition, $\dot{f}_{\dot{G}, \dot{\Xi}}(\dot{\phi}, \overline{x})= \dot{f}_{\dot{G}, \dot{\Xi}}(\dot{\phi}, \overline{u}_{\overline{x}})$.
We can find a $\dot{s}_{\ol{x}} \in S_{\dot{\phi}, \mathrm{ss}}$ and $\dot{u}_{\ol{x}} \in N^{\natural}_{\dot{\phi}}$ such that
 $\dot{s}_{\overline{x}}$ and $\dot{u}_{\overline{x}}$ have the same image $\dot{x} \in S^{\natural}_{\dot{\phi}}$ and $\dot{u}_{\ol{x}}$ lies over $\overline{u}_{\overline{x}}$.  Then
\[
	\dot{f}'_{\dot{G}, \dot{\Xi}}(\dot{\phi}, \overline{x}^{-1}) = \dot{f}'_{\dot{G}, \dot{\Xi}}(\dot{\phi}, \dot{s}_{\overline{x}}^{-1})
		=  \prod_{v}\dot{f}'_{\dot{G}, \dot{\Xi}, v}(\dot{\phi}_v, \dot{s}_{\overline{x}, v}^{-1})
\]
and
\[
	\dot{f}_{\dot{G}, \dot{\Xi}}(\dot{\phi}, \overline{x}) = \dot{f}_{\dot{G}, \dot{\Xi}}(\dot{\phi}, \dot{u}_{\overline{x}})
		=  \prod_{v}\dot{f}_{\dot{G}, \dot{\Xi}, v}(\dot{\phi}_v, \dot{u}_{\overline{x}, v}).
\]

The local intertwining relation for $\dot\phi_v$ where $v\neq u$ is known as the group $\dot{G}_v$ is quasi-split.
This implies that for all $v\neq u$,
$\dot{f}'_{\dot{G}, \dot{\Xi}, v}(\dot{\phi}_v, \dot{s}^{-1}_{\overline{x}, v}) = e(\dot{G}_v) \dot{f}_{\dot{G}, \dot{\Xi}, v}(\dot{\phi}_v, \dot{u}_{\overline{x}, v})$. (Of course $e(\dot{G}_v)=1$ for $v\neq u$.)
Consequently, Equation \eqref{equation_1_lem_lir_equality_first_lem_N_odd} becomes
\begin{equation}
\label{equation_2_lem_lir_equality_first_lem_N_odd}
\sum_{\overline{x} \in \overline{\cS}_{\dot{\phi}, \mathrm{ell}}} \left(
\prod_{v\neq u}  \dot{f}_{\dot{G}, \dot{\Xi}, v}(\dot{\phi}_v, \dot{u}_{\overline{x}, v})\right)
(\dot{f}'_{\dot{G}, \dot{\Xi}, u}(\dot{\phi}_u, \dot{s}^{-1}_{\overline{x}, u}) - e(\dot{G}_u) \dot{f}_{\dot{G}, \dot{\Xi}, u}(\dot{\phi}_u, \dot{u}_{\overline{x}, u})) = 0.
\end{equation}

As $\dot{G}_{v_1}$ is quasi-split, both the
local intertwining relation and the local classification theorem for $\dot{\phi}_{v_1}$ are known.
It follows that
\[
\dot{f}_{\dot{G}, \dot{\Xi}, v_1}(\dot{\phi}_{v_1}, \dot{u}_{\overline{x}, v_1}) =
	\sum_{\pi_{v_1} \in \Pi_{\dot{\phi}_{v_1}}} \langle \pi_{v_1} ,\dot{x}_{v_1} \rangle \dot{f}_{\dot{G}, \dot{\Xi}, v_1}(\pi_{v_1}).
\]
The place $v_1$ is finite and $\dot{\phi}_{v_1}$ is a generic bounded  parameter.
  Consequently
\[
	\langle \, , \, \rangle : \Pi_{\dot{\phi}_{v_1}} \times \overline{\cS}_{\dot{\phi}_{v_1}}
		\rightarrow \C
\]
is a perfect pairing in the sense that it induces a bijection from $\Pi_{\dot{\phi}_{v_1}}$ onto $\Irr(\overline{\cS}_{\dot{\phi}_{v_1}})$.
It follows that the linear forms $\dot{f}_{\dot{G}, \dot{\Xi}, v_1}(\dot{\phi}_{v_1}, \dot{u}_{\overline{x}, v_1})$
are linearly independent as $\ol{x}$ runs over the set $\ol{\cS}_{\dot\phi,\el}$.
We may identify $\overline{\cS}_{\dot{\phi}, \mathrm{ell}}$ with
$\overline{\cS}_{\dot{\phi}_{v_1}, \mathrm{ell}}$.
Using linear independence, we deduce that
\[
\prod_{v\neq u} \dot{f}_{\dot{G}, \dot{\Xi}, v}(\dot{\phi}_v, \dot{u}_{\overline{x}, v})
(\dot{f}'_{\dot{G}, \dot{\Xi}, u}(\dot{\phi}_u, \dot{s}^{-1}_{\overline{x}, u}) - e(\dot{G}_u) \dot{f}_{\dot{G}, \dot{\Xi}, u}(\dot{\phi}_u, \dot{u}_{\overline{x}, u})) = 0.
\]
As the parameter $\dot{\phi}_v$ is relevant for all $v \neq u$, the linear form
$\prod_{v\neq u} \dot{f}_{\dot{G}, \dot{\Xi}, v}(\dot{\phi}_v, \dot{u}_{\overline{x}, v})$ does not vanish identically.
It follows that
$$\dot{f}'_{\dot{G}, \dot{\Xi}, u}(\dot{\phi}_u, \dot{s}^{-1}_{\overline{x}, u}) = e(\dot{G}_u)  \dot{f}_{\dot{G}, \dot{\Xi}, u}(\dot{\phi}_u, \dot{u}_{\overline{x}, u}).$$

By definition $\dot{f}'_{\dot{G}, \dot{\Xi}, u}(\dot{\phi}_u, \dot{s}^{-1}_{\overline{x}, u})
= \dot{f}'_{\dot{G}, \dot{\Xi}, u}(\dot{\phi}_u, \dot{x}^{-1}_{u})$.
If $\phi$ is exceptional, then the image of $\dot{u}_{\overline{x}, u}$ in $W_{\dot{\phi}_u}$
lies in $W_{\dot{\phi}_u, \mathrm{reg}}$.  Consequently by definition,
$\dot{f}_{\dot{G}, \dot{\Xi}, u}(\dot{\phi}_u, \dot{u}_{\overline{x}, u})
= \dot{f}_{\dot{G}, \dot{\Xi}, u}(\dot{\phi}_u, \dot{x}_{u})$.  Thus, we have shown that
$\dot{f}'_{\dot{G}, \dot{\Xi}, u}(\dot{\phi}_u, \dot{x}^{-1}_{u}) =  e(\dot{G}_u)  \dot{f}_{\dot{G}, \dot{\Xi}, u}(\dot{\phi}_u, \dot{x}_{u})$.
Our globalization allows to identify $\overline{\cS}_{\dot{\phi}, \mathrm{ell}}$ with $\overline{\cS}_{\dot{\phi}_{u}, \mathrm{ell}}$ and $S^\natural_\phi$ with $S^\natural_{\dot\phi_u}$, and also to rewrite the above identity as
$$f'_{G,\Xi}(\phi,\dot x^{-1}_u)=e(G)f_{G,\Xi}(\phi,\dot x_u),$$ where $\dot x_u$ lifts $\ol{x}$ by construction.

\medskip

\underline{Case $N$ even:}  The method of proof is much the same as the odd case except we must appeal to the more complicated globalizations. Let us begin with the case that $M^*$ is not linear.
If $N \neq 4$, then the proof is the same as for the case $N$ odd. The only difference is that we appeal  to
Proposition \ref{prop_globalize_data_N_even_and_M_neq_G_and_N_geq_6} instead of Proposition \ref{prop_globalize_data_N_odd}, and we note that via the induction hypothesis, we know the local intertwining relation for the  non elliptic non exceptional  relevant parameter $\dot{\phi}_{v_2}$.

Assume now that $N=4$.  We firstly obtain the result for exceptional parameters  by arguing as above and appealing to
Proposition \ref{prop_globalize_data_N_4_exceptional_case} using the fact that by Proposition \ref{pro:u31lir} the local intertwining relation is already known for a specific choice of exceptional parameter of both type (exc1) and (exc2) at archimedean places for $N=4$.
We then obtain the results for elliptic parameters by arguing as above and appealing to Proposition
\ref{prop_globalize_data_N_4_reduction_exceptional_case} and using the fact that the local intertwining relation is now known for all exceptional parameters for $N=4$.

We are left to consider the case that $M^*$ is  linear.  If the group is $G$ is quasi-split, then we know the local intertwining relation and hence the result is known.  We may now assume that $G$ is not quasi-split.  For such a group the Levi $M^*$ will not transfer to a Levi of $G$.  By definition, the form
$f_{G,\Xi}(\phi, x)  = 0$ vanishes and we must show that the form $f'_{G,\Xi}(\phi, x^{-1}) = 0$ also vanishes for some
lift $x \in S^{\natural}_{\phi}$ of $\overline{x}$.  By arguing as in the proof of
the case $N$ odd, but appealing to Proposition \ref{prop_globalize_data_N_even_and_M_neq_G_split}, we deduce using the obvious notation that
\[
\dot{f}'_{\dot{G}, \dot{\Xi}, u_1}(\dot{\phi}_{u_1}, \dot{u}^{-1}_{\overline{x}, u_1})
\dot{f}'_{\dot{G}, \dot{\Xi},  u_2}(\dot{\phi}_{u_2}, \dot{u}^{-1}_{\overline{x}, u_2})
= 0.
\]
That is either $\dot{f}'_{\dot{G}, \dot{\Xi}, u_1}(\dot{\phi}_{u_1}, \dot{u}^{-1}_{\overline{x}, u_1}) = 0$
or
$\dot{f}'_{\dot{G}, \dot{\Xi}, u_2}(\dot{\phi}_{u_2}, \dot{u}^{-1}_{\overline{x}, u_2})
= 0$.  The result follows.
\end{proof}

  We remind the reader that our convention here is to write $\phi_{M^*}$ for $\psi$ of Theorem \ref{thm:lir} when the parameter is generic. The image of $\phi_{M^*}$ in $\Phi_\tx{bdd}(G^*)$ is denoted by $\phi$ as before.
  Now we are ready to complete the proof of LIR for all generic parameters.

\begin{lem}
\label{lem_lir_part1_lem}
Parts 2 and 3 of Theorem \ref{thm:lir} hold for  generic  parameters.
\end{lem}
\begin{proof}
Clearly part 2 is implied by part 3, so we will concentrate on the latter. Given the results of Sections \ref{sub:prelim-local-intertwining1} and \ref{sub:prelim-local-intertwining} we can assume that $\phi_{M^*}\in \Phi_{2,\tx{bdd}}(M^*)$ and $\phi\in \Phi^2_{\el}(G^*)\coprod \Phi_{\EXC}(G^*)$.

If $\phi \in \Phi^2_{\mathrm{ell}}(G^*)$ is elliptic non-square integrable, then by the induction hypothesis we know the following.
\begin{itemize}
\item[i)]
$f_{G, \Xi}(\phi, u)$ is the same for every $u \in N^{\natural}_\phi$ mapping to the same $x \in S^{\natural}_\phi$.
\item[ii)]
$f'_{G, \Xi}(\phi, s)$ is the same for every $s \in S_{\phi, \mathrm{ss}}$ mapping to the same
	$x \in S^{\natural}_\phi$.
\item[iii)]
$f'_{G, \Xi}(\phi, s^{-1}) = e(G) f_{G, \Xi}(\phi, u) $ for all  $u \in N^{\natural}_\phi$ and $s \in S_{\phi, \mathrm{ss}}$ mapping to the same $x \not\in S^{\natural}_{\phi}$.

\end{itemize}

We shall now deduce
\begin{itemize}
  \item[iv)] $f'_{G, \Xi}(\phi, s^{-1}) = e(G) f_{G, \Xi}(\phi, u) $ for all  $u \in N^{\natural}_\phi$ and $s \in S_{\phi, \mathrm{ss}}$ mapping to the same $x \in S^{\natural}_{\phi}$.
\end{itemize}
after possibly changing $\Xi$, but not the underlying inner twist.
Fix any $x \in S^{\natural}_{\phi,\mathrm{ell}}$ and let $\overline{x} \in \overline{\cS}_{\phi}$ denote the image of $x$, which belongs to $\ol{\cS}_{\phi,\el}$ by Lemma \ref{lem:S-elliptic}.
Lemma \ref{lem_lir_equality_first_lem} implies that $f'_{G, \Xi}(\phi, s^{-1}) = e(G) f_{G, \Xi}(\phi, u) $ whenever $u$ and $s$ map to some $x'\in S^{\natural}_{\phi,\mathrm{ell}}$ which lifts $\ol{x}$. Then we conclude that iv) holds true for the given $x$ thanks to Lemma \ref{lem:central-action}. Now that i)-iv) above give us the desired result for elliptic parameters for a particular class of (extended) pure inner twists, we extend the same result to all extended pure inner twists with the same underlying inner twist by Lemma \ref{lem:lirind}.

We shall now consider exceptional parameters and argue in a similar way.
If $\phi$ is an exceptional parameter,
then by the induction hypothesis we know the following.
\begin{itemize}
\item[i)]
$f_{G, \Xi}(\phi, u)$ is the same for every $u \in N_{\phi, \mathrm{reg}}^{\natural}$
(resp. $u \not\in  N_{\phi, \mathrm{reg}}^{\natural}$)
 mapping to the same $x \in S^{\natural}_{\phi}$.
\item[ii)]
$f'_{G, \Xi}(\phi, s)$ is the same for every $s \in S_{\phi, \mathrm{ss}}$ mapping to the same
	$x \in S^{\natural}_\phi$.
\item[iii)]
$f'_{G, \Xi}(\phi, s^{-1}) = e(G) f_{G, \Xi}(\phi, u) $ for all  $u \not\in N^{\natural}_{\phi, \mathrm{reg}}$ and $s \in S_{\phi, \mathrm{ss}}$ mapping to the same $x \in S^{\natural}_{\phi, \mathrm{ell}}$.

\end{itemize}

  Now we are about to prove
\begin{itemize}
  \item[iv)] $f'_{G, \Xi}(\phi, s^{-1}) = e(G) f_{G, \Xi}(\phi, u) $ for all  $u \in N^{\natural}_{\phi, \mathrm{reg}}$ and $s \in S_{\phi, \mathrm{ss}}$ mapping to the same $x \in S^{\natural}_{\phi, \mathrm{ell}}$.
\end{itemize}
Again fix $x \in S^{\natural}_{\phi,\mathrm{ell}}$ and let $\overline{x} \in \overline{\cS}_{\phi,\el}$ denote its image. We deduce from Lemmas \ref{lem_lir_equality_first_lem} and \ref{lem:central-action} that $f'_{G, \Xi}(\phi, s^{-1}) = e(G) f_{G, \Xi}(\phi, u) $ whenever $u$ and $s$ as in iv) map to the particular $x\in S^{\natural}_{\phi,\mathrm{ell}}$, after possibly changing $\Xi$, but not the underlying inner twist. (Recall that $f_{G,\Xi}(\phi,x)$ is defined to be $f_{G,\Xi}(\phi,u)$ for the unique $u \in N^{\natural}_{\phi, \mathrm{reg}}$ lifting $x$ in this case.) It follows from i)-iv) that the theorem holds for $\phi\in \Phi_{\EXC}(G^*)$ and a particular equivalence class of extended pure inner twists with the underlying inner twist given by the original $\Xi$. Then the theorem holds for all extended pure inner twists by Lemma \ref{lem:lirind}.

\end{proof}

\begin{lem}
\label{lem_lir_triv_R}
Let $\Xi : G^* \rightarrow G$ be an equivalence class of extended pure inner twist and let $\phi_{M^*} \in \Phi_{2,\tx{bdd}}(M^*)$.
Then for all $u^\natural \in W^\tx{rad}_{\phi}(M,G)$ and $\pi_M \in \Pi_{\phi_{M^*}}$ we have
\[
R_P(u^\natural, \Xi, \pi_M, \phi_{M^*}, \psi_F) = 1
\]
\end{lem}
\begin{proof}
If $\phi$ is either elliptic, or non-elliptic and non-exceptional, then the lemma is already known by Lemma \ref{lem:local-indep-f} (and Lemma \ref{lem:local-exc}). So we are reduced to the case that $\phi$ is of the form (exc1) or (exc2). For each inner twist $(G,\xi)$, we remark that the operator $R_P(u^\natural, \Xi, \pi_M, \phi_{M^*}, \psi_F)$ is the same for every $\Xi$ whose underlying inner twist is $(G,\xi)$ in view of part 2 of Lemma \ref{lem:rpequiv} and the fact that the pairing $\lg c,u^\natural\rg$ there is trivial if $u^\natural \in W^\tx{rad}_{\phi}(M,G)$ (since $N(A_{\hat M^*},S^{\rad}_{\psi_{M^*}})\subset \hat G_{\der}$).

We recall that $\lvert W^\rad_\phi \rvert = 2$ when $\phi\in \Phi_{\EXC}(G^*)$.
Let $u_0 \in W^\rad_\phi$ (resp. $u_1 \in W^\rad_\phi$) denote the trivial (resp. non-trivial) element.
By construction, the intertwining operator is trivial for the identity element $u_0$.

We globalize as in the proof of Lemma \ref{lem_lir_equality_first_lem} if $N$ is odd and as in Proposition \ref{prop_globalize_data_N_even_and_M_neq_G_split} if $N$ is even, and then apply the last part of Lemma \ref{lem:EXC1-EXC2}. Since Lemma \ref{lem_lir_triv_R} is known at quasi-split places (which we take on faith; see the discussion of \S\ref{sub:results-qsuni}), we deduce that $R_P(u_1, \Xi, \pi_M, \phi_{M^*}, \psi_F)$ is the identity operator if $N$ is odd, and
$R_P(u_1, \Xi, \pi_M, \phi_{M^*}, \psi_F)^{\otimes 2}$ is the identity on $\cH_P(\pi_M)^{\otimes 2}$ and thus
$$R_P(u_1, \Xi, \pi_M, \phi_{M^*}, \psi_F)~=~+1~\mbox{or}~-1$$
if $N$ is even. To be precise we may have changed $\Xi$ without disturbing the underlying inner twist in the globalization, but we keep the same notation $\Xi$ as this does not affect the intertwining operator by the remark above.

We may focus on the even case from now. Lemma \ref{lem_lir_part1_lem} implies that $f_{G,\Xi}(\phi,u_1u_0)=f_{G,\Xi}(\phi,u_0)$ so we obtain by unraveling the definition that
$$\sum_{\pi_M\in \Pi_{\psi_{M^*}}(M,\Xi)}  \tr \left((1-R_P(u_1,\Xi,\pi_M,\phi_{M^*},\psi_F))\cI_P(\pi_M)(f)\right)=0.$$
Since the left hand side is the trace of a linear combination of (nonzero) representations with nonnegative integer coefficients, it follows that $R_P(u_1,\Xi,\pi_M,\phi_{M^*},\psi_F)=1$ as desired.

\end{proof}

We have now completed the proof of Theorem \ref{thm:lir} for all parameters $\phi_{M^*} \in \Phi_{2,\tx{bdd}}(M^*)$. We will now argue that Theorem \ref{thm:lir} holds in the more general case $\phi_{M^*} \in \Phi_\tx{bdd}(M^*)$. Most of the work for this has already been done in Section \ref{sub:prelim-local-intertwining1}. Namely, Proposition \ref{pro:lirreddisc} tells us that parts 2 and 3 of Theorem \ref{thm:lir} are valid for any $\phi_{M^*} \in \Phi_\tx{bdd}(M^*)$. We will now treat part 1. At the same time, we will show the closely related statement that the representation-theoretic and endoscopic $R$-groups coincide.

We recall that the $R$-group $R_{\pi_M}(M,G)$ of a $\pi_M \in \Pi_{\phi_{M^*}}$ is a priori different from the $R$-group $R_\phi(M,G)$ of $\phi$. We have seen that the stabilizer $W_\phi$ of $\phi_{M^*}$ is identified with a subgroup of $W(M,G)(F)$, the Weyl group that
contains the stabilizer $W_{\pi_M}(M,G)$ of $\pi_M$. It is a consequence of the disjointness of tempered $L$-packets for $M$ that $W_\phi$ contains $W_{\pi_M}(M,G)$.  Conversely, it follows from the form of $M$ and the induction hypothesis that any element in $W_\phi(M,G)$ stabilizes $\pi_M$.  Therefore $W_\phi(M,G) = W_{\pi_M}(M,G)$.

If we assume that $W^0_\phi(M,G)$ acts trivially on the induced representation $\mathcal{I}_P(\pi_M)$, then $W^0_\phi(M,G) \subset W_{\pi_M}^0(M,G)$ and we obtain a surjection
\begin{equation}
\label{equation_map_of_r_groups_surjective_generic}
	R_\phi(M,G) = W_\phi(M,G)/W^0_\phi(M,G) \twoheadrightarrow R_{\pi_M}(M,G) = W_{\pi_M}(M,G)/W^0_{\pi_M}(M,G).
\end{equation}
We will now see that this $W^0_\phi(M,G)$ indeed acts trivially on the induced representation $\mathcal{I}_P(\pi_M)$ and that moreover \eqref{equation_map_of_r_groups_surjective_generic} is a bijection.

\begin{lem}
\label{lem_lir_triv_R2}
Let $\Xi : G^* \rightarrow G$ be an equivalence class of extended pure inner twist and let $\phi_{M^*} \in \Phi_\tx{bdd}(M^*)$.
Then for all $u^\natural \in W^\tx{rad}_{\phi}(M,G)$ and $\pi_M \in \Pi_{\phi_{M^*}}$ we have
\[
R_P(u^\natural, \Xi, \pi_M, \phi_{M^*}, \psi_F) = 1
\]
\end{lem}
\begin{proof}
The case of $\phi_{M^*} \in \Phi_{2,\tx{bdd}}(M^*)$ was already treated in Lemma \ref{lem_lir_triv_R}. So we may assume that $\phi_{M^*}$ is the image of $\phi_{M^*_0} \in \Phi_{2,\tx{bdd}}(M^*_0)$ for a proper standard Levi subgroup $M^*_0 \subset M^*$. As in the proof of Proposition \ref{pro:lirreddisc} we choose a lift $u \in N_{\phi_{M^*_0}}(M_0,G) \cap N_{\phi_{M^*}}(M,G)$ of $u^\natural$. The image of $u$ in $N_{\phi_{M^*_0}}^\natural(M_0,G)$ belongs to $W_{\phi_{M^*_0}}^\tx{rad}(M_0,G)$. Lemma \ref{lem_lir_triv_R} then implies that for all $\pi_{M_0} \in \Pi_{\phi_{M^*_0}}(M_0,\Xi)$ we have
\[ R_P(u^\natural, \Xi, \pi_{M_0}, \phi_{M^*_0}, \psi_F) = 1. \]
By definition then
\[ f_{G,\Xi}(u^\natural,\phi_{M^*_0}) = \sum_{\pi_{M_0} \in \Pi_{\phi_{M^*_0}}(M_0,\Xi)} \tx{tr}(\mc{I}_{P_0}(\pi_{M_0},f)). \]
We now consider the homomorphism \eqref{equation_map_of_r_groups_surjective_generic} for the Levi subgroup $M_0$ and the parameter $\phi_{M^*_0}$. This homomorphism is well-defined according to \ref{lem_lir_triv_R} and tells us that the representation-theoretic $R$-group $R_{\pi_{M_0}}(M_0,G)$ is abelian. As a result, the representation $\mc{I}_{P_0}(\pi_{M_0})$ decomposes as a direct sum of irreducible representations of $G(F)$, each occurring with multiplicity 1. The linear form $f_G(u^\natural,\phi_{M^*_0})$ is then the sum of the traces of all these constituents, as $\pi_{M_0}$ varies over $\Pi_{\phi_{M^*_0}}(M_0,\Xi)$.

Now let $\pi_M \in \Pi_\phi(M,\Xi)$. By induction in stages we know that $\mc{I}_P(\pi_M)$ must also be multiplicity-free. This forces the intertwining operator $R_P(u^\natural, \Xi, \pi_M, \phi_{M^*}, \psi_F)$ to act as the multiplication by scalar on each constituent of $\mc{I}_P(\pi_M)$. A-priori the scalars for different constituents could be different, but we will see momentarily that this is not the case. At any rate, this tells us that the linear form $f_{G,\Xi}(u^\natural,\phi_{M^*})$ is a linear combination of the traces of the constituents of $\mc{I}_P(\pi_M)$, as $\pi_M$ varies over $\Pi_{\phi_{M^*}}(M,\Xi)$. But recalling the construction of $\Pi_{\phi_{M^*}}(M,\Xi)$ given in Section \ref{sec:lir} and using induction in stages we see that $f_G(u^\natural,\phi_{M^*})$ is in fact a linear combination of the traces of the constituents of $\mc{I}_{P_0}(\pi_{M_0})$, as $\pi_{M_0}$ varies over $\Pi_{\phi_{M^*_0}}(M_0,\Xi)$.

We now apply Lemma \ref{lem:lirdesc1} and see $f_{G,\Xi}(u^\natural,\phi_{M^*})=f_{G,\Xi}(u^\natural,\phi_{M^*_0})$. This tells us that all the scalars in the linear combination of the constituents of $\mc{I}_P(\pi_M)$ that make out $f_{G,\Xi}(u^\natural,\phi_{M^*})$ are equal to 1.
\end{proof}

With this lemma, the proof of Theorem \ref{thm:lir} is finally complete for all $\phi_{M^*} \in \Phi_\tx{bdd}(M^*)$. Furthermore, we now know that the homomorphism \eqref{equation_map_of_r_groups_surjective_generic} is well-defined for all $\phi_{M^*} \in \Phi_\tx{bdd}(M^*)$. We will now show that it is in fact an isomorphism. For this, recall that the definition of regular %
 implies that the set $R_{\phi, \mathrm{reg}}$ is the pre-image of $R_{\mathrm{reg}}(\pi_M)$ under this map. Recall also that the set $R_{\phi, \mathrm{reg}}$ is empty if $\phi$ is not elliptic, and it has cardinality $1$ if $\phi$ is elliptic.  (Compare with the last paragraph in the proof of Proposition 6.3.1 in \cite{Arthur}.)

\begin{lem}
\label{lem_tempered_generic_repns_r_groups_isomorphic}
For any relevant $\phi_{M^*} \in \Phi_\tx{bdd}(M^*)$ the homomorphism \eqref{equation_map_of_r_groups_surjective_generic} is an isomorphism.
\end{lem}
\begin{proof}
If $\phi$ is elliptic, then the set $R_{\phi, \mathrm{reg}}$ has cardinality $1$.
As  the set $R_{\phi, \mathrm{reg}}$ is the pre-image of
$R_{\mathrm{reg}}(\pi_M)$ under the morphism  \eqref{equation_map_of_r_groups_surjective_generic}, it
follows that the kernel of this morphism has cardinality $1$.  That is
 the
surjective
morphism  \eqref{equation_map_of_r_groups_surjective_generic} is an isomorphism.

Let us now show that this holds even without the additional hypothesis that $\phi$ be elliptic.
The set $R_{\phi}$ is equal to the  disjoint union of $R_{\phi_L, \mathrm{reg}}$ where $L$ ranges over all the Levi subgroups $L$ of $G$  containing $M$ such that $\phi$ is the image of an elliptic $\phi_L \in \Phi_{\mathrm{ell}}(L)$.
By arguing as above for each $\phi_L$, we deduce that the surjective morphism $R_\phi  \rightarrow R(\pi_M)$ is bijective and hence an isomorphism.
\end{proof}

\subsection{Local packets for non-discrete parameters} \label{sec:lpackns}

Let $G^*=U_{E/F}(N)$ and let $\Xi: G^* \rightarrow G$ be an equivalence class of extended pure inner twists. Having proved the local intertwining relation in Section \ref{sub:LIR-proof} we can now complete the proof of the local classification theorem \ref{thm:locclass-single} for parameters $\phi \in \Phi_{\mathrm{bdd}}(G^*) \sm \Phi_2(G^*)$. We remind the reader that the archimedean case is already known by the work of Langlands \cite{Lan89}, and Shelstad \cite{She82}, \cite{SheTE2}, \cite{SheTE3}. The reader may also consult \cite[\S5.6]{Kal13} for an exposition that parallels closely the statement of Theorem \ref{thm:locclass-single}. Thus, we assume from now on until the rest of the chapter that $E/F$ is a quadratic extension of non-archimedean local fields.

Let $(\xi, z) \in \Xi$ be as in Section \ref{sec:iop3u}. Recall that the $L$-packet $\Pi_\phi(G,\Xi)$ and the map $\Pi_\phi(G,\Xi) \rw \tx{Irr}(S_\phi^\natural,\chi_\Xi)$ was already defined at the end of Section \ref{sec:lir}. Moreover, the character identities expressed in part 4 of Theorem \ref{thm:locclass-single} were also deduced there. Both the construction as well as the proof of the character identities were conditional on the validity of Theorem \ref{thm:lir}, which has now been established.

Thus parts 1, 2, and 4 of Theorem \ref{thm:locclass-single} have now been proved for $\phi \in \Phi_{\mathrm{bdd}}(G^*) \backslash \Phi_2(G^*)$. What remains to be shown is the bijectivity of the map $\Pi_\phi(G,\Xi) \rw \tx{Irr}(S_\phi^\natural,\chi_\Xi)$ and the fact that the packets $\Pi_\phi(G,\Xi)$ are disjoint and exhaust $\Pi_\tx{temp}(G)$.

Recall the general classification of $\Pi_{\mathrm{temp}}(G)$ by harmonic analysis from \cite[\S1]{ArtETC} (cf. \cite[p.154]{Arthur}). It characterizes $\Pi_{\mathrm{temp}}(G)$ as the image of a bijection
\[
	\{(M, \sigma, \mu) \} \stackrel{\sim}{\rightarrow} \{\pi_\mu\}, \ M \in \mathcal{L}, \sigma \in \Pi_{2, \mathrm{temp}}(M), \mu \in \Pi(\tilde R_\sigma(M,G)).
\]
The left hand side is comprised of $G(F)$-orbits of triples consisting of a Levi subgroup $M$, a discrete series representation $\sigma$ of $M(F)$, and an irreducible representation $\mu$ of the representation-theoretic $R$-group $R_\sigma(M,G)$. While in general one needs to use an extension of $R_\sigma(M,G)$, in the case of unitary groups this extension is not necessary, as we have shown in Section \ref{sec:iop3u} that $\sigma$ has a canonical extension to the group $M(F) \rtimes W_\sigma(M,G)$. The right hand side is the set of irreducible constituents of $\mc{I}_P^G(\sigma)$, where $P$ is any parabolic subgroup of $G$ with Levi factor $M$.

This discussion immediately implies that the packets $\Pi_\phi(G,\Xi)$ exhaust the set $\Pi_\tx{temp}(G)$. It also implies that two packets $\Pi_{\phi_i}(G,\Xi)$ for inequivalent $\phi_1,\phi_2 \in \Phi_{\mathrm{bdd}}(G^*) \backslash \Phi_2(G^*)$ are disjoint. Indeed, the packet $\Pi_{\phi_1}(G,\Xi)$ is constructed to consist of the irreducible constituents of $\mc{I}_P(\sigma)$ for all $\sigma \in \Pi_{\phi_{1,M}}(M,\Xi)$, where $\phi_{1,M} \in \Phi_{2,\tx{bdd}}(M^*)$ is any parameter with image $\phi_1$. An non-trivial intersection between $\Pi_{\phi_1}(G,\Xi)$ and $\Pi_{\phi_2}(G,\Xi)$ would imply by the above discussion that both $\phi_1$ and $\phi_2$ are images of two discrete parameters $\phi_{1,M_1},\phi_{2,M_2}$ for the $W_0^G$-conjugate Levi subgroups $M_1,M_2$ of $G$. We may assume $M_1=M_2=M$ after replacing $\phi_2$ by an equivalent parameter. Furthermore, we must have $\sigma_1 \in \Pi_{\phi_{1,M}}(M,\Xi)$ and $\sigma_2 \in \Pi_{\phi_{2,M}}(M,\Xi)$ with $\mc{I}_P(\sigma_1) \cap \mc{I}_P(\sigma_2) \neq \emptyset$, but by \cite[Prop 1.1]{ArtETC} this implies $\sigma_1 = \sigma_2$, which by the disjointness of tempered $L$-packets for $M$, assumed by induction, implies $\phi_{1,M}=\phi_{2,M}$ and thus $\phi_1=\phi_2$.

What remains to be shown is the bijectivity of the map $\Pi_\phi(G,\Xi) \rw \tx{Irr}(S_\phi^\natural,\chi_\Xi)$. For this we recall that the map $(M,\sigma,\mu) \mapsto \pi_\mu$ discussed above can be explicitly described. It depends on a choice of normalization of the self-intertwining operator $R_P(r,\sigma) \in \tx{Aut}_G(\mc{I}_P^G(\sigma))$ for each $r \in R_\sigma(M,G)$ so that the map
\[ R_\sigma(M,G) \rw \tx{Aut}_G(\mc{I}_P^G(\sigma)),\qquad r \mapsto R_P(r,\sigma) \]
is a homomorphism. Given this choice, the bijection is determined by decomposing the representation $\mc{R}$ of $R_\sigma(M,G) \times G(F)$ on the space $\mc{H}_P^G(\sigma)$ given by $\mc{R}(r,f) = R_P(r,\sigma)\circ\mc{I}_P^G(\sigma,f)$. This decomposition takes the form
\[ \mc{R} = \bigotimes_\mu \mu^\vee \otimes \pi_\mu, \]
where $\mu$ runs over the set of characters of the (in our case abelian) group $R_\sigma(M,G)$. This decomposition has the character-theoretic form
\[ \tx{tr}(R_P(r,\sigma)\circ \mc{I}_P^G(\sigma,f)) = \sum_\mu \mu^\vee(r)\tx{tr}(\pi_\mu(f)),\qquad r \in R_\sigma(M,G), f \in \mc{H}(G). \]
According to Lemma \ref{lem_tempered_generic_repns_r_groups_isomorphic} we have $R_\phi(M,G) = R_\sigma(M,G)$. This allows us to inflate the homomorphism $R_\sigma(M,G) \rw \tx{Aut}_G(\mc{I}_P(\sigma))$ to a homomorphism $S_\phi^\natural \rw \tx{Aut}_G(\mc{I}_P(\sigma))$. On the other hand, we have a second homomorphism with the same source and target, namely
\[ S_\phi^\natural \rw \tx{Aut}_G(\mc{I}_P(\sigma)),\qquad s \mapsto R_P(u^\natural,\Xi,\sigma,\phi,\psi_F), \]
where $u^\natural \in N_\phi^\natural(M,G)$ is any lift of $s$, and where the operator on the right is the normalization of the self-intertwining operator introduced in Section \ref{sec:iop}. We have implicitly identified $\phi$ with an element of $\Phi_{2,\tx{bdd}}(M^*)$ here. The right hand side is independent of the choice of lift due to part 1 of Theorem \ref{thm:lir}. If $r \in R_\sigma(M,G)$ is the image of $s \in S_\phi^\natural$, then up to a scalar multiple the two operators $R_P(u^\natural,\Xi,\sigma,\phi,\psi_F)$ and $R_P(r,\sigma)$ are equal. Thus there exists a character $\chi : S_\phi^\natural \rw \C^\times$ such that $R_P(u^\natural,\sigma,\phi,\psi_F)=\chi(s)R_P(r,\sigma)$. Of course, $\chi$ depends on the (arbitrary) choice of $R_P(r,\sigma)$ we are using here. It follows that
\[ \tx{tr}(R_P(u^\natural,\Xi,\sigma,\phi,\psi_F)\circ \mc{I}_P^G(\sigma,f)) = \sum_\mu \chi(s)\mu^\vee(s)\tx{tr}(\pi_\mu(f)),\quad s\in S_\phi^\natural, f \in \mc{H}(G), \]
again with the understanding that $u^\natural \in N_\phi^\natural(M,G)$ is an arbitrary lift of $s$. The summation index is the set of characters $\mu$ of $S_\phi^\natural$ that are inflated from $R_\phi(M,G)$. Since $S_\phi^\natural$ is abelian, this set is in bijection with the set of characters whose restriction to $S_\phi^{\natural\natural}(M)$ is equal to the restriction of $\chi^{-1}$ to that group. Thus, after reindexing, we arrive at
\[ \tx{tr}(R_P(u^\natural,\sigma,\phi,\psi_F)\circ \mc{I}_P^G(\sigma,f)) = \sum_{\mu'} (\mu')^\vee(s)\tx{tr}(\pi_{\mu'}(f)),\qquad s\in S_\phi^\natural, f \in \mc{H}(G), \]
where now $\mu'$ runs over the set of characters of $S_\phi^\natural$ whose restriction to $S_\phi^{\natural\natural}(M)$ is equal to the restriction of $\chi^{-1}$. But the restriction of $\chi$ can be easily determined. Indeed, for $s \in S_\phi^{\natural\natural}(M)$ we have $\chi(s)=R_P(s,\sigma,\phi,\psi_F)=\<\sigma,s\>_{\xi,z}^{-1}$ by Lemma \ref{lem:rpequiv}. It thus follows that $\mu'$ runs over the set of characters of $S_\phi^\natural$ that extend the character $\<\sigma,-\>_{\xi,z}$ of $S_\phi^{\natural\natural}(M)$. Taking the sum over all $\sigma \in \Pi_\phi(M,\Xi_M)$ on the left is equivalent to taking the sum on the right over all characters $\mu'$ of $S_\phi^\natural$ whose restriction to $Z(\hat G)^\Gamma$ is equal to the character $\<z,-\>$. We then arrive at the equation that defines the map $\mu \mapsto \pi_\mu$ in Section \ref{sec:lir}. Thus we conclude that the map $\mu' \mapsto \pi_{\mu'}$ defined there is a bijection, as claimed. We summarize the results of this discussion.
\begin{prop}
The local classification Theorem \ref{thm:locclass-single} holds for $\Xi : G^* \rw G$ and generic parameters $\phi$ in the complement of $\Phi_{2,\tx{bdd}}(G^*)$ in $\Phi_{\mathrm{bdd}}(G^*)$.
\end{prop}

\subsection{Elliptic orthogonality relations}
\label{subsection_elliptic_orthog_reln}
In this section we will discuss the orthogonality relations for elliptic tempered characters, originally derived in \cite{ArtETC}, from the point of view of the normalized intertwining operators introduced in Section \ref{sec:iop}. This discussion will serve as a technical input for the proof of the local classification theorem for generic discrete parameters in the next section. As before $E/F$ is a quadratic extension of non-archimedean local fields, $G^*=U_{E/F}(N)$, and $\Xi: G^* \rightarrow G$ denotes an equivalence class of extended pure inner twists.

In \cite{ArtETC}, Arthur introduces the set $T_\tx{ell}(G)$ of $G(F)$-orbits of triples $\tau=(M,\sigma,r)$, where $M$ is a Levi subgroup of $G$, $\sigma$ is a unitary discrete series representation of $M(F)$, and $r \in R_\sigma(M,G)$ is a regular element. As recalled in the previous section, given such a triple $\tau$ and a parabolic subgroup $P$ of $G$ with Levi factor $M$, there is a normalized intertwining operator $R_P(r,\sigma)$ of the induced representation $\mc{I}_P(\sigma)$, well-defined up to a scalar multiple. Note that the set $\Pi_{2,\tx{temp}}(G)$ of equivalence classes of unitary discrete series representations of $G$ is naturally a subset of $T_\tx{ell}(G)$ by interpreting $\pi \in \Pi_{2,\tx{temp}}(G)$ as the triple $(G,\pi,1) \in T_\tx{ell}(G)$. In this case we take $R_P(1,\sigma)=\tx{id}$. We will denote the complement of $\Pi_{2,\tx{temp}}(G)$ in $T_\tx{ell}(G)$ by $T_\tx{ell}^2(G)$.

The elliptic tempered character of $\tau$ is the distribution
\[ f \mapsto f_G(\tau) = \tx{tr}(R_P(r,\sigma)\circ\mc{I}_P(\sigma,f)), \]
itself well-defined up to a scalar multiple, unless $\tau \in \Pi_{2,\tx{temp}}(G)$. This distribution can be restricted to the space $\mc{H}_\tx{cusp}(G)$ of cuspidal functions. In this way, each $f_G \in \mc{I}_\tx{cusp}(G)$ gives a function
\[ T_\tx{ell}(G) \rw \C,\qquad \tau \mapsto f_G(\tau). \]
A central result of \cite{ArtETC}, tracing back to work of Kazhdan, is that the map $f_G \rw f_G(\tau)$ is an isomorphism of $\C$-vector spaces between $\mc{I}_\tx{cusp}(G)$ and the space of complex valued functions on $T_\tx{ell}(G)$ of finite support.

We will now formulate a version of this result involving the specific normalization of the intertwining operators introduced in Section \ref{sec:iop}. For this, we consider the set $T^{2,\natural}_\tx{ell}(G)$ of $G(F)$-conjugacy classes of triples $(M,\sigma,s)$, where $M$ is a proper Levi subgroup of $G$, $\sigma$ is a unitary discrete series representation of $M$, and $s \in S_{\phi_\sigma}^\natural$ is an element whose image in $R_{\phi_\sigma}(M,G) \cong R_\sigma(M,G)$ is regular. We have used here the isomorphism given by Lemma \ref{lem_tempered_generic_repns_r_groups_isomorphic}, as well as the validity of Theorem \ref{thm:locclass-single} for $M$, which assigns to $\sigma$ a discrete generic parameter ${\phi_\sigma} \in \Phi_{2,\tx{bdd}}(M^*)$ whose image in $\Phi_\tx{bdd}(G^*)$ we also denote by ${\phi_\sigma}$. To an element $\tau^\natural=(M,\sigma,s) \in T^{2,\natural}_\tx{ell}(G)$ we assign the distribution
\[ f_G(\tau^\natural)=\tx{tr}(R_P(u^\natural,\Xi,\sigma,{\phi_\sigma},\psi_F)\circ\mc{I}_P(\sigma,f)), \]
where $u^\natural \in N_{\phi_\sigma}^\natural(M,G)$ is any preimage of $s$. According to Theorem \ref{thm:lir}, which we have now proved, the choice of $u^\natural$ doesn't matter. We set $T^\natural_\tx{ell}(G)=T^{2,\natural}_\tx{ell}(G) \sqcup \Pi_{2,\tx{temp}}(G)$. To an element $\sigma \in \Pi_{2,\tx{temp}}(G)$ we assign the distribution $f_G(\sigma)=\tx{tr}(\sigma,f)$.

For an element $\tau=(M,\sigma,r) \in T^2_\tx{ell}(G)$ we write $\phi_\tau=\phi_\sigma \in \Phi_{2,\tx{bdd}}(M^*)$ for the unique parameter with $\sigma \in \Pi_{\phi_\tau}(M,\Xi_M)$ as well as
\[ \<\tau,y\> = \<\sigma,y\>_{\Xi_M}, \qquad y \in S_{\phi_\tau}^{\natural\natural}(M) \]
where on the right the character $\<\sigma,-\>_{\Xi_M}$ is the one associated to the representation $\sigma$ by Theorem \ref{thm:locclass-single}.

There is an obvious surjective map $T^\natural_\tx{ell}(G) \rw T_\tx{ell}(G)$. If $\tau^\natural_1,\tau^\natural_2 \in T^{2,\natural}_\tx{ell}(G)$ belong to the same fiber of this map, then we have $\tau^\natural_1=(M,\sigma,s)$ and $\tau^\natural_2=(M,\sigma,ys)$ with $y \in S_{\phi_\tau}^{\natural\natural}(M)$. We will write $\tau_2^\natural=y\tau_1^\natural$ for short. Lemma \ref{lem:rpequiv} then asserts that
\begin{equation} \label{eq:tpwn} f_G(\tau^\natural_2) = \<\tau,y\>^{-1} \cdot f_G(\tau^\natural_1). \end{equation}
The trace Paley-Wiener theorem in this setting asserts that the map $f_G \mapsto f_G(\tau^\natural)$ is an isomorphism of $\C$-vector spaces between $\mc{I}_\tx{cusp}(G)$ and the space of complex valued functions on $T^\natural_\tx{ell}(G)$ of finite support (recall that the groups $S_{\phi_\tau}^{\natural\natural}(M)$ are finite for unitary groups) and satisfy the equivariance property \eqref{eq:tpwn}.

To a pair $(\phi, x)$ where $\phi \in \Phi_{\mathrm{ell}}(G^*)$ and $x \in S_{\phi, \mathrm{ell}}^\natural$,
we can associate the endoscopic pair $(\fke,\phi^\fke)$ by \S\ref{sub:endo-correspondence} and thus obtain the functional
\begin{eqnarray*}
	\mathcal{H}(G) &\mapsto& \C	\\
	f &\mapsto& f^\fke(\phi^\fke).
\end{eqnarray*}
We shall consider the restriction of this functional to the space $\mathcal{H}_{\mathrm{cusp}}(G)$. The trace Paley-Wiener theorem then implies that for $f \in \mathcal{H}_{\mathrm{cusp}}(G)$ the following identity holds,
\begin{equation} \label{eq:weakcharid}
	f'_{G, \Xi}(\phi,x)=f^\fke(\phi^\fke) = e(G)\sum_{\tau \in T_{\mathrm{ell}}(G)} c_{\phi, x}(\tau^\natural) f_G(\tau^\natural),
\end{equation}
where $\tau^\natural \in T^\natural_\tx{ell}(G)$ is any lift of $\tau$ and $c_{\phi,x}(\tau^\natural) \in \C$ are uniquely determined scalar coefficients having the equivariance property
\begin{equation} \label{eq:cpxeq1} c_{\phi,x}(y\tau^\natural) = \<\tau,y\> \cdot c_{\phi,x}(\tau^\natural),\qquad y \in S_{\phi_\tau}^{\natural\natural}(M), \tau^\natural \in T^{2,\natural}_\tx{ell}(G). \end{equation}
The coefficients $c_{\phi,x}(\tau^\natural)$ depend only upon $\phi$, $x$, and $\tau^\natural$. Furthermore, the dependence on $x$ has an equivariance property under translations by $Z(\hat G)^\Gamma$ parallel to that of $f'_{G, \Xi}(\phi,x)$, namely
\begin{equation} \label{eq:cpxeq2} c_{\phi,yx}=\<\Xi,y\>\cdot c_{\phi,x},\qquad y \in Z(\hat G)^\Gamma. \end{equation}

When $\phi \in \Phi_\tx{ell}^2(G^*)$, the local intertwining relation provides a formula for the coefficients $c_{\phi,x}(\tau^\natural)$. Indeed, let $M^*$ be the minimal standard Levi subgroup of $G^*$ through which $\phi$ factors. We have the $L$-packet $\Pi_\phi(M)$ and can form the subset of $T_\tx{ell}(G)$ given by
\[ T_{\phi,\tx{ell}}(G) = \{\tau=(M,\sigma,r)| \sigma \in \Pi_\phi(M),r \in {R_\sigma(M,G)}_\tx{reg} \} \]
as well as its preimage $T^\natural_{\phi,\tx{ell}}(G)$ in $T^{\natural}_\tx{ell}(G)$.
\begin{prop}
\label{prop_definition_coefficients_agrees_with_pairing_using_lir}
Let $\phi \in \Phi_\tx{ell}^2(G^*)$, $x \in S_{\phi,\tx{ell}}^\natural$, and let $\tau^\natural\in T^\natural_\tx{ell}(G)$. Then
\[
	c_{\phi, x}(\tau^\natural) = \begin{cases}
		\<\tau,xs\>, & \textrm{ if } \tau^\natural=(M,\sigma,s)  \in T^\natural_{\phi,\tx{ell}}(G) \\
		0 & \textrm{ else }
	\end{cases}
\]
\end{prop}
\begin{proof}
Given the equivariance property \eqref{eq:cpxeq1} it is enough to assume $x=s^{-1}$, in which case the statement follows from applying the local intertwining relation to the definition of $c_{\phi, x}$ coming from our chosen normalization.
\end{proof}

We can now state the orthogonality relations. Given $\tau = (M,\sigma,r) \in T_\tx{ell}(G)$ we define
\[ d(\tau) := \det(r-1)_{\mathfrak{a}_M^G},\quad R(\tau) = R_\sigma(M,G),\quad b(\tau) = \abs{d(\tau)} \cdot \abs{R(\tau)}. \]
Let $\tau^\natural \in T^\natural_\tx{ell}(G)$ be a lift of $\tau$. Then for any two pairs $(\phi_i,x_i)$, $i=1,2$, the product
\[ c_{\phi_1, x_1}(\tau^\natural) \overline{c_{\phi_2, x_2}(\tau^\natural)} \]
depends only on $\tau$ and not on the choice of lift $\tau^\natural$, as one sees from the equivariance property \eqref{eq:cpxeq1}. Furthermore, for a single pair $(\phi,x)$, the product
\[ c_{\phi,x}(\tau^\natural) \ol{c_{\phi,x}(\tau^\natural)} \]
depends on $x$ only through its image in $\ol{\cS_\phi}$, as one sees from the equivariance property \eqref{eq:cpxeq2}.

\begin{prop}
\label{prop_orthog_rel_1}
Suppose that $(\phi_i, x_i)$ for $i=1,2$ are two pairs where $\phi_i \in \Phi_{\mathrm{ell}}(G^*)$ and $x_i \in S_{\phi, \mathrm{ell}}^\natural$. If either $\phi_1 \neq \phi_2$, or $\phi_1=\phi_2=\phi$ but the images of $x_1$ and $x_2$ in $\ol{\cS}_{\phi}$ are not equal, we have
\[ \sum_{\tau \in T_{\mathrm{ell}}(G)} b(\tau) c_{\phi_1, x_1}(\tau^\natural) \overline{c_{\phi_2, x_2}(\tau^\natural)} = 0. \]
On the other hand,
\[ \sum_{\tau \in T_{\mathrm{ell}}(G)} b(\tau) c_{\phi, x}(\tau^\natural) \overline{c_{\phi, x}(\tau^\natural)} = \abs{\ol{\cS}_{\phi}}. \]
\end{prop}
\begin{proof}
This result is analogous to \cite[Cor 6.5.2]{Arthur} and \cite[Prop 6.5.1]{Arthur}, and is obtained in the same way. We will limit ourselves to a brief sketch. One begins by considering the elliptic inner product $I(f_G,g_G)$ on the space $\mc{I}_\tx{cusp}(G)$ defined in \cite[\S1]{ArtLCR} by
\[ I(f_G,g_G) = \int_{\Gamma_\tx{ell}(G)} f_G(\gamma)\ol{g_G(\gamma)}d\gamma, \]
where $\Gamma_\tx{ell}(G)$ is the space of regular semi-simple elliptic conjugacy classes in $G(F)$ and $d\gamma$ is a certain natural measure on this set. There are two spectral expressions for this scalar product. The first one is furnished by the local trace formula and takes the form
\[ I(f_G,g_G) = \sum_{\tau \in T_\tx{ell}(G)} |d(\tau)|^{-1}|R(\tau)|^{-1}f_G(\tau) \ol{g_G(\tau)}. \]
This is a specialization of \cite[Cor. 3.2]{ArtETC}, see also \cite[(6.5.6)]{Arthur}. This formula is valid with an arbitrary normalization of the distributions $f_G(\tau)$, since the effect of changing normalization cancels in the product $f_G(\tau)\ol{g_G(\tau)}$. In particular, we are free to use the normalization $f_G(\tau^\natural)$ that is encoded in the cover $T^\natural_\tx{ell}(G)$ of $T_\tx{ell}(G)$ introduced above and in the normalization of the intertwining operators introduced in Chapter \ref{chapter2}. Thus
\begin{equation} \label{eq:eips1} I(f_G,g_G) = \sum_{\tau \in T_\tx{ell}(G)} |d(\tau)|^{-1}|R(\tau)|^{-1}f_G(\tau^\natural) \ol{g_G(\tau^\natural)},  \end{equation}
where $\tau^\natural$ is an arbitrary lift of $\tau$, and the independence of the lift follows from the equivariance property \eqref{eq:cpxeq1}.

The second spectral expansion of $I(f_G,g_G)$ is of endoscopic nature. It takes the form
\begin{equation} \label{eq:eips2} I(f_G,g_G) = \sum_{\phi \in \Phi_\tx{ell}(G)} |\ol{\mc{S}_\phi}|^{-1} \sum_{x \in \ol{\mc{S}}_{\phi,\tx{ell}}} f'(s,\phi)\ol{g'(s,\phi)}, \end{equation}
where $s \in S_\phi$ is any lift of $x$, and the product $f'(s,\phi)\ol{g'(s,\phi)}$ is independent of the choice of $s$. To obtain this expression, we follow the argument of \cite[Prop 6.5.1]{Arthur}, and consider the stabilization identity \cite[(6.5.8)]{Arthur}
\[ I(f_G,g_G) = \sum_{\mf{e} \in \mc{E}_\tx{ell}(G)}\iota(G,G^\mf{e})S(f^\mf{e},g^\mf{e}). \]
This identity is proved for general reductive groups in \cite[Prop 3.5]{ArtLCR}. The terms $S(f^\mf{e},g^\mf{e})$ on the right hand side are the stable analogs of the elliptic inner product $I(f_G,g_G)$. They have a spectral expansion similar to that of $I(f,g)$, namely
\[ S(f^\mf{e},g^\mf{e}) = \sum_{\phi^\mf{e} \in \Phi_{2,\tx{bdd}}(G^\mf{e})} |\ol{\cS_{\phi^\mf{e}}}| f^\mf{e}(\phi^\mf{e}) \ol{g^\mf{e}(\phi^\mf{e})}. \]
This expansion, stated as \cite[(6.5.12)]{Arthur}, applies to the quasi-split groups $G^\mf{e}$ and is part of our assumptions stated in \ref{sub:results-qsuni}. Combining the last two identities leads to \eqref{eq:eips2}.

With the two identities \eqref{eq:eips1} and \eqref{eq:eips2} at hand, the statement follows from elementary linear algebra, as explained in the proof of \cite[Cor. 6.5.2]{Arthur}.
\end{proof}

\begin{prop}
\label{prop_orthog_rel_2}

\begin{enumerate}
\item[i)]
	Given a pair $(\phi, x)$ where  $\phi \in \Phi_{2,\tx{bdd}}(G^*)$ and $x \in S^{\natural}_{\phi, \mathrm{ell}}$, we have that
$c_{\phi, x}(\tau^\natural) = 0$ for all $\tau^\natural \in T^{2,\natural}_{\mathrm{ell}}(G)$.

\item[ii)] Suppose that
\[
	y_i = (\phi_i, x_i), \ i=1,2
\]
are two pairs where $\phi_i \in \Phi_{2,\tx{bdd}}(G^*)$ and $x_i \in S_{\phi, \mathrm{ell}}^\natural$.
Then the expression
\[
	\sum_{\pi\in\Pi_{2,\tx{temp}}(G)} c_{\phi_1, x_1}(\pi) \overline{c_{\phi_2, x_2}(\pi)}
\]
is equal to zero, unless $\phi_1 = \phi_2$ and the images of $x_1$ and $x_2$ in $\ol{\cS}_{\phi_1}$ are equal. Moreover, \[ \sum_{\pi\in\Pi_{2,\tx{temp}}(G)} c_{\phi, x}(\pi) \overline{c_{\phi, x}(\pi)}=|\ol{\cS_\phi}|. \]
\end{enumerate}

\end{prop}
\begin{proof}
Part (ii) follows at once from part (i) and Proposition \ref{prop_orthog_rel_1} whilst noting that $b(\tau) = 1$ for $\tau = \pi \in \Pi_{2,\tx{temp}}(G)$. It is thus enough to show part (i).

Let  $\tau_1=(M_1,\sigma_1,r_1) \in T^2_{\mathrm{ell}}(G)$ and let $\tau_1^\natural=(M_1,\sigma_1,s_1)$ be a lift. Then $\phi_1 := \phi_{\tau_1} \in \Phi^2_{\mathrm{bdd}}(G)$. The isomorphism of $R$-groups $R_{\phi_1}(M,G) = R(\sigma_1)$ provided by Lemma \ref{lem_tempered_generic_repns_r_groups_isomorphic} implies that $R_{\phi_1, \mathrm{reg}} \neq \emptyset$, that is $\phi_1 \in \Phi^2_{\mathrm{ell}}(G)$.

For any $\tau^\natural=(M,\sigma,s) \in T^\natural_\tx{ell}(G)$, the formula for $c_{\phi_1, x_1}(\tau^\natural)$
appearing in Proposition \ref{prop_definition_coefficients_agrees_with_pairing_using_lir} and the fact that the subset of elliptic elements $S^{\natural}_{\phi_1, \mathrm{ell}} \subset S^{\natural}_{\phi_1}$ is a $S^{\natural\natural}_{\phi_1}(M)$-torsor imply
\[
	\frac{1}{\abs{S_{\phi_1,\tx{ell}}^{\natural}}} \sum_{x_1 \in S_{\phi_1,\tx{ell}}^{\natural}} \<\tau_1,x_1 s_1\> \cdot  \ol{c_{\phi_1,x_1}(\tau^\natural)}
		= \begin{cases}
			\<\tau,s_1 s^{-1}\> & \textrm{ if } \tau = \tau_1,	\\
			0 & \textrm{otherwise.}
		\end{cases}
\]
Hence, for any $\phi \in \Phi_{2,\tx{bdd}}(G^*)$ and  $x \in S^{\natural}_{\phi_1, \mathrm{ell}}$, we have
\[ b(\tau_1)c_{\phi,x}(\tau_1^\natural) = \sum_{\tau \in T_\tx{ell}(G)} b(\tau)c_{\phi,x}(\tau^\natural)\frac{1}{\abs{S_{\phi_1,\tx{ell}}^{\natural}}} \sum_{x_1 \in S_{\phi_1,\tx{ell}}^{\natural}} \<\tau_1,x_1 s_1\> \cdot  \ol{c_{\phi_1,x_1}(\tau^\natural)}. \]
Here $\tau^\natural$ is an arbitrary lift of $\tau$, and the summand is independent of the choice of lift. Interchanging the two sums and applying Proposition \ref{prop_orthog_rel_1} we see that the entire expression is equal to zero, since $\phi_1 \in \Phi^2_\tx{ell}(G^*)$ and $\phi \in \Phi_{2,\tx{bdd}}(G^*)$ preclude $\phi_1 = \phi$. But $b(\tau_1) \neq 0$, forcing $c_{\phi,x}(\tau_1^\natural)=0$.
\end{proof}

\subsection{Local packets for square-integrable parameters}
Let $E/F$ be a quadratic extension of non-archimedean local fields and let $G^*=U_{E/F}(N)$. Let $\Xi : G^* \rw G$ be an equivalence class of extended pure inner twists. In this section we are going to complete the proof of the local classification theorem \ref{thm:locclass-single} in the generic case by treating generic discrete parameters $\phi \in \Phi_{\tx{bdd},2}(G^*)$. Unlike the treatment of non-discrete parameters in Section \ref{sec:lpackns}, which was purely local, the treatment of discrete parameters here will have to be global.

We first make the following reduction. According to Proposition \ref{p:local-thm-other-ext-pure-inner-tw}, it is enough to prove the local classification theorem for any equivalence class $\Xi' : G^* \rw G$ of extended pure inner twists which gives rise to the same equivalence class of inner twists as $\Xi$. If $N$ is odd, then $\Xi$ necessarily gives rise to the trivial equivalence class of inner twists and we may thus take $\Xi'$ to be the trivial equivalence class of extended pure inner twists. The local classification theorem then follows by assumption from the quasi-split case.

The case that $N$ is even is our main concern here. We use Proposition  \ref{prop_globalize_data_N_even_M_eq_G} to obtain the datum $(\dot{E}/\dot{F}, u, v_1,v_2,  \dot{G}^*, \dot{\eta},  \dot{\Xi}: \dot{G}^* \rightarrow \dot{G}, \dot{\phi})$. We will prove the local classification theorem for the equivalence class of pure inner twists $\dot\Xi_u : \dot G^*_u \rw \dot G_u$.

\begin{lemm}
\label{lemm_loc_sq_integrable_packets_global_reln} For every $\dot f \in \mc{H}(\dot G)$ the following equality holds
\[
\sum_{\dot{\pi}} n_{\dot\phi}(\dot{\pi}) \dot{f}_{\dot{G}}(\dot{\pi})
	= \frac{1}{\abs{\ol{\cS_{\dot{\phi}}}}} \sum_{\overline{x} \in \ol{\cS}_{\dot{\phi}}} \dot{f}^\fke(\dot{\phi}, \overline{x}),
\]
where $\dot{\pi}$ runs over the irreducible representations of ${\dot{G}}(\dot{\A})$ and
$n_{\dot\phi}(\dot{\pi})$ denotes the multiplicity of $\dot{\pi}$ in $R^{\dot{G}}_{\mathrm{disc}, \dot{\phi}}$.
\end{lemm}
\begin{proof}
The globalization propositions ensure that $\dot\phi \in \Phi_2(\dot G,\dot\eta)$. We claim that $\dot\phi$ cannot contribute to the discrete spectrum of any proper Levi subgroup $\dot M$ of $\dot G$. Indeed, if $\dot\phi=(\dot\phi^N,\tilde{\dot\phi})$ did contribute then it should come from $\dot\phi_{\dot M}\in\Phi_2(\dot M,\dot\eta)$ by the induction hypothesis applied to $\dot M$. Consider the decomposition into simple parameters $$\dot\phi^N=\boxplus_{i=1}^r \ell_i \dot\phi_i^{N_i}.$$ Write $\dot\phi_{\dot M}=\dot\phi_{\dot M_-}\times \dot\phi_{\dot M_+}$ according to the decomposition $\dot M = \dot M_- \times \dot M_+$ into linear and hermitian parts. A simple factor of $\dot\phi_{\dot M_-}$ will contribute a factor with even multiplicity if it is conjugate self-dual, and a factor which is not conjugate self-dual otherwise. However $\dot\phi$ is a discrete parameter so all $\dot\phi_i^{N_i}$ are conjugate self-dual with multiplicity one by Lemma \ref{lem:global-disc-param}. The resulting contradiction proves the claim.

  Equation \eqref{eq:I_disc-psi-spec} thus reduces to
\[	I^{\dot{G}}_{\mathrm{disc}, \dot{\phi}}(\dot{f}) = \tr R^{\dot{G}}_{\mathrm{disc}, \dot{\phi}}(\dot{f})
		= \sum_{\dot{\pi}} n_{\dot\phi}(\dot{\pi}) \dot{f}_{\dot{G}}(\dot{\pi}).
\]
On the other hand, Proposition \ref{p:std-model-end} gives the expansion
\[
I^{\dot{G}}_{\mathrm{disc}, \dot{\phi}}(\dot{f}) = \frac{1}{\abs{\ol{\cS_{\dot{\phi}}}}} \sum_{\overline{x} \in \ol{\cS}_{\dot{\phi}}}
	 \dot{f}_{\dot G,\dot \Xi}(\dot{\phi}, \overline{x})
\]
where we have used the fact that $s_{\dot{\phi}} = 1$ and $\epsilon_{\dot{\phi}}(\overline{x}) = 1$ as $\dot{\phi}$ is generic and that
 $e_{\dot{\phi}}(\overline{x}) = 1$ as $\dot{\phi}$ is square-integrable. The result follows.
\end{proof}

We shall now apply Lemma \ref{lemm_loc_sq_integrable_packets_global_reln} with a decomposable function $\dot f = \otimes_v \dot f_v$ whose $u$-component $\dot f_u$ belongs to $\mathcal{H}_{\mathrm{cusp}}(G)$. To interpret the right-hand side of the equation of the lemma, we use the character identities, which are part 4 of Theorem \ref{thm:locclass-single}, i.e.
\begin{equation}
\label{eq:expand-part4-thm1.6.1}
	\dot{f}'_{\dot G_v,\dot\Xi_v}(\dot{\phi}_v, \dot{x}_v) = e(\dot{G}_v)
	\sum_{\dot{\pi}_v \in \Pi_{\dot{\phi}_v}} \langle \dot{\pi}_v, \dot{x}_v \rangle_{\dot\Xi_v}
	 f_{v, \dot{G}_v}(\dot{\pi}_v),\qquad \dot{x} \in S^{\natural}_{\dot{\phi}}.
\end{equation}
These character identities are known for $v \notin \{u, v_2\}$, because then the group $\dot{G}_v$ is quasi-split and these identities are part of our list of assumptions in Section \ref{sub:results-qsuni}. At the place $v=v_2$ the group $\dot{G}_v$ is an archimedean unitary group.
The character identities are thus known by the work of Shelstad. We refer the reader to \cite[\S5.6]{Kal13} for an exposition in the language of pure inner forms.

At the place $v=u$ we do not have the character identities yet. Instead, we have the following weaker statement, coming from \eqref{eq:weakcharid} together with Proposition \ref{prop_orthog_rel_2}, part i,
\begin{eqnarray}
\label{eqn_loc_sq_int_packets_char_expr_prelim}
	\dot{f}'_{\dot G_u,\dot\Xi_u}(\dot{\phi}_u, {\dot{x}_u}) =  e(\dot G_u)\sum_{\pi \in \Pi_{2,\tx{temp}}(G)}  c_{\phi, \dot{x}_u}(\pi) \dot{f}_{u, G}(\pi),\qquad \dot{x} \in S^{\natural}_{\dot{\phi}}.
\end{eqnarray}

Lemma \ref{lemm_loc_sq_integrable_packets_global_reln} thus leads to the equation
\begin{equation}
\label{equation_used_sq_int_generic_1}
\sum_{\dot{\pi}} n_{\dot\phi}(\dot{\pi}) \dot{f}_{\dot{G}}(\dot{\pi})
	= \frac{1}{\abs{\ol{\cS_{\dot{\phi}}}}}  \sum_{\dot{\pi}^u} \sum_{\pi \in \Pi_{2,\tx{temp}}(G)}
		 \sum_{\overline{x} \in \ol{\cS}_{\dot{\phi}}} c_{\phi, \dot{x}_u}(\pi) \langle \dot{\pi}^u, \dot{x}^u  \rangle_{\dot\Xi^u}
		\dot{f}_{\dot{G}}(\pi \otimes \dot{\pi}^u).
\end{equation}
Here on the left $\dot\pi$ runs over the irreducible unitary representations of $\dot G(\A)$, while on the right $\dot\pi^u$ runs over the elements in the packet
\begin{eqnarray*}
	\Pi_{\dot{\phi}^u} &=& \otimes_{v \neq u} \Pi_{\dot{\phi}_v}	\\
			&=& \left\{
				\dot{\pi} = \otimes_{v \neq u} \dot{\pi}_v  :  \dot{\pi}_v \in \Pi_{\dot{\phi}_v}, \langle \dot{\pi}_v, \cdot \rangle_{\dot\Xi_v} = 1
					\textrm{ for almost all $v$}
			\right\}.
\end{eqnarray*}
As before, $\dot{x} \in S^{\natural}_{\dot{\phi}}$ denotes a lift of $\overline{x} \in \ol{\cS}_{\dot{\phi}}$. The pairing $\langle \dot{\pi}^u, \dot{x}^u  \rangle_{\dot{\Xi}^u}$ is given by the product over $v \neq u$ of the local pairings $\<\dot\pi_v,\dot x_v\>_{\dot\Xi_v}$.

Recall that our particular globalization satisfies
\[ S^{\natural}_{\phi} = S^{\natural}_{\dot{\phi}_u} = S^{\natural}_{\dot{\phi}_{v_1}} = S^{\natural}_{\dot{\phi}}. \]
as well as $\<\dot\Xi_v,-\>=1$ %
 for $v \notin \{u,v_2\}$ and $\<\dot\Xi_{v_2},-\>=\<\dot\Xi_u,-\>^{-1}$, the latter being an equality of characters of
\[ Z(\hat{\dot G})^{\Gamma_{v_2}}=Z(\hat{\dot G})^\Gamma=Z(\hat{\dot G})^{\Gamma_u}. \]
This allows us to fix the representation $\dot{\pi}^{u,v_1,v_2} \in \Pi_{\dot{\phi}^{u,v_1,v_2}}$ given by the restricted tensor product $\otimes'_{v\neq u,v_1,v_2} \pi_v$, where $\pi_v \in\Pi_{\dot\phi_v}$ is the unique element with $\<\pi_v,-\>_{\dot\Xi_v}=1$. Fix furthermore $\dot\pi_{v_2} \in \Pi_{\dot\phi_{v_2}}$ arbitrarily. Then the character $\<\dot\pi^{u,v_1},-\>_{\dot\Xi^{u,v_1}}$ restricts to $\<\dot\Xi_{v_2},-\>$ on $Z(\hat{\dot G})^\Gamma$.

Let $\mu \in X^*(S^{\natural}_{\phi})$ be a character restricting to $\<\dot\Xi_u,-\>$ on $Z(\hat{\dot G})^{\Gamma}$. Then $\mu\cdot\<\dot\pi^{u,v_1},-\>_{\dot\Xi^{u,v_1}}$ restricts trivially to $Z(\hat{\dot G})^\Gamma$ and we may choose $\dot{\pi}_{v_1} \in \Pi_{\dot{\pi}_{v_1}}$ to be the element for which $\<\dot\pi^{u},-\>_{\dot\Xi^u}=\mu^{-1}$ holds. We then define for all $\pi \in \Pi_{2,\tx{temp}}(G)$,
\[
	n_{\phi}(\mu, \pi) := n_{\dot\phi}(\pi \otimes \dot{\pi}^u)
\]
which is a non-negative integer

\begin{prop}
\label{prop_exp_n_phi_xi_pi}
We have the equality
\[
	n_\phi(\mu, \pi) = \frac{1}{\abs{\ol{\cS_{{\phi}}}}}  \sum_{x} c_{\phi, x}(\pi) \mu(x)^{-1}
\]
where $x \in S^{\natural}_\phi$ runs through a set of representatives for the quotient $\ol{\cS}_\phi$.
In particular the non-negative integer $n_\phi(\mu, \pi)$ depends only upon $\mu$, $\phi$ and $\pi$.
\end{prop}
Note that each summand depends only on the image of $x$ in $\ol{\cS}_\phi$, as its two factors have a cancelling equivarance behavior under translations by $Z(\hat G)^\Gamma$. Hence the choice of set of representatives is irrelevant.
\begin{proof}
Applying our choice of $\dot{\pi}^u$ to Equation \eqref{equation_used_sq_int_generic_1} and using linear independence of
characters of $\dot{G}(\A^u)$, we see that for each $f \in \mathcal{H}_{\mathrm{cusp}}(G)$ we have the equality %
\[
	\sum_{{\pi} \in \Pi_{2,\tx{temp}}(\dot{G}_u)} n_{\phi}(\mu, \pi)  \dot{f}_{\dot{G}_u}({\pi})
	= \frac{1}{\abs{\ol{\cS_{{\phi}}}}}  \sum_{\pi \in \Pi_{2,\tx{temp}}(G)}
		 \sum_{\overline{x} \in \ol{\cS}_{{\phi}}} c_{\phi, \dot{x}_u}(\pi) \mu(\dot{x})^{-1}
		\dot{f}_{\dot{G}_u}(\pi).
\]
The result follows by using the linear independence of the representations in $\Pi_{2,\tx{temp}}(G)$ as characters of
$\mathcal{H}_{\mathrm{cusp}}(G)$.
\end{proof}

\begin{prop}
\label{prop_sq_int_local_l_packet_generic_pi_exists_associated_to_char}
For all discrete parameters $\phi, \phi' \in \Phi_{2,\tx{bdd}}(G)$ and all characters
$\mu \in X^*(S^{\natural}_{\phi})$ and $\mu' \in X^*(S^{\natural}_{\phi'})$ whose restriction to $Z(\hat G)^\Gamma$ is equal to $\<\dot\Xi_u,-\>$, we have that
\[
	\sum_{\pi \in \Pi_{2,\tx{temp}}(G)} n_{\phi}(\mu, \pi) n_{\phi'}(\mu', \pi) = \begin{cases}
		1 & \textrm{ if } (\phi, \mu) = (\phi', \mu'),	\\
		0 & \textrm{ otherwise.}	\\
	\end{cases}
\]
\end{prop}
\begin{proof}
By Proposition \ref{prop_exp_n_phi_xi_pi}, we obtain that
\begin{eqnarray*}
&&\sum_{\pi \in \Pi_{2,\tx{temp}}(G)} n_{\phi}(\mu, \pi) n_{\phi'}(\mu', \pi)\\
&=&\sum_{\pi \in \Pi_{2,\tx{temp}}(G)} n_{\phi}(\mu, \pi) \overline{n_{\phi'}(\mu', \pi)} 	\\
&=& \frac{1}{\abs{\ol{\cS_{{\phi}}}}\abs{\ol{\cS_{{\phi'}}}}}
		 \sum_{x} \sum_{x'}  \mu(x)^{-1} \mu'(x')
		\sum_{\pi \in \Pi_{2,\tx{temp}}(G)}
		c_{\phi, x}(\pi) \overline{c_{\phi', x'}(\pi)}
\end{eqnarray*}
According to Proposition \ref{prop_orthog_rel_2}, part ii, the inner sum is zero unless $\phi = \phi'$ and
the image of $x$ and $x'$ in $\ol{\cS_{\phi}}$ is equal. Thus the double sum over $x,x'$ collapses to a single sum over $x$ and another look at Proposition \ref{prop_orthog_rel_2} shows that the above expression is equal to
\[ \frac{1}{\abs{\ol{\cS_{{\phi}}}}} \sum_{x} \mu(x)^{-1} \mu'(x). \]
The sum over $x$ still runs over a set of representatives in $S_\phi^\natural$ for the quotient $\ol{\cS_\phi}$. However, the summand now descends to this quotient and the result follows.
\end{proof}

Proposition \ref{prop_sq_int_local_l_packet_generic_pi_exists_associated_to_char} and the fact that
$n_{\phi}(\mu, \pi)$ is a non-negative integer imply that for every $\mu \in X^*(S^{\natural}_{\phi})$
lying above $\<\dot\Xi_u,-\>$ there exists a unique $\pi = \pi(\mu) \in \Pi_{2,\tx{temp}}(G)$ such that $n_\phi(\mu, \pi) \neq 0$, in which case in fact $n_\phi(\mu, \pi) = 1$. Moreover, the assignment
\[
	\mu \mapsto \pi(\mu)
\]
is injective.
We define the packet
\[
	\Pi_\phi = \left\{
		\pi(\mu) : \mu \in X^*(S^{\natural}_{\phi}), \mu|_{Z(\hat G)^\Gamma}=\<\dot\Xi_u,-\>\right\},
\]
as well as the character
\[
	\langle \pi(\mu), \cdot \rangle_{\dot\Xi_u} := \mu(\cdot).
\]
Proposition \ref{prop_sq_int_local_l_packet_generic_pi_exists_associated_to_char} implies that these packets are disjoint from each other. Moreover, the map $\pi \mapsto \<\pi,-\>_{\dot\Xi_u}$ is by construction a bijection from $\Pi_\phi$ to the set of characters of $S_\phi^\natural$ lying above $\<\dot\Xi_u,-\>$. The following proposition shows that the desired character identity holds.
\begin{prop}
For $x \in S_\phi^\natural$ and $f \in \mathcal{H}(G)$ the following equality holds
\[
	f^\fke(\phi, x) = e(G)\sum_{\pi \in \Pi_\phi} \langle \pi, x \rangle_{\dot\Xi_u} f_G(\pi).
\]
\end{prop}
\begin{proof}
Firstly consider the case where $f \in \mathcal{H}_{\mathrm{cusp}}(G)$. Propositions \ref{prop_exp_n_phi_xi_pi} and \ref{prop_sq_int_local_l_packet_generic_pi_exists_associated_to_char} imply that
\[
	\frac{1}{\abs{\ol{\cS_\phi}}} \sum_{{x} \in \ol{\cS_\phi}} c_{\phi, x}(\pi) \mu(x)^{-1} = n_\phi(\mu,\pi) =
		\begin{cases}
			1 & \textrm{ if } \pi = \pi(\mu),	\\
			0 & \textrm{otherwise}.
		\end{cases}
\]
Inverting this equation, we have that
\[
	c_{\phi, x}(\pi) =
		\begin{cases}
			\mu(x) = \langle \pi, x\rangle_{\dot{\Xi}_u} & \textrm{ if } \pi = \pi(\mu),	\\
				0 & \textrm{otherwise}.
		\end{cases}
\]
Applying this to Equation \eqref{eqn_loc_sq_int_packets_char_expr_prelim}, we obtain the desired result.

It remains to show that the character identity holds for non-cuspidal $f \in \mathcal{H}(G)$. This can be done by arguing as in the proof of \cite[Cor 6.7.4]{Arthur}.

\end{proof}

 The proposition implies that the packet $\Pi_\phi$ depends only on $\phi$ and not on the globalization $\dot\phi$.
 It follows from our construction that these packets are disjoint.
  To show that they exhaust the set $\Pi_{2,\tx{temp}}(G)$, it is enough to show that for each $\pi \in \Pi_{2,\tx{temp}}(G)$ there exist $\phi \in \Phi_{2,\tx{bdd}}(G)$ and $\mu \in S_\phi^\natural$ lying above $\<\Xi,-\>$ such that $n_\phi(\mu,\pi) \neq 0$. For this it suffices to find $\dot\pi$ and $\dot\phi$ such that $\dot\pi_u = \pi$, $\dot\pi^u \in \Pi_{\dot\phi^u}$ and $n_{\dot\phi}(\dot \pi) \neq 0$. This is further reduced to
just finding $\dot\pi$ with $\dot\pi_u=\pi$, as $\dot\phi$ is then determined by weak base change of $\dot\pi$, an argument implicit in the display below Corollary \ref{cor:outside-Psi(G)}, and then descend via Proposition \ref{p:stable-multiplicity}. But the existence of $\dot\pi$ with $\dot\pi_u=\pi$ is immediate from Lemma \ref{lemma_globalization_of_the_representation_1}.

This completes the proof of Theorem \ref{thm:locclass-single} for discrete generic parameters, which was also an inductive assumption in the proof of this theorem for general generic parameters. In other words, Theorem \ref{thm:locclass-single} has now been established for all generic parameters and all pure inner twists of unitary groups.

\section{Proof of the main global theorem}\label{chapter5}

  This last chapter is devoted to the proof of Theorem \ref{thm:main-global}.   It turns out that the argument is much simpler than the analogue in the quasi-split case. For one thing the twisted trace formula for general linear groups plays no more direct roles.
  
   We adopt the global notation so that $G^*=U_{E/F}(N)$ is a global quasi-split unitary group. Fix $\kappa\in \{\pm1\}$ as well as $\chi=\chi_\kappa\in \cZ_E^\kappa$. Let $(G,\xi)$ be an inner twist of $G^*$. 
  Recall from \S\ref{sub:stable-multiplicity} that
$L^2_\disc(G(F)\bs G(\A_F))=\oplus_{\psi} L^2_{\disc,\psi}(G(F)\bs G(\A_F))$ as $\psi$ runs over $\Psi(G^*,\eta_\chi)$ and that there is a corresponding decomposition $\tr R_{\disc}(f)=\sum_{\psi} \tr R_{\disc,\psi}(f)$ for $f\in \cH(G)$. For the proof it suffices to identify $L^2_{\disc,\psi}(G(F)\bs G(\A_F))$ for each $\psi\in\Psi(G^*,\eta_\chi)$. 
  We assume the following on $\psi$:
 \begin{itemize}
   \item Hypothesis \ref{hypo:Hyp(psi)} (i.e. the local classification theorem holds for $\psi_v$ at every $v$) and
   \item Theorem \ref{thm:lir} (the local intertwining relation) for $\psi_v$ at every place $v$.
 \end{itemize}
 These have been established if $\psi=\phi$ is generic and if $(G,\xi)$ is realized as a pure inner twist of $G^*$. (The latter condition implies that $G_v$ is a split group at every place $v$ of $F$ split in $E$ since general linear groups do not have any nontrivial pure inner twists.) In the remaining cases the above assumptions will be resolved in \cite{KMS_A} and \cite{KMS_B}.

\begin{thm}\label{thm:global-conditional} Let $\psi\in\Psi(G^*,\eta_\chi)$. Under the two assumptions above, 
\begin{enumerate}
  \item $L^2_{\disc,\psi}(G(F)\bs G(\A_F))=0~$ if $\psi\notin\Psi_2(G^*,\eta_\chi)$.
  \item $L^2_{\disc,\psi}(G(F)\bs G(\A_F))=\bigoplus_{\pi\in \Pi_\psi(G,\xi,\epsilon_\psi)} \pi~$ if $\psi\in \Psi_2(G^*,\eta_\chi)$.
\end{enumerate}
\end{thm}

\begin{proof}

If $\psi\notin \Psi_2(G^*,\eta_\chi)$ then Corollary \ref{cor:std-model-result} and the local intertwining relation imply (via Theorem \ref{thm:global-intertwining}) that $\tr R_{\disc,\psi}(f)=0$ identically, so $\psi$ does not contribute to the discrete spectrum. It remains to check that for $\psi\in \Psi_2(G^*,\eta_\chi)$,
$$\tr R_{\disc,\psi}(f) = \sum_{\pi\in \Pi_\psi(G,\xi,\epsilon_\psi)}  f(\pi),\qquad f\in \cH(G).$$
As usual our notation is that $f(\pi)=\tr \pi(f)$. Fourier transform on the finite group $\ol{\cS}_\psi$ allows us to rewrite the right hand side as
$$\frac{1}{|\ol{\cS}_\psi|} \sum_{\pi\in \Pi_\psi(G,\xi)} \sum_{x\in \ol{\cS}_\psi} \epsilon^{G^*}_\psi(x)\lg x,\pi\rg_{G,\xi}f(\pi)
=\frac{1}{|\ol{\cS}_\psi|}\sum_{x\in \ol{\cS}_\psi} \epsilon^{G^*}_\psi(x)f_G(\psi,x).$$
In view of the definition \eqref{eq:0rG} it is enough to show that ${}^0 r^G_{\disc,\psi}(f)=0$. This results from Corollary \ref{cor:std-model-result} since the local intertwining relation and the local classification theorem imply that $f_G(\psi,x)=f'_G(\psi,s_\psi x^{-1})$.

\end{proof}

\appendix
\section{The Aubert involution and $R$-groups}\label{sec:appendix}
In \cite{Ban, Ban2}, Ban shows that the Knapp-Stein $R$-group of a connected reductive group $G$ is invariant under the Aubert involution $\pi \mapsto \widehat{\pi}$. More precisely she proves:
\begin{thm}
Let $M$ be any Levi subgroup of $G$ and let $\pi$ be a square integrable irreducible representation  of  $M$. Denote by $R$  the Knapp-Stein $R$-group for $\pi$. Then the set of normalized self-intertwining operators
\[\{R_P(r,\widehat{\pi}), r \in R\}
\]
is a basis for the commuting algebra $\End_G\left(i_M^G(\widehat{\pi})\right)$.
\end{thm}
In \cite{Ban, Ban2} the following three hypothesis are needed:
\begin{enumerate}
\item $G$ is a split group.
\item One can normalize the self-intertwining operators $R_P(w,\pi)$ to be multiplicative on $w$.
\item The Aubert involution of a (square integrable) unitary representation of $M$ is still unitary.
\end{enumerate}
In this appendix we get rid of the two first hypothesis. However Hypothesis (3) will be still needed to grant that the normalized intertwining operator $R_{P',P}(\widehat{\pi}_\chi)$ is well defined at $\chi=1$ (see below for more details on notation).

For our purposes, Hypothesis (3) will not be a problem since all representations we will care about appear as local components of automorphic representations and are thus unitary. By induction hypothesis and Theorem \ref{thm:locclass-single}(1) they will be automatically unitary.
\subsection{Notation}
Let $F$ ba a local non-archimedean field. Denote by $| \, |_F$ the normalized absolute value of $F$. Let $G$ be any connected reductive group defined over $F$. Write $A_G$ for the maximal split torus in the center of $G$. By a representation of $G(F)$ we will always mean a smooth representation.

Define $X(G)=\Hom(G,\mathbb{G}_m)$ and $\mathfrak{a}_G=\Hom(X(G), \R)$. The dual of this space is denoted by $\mathfrak{a}_G^\ast$ and we set $\mathfrak{a}^\ast_{G,\C}=\mathfrak{a}^\ast_G\otimes \C$. For every Levi subgroup $M$ there is a canonically split short exact sequence:
\[0 \to \mathfrak{a}_G \to \mathfrak{a}_M \to \mathfrak{a}_M^G \to 0.
\]

Let $\Psi(G)$ be the complex torus of unramified characters of $G$. It is endowed with a structure of complex algebraic variety coming from the surjection
\begin{eqnarray}\label{ap:HC}
\mathfrak{a}^\ast_{G,\C} & \longrightarrow& \Psi(G)\\
\chi \otimes s &\mapsto &(g \mapsto |\chi(g)|_F^s). \notag
\end{eqnarray}

For any $\chi \in\Psi(G)$  and any representation $\pi$ of $G(F)$ we will write $\pi_\chi$ for $\pi \otimes \chi$.

Fix a minimal parabolic subgroup $P_0=M_0U_0$ of $G$. We set $\Delta_0$ to be the set of simple roots. For any $\Theta \subset \Delta_0$, let $P_\Theta$ be the associated standard parabolic.

Denote by $\mathcal{L}^G$ the set of standard Levi subgroups of $G$. For any Levi subgroup $M$ in $\mathcal{L}^G$, denote by $\mathcal{P}(M)$ the set of parabolic subgroups of $G$ with Levi component $M$ and by $W(M)$ the Weyl group relative to $M$. For any $w \in W(M)$ with representative $\tilde{w} \in G(F)$ and any representation $\pi$ of $M(F)$, we define $\tilde{w}\pi$ to be the representation on the same underlying vector space as $\pi$ with action $(\tilde{w}\pi)(m)=\pi(\tilde{w}^{-1} m \tilde{w})$ (observe that the equivalence class depends only on $w$ and if no confusion arises we will write $w\pi$ for $\tilde{w}\pi$.) We also define $$W_\pi=\{w\in W(M)\,:\,w\pi\simeq \pi\}.$$

For any Levi subgroup $M$ in $\mathcal{L}^G$, and any parabolic $P$ containing $M$, denote by $i_P^G$ the normalized parabolic induction (so that $i_P^G(\pi)$ is unitary if $\pi$ is) and
 by $r_P^G$ its left adjoint functor, the Jacquet functor. If $P$ is the standard parabolic subgroup of $G$ containing $M$ we will simply write $i_M^G$ and $r_M^G$. For any parabolic subgroup $P$ denote by $P^-$ the parabolic subgroup opposite to $P$.

Write $\mc{K}(G)$ for the Grothendieck group of finite length representations of $G(F)$, that is, the set of finite integral linear combinations of isomorphism classes of irreducible representations of $G(F)$. For any smooth representation  $\pi$ of $G(F)$ write $\pi^\vee$ for its contragredient and, if we suppose moreover that $\pi$ is of finite length, denote by $[\pi]$ its image in $\mc{K}(G)$.

We also use the notation $\tx{Alg}(G)$ for the category of smooth representations of $G$, $\Pi(G)$ for the set of isomorphism classes of irreducible representations of $G$ and $\Pi_2(G)$ (resp. $\Pi_{\tx{unit}}(G)$) for its subset consisting of square integrable (resp. unitary) representations.

\subsection{The Aubert involution}
\subsubsection{The Aubert involution at the level of Grothendieck groups}
The functors $i_M^G$ and $r_M^G$ induce, by linearity, homomorphisms between Gro\-thendieck groups that, by a little abuse of notation, we will still write $i_M^G:\mc{K}(M) \to \mc{K}(G)$ and $r_M^G:\mc{K}(G) \to \mc{K}(M)$ respectively.  The Aubert involution \cite{Aubert,Aubert2} is the homomorphism of $\Z$-modules
$\tb{D}_G: \mc{K}(G) \to \mc{K}(G)$ defined by
$$\tb{D}_G= \underset{M \in \mathcal{L}^G}{\sum}(-1)^{\dim(A_{M_0}/A_M)}i_M^G\circ r_M^G.$$

A key property of the Aubert involution is that it preserves irreducibility up a sign. Namely for any irreducible representation $\pi$ of $G(F)$, there exists $\beta(\pi)\in \{\pm 1\}$ such that $\beta(\pi)\tb{D}_G([\pi])$ is represented by an irreducible representation of $G(F)$, to be denoted by $[\widehat{\pi}]$.

\subsubsection{The Aubert involution on the level of representations}\label{ap:22}
Let $\pi$ be a smooth representation of $G$ and let $P \subset P' \subset G$ be two standard parabolic subgroups of $G$. Denote by $\pi_{P,P'}$ the canonical projection $r_P^G(\pi) \to r_{P'}^G(\pi) $ and define the map
\begin{eqnarray}\label{eq:defd}
d_{P,P'}: i_{P}^G\circ r_{P}^G(\pi) & \longrightarrow & i_{P'}^G\circ r_{P'}^G(\pi)  \notag\\
f & \mapsto & f_{P,P'}.
\end{eqnarray}
where  $f_{P,P'}(g) = \pi_{P,P'}f(g)$.

Let $\Theta_0 \subset \Delta_0$. We denote by $\tx{Alg}(\Theta_0)$ the full abelian subcategory of $\tx{Alg}(G)$ consisting of $\pi$ such that every subquotient of $\pi$ is a subquotient of a representation of the form $i_{P}^G(\rho)$ where $P$ is a standard parabolic subgroup with Levi component conjugate to $M_{\Theta_0 }$, and $\rho$ is a cuspidal representation of the Levi quotient of $P(F)$.

Aubert defined a map
\begin{eqnarray*}
\tx{Alg}(\Theta_0) &\longrightarrow &\tx{Alg}(\Theta_0)\\
\pi & \mapsto& \widehat\pi
\end{eqnarray*}
in the following way. She proved in \cite[Th\'eor\`eme 3.6]{Aubert} that the following sequence of representations of $G(F)$ is exact:
\begin{align}\label{eq:seqAubert}
0 \to \pi \overset{d_{|\Delta_0|}}{\longrightarrow }\underset{\begin{array}{c} |\Theta|=|\Delta_0|-1\\ \Theta_0 \subset \Theta \end{array}}{\bigoplus}i_{P_\Theta}^G\circ r_{P_\Theta}^G(\pi) \overset{d_{|\Delta_0|-1}}{\longrightarrow } \dots \to \underset{\begin{array}{c} |\Theta|=|\Theta_0| \\  \Theta_0 \subset \Theta \end{array} }{\bigoplus}i_{P_\Theta}^G\circ r_{P_\Theta}^G(\pi),
\end{align}
where $$d_i= \underset{\begin{array}{c} |\Theta|=i, |\Theta'|=i-1, \\ \Theta_0 \subset \Theta' \subset \Theta\end{array}}{\bigoplus}d_{P_\Theta,P_{\Theta'}} \epsilon_{\Theta,\Theta'},$$
and $\epsilon_{\Theta,\Theta'}$ is a sign defined in \cite[page 2187]{Aubert}. Then she set $\widehat\pi$ the cokernel of $d_{|\Theta_0|+1}$.

Let $\chi$ be a character of the center of $G(F)$ and $\Omega$ be a Bernstein block of $\tx{Alg}(G)$. Denote by $\tx{Alg}^{\tx{fl}}_\chi(\Omega)$  the full subcategory of $\tx{Alg}(G)$ of finite length representations whose irreducible subquotients have central character $\chi$ and are all of type $\Omega$.

Schneider-Stuhler in  \cite[III.3.1]{SchStu} and Bernstein-Bezrukavnikov in \cite[IV.5.1]{Bernstein} and \cite[\textsection 4]{Bezru} have defined, in terms of cohomology, an involution
$$\mc{E}:\tx{Alg}^{\tx{fl}}_\chi(\Omega) \longrightarrow \tx{Alg}^{\tx{fl}}_{\chi^{-1}}(\Omega)$$
that coincides, by \cite[p. 184]{SchStu},  with the contragredient of the Aubert involution. That is, for every $\pi \in \tx{Alg}^{\tx{fl}}_\chi(\Omega)$ we have that $\mc{E}(\pi^\vee)=\widehat{\pi}$.

The Aubert involution has the following properties:\\

{\bf Irreducibility.} The representation $\widehat{\pi}$ is  irreducible if and only if $\pi$ is irreducible \cite[Corollaire 3.9]{Aubert}.\\

{\bf Induction.} If $P$ is a parabolic subgroup of $G$ with Levi component $M$  then \cite[Theorem 31(3)]{Bernstein}
\begin{equation}\label{ap:ber}
\widehat{i_P^G(\cdot)} = i_{P^-}^G(\widehat{\cdot}).
\end{equation}

{\bf Conjugation.} Let $P$ be a parabolic subgroup of $G$ and let $h\in G(F)$. Let $\pi \in \tx{Alg}(\Theta_0)$ and denote by $h\pi$ the representation of $G(F)$ with same underlying vector space as $\pi$ and action given by $h\pi(g)=\pi(h^{-1}gh)$. The canonical isomorphism
\begin{eqnarray*}
h(i_P^G r_P^G(\pi)) & \longrightarrow  &i_{hP}^G r_{hP}^G(h\pi)\\
f & \mapsto & f(h^{-1} \cdot h)
\end{eqnarray*}
gives rise to a canonical identification $h\widehat{\pi}=\widehat{h\pi}$.\\

{\bf Functoriality.} The involution
\begin{eqnarray*}\tx{Alg}^{\tx{fl}}_\chi(\Omega)&\longrightarrow &\tx{Alg}^{\tx{fl}}_\chi(\Omega) \\
\pi &\mapsto&\widehat{\pi}
\end{eqnarray*}
is an exact covariant involutive functor \cite[III.3]{SchStu}. If $\pi, \pi' \in \tx{Alg}^{\tx{fl}}_\chi(\Omega)$, we deduce that there is a canonical involutive isomorphism
\begin{eqnarray}\label{ap:222}
\Hom_G\left(\pi, \pi'\right)& \longrightarrow &\Hom_G\left(\widehat{\pi}, \widehat{\pi'} \right) \notag \\
\varphi & \mapsto & \widehat{\varphi}
\end{eqnarray}
such that $\widehat{\varphi(\pi)}=\widehat{\varphi}(\widehat{\pi})$. Indeed this isomorphism can be constructed as follows:
Let $\varphi \in \Hom_G\left(\pi, \pi'\right)$. For any  parabolic subgroup $P$ of $G$ the map $\varphi$ induces, by functoriality, \begin{equation}\label{ap:phiP}\varphi_P\in \Hom_G\left(i_{P}^G\circ r_{P}^G(\pi),i_{P}^G\circ r_{P}^G(\pi')\right).\end{equation} For any two  parabolic subgroups $P \subset P'$, by definition \eqref{eq:defd}, these maps make the following diagram commutative:
\[ \xymatrix{ i_{P}^G\circ r_{P}^G(\pi)  \ar[d]^{\varphi_P}\ar[r]^{d_{P,P'}} & i_{P'}^G\circ r_{P'}^G(\pi) \ar[d]^{\varphi_{P'}}\\  i_{P}^G\circ r_{P}^G(\pi')  \ar[r]^{d_{P,P'}} & i_{P'}^G\circ r_{P'}^G(\pi') }
\]
We deduce from \eqref{eq:seqAubert} a commutative diagram:

\[\hspace{-2cm}\xymatrix{
0 \ar[r]\ar[d]& \pi \ar[d]^\varphi \ar[r]& \underset{|\Theta|=|\Delta_0|-1}{\bigoplus}i_{P_\Theta}^G\circ r_{P_\Theta}^G(\pi) \ar[r]\ar[d]^{\varphi_{P_{|\Delta_0|-1}}}& \dots \ar[r]& \underset{|\Theta|=|\Theta_0|}{\bigoplus}i_{P_\Theta}^G\circ r_{P_\Theta}^G(\pi) \ar[r]\ar[d]^{\varphi_{|\Theta_0|}}&\widehat{\pi}  \ar[r]& 0\\
0 \ar[r]& \pi' \ar[r]& \underset{|\Theta|=|\Delta_0|-1}{\bigoplus}i_{P_\Theta}^G\circ r_{P_\Theta}^G(\pi') \ar[r]& \dots \ar[r]& \underset{|\Theta|=|\Theta_0|}{\bigoplus}i_{P_\Theta}^G\circ r_{P_\Theta}^G(\pi') \ar[r]&\widehat{\pi'}  \ar[r]& 0
}\]
(to simplify notation we drop the fact that the sum is taken over parabolic subgroups $P_\Theta$ such that $\Theta_0 \subset \Theta $).
Thus there exists a unique morphism $\widehat{\varphi}\in  \Hom_G\left(\widehat{\pi}, \widehat{\pi'} \right)$ making the diagram

\begin{equation}
\label{eq:commBan10}\hspace{-2cm}\xymatrix{
0 \ar[r]\ar[d]& \pi \ar[d]^\varphi\ar[r]& \underset{|\Theta|=|\Delta_0|-1}{\bigoplus}i_{P_\Theta}^G\circ r_{P_\Theta}^G(\pi) \ar[r]\ar[d]^{\varphi_{P_{|\Delta_0|-1}}}& \dots \ar[r]& \underset{|\Theta|=|\Theta_0|}{\bigoplus}i_{P_\Theta}^G\circ r_{P_\Theta}^G(\pi) \ar[r]\ar[d]^{\varphi_{|\Theta_0|}}&\widehat{\pi}  \ar[r]\ar[d]^{\widehat{\varphi}}& 0\\
0 \ar[r]& \pi' \ar[r]& \underset{|\Theta|=|\Delta_0|-1}{\bigoplus}i_{P_\Theta}^G\circ r_{P_\Theta}^G(\pi') \ar[r]& \dots \ar[r]& \underset{|\Theta|=|\Theta_0|}{\bigoplus}i_{P_\Theta}^G\circ r_{P_\Theta}^G(\pi') \ar[r]&\widehat{\pi'}  \ar[r]& 0
}
\end{equation}
commute. The commutativity of the diagram implies that  $\widehat{\varphi(\pi)}=\widehat{\varphi}(\widehat{\pi})$.\\

{\bf Rationality.} Here we follow \cite[\textsection IV.1]{Walds} and \cite{BernDel}. Let $M \in \mathcal{L}^G$ and $(\pi,V)$ a smooth admissible representation of $M(F)$. Let $B$ be  the polynomial algebra on the variety $\Psi(M)$ and, for $m \in M(F)$, denote $b_m\in B$ the polynomial defined by $b_m(\chi)=\chi(m)$.

Define the algebraic $B$-family $(\pi_B,V_B)$ of admissible representations (in the sense of \cite[1.16]{BernDel}, \cite[\textsection I.5]{Walds}) of $M(F)$ by
\[V_B=V \otimes_\C B, \qquad \pi_B(m)(v \otimes b)=\pi(m)v \otimes b_m b
\]
for every $m \in M(F)$, $v \in V$ and $b \in B$.

For $\chi \in \Psi(M)$ set $\tx{sp}_\chi V_B:= V_B \otimes_B (B/B_\chi)$, the specialization map, where $B_\chi$ is the maximal ideal in $B$ made of functions that are zero on $\chi$. As a representation of $M$, $\tx{sp}_\chi V_B$ is isomorphic to $\pi _\chi$.

As the functors $i_P^G$ and $r_P^G$ of parabolic induction and restriction are exact, they commute with $\otimes_B N$ for any $B$-module $N$ (\cite[2.5]{BernDel}). We deduce that $\tx{sp}_\chi i_M^G(V_B) $ is isomorphic to $i_M^G(\pi_\chi)$.

Let now $P,P' \in \mathcal{P}(M)$ and $\pi,\pi'$ admissible representations of $M(F)$. Suppose that we have a family of intertwining operators
$$A(\chi):i_P^G(\pi _\chi) \longrightarrow i_{P'}^G(\pi' _\chi)$$
for every $\chi \in \Psi(M)$. We say that $A(\chi)$ is \emph{rational} if there exists a homomorphism of $G-B$-modules
$$A_B:i_P^G(\pi_B) \longrightarrow i_{P'}^G(\pi'_B)$$
and $b\in B$ such that $b(\chi)A(\chi)\tx{sp}_\chi=\tx{sp}_\chi A_B$.

The Aubert involution preserves rationality in the following sense.
\begin{prop} Let $\pi,\pi'$ be smooth  finite length representations of $M(F)$ in the same Bernstein block.
With notation as above, suppose that the family of intertwining operators
$$A(\chi):i_P^G(\pi _\chi) \longrightarrow i_{P'}^G(\pi' _\chi)$$
is rational. Then the family of intertwining operators
$$\widehat{A(\chi)}:\widehat{i_P^G(\pi _\chi)} \longrightarrow \widehat{i_{P'}^G(\pi' _\chi)}$$
is also rational.
\end{prop}
\begin{proof}
Let $Q$ be a parabolic subgroup of $G$. Then the intertwining operator
$$A(\chi)_Q: i_{Q}^G\circ r_{Q}^G(i^G_P(\pi_\chi))  \to  i_{Q}^G\circ r_{Q}^G(i^G_{P}(\pi'_\chi)),$$
induced by functoriality,
is rational. Indeed, let
$$A_{B,Q}: i_{Q}^G\circ r_{Q}^G(i^G_P(\pi_B))  \to  i_{Q}^G\circ r_{Q}^G(i^G_{P}(\pi'_B)),$$
be the intertwining deduced by functoriality from $A_B$. Then, for every $\chi \in \Psi(M)$, we have
\begin{eqnarray*}
\tx{sp}_\chi A_{B,Q} &=& (\tx{sp}_\chi A_{B})_{Q} \qquad \text{ ( $\tx{sp}_\chi$ commutes to $r_{Q}^G\circ i^G_P$)} \\
&=& (b(\chi)A(\chi)\tx{sp}_\chi)_{Q} \qquad \text{ ( $A(\chi)$ is rational)}\\
 &=& b(\chi)A(\chi)_{Q}\tx{sp}_\chi \qquad \text{ ( $\tx{sp}_\chi$ commutes to $r_{Q}^G\circ i^G_P$)}.
\end{eqnarray*}
We deduce from our construction \eqref{eq:commBan10} that the operator $$\widehat{A(\chi)}:\widehat{i_P^G(\pi _\chi)} \longrightarrow \widehat{i_{P'}^G(\pi' _\chi)}$$
is also rational.
\end{proof}

\subsection{Intertwining operators and $R$-groups}\label{sub:A2}

\subsubsection{Definition and properties}
The main references here are \cite{ArtIOR1} \cite{Walds}. Let $G$ be a connected reductive group. Fix $M \in \mathcal{L}^G$, $P,P' \in \mathcal{P}(M)$, $\pi$ a smooth representation of $M(F)$ of finite length and $\chi \in \Psi(M)$. Write $U_P$ and $U_{P'}$ for the unipotent radicals of $P$ and $P'$, respectively.
Let
\[J_{P'|P}(\pi_\chi) : i^G_P(\pi_\chi)  \to i^G_{P'}(\pi_\chi)
\]
be the standard intertwining operator defined by the absolutely convergent integral
\[(J_{P'|P}(\pi_\chi) \phi)(x)= \int_{U_P(F)\cap U_{P'}(F)\backslash U_{P'}(F)}\phi(ux)du, \qquad x \in G(F),
\]
when the real part of $\chi$ lies in a certain chamber (see \cite[IV.1]{Walds} for a precise meaning of this convergence). This intertwining operator satisfies a series of properties (properties $J_1-J_5$ in \cite[pag. 26]{ArtIOR1}). In particular, the family of intertwining operators is rational in the sense explained before.

\subsubsection{Normalized self-intertwining operators and the Knapp-Stein $R$-group}\label{subsub:A22}

Let $r_{P'|P}(\pi_\chi)$ be the normalizing factor, that is a rational function on $\Psi(M)$, such that the \textit{normalized} intertwining operator
\[R_{P'|P}(\pi_\chi):=r_{P'|P}(\pi_\chi)^{-1}J_{P'|P}(\pi_\chi)
\] satisfies properties $R_1-R_8$ in \cite[pag. 28]{ArtIOR1}. The family of such normalizing factors exists by \cite[Theorem 2.1]{ArtIOR1} and  $R_{P'|P}(\pi_\chi)$ is a rational family (in $\chi$) of intertwining operators $i^G_P(\pi_\chi) \to i^G_{w^{-1}P}(\pi_\chi)$. If $R_{P'|P}(\pi_\chi)$ is well defined at $\chi=1$, we set $R_{P'|P}(\pi)=R_{P'|P}(\pi_1)$ (for example if $\pi$ unitary \cite[pag. 28, $R_4$]{ArtIOR1}).

Now let $\pi \in \Pi_{\tx{unit}}(M)$, $w \in W_\pi$ and fix $P \in \mathcal{P}(M)$. Choose a representative $\tilde{w} \in G(F)$ of $w$ and an isomorphism $\pi(\tilde{w}):\tilde{w}\pi \to \pi$.
Define the normalized self-intertwining operator $R_P( \tilde{w},\pi)$ by composition of the following intertwining operators:
\[i^G_P(\pi) \overset{R_{w^{-1}P|P}(\pi)}{\longrightarrow}i^G_{w^{-1}P}(\pi) \overset{l(w):\phi\mapsto \phi(\tilde{w}^{-1}\cdot)}{\longrightarrow}i^G_P(\tilde{w}\pi) \overset{\pi(\tilde{w})}{\longrightarrow}i^G_P(\pi) \]

\begin{rem}
The class of $R_P( \tilde{w}, \pi) \mod \C^\times$ is independent of the choices of $\tilde{w}$ and $\pi(\tilde{w})$  and  for any $w,w' \in W_\pi$, $R_P( \tilde{w},\pi) R_P( \tilde{w}',\pi) =R_P( \tilde{w}\tilde{w}',\pi) \mod \C^\times$.

However it depends on the choice of $P$. For another parabolic subgroup $Q \in \mathcal{P}(M)$ one has
\[R_Q( \tilde{w}, \pi)=R_{P|Q}(  \pi)^{-1} R_P( \tilde{w}, \pi)R_{P|Q}(  \pi).
\]
So in particular, $R_P( \tilde{w}, \pi)=R_{P^-}( \tilde{w}, \pi)\mod \C^\times$.
\end{rem}

Suppose now that $\pi\in \Pi_2(M)$. The Knapp-Stein $R$-group $R(\pi)$ can be defined as a subgroup of $W_\pi$ satisfying the following properties (this is a theorem of Harish-Chandra and Silberger \cite{Silberger}):
\begin{enumerate}
\item $W_\pi =  \{w \in W_\pi: R_P(\tilde{w},\pi) \in \C^\times \} \rtimes R(\pi)$.
\item\label{eq:Ban2} The set $\{ R_P( \tilde{w},\pi), r \in R_\pi \}$ is a basis for $\End_G\left(i_M^G(\pi)\right)$.
\end{enumerate}

\subsubsection{Intertwining operators and the Aubert involution}
Let $M$ be a Levi subgroup of $G$ and let $\pi\in \Pi_{\tx{unit}}(M)$. Let $P$ be a parabolic subgroup with Levi component $M$. Let $w \in W(M)$ and for any $\chi \in \Psi(M)$, put
$${\bf R}_P(\tilde{w},\pi_\chi):= l(w) \circ R_{w^{-1}P|P}(\pi_\chi):i^G_P(\pi_\chi) \longrightarrow i^G_P(\tilde{w}\pi_\chi).$$
As $\pi$ is unitary, the operator ${\bf R}_P(\tilde{w},\pi_\chi)$ is well defined at $\chi=1$ so we define
${\bf R}_P(\tilde{w},\pi)={\bf R}_P(\tilde{w},\pi_1)$.

\begin{prop}\label{an:propban}
Suppose moreover that $\widehat{\pi}$ is unitary (so that the operator ${\bf R}_P(\tilde{w},\widehat{\pi}_\chi)$ is well defined at $\chi=1$). Then $\widehat{{\bf R}_P(\tilde{w},\pi)} ={\bf R}_{P^-}(\tilde{w},\widehat{\pi})\mod \C^\times$.
\end{prop}
\begin{proof}
The intertwining operator
\[
\widehat{{\bf R}_P(\tilde{w},\pi_\chi)} :i_{P^-}^G(\widehat{\pi}_\chi) \longrightarrow :i_{P^-}^G(\tilde{w}\widehat{\pi}_\chi)
\]
is by construction rational and bijective whenever ${\bf R}_P(\tilde{w},\pi_\chi)$ is (in particular it is well defined at $\chi=1$).

If we compose it with ${\bf R}_{P^-}(\tilde{w}^{-1},\widehat{\tilde{w}\pi}_\chi)$, which equals ${\bf R}_{P^-}(\tilde{w}^{-1},\tilde{w}\widehat{\pi}_\chi)$ by the conjugation property, we get a self-intertwining operator of $i_{P^-}^G(\widehat{\pi}_\chi)$. For regular $\sigma$ this latter representation is irreducible so we deduce
\[ {\bf R}_{P^-}(\tilde{w}^{-1},\tilde{w}\widehat{\pi}_\chi) \circ \widehat{{\bf R}_P(\tilde{w},\pi_\chi)}= a(\chi)\cdot \tx{Id}_{i_{P^-}^G(\widehat{\pi}_\chi)}
\]
where $a(\chi)$ is a rational function on $\Psi(M)$. By hypothesis and construction both intertwining operators $ {\bf R}_{P^-}(\tilde{w}^{-1},\tilde{w}\widehat{\pi}_\chi)$ and $\widehat{{\bf R}_P(\tilde{w},\pi_\chi)}$ are well defined at $\chi=1$ and  bijective. Hence
\[ {\bf R}_{P^-}(\tilde{w}^{-1},\tilde{w}\widehat{\pi}) \circ \widehat{{\bf R}_P(\tilde{w},\pi)} = a \cdot \tx{Id}_{i_{P^-}^G(\widehat{\pi})}
\]
with $a \in \C^\times$. The intertwining operator $\widehat{{\bf R}_P(\tilde{w},\pi)}$ is the inverse of ${\bf R}_{P^-}(\tilde{w}^{-1},\tilde{w}\widehat{\pi})  $ up to a constant so it must be equal to ${\bf R}_{P^-}(\tilde{w},\widehat{\pi})$ up to a constant.
\end{proof}

\begin{coro}\label{ap:co221}
Let $M$ be a Levi subgroup of $G$ and let $\pi\in \Pi_{\tx{unit}}(M)$. Suppose $\widehat{\pi}$ is unitary.  Let $P$ be a parabolic subgroup with Levi component $M$ and let $w \in W_\pi$. Then  $R_{P^-}( \tilde{w},\widehat{\pi})=\widehat{R_P( \tilde{w}, \pi)} \mod \C^\times$.
\end{coro}
\begin{proof}
Recall that the normalized self-intertwining operator $R_P( \tilde{w},\pi)$ is defined by composition of the following intertwining operators:
\[i^G_P(\pi) \overset{{\bf R}_P(\tilde{w},\pi)}{\longrightarrow} i^G_P(\tilde{w}\pi) \overset{\pi(\tilde{w})}{\longrightarrow}i^G_P(\pi). \]
By functoriality of \eqref{ap:222} and Proposition \ref{an:propban}, it is enough to prove that we have commutative diagram
\begin{equation}\label{ap:square} \xymatrix{i^G_P(\tilde{w}\pi)  \ar[d]^{\widehat{\,}}  \ar[r]^{\pi(\tilde{w})} & i^G_P(\pi)\ar[d]^{\widehat{\,}}
\\  i^G_{P^-}(\tilde{w}\widehat{\pi})\ar[r]^{\widehat{\pi}(\tilde{w})} &   i^G_{P^-}(\widehat{\pi})}
\end{equation}
up to isomorphism.
 But
\[\widehat{\pi(\tilde{w})}
\]
is an isomorphism between  $\tilde{w}\widehat{\pi}$ and $\widehat{\pi}$. By Schur's Lemma it is equal to $\widehat{\pi}(\tilde{w})$ up to a nonzero constant. We deduce the commutativity of the diagram by \eqref{ap:ber}.
\end{proof}

\begin{rem}
In the article we will be in the following setting.
Let $M$ be a Levi subgroup of $G$ and let $\pi\in \Pi_{\tx{unit}}(M)$. Suppose we have a group $N_\pi$ endowed with a surjective morphism $p: N_\pi \to W_\pi$. Suppose moreover that $\widehat{\pi}$ is unitary and that we have normalizations $R_{P}( u,\pi)$, $u \in N_\pi$, of the intertwining operators such that:
\begin{enumerate}
\item $R_{P}( u,\pi)=R_{P}( p(u),\pi) \mod \C^\times,$ for every $\pi\in \Pi_{\tx{unit}}(M)$ and every $u \in N_\pi$.
\item The intertwining operators $R_{P}( u,\pi)$  are mutiplicative on $N_\pi$ \textit{i.e.} they are such that for any $u,u' \in N_\pi$, $R_P( u,\pi) R_P( u',\pi) =R_P( uu',\pi)$.
\end{enumerate}
Then there exists a multiplicative character $\epsilon_\pi:N_\pi \to \C^\times$ such that for every $u \in N_\pi$ and every $f$ in the Hecke algebra of $G(F)$ we have
\[f_G(\hat{\pi},u)=\epsilon_\pi(u) \tx{tr}(\widehat{R_P( u,\pi)}\widehat{i_P^G(\pi)},f) .\]

\end{rem}

We finish this appendix proving the corollary announced in the introduction.
\begin{coro}\label{ap:co222}
Let $M$ be any Levi subgroup of $G$ and let $\pi\in \Pi_{2}(M)$. Suppose $\widehat{\pi}$ is unitary. Let $P$ be a parabolic subgroup with Levi component $M$. Denote by $R(\pi)$  the Knapp-Stein $R$-group for $\pi$. Then the set of normalized self-intertwining operators
\[\{R_P(r,\widehat{\pi}), r \in R(\pi)\}
\]
is a basis for the commuting algebra $\End_G\left(i_M^G(\widehat{\pi})\right)$.
\end{coro}
\begin{proof}
The set $\{ R_{P^-}( \tilde{w},\pi), r \in R_\pi \}$ is a basis for $\End_G\left(i_{P^-}^G(\pi)\right)$. Thus by \eqref{ap:222}, the set $\{ \widehat{R_{P^-}( \tilde{w},\pi)}, r \in R_\pi \}$ is a basis for $\End_G\left(i_P^G(\widehat{\pi})\right)$, so by Corollary \ref{ap:co221} we deduce that the set $\{ R_P(r,\widehat{\pi}), r \in R(\pi), r \in R_\pi \}$ is a basis for $\End_G\left(i_P^G(\widehat{\pi})\right)$.

\end{proof}

\bibliographystyle{amsalpha}
\bibliography{unitary}

\bigskip
  \footnotesize{

  Tasho Kaletha, \textsc{Department of Mathematics, Harvard University,
   Cambridge, MA 02138, USA}\par\nopagebreak
  \textit{E-mail address}: \texttt{tasho.kaletha@gmail.com}

  \medskip

  Alberto Minguez, \textsc{Institut de Math\'ematiques de Jussieu - Paris Rive Gauche, Universit\'e Paris 6,
        4 place Jussieu, 75005, Paris, France}\par\nopagebreak
  \textit{E-mail address}: \texttt{alberto.minguez@imj-prg.fr}

  \medskip

  Sug Woo Shin, \textsc{Department of Mathematics, University of California, Berkeley,
  Berkeley, CA 94720-3840, USA / Korea Institute for Advanced Study, Dongdaemun-gu, Seoul 130-722, Republic of Korea}\par\nopagebreak
  \textit{E-mail address}: \texttt{prmideal@gmail.com}

\medskip

  Paul-James White, \textsc{Mathematical Institute, University of Oxford, %
 Oxford, OX2 6GG, UK}\par\nopagebreak
  \textit{E-mail address}: \texttt{pauljames.white@gmail.com}

}

\end{document}